\documentclass[article,lineno]{biometrika2}

\usepackage{amsmath}

\usepackage{times}
\usepackage{bm}

\newcommand{\RNum}[1]{\uppercase\expandafter{\romannumeral #1\relax}}

\def\T{{ \mathrm{\scriptscriptstyle T} }}

\newcommand{\MBF}{} %\mathbf
\newcommand{\MBB}{} %\mathbb
\newcommand{\BSM}{} %\boldsymbol
\newcommand{\BIEN}{\mathrm{I}_p}
\newcommand{\MYTRAN}{\T}
\newcommand{\MYPR}{\mathrm{pr}} %\Pr

\usepackage{color,xcolor}

\usepackage{graphicx}
\usepackage{subcaption}

\usepackage{booktabs}
\usepackage{enumitem}
\usepackage{lscape}
\usepackage{rotating}

\usepackage{natbib}
\begin{document}

\jname{}
%% The year, volume, and number are determined on publication
\jyear{}
\jvol{}
\jnum{}
%% The \doi{...} and \accessdate commands are used by the production team
%\doi{10.1093/biomet/asm023}
%\accessdate{Advance Access publication on 31 July 2018}

%% These dates are usually set by the production team
%\received{{\rm 2} January {\rm 2017}}
%\revised{{\rm 1} April {\rm 2017}}

%\received{January 2017}
%\revised{April 2017}

%% The left and right page headers are defined here:
\markboth{He et~al.}{On the Phase Transition of Wilk's Phenomenon}

%% Here are the title, author names and addresses
\title{On the Phase Transition of Wilk's Phenomenon}

\author{Yinqiu He}
\affil{Department of Statistics, University of Michigan, MI 48109, U.S.A. 
 \email{yqhe@umich.edu}}

\author{Bo Meng,  Zhenghao Zeng}
\affil{University of Science and Technology of China, Anhui, 230026, P.R. China
\email{mb0529@mail.ustc.edu.cn, zzh98052@mail.ustc.edu.cn}}

%\author{Zhenghao Zeng}
%\affil{School of , University of Science and Technology of China
%\email{editor@biometrika.org.uk}}

\author{\and Gongjun Xu}
\affil{Department of Statistics, University of Michigan, MI 48109, U.S.A. 
\email{gongjun@umich.edu}}

%\author{\and D. M. TITTERINGTON}
%\affil{Department of Statistics, University of Glasgow, Glasgow G12 8QQ, U.K. \email{mike@stats.gla.ac.uk}}

\maketitle

%\begin{abstract}
%There should be a single paragraph summary which should not contain formulae or symbols, followed by some key words in alphabetical order.  Typically there are 3--8 key words, which should contain nouns and be singular rather than plural.  The summary contains bibliographic references only if they are essential.  It should indicate results rather than describe the contents of the paper: for example, `A simulation study is performed' should be replaced by a more informative phrase such as `In a simulation our estimator had smaller mean square error than its main competitors.'
%\end{abstract}

%\begin{document}
%
%\begin{frontmatter}
%
%\title{A Note on the Phase Transition Boundaries of the Chi-Square Approximations for Likelihood Ratio Tests}
%\tnotetext[mytitlenote]{Fully documented templates are available in the elsarticle package on }
%% \href{http://www.ctan.org/tex-archive/macros/latex/contrib/elsarticle}{CTAN}.}
%
%%% Group authors per affiliation:
%\author{USTCers}
\begin{abstract}
 Wilk's theorem, which offers universal chi-squared  approximations for likelihood ratio tests, is widely used in many scientific hypothesis testing problems.  For modern datasets with increasing dimension, researchers have found that the conventional Wilk's phenomenon of the likelihood ratio test statistic often fails. Although new approximations have been proposed in high dimensional settings, there still lacks a clear statistical guideline regarding how to choose between the conventional and newly proposed  approximations, especially for moderate-dimensional data. To address this issue, we  develop the necessary and sufficient phase transition conditions for Wilk's phenomenon under popular tests on multivariate mean and covariance structures. Moreover, we provide an in-depth analysis of the accuracy of chi-squared approximations by deriving their asymptotic biases. These results may provide helpful insights into the use of  chi-squared approximations in scientific practices. 
%The results lead to  quantitative guidelines to check, and would provide useful statistical insights in  scientific practices.
\end{abstract}


%\begin{keywords}
%Address; Appendix; Figure; Length; Reference; Style; Summary; Table.
%\end{keywords}

%\linenumbers

\section{Introduction} \label{Intro}

The likelihood ratio test is a standard testing method for many hypothesis testing problems due to its nice statistical properties  \citep{anderson2003introduction,Muirhead2009}.   
% It enjoys the optimality property  under regularity conditions
%a general test for the goodness-of-fit between two models, which  enjoys certain optimality under regularity conditions and is widely applied in many hypothesis testing problems \citep{dasgupta2008asymptotic,Muirhead2009}. 
Under the low-dimensional setting with a fixed number of parameters $p$ and large sample size $n$,
%when the number of parameters $p$ is fixed compared to the sample size $n$, 
classic theorems offer general asymptotic results for various likelihood ratio test  statistics. One of the most celebrated and fundamental results is Wilks' theorem, which  states that, %when $p$ is fixed compared to $n$, 
under the null hypothesis, twice the negative log-likelihood ratio asymptotically approaches a  $\chi^2_f$ distribution, where $f$ is the difference of the degrees of freedom between the null and alternative hypotheses. 
 The popularly used Bartlett correction provides a general rescaling strategy that further improves the finite sample accuracy of the chi-squared  approximations  \citep{cordeiro2014introduction,barndorff1988level}. 
% \citep[see,][]{dasgupta2008asymptotic,barndorff1988level}. 
Similar Wilk's phenomenon and Bartlett correction were also studied for empirical likelihood  \citep{owen1990empirical,diciccio1991empirical,chen2006bartlett}.  

%More generally, similar  asymptotic distributions of the LRTs have also been established in the context of nonparametric problems \cite[e.g.,][]{fan2000geometric,fan2001generalized}.  

%there also exist works relaxing the constraint on the dimensionality and provide sufficient conditions for the chi-square-type approximation to hold. ... Or other modifications of the LRT ....

%specifying the sufficient conditions for the chi-square-type approximation of the likelihood ratio test as $p$ increases with $n$.  ...

%Traditional multivariate statistics  

%In the traditional field of Multivariate Analysis, most of the modeling and inference methods are well developed under the assumption that the number of variables, denoted by $p$, is considered to be fixed or negligible compared with the sample size, denoted by $n$. One important example is the widely used $\chi^2$ approximation in different likelihood ratio tests. However, with rapid development of modern technology in finance, computers and manufacturing, nowadays we are faced with some high-dimensional data featuring simultaneously large $n$ and $p$. Some examples of high dimensional data can be found in \cite{Donoho2000}.

%However
%Given modern applications in various scientific disciplines, 

%Despite the extensive literature on the likelihood ratio tests under finite dimensions,
Despite the extensive literature on the Wilk's-type phenomenon of likelihood ratio tests under finite dimensions,     
it is of emerging interest to study  the large $n$, diverging $p$ asymptotic regions in a wide variety of modern applications. 
To understand how large the dimension $p$ can be to ensure the validity of the 
classical 
Wilk's phenomenon,
% to hold, 
various works establish sufficient conditions on the growth rate  of $p$ as $n$ increases.
% the chi-square approximation when $p$ increases with $n$. 
For instance, \citet{portnoy1988asymptotic} showed that the chi-squared approximation of the likelihood ratio test statistic for a simple hypothesis in canonical exponential families  holds if $p/n^{2/3}\to 0$.
Moreover, \citet{hjort2009extending}, \citet{chen2009effects}, and \cite{tang2010penalized} studied the empirical likelihood ratio statistic when $p \to \infty$.
%, and showed that the results depend on both the data dimension and dependence structure.  
% effects of the data dimension and dependence  on the empirical likelihood ratio statistic for the mean. 
%under certain regularity conditions,
 Particularly,  \citet{chen2009effects} argued that $p/n^{1/2}\to 0$ is likely to be the  best rate for the chi-squared  approximation of general empirical likelihood ratio test,  and showed that for the least-squares empirical likelihood, a simplified version of the empirical likelihood, the chi-squared  approximation holds if $p/n^{2/3}\to 0.$ 
  The effect of data dimension was also studied in other inference problems; see, for example,  \cite{portnoy1985asymptotic}, \cite{he2000parameters},  and \cite{wang2011gee}. 
% Further relaxation on the data dimension was obtained in \cite{tang2010penalized} for empirical likelihood with penalization.  
% Beyond the parametric inference problems, 
% \citet{fan2000geometric},  \citet{fan2001generalized}, and \citet{boucheron2011high} also  explored extensions of the Wilk's theorem from non-parametric or statistical learning perspectives. 
% \citet{fan2000geometric} and \citet{fan2001generalized} also  explored extensions to infinite dimensional non-parametric inference problems. These results also motivated the study on the non-asymptotic Wilk's phenomenon    in \citet{boucheron2011high} from a perspective of statistical learning theory. 
% The effect of data dimension was also studied in other problems; see, for example,  \cite{portnoy1985asymptotic}, \cite{he2000parameters},  and \cite{wang2011gee}. 

When the dimension $p$ further increases, researchers have found that the chi-squared approximations based on Wilk's theorem 
%of the likelihood ratio tests 
often become inaccurate, resulting in the failure of the corresponding likelihood ratio tests. 
%can perform poorly and become invalid. 
%may be no longer valid and perform badly
To address this issue, various corrections and alternative approximations for the likelihood ratio tests have been proposed. 
%To relax the constraints of dimension $p$ for the LRTs,
%For example, 
%For example, early studies in \citet{portnoy1988asymptotic} considered testing a simple null hypotheses in canonical exponential families, and established normal approximations for the likelihood ratio tests when $p\to \infty$ and $p/n^{2/3} \to 0$. 
%\citet{morris1975central}
%\citet{fan2000geometric} and \citet{fan2001generalized} also  explored extensions to infinite dimensional non-parametric inference problems. 
%Moreover, researchers further studied asymptotic region when the dimension $p$  is asymptotically proportional to $n$, namely, $p/n \rightarrow y \in (0,1)$ as $n \rightarrow \infty$. In particular, the seminal work \cite{Bai2009} showed  the failure of the chi-square approximations for two LRTs on testing covariance matrices and correspondingly  proposed corrected normal approximations based on the random matrix theory.
For example, when $p$ is asymptotically proportional to $n$, namely, $p/n \rightarrow y \in (0,1)$ as $n \rightarrow \infty$,  \citet{Bai2009},  \citet{Jiang13}, and \citet{Jiang15} proposed normal approximations for the corrected likelihood ratio tests on testing mean vectors and covariance matrices. 
 \cite{zheng2012central}, \citet{bai2013testing}, and \citet{He2018} proposed  normal approximations for corrected likelihood ratio tests  in multivariate linear regression models. 
%When the dimension $p$ further increases  to be asymptotically proportional to $n$, namely, $p/n \rightarrow y \in (0,1)$ as $n \rightarrow \infty$, \citet{Bai2009} proposed normal approximations for two corrected likelihood ratio tests on testing covariance matrices using the random matrix theory.  
%Moreover, \citet{Jiang13} and \citet{Jiang15}  studied six classical likelihood ratio tests on testing means and covariance matrices, and derived corrected likelihood ratio tests with normal approximations based on the moment generating functions. 
%For other problems such as 
%In addition, for testing the general linear hypothesis in multivariate linear regression, \cite{zheng2012central}, \citet{bai2013testing}, and \citet{He2018} proposed  normal approximations for corrected likelihood ratio tests  under general asymptotic regions. 
Furthermore,   \citet{sur2019modern}, \citet{sur2019likelihood}, and \citet{candes2018phase} studied the phase transition of the maximum likelihood estimator for the logistic regression, and proposed a rescaled chi-squared approximation for the likelihood ratio test.

Despite the proposed distributional theory of the likelihood ratio tests   for low- or high-dimensional data,  there still lacks a quantitative guideline on which approximation should be chosen to use in practice,
%how to choose an appropriate approximation in practice, 
especially for  moderate-dimensional data. For instance, when analyzing a dataset with the number of parameters $p\leq 5$ and sample size $n=100$, 
the chi-squared approximation  may be considered as reliable. 
However, when studying a data set with moderate dimension, e.g., $p$ is between $6$ to $20$ and sample size $n=100$, it may be unclear to  practitioners whether they can still apply the classical chi-squared approximations   or they should turn to other high-dimensional asymptotic results. 
To address this practical issue, 
it is of interest to investigate the phase transition boundary where  the chi-squared approximation starts to fail as $p$ increases,
and also characterize   the approximation accuracy.
%it requires a 
%%deep understanding 
%careful investigation  
%%of the phase transition  behavior 
%of Wilk's phenomenon, including 
%evaluating the phase transition boundary where  
%%when
% the chi-squared approximation starts to fail as $p$ increases,
%  and characterizing the approximation accuracy near the boundary. 
 Theoretically, this needs a deep understanding of the limiting behavior of the likelihood ratio test statistics from low to high dimensions.

In this work, we focus on several standard likelihood ratio tests on  multivariate mean and covariance structures that are  widely used  in biomedical and social sciences \citep{pituch2015applied,cleffapplied}.
For each considered likelihood ratio test, we derive its phase transition boundary of Wilk's phenomenon and also provide an in-depth analysis of the accuracy of the chi-squared approximation.
First, 
%Specifically, first, 
in terms of the phase transition boundary, 
%First, we study
%In particular, 
we establish  
the  \textit{necessary} and \textit{sufficient} condition 
for Wilk's theorem to hold when $p$ increases with $n$. 
Specifically, 
we show that the chi-squared approximations hold if and only if  $p/n^d\to 0$, where the value of $d$ depends on the testing problem and whether the Bartlett correction is used.
Interestingly, the proposed phase transition boundaries resonate with the abovementioned literature   \citep[e.g.,][]{portnoy1988asymptotic,chen2009effects},  which mostly  focused on sufficient conditions without the Bartlett correction.
%; please see Section \ref{sec:main} for details.
Second, we provide a detailed characterization of the asymptotic bias of each  chi-squared approximation. 
Specifically, we consider two 
local asymptotic regimes, depending on whether  Wilk's theorem holds or not.
%the derived necessary and sufficient condition is satisfied or violated, respectively. 
%When Wilk's theorem holds, 
%the chi-squared approximation holds.
Under the asymptotic regime when Wilk's theorem holds,
the derived asymptotic bias 
%gives the sharp
sharply characterizes
%describes 
the convergence rate of the distribution of the likelihood ratio test statistic to the limiting chi-squared distribution, and thus provides a useful measure on the accuracy of the chi-squared approximation. 
%and asymptotically converges to 0 as $p/n^d \to 0$. 
%When the condition is violated,  the chi-squared approximation fails, 
When Wilk's theorem fails,
  the derived asymptotic bias describes the unignorable  discrepancy between the chi-squared approximation and the true distribution of the likelihood ratio test statistic.  
As illustrated in the simulation studies, our theoretical results of the phase transition boundaries and the asymptotic biases may provide a helpful guideline on the use of the chi-squared approximations in practice. 
\section{Results of One-Sample Tests} \label{sec:main}
%This section presents our main theoretical results on the phase transition boundaries of the LRTs.  
%% phase transition boundaries of the LRTs. 
%Particularly, Section \ref{sec:mainone} considers testing mean vectors, covariance matrices, and  jointly testing mean vectors and covariance matrices for one sample; Section \ref{sec:mainmult} considers the same testing problems for multiple samples; Section \ref{sec:mainindp} considers testing the independence of subvectors of a multivariate normal vector. 
%For the LRT in each question, we will derive the necessary and sufficient condition for the chi-square approximation and also the corresponding chi-square approximation with Bartlett correction.

%phase transition boundaries of the chi-square approximation and also the corresponding chi-square approximation with Bartlett correction. 

%\subsection{One-Sample Tests}\label{sec:mainone}
%In this section, we focus on several LRTs for   mean vectors and covariance matrices of   one-sample multivariate normal distribution. For the LRT in each question, we will derive the phase transition boundaries of the chi-square approximation and also the corresponding chi-square approximation with Bartlett correction. 

%In this section, we examine three one-sample testing problems.
%   multivariate normal  vectors.  
%This section examines three one-sample testing problems.
In this section, we present the  theoretical results under three one-sample testing problems. 
We also obtain similar results for other multiple-sample  testing problems, which are introduced in $\S$\,\ref{sec:exten}, and please see their details in the Supplementary Material. 
%Other multiple-sample testing problems are introduced in Section \ref{sec:exten} and please see their details in the Supplementary Material. 
Under one-sample problems, 
suppose $\MBF{x}_1,\dots,\MBF{x}_n\in \MBB{R}^p$ are independent and identically distributed random vectors with distribution $\mathcal{N}_p(\BSM{\mu},\BSM{\Sigma}),$ which denotes a $p$-variate multivariate normal distribution with mean vector $\BSM{\mu}$ and covariance matrix $\BSM{\Sigma}$.  We define $\overline{\MBF{x}} = n^{-1}\sum_{i=1}^{n}\MBF{x}_i$ and $ \MBF{A} = \sum_{i=1}^{n}(\MBF{x}_i-\overline{\MBF{x}})(\MBF{x}_i-\overline{\MBF{x}})^{\MYTRAN}$, and denote the determinant and the trace of  $\MBF{A}$ by $|\MBF{A}|$ and  $\mathrm{tr}(\MBF{A})$,   respectively. We next introduce the considered testing problems and the corresponding likelihood ratio tests  \citep{anderson2003introduction,Muirhead2009}.

\smallskip

(I) \textit{Testing Specified Value for the Mean Vector.}\ This test  examines whether the population mean vector $\BSM{\mu}$ is equal to a specified vector $\BSM{\mu}_0\in \MBB{R}^p$, that is,  $H_0: \BSM{\mu}=\BSM{\mu}_{0}$ against $H_a: H_0$ is not true. 
%a given vector $\BSM{\mu}_0$ 
%We next examine one-sample test on only mean vector, i.e.,  
%%$H_0: \BSM{\mu}=\MBF{0}$.  We examine the one-sample mean test 
%$H_0: \BSM{\mu}=\BSM{\mu}_0$. Similarly, by transforming the data to $\tilde{\MBF{x}}_{i}=\MBF{x}_{i}-\BSM{\mu}_{0}$, we know it suffices to consider the test 
%\begin{align*}
%	H_0: \BSM{\mu}=\BSM{\mu}_{0} \mathrm{\quad vs \quad} H_a: H_0 \mbox{ is not true.} %\label{eq:onesammean}
%\end{align*}
Through the transformation $\MBF{x}_i-\BSM{\mu}_0$, 
we consider, without loss of generality, $\BSM{\mu}_0=(0,\ldots,0)^{\MYTRAN}$.
%, where $\MBF{0}$ denotes an all zero vector.  
%By the transformation $\tilde{\MBF{x}}_i=\MBF{x}_i-\BSM{\mu}_0$, we 
%Then by Theorem 6.3.2 in \citet{Muirhead2009}, 
Then,   the likelihood ratio test statistic is  $\Lambda_n=|\MBF{A}|^{{n}/{2}}(\MBF{A} + n \bar{\MBF{x}} \bar{\MBF{x}}^{\MYTRAN} )^{-{n}/{2}}.$
% where $|\MBF{A}|$ denotes the determinant of the matrix $\MBF{A}$.  
 When $p$ is fixed and $n\to \infty$,  under the null hypothesis,
%  $\BSM{\mu}=\MBF{0}$,   
%for the log likelihood ratio statistic $\log \Lambda_n$, 
 the classical chi-squared approximation without correction is $-2\log \Lambda_n \xrightarrow{d} \chi^2_{f}$, where $\xrightarrow{d}$ represents the convergence in distribution and $f=p$, and the chi-squared approximation with the Bartlett correction is   $-2\rho \log \Lambda_n \xrightarrow{d} \chi^2_{f}$, where  $\rho = 1-(1+p/2)/n$. 
 
\smallskip
%The likelihood ratio statistic  is $\Lambda_n=|\MBF{A}|^{\frac{n}{2}}(\MBF{A} + n \bar{\MBF{x}} \bar{\MBF{x}}^{\MYTRAN} )^{-\frac{n}{2}};$ see, for example, Theorem 6.3.2 in \citet{Muirhead2009}. When $p$ is fixed and $n\to \infty$, under the null hypothesis \eqref{eq:onesammean}, the classical chi-square approximation is $-2\log \Lambda_n \xrightarrow{d} \chi^2_{p}$, where $\xrightarrow{d}$ represents the convergence in distribution; the chi-square approximation with Bartlett correction is $-2\rho \log \Lambda_n \xrightarrow{d} \chi^2_{p}$, where  $\rho = 1-(1+p/2)/n$. 

(II) \textit{Testing the Sphericity of the Covariance Matrix.}\   This test examines whether the covariance matrix $\BSM{\Sigma}$ is proportional to an identity matrix; that is, $H_0: \BSM{\Sigma} = \lambda \BIEN$ against $H_a: H_{0}$ is not true, 
%\quad Consider the spherical test of covariance matrix, that is
%\begin{align}\label{test1}
%H_0: \BSM{\Sigma} = \lambda \BIEN \quad \text{vs} \quad H_a: H_{0} \text { is not true},
%% \BSM{\Sigma} \neq \lambda \BIEN,
%\end{align}
where $\lambda>0$ is an unspecified constant and $\BIEN$ denotes the $p\times p$ identity matrix. 
%This is  also called the sphericity test because under the null hypothesis, the contours of equal density in the multivariate normal distribution are spheres.
%The null hypothesis in \eqref{test1} is also called the hypothesis of sphericity since under \eqref{test1} the contours of equal density in the multivariate normal distribution are spheres. 
%The classical multivariate analysis \citep{Muirhead2009} provides  the likelihood 
%Following Theorem 8.3.2 and Section 8.3.3 in \cite{Muirhead2009}, w
The likelihood ratio test statistic is 
$\Lambda_n = |\MBF{A}|^{{(n-1)}/{2}} \left\{{\mathrm{tr}(\MBF{A})}/{p}\right\}^{-p(n-1)/2}.$
  %\begin{equation}\label{statistic1}
%\Lambda_n = |\MBF{A}|^{\frac{n-1}{2}} \times \left\{{\mathrm{tr}(\MBF{A})}/{p}\right\}^{-p(n-1)/2}.
%\end{equation}
%which is commonly called the ellipticity statistic.
When $p$ is fixed and $n\to \infty$, under the null hypothesis,  the  chi-squared approximation is $ -2\log \Lambda_n \xrightarrow{d} \chi_{f}^2 $, where $f=(p-1)(p+2)/2$, and the chi-squared approximation with the Bartlett correction is $-2\rho\log \Lambda_n \xrightarrow{d} \chi_{f}^2$, where $\rho = 1- \{6(n-1)p\}^{-1}(2p^2+p+2)$. 
 
 \smallskip

(III)\ \textit{Joint Testing Specified Values for the Mean Vector and Covariance Matrix.}\ Consider a specified vector $\BSM{\mu}_0\in \MBB{R}^p$ and a specified positive-definite matrix $\BSM{\Sigma}_0 \in \MBB{R}^{p\times p}$. We study the test $H_{0}:  \BSM{\mu}=\BSM{\mu}_0$ and $\BSM{\Sigma}=\BSM{\Sigma}_0$ against   $H_{a}: H_{0}$ is not true.
% $H_{0}:  \BSM{\mu}=\BSM{\mu}_0$ and $\BSM{\Sigma}=\BSM{\Sigma}_0$ vs $H_{a}: H_{0}$ is not ture. 
%\begin{equation*} %\label{test-22}
%H_{0}:  \BSM{\mu}=\BSM{\mu}_0\ \text { and }\ \BSM{\Sigma}=\BSM{\Sigma}_0 \quad \text { vs } \quad H_{a}: H_{0} \text { is not true}.
%\end{equation*}
By applying the transformation $\BSM{\Sigma}_0^{-1/2}(\MBF{x}_i -\BSM{\mu}_0)$, we assume, without loss of generality, that $\BSM{\mu}_0=\MBF{0}$ and $\BSM{\Sigma}_0=\BIEN$.
% in \eqref{test-22}.  
%Then by  Theorem 8.5.1 in  \citet{Muirhead2009},
Then, the likelihood ratio test statistic is 
$\Lambda_{n}=\left({e}/{n}\right)^{n p / 2}|\MBF{A}|^{n / 2}\exp\{ -\operatorname{tr}(\MBF{A}) / 2 -n \overline{\MBF{x}}^{\MYTRAN} \overline{\MBF{x}} / 2 \}.$
%$$\Lambda_{n}=\left({e}/{n}\right)^{n p / 2}|\MBF{A}|^{n / 2} e^{-\operatorname{tr}(\MBF{A}) / 2} e^{-n \overline{\MBF{x}}^{\MYTRAN} \overline{\MBF{x}} / 2}.$$
%It follows that the likelihood ratio test statistic for hypothesis \eqref{test-22}  is; see, for example, Theorem 8.5.1 in  \citet{Muirhead2009}. 
%\begin{equation}\label{stat-2-1}
%\Lambda_{n}=\left({e}/{n}\right)^{n p / 2}|\MBF{A}|^{n / 2} e^{-\operatorname{tr}(\MBF{A}) / 2} e^{-n \overline{\MBF{x}}^{\MYTRAN} \overline{\MBF{x}} / 2}.  
%\end{equation}
%When $p$ is fixed and $n \to \infty$, under the null hypothesis，% $ \BSM{\mu}=\MBF{0}$ and $\BSM{\Sigma}=\BIEN$,
%When $p$ is fixed and $n \to \infty$, under the null hypothesis，
When $p$ is fixed and $n\to \infty$, under the null hypothesis, 
the chi-squared approximation is $-2  \log \Lambda_{n} \xrightarrow{d} \chi_{f}^{2},$ where $f=  p(p+3)/2$, and the chi-squared approximation with the  Bartlett correction is $-2\rho \log \Lambda_n \xrightarrow{d} \chi_{f}^{2}$, where $\rho=1-\{6 n(p+3) \}^{-1}(2 p^{2}+9 p+11).$

 \smallskip
For the likelihood ratio tests of the above three testing problems, Theorem  \ref{thm:onesam} gives  
%the necessary and sufficient conditions of the  the chi-square approximations in the following Theorem  \ref{thm:onesam}.  
the phase transition boundaries of the chi-squared approximations without and with the Bartlett correction.

\begin{theorem}\label{thm:onesam}
Assume $n>p+1$ for all $n\geq 3$ and $n-p\to \infty$ as $n\to \infty$. 
Under $H_0,$ for the chi-squared approximations without and with the Bartlett correction of each likelihood ratio test in (I)--(III), we have the following necessary and sufficient conditions: 
%in the sense that under $H_0,$  as $n\to \infty,$ 
%for any significance level $\alpha$,
%\begin{itemize}
%	\item[(i)] $\sup_{\alpha\in (0,1)}|\MYPR\{-2\times \log \Lambda_{n}>\chi^2_{f}(\alpha)\}-\alpha|\to 0$ if and only if $p/n^{d_1} \to 0;$
%	\item[(ii)]  $\sup_{\alpha\in (0,1)}|\MYPR\{-2\rho\times  \log \Lambda_{n}>\chi^2_{f}(\alpha)\}- \alpha| \to 0$ if and only if $p/n^{d_2} \to 0,$
%\end{itemize} where the values of $d_1$ and $d_2$ under three testing problems are listed in the table below. 
%\smallskip
\smallskip

(i) $\sup_{\alpha\in (0,1)}|\MYPR\{-2 \log \Lambda_{n}>\chi^2_{f}(\alpha)\}-\alpha|\to 0$ if and only if $p/n^{d_1} \to 0;$

(ii)   $\sup_{\alpha\in (0,1)}|\MYPR\{-2\rho  \log \Lambda_{n}>\chi^2_{f}(\alpha)\}- \alpha| \to 0$ if and only if $p/n^{d_2} \to 0,$

\smallskip

 \noindent where the values of $d_1$ and $d_2$ under the three testing problems are listed in the table below. 
\begin{center}
	\setlength{\tabcolsep}{18pt}
\begin{tabular}{lccc} 
% \hline 
%\smallskip
      & (I)  Mean & (II) Covariance & (III) Joint   \\  
(i) without correction $d_1$:  &    $2/3$ &   $1/2$    &  $1/2$  \\ 
(ii) with correction $d_2$:  &    $4/5$ &   $2/3$    &  $2/3$  \\ 
%\hline
\end{tabular}
\end{center} 
\end{theorem}

\medskip

%We discuss the conditions in Theorem \ref{thm:onesam}.  Here 

%To see this, consider $\BSM{\Sigma}=\MBF{I}_p$. Then by Theorem 3.1.2 in \cite{Muirhead2009}, $\MBF{A}$ has the distribution same as $\MBF{Z}^{\MYTRAN}\MBF{Z},$ where $\MBF{Z}$ is a matrix of size $(n-1)\times p$ with  entries being i.i.d. $\mathcal{N}(0,1)$. It follows that with probability one, $\MBF{A}$ is not of full rank when $p\geq n$ and thus $|\MBF{A}|=0$. For a general matrix $\BSM{\Sigma}$, similar conclusion holds by the invariant properties of the three likelihood ratio tests above; please see Theorem 6.3.3 and 8.3.1 and Section 8.5.1 in \cite{Muirhead2009}. 
%We next discuss the obtained phase transition boundaries in Theorem \ref{thm:onesam}. Consider the classical chi-square approximations without correction. 
In Theorem \ref{thm:onesam}, $n>p+1$ is assumed to ensure the existence of the likelihood ratio tests. 
%The condition $n-p \to \infty$ is assumed to obtain the chi-square approximation for problem (I), and  can  be removed for problems (II) and (III). 
We next discuss the obtained phase transition boundaries of the classical chi-squared approximations without correction.  
When only testing mean parameters,  Theorem \ref{thm:onesam} suggests that the chi-squared approximation holds if and only if $p/n^{2/3}\to 0$. This asymptotic regime is similarly assumed 
%consistent with the sufficient conditions 
in \cite{portnoy1988asymptotic}, which considered testing $p$ natural parameters in exponential families.
% for the chi-square approximation to be applied
 However, \citet{portnoy1988asymptotic} only showed the sufficiency of  $p/n^{2/3}\to 0$  for the chi-squared  approximation to be applied,  and did not   establish the necessary and sufficient result, which is essential for  understanding the phase transition behaviors. 
 %Moreover, \citet{portnoy1988asymptotic} actually considered testing $p$ natural parameters in exponential families, which is not exactly the same as only testing mean parameters here.
  In addition, when the likelihood ratio tests involve covariance matrices as in (II) and (III), Theorem \ref{thm:onesam} shows that the chi-squared approximation holds if and only if $p/n^{1/2}\to 0$, which is consistent with the discussion in \cite{chen2009effects}. Particularly, under certain regularity conditions, \citet{chen2009effects} established that the chi-squared approximation of the empirical likelihood ratio test holds if $p/n^{1/2}\to 0$. The authors further argued that $p/n^{1/2}\to 0$ is likely to be the best rate for $p$, because it is the necessary and sufficient condition for the convergence of the sample covariance matrix to the true covariance matrix $\BSM{\Sigma}$ under the trace norm  when the eigenvalues of $\BSM{\Sigma}$ are bounded. The analysis provides an intuitive explanation for the phase transition boundaries obtained above, and  our  necessary and sufficient result would serve as another support for their conjecture, despite the different problem settings  in \citet{chen2009effects} and here.  
 %Despite the similarity, we point out that \citet{chen2009effects} studied the empirical likelihood ratio test,  which is different from  the problem settings here. 

Additionally, for the chi-squared approximations with the Bartlett correction, Theorem \ref{thm:onesam} also explicitly characterizes their phase transition boundaries, which generally achieve a larger asymptotic region than those without correction.
When $p$ is fixed, the Bartlett correction serves as a rescaling strategy that can improve the convergence rate of the likelihood ratio statistic from $O(n^{-1})$ to $O(n^{-2})$; however, 
when $p$  grows with sample size $n$, the classical result cannot apply directly. Alternatively, the results in Theorem \ref{thm:onesam} 
%consider both $p$ and $n$, and  
provide a   precise illustration of how the Bartlett correction improves the chi-squared approximations in terms of the phase transition boundaries.

The phase transition boundaries in Theorem \ref{thm:onesam} give the necessary and sufficient conditions on the asymptotic regimes of $(n,p)$ in Wilk's phenomenon.  
%for the chi-squared approximations to hold.
When applying the likelihood ratio test in practice, it is desired to have 
%a further characterization on 
a better understanding of 
the accuracy of the chi-squared approximation,  especially near its phase transition boundary. 
%When $p$ is fixed compared to $n$,  the accuracy of the chi-squared approximation is well-established in classical literature; see, e.g.,    \cite{anderson2003introduction}.
%To better understand 
%the accuracy of the chi-squared approximation as  $p$ increases with $n$, 
%we next consider  $p\to \infty$ as $n\to \infty$. 
%Particularly, 
The following Theorem \ref{thm:onesamchisq} characterizes the accuracy of each chi-squared approximation for tests (I)--(III) 
% for each test  
% under the asymptotic regime 
when Wilk's phenomenon holds. 
%the  chi-squared  approximation holds.  
Specifically, we consider the asymptotic regime where  $(n,p)$ satisfies the corresponding necessary and sufficient condition  in Theorem \ref{thm:onesam},
i.e.,
 $p/n^{d_1}\to 0$ and $p/n^{d_2}\to 0$ for the chi-squared approximations without and with the Bartlett correction, respectively.  
%The results are of interest 
%We focus on this particular local asymptotic regime 
%since  the asymptotic behavior of the chi-squared approximation changes at this neighborhood as suggested by Theorem \ref{thm:onesam}.  
%Then the derived accuracy measure  under this local asymptotic regime would lead to a useful  quantitative guideline on the use of the chi-squared approximations in practice.    

\begin{theorem}\label{thm:onesamchisq}
%Assume $n>p+1$ for all $n\geq 3$. 
%, and $p \to \infty$ as $n\to \infty$. 
For each likelihood ratio test (I)--(III), let $d_i$, $i=1,2$ take the corresponding values in Theorem \ref{thm:onesam}. Let $z_{\alpha}$ denote the upper $\alpha$-level quantile of the standard normal distribution. Consider $p\to \infty$ as $n\to \infty$. 
 Then under $H_0$, given $\alpha\in (0,1)$, 
\begin{itemize}
\item[(i)] when $p/n^{d_1}\to 0$, the chi-squared approximation satisfies 
\begin{align}
\MYPR\{-2 \log \Lambda_{n}>\chi^2_{f}(\alpha)\}-\alpha=\frac{\vartheta_1(n,p)}{\sqrt{\pi}} \exp\biggr(-\frac{z_{\alpha}^2}{2}\biggr) + o\biggr( \frac{p^{1/d_1}}{n}\biggr); \label{eq:chisqapprox}
\end{align}
%\begin{align}
%\MYPR\{-2 \log \Lambda_{n}>\chi^2_{f}(\alpha)\}-\alpha=\sqrt{2}\phi(z_{\alpha})\times \vartheta_1(n,p) + o\biggr( \frac{p^{1/d_1}}{n}\biggr); \label{eq:chisqapprox}
%\end{align}
\item[(ii)] when $p/n^{d_2}\to 0$, the chi-squared approximation with the Bartlett correction satisfies 
\begin{align}
\MYPR\{-2\rho \log \Lambda_{n}>\chi^2_{f}(\alpha)\}-\alpha	= \frac{\vartheta_2(n,p)}{\sqrt{\pi}}\exp\biggr(-\frac{z_{\alpha}^2}{2}\biggr) + o\biggr( \frac{p^{2/d_2}}{n^2}\biggr).   \label{eq:chisqapproxbartcorr} 
\end{align}
\end{itemize}
The values of $\vartheta_1(n,p)$ and $\vartheta_2(n,p)$ under three testing problems (I)--(III) are listed below. 
%\begin{center}
%\vspace{1.5em}
\begingroup
\allowdisplaybreaks
\small
\begin{eqnarray}
 &\hspace{-8.6em} \text{(I)}\hspace{1.15em} \displaystyle \vartheta_1(n,p)=\displaystyle \frac{p^2+2p}{4n\sqrt{f}},\hspace{0.75em}   & \vartheta_2(n,p)=\frac{p(p^2-4)}{24(\rho n)^2\sqrt{f}}; \notag  \\[5pt]
 &\hspace{-2.1em} \text{(II)} \hspace{0.95em} \displaystyle \vartheta_1(n,p)=\displaystyle \frac{p(2p^2+3p-1)-4/p}{24(n-1)\sqrt{f}},\hspace{1.2em}  &\vartheta_2(n,p)=\frac{(p-2)(p-1)(p+2)(2p^3+6p^2+3p+2)}{144p^2\rho^2(n-1)^2\sqrt{f}}; \notag \\[5pt]
% &\hspace{-0.6em} \text{(II)} \hspace{0.95em} \displaystyle \vartheta_1(n,p)=\displaystyle \frac{p(2p^2+3p-1)}{24n\sqrt{f}}-\frac{1}{6pn\sqrt{f}},\hspace{0.8em}  &\vartheta_2(n,p)=\frac{(p-2)(p-1)(p+2)(2p^3+6p^2+3p+2)}{144p^2(\rho n)^2\sqrt{f}}; \notag \\[5pt]
 &\hspace{-4.1em} \text{(III)} \hspace{0.85em} \displaystyle \vartheta_1(n,p)=\frac{p\left(2 p^{2}+9 p+11\right)}{24n\sqrt{f}}, \hspace{1.05em} &\vartheta_2(n,p)=\frac{p(2p^4+18p^3+49p^2+36p-13)}{144(p+3)(\rho n)^2\sqrt{f}}.   \notag 
\end{eqnarray}
\normalsize
\endgroup
\end{theorem}

In Theorem \ref{thm:onesamchisq}, the forms of $\vartheta_1(n,p)$ and $\vartheta_2(n,p)$   are derived from a nontrivial calculation of  certain complicated infinite series (see Eq. \eqref{eq:probexpaniii1} and \eqref{eq:probexpaniii1bart} in the Supplementary Material).  
 We can see that for each test,  $\vartheta_1(n,p)$ and $\vartheta_2(n,p)$  are of orders 
of $p^{1/d_1}n^{-1}$ and  $p^{2/d_2}n^{-2}$, respectively. 
%The leading term of the  biases of the type-\RNum{1} error of the chi-squared approximations without and with the Bartlett correction are given by the first terms in  \eqref{eq:chisqapprox} and \eqref{eq:chisqapproxbartcorr}, respectively, which therefore can be used to measure the accuracy of the chi-squared approximations. 
It follows that  $\vartheta_1(n,p)\exp(-z_{\alpha}^2/2)/\sqrt{\pi}$ in \eqref{eq:chisqapprox}
is the leading term for the  chi-squared  approximation bias $\MYPR\{-2 \log \Lambda_{n}>\chi^2_{f}(\alpha)\}-\alpha$, and  therefore can be used to measure the accuracy of the chi-squared approximation.
%It follows that $\vartheta_1(n,p)\exp(-z_{\alpha}^2/2)/\sqrt{\pi}$ in \eqref{eq:chisqapprox}
%%the first term in the right hand side of  \eqref{eq:chisqapprox}
%  provides the leading term of the  bias of the chi-squared approximation, 
%and  therefore can be used to measure the accuracy of the approximation.
Similar conclusion also holds for $\vartheta_2(n,p)\exp(-z_{\alpha}^2/2)/\sqrt{\pi}$ in \eqref{eq:chisqapproxbartcorr} when using the  chi-squared approximation with the Bartlett correction. 
%It follows that the first terms in  \eqref{eq:chisqapprox} and \eqref{eq:chisqapproxbartcorr} provide the leading terms of the  biases of the type-\RNum{1} error of the chi-squared approximations without and with the Bartlett correction, respectively, and therefore can be used to measure the accuracy of the chi-squared approximations. 
We demonstrate the usefulness of   \eqref{eq:chisqapprox} and \eqref{eq:chisqapproxbartcorr} in practice by simulation studies in $\S$\,\ref{sec:simulations}.

In the above discussion, we  focus   on the local asymptotic regime 
when Wilk's phenomenon holds, 
%where the chi-squared approximation holds, 
and the derived bias 
describes the accuracy of the chi-squared approximation. 
When $p$ further increases beyond this local asymptotic regime, the chi-squared approximation starts to fail, and the approximation bias becomes asymptotically unignorable. 
The following  Theorem \ref{thm:onesamnormal}  characterizes such unignorable  biases of the chi-squared approximations.
Particularly, we consider 
%a general 
the local asymptotic regime $p/n \to 0$, which includes the case 
when Wilk's theorem fails, 
% when the chi-squared approximation does not hold, 
 that is,  $p/n^{d_1}\not \to 0$ for the chi-squared approximation, and $p/n^{d_2}\not \to 0$ for the chi-squared approximation with the Bartlett correction.

\begin{theorem}\label{thm:onesamnormal}
%Assume $n>p+1$ for all $n\geq 3$, and 
Assume $p\to \infty$ and $p/n \to 0$ as $n \to \infty$. 
%$n-p\to \infty$ and $p\to \infty$ as $n\to \infty$. 
For each likelihood ratio test  (I)--(III), under $H_0$, there  exists a small constant $\delta\in (0,1)$  such that for any $\alpha \in (0,1)$,  

%\begin{enumerate}
%	\item[(i)] the chi-squared approximation satisfies 
%\begin{align}
%\MYPR\big\{-2 \log \Lambda_{n}>\chi^2_{f}(\alpha)\big\}-\alpha =  \bar{\Phi} \biggr(z_{\alpha}+\frac{f+2\mu_n}{2n\sigma_n} \biggr)- \alpha + O\left\{\left(\frac{p}{n}\right)^{\frac{1-\delta}{2}}+f^{-\frac{1-\delta}{6}} \right\};   \label{eq:normalbias1}
%\end{align}
%\begin{align}
%\hspace{-0.7em}\MYPR\big\{-2 \log \Lambda_{n}>\chi^2_{f}(\alpha)\big\}-\alpha =  \bar{\Phi} \biggr(z_{\alpha}+\frac{f+2\mu_n}{2n\sigma_n} \biggr)- \bar{\Phi}(z_{\alpha}) + O\left\{\left(\frac{p}{n}\right)^{\frac{1-\delta}{2}}+f^{-\frac{1-\delta}{6}} \right\};   \label{eq:normalbias1}
%\end{align}
\vspace{0.3em}

%(i)  the chi-squared approximation satisfies 
%\begin{align}
%\hspace{-1.5em}\MYPR\big\{-2 \log \Lambda_{n}>\chi^2_{f}(\alpha)\big\}-\alpha =  \Phi(z_{\alpha})-\Phi\biggr(z_{\alpha}+\frac{f+2\mu_n}{2n\sigma_n} \biggr) + O\left\{\left(\frac{p}{n}\right)^{\frac{1-\delta}{2}}+f^{-\frac{1-\delta}{6}} \right\},   \label{eq:normalbias1}
%\end{align}
%\begin{align}
%\hspace{-1.5em}\MYPR\big\{-2 \log \Lambda_{n}>\chi^2_{f}(\alpha)\big\}-\alpha =  \Phi(z_{\alpha})-\Phi\biggr\{\frac{\chi^2_f(\alpha)+2\mu_n}{2n\sigma_n} \biggr\} + O\left\{\left(\frac{p}{n}\right)^{\frac{1-\delta}{2}}+f^{-\frac{1-\delta}{6}} \right\},   \label{eq:normalbias1}
%\end{align}
%\begin{align}
%\hspace{-1.5em}\MYPR\big\{-2 \log \Lambda_{n}>\chi^2_{f}(\alpha)\big\}-\alpha =  1-\Phi\biggr\{\frac{\chi^2_f(\alpha)+2\mu_n}{2n\sigma_n} \biggr\} - \alpha + O\left\{\left(\frac{p}{n}\right)^{\frac{1-\delta}{2}}+f^{-\frac{1-\delta}{6}} \right\},   \label{eq:normalbias1}
%\end{align}
\begin{enumerate}
	\item[(i)] the chi-squared approximation satisfies 
\begin{align}
\hspace{-1.5em}\MYPR\big\{-2 \log \Lambda_{n}>\chi^2_{f}(\alpha)\big\}-\alpha =  \bar{\Phi}\biggr\{\frac{\chi^2_f(\alpha)+2\mu_n}{2n\sigma_n} \biggr\} - \alpha + O\left\{\left(\frac{p}{n}\right)^{\frac{1-\delta}{2}}+f^{-\frac{1-\delta}{6}} \right\},   \label{eq:normalbias1}
\end{align}
%\hspace{0.3em} 
 where $\bar{\Phi}(\cdot)=1-\Phi(\cdot)$, and $\Phi(\cdot)$ denotes the cumulative distribution function of the standard normal distribution; 

\vspace{0.4em}
%where $\Phi(\cdot)$ denotes the cumulative distribution function of the standard normal distribution; 
%where  $\bar{\Phi}(z)=\MYPR(Z_0>z)$ and $Z_0$ denotes the standard normal distribution.   
%where $\bar{\Phi}(\cdot)$ denote the upper tail probability of the standard normal distribution; 
%where $\delta>0$ is a small positive constant; 
%\item[(ii)] the chi-squared approximation with the Bartlett correction satisfies
%\begin{align}
%\MYPR\big\{-2\rho \log \Lambda_{n}>\chi^2_{f}(\alpha)\big\}-\alpha =    \bar{\Phi} \biggr(z_{\alpha}+\frac{f+2\rho \mu_n}{2\rho n\sigma_n}\biggr)- \alpha  + O\left\{\left(\frac{p}{n}\right)^{\frac{1-\delta}{2}}+f^{-\frac{1-\delta}{6}} \right\}.  \label{eq:normalbias2}
%\end{align} 
%\begin{align}
%\MYPR\big\{-2\rho \log \Lambda_{n}>\chi^2_{f}(\alpha)\big\}-\alpha =    \bar{\Phi} \biggr(z_{\alpha}+\frac{f+2\rho \mu_n}{2\rho n\sigma_n}\biggr)- \bar{\Phi}(z_{\alpha})  + O\left\{\left(\frac{p}{n}\right)^{\frac{1-\delta}{2}}+f^{-\frac{1-\delta}{6}} \right\}.  \label{eq:normalbias2}
%\end{align} 
	\item[(ii)] the chi-squared approximation with the Bartlett correction satisfies
%(ii) the chi-squared approximation with the Bartlett correction satisfies
\begin{align}
\hspace{-0.1em}\MYPR\big\{-2\rho \log \Lambda_{n}>\chi^2_{f}(\alpha)\big\}-\alpha =   \bar{\Phi}\Biggr\{\frac{\chi^2_f(\alpha)+2\rho \mu_n}{2\rho n\sigma_n} \Biggr\} - \alpha + O\left\{\left(\frac{p}{n}\right)^{\frac{1-\delta}{2}}+f^{-\frac{1-\delta}{6}} \right\}.  \label{eq:normalbias2}
\end{align}
\end{enumerate}
%\begin{align}
%\hspace{-0.1em}\MYPR\big\{-2\rho \log \Lambda_{n}>\chi^2_{f}(\alpha)\big\}-\alpha =     \Phi(z_{\alpha}) - \Phi\Biggr\{\frac{\chi^2_f(\alpha)+2\rho \mu_n}{2\rho n\sigma_n} \Biggr\}  + O\left\{\left(\frac{p}{n}\right)^{\frac{1-\delta}{2}}+f^{-\frac{1-\delta}{6}} \right\}.  \label{eq:normalbias2}
%\end{align} 
%\begin{align}
%\hspace{-0.1em}\MYPR\big\{-2\rho \log \Lambda_{n}>\chi^2_{f}(\alpha)\big\}-\alpha =   1 - \Phi\Biggr\{\frac{\chi^2_f(\alpha)+2\rho \mu_n}{2\rho n\sigma_n} \Biggr\} - \alpha + O\left\{\left(\frac{p}{n}\right)^{\frac{1-\delta}{2}}+f^{-\frac{1-\delta}{6}} \right\}.  \label{eq:normalbias2}
%\end{align}
%\begin{align}
%\hspace{-1.3em}\MYPR\big\{-2\rho \log \Lambda_{n}>\chi^2_{f}(\alpha)\big\}-\alpha =     \Phi(z_{\alpha}) - \Phi \biggr(z_{\alpha}+\frac{f+2\rho \mu_n}{2\rho n\sigma_n}\biggr) + O\left\{\left(\frac{p}{n}\right)^{\frac{1-\delta}{2}}+f^{-\frac{1-\delta}{6}} \right\}.  \label{eq:normalbias2}
%\end{align} 
%\end{enumerate}
The values of $\mu_n$ and $\sigma_n$ under each   problem are listed below, where $L_{x,p}=\log(1-p/x)$ for $x>p$. 
\begingroup
\allowdisplaybreaks
\small
\begin{eqnarray}
 &\hspace{0.2em} \text{(I)}\hspace{1em} \displaystyle \mu_n=\frac{n}{2}\Big\{\Big(n-p-\frac{3}{2}\Big)(L_{n,p}-L_{n-1,p}) +L_{n,p}+pL_{n,1}\Big\},\hspace{0.8em}   &  \sigma_n^2=\frac{1}{2}(L_{n,p}-L_{n-1,p}); \notag  \\[5pt]
 & \hspace{-6.2em}  \text{(II)} \hspace{1em} \displaystyle 	\mu_n =  -\frac{n-1}{2}\big\{(n-p-3/2)L_{n-1,p}+p\big\},\hspace{0.8em}  & \sigma_n^2=-\frac{1}{2}\left(\frac{p}{n-1} + L_{n-1,p} \right)\frac{(n-1)^2}{n^2}; \notag \\[5pt]
 &\hspace{-5.6em} \text{(III)} \hspace{0.85em} \displaystyle   \mu_{n}=-\frac{n}{2}\big\{( n- p-3/2)L_{n-1,p}+p\big\}-\frac{p}{2}, \hspace{1.05em} &\sigma_{n}^{2}=-\frac{1}{2}\left(\frac{p}{n-1} + L_{n-1,p} \right).   \notag 
\end{eqnarray}
\normalsize
\endgroup
\end{theorem}

 Theorem \ref{thm:onesamnormal} is derived by quantifying the difference between the  characteristic functions of $\log \Lambda_{n}$ and a normal distribution (see Lemma \ref{lm:chardiffgoal} in the Supplementary Material).
The local asymptotic regime $p/n\to 0$ is assumed mainly for the technical simplicity of evaluating the asymptotic expansions of the characteristic functions. 
Under the conditions of Theorem \ref{thm:onesamnormal}, 
$ \bar{\Phi}[ \{ \chi^2_f(\alpha)+2\mu_n \}/(2n\sigma_n) ] - \alpha $ in \eqref{eq:normalbias1} can be approximated by  
$\bar{\Phi}\{ z_{\alpha} + (f+2\mu_n)/(2n\sigma_n)\} - \bar{\Phi}(z_{\alpha})$,
where 
%$z_{\alpha}$ denotes the upper $\alpha$-level quantile of the standard normal distribution, and 
$(f+2\mu_n)/(2n\sigma_n)$ is of the order of $pn^{-d_1}$ (see Remark \ref{rm:order} in the Supplementary Material). 
Consequently, when 
the chi-squared approximation fails, i.e., 
$pn^{-d_1}\not \to 0$, 
we know that 
$ \bar{\Phi}[ \{ \chi^2_f(\alpha)+2\mu_n \}/(2n\sigma_n) ] - \alpha $  in 
\eqref{eq:normalbias1} 
characterizes the corresponding unignorable bias of the chi-squared approximation. 
Similarly,  
%for the chi-squared approximation with the Bartlett correction,
we can show that  
$ \bar{\Phi}[ \{ \chi^2_f(\alpha)+2\rho \mu_n \}/(2\rho n\sigma_n) ] - \alpha $
can be approximated by 
$\bar{\Phi}\{ z_{\alpha} + (f+2\rho \mu_n)/(2\rho n\sigma_n)\} - \bar{\Phi}(z_{\alpha})$,
where $(f+2\rho \mu_n)/(2\rho n\sigma_n)$ is of the order of $p^{2/d_2}n^{-2}$. 
Therefore, when the chi-squared approximation with the Bartlett correction fails, i.e., $pn^{-d_2}\not \to 0$,
we know that \eqref{eq:normalbias2} characterizes the corresponding unignorable approximation bias.

\begin{remark}\label{rm:orderelation}
Although the above discussions consider  $p/n^{d_1}\not \to 0$ and $p/n^{d_2}\not \to 0$,  \eqref{eq:normalbias1} and \eqref{eq:normalbias2} in Theorem \ref{thm:onesamnormal} also hold under  the asymptotic regimes $p/n^{d_1}\to 0$ and $p/n^{d_2}\to 0$ examined in Theorem \ref{thm:onesamchisq}.  
However,  since Theorems \ref{thm:onesamchisq} and  \ref{thm:onesamnormal} focus on different asymptotic regimes 
and are proved using different techniques, 
we can show that when $p/n^{d_1}\to 0$ and $p/n^{d_2}\to 0$, \eqref{eq:normalbias1} and \eqref{eq:normalbias2} have an additional remainder term 
$O\{({p}/{n})^{(1-\delta)/2}+f^{-(1-\delta)/6} \}$ compared to \eqref{eq:chisqapprox} and \eqref{eq:chisqapproxbartcorr}, respectively;  see  Remark \ref{rm:order} in the Supplementary Material. 
Therefore, 
%when $p/n^{d_1}\to 0$ and $p/n^{d_2}\to 0$, 
under the asymptotic regimes of Theorem \ref{thm:onesamchisq},  
 \eqref{eq:chisqapprox} and \eqref{eq:chisqapproxbartcorr}
 provide a sharper characterization of the accuracy of  the chi-squared approximations than \eqref{eq:normalbias1} and \eqref{eq:normalbias2}, respectively. 
 \end{remark}

\section{Simulations} \label{sec:simulations}
% We conduct simulation studies to show that the theoretical phase transition boundaries match the empirical ones and thus provide useful quantitative guidelines in scientific practices. 

% Particularly, Figure \ref{fig:simone} presents the numerical results  of three one-sample tests (I)--(III). To examine the effect of increasing data dimension, we let $p = \lfloor n^\epsilon \rfloor$, where $n\in\{100,500,1000,5000\}$ and $\epsilon \in\{6/24,\ldots, 23/24 \}$, and set $\BSM{\mu} = (0,\ldots,0)^{\MYTRAN}$ and $\BSM{\Sigma} = \BIEN$ under the null hypothesis. With the significance level $0.05$, we plot the estimated type I error versus $\epsilon$ in Figure \ref{fig:simone}, where the three columns correspond to three problems (I)--(III), and the the two rows correspond to using the chi-square approximations without and with the Bartlett correction, respectively.  

% We observe that each approximation starts to fail around the value of $\epsilon$ that matches their theoretical phase transition boundary. For instance, plots (a) and (d) in Figure \ref{fig:simone} correspond to the chi-square approximations without and with the Bartlett correction of problem (I), and their estimated type I errors start to inflate around $\epsilon=2/3$ and $4/5$, respectively. This is consistent with two theoretical phase transition boundaries of problem (I) in Theorem \ref{thm:onesam}. Moreover, in (b) and (e) of Figure \ref{fig:simone}, the estimated type I errors start to inflate around $\epsilon=1/2$ and $2/3$, respectively. 

We conduct simulation studies to evaluate the finite-sample performance of the theoretical results.  
Particularly, under the null hypothesis of the one-sample tests, we generate data with $\BSM{\mu} = (0,\ldots,0)^{\MYTRAN}$ and $\BSM{\Sigma} = \BIEN$ and use $\alpha=0.05$.  
% in Theorems \ref{thm:onesam}--\ref{thm:onesamnormal}.  
%of the theoretical phase transition results and asymptotic biases. 
%In this section, 
We next consider problem (III), jointly testing mean and covariance, as an illustration example, and present the results of the chi-squared approximation without the Bartlett correction. 
%We next present the numerical results of the chi-squared approximation without the Bartlett correction for 
% problem (III), testing mean and covariance jointly, as an illustration example.  
For test (III) with the Bartlett correction and problems (I)--(II),
testing mean and covariance separately, the simulation results are similar and thus presented in $\S$\,\ref{sec:suppsim} of the Supplementary Material. 

%We next present the numerical results of the chi-squared approximation without the Bartlett correction for 
%test (III) as an illustration example, 
%while the results for test (III) with the Bartlett correction and tests (I)--(II) 
% are presented in $\S$\,\ref{sec:suppsim} of the Supplementary Material. 

First, to examine the phase transition boundary in Theorem \ref{thm:onesam},  
 we take  $p = \lfloor n^\epsilon \rfloor$, where $n\in\{100,500,1000,5000\}$, $\epsilon \in\{6/24,\ldots, 23/24 \}$, and $\lfloor \cdot \rfloor$ denotes the floor function. 
We plot the empirical type-\RNum{1} error 
% for the chi-squared approximation without the Bartlett correction 
 versus $\epsilon$ in Part (a) of Fig. \ref{fig:simoneexam}, which is based on 1,000 simulation replications. 
% ,  to evaluate the effect of increasing data dimension $p$. 
%Under the null hypothesis, we set $\BSM{\mu} = (0,\ldots,0)^{\MYTRAN}$ and $\BSM{\Sigma} = \BIEN$. 
%In the part (a) of Fig. \ref{fig:simoneexam}, 
We can see that for all considered sample sizes, the empirical type-\RNum{1} errors start to inflate around $\epsilon=1/2$, matching the  phase transition boundary $d_1=1/2$ of test (III) in Theorem \ref{thm:onesam}. Similar results are obtained for  other tests as shown in the Supplementary Material.

\begin{figure}[!htbp]
\captionsetup[subfigure]{labelformat=empty}
\centering
\begin{subfigure}[t]{0.32\textwidth}
\centering
\includegraphics[width=\textwidth]{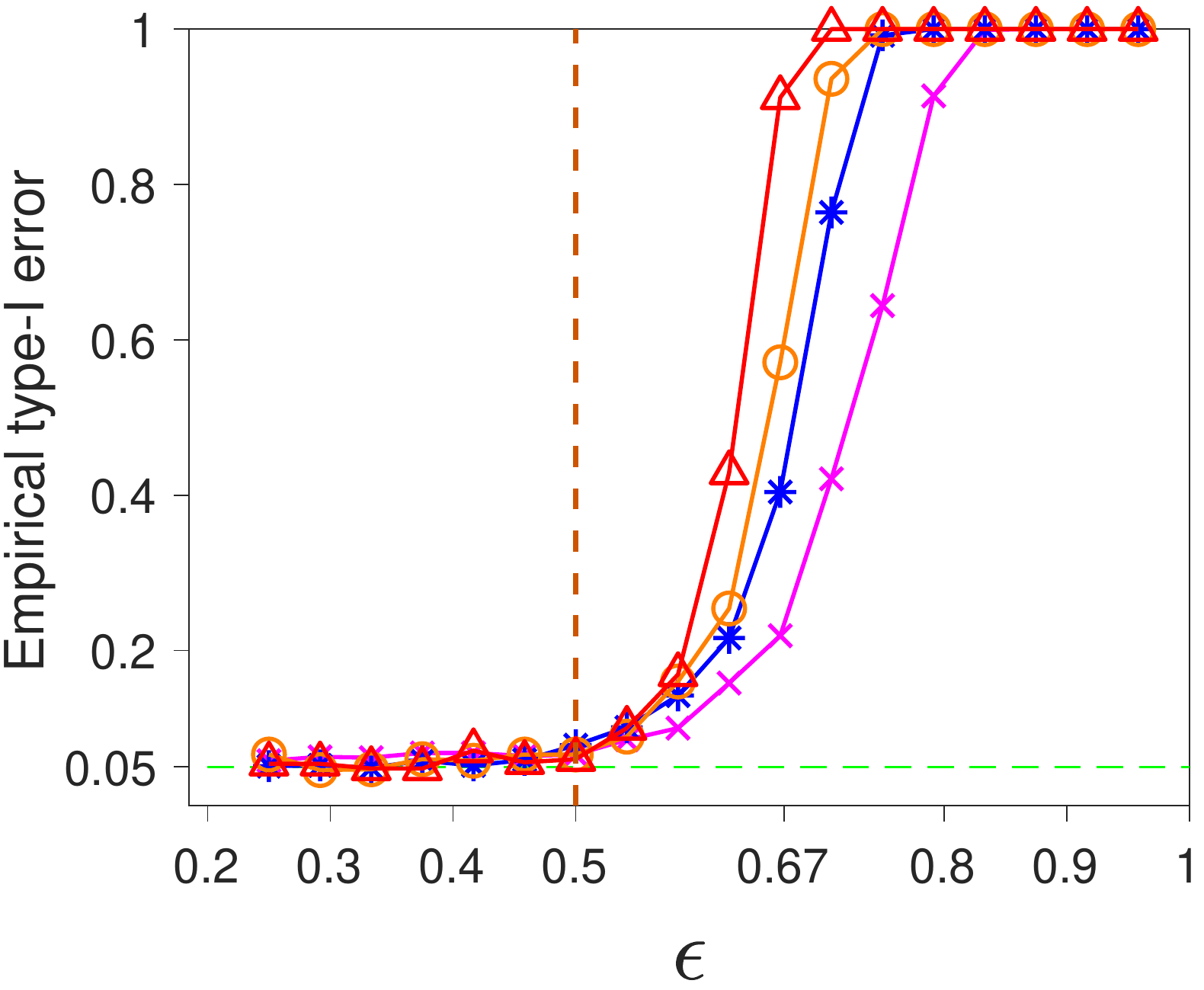}
\centerline{(a)}	
\end{subfigure}%
\hfill
\begin{subfigure}[t]{0.32\textwidth}
\centering
\includegraphics[width=\textwidth]{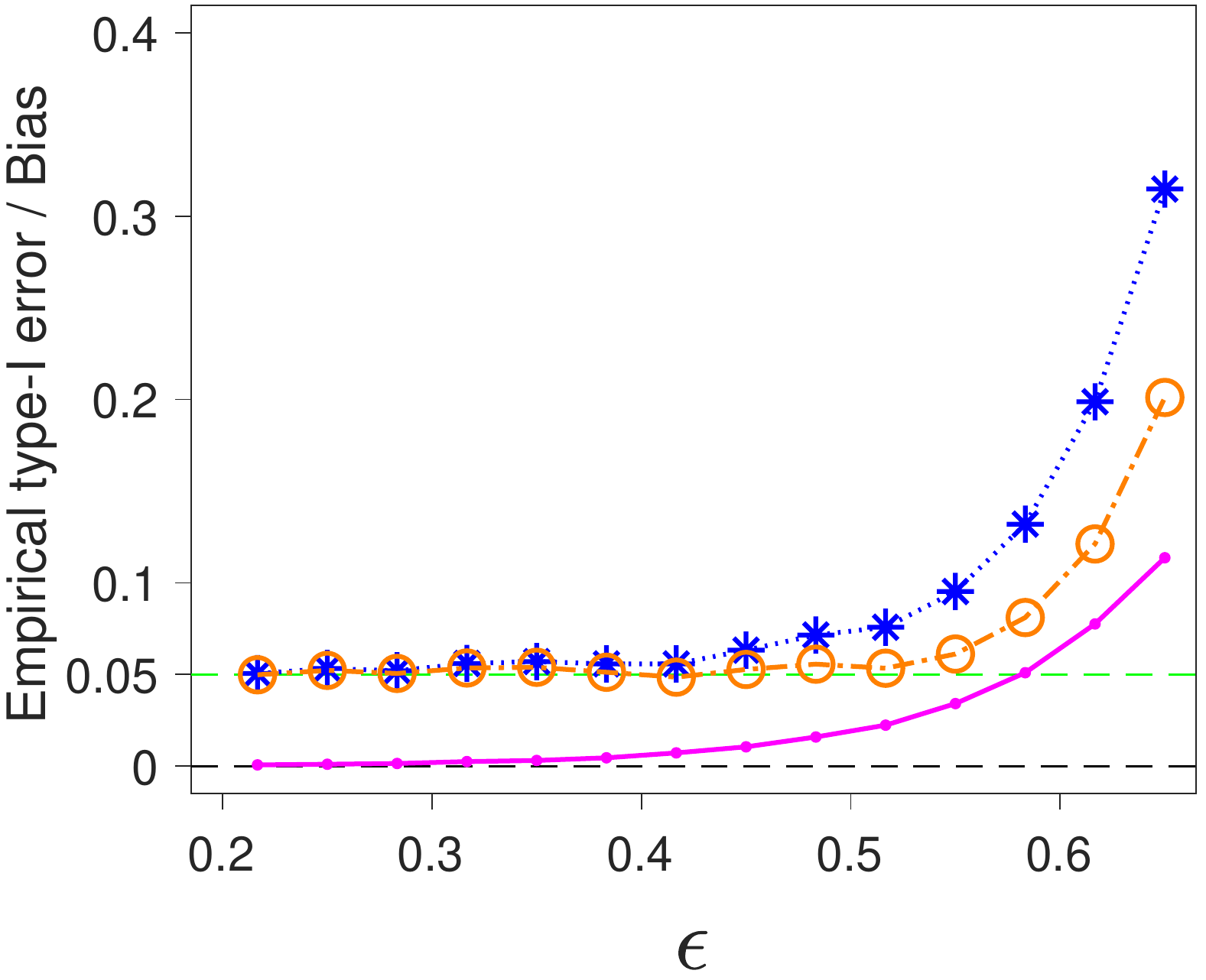}
%{(I)Meanalone_Bias1.pdf}
\centerline{(b)}	
\end{subfigure}%
\hfill
\begin{subfigure}[t]{0.32\textwidth}
\centering
\includegraphics[width=\textwidth]{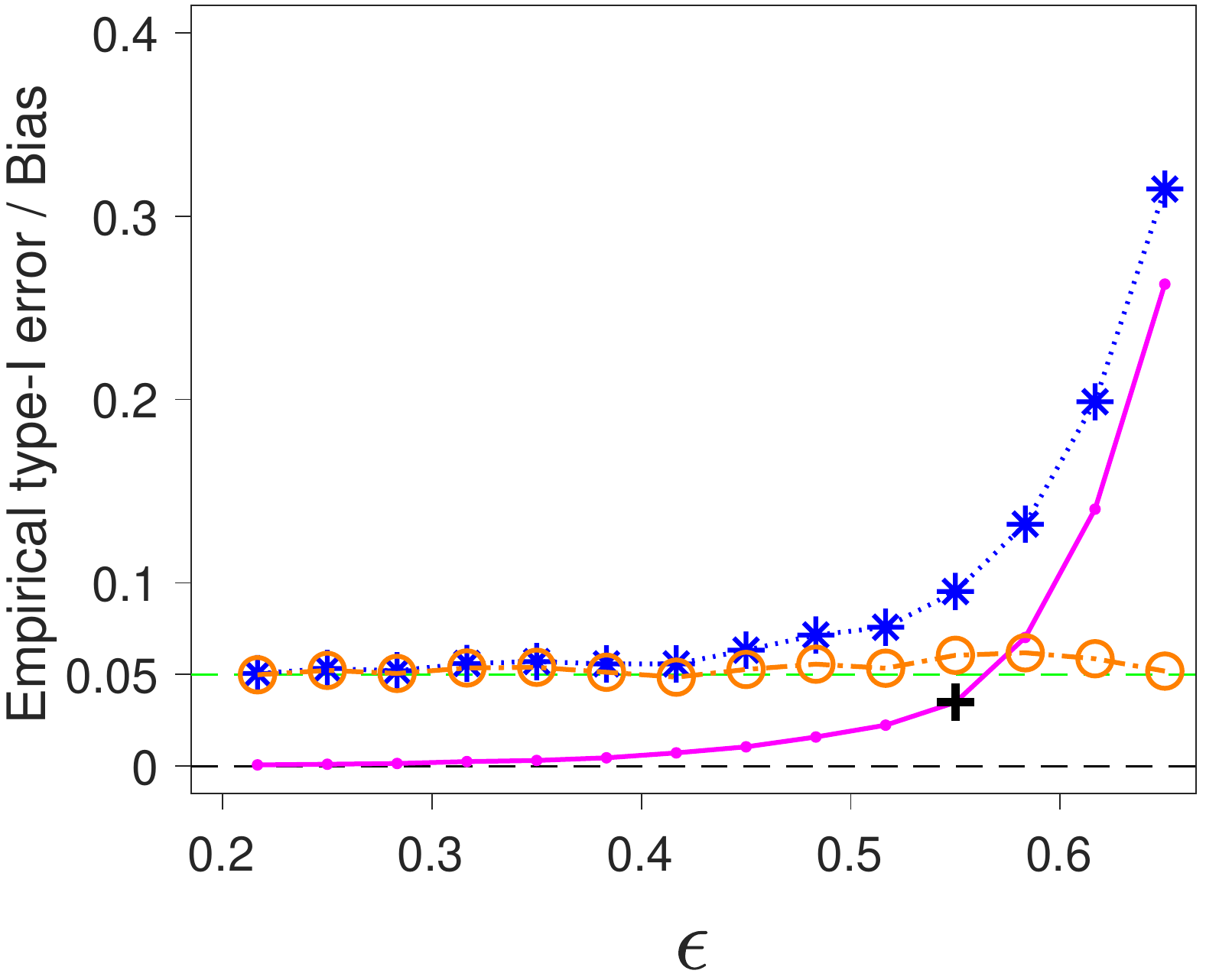}
%{(I)Meanalone_Bias1_Max.pdf}
\centerline{(c)}		
\end{subfigure}
\caption{Chi-squared approximation without the Bartlett correction for test (III):\ \ (a) \ Empirical type-\RNum{1} error  for $n=100$ (cross), $500$ (asterisk), $1000$ (square), and $5000$ (triangle);  the theoretical phase transition boundary $\epsilon=1/2$ (vertical dashed line). \ \ (b) \  Empirical type-\RNum{1} error  for $n=500$ (asterisk); asymptotic bias $\vartheta_1(n,p)\exp(-z_{\alpha}^2/2)/\sqrt{\pi}$ in \eqref{eq:chisqapprox}  (dot); the difference between the empirical type-\RNum{1} error  and  the asymptotic bias in  \eqref{eq:chisqapprox}  (circle). \ \ (c) \  Empirical type-\RNum{1} error   for $n=500$ (asterisk); the maximum bias over the bias in \eqref{eq:chisqapprox} and the bias $ \bar{\Phi}[ \{ \chi^2_f(\alpha)+2\mu_n \}/(2n\sigma_n) ] - \alpha $ in \eqref{eq:normalbias1} (dot);  the location where the bias in \eqref{eq:normalbias1} starts to dominate  the bias in \eqref{eq:chisqapprox} (plus sign); the difference between the empirical type-\RNum{1} error and the maximum bias (circle).
%$\vartheta_1(n,p)\exp(-z_{\alpha}^2/2)/\sqrt{\pi}\leq \bar{\Phi}[ \{ \chi^2_f(\alpha)+2\mu_n \}/(2n\sigma_n) ] - \alpha$.
}\label{fig:simoneexam} 

\end{figure}

Second, we numerically evaluate the asymptotic biases in  Theorems \ref{thm:onesamchisq} and \ref{thm:onesamnormal}  with   
%to examine the approximation bias, we take 
$p=\lfloor n^{\epsilon} \rfloor$, where  $n\in \{100, 500\}$ and $\epsilon\in (0,1)$. 
%Parts (b) and (c) in 
Parts (b) and (c) in Fig. \ref{fig:simoneexam} present the results with $n=500$,
 while the results with $n=100$ are similar and thus reported in the Supplementary Material. 
%Specifically, 
Part (b) shows that the asymptotic bias in \eqref{eq:chisqapprox} 
can be  an informative indicator of  the failure of Wilk's theorem. 
%the chi-squared approximation. 
%the asymptotic bias in \eqref{eq:chisqapprox} can indicate the failure of Wilk's theorem. 
%whether the chi-squared approximation correctly controls the type-\RNum{1} error or not.   
% indicates whether the chi-squared approximation can correctly control the type-\RNum{1} error or not.
%when $\epsilon$ is small, i.e., $p$ is small compared to $n$, the  theoretical asymptotic bias is small, and the chi-squared approximation controls the type-\RNum{1} error well. 
%at the pre-specified significance level 0.05. 
% when $\epsilon$ increases,  
% and becomes close to the phase transition boundary, 
Particularly,  as $\epsilon$ increases, 
%the empirical type-\RNum{1} error inflates, 
%begins to inflate,
  the asymptotic bias in \eqref{eq:chisqapprox}  increases accordingly. 
%%  of the chi-squared approximation 
%  begins to inflate, and  
%  the asymptotic bias increases accordingly. 
At the $\epsilon$ values where the empirical type-\RNum{1} error begins to inflate (e.g. $\epsilon\in[0.4,0.5]$),  
the difference between the empirical type-\RNum{1} error and the asymptotic bias
is still close to 0.05 as shown in the circle line,
suggesting that  \eqref{eq:chisqapprox} can approximate the bias well. 
%Theorem \ref{thm:onesamchisq} holds. 
%the increasing trends of the empirical type-\RNum{1} error and the asymptotic bias are consistent,
%  and their difference is close to 0.05. 
%  This is consistent with  Theorem \ref{thm:onesamchisq}. 
%  the two lines ``Chi-squared approximation'' and ``Approximate bias'' are consistent, 
%and the difference between the two lines, i.e., the ``Difference'' is close to 0.05.   
% This suggests that the asymptotic bias in \eqref{eq:chisqapprox} can be  an informative indicator on the failure of the chi-squared approximation, 
% and is consistent with  Theorem \ref{thm:onesamchisq}.   
 When $\epsilon$ further increases beyond the phase transition boundary (e.g. $\epsilon>0.5$), 
 the asymptotic bias keeps increasing, and its large value indicates the failure of the chi-squared approximation, even though it now underestimates  the approximation bias in this regime.  
 To better characterize the approximation bias when $\epsilon$ is beyond the phase transition boundary, we can combine the results in Theorem \ref{thm:onesamnormal} together with those in  Theorem \ref{thm:onesamchisq}.
Specifically, Part (c) shows that taking the maximum over the two asymptotic biases in \eqref{eq:chisqapprox} and \eqref{eq:normalbias1}   gives a good evaluation    of the   approximation bias for a full range of  $\epsilon$, below or above the phase transition boundary. 
%when $\epsilon$ is large,
%improve the characterization of the bias when $\epsilon$ is large,
%which is consistent with Theorem \ref{thm:onesamnormal}. 
% as suggested by Theorem \ref{thm:onesamnormal}. 
We also find that using \eqref{eq:normalbias1} itself does not evaluate the approximation bias well for small $\epsilon$ (results are not presented). 
Based on our theoretical and numerical results, when applying Wilk's theorem, we would recommend practitioners  to compare the asymptotic bias, either \eqref{eq:chisqapprox} or the maximum over \eqref{eq:chisqapprox} and \eqref{eq:normalbias1},  with a small threshold value that they may specify beforehand, e.g.,  0.01-0.02.
If the asymptotic bias is larger than the threshold, the chi-squared approximation should not be directly used, and other methods would be needed.

\section{Results of Other Tests} \label{sec:exten}
In addition to  three one-sample tests in $\S$\,\ref{sec:main}, we  also obtain similar theoretical and numerical  results for   other four  popular testing problems in the Supplementary Material.
Particularly, we consider three multiple sample tests: 
(IV) Testing the equality of several mean vectors;  (V) Testing the equality of several covariance matrices; (VI) Jointly testing the equality of several  mean vectors and covariance matrices.
We also study  (VII) Testing independence between multiple vectors.
%Testing the independence of subvectors of multivariate normal distribution.
Similarly to the results in $\S$\,\ref{sec:main}, for each likelihood ratio test, we establish not only the phase transition boundary of Wilk's theorem,  but also the approximation biases under the two asymptotic regimes, where Wilk's theorem holds or not, respectively. 
%the chi-squared approximation holds or not. 
%Note that tests (IV)--(VI) correspond to tests (I)--(III) under multiple samples, and the results of (IV)--(VI) are parallel to those of (I)--(III), respectively. 
Please see the details in $\S$\,\ref{sec:extensions} of the Supplementary Material.

\section{Discussion} \label{sec:suppdisc}

%Likelihood ratio tests are commonly applied in many real-world  problems and readily available in various softwares. 
This study derives the phase transition boundary and  characterizes the approximation  bias of Wilk's theorem in seven standard 
likelihood ratio tests. 
%testing problems. 
It is interesting to see that the phase transition boundary 
%of Wilk's phenomenon 
generally depends on the problem setting and whether the Bartlett correction is used or not, which emphasizes the  necessity of statistically-principled guidelines.  
%We mention that \cite{xu2019pearson} also demonstrated a similar phase transition phenomenons in Pearson's chi-squared test.
%Similarly, the chi-squared approximation biases for the likelihood ratio tests were  also studied in  \cite{anastasiou2018bounds}, 
%For the  chi-squared approximation biases of the likelihood ratio tests, 
%We note that similar problems were also studied under other settings. 
%We note 
The approximation bias of Wilk's theorem 
%performance of the chi-squared approximation 
was also recently studied by \cite{anastasiou2018bounds}, which derived an explicit bound of the chi-squared approximation bias for a general family of regular likelihood ratio test statistics. 
However, 
as noted in that paper, 
%their paper mentioned  that
%as noted by the authors, their 
their bounds are generally not optimized. It is thus of interest to further study the 
necessary and sufficient conditions for  
%accuracy of 
Wilk's phenomenon  and the approximation accuracy in such a general setting. 
%and proposed a modified statistic to address dichotomy. 
% \cite{xu2019pearson} demonstrated that a similar phase transition phenomenon also exists in Pearson's chi-squared test. 
%Moreover, 
%\cite{anastasiou2018bounds} studied the chi-squared approximation bias for a family of  regular likelihood ratio test statistics,
%where they derived the bounds for the bias,  
%% without a sharp characterization. 
% while we focus on the sharp characterization of bias in this paper. 
%and did not provide sharp characterization of bias. 
%thus differed from the sharp characterization of bias in  this paper.  
%but derived bounds  
%Similarly, \cite{xu2019pearson} demonstrated that a phase transition phenomenon also exists in Pearson's chi-squared test.
%,  and proposed a modified statistic.    
% the derived approximation bias provides sharp characterization 
% Moreover, the chi-squared approximation biases for the likelihood ratio tests were  also studied in  \cite{anastasiou2018bounds}, where their results focused on the bounds of the bias and thus were different from 
Beyond the regular parametric inference problems,  Wilk's-type phenomenon has also been studied in 
geometrically irregular parametric models 
% non-regular parametric inference questions
\citep{drton2011quantifying,chen2018conditional}, and extended to nonparametric models and statistical learning theory \citep[e.g.,][]{fan2000geometric,fan2001generalized,fan2004generalised,boucheron2011high}.  
 Understanding the phase transition behavior of  Wilk's phenomenon for the  likelihood ratio tests would shed light on   
studying the general Wilk's phenomenon under these complicated statistical models. 
Besides the likelihood ratio tests, 
similar phase transition  phenomena can also occur for other 
 popular test statistics. 
%another related research direction is the investigation of similar phase transition  behaviors in other 
%popular test statistics. 
For instance,  \cite{xu2019pearson} recently studied the  approximation theory for Pearson's chi-squared statistics when the number of cells is large, and  demonstrated  
a similar phase transition phenomenon that  
%the classical chi-squared approximation can fail 
%, depending on the conditions.
 the asymptotic distribution of the test statistic can be either a chi-squared or a normal distribution.
 %, depending on the conditions. 
It is interesting to further investigate the phase transition boundaries of these tests.

%Beyond the regular parametric inference problems,  Wilk's-type phenomenon has also been studied in 
%geometrically irregular parametric models 
%% non-regular parametric inference questions
%\citep{drton2011quantifying,chen2018conditional}, and extended to nonparametric models and statistical learning theory \citep[e.g.,][]{fan2000geometric,fan2001generalized,fan2004generalised,boucheron2011high}.  Understanding the phase transition behavior of  Wilk's phenomenon for the standards likelihood ratio tests would   shed light on   
%studying the general Wilk's phenomenon under these complicated statistical models. 
%Similar problem on the distance between the likelihood ratio test statistics and the limiting chi-squared distribution was also studied in \cite{anastasiou2018bounds}
% Beyond the likelihood ratio tests considered in this paper,
% Wilk's-type phenomenon has also been studied in 
% other parametric models \citep{drton2011quantifying,chen2018conditional},  
% and extended to nonparametric models and statistical learning theory \citep[e.g.,][]{fan2000geometric,fan2001generalized,fan2004generalised,boucheron2011high}.  

\section*{Acknowledgement}
The authors are grateful to the editor, Professor Paul Fearnhead, an associate editor and three referees for their valuable comments and suggestions.  
This research was partially supported by the U.S. National Science Foundation. 

%Acknowledgements should appear after the body of the paper but before any appendices and be as brief as possible
%subject to politeness. Information, such as contract numbers, of no interest to readers, must
%be excluded.

\section*{Supplementary material}
\label{SM}
The supplementary material available at \textit{Biometrika}  online includes   theoretical results for the other four  testing problems in  Section \ref{sec:exten}, additional simulation studies, and the proofs of the theorems.
\bibliographystyle{chicago}
\bibliography{testrefer}

%\title{Supplementary Material for\\[2pt] ``On the Phase Transition of Wilk's Phenomenon''}

\vspace{6em}

\numberwithin{equation}{section}
\begin{center}
\textbf{\large{Supplementary Material for\\[2pt] ``On the Phase Transition of Wilk's Phenomenon''}}	
\end{center}
       
\bigskip
\medskip

\appendix

In this supplementary material, we present additional results  in $\S$\,\ref{sec:extensions}. 
Particularly, the theoretical results for tests (IV)--(VI) and test  (VII) are given in $\S$\,\ref{sec:extmulti} and $\S$\,\ref{sec:extindp}, respectively.   
All the simulations for tests (I)--(VII) are provided in $\S$\,\ref{sec:suppsim}. 
We next present the proofs for  the testing problem (III) as an illustration example in $\S$\,\ref{sec:pfthemexam}, where the corresponding results in Theorems \ref{thm:onesam}--\ref{thm:onesamnormal} are proved in $\S\S$\,\ref{sec:pfonesam3}--\ref{sec:pfonsamnomral3}, respectively. 
The proofs for other tests are similar and given in $\S$\,\ref{sec:pfthms}. 
The technical  lemmas are proved in $\S$\,\ref{sec:assitedlemmas}. 

% in Section 

\bigskip

\contentsline {chapter}{\numberline {\ref{sec:extensions}}Additional Results}{\pageref{sec:extensions}}
\contentsline {section}{\numberline {\ref{sec:extmulti}}Multiple-Sample Tests}{\pageref{sec:extmulti}}
\contentsline {section}{\numberline {\ref{sec:extindp}}Testing Independence between Multiple Vectors}{\pageref{sec:extindp}}
\contentsline {section}{\numberline {\ref{sec:suppsim}}Additional Simulations}{\pageref{sec:suppsim}}
%\contentsline {section}{\numberline {\ref{sec:suppdisc}}Discussion}{\pageref{sec:suppdisc}}

%\contentsline {subsection}{\numberline {1.2.1}Objective 1}{5}

\contentsline {chapter}{\numberline {\ref{sec:pfthemexam}}Proof Illustration with Problem (III)}{\pageref{sec:pfthemexam}}
\contentsline {section}{\numberline {\ref{sec:pfonesam3}}Proof of Theorem \ref{thm:onesam} (III)}{\pageref{sec:pfonesam3}}
\contentsline {section}{\numberline {\ref{sec:pfthmchisqiii}}Proof of Theorems \ref{thm:onesamchisq} (III)}{\pageref{sec:pfthmchisqiii}}
\contentsline {section}{\numberline {\ref{sec:pfonsamnomral3}}Proof of Theorems \ref{thm:onesamnormal} (III)}{\pageref{sec:pfonsamnomral3}}

\contentsline {chapter}{\numberline {\ref{sec:pfthms}}Proofs of Other Problems}{\pageref{sec:pfthms}}
\contentsline {section}{\numberline {\ref{sec:pfthm1and4}}Proofs of Theorems \ref{thm:onesam},  \ref{thm:allmult} \& \ref{Thm2.1}}{\pageref{sec:pfthm1and4}}
\contentsline {section}{\numberline {\ref{sec:proofprop}}Proofs of Propositions \ref{prop:1} \& \ref{prop:2}}{\pageref{sec:proofprop}}
\contentsline {section}{\numberline {\ref{sec:pfthm2}}Proofs of Theorems \ref{thm:onesamchisq}, \ref{thm:multisamchisq} \& \ref{thm:chisqindp}}{\pageref{sec:pfthm2}}
\contentsline {section}{\numberline {\ref{sec:pfthmnormal36}}Proofs of Theorems \ref{thm:onesamnormal},  \ref{thm:multsamnormal}, \& \ref{thm:indepnormal}}{\pageref{sec:pfthmnormal36}}

\contentsline {chapter}{\numberline {\ref{sec:assitedlemmas}} Proofs of Assisted Lemmas}{\pageref{sec:assitedlemmas}}
\contentsline {section}{\numberline {\ref{sec:gammafunc}}Results on Asymptotic Expansions of the Gamma Function}{\pageref{sec:gammafunc}}
\contentsline {section}{\numberline {\ref{sec:chisqlemma}}Lemmas for Theorems \ref{thm:onesamchisq}, \ref{thm:multisamchisq} \& \ref{thm:chisqindp}}{\pageref{sec:chisqlemma}}
%\contentsline {subsection}{\numberline {C.2.1}Notation and results on finite difference}{\pageref{sec:pfthm3}}
\contentsline {section}{\numberline {\ref{sec:normallemmas}}Lemmas for Theorems \ref{thm:onesamnormal},  \ref{thm:multsamnormal}, \& \ref{thm:indepnormal}}{\pageref{sec:normallemmas}}

%\contentsline {chapter}{\numberline {\ref{sec:extensions}}Supplementary Results}{\pageref{sec:extensions}}
%\contentsline {section}{\numberline {\ref{sec:ext}}Extensions}{\pageref{sec:ext}}
%\contentsline {section}{\numberline {\ref{sec:suppsim}}Supplementary Simulations}{\pageref{sec:suppsim}}
%%\contentsline {subsection}{\numberline {1.2.1}Objective 1}{5}
%\contentsline {chapter}{\numberline {\ref{sec:extensions}}Supplementary Results}{\pageref{sec:extensions}}
%

%\section*{Appendix 1: Proofs of Theorems}

%\tableofcontents

\appendix

%\section*{Appendix 1: Proofs of Theorems}

%\tableofcontents

%\section{Supplementary Results} \label{sec:extensions}
%
%\subsection{Extensions} \label{sec:ext}
%
%\subsection{Supplementary Simulations} \label{sec:suppsim}
%
%
%

\newpage

\section{Additional Results} \label{sec:extensions}

%This section  provides additional results. 
%Particularly, we present theoretical results for three multiple-sample tests (IV)--(VI) in Section \ref{sec:extmulti},
%and another test of independence (VII) in Section \ref{sec:extindp}. 
%Moreover, additional simulation studies are provided in Section \ref{sec:suppsim}, and   
%a discussion is included in Section \ref{sec:suppdisc}. 

% on other likelihood ratio tests mentioned in Section \ref{sec:exten}. 

\subsection{Multiple-Sample Tests} \label{sec:extmulti}

%\subsubsection{Multiple-Sample Tests}\label{sec:mainmult}
%examines three likelihood ratio tests for multiple samples.
This subsection presents the theoretical results of three multiple-sample tests (IV)--(VI).  
Under the multiple-sample problems, 
let $k$ denote the number of samples, which is assumed to be fixed compared to the sample size.  
% $k$ is fixed compared to the sample size $n$. 
 In each sample $i=1,\ldots,k,$ the observations $\MBF{x}_{i 1}, \cdots, \MBF{x}_{i n_{i}}$ are  independent and identically distributed   $\mathcal{N}_{p}\left(\BSM{\mu}_i, \BSM{\Sigma}_i\right)$  random vectors. In this subsection, we define $\overline{\MBF{x}}_{i}=n_i^{-1} \sum_{j=1}^{n_{i}} \MBF{x}_{i j}$ and $\MBF{A}_{i}=\sum_{j=1}^{n_{i}}(\MBF{x}_{i j}-\overline{\MBF{x}}_{i})(\MBF{x}_{i j}-\overline{\MBF{x}}_{i})^{\MYTRAN}$ for $i=1,\ldots, k$, and let $\MBF{A}=\MBF{A}_{1}+\ldots+\MBF{A}_{k}$ and $n=n_{1}+\ldots+n_{k}$.  
 We next briefly review the likelihood ratio tests for the problems (IV)--(VI). 
 
 \smallskip
 
(IV)\ \textit{Testing the Equality of Several Mean Vectors.}\  
Consider 
%testing the equality of mean vectors across multiple samples, that is,  
$H_0: \BSM{\mu}_1=\ldots= \BSM{\mu}_k$ agains  $H_a: H_0$ is not true, 
%\begin{align*}
%	H_0: \BSM{\mu}_1=\ldots= \BSM{\mu}_k,\ \mbox{ vs }\ H_a: H_0 \mbox{ is not true},
%\end{align*} 
where the covariances of the $k$ samples are assumed to be the same.  
Define $\MBF{B}=\sum_{i=1}^{k} n_i (\overline{\MBF{x}}_i-\overline{\MBF{x}})(\overline{\MBF{x}}_i-\overline{\MBF{x}})^{\MYTRAN}$ and $ \overline{\MBF{x}} =n^{-1}\sum_{i=1}^{k}n_i\overline{\MBF{x}}_i$. 
Then, the likelihood ratio test statistic is $\Lambda_n = |\MBF{A}|^{{n}/{2}} |\MBF{A}+\MBF{B}|^{-{n}/{2}}$. 
%;  see, for example, Section 10.7 in \citet{Muirhead2009}. 
%When $k=2$, this likelihood ratio test is equivalent to a two-sample Hotelling's  $T^2$ test. 
% $T^2$-test of Hotelling test. 
%; see, Theorem 6.3.5 in \citet{Muirhead2009}. 
When $p$ is fixed and $n\to \infty$, the chi-squared approximation is $-2\log \Lambda_n \xrightarrow{d} \chi^2_f$, where $f=(k-1)p$, and the chi-squared approximation with the Bartlett correction is  $-2\rho \log \Lambda_n \xrightarrow{d} \chi^2_f$, where $\rho = 1 - \{1 + (k+p)/2\}/n$. 

 \smallskip
 
(V)\ \textit{Testing the Equality of Several Covariance Matrices.}\  
%The equality of several covariance matrices is a common assumption in multivariate analysis of variance.
%% Testing the equality of $k$ covariance matrices considers 
% We study testing 
Consider $H_{0}: \BSM{\Sigma}_1=\ldots=\BSM{\Sigma}_k$ against  $H_{a}: H_{0} $  is not true. 
%, that is,
%\begin{equation} \label{01}
%H_{0}: \BSM{\Sigma}_1=\ldots=\BSM{\Sigma}_k \quad \text { vs } \quad H_{a} : H_{0} \text { is not true}. 
%\end{equation}
%The LRT statistic for this test was first derived in \cite{Wilks1932}:
%\begin{equation}\label{eq241}
%	\Lambda_n = \frac{\prod_{i=1}^{k} |\MBF{A}_i|^{n_i/2}}{|\MBF{A}|^{n/2}}\cdot \frac{n^{np/2}}{\prod_{i=1}^{k}n_i^{n_ip/2}}
%\end{equation}
For this test, $\Lambda_n=|\MBF{A}|^{-(n-k)/2}(n-k)^{(n-k)p/2}\times\prod_{i=1}^{k}(n_i-1)^{-(n_i-1)p/2}|\MBF{A}_i|^{(n_i-1)/2}$ is the  modified likelihood ratio test statistic with the unbiasedness property.  
%\begin{align*}
%	\Lambda_n = \frac{\prod_{i=1}^{k} |\MBF{A}_i|^{n_i/2}}{|\MBF{A}|^{n/2}}\cdot \frac{n^{np/2}}{\prod_{i=1}^{k}n_i^{n_ip/2}}.
%\end{align*}As noted in Section 8.2.2 of \cite{Muirhead2009}, the LRT above is biased, that is, the probability of rejecting $H_0$ when $H_0$ is false can be smaller than that of rejecting $H_0$ when $H_0$ is true. To compensate for this drawback, \citet{Bartlett1937} proposed a modified test statistic which redefine $\Lambda_n$ as follows:
%\begin{equation}\label{stat-1}
%\Lambda_{n}^{*}=\frac{\prod_{i=1}^{k}\left(\left|\MBF{A}_{i}\right|\right)^{\left(n_{i}-1\right) / 2}}{(|\MBF{A}|)^{(n-k) / 2}} \cdot \frac{(n-k)^{(n-k) p / 2}}{\prod_{i=1}^{k}\left(n_{i}-1\right)^{\left(n_{i}-1\right) p / 2}}.
%\end{equation}
%\begin{equation}\label{stat-1}
%\Lambda_{n}=\frac{\prod_{i=1}^{k}\left(\left|\MBF{A}_{i}\right|\right)^{\left(n_{i}-1\right) / 2}}{(|\MBF{A}|)^{(n-k) / 2}} \cdot \frac{(n-k)^{(n-k) p / 2}}{\prod_{i=1}^{k}\left(n_{i}-1\right)^{\left(n_{i}-1\right) p / 2}}.
%\end{equation}
%\noindent Note that the redefined $\Lambda_{n}$ is obtained by replacing sample size $n_i$ with degree of freedom $n_i-1$. The unbiasedness property of this modified LRT was established in \cite{Perlman1980}. 
When $p$ is fixed and $\min _{1 \leq i \leq k} n_{i} \rightarrow \infty$,   
%under the null hypothesis in \eqref{01}, for $\Lambda_n$ in \eqref{stat-1}, 
the  chi-squared approximation is 
$-2  \log \Lambda_{n}\xrightarrow{d} \chi_{f}^{2}$, where $f=p(p+1)(k-1)/2,$ and the chi-squared approximation with the Bartlett correction is  $-2 \rho \log \Lambda_{n}\xrightarrow{d} \chi_{f}^{2}$,  where $\rho=1-\{6(p+1)(k-1)\}^{-1}(2p^2+3p-1)\{\sum_{i=1}^k(n_i-1)^{-1}-(n-k)^{-1}\}.$
%\begin{align*}
%\rho=1-\frac{2 p^{2}+3 p-1}{6(p+1)(k-1)(n-k)}\left(\sum_{i=1}^{k} \frac{n-k}{n_{i}-1}-1\right).	
%\end{align*}
%\smallskip

\smallskip

(VI) \textit{Joint Testing the Equality of Mean Vectors and Covariance Matrices.} \ 
%This test examines 
%whether the distributions of $k$ samples are identical, that is, 
Consider 
$H_0:\BSM{\mu}_1=\ldots =\BSM{\mu}_k, ~ \BSM{\Sigma}_1=\ldots =\BSM{\Sigma}_k$ against $H_a: H_0$ is not true. 
%\begin{equation*}
%H_0:\BSM{\mu}_1=\ldots =\BSM{\mu}_k, ~ \BSM{\Sigma}_1=\ldots =\BSM{\Sigma}_k \quad \mathrm{vs }\quad H_a:H_0 \;\mathrm{is \; not \;true}. 
%\end{equation*}
The likelihood ratio test statistic is 
$\Lambda_n=n^{pn/2} |\MBF{A}+\MBF{B}|^{-n/2}\times \prod_{i=1}^{k}n_i^{-pn_i/2}|\MBF{A}_i|^{n_i/2} $.  
%was derived in \cite{Wilks1932} as follows:
%\begin{equation*}%\label{statistic3}
%\Lambda_n=\frac{\prod_{i=1}^{k}|\MBF{A}_i|^{n_i/2}}{|\MBF{A}+\MBF{B}|^{n/2}} \times \frac{n^{pn/2}}{\prod_{i=1}^{k}n_i^{pn_i/2}},
%\end{equation*}
%which was shown to be unbiased in \cite{Perlman1980}.
When $p$ is fixed and $\min_{1\leq i \leq k}n_i \rightarrow \infty$, the chi-squared approximation  is  $-2\log \Lambda_n\xrightarrow{d}\chi_{f}^2$, where $f=p(k-1)(p+3)/2$, and the chi-squared approximation with the Bartlett correction is $-2\rho \log \Lambda_n\xrightarrow{d}\chi_{f}^2$, where $\rho=1-\{6(k-1)(p+3)\}^{-1}(2p^2+9p+11)(\sum_{i=1}^kn_i^{-1}-n^{-1}).$

\smallskip

For the likelihood ratio tests (IV)--(VI), Theorem \ref{thm:allmult} gives the phase transition boundaries of the chi-squared approximations without and with the Bartlett correction.

\begin{theorem}\label{thm:allmult}
Assume $n_i>p+1$ for $i=1,\dots,k$, and there exists a constant $\delta \in (0,1)$ such that $\delta < n_i/n_j<\delta^{-1}$ for any  $1\leq i,j\leq k.$
%$1\leq i,j \leq k$. 
Under $H_0,$ for the chi-squared approximations without and with the Bartlett correction, we have the following necessary and sufficient conditions:
% as follows:
 
\smallskip
 
(i) $\sup_{\alpha\in (0,1)}|\MYPR\{-2 \log \Lambda_{n}>\chi^2_{f}(\alpha)\}-\alpha|\to 0$ if and only if $p/n^{d_1} \to 0;$

(ii)  when $p=o(n),$  $\sup_{\alpha\in (0,1)}|\MYPR\{-2\rho  \log \Lambda_{n}>\chi^2_{f}(\alpha)\}- \alpha| \to 0$ if and only if $p/n^{d_2} \to 0$,

% \begin{center}
% 	\setlength{\tabcolsep}{20pt}
% \begin{tabular}{lrrr}
% \hline 
%         & (I) Joint & (II) Mean & (III) Covariance \smallskip \\ \smallskip
% $d$  &    $1/2$ &   $2/3$    &  $1/2$  \\ \hline
% \end{tabular}
% \end{center}
%\noindent (ii) consider $p/n\to 0,$ the chi-square approximation with Bartlett correction $\MYPR\{-2\rho\log \Lambda_{n}>\chi^2_{f}(\alpha)\}\rightarrow \alpha$ holds if and only if $p/n^{d} \to 0:$

\smallskip

 \noindent where the values of $d_1$ and $d_2$ under the three testing problems are listed in the table below. 
\begin{center}
	\setlength{\tabcolsep}{18pt}
\begin{tabular}{lccc} 
% \hline 
%\smallskip
      & (IV)  Mean & (V) Covariance & (VI) Joint   \\  
(i) without correction $d_1$:  &    $2/3$ &   $1/2$    &  $1/2$  \\ 
(ii) with correction $d_2$:  &    $4/5$ &   $2/3$    &  $2/3$  \\ 
%\hline
\end{tabular}
\end{center} 
%\begin{center}
%	\setlength{\tabcolsep}{20pt}
%\begin{tabular}{lrrr} 
% \hline 
%      & (IV)  Mean & (V) Covariance & (VI) Joint   \\  \smallskip
%(i) no correction $d$  &    $2/3$ &   $1/2$    &  $1/2$  \\ 
%(ii) with correction $d$  &    $4/5$ &   $2/3$    &  $2/3$  \\ \hline
%\end{tabular}
%\end{center}	
\end{theorem}

In Theorem \ref{thm:allmult}, the boundedness of $n_{i}/n_{j}$ suggests that the sizes of all the samples are comparable. The additional regularity  condition $p=o(n)$ in (ii) specifies a local asymptotic region, which is of practical interest, and simulation studies suggest that the conclusion can hold more generally without this condition.   
%Moreover, we mention that for the chi-square approximations with Bartlett correction 
With a fixed $k$, the phase transition boundaries in Theorem \ref{thm:allmult} are parallel to those in Theorem \ref{thm:onesam}, and the analyses after Theorem \ref{thm:onesam}   apply to   Theorem \ref{thm:allmult} similarly. 
%It is worth mentioning that 
Particularly, examining covariances or not will yield different phase transition boundaries in the three problems. When $k$   also increases with $n$,
% the analyses would be more complicated and 
the phase transition boundaries would involve  $k, p$, and $n$, as illustrated in the following proposition.
%   the mean testing problem (IV). 
%and provides the phase transition boundary when both $k$ and $p$ can increases with $n$.
% We provide one example for the mean testing problem (IV). 
\begin{proposition}\label{prop:1}
%Consider $n>p+k$ and $\max\{p,k\}=o(n)$. 
Consider $n>p+k$, $n-k \to \infty$, and $n-p\to \infty$. 
For $\Lambda_n$ in problem (IV), under $H_0$, as $n\to \infty$,  %We have  
\smallskip
 
(i) $\sup_{\alpha\in (0,1)}|\MYPR\{-2 \log \Lambda_{n}>\chi^2_{f}(\alpha)\}-\alpha|\to 0$ if and only if $\sqrt{pk}(p+k)/n \to 0;$

(ii) $\sup_{\alpha\in (0,1)}|\MYPR\{-2\rho  \log \Lambda_{n}>\chi^2_{f}(\alpha)\}- \alpha| \to 0$ if and only if $\sqrt{pk}(p^2+k^2)/n^2 \to 0$. 
%$n_i>p+1$ for $i=1,\dots,k$,
\end{proposition}

%When $k$ is fixed, the conclusions in Proposition \ref{prop:1} are consistent with those in Theorem \ref{thm:allmult}. 

Proposition \ref{prop:1} suggests that the total number of samples $k$ and the dimension of each observation $p$ play symmetric roles in the phase transition boundary of problem (IV).  
When $k$ is fixed, Proposition \ref{prop:1} is consistent with Theorem \ref{thm:allmult}.
To further  illustrate the cases with increasing $k$, 
% the implications of Proposition \ref{prop:1}, 
 we consider $p=\lfloor n^{\epsilon} \rfloor$ and $k=\lfloor n^{\eta} \rfloor$, where $0<\epsilon,\eta <1$ and $\lfloor \cdot \rfloor$ denotes the floor of a number. Then the two phase transition boundaries in Proposition \ref{prop:1} become (i) $\max\{\epsilon, \eta\}+(\epsilon+\eta)/2 < 1$ and (ii) $\max\{\epsilon, \eta\}+(\epsilon+\eta)/4 < 1$, respectively. 
%These give two sub-areas of the two-dimensional region formed by $\epsilon,\eta \in (0,1)$. 
Specifically, for (i), when $\epsilon$ is close to 0, the largest value of $\eta$ is around $2/3,$ and vice versa;
% and symmetric conclusion holds when $\eta$ approaches 0. 
when $\epsilon=\eta$, suggesting $p$ and $k$ are of the same order, the largest value of $\epsilon$ is $1/2$. For (ii), when  $\epsilon$ is close to 0, the largest value of $\eta$ is around $4/5,$ and vice versa; when $\epsilon=\eta$, the largest value of $\epsilon$ becomes $2/3$.

In addition to the phase transition boundaries above,  
the following Theorem \ref{thm:multisamchisq}, similarly to Theorem \ref{thm:onesamchisq}, 
further characterizes  
the accuracy of each chi-squared approximation for tests (IV)--(VI) when Wilk's theorem holds. 
Specifically, we consider  $p/n^{d_1}\to 0$ and $p/n^{d_2}\to 0$ for the chi-squared approximations without and with the Bartlett correction, respectively.

%Similarly to Theorem \ref{thm:onesamchisq},
%the following Theorem \ref{thm:multisamchisq} further 
%characterizes  
%the accuracy of each chi-squared approximation under the asymptotic regime where $p$ satisfies the corresponding necessary and sufficient condition  in Theorem \ref{thm:allmult}.

%when the approximation holds. Particularly, for each test, Theorem \ref{thm:multisamchisq} considers the asymptotic regime where $p$ satisfies the corresponding necessary and sufficient condition  in Theorem \ref{thm:allmult}.  

\begin{theorem}\label{thm:multisamchisq}
%$n_i>p+1$ for $i=1,\dots,k$ $p \to \infty$ as $n\to \infty$, and
Assume that there exists a constant $\delta \in (0,1)$ such that $\delta < n_i/n_j<\delta^{-1}$ for any  $1\leq i,j\leq k,$ and $p\to \infty$ as $n\to \infty$.  
%Consider $p \to \infty$. 
For each likelihood ratio test (IV)--(VI), let $d_i$, $i=1,2$ take the corresponding values in Theorem \ref{thm:allmult}. Then under $H_0$, for any $\alpha\in (0,1)$, 
\begin{itemize}
	\item[(i)] when $p/n^{d_1} \to 0$, 
\eqref{eq:chisqapprox} in Theorem \ref{thm:onesamchisq} holds with the value of $\vartheta_1(n,p)$ listed below;
\item[(ii)] when $p/n^{d_2}\to 0$,  \eqref{eq:chisqapproxbartcorr} in Theorem \ref{thm:onesamchisq} holds with the 
 values of $\vartheta_2(n,p)$ listed below.
\end{itemize}
%when $p/n^{d_1} \to 0$, 
%\eqref{eq:chisqapprox} holds with the value of $\vartheta_1(n,p)$ listed below, and when $p/n^{d_2}\to 0$,  \eqref{eq:chisqapproxbartcorr} holds with the 
% values of $\vartheta_2(n,p)$ listed below.  
Let $D_{n,r}=\sum_{i=1}^k n_i^{-r} - n^{-r}$ and $\tilde{D}_{n,r}=\sum_{i=1}^k (n_i-1)^{-r} - (n-k)^{-r}$. 
% Let $D_{n,1}=\sum_{i=1}^k n_i^{-1} - n^{-1}$, $D_{n,2}=\sum_{i=1}^k n_i^{-2} - n^{-2}$, $\tilde{D}_{n,1}=\sum_{i=1}^k (n_i-1)^{-1} - (n-k)^{-1}$, and $\tilde{D}_{n,2}=\sum_{i=1}^k (n_i-1)^{-2} - (n-k)^{-2}$. 
\begingroup
\allowdisplaybreaks
\begin{eqnarray}
&\hspace{-6.6em} \text{ (IV)~ Mean: } &\vartheta_1(n,p)=	\frac{p(k-1)(p+2+k)}{4n\sqrt{f}}, \notag \\[3pt]
&&  \vartheta_2(n,p)= \frac{(k - 1) p (p^2+k^2-2k - 4)}{24n^2\rho^2\sqrt{f}}; \notag \\[5.5pt] % \frac{(k - 1) p (p^2+k^2-2k - 4)}{24\{ n -(p+k+2)/2\}^2\sqrt{f}}. \notag \\[5.5pt]
&\hspace{-3.6em}  \text{  (V)~  Covariance: }\quad \   & \vartheta_1(n,p)=\frac{\tilde{D}_{n,1}p(2p^2+3p-1)}{24\sqrt{f}}, \notag \\[3pt]
&&\vartheta_2(n,p)=\frac{p(p+1)}{24\rho^2\sqrt{f}}\Big\{ (p-1)(p+2)\tilde{D}_{n,2} - 6(k-1)(1-\rho)^2 \Big\};  \notag \\[5.5pt]
&\hspace{-7.2em} \text{ (VI)~ Joint: }&\vartheta_1(n,p)=\frac{D_{n,1}p\left(2 p^{2}+9 p+11\right)}{24\sqrt{f}}, \notag \\[3pt]
&& \vartheta_2(n,p)=\frac{p(p+3)}{24\rho^2\sqrt{f}}\Big\{(p+1)(p+2)D_{n,2}-6(k-1)(1-\rho)^2\Big\}. \notag
\end{eqnarray}
\endgroup 
\end{theorem}
\smallskip

% in Theorem \ref{thm:onesamchisq}

% The conclusions in Theorem \ref{thm:multisamchisq} are parallel to those in Theorem \ref{thm:onesamchisq}. 
%Particularly, 
 Theorem \ref{thm:multisamchisq} shows that for multiple-sample tests (IV)--(VI),
  \eqref{eq:chisqapprox} and \eqref{eq:chisqapproxbartcorr} in Theorem \ref{thm:onesamchisq} still hold.
%   for the chi-squared approximations without and with the Bartlett correction,  respectively. 
%the biases of  the chi-squared approximations without and with the Bartlett correction still take  the forms of \eqref{eq:chisqapprox} and \eqref{eq:chisqapproxbartcorr}, respectively.  
However, the values of $\vartheta_1(n,p)$ and $\vartheta_2(n,p)$  depend on the testing problems, and are different from those in Theorem \ref{thm:onesamchisq}. 
Similarly to Theorem \ref{thm:onesamchisq}, in each test (IV)--(VI), 
we also know that $\vartheta_1(n,p)$ and $\vartheta_2(n,p)$ are of the orders of  
$p^{1/d_1}n^{-1}$ and $p^{2/d_2}n^{-2}$, respectively. 
Then $\vartheta_1(n,p)\exp(-z_{\alpha}^2/2)/\sqrt{\pi}$ in \eqref{eq:chisqapprox} and  $\vartheta_2(n,p)\exp(-z_{\alpha}^2/2)/\sqrt{\pi}$ in \eqref{eq:chisqapproxbartcorr} are the leading terms of the biases of the chi-squared approximations without and with the Bartlett correction, respectively. 
We can similarly use the derived asymptotic biases to measure the  approximation accuracy, and 
please see the simulation studies for multiple-sample tests (IV)--(VI) in $\S$\,\ref{sec:suppsim}.

Theorem \ref{thm:multisamchisq} focuses on the local asymptotic regime of $(n,p)$ when Wilk's theorem holds. 
%the necessary and sufficient conditions  in Theorem \ref{thm:allmult} are satisfied.
When $p$ further increases such that Wilk's theorem fails, 
the biases of the chi-squared approximations  become  unignorable.  
The following Theorem \ref{thm:multsamnormal} characterizes such unignorable biases of the chi-squared approximations in testing problems (IV)--(VI). 
Similarly to Theorem \ref{thm:onesamnormal}, 
we consider a general local asymptotic regime $p/n\to 0$, which includes the case when Wilk's theorem fails, i.e.,  $p/n^{d_1}\not \to 0$ and $p/n^{d_2}\not \to 0$ for the chi-squared approximations without and with the Bartlett correction, respectively. 

% for the chi-squared approximation, and $p/n^{d_2}\not \to 0$ for the chi-squared approximation with the Bartlett correction. 

\begin{theorem}\label{thm:multsamnormal}
%$n_i>p+1$ for $i=1,\dots,k$, and
Assume that there exists a constant $\delta \in (0,1)$ such that $\delta < n_i/n_j<\delta^{-1}$ for any  $1\leq i,j\leq k.$
Moreover, assume $p\to \infty$ and $p/n_i \to 0$ as $n_i \to \infty$. 
%$n-p\to \infty$ and $p\to \infty$ as $n\to \infty$. 
For each likelihood ratio test  (I)--(III), under $H_0$, for any $\alpha\in (0,1)$,  \eqref{eq:normalbias1} and \eqref{eq:normalbias2} in Theorem \ref{thm:onesamnormal} hold under three testing problems (IV)--(VI)  with $\mu_n$ and $\sigma_n$ listed below. 
	%For easy presentation, we define  and $L_{x,p}=\log(1-p/x)$ for $x>p$. 
%(IV) Mean:
\begingroup
\allowdisplaybreaks
\begin{eqnarray}
&\hspace{-3.6em} \text{(IV) Mean: }  &\hspace{-0.2em}  \mu_n=\frac{n}{2}\big\{(n-p-k-1/2) (L_{n-1,p}-L_{n-k,p} )+(k-1)L_{n-1,p}+pL_{n-1,k-1}  \big\}, \notag \\ %\label{eq:multisammeanmu} \\
&&\hspace{0.1em}	\sigma_n^2=\frac{1}{2}\big( L_{n-1,p} - L_{n-k,p} \big);\notag \\
& \hspace{-1.6em}  \text{(V) Covariance: } & \hspace{-0.2em} \mu_n=\frac{1}{2} \sum_{i=1}^{k}(n_i-1)\Big\{ ( n- p-k-1/2)L_{n-k,p}-\left( n_{i}- p-3/2\right)  L_{n_i-1, p} \Big\}, \notag \\	 %\label{eq:multicovmun} \\
& &\hspace{0.1em} \sigma_{n}^{2}= \frac{(n-k)^2}{2n^2} \Biggr\{L_{n-k,p}-\sum_{i=1}^{k}\left(\frac{n_{i}-1}{n-k}\right)^{2}L_{n_i-1,p} \Biggr\}; \notag \\
&\hspace{-3.7em} \text{(VI) Joint: } &\hspace{-0.2em} \mu_n =\frac{1}{2} \left[-kp +n\Big(n-p-\frac{3}{2}\Big)L_{n,p}-\sum_{i=1}^{k}\biggr\{\frac{p}{2n_i}+n_i\Big(n_i-p-\frac{3}{2}\Big)L_{n_i-1,p}\biggr\}\right], \hspace{0.4em} \notag \\ %\label{eq:multijointmu} \\
%&\hspace{-4em} \text{(VI) Joint: } & \mu_n =\frac{1}{2} \left\{-kp-\sum_{i=1}^k \frac{p}{2n_i}+n(n-p-3/2)L_{n,p}-\sum_{i=1}^{k}n_i(n_i-p-3/2)L_{n_i-1,p}\right\}, \hspace{1em}  \label{eq:multijointmu} \\
&&\hspace{0.1em} \sigma_n^2 = \frac{1}{2}\left(L_{n,p}- \sum_{i=1}^{k}\frac{n_i^2}{n^2} \times L_{n_i-1,p} \right). \notag
\end{eqnarray}
\endgroup
\end{theorem}

Theorem \ref{thm:multsamnormal} shows that \eqref{eq:normalbias1} and  \eqref{eq:normalbias2} still hold for multiple-sample tests (IV)--(VI),
where the values of $\mu_n$ and $\sigma_n^2$ depend on the specific testing problem.  
Similarly to Theorem \ref{thm:onesamnormal}, 
the analysis in Remark \ref{rm:order} also applies here, and   we know that when $pn^{-d_1}\not \to 0$, \eqref{eq:normalbias1}  characterizes the unignorable biases for
 the chi-squared approximation, and when $pn^{-d_2}\not \to 0$,
  \eqref{eq:normalbias2} characterizes the  unignorable biases for the chi-squared approximation with the Bartlett correction.  
Moreover, the analysis in Remark \ref{rm:orderelation} also applies similarly to the multiple-sample tests (IV)--(VI), and thus is not repeated here.

\subsection{Testing Independence between Multiple Vectors} \label{sec:extindp} 
%\subsubsection{Testing Independence between Multiple Vectors}\label{sec:mainindp}
% Subvectors of Multivariate Normal Distribution
%In this section, we consider testing the independence of subvectors of multivariate normal distribution. 
This subsection studies testing the independence between $k$ sets of multivariate normal variables. 
%\subsection{Testing Independence of Subvectors of Multivariate Normal Distribution} \label{Independence}
%\subsection{Testing Independence of Components of Normal Distributions} \label{Independence}
%Given a multivariate normal vector, one useful test is to check the independence between  
%A common assumption widely used in statistics is that a certain number of random vectors are mutually independent. Under the assumption of independence many procedures can be developed. And thus before applying these methods it is important to test the independence of several random vectors. 
%This is an old problem and there are some recent developments. For instance, \cite{Fan2017} constructed a test based on characteristic function without a normal assumption. \cite{Genest2019} proposed model-free tests of mutual independence where they allow discrete, continuous variables and even a mixture of the two types to exist in a vector. 
%In this paper we examine the validity of $\chi^2$ approximation for LRT statistic under the normal assumption. 
%Suppose the $p$-dimensional vector $\MBF{x}_i$ can be written into  $k$ subvectors $(\BSM{\xi}_{i,1}^{\MYTRAN},\dots,\BSM{\xi}_{i,k}^{\MYTRAN})^{\MYTRAN}$, which are of  dimensions  $p_1,\ldots, p_k$, respectively.   
Suppose $\MBF{x}_1,\dots,\MBF{x}_n\in \MBB{R}^p$ are independent and identically distributed   $\mathcal{N}_p(\BSM{\mu},\BSM{\Sigma})$ random vectors, and we partition $x_i$ and $\BSM{\Sigma}$ as $\MBF{x}_i = (\BSM{\xi}_{i1}^{\MYTRAN},\dots,\BSM{\xi}_{ik}^{\MYTRAN})^{\MYTRAN}$ and $\BSM{\Sigma} = (\Sigma_{jl})_{1\leq j,l \leq k}$, respectively, where $\BSM{\xi}_{i,j}$ is of size $p_j\times 1$, $\Sigma_{jl}$ is a  $p_j \times p_l$ sub-matrix of $\BSM{\Sigma}$, and $\sum_{j=1}^k p_j=p$. In this subsection, we define $\overline{\MBF{x}} = n^{-1}\sum_{i=1}^{n}\MBF{x}_i$, $\bar{\BSM{\xi}}_j=n^{-1}\sum_{i=1}^n \xi_{ij}$, $ \MBF{A} = \sum_{i=1}^{n}(\MBF{x}_i-\overline{\MBF{x}})(\MBF{x}_i-\overline{\MBF{x}})^{\MYTRAN}$, and $A_{jj}=\sum_{i=1}^n(\BSM{\xi}_{ij}-\bar{\BSM{\xi}}_{j})(\BSM{\xi}_{ij}-\bar{\BSM{\xi}}_{j})^{\MYTRAN}.$  

\smallskip

%\begin{enumerate}
%    \item[(VII)] \textit{Testing the Independence of Subvectors of Multivariate Normal Distribution.}
(VII)\ \textit{Testing Independence of Subvectors of Multivariate Normal Distribution.}
% \noindent \textbf{(VII) Testing the Independence of Subvectors of Multivariate Normal Distribution}\quad 
For the multivariate normal distribution, 
%Suppose $p_1,\dots,p_k$ are $k$ positive integers and $k\geq2$. Denote $p = \sum_{i=1}^{k}p_i$ and let
%$
%\BSM{\Sigma} = (\Sigma_{ij})_{1\leq i,j \leq k}
%$
%be a positive definite matrix, where $\Sigma_{ij}$ is a $p_i \times p_j$ sub-matrix for all $1\leq i,j \leq k$. 
%Assume $\BSM{\xi}_i \in \MBB{R}^{p_i} $ is a normal vector for each $1 \leq i \leq k$ , and $(\BSM{\xi}_1^{\MYTRAN},\dots,\BSM{\xi}_k^{\MYTRAN})^{\MYTRAN}$ is distributed as $\mathcal{N}_p(\BSM{\mu},\BSM{\Sigma})$. We want to
 testing the independence between $k$ sets of vectors $\BSM{\xi}_{i,1},\dots,\BSM{\xi}_{i,k}$ is equivalent to testing $H_0:\Sigma_{jl} = 0$, for  
 %$j,l = 1,\ldots, k$, and $j\neq l$ 
  $1\leq j<l \leq k$, 
 against $H_a: H_0$ is not true.  
%\begin{equation} \label{eq:subindep}
%H_0:\Sigma_{jl} = 0 \; \mathrm{for \; all} \;1\leq j<l \leq k \mathrm{\quad vs} \quad H_a:H_0 \mathrm{\;is \;not \;true}
%\end{equation}
%For the hypothesis in \eqref{eq:subindep}, 
The likelihood ratio statistic is $\Lambda_n = |\MBF{A}|^{n/2}\prod_{j=1}^k |{A}_{jj}|^{-n/2}.$
%\begin{equation}\label{statistic2}
%\Lambda_n = |\MBF{A}|^{n/2}\times {\prod_{j=1}^{k}|\MBF{A}_{jj}|^{-n/2}} : = (W_n)^{n/2},
%\end{equation}
%\begin{equation}\label{statistic2}
%\Lambda_n = \frac{|\MBF{A}|^{n/2}}{\prod_{j=1}^{k}|\MBF{A}_{jj}|^{n/2}} : = W_n^{n/2},
%\end{equation}
%\begin{equation}\label{statistic2}
%\Lambda_n = \frac{|\MBF{A}|^{n/2}}{\prod_{j=1}^{k}|{A}_{jj}|^{n/2}},
%\end{equation}
%where ${A}_{jj}/n$ is the sample covariance matrix of $\BSM{\xi}_{i,j}$, $ i=1,\ldots, n$, and we have $\MBF{A}=({A}_{jl})_{1\leq j,l\leq k}$ corresponding to the partition $\BSM{\Sigma} = (\Sigma_{ij})_{1\leq i,j \leq k}$. 
%denotes the $p_j\times p_l$ submatrix in $\MBF{A}$ and $\MBF{A}=(\MBF{A}_{jl})_{1\leq j,l\leq k}$ corresponds the  
%we partition the matrix $\MBF{A}$ 
%Assume that $\BSM{x}_1,\dots,\BSM{x}_N$ are i.i.d. sampled from distribution $\mathcal{N}_p(\BSM{\mu},\BSM{\Sigma})$. Set $n=N-1$ and define
%\[
%\BSM{A}=\sum_{i=1}^{n+1}(\BSM{x}_i-	\overline{\BSM{x}})(\BSM{x}_i-	\overline{\BSM{x}})' \quad \mathrm{with} \quad 	\overline{\BSM{x}}=\frac{1}{n+1}\sum_{i=1}^{n+1}\BSM{x}_i
%\] 
%and partition it similarly to $\BSM{\Sigma} = (\Sigma_{ij})_{1\leq i,j \leq k}$ as follows
%\[
%\BSM{A} = \;
%\begin{pmatrix}
%{A}_{11} & {A}_{12} & \ldots & {A}_{1k}\\
%{A}_{21} & {A}_{22} & \ldots & {A}_{2k}\\
%\vdots & \vdots & \ddots & \vdots\\
%{A}_{k1} &{A}_{k2} & \ldots & {A}_{kk}\\
%\end{pmatrix}
%\]
%where ${A}_{ij}$ is a $p_i\times p_j$ sub-matrix. 
%The likelihood ratio test is shown to be unbiased in \cite{Narain1950}.
When $p_1,\dots,p_k$ are fixed, the chi-squared approximation is $-2\log \Lambda_n\xrightarrow{d} \chi_{f}^2$, where $f =(p^2-\sum_{i=1}^{k}p_i^2)/2;$ the chi-squared approximation with the Bartlett correction is  $-2\rho \log \Lambda_n\xrightarrow{d} \chi_{f}^2$, where $\rho=1- (3/2n)^{-1}-\{3n(p^2-\sum_{i=1}^{k}p_i^2)\}^{-1}(p^3-\sum_{i=1}^{k}p_i^3).$

\smallskip

Theorem \ref{Thm2.1} below gives the phase transition boundaries of the chi-squared approximations without and with the Bartlett correction for test (VII). 

\begin{theorem}\label{Thm2.1}
Assume $n>p+1$ and there exists $\delta \in (0,1)$ such that $\delta < p_i/p_j<\delta^{-1}$ for 
%all $i,j=1,\ldots, k.$
 	$1\leq i,j \leq k$. 
	For $\Lambda_n$ in problem (VII), under $H_0$, as $n\to \infty$, 
\smallskip
 
(i) $\sup_{\alpha\in (0,1)}|\MYPR\{-2 \log \Lambda_{n}>\chi^2_{f}(\alpha)\}-\alpha|\to 0$ if and only if $p/n^{1/2} \to 0;$

(ii)   when $p=o(n),$  $\sup_{\alpha\in (0,1)}|\MYPR\{-2\rho  \log \Lambda_{n}>\chi^2_{f}(\alpha)\}- \alpha| \to 0$ if and only if $p/n^{2/3} \to 0$.
%, when $p/n\to 0.$
	%	for any significance level $\alpha$,  as $n\to \infty,$
%\begin{itemize}
%	\item[(i)] $\MYPR(-2 \log \Lambda_n>\chi_{f}^2(\alpha))$ if and only if $p/n^{1/2}\to 0$ for $i=1,\ldots,k.$
%	\item[(ii)] When ${p}/{n} \rightarrow 0$, $\MYPR\{-2\rho \log \Lambda_n>\chi_{f}^2(\alpha)\}$ if and only if $p/n^{2/3}\to 0$ for $i=1,\ldots,k.$
%\end{itemize}		
\end{theorem}

The phase transition boundaries in Theorem \ref{Thm2.1} are 
 consistent with those in Theorems \ref{thm:onesam} and \ref{thm:allmult} for 
 testing problems (II), (III), (V), and (VI).
% the tests involving covariance matrices. 
 This is reasonable because testing independence between multivariate normal vectors  examines the structures of covariance matrices. 
In Theorem \ref{Thm2.1}, the boundedness of $p_i/p_j$ suggests that the dimensions of the multiple vectors are comparable. The following Proposition \ref{prop:2}   relaxes this constraint for $k=2$, a case closely related to the canonical correlation analysis. %We mention that testing independence between two sets of variables is closely related to the canonical correlation analysis, which may be of separate interest. 
\begin{proposition}\label{prop:2}
Consider $n>p_1+p_2$ and $n-\max\{p_1,p_2\}\to \infty$. For $\Lambda_n$ in problem (VII), under $H_0$, as $n\to \infty$,
%$\max\{p_1,p_2\}=o(n)$. As $n\to \infty,$ %We have
\smallskip
 
(i) $\sup_{\alpha\in (0,1)}|\MYPR\{-2 \log \Lambda_{n}>\chi^2_{f}(\alpha)\}-\alpha|\to 0$ if and only if $\sqrt{p_1p_2}(p_1+p_2)/n \to 0;$

(ii) $\sup_{\alpha\in (0,1)}|\MYPR\{-2\rho  \log \Lambda_{n}>\chi^2_{f}(\alpha)\}- \alpha| \to 0$ if and only if $\sqrt{p_1p_2}(p_1^2+p_2^2)/n^2 \to 0$. 
\end{proposition}

Proposition \ref{prop:2} shows that the effects of $p_1$ and $p_2$  on the phase transition boundaries are symmetric.  
%When $p_1$ and $p_2$ are comparable, e.g., $p_1=p_2=p/2$, Proposition \ref{prop:2} provides conclusions consistent with those in Theorem \ref{Thm2.1}. To illustrate the settings when $p_1$ and $p_2$ vary differently,
To further illustrate, consider $p_1=\lfloor n^{\epsilon} \rfloor$ and $p_2=\lfloor n^{\eta} \rfloor$, where $0<\epsilon,\eta <1$. Then the two phase transition boundaries in Proposition \ref{prop:2} become (i) $\max\{\epsilon, \eta\}+(\epsilon+\eta)/2 < 1$ and (ii) $\max\{\epsilon, \eta\}+(\epsilon+\eta)/4 < 1$, respectively. When $\epsilon=\eta$, i.e., $p_1$ and $p_2$ are of the same order, the largest value of $\epsilon$ and $\eta$ achievable is (i) $1/2$ and (ii) $2/3$ respectively, which are consistent with Theorem \ref{Thm2.1}. When $\eta$ is close to 0, the largest value of $\epsilon$ is (i) $2/3$ and (ii) $4/5$ respectively. 
%A similar result holds when $\epsilon$ is close to 0. 
Therefore when one set of the vectors is of finite  dimension, the chi-squared approximations without and with the Bartlett correction can be applied when $p/n^{2/3}\to 0$ and $p/n^{4/5}\to 0$, respectively. This demonstrates an interesting phenomenon that for the phase transition boundary, the growth rate of $p$ changes as the ratio of $p_1$ and $p_2$ varies. 
%dimensions of different sets of vectors vary. 

 Similarly to Theorems \ref{thm:onesamchisq} and \ref{thm:multisamchisq}, 
the following Theorem \ref{thm:chisqindp}  further 
characterizes  
the accuracy of the chi-squared approximation under the asymptotic regime where $p$ satisfies the corresponding necessary and sufficient conditions in Theorem \ref{Thm2.1}.

%when the chi-squared approximation holds.
%Particularly, Theorem \ref{thm:chisqindp} considers the asymptotic regime where $p$ satisfies the corresponding necessary and sufficient conditions in Theorem \ref{Thm2.1}.

\begin{theorem}\label{thm:chisqindp}
Assume that there exists $\delta \in (0,1)$ such that $\delta < p_i/p_j<\delta^{-1}$ for $1\leq i,j\leq k$,
and $p\to \infty$ as $n\to \infty$. 
Let $d_1=1/2$ and $d_2=2/3$ as in Theorem \ref{Thm2.1}. 
%Let $d_i$, $i=1,2$ take the corresponding values in Theorem \ref{Thm2.1}. 
For $\Lambda_n$ in problem (VII), under $H_0$, for any $\alpha\in (0,1)$, 
\begin{itemize}
	\item[(i)] when $p/n^{d_1} \to 0$, 
\eqref{eq:chisqapprox} in Theorem \ref{thm:onesamchisq} holds with the value of $\vartheta_1(n,p)$  below;
\item[(ii)] when $p/n^{d_2}\to 0$,  \eqref{eq:chisqapproxbartcorr} in Theorem \ref{thm:onesamchisq} holds with the 
 value of $\vartheta_2(n,p)$ below.
\end{itemize}
Let $D_{p,r}=p^r-\sum_{j=1}^k p_j^r$. Then
\begin{align*}
\vartheta_1(n,p)=\frac{2D_{p,3}+9D_{p,2}}{24n\sqrt{f}},\quad \quad   		 \vartheta_2(n,p)= \frac{1}{(\rho n)^2\sqrt{f}}\biggr(\frac{1}{24}D_{p,4}-\frac{5D_{p,2}}{48} -\frac{D_{p,3}^2}{36D_{p,2}} \biggr).
\end{align*}
%\begin{align*}
%\vartheta_1(n,p)=&~\frac{2D_{p,3}+9D_{p,2}}{24n\sqrt{f}}, \notag \\
%	 \vartheta_2(n,p)=&~ \frac{1}{(\rho n)^2\sqrt{f}}\left(\frac{1}{24}D_{p,4}-\frac{5D_{p,2}}{48} -\frac{D_{p,3}^2}{36D_{p,2}} \right).
%\end{align*}
\end{theorem}
 
 \smallskip

 Similar to Theorems  \ref{thm:onesamchisq} and  \ref{thm:multisamchisq}, Theorem \ref{thm:chisqindp}  focuses on the local asymptotic regime when  Wilk's theorem holds, 
 and we  know from a similar analysis that \eqref{eq:chisqapprox} and  \eqref{eq:chisqapproxbartcorr} 
  provide useful information on the accuracy of the chi-squared approximations.
  Please see the simulations for test (VII) in $\S$\,\ref{sec:suppsim}. 
 When $p$ further increases such that Wilk's theorem fails,  
 the following Theorem \ref{thm:indepnormal} characterizes the unignorable chi-squared approximation biases for test (VII) similarly as in Theorems \ref{thm:onesamnormal} and \ref{thm:multsamnormal}.

%  focuses on the local asymptotic regime when  Wilk's theorem holds. 
%%  the chi-squared approximation holds.  
%When $p$ further increases such that the necessary and sufficient conditions  in Theorem \ref{Thm2.1} are  violated, 
%the biases of the corresponding chi-squared approximations  become  unignorable. 
%Similarly to Theorems \ref{thm:onesamchisq} and \ref{thm:multisamchisq}, 
%the following Theorem \ref{thm:indepnormal} characterizes such unignorable biases of the chi-squared approximations in the testing problem (VII).  
%%Similarly to Theorems \ref{thm:onesamchisq} and \ref{thm:multisamchisq}, 
%In particular, we consider a general local asymptotic regime $p/n \to 0$, which includes the case  when the chi-squared approximation does not hold, that is,  $p/n^{d_1}\not \to 0$ for the chi-squared approximation, and $p/n^{d_2}\not \to 0$ for the chi-squared approximation with the Bartlett correction. 

\begin{theorem}\label{thm:indepnormal}
Assume that
%Assume $n>p+1$ and 
there exists $\delta \in (0,1)$ such that $\delta < p_i/p_j<\delta^{-1}$ for $1\leq i,j\leq k$, and
$p\to \infty$ and $p/n \to 0$ as $n \to \infty$. 
For $\Lambda_n$ in problem (VII), under $H_0$, as $n\to \infty$, 
for any $\alpha\in (0,1)$,  \eqref{eq:normalbias1} and \eqref{eq:normalbias2} in  Theorem \ref{thm:onesamnormal}  hold with $\mu_n$ and $\sigma_n$ listed below.
%\begin{align*}
%\mu_n=\frac{n}{2}\left[\left(n-p-\frac{3}{2}\right) L_{n-1, p}+\sum_{j=1}^{k}\left\{\left(n-p_{j}-\frac{3}{2}\right) L_{n-1, p_{j}}\right\}\right],\quad  \sigma_n^2=\frac{1}{2}\biggr(L_{n-1, p}+\sum_{j=1}^{k}L_{n-1, p_{j}}\biggr). 	
%\end{align*}
\begin{align*}
\mu_n=&~\frac{n}{2}\biggr[-\left(n-p-\frac{3}{2}\right) L_{n-1, p}+\sum_{j=1}^{k}\left\{\left(n-p_{j}-\frac{3}{2}\right) L_{n-1, p_{j}}\right\}\biggr], \notag \\
\sigma_n^2=&~\frac{1}{2}\biggr(-L_{n-1, p}+\sum_{j=1}^{k}L_{n-1, p_{j}}\biggr). 	
\end{align*}
\end{theorem} 

Note that Theorem \ref{thm:indepnormal} is analogous to Theorems \ref{thm:onesamnormal} and \ref{thm:multsamnormal}, and therefore similar analyses and conclusions as in Remarks \ref{rm:orderelation} and \ref{rm:order} also hold for test (VII), which are  not repeated here.

 \smallskip

\subsection{Additional Simulations} \label{sec:suppsim}

We next introduce the simulation settings of each test and afterwards analyze the numerical results. 

%This subsection presents additional simulations of tests (I)--(VII).  
%We next introduce the simulation settings of each test. 

%(I)--(III) in Section \ref{sec:main} and tests (IV)--(VII)  in Section \ref{sec:extensions}.   

\subsubsection{One-Sample Tests (I)--(III).}
Similarly to Section \ref{sec:simulations}, 
under the null hypothesis of each one-sample test (I)--(III), we set $\BSM{\mu} = (0,\ldots,0)^{\MYTRAN}$ and $\BSM{\Sigma} = \BIEN$. 
\smallskip

\noindent \textit{(1) On the phase transition boundaries.}\  We take $p = \lfloor n^\epsilon \rfloor$, where $n\in\{100,500,1000,5000\}$ and $\epsilon \in\{6/24,\ldots, 23/24 \}$. 
We next plot the  empirical type-\RNum{1} error rates (over 1000 replications) versus $\epsilon$ for each chi-squared approximation in Fig. \ref{fig:simone}. 
We still include the results in $\S$\,\ref{sec:simulations} for easy  presentation of the figure.   

\smallskip

\noindent \textit{(2) On the asymptotic biases.}\  To evaluate the asymptotic biases in  Theorems \ref{thm:onesamchisq} and \ref{thm:onesamnormal}, we take $p=\lfloor n^{\epsilon} \rfloor$, where  $n\in \{100, 500\}$ and $\epsilon\in (0,1)$. 
 The results of $n=100$ and $500$ (over 3000 replications) are given in Fig. \ref{fig:bias13n100} and Fig. \ref{fig:bias13n500}, respectively. 
In each setting, the range of $\epsilon$ is chosen such that the largest empirical type-\RNum{1} error is below 0.5.
%We mention that a smaller number of $\epsilon$ values are included under $n=100$ compared to that under $n=500$ to avoid that $\lfloor n^{\epsilon} \rfloor$ returns the same value under multiple $\epsilon$ values. 

To facilitate the presentation of figures and the discussions below, we define 
\begin{align*}
&\varpi_1=\vartheta_1(n,p)\exp(-z_{\alpha}^2/2)/\sqrt{\pi}, \quad 	\varpi_3=\bar{\Phi}\big[ \big\{ \chi^2_f(\alpha)+2\mu_n \big\}/(2n\sigma_n) \big] - \alpha, \notag \\
&\varpi_2=\vartheta_2(n,p)\exp(-z_{\alpha}^2/2)/\sqrt{\pi}, \quad \varpi_4=\bar{\Phi}\big[ \big\{ \chi^2_f(\alpha)+2\rho \mu_n \big\}/(2\rho n\sigma_n) \big] - \alpha.
\end{align*}
Then $\varpi_1, \varpi_2, \varpi_3,$ and $\varpi_4$ 
 denote the asymptotic biases in \eqref{eq:chisqapprox}--\eqref{eq:normalbias2}, respectively. 
For each test in Fig. \ref{fig:bias13n100} and Fig. \ref{fig:bias13n500}, 
we plot $\varpi_1$ and $\varpi_2$ in the subfigures in the columns (a) and (c), respectively. 
Similarly to $\S$\,\ref{sec:simulations}, 
to better characterize each approximation bias when $\epsilon$ is beyond the corresponding phase transition boundary, 
we combine the results in Theorem \ref{thm:onesamchisq} and those in Theorem \ref{thm:onesamnormal}.  
%utilize $\max\{\varpi_1, \varpi_3\}$ and $\max\{\varpi_2, \varpi_4\}$,
%combining Theorem \ref{thm:onesamchisq} and Theorem \ref{thm:onesamnormal}.  
Specifically, 
%since   $\varpi_1$ characterizes the approximation bias $\MYPR\{-2 \log \Lambda_{n}>\chi^2_{f}(\alpha)\}-\alpha$ well when $\epsilon$ is small, 
%while $\varpi_2$ performs better when $\epsilon$ is larger,
in the column (b) of Fig. \ref{fig:bias13n100} and Fig. \ref{fig:bias13n500}, 
we plot $M_c(\varpi_1, \varpi_3)\equiv \varpi_1 1\{\varpi_1 < c \}+ \max\{ \varpi_1 , \varpi_3 \} 1\{\varpi_1 \geq c \}$, where  $1\{\cdot\}$ denotes an indicator function, and $c$ denotes a small positive threshold, and we choose $c=0.002$ in the simulations.  
This definition of $M_c(\varpi_1, \varpi_3)$ 
suggests that $\varpi_1$ is used when the approximation bias is smaller than $c$, and $ \max\{ \varpi_1 , \varpi_3 \} $ is used when the approximation bias becomes larger.  
Similarly, we define 
$M_c(\varpi_2, \varpi_4)\equiv \varpi_2 1\{\varpi_2 < c \}+ \max\{ \varpi_2 , \varpi_4 \} 1\{\varpi_2 \geq c \}$, 
and plot it in the column (d) of  Fig. \ref{fig:bias13n100} and Fig. \ref{fig:bias13n500}. 
%Then in each row of Fig. \ref{fig:bias13n100} and Fig. \ref{fig:bias13n500}, we plot $M_c(\varpi_1, \varpi_3)$ and $M_c(\varpi_2, \varpi_4)$ in columns (b) and (d), respectively.  

\begin{remark}
For each chi-squared approximation, $\max\{ \varpi_1 , \varpi_3 \}$  already characterizes the bias well most of the time.  
We use $M_c(\varpi_1, \varpi_3)$ instead of $\max\{ \varpi_1 , \varpi_3 \} $  because $\varpi_3$ can mistakenly indicate a large bias under small $\epsilon$, especially when $n$ is small.  
% which might happen when $n$ is small.   
Compared to  $\max\{ \varpi_1 , \varpi_3 \}$, 
$M_c(\varpi_1, \varpi_3)$ does not use $\varpi_3$ when $\varpi_1$ indicates that the bias is still small. 
%The definition of 
%$M_c(\varpi_1, \varpi_3)$ suggests 
% that $\varpi_3$ is not used when $\varpi_1$ indicates that the bias is still small.
%% ,  
%and chooses the larger one between $\varpi_1$ and $\varpi_3$ when $\varpi_1$ implies that the type-\RNum{1} error begins to inflate. 
As long as $c$ is sufficiently small but not too close to zero, 
%$\varpi_3$ is not used when $\varpi_1$ indicates that the bias is still small. 
%As then 
$M_c(\varpi_1, \varpi_3)$ will not take the wrong value given by $\varpi_3$, 
%large value of $\varpi_3 $, 
and thus gives a good evaluation of the approximation bias under a wide range of $\epsilon$ values. 
%$M_c(\varpi_1, \varpi_3)$ can give a good evaluation of the approximation bias under a wide range of $\epsilon$ values.  
%as shown in the simulations. 
Despite the difference between  $M_c(\varpi_1, \varpi_3)$ and  $\max\{ \varpi_1 , \varpi_3 \}$, we note that 
 $M_c(\varpi_1, \varpi_3)$ is equal to $\max\{ \varpi_1 , \varpi_3 \} $ under most cases.  
 For instance, in all our simulations with $n=500$ and $c=0.002$, $M_c(\varpi_1, \varpi_3)=\max\{ \varpi_1 , \varpi_3 \}$. Thus in  $\S$\,\ref{sec:simulations}, we did not highlight this difference. 
 When the Bartlett correction is used, we know that similar analysis applies to $\max\{ \varpi_2 , \varpi_4 \}$ and $M_c(\varpi_2, \varpi_4)$. 
\end{remark}

\subsubsection{Multiple-Sample Tests (IV)--(VI).}
Consider $k=3$, $n_1=n_2=n_3$, and $n=n_1+n_2+n_3$. 
Under the null hypothesis of each multiple-sample test (IV)--(VI), 
we set  $\BSM{\mu}_i = (0,\ldots,0)^{\MYTRAN}$, and $\BSM{\Sigma}_i = \BIEN$ for $i=1,2,3$. 

\smallskip

\noindent \textit{(1) On the phase transition boundaries.}\ Let $p = \lfloor n^\epsilon \rfloor$, where $n=n_1+n_2+n_3$ and $n_i\in \{100,500,1000,5000\}$ for $i=1,2,3$. 
We then plot the empirical type-\RNum{1} error rates (over 1000 replications) versus $\epsilon$ for each chi-squared approximation in Fig. \ref{fig:simphasemult}.

\smallskip

\noindent \textit{(2) On the asymptotic biases.} \  To evaluate the asymptotic biases in  Theorems \ref{thm:multisamchisq} and \ref{thm:multsamnormal},
we take  $p=\lfloor n^{\epsilon} \rfloor$, where $n=n_1+n_2+n_3$, $n_i\in \{100, 500\}$ for $i=1,2,3$, and $\epsilon \in (0,1)$. 
 The results of $n_i=100$ and $500$ (over 3000 replications) are given in Fig. \ref{fig:biasmultin100} and Fig. \ref{fig:biasmultin500}, respectively. 
Similarly to  Fig. \ref{fig:bias13n100} and Fig. \ref{fig:bias13n500}, in each row of Fig. \ref{fig:biasmultin100} and Fig. \ref{fig:biasmultin500}, 
the lines with dot markers in the four columns (a)--(d) give $\varpi_1$, $M_c(\varpi_1, \varpi_3)$, $\varpi_2$, and $M_c(\varpi_2, \varpi_4)$, respectively. 

%the results in the four columns have the same interpretation as those of Fig. \ref{fig:bias13n100} and Fig. \ref{fig:bias13n500}, that is, 
% are plotted same as 
%that is, results of each test  in Fig. \ref{fig:biasmultin100} and Fig. \ref{fig:biasmultin500}, 

%First, to examine the phase transition boundaries, we take $p = \lfloor n^\epsilon \rfloor$, where $n_1=n_2=n_3=n/3$ and  $n_i\in \{100,500,1000,5000\}$ for $i=1,2,3$. 
%We then plot the  type-\RNum{1} error rates versus $\epsilon$ for each chi-squared approximation in Fig. \ref{fig:simphasemult}.  

%For the multiple-sample tests (IV)--(VI), we take $k=3$, $n_1=n_2=n_3$, and $n_i\in \{100,500,1000,5000\}$, and set $\BSM{\mu}_i = (0,\ldots,0)^{\MYTRAN}$ and $\BSM{\Sigma}_i = \BIEN$ for $i=1,2,3$ under the null hypothesis. 

\subsubsection{Testing Independence between Multiple Tests (VII).} 
Consider $k=3$. 
Under the null hypothesis of test (VII), we set  $\BSM{\mu}= (0,\ldots,0)^{\MYTRAN}$ and $\BSM{\Sigma}=\BIEN$. 

\smallskip

\noindent \textit{(1) On the phase transition boundaries.}\  
%First, to examine the phase transition boundaries,
Let  $p = \lfloor n^\epsilon \rfloor$, where $\epsilon \in\{ 6/24, 7/24, \ldots, 23/24 \}$ and $n\in \{100,500,1000,5000\}$. 
Under each $(n,p)$, we set $p_1=p_2=\lfloor p/3 \rfloor$ and $p_3=p-p_1-p_2$, 
 and then plot the empirical type I error (over 1000 replications) versus $\epsilon$ in Fig. \ref{fig:simphasemult}. 

\smallskip

\noindent \textit{(2) On the asymptotic biases.}\ To evaluate the asymptotic biases in  Theorems \ref{thm:chisqindp} and \ref{thm:indepnormal}, 
we  set $p = \lfloor n^\epsilon \rfloor$, where  $n\in \{100,500\}$ and $\epsilon \in (0,1)$. 
Under each $(n,p)$, we take $p_1=p_2=\lfloor p/3 \rfloor$ and $p_3=p-p_1-p_2$. 
 The results of $n=100$ and $500$ (over 3000 replications) are given in Fig.  \ref{fig:bias7n100} and Fig. \ref{fig:bias7n500}, respectively. 
Similarly to  Figures \ref{fig:bias13n100}--\ref{fig:biasmultin500}, in Fig. \ref{fig:bias7n100} and Fig. \ref{fig:bias7n500}, 
the lines  with dot markers in the four columns (a)--(d) give $\varpi_1$, $M_c(\varpi_1, \varpi_3)$, $\varpi_2$, and $M_c(\varpi_2, \varpi_4)$, respectively.

\medskip

We next analyze the simulation results. First, as shown in Figures \ref{fig:simone} and \ref{fig:simphasemult}, 
 the theoretical phase transition boundary, denoted by a vertical line, is observed to be consistent with where each chi-squared approximation starts to fail. 
 For instance, the two plots in the first row of Fig. \ref{fig:simone} show that for test (I), the type-\RNum{1} error rates of the chi-squared approximations without and with the Bartlett correction begin to inflate when $\epsilon$ is around $2/3$ and $4/5$, respectively. 
 These are consistent with $d_1=2/3$ and $d_2=4/5$ for test (I) in Theorem \ref{thm:onesam}. 
 Similarly for other tests, we can see that the numerical results are also consistent with the corresponding conclusions in Theorems \ref{thm:onesam}, \ref{thm:allmult}, and \ref{Thm2.1}.

Second, similarly to $\S$\,\ref{sec:simulations}, the results in Figures \ref{fig:bias13n100}--\ref{fig:bias7n500}  show that the derived theoretical asymptotic biases provide good evaluations of the corresponding chi-squared approximation  biases.  
From the subfigures in the column (a) of Figures \ref{fig:bias13n100}--\ref{fig:bias7n500}, 
we can see that 
as $\epsilon$ increases, the empirical type-\RNum{1} error inflates, 
and $\varpi_1$
%, the asymptotic bias  in \eqref{eq:chisqapprox}, 
also increases accordingly. 
At the $\epsilon$ values where the  type-\RNum{1} error begins to inflate,  
the difference between the empirical type-\RNum{1} error and  $\varpi_1$ is close to 0.05, as shown by the circle line, 
which suggests that $\varpi_1$
%, i.e., the asymptotic bias in \eqref{eq:chisqapprox},  
approximates the chi-squared approximation bias $\MYPR\{-2 \log \Lambda_{n}>\chi^2_{f}(\alpha)\}-\alpha$ well in this regime.   
 When $\epsilon$ further increases beyond the corresponding phase transition boundary, 
 the asymptotic bias $\varpi_1$ keeps increasing, and its large value indicates the failure of the chi-squared approximation, even though now  $\varpi_1$ underestimates  the approximation bias in this regime.   
To better characterize the approximation bias when $\epsilon$ is beyond the phase transition boundary, we combine $\varpi_1$ and $\varpi_3$  
by plotting $M_c(\varpi_1, \varpi_3)$ in the column (b) of Figures \ref{fig:bias13n100}--\ref{fig:bias7n500}. 
%Moreover,  the column (b) of Figures \ref{fig:bias13n100}--\ref{fig:bias7n500}  
The results suggest that 
%taking the maximum over two asymptotic biases in \eqref{eq:chisqapprox} and in \eqref{eq:normalbias1} 
utilizing the two asymptotic biases in \eqref{eq:chisqapprox} and in \eqref{eq:normalbias1} together can 
give a good evaluation of the approximation bias under a wide range of  $\epsilon$ values, either below or above the phase transition boundary. 
%when $\epsilon$ is large, which is consistent with the theoretical results.  
%Within subfigures that give $M_c(\varpi_1, \varpi_3)$, i.e., those in the column (b) of Figures \ref{fig:bias13n100}--\ref{fig:bias7n500}, 
% of Figures \ref{fig:bias13n100}--\ref{fig:bias7n500}, 
Moreover, in each subfigure in the column (b),
we also highlight the location with $x$-axis $\epsilon^*$  where  $M_c(\varpi_1, \varpi_3)$ starts to be larger than $\varpi_1$ (the plus sign).
When $\epsilon < \epsilon^*$, $M_c(\varpi_1, \varpi_3)=\varpi_1$, 
indicating that $\varpi_1$ approximates the bias better than $\varpi_3$ does in this regime, 
while  $\varpi_3$ performs better than $\varpi_1$ when $\epsilon \geq \epsilon^*$. 
%When $\epsilon \geq \epsilon^*$,
%although $\varpi_3$ performs better than $\varpi_1$, 
% $M_c(\varpi_1, \varpi_3)=\varpi_3$, suggesting that $\varpi_3$ performs better than $\varpi_1$, 
%whereas we note that under this regime, 
%However, we note that when $\epsilon \geq \epsilon^*$, 
%the approximation bias can be large so that $\varpi_1$ can also tell the failure of the chi-squared approximation. 
%%is usually already large. 
%%we note that 
%%the empirical type-\RNum{1} already begins to inflate, and the approximation bias is usually large.  
%For instance, in the first row and column (b) of Fig. \ref{fig:bias13n100},
%although  $\varpi_3$  approximates
%On the other hand, $\varpi_3$ performs better than $\varpi_3$ when $\epsilon \geq \epsilon^*$, where the approximation bias is relatively large. 
%By construction, $M_c(\varpi_1, \varpi_3)>\varpi_1$ indicates that $\varpi_3$ dominates 
%, indicating $\varpi$
%with $x$-axis $\epsilon^*$ satisfying 
Similarly,  for the chi-squared approximation with the Bartlett correction, 
similar conclusions can be obtained  by the results in the columns (c) and (d) of Figures \ref{fig:bias13n100}--\ref{fig:bias7n500}.

%the column (c) of Figures \ref{fig:bias13n100}--\ref{fig:bias7n500} shows that 
%$\varpi_3$, i.e., the asymptotic bias in \eqref{eq:chisqapproxbartcorr}, can approximate the empirical bias
% well. 
%%increases when the empirical type-\RNum{1} error rates inflate.  
%In addition, the column (d) of Figures \ref{fig:bias13n100}--\ref{fig:bias7n500} also shows that  taking the maximum over two asymptotic biases in \eqref{eq:chisqapproxbartcorr} and in \eqref{eq:normalbias2} can 
%give a better characterization of the approximation bias when $\epsilon$ is large, as expected by the theoretical results. 

The simulations under the finite sample suggest that the derived asymptotic biases 
%provide good characterization of the approximation biases,
%and therefore 
can be used as practical guidelines for the considered likelihood ratio tests. 
%when applying Wilk's theorem.    
Specifically, when using the chi-squared approximation in each test, 
similarly to our recommendation in $\S$\,\ref{sec:simulations}, 
the  practitioners can compare the asymptotic bias, either $\varpi_1$ or $M_c(\varpi_1, \varpi_3)$, with a small threshold value that they may specify in advance, e.g., 0.01--0.02.    
If the asymptotic bias is larger than the threshold, the chi-squared approximation should not be directly used, and other methods would be needed.  
In addition, when using the chi-squared approximation with the Bartlett correction  in each test, 
we can compare the asymptotic bias, either $\varpi_2$ or $M_c(\varpi_2, \varpi_4)$ with the pre-specified threshold value.
Similarly, if the asymptotic bias is larger than the threshold, the chi-squared approximation with the Bartlett correction should not be directly applied, and other methods would be needed.

\newpage

\begin{figure}[!htbp]
\captionsetup[subfigure]{labelformat=empty}
\centering

\begin{turn}{90}
\begin{minipage}{0.26\textwidth}
 \hspace{6em} Test (I)\vspace{0.12em} 
\end{minipage}
\end{turn}%
~ 
\begin{subfigure}[t]{0.38\textwidth}
\centering
%\caption{(I.i) without correction}
\includegraphics[width=\textwidth]{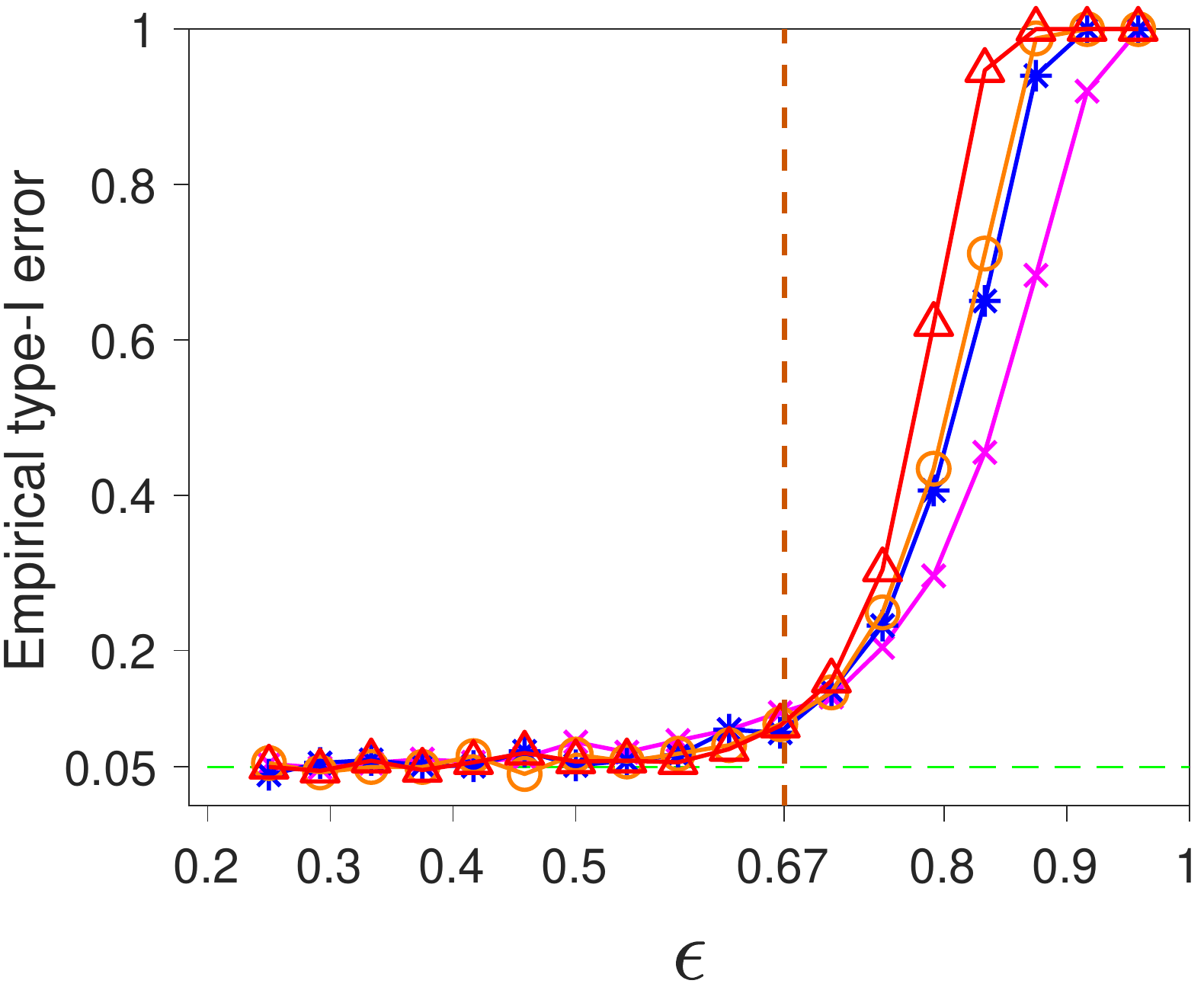}
\end{subfigure}%
~ \ 
\begin{subfigure}[t]{0.38\textwidth}
\centering
%\caption{(I.ii) with correction}		
\includegraphics[width=\textwidth]{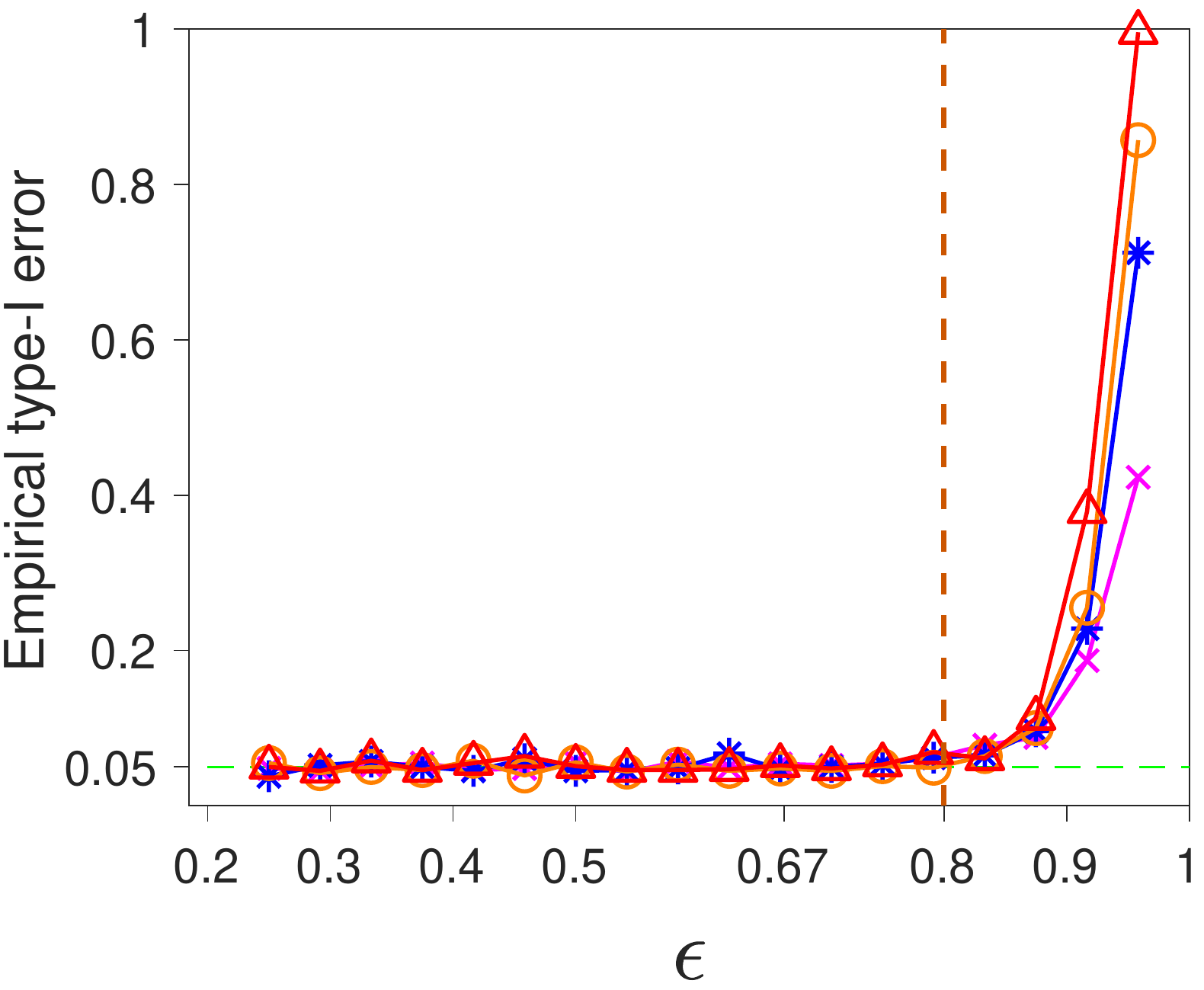}
%\caption{(I.ii) with correction}		
\end{subfigure}

\quad \\

\centering
\begin{turn}{90}
\begin{minipage}{0.26\textwidth}
 \hspace{6em} Test (II)\vspace{0.12em} 
\end{minipage}
\end{turn}%
~ \
\begin{subfigure}[t]{0.38\textwidth}
\centering
\includegraphics[width=\textwidth]{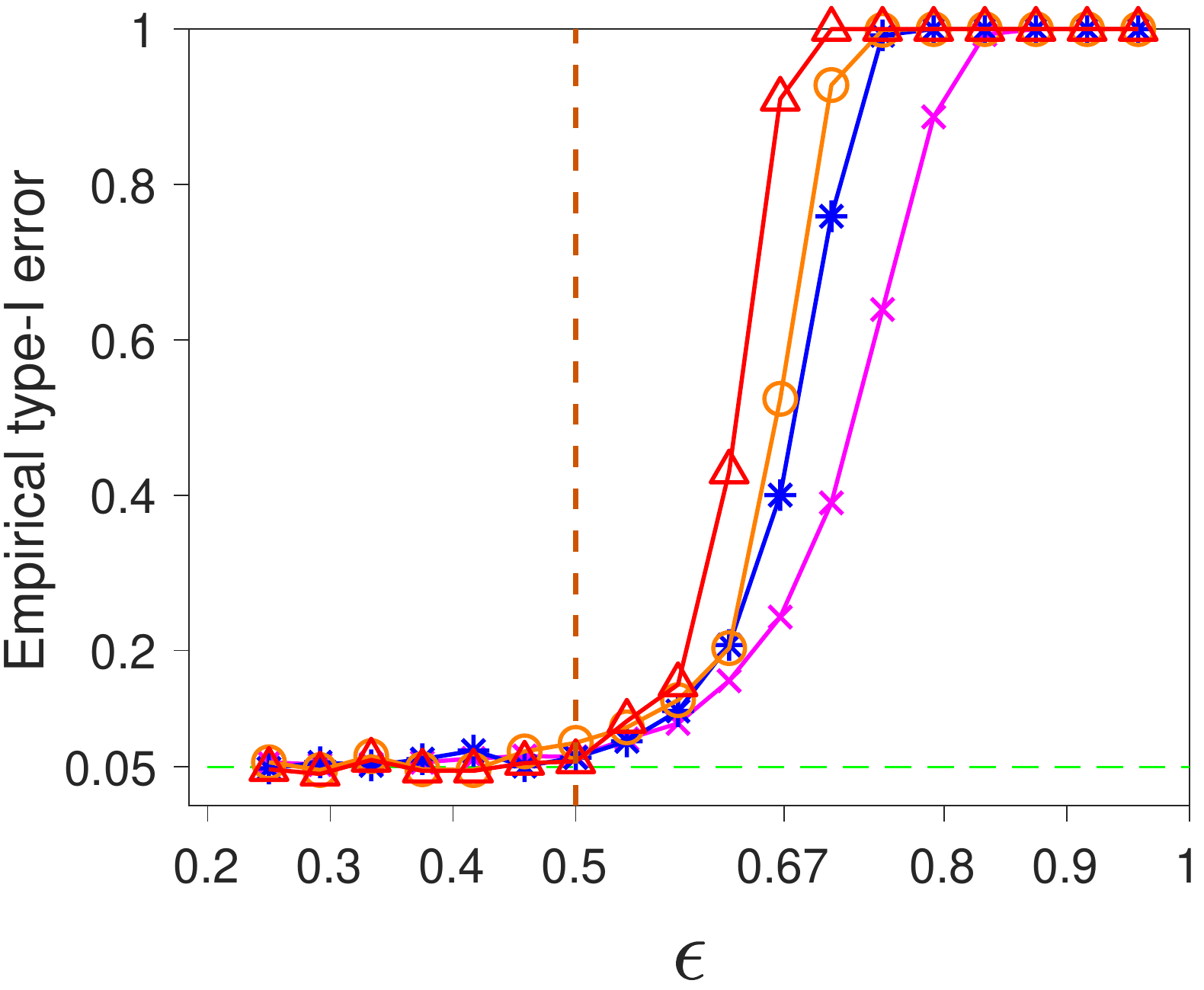}
%\caption{(II.i) without correction}	
\end{subfigure}%
~\ 
\begin{subfigure}[t]{0.38\textwidth}
\centering
\includegraphics[width=\textwidth]{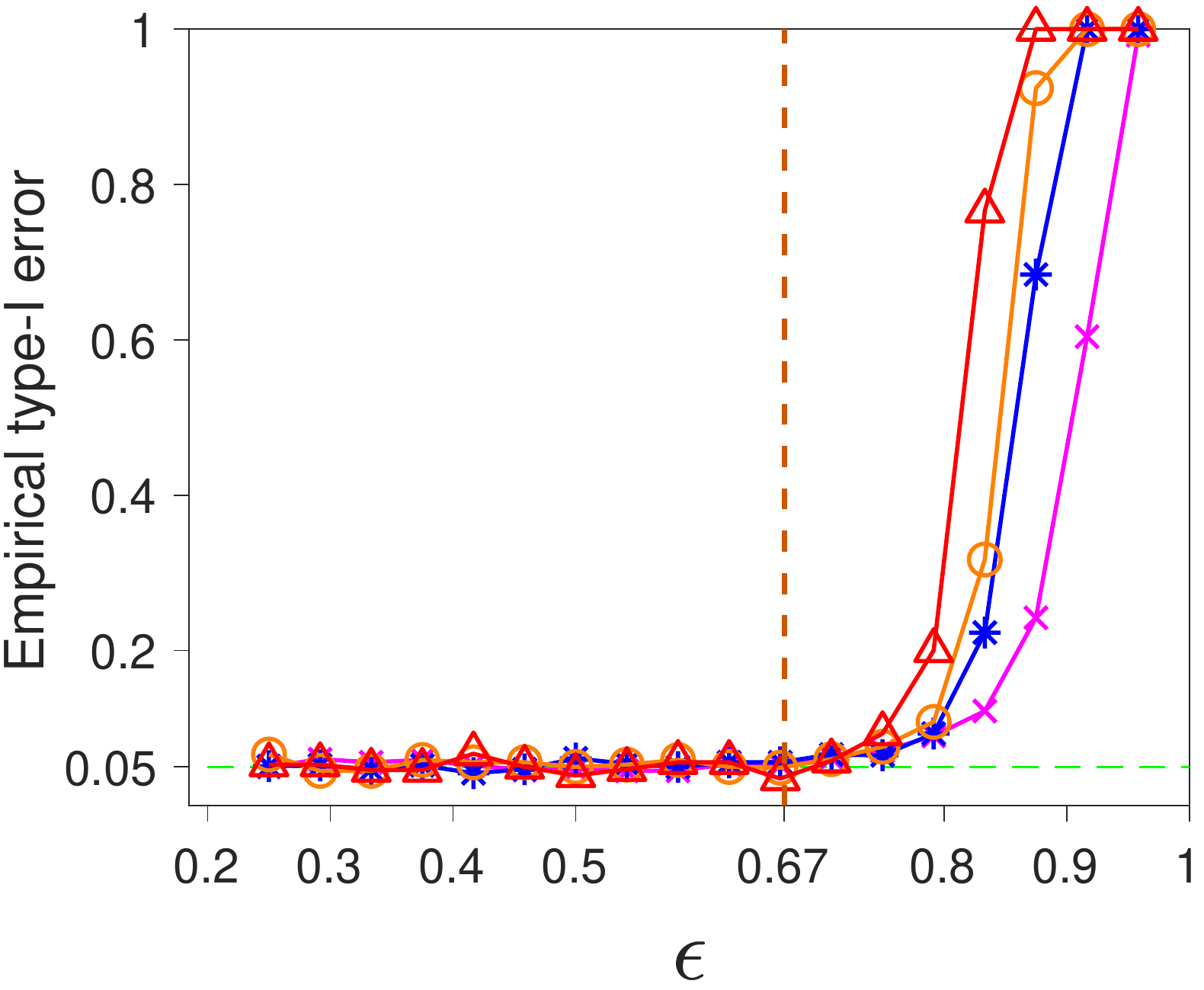}
%\caption{(III.ii) with correction}		
\end{subfigure}

\quad \\

\centering
\begin{turn}{90}
\begin{minipage}{0.26\textwidth}
 \hspace{6em} Test (III)\vspace{0.12em} 
\end{minipage}
\end{turn}%
~ \
\begin{subfigure}[t]{0.38\textwidth}
\centering
\includegraphics[width=\textwidth]{IIIJoint1.pdf}
\caption{\quad (i) Without the Bartlett correction}		
\end{subfigure}%
~ \ 
\begin{subfigure}[t]{0.38\textwidth}
\centering
\includegraphics[width=\textwidth]{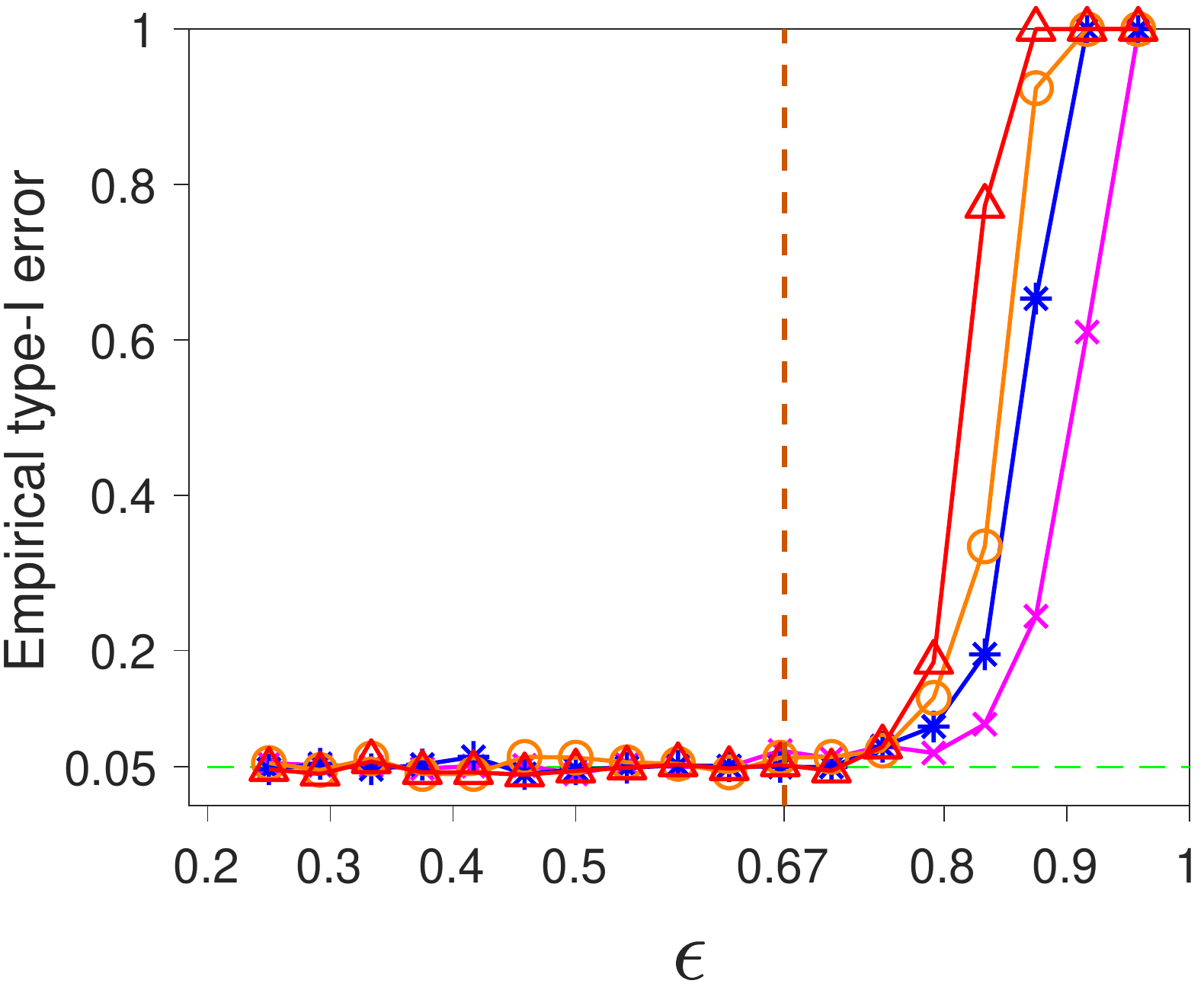}
\caption{\quad (ii) With the Bartlett correction correction}	
\end{subfigure}

\caption{One-sample tests (I)--(III). Rows 1-3 give the results for tests (I)--(III), respectively.  Columns (i) and (ii) correspond to the chi-squared approximations without and with the Bartlett correction, respectively.  Within each subfigure: empirical type-\RNum{1} error versus $\epsilon$ with $n=100$ (cross), $500$ (asterisk), $1000$ (square), and $5000$ (triangle);   theoretical phase transition boundary (vertical dashed line).}\label{fig:simone}
\end{figure}
	
%\quad		

%\newpage
%\quad 
%\vspace{1em}
%\newpage
%\quad \\
%\begin{minipage}[c][0.3em][c]{0.2\textwidth}
%	\quad
%\end{minipage}
%\begin{minipage}[c][0.3em][c]{0.35\textwidth}
%\hspace{-4em}	(i) Without the Bartlett correction
%\end{minipage}
%\begin{minipage}[c][0.3em][c]{0.3\textwidth}
%\hspace{-5.4em} (ii) With the Bartlett correction
%\end{minipage}
%\begin{minipage}[c][0.3em][c]{0.38\textwidth}
%	\quad 
%\end{minipage}
%
%\quad  
%
%\vspace{-3em} \quad
%
\begin{figure}[!htbp]
\captionsetup[subfigure]{labelformat=empty}
\centering

\begin{turn}{90}
\begin{minipage}{0.26\textwidth}
 \hspace{5em} Test (IV)\vspace{0.12em} 
\end{minipage}
\end{turn}%
~ 
\begin{subfigure}[t]{0.38\textwidth}
\centering
%\caption{(I.i) without correction}
\includegraphics[width=\textwidth]{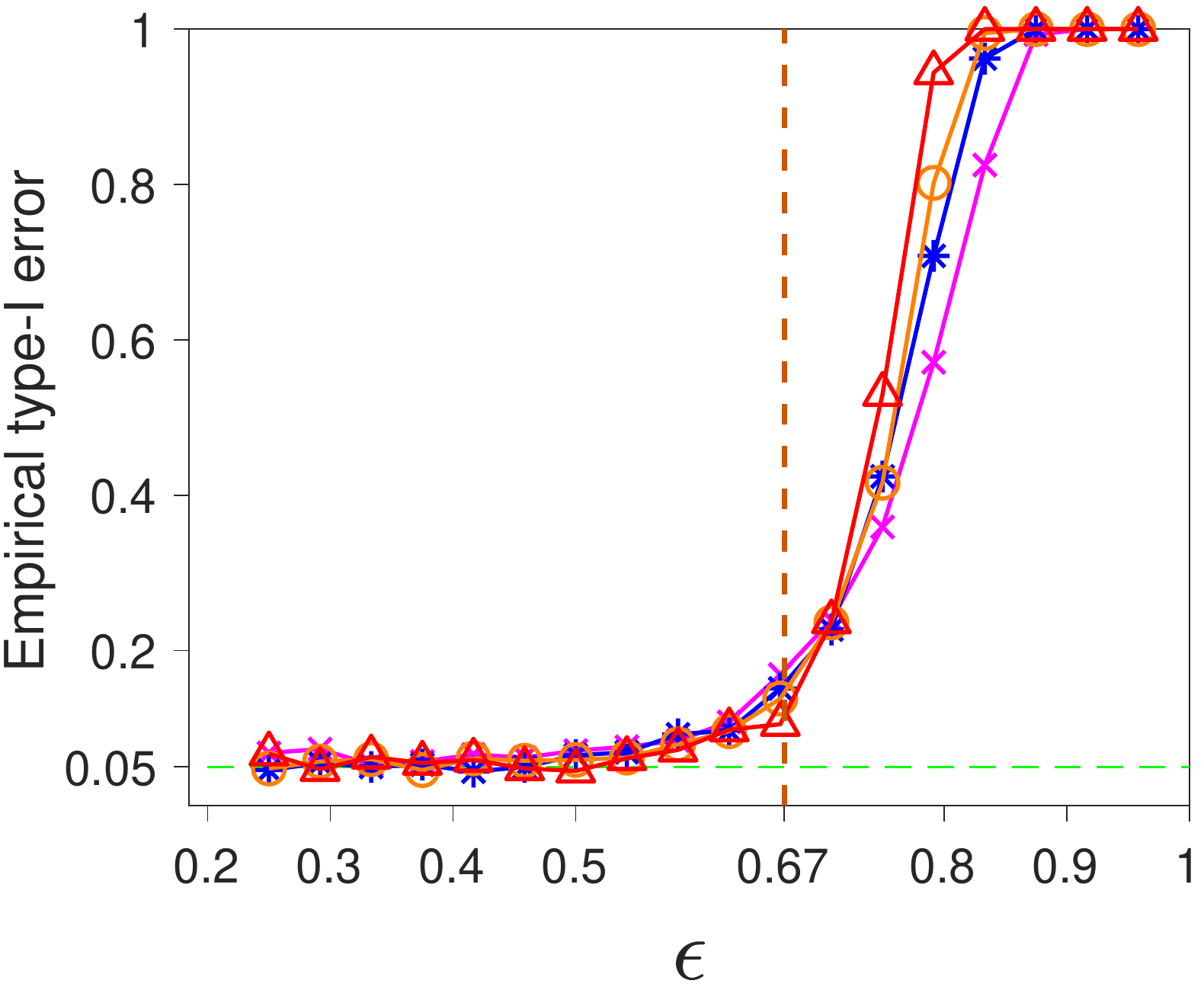}
\end{subfigure}%
~ \ 
\begin{subfigure}[t]{0.38\textwidth}
\centering
%\caption{(I.ii) with correction}		
\includegraphics[width=\textwidth]{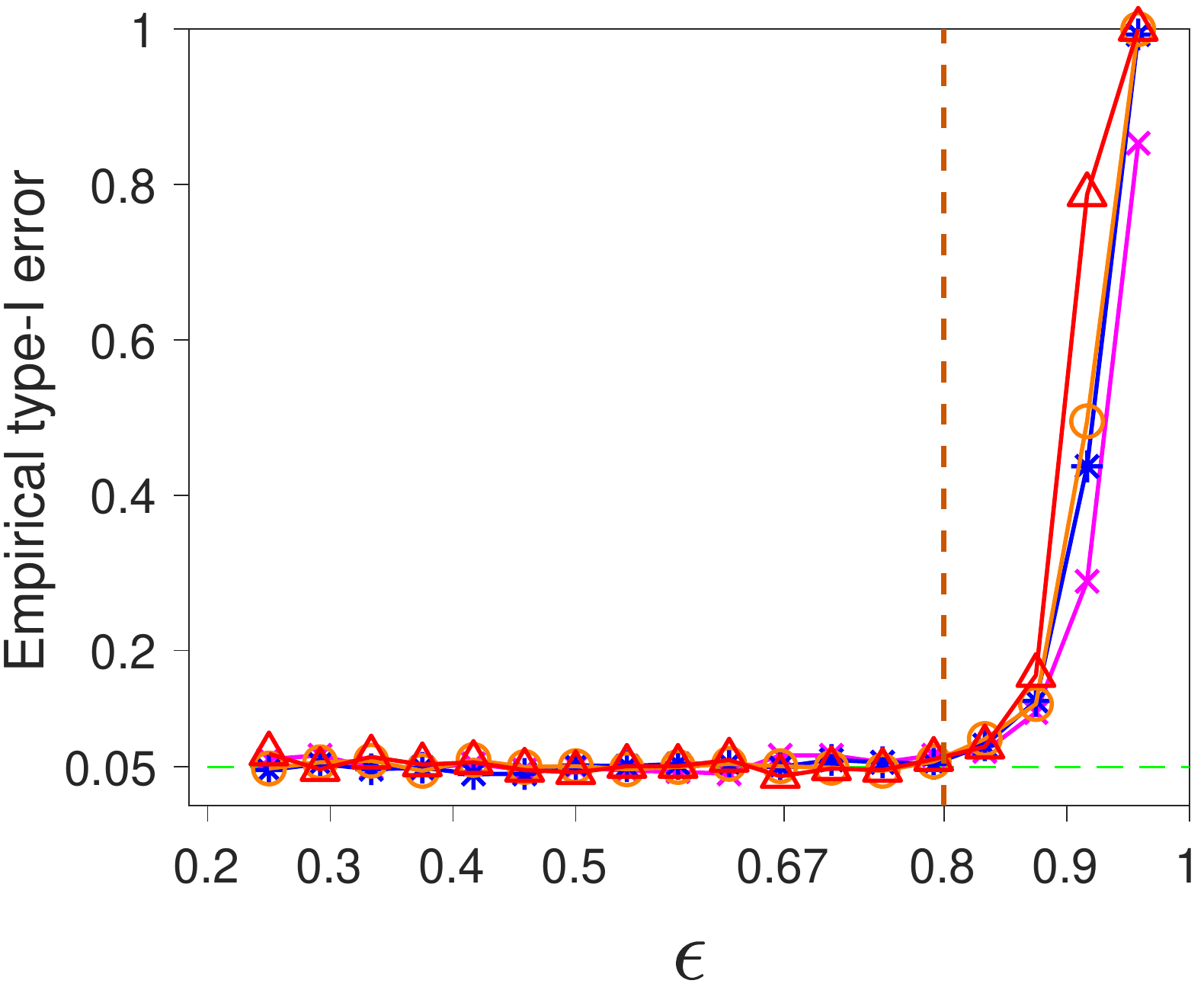}
%\caption{(I.ii) with correction}		
\end{subfigure}

\quad \\

\centering
\begin{turn}{90}
\begin{minipage}{0.26\textwidth}
 \hspace{6em} Test (V)\vspace{0.12em} 
\end{minipage}
\end{turn}%
~ \
\begin{subfigure}[t]{0.38\textwidth}
\centering
\includegraphics[width=\textwidth]{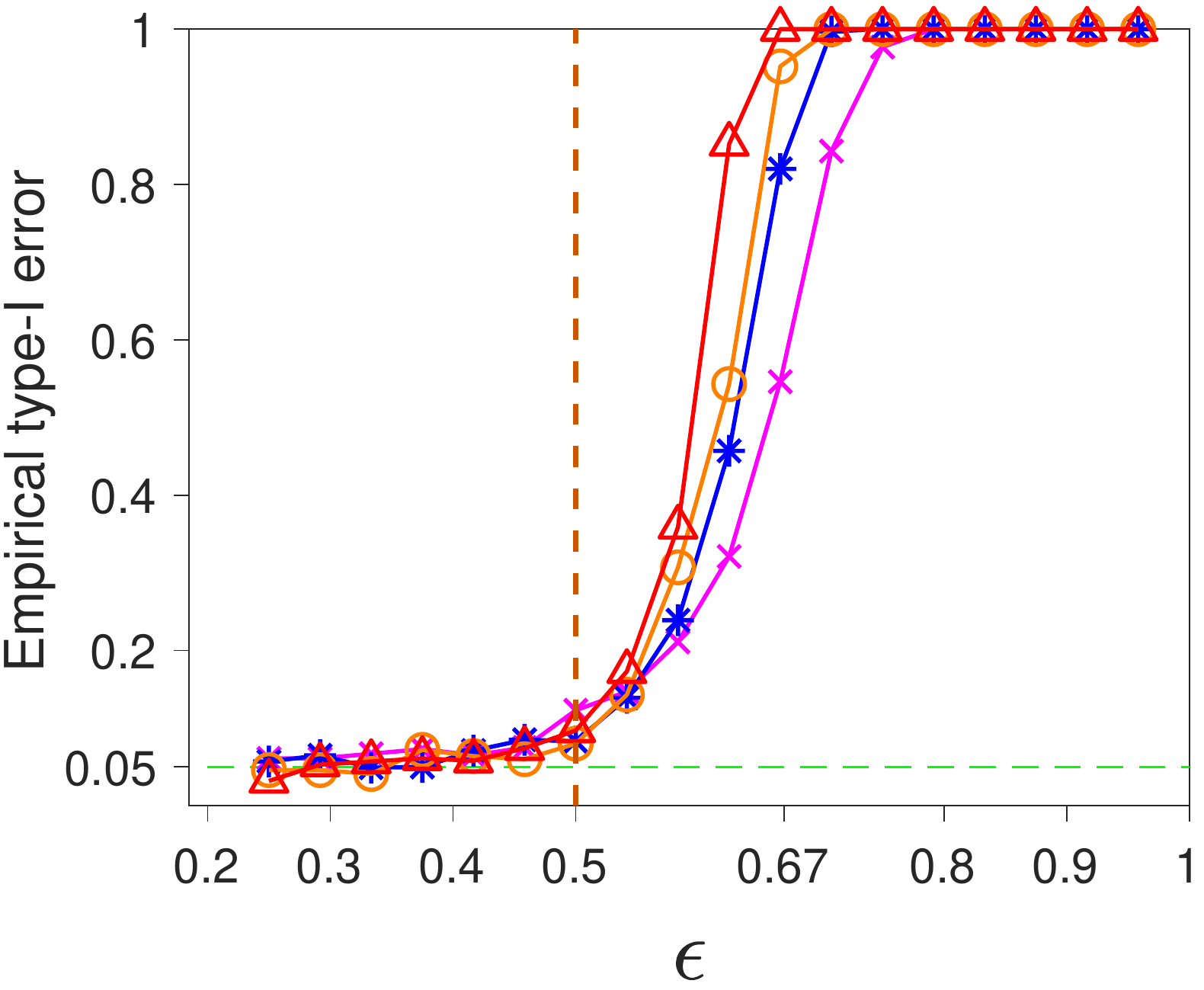}
%\caption{(II.i) without correction}	
\end{subfigure}%
~\ 
\begin{subfigure}[t]{0.38\textwidth}
\centering
\includegraphics[width=\textwidth]{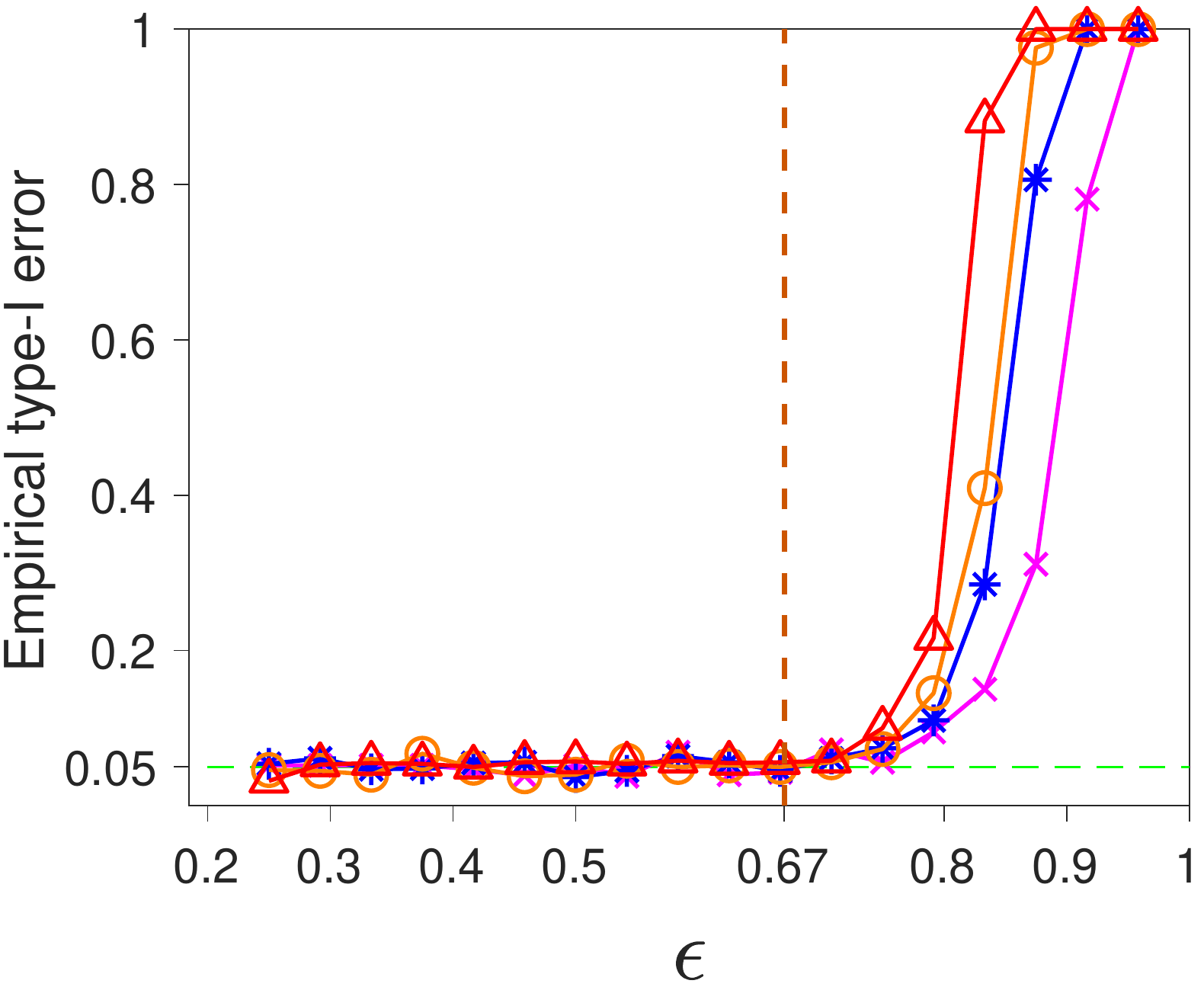}
%\caption{(III.ii) with correction}		
\end{subfigure}

\quad \\

\centering
\begin{turn}{90}
\begin{minipage}{0.26\textwidth}
 \hspace{6em} Test (VI)\vspace{0.12em} 
\end{minipage}
\end{turn}%
~ \
\begin{subfigure}[t]{0.38\textwidth}
\centering
\includegraphics[width=\textwidth]{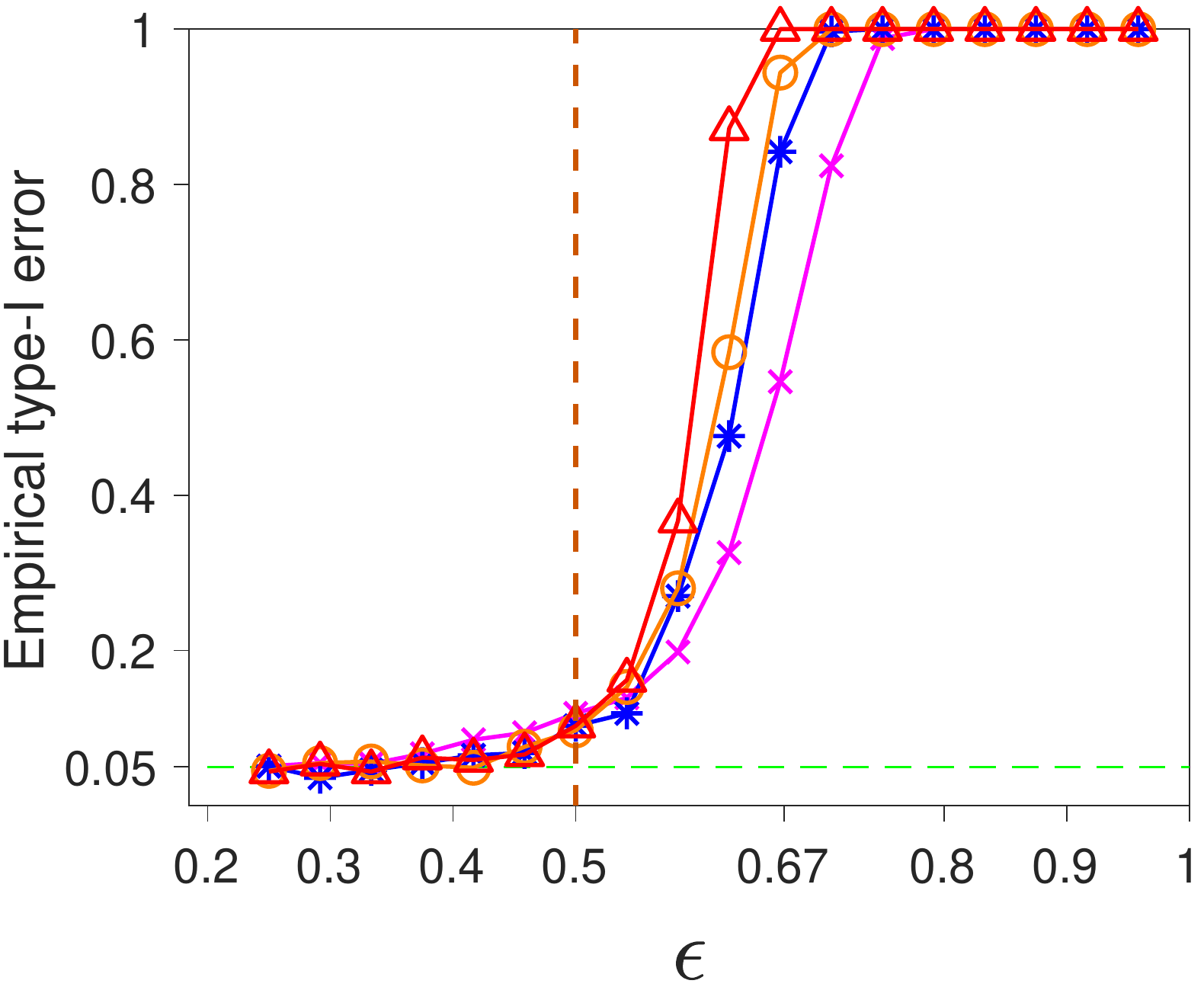}
%\caption{(i) Without the Bartlett correction}		
\end{subfigure}%
~ \ 
\begin{subfigure}[t]{0.38\textwidth}
\centering
\includegraphics[width=\textwidth]{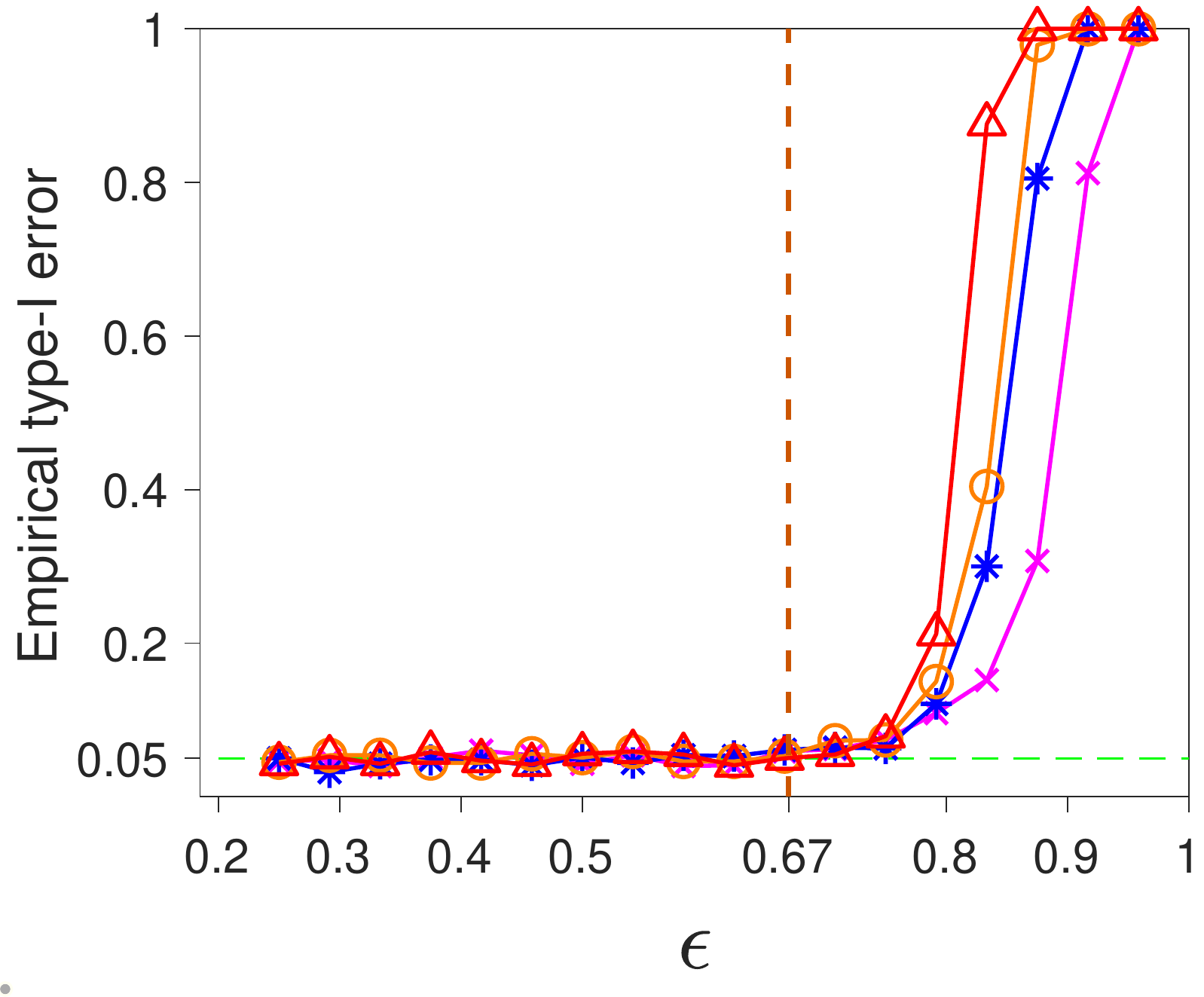}
%\caption{(ii) With the Bartlett correction}	
\end{subfigure}

\quad \\

\centering
\begin{turn}{90}
\begin{minipage}{0.26\textwidth}
 \hspace{6em} Test (VII)\vspace{0.12em} 
\end{minipage}
\end{turn}%
~ \
\begin{subfigure}[t]{0.38\textwidth}
\centering
\includegraphics[width=\textwidth]{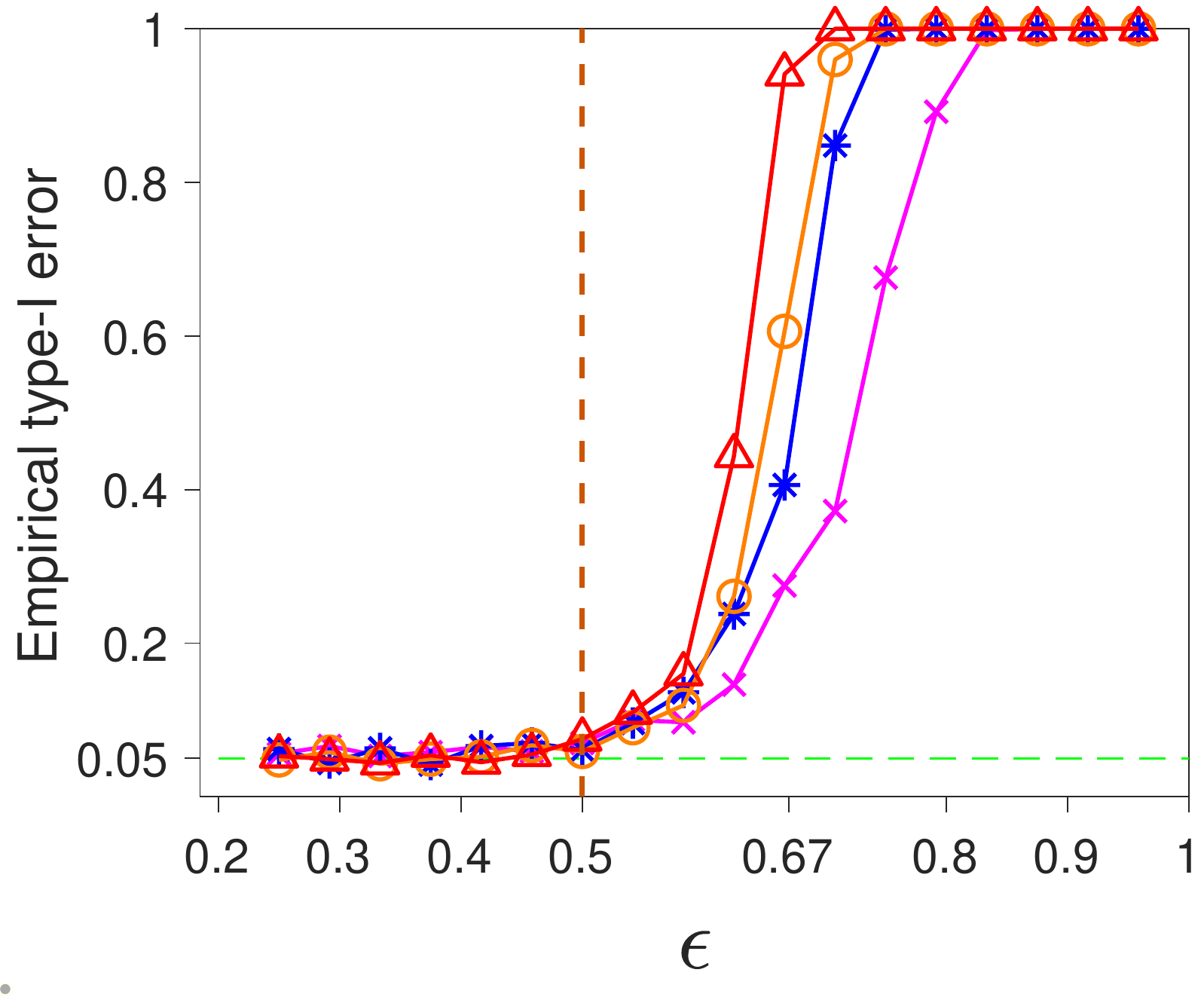}
\centerline{\quad \quad  (i) Without the Bartlett correction}	
\end{subfigure}%
~
\begin{subfigure}[t]{0.38\textwidth}
\centering
\includegraphics[width=\textwidth]{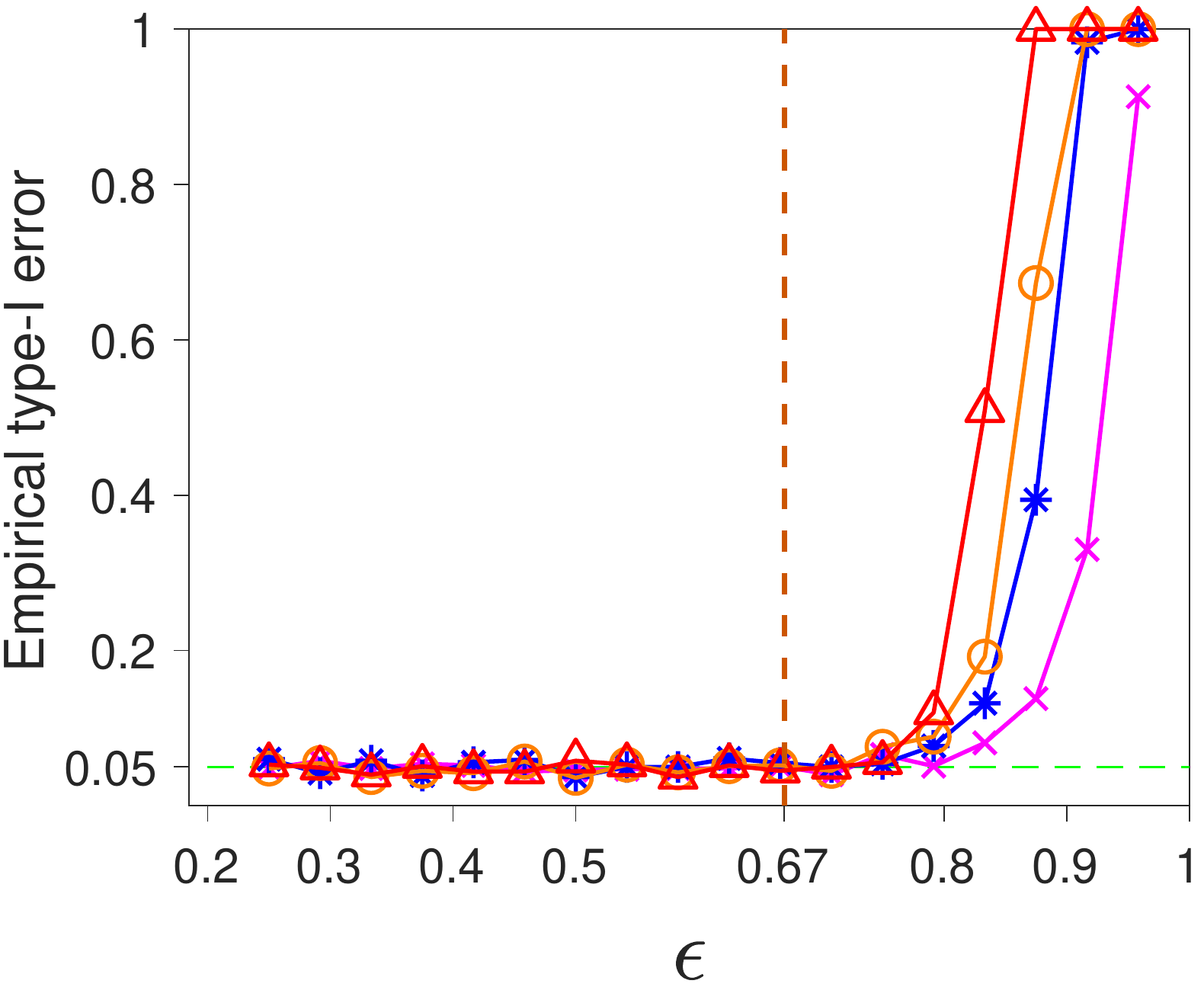}
\centerline{\quad \quad (ii) With the Bartlett correction}	
%\caption{(b)}	
%\centerline{(b)}
\end{subfigure}

\caption{Multiple-sample tests (IV)--(VI) and the independence test (VII). Rows 1-4 give results for tests (IV)--(VII), respectively. Columns (i) and (ii) are for the chi-squared approximations without and with the Bartlett correction, respectively. Within each subfigure, please see the caption description in Fig. \ref{fig:simone}. } \label{fig:simphasemult}
%\caption{Multiple-sample tests (IV)--(VI): Empirical type-\RNum{1} error versus $\epsilon$ with $n=100$ (cross), $500$ (asterisk), $1000$ (square), and $5000$ (triangle);   theoretical phase transition boundary (vertical dashed line).}\label{fig:simone}
\end{figure}

\begin{landscape}
\begin{figure}[!htbp]
\captionsetup[subfigure]{labelformat=empty}
\centering
\begin{turn}{90}
\begin{minipage}{0.26\textwidth}
 \hspace{4.5em} Test (I)\vspace{0.12em} 
\end{minipage}
\end{turn}%
~ \
\begin{subfigure}[t]{0.34\textwidth}
\centering
\includegraphics[width=\textwidth]{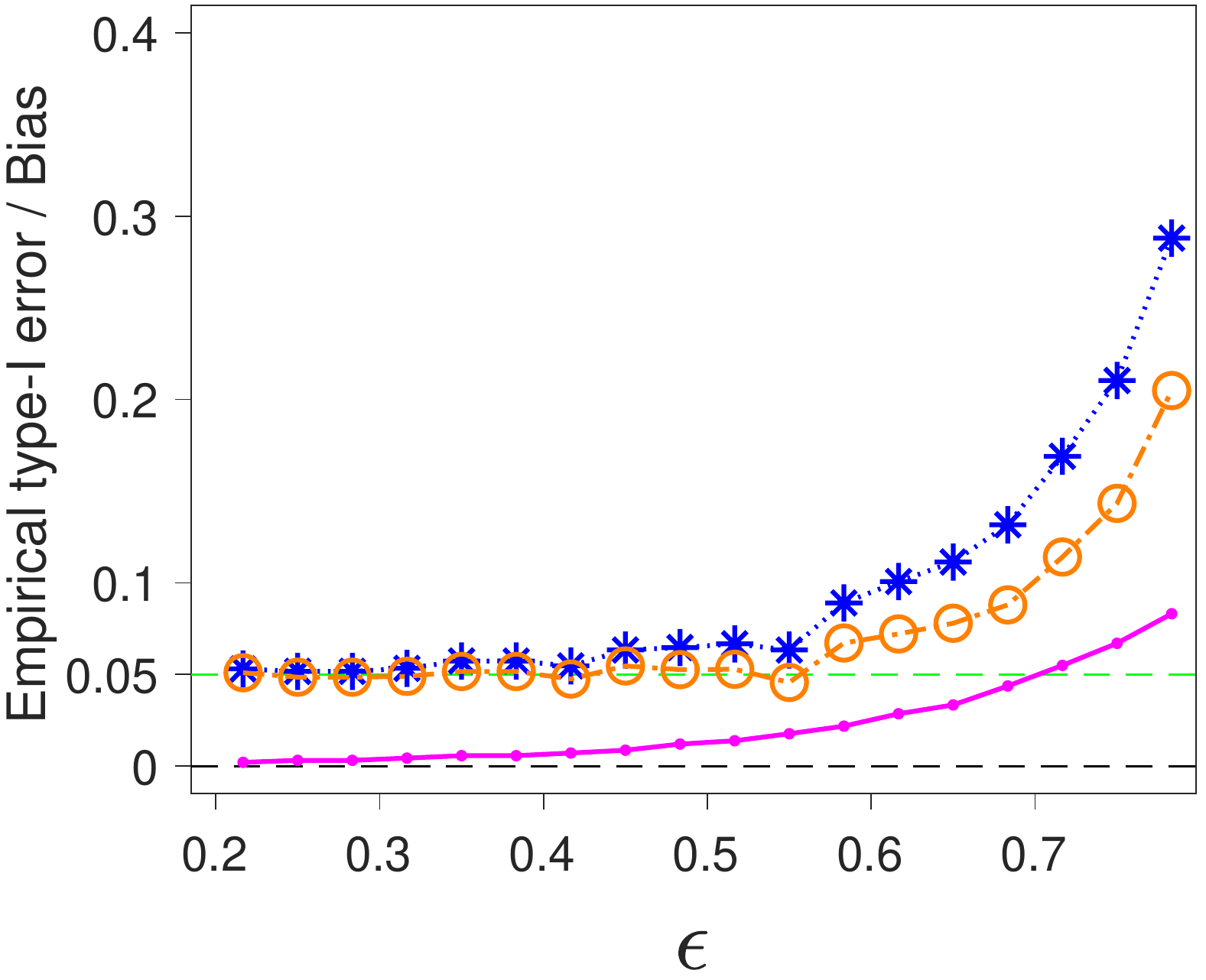}
%\centerline{(a)}
\end{subfigure}%
~ \ 
\begin{subfigure}[t]{0.34\textwidth}
\centering
\includegraphics[width=\textwidth]{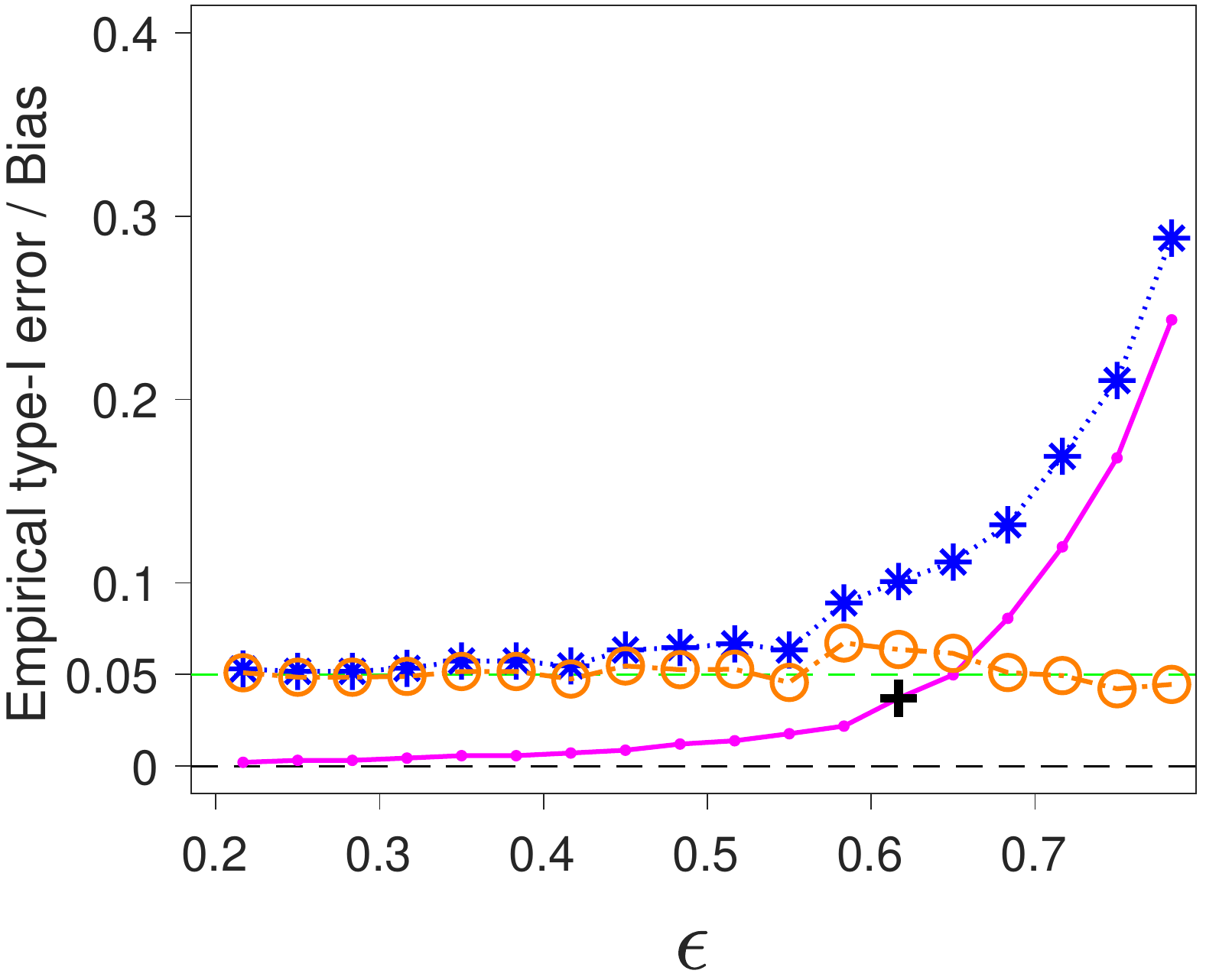}
%\centerline{(b)}
\end{subfigure}
~ \ 
\begin{subfigure}[t]{0.34\textwidth}
\centering
\includegraphics[width=\textwidth]{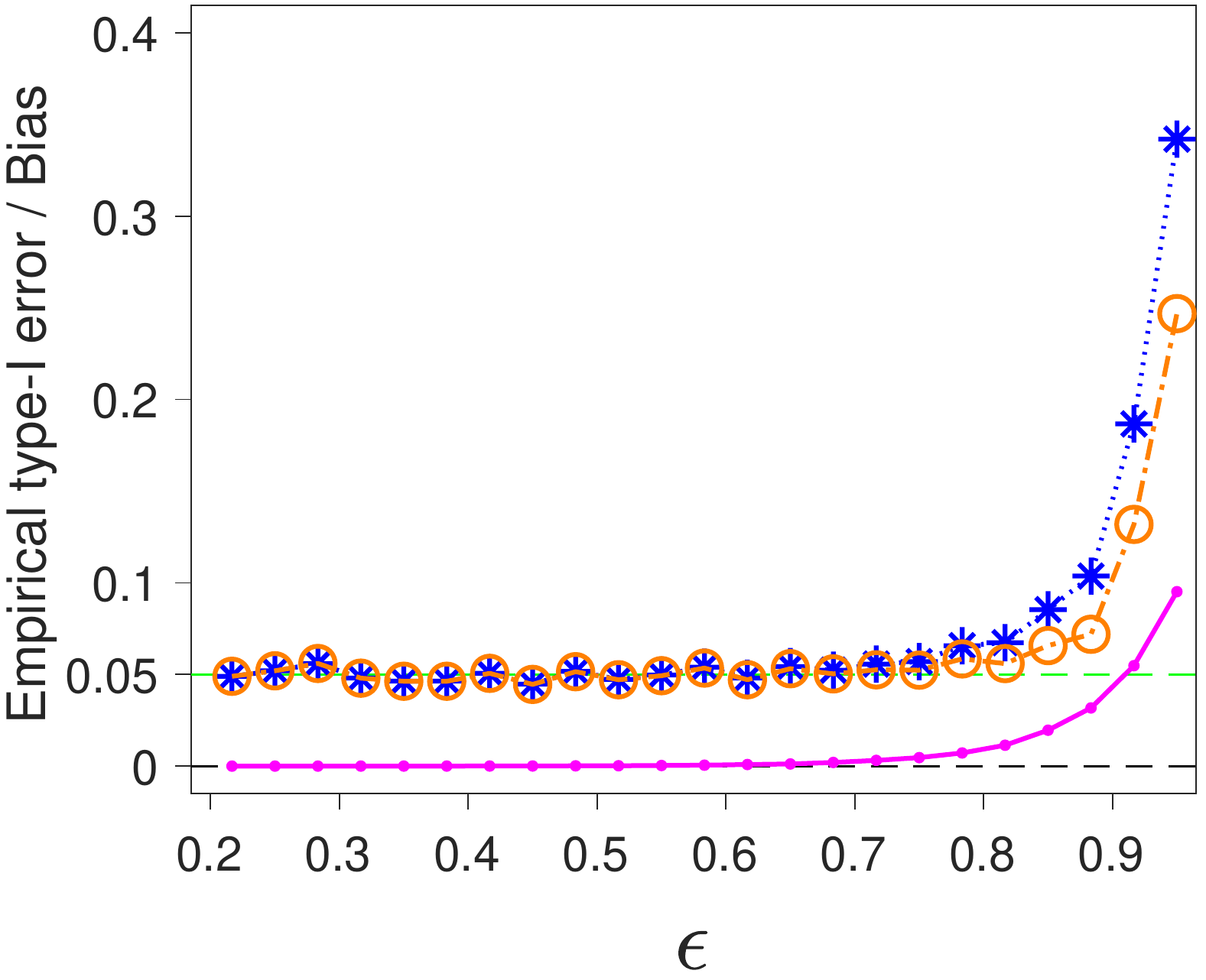}
%\centerline{(c)}
\end{subfigure}%
~ \ 
\begin{subfigure}[t]{0.34\textwidth}
\centering
\includegraphics[width=\textwidth]{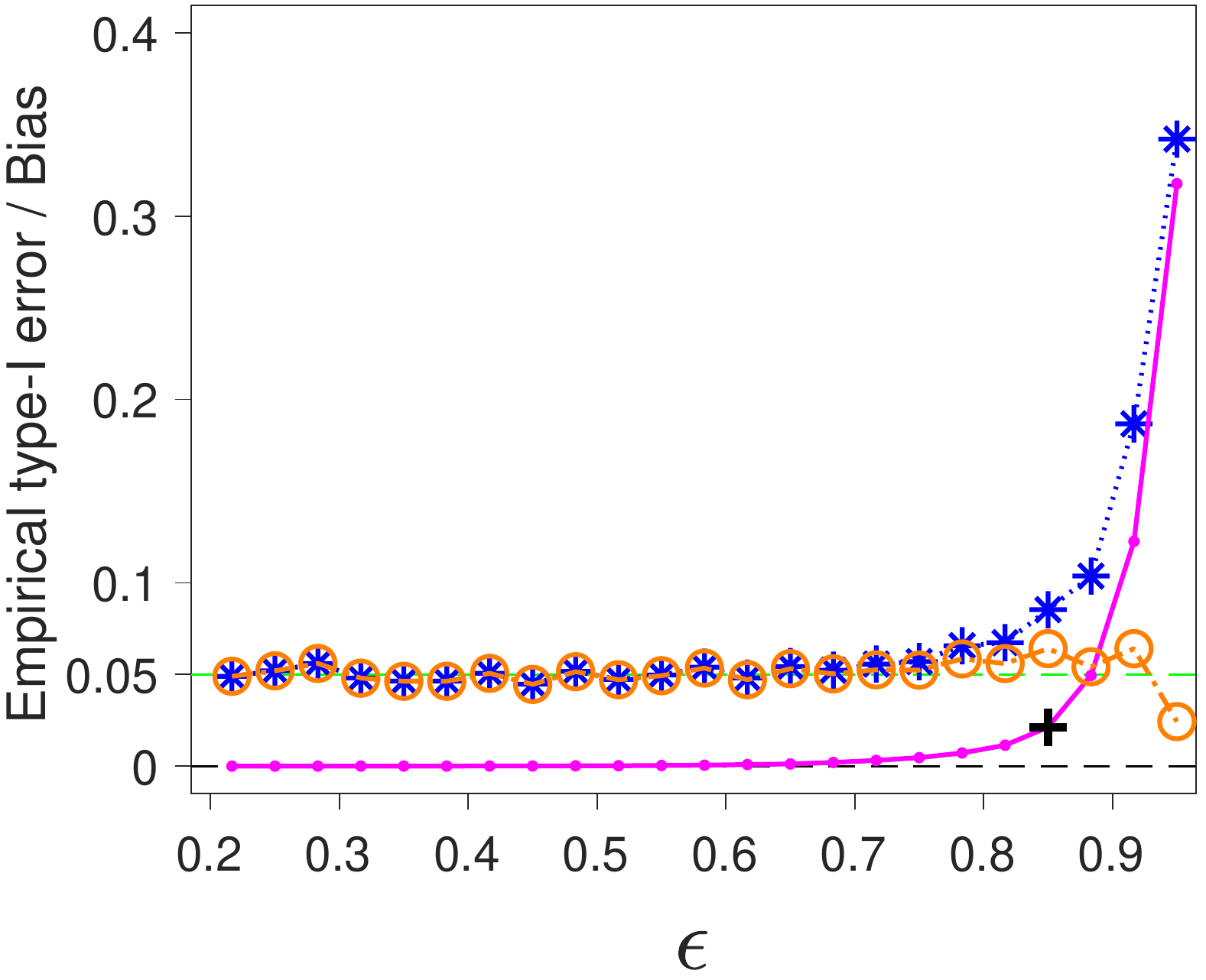}
%\centerline{(d)}
\end{subfigure}
\ \vspace{0.2em}

\begin{turn}{90}
\begin{minipage}{0.26\textwidth}
 \hspace{4.5em} Test (II)\vspace{0.12em} 
\end{minipage}
\end{turn}%
~ \
\begin{subfigure}[t]{0.34\textwidth}
\centering
\includegraphics[width=\textwidth]{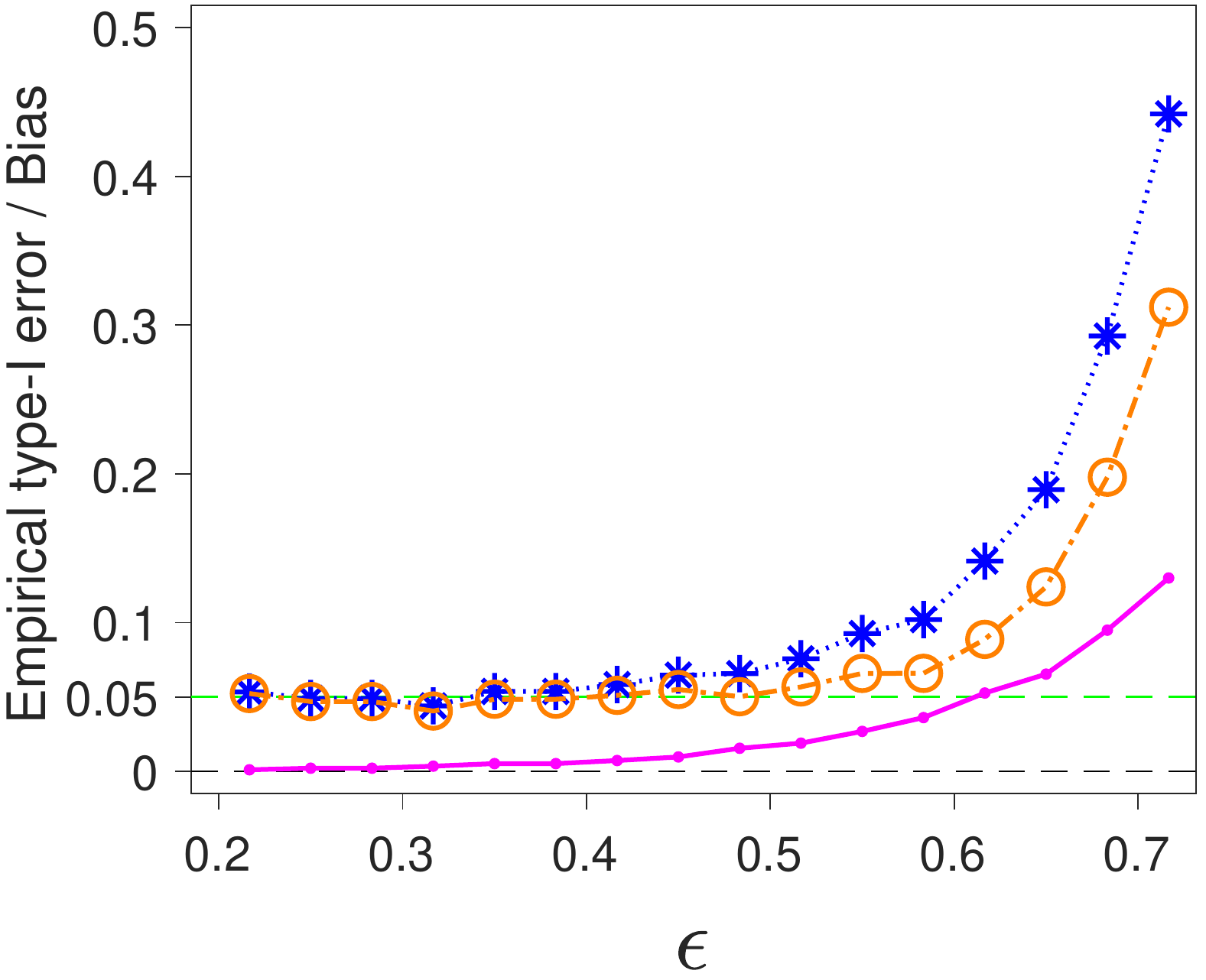}
%\centerline{(a)}
\end{subfigure}%
~ \ 
\begin{subfigure}[t]{0.34\textwidth}
\centering
\includegraphics[width=\textwidth]{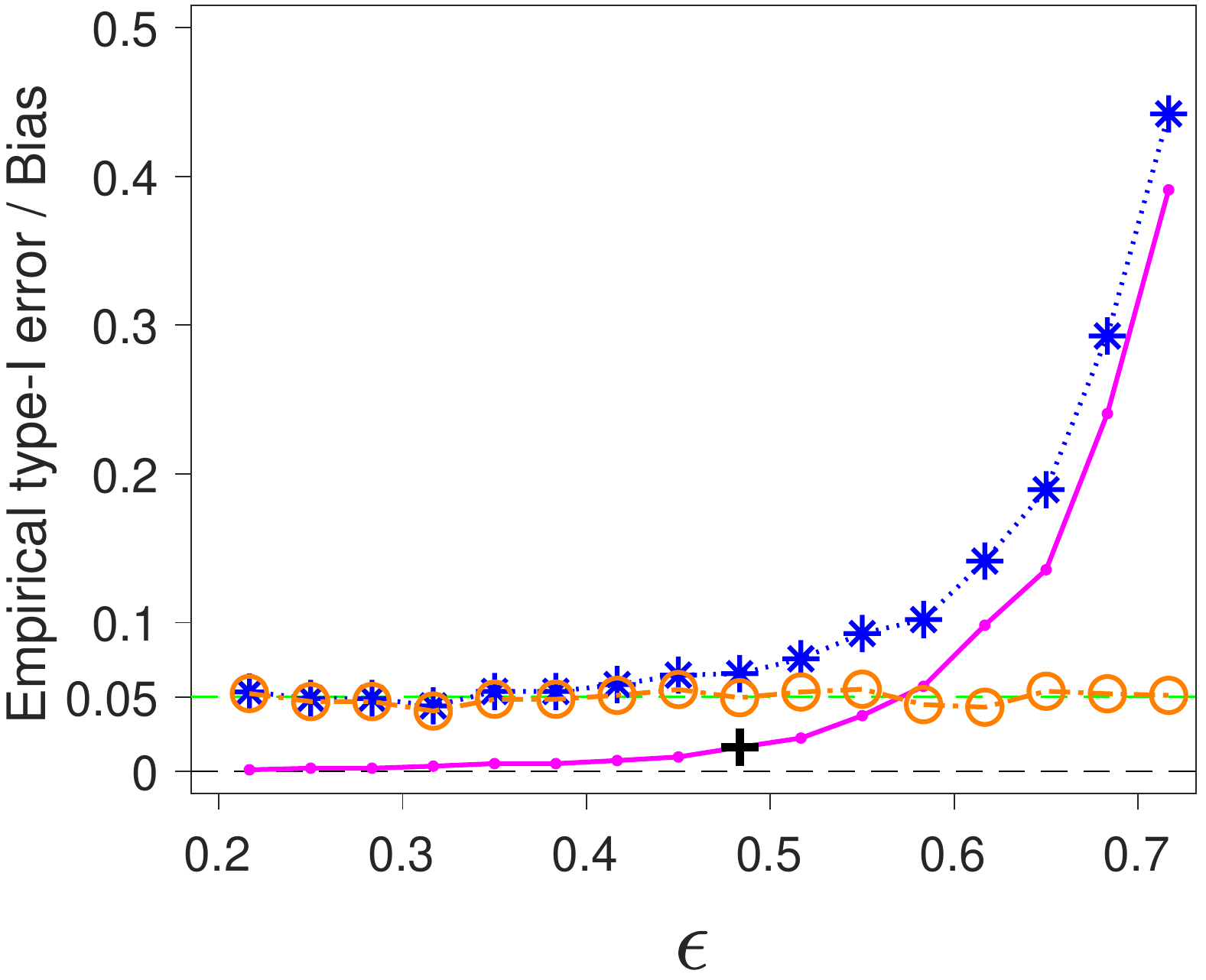}
%\centerline{(b)}
\end{subfigure}
~ \ 
\begin{subfigure}[t]{0.34\textwidth}
\centering
\includegraphics[width=\textwidth]{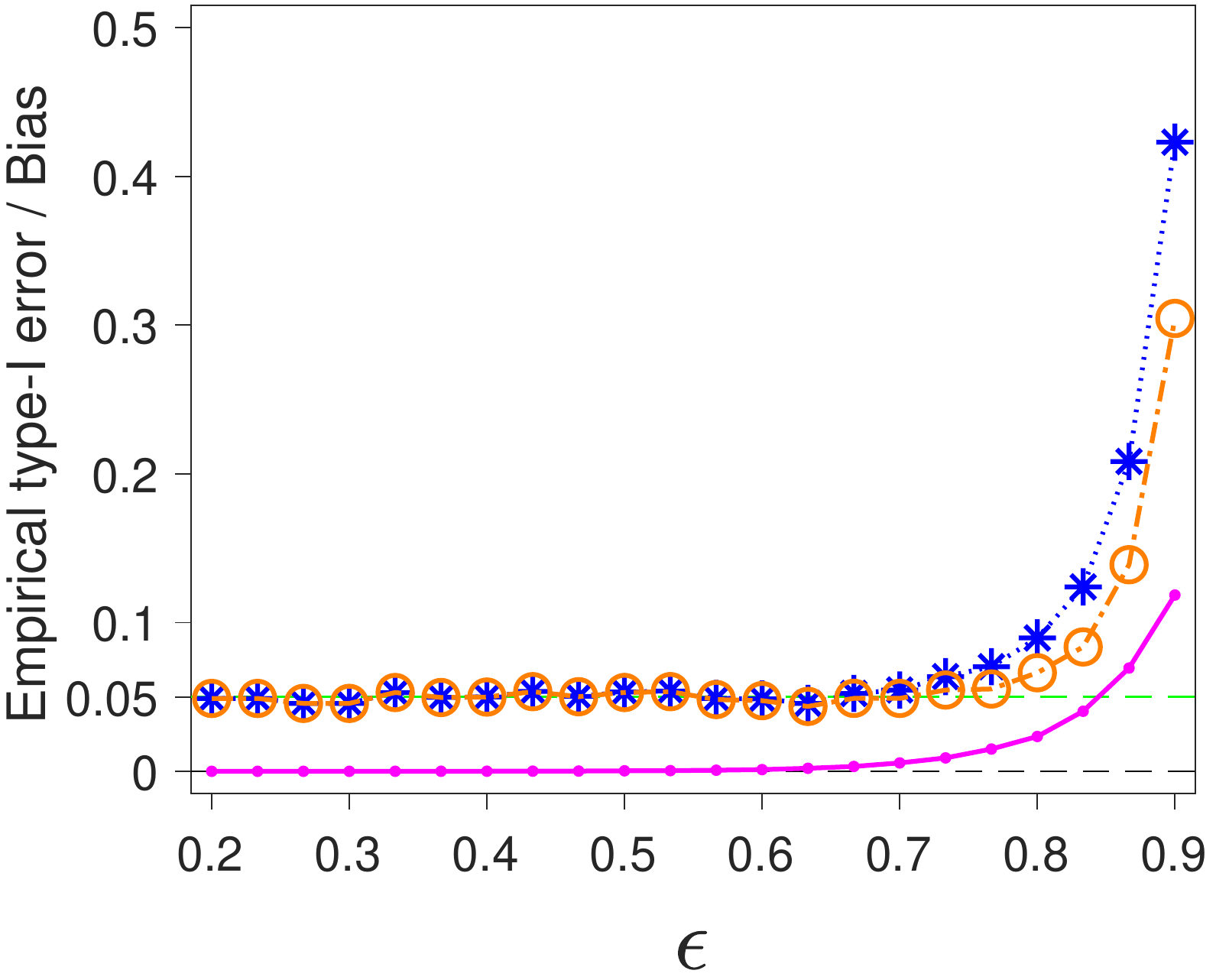}
%\centerline{(c)}
\end{subfigure}%
~ \ 
\begin{subfigure}[t]{0.34\textwidth}
\centering
\includegraphics[width=\textwidth]{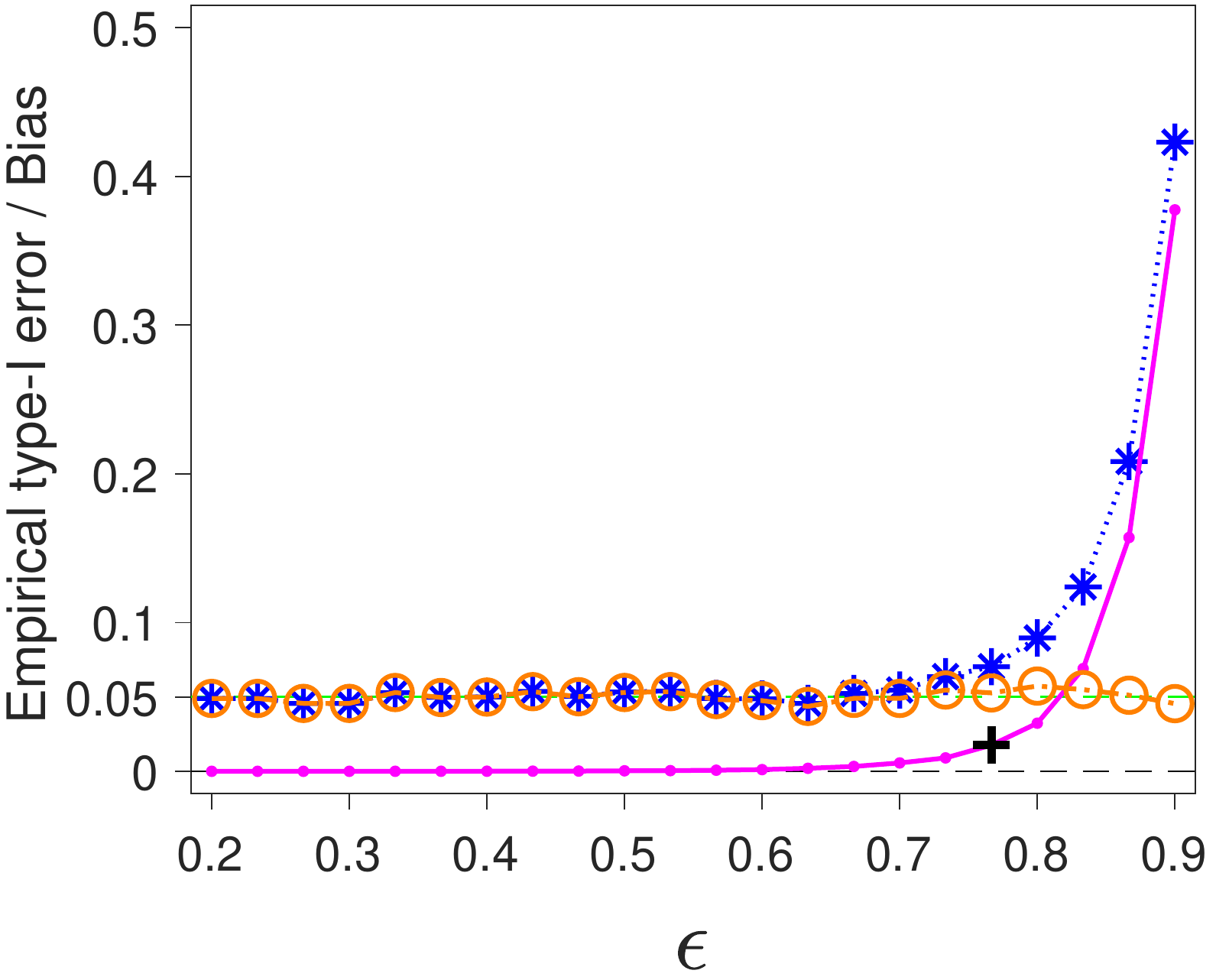}
%\centerline{(d)}
\end{subfigure}
\ \vspace{0.2em}

\begin{turn}{90}
\begin{minipage}{0.26\textwidth}
 \hspace{4.5em} Test (III)\vspace{0.12em} 
\end{minipage}
\end{turn}%
~ \
\begin{subfigure}[t]{0.34\textwidth}
\centering
\includegraphics[width=\textwidth]{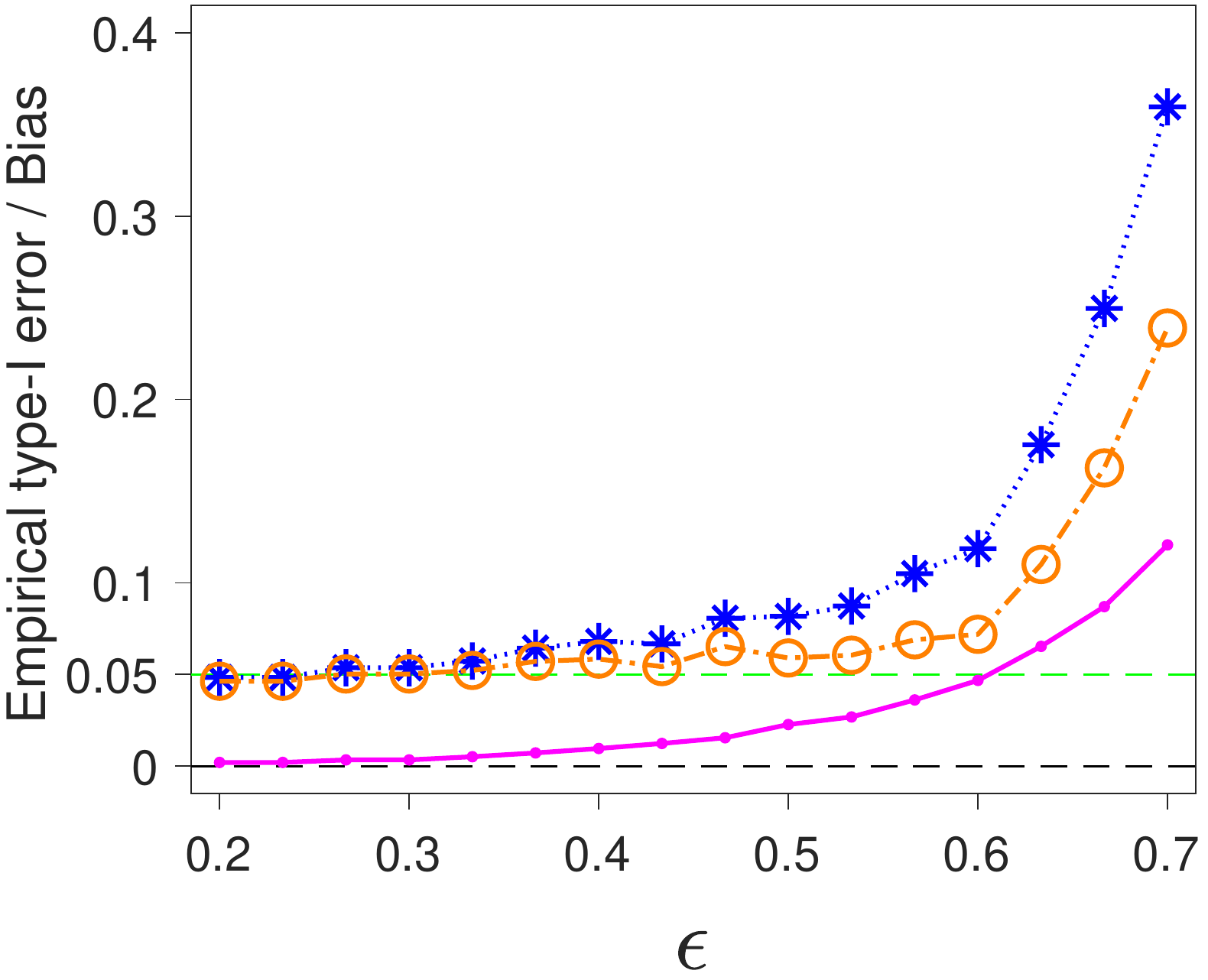}
\caption{\quad (a)\ Without the Bartlett correction}
\end{subfigure}%
~ \ 
\begin{subfigure}[t]{0.34\textwidth}
\centering
\includegraphics[width=\textwidth]{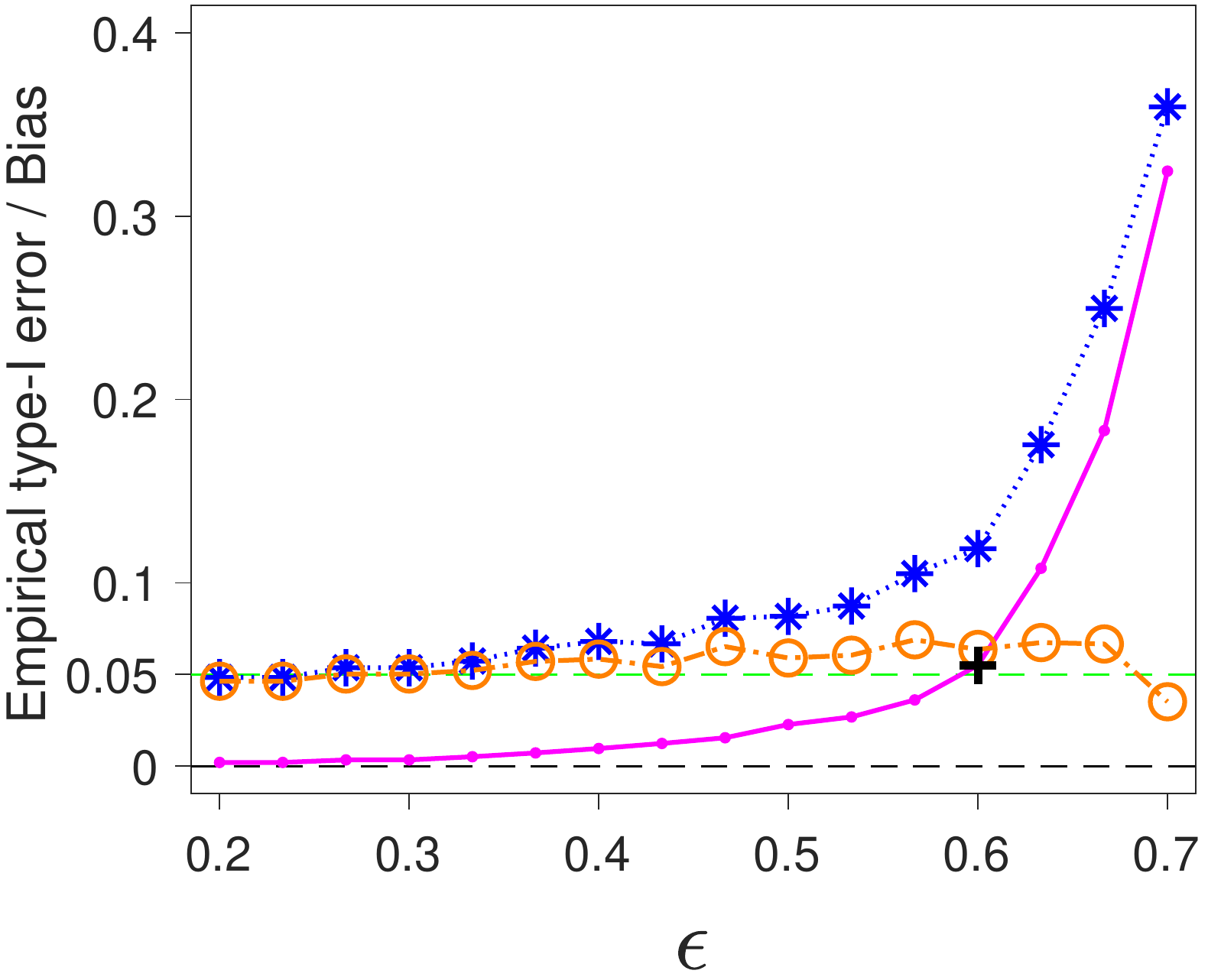}
\caption{\quad (b)\ Without the Bartlett correction}
\end{subfigure}
~ \ 
\begin{subfigure}[t]{0.34\textwidth}
\centering
\includegraphics[width=\textwidth]{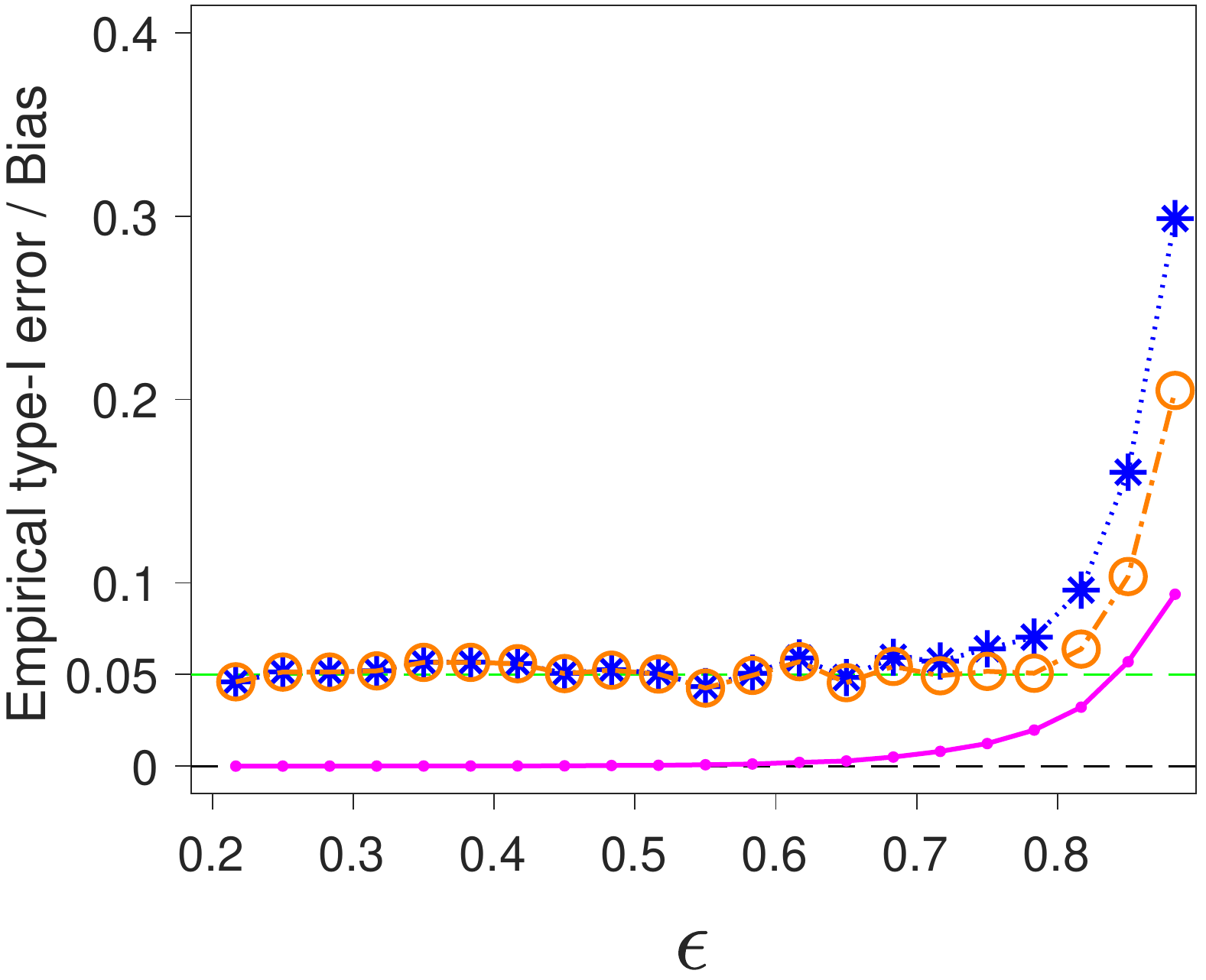}
\caption{\quad (c)\ With the Bartlett correction}
\end{subfigure}%
~ \ 
\begin{subfigure}[t]{0.34\textwidth}
\centering
\includegraphics[width=\textwidth]{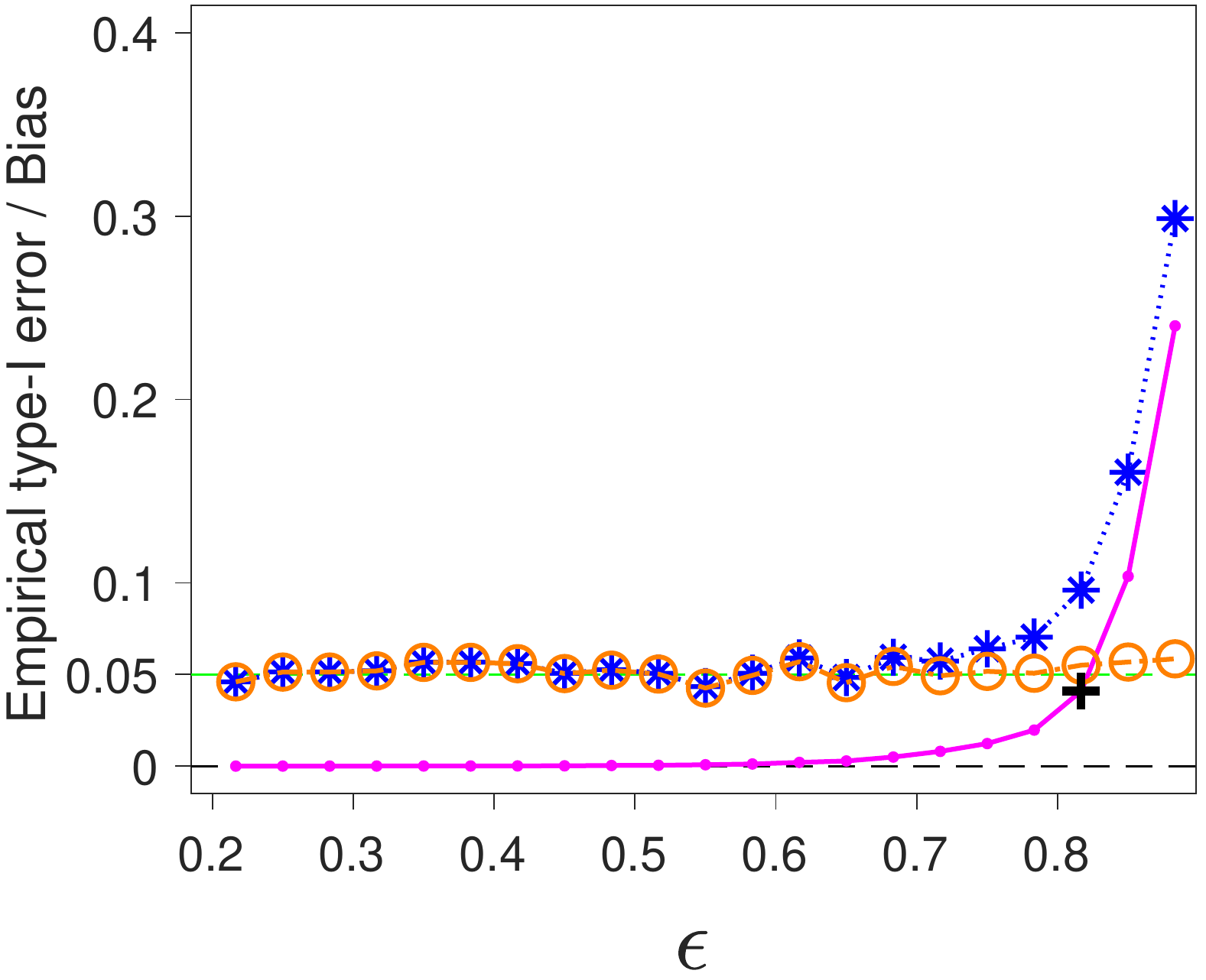}
\caption{\quad (d)\ With the Bartlett correction}
\end{subfigure}
%\ \vspace{0.2em}

\caption{One-sample tests (I)--(III) when $n=100$. 
Rows 1--3 present the results for tests (I)--(III), respectively. 
For four columns in each row: (a) Without the Bartlett correction: empirical type-\RNum{1} error versus $\epsilon$ (asterisk); $\varpi_1$, i.e., the asymptotic bias in \eqref{eq:chisqapprox} (dot); 
the difference between the empirical type-\RNum{1} error and $\varpi_1$ (circle). 
%the difference between the asterisk line and the cross line (circle). 
(b) Without the Bartlett correction: empirical type-\RNum{1} error versus $\epsilon$ (asterisk); 
%$\varpi_1 1\{\varpi_1 < c \}+ \max\{ \varpi_1 , \varpi_3 \} 1\{\varpi_1 > c \}$ 
$M_c(\varpi_1, \varpi_3)$ with $c=0.002$ 
(dot); 
%the maximum over two asymptotic biases in \eqref{eq:chisqapprox} and \eqref{eq:normalbias1} (cross); 
%the location where $M_c(\varpi_1, \varpi_3)=\varpi_1$ on its left and 
the location with $x$-axis $\epsilon^*$ satisfying $M_c(\varpi_1, \varpi_3)=\varpi_1$ when $\epsilon<\epsilon^*$ and  $M_c(\varpi_1, \varpi_3)>\varpi_1$ when $\epsilon \geq \epsilon^*$   
(plus sign);  
%the difference between the asterisk line and the cross line (circle). 
the difference between the empirical type-\RNum{1} error and $M_c(\varpi_1, \varpi_3)$ (circle).  
(c) With the Bartlett correction: empirical type-\RNum{1} error versus $\epsilon$ (asterisk); $\varpi_2$, i.e., the asymptotic bias  in \eqref{eq:chisqapproxbartcorr} (dot); 
the difference between the empirical type-\RNum{1} error and $\varpi_2$ (circle). 
%the difference between the empirical type-\RNum{1} error and $\varpi_2$ (circle). 
(d) With the Bartlett correction:  empirical type-\RNum{1} error versus $\epsilon$ (asterisk); 
%$\varpi_2 1\{\varpi_2 < c \}+ \max\{ \varpi_2 , \varpi_4 \} 1\{\varpi_2 > c \}$
$M_c(\varpi_2, \varpi_4)$ with $c=0.002$
 (dot);
 the location with $x$-axis $\epsilon^*$ satisfying $M_c(\varpi_2, \varpi_4)=\varpi_2$ when $\epsilon<\epsilon^*$ and  $M_c(\varpi_2, \varpi_4)>\varpi_2$ when $\epsilon \geq \epsilon^*$   
(plus sign);  
the difference between the empirical type-\RNum{1} error and $M_c(\varpi_2, \varpi_4)$ (circle). 
%the difference between the asterisk line and the dot line (circle).
%the maximum over the asymptotic biases in \eqref{eq:chisqapproxbartcorr} and \eqref{eq:normalbias2} (cross).
}  \label{fig:bias13n100}

\end{figure}

\end{landscape}

%\quad 
%\newpage

\begin{landscape}
\begin{figure}[!htbp]
\captionsetup[subfigure]{labelformat=empty}
\centering
\begin{turn}{90}
\begin{minipage}{0.26\textwidth}
 \hspace{4.5em} Test (I)\vspace{0.12em} 
\end{minipage}
\end{turn}%
~ \
\begin{subfigure}[t]{0.34\textwidth}
\centering
\includegraphics[width=\textwidth]{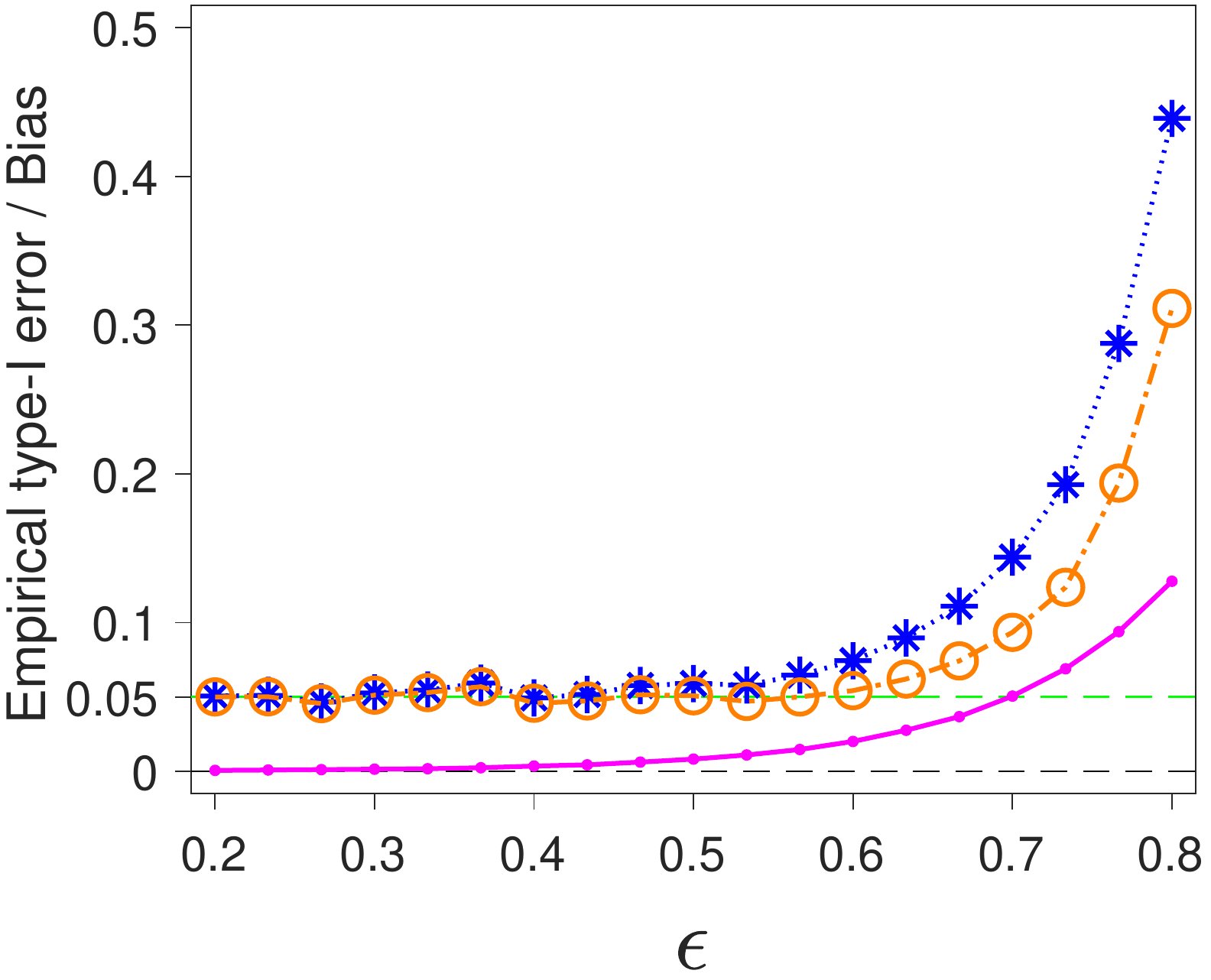}
%\centerline{(a)}
\end{subfigure}%
~ \ 
\begin{subfigure}[t]{0.34\textwidth}
\centering
\includegraphics[width=\textwidth]{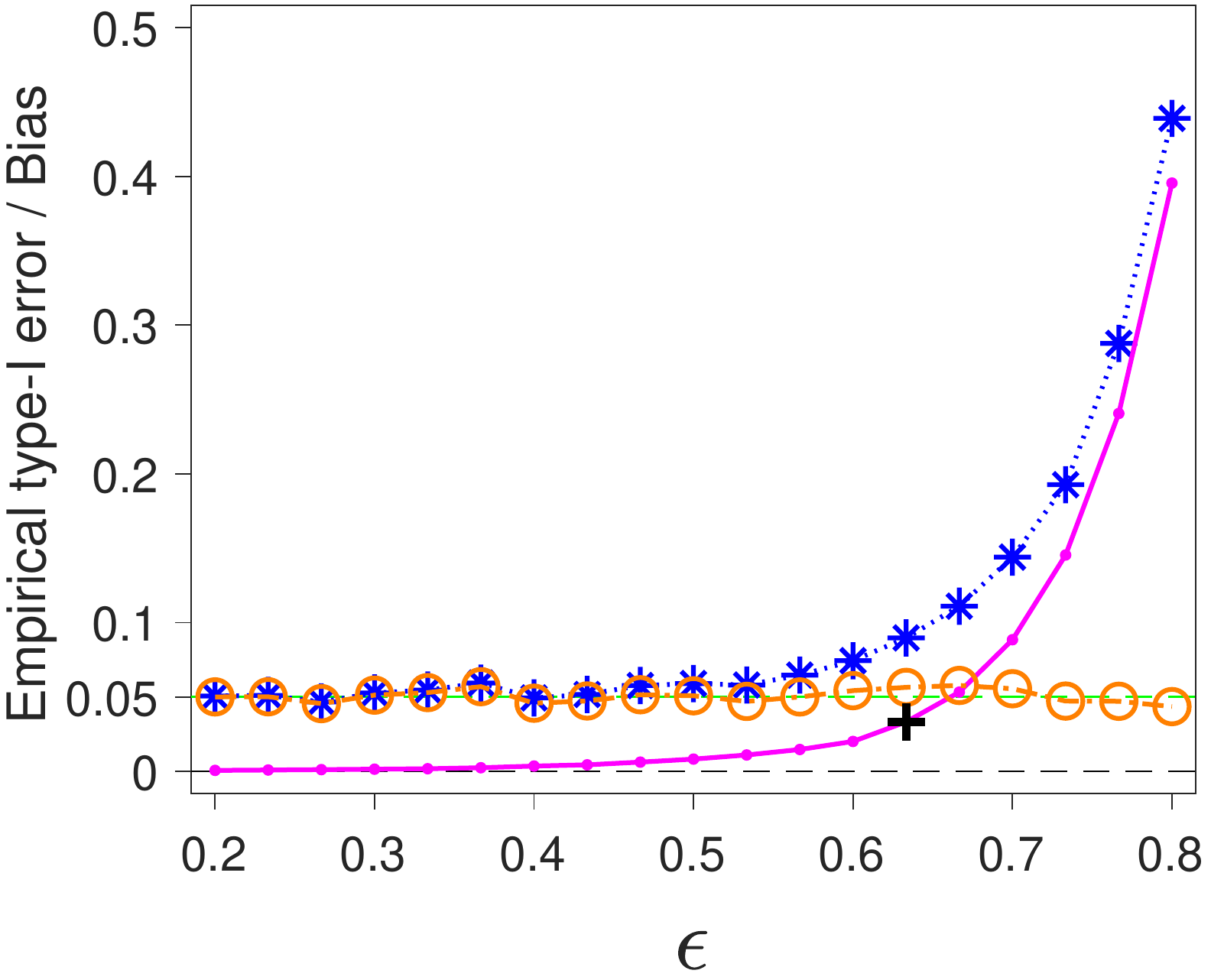}
%\centerline{(b)}
\end{subfigure}
~ \ 
\begin{subfigure}[t]{0.34\textwidth}
\centering
\includegraphics[width=\textwidth]{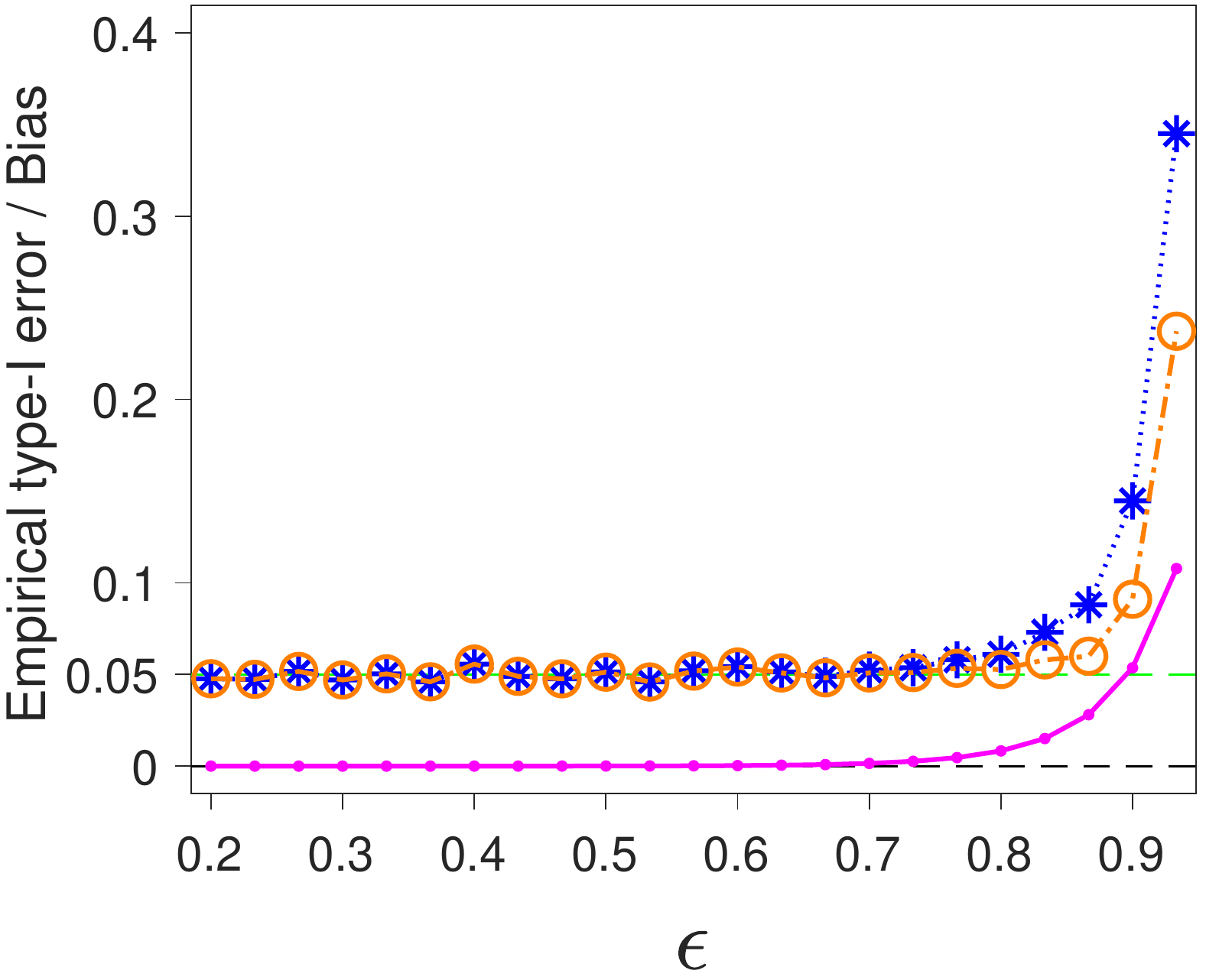}
%\centerline{(c)}
\end{subfigure}%
~ \ 
\begin{subfigure}[t]{0.34\textwidth}
\centering
\includegraphics[width=\textwidth]{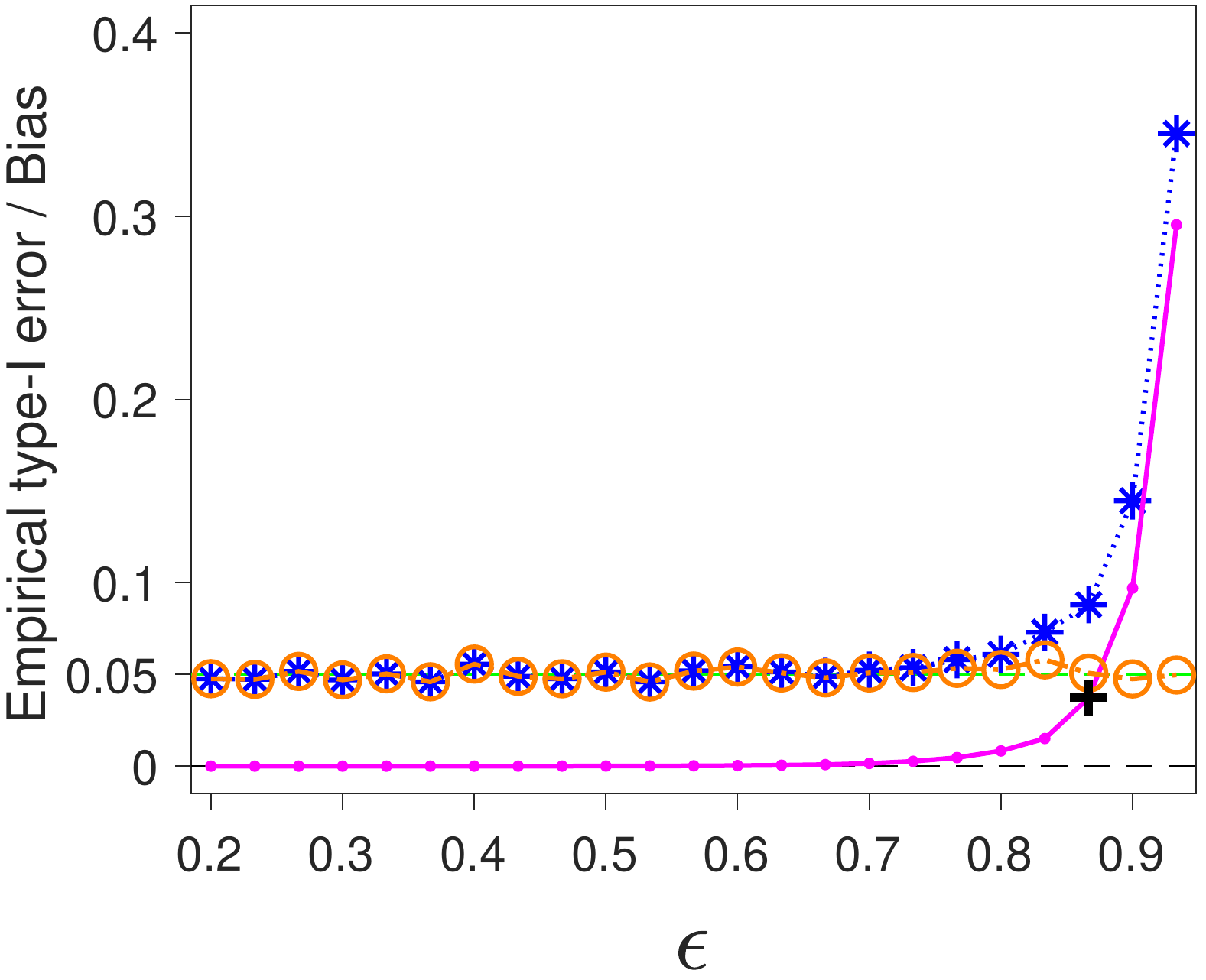}
%\centerline{(d)}
\end{subfigure}
\ \vspace{0.2em}

%\centerline{Test (V)}
%\centering
\begin{turn}{90}
\begin{minipage}{0.26\textwidth}
 \hspace{4.5em} Test (II)\vspace{0.12em} 
\end{minipage}
\end{turn}%
~ \
\begin{subfigure}[t]{0.34\textwidth}
\centering
\includegraphics[width=\textwidth]{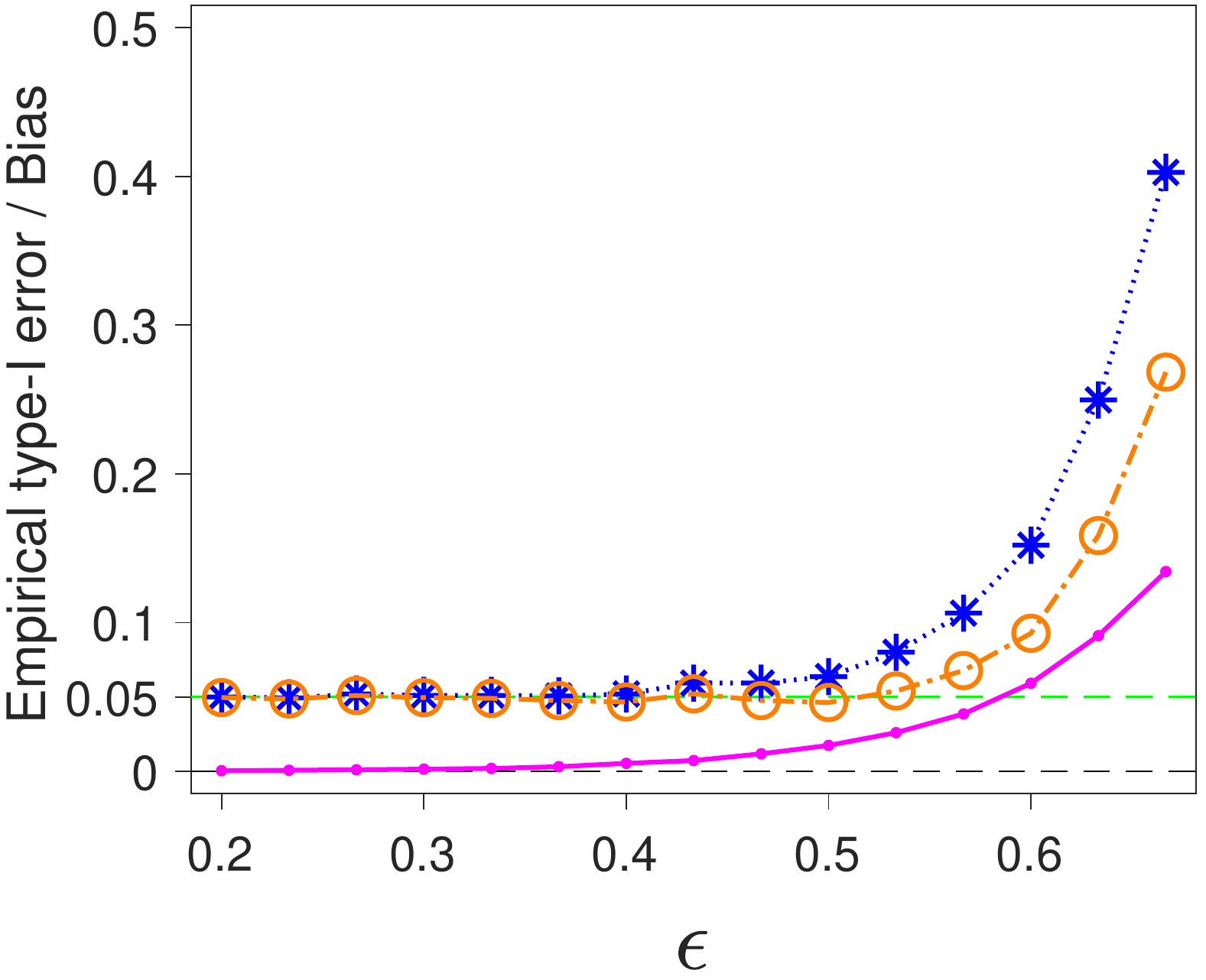}
%\centerline{(a)}
\end{subfigure}%
~ \ 
\begin{subfigure}[t]{0.34\textwidth}
\centering
\includegraphics[width=\textwidth]{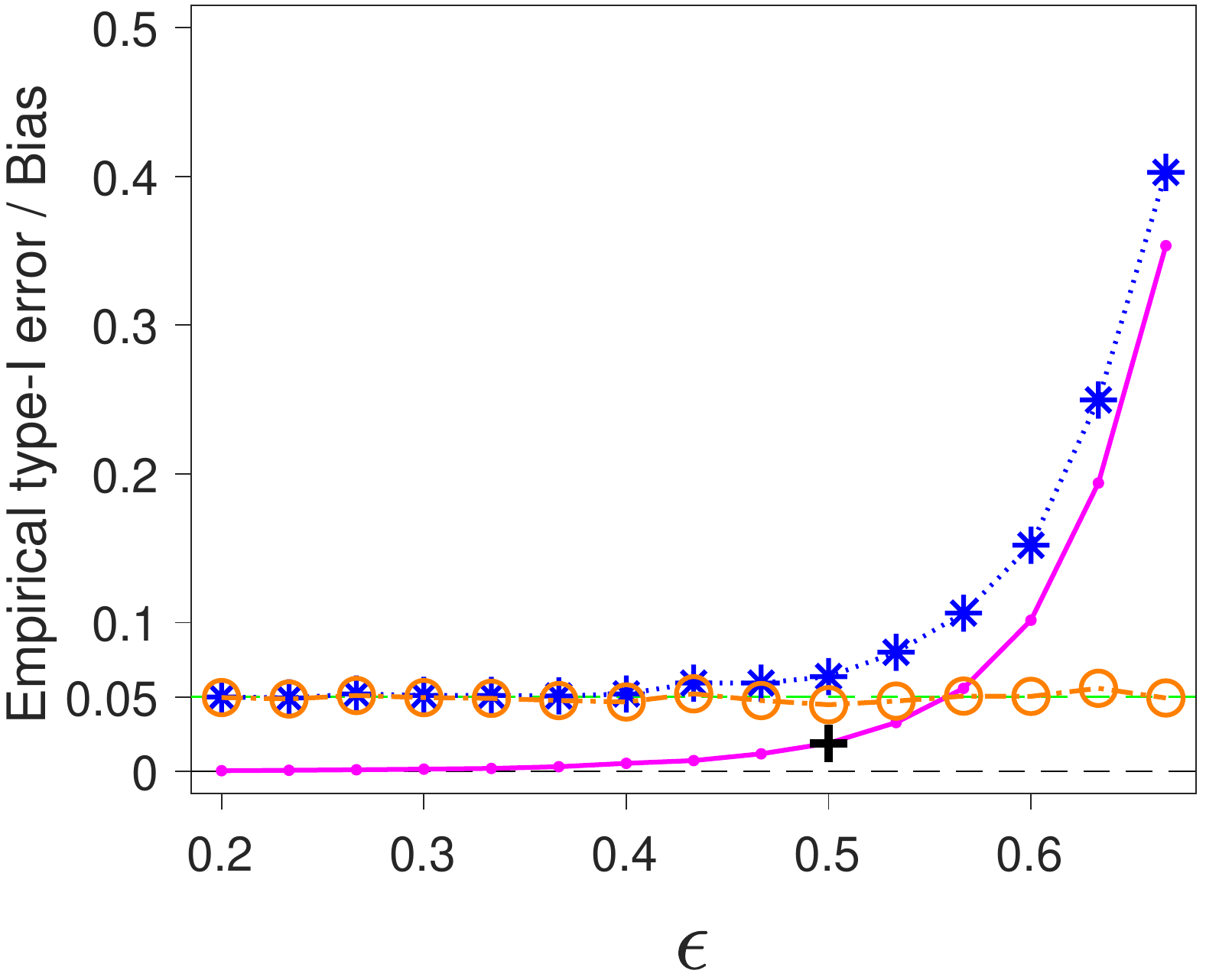}
%\centerline{(b)}
\end{subfigure}
~ \ 
\begin{subfigure}[t]{0.34\textwidth}
\centering
\includegraphics[width=\textwidth]{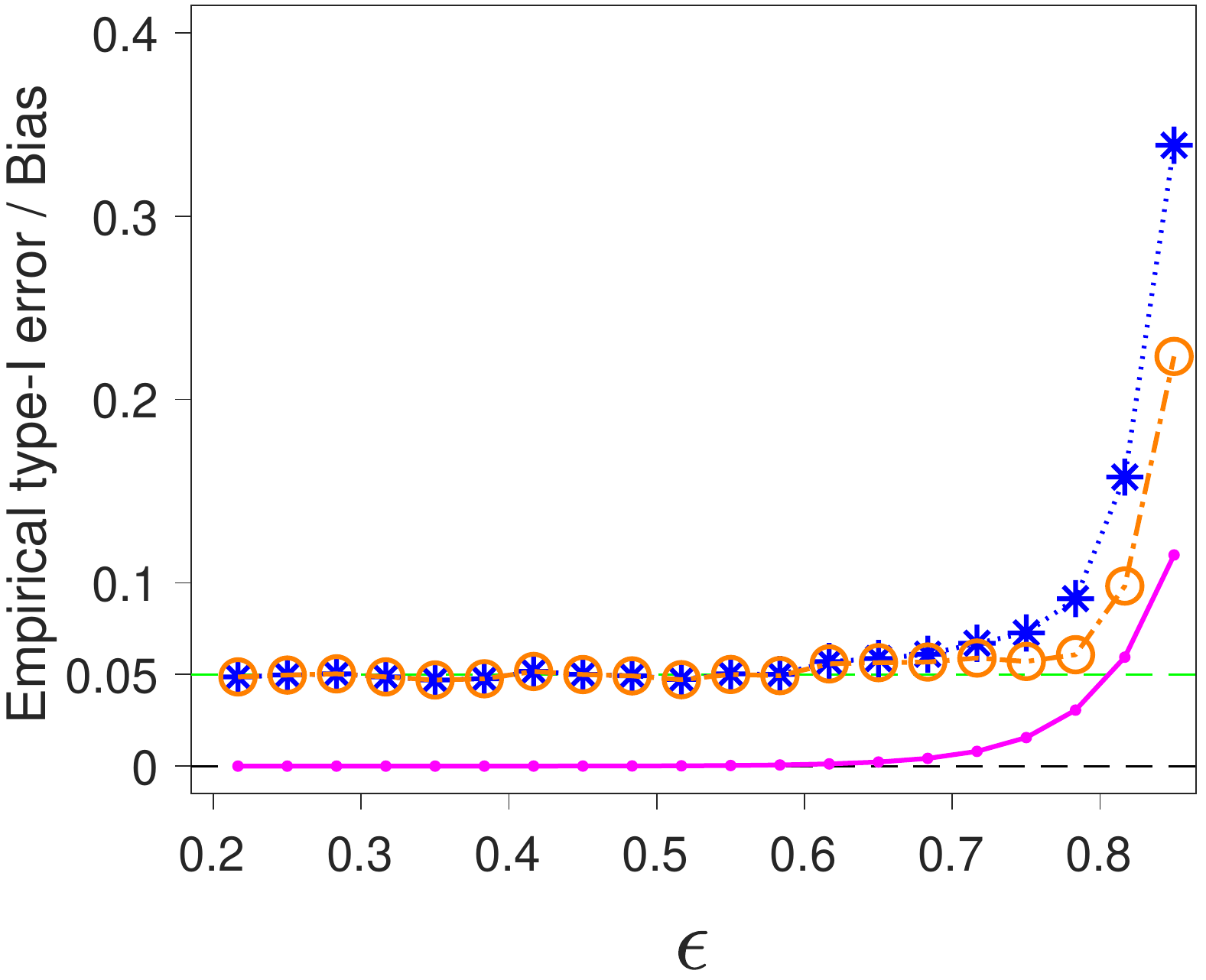}
%\centerline{(c)}
\end{subfigure}%
~ \ 
\begin{subfigure}[t]{0.34\textwidth}
\centering
\includegraphics[width=\textwidth]{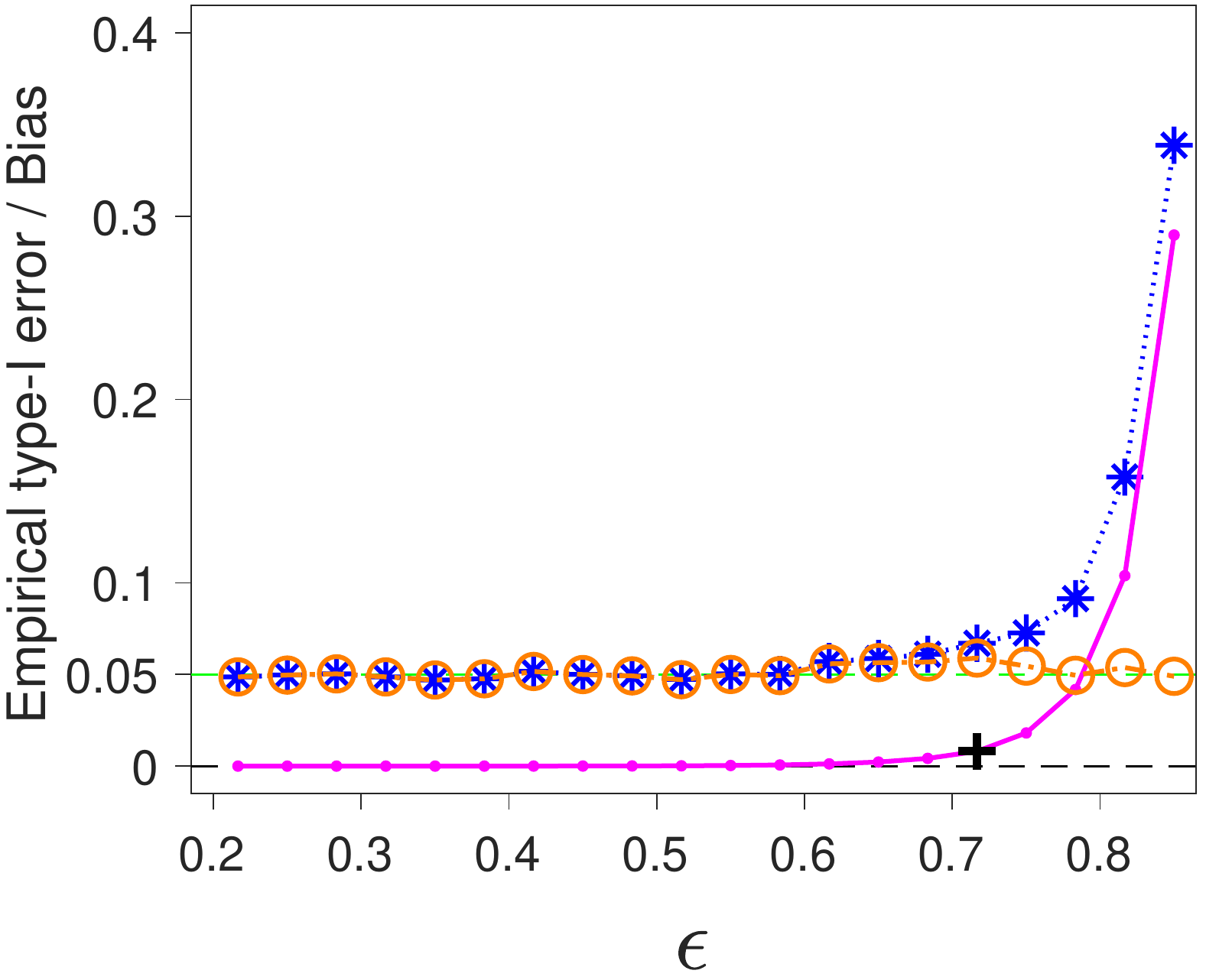}
%\centerline{(d)}
\end{subfigure}
\ \vspace{0.2em}

%\quad 
%\vspace{0.1em}

%\centerline{Test (VI)}
%\centering
%\hspace{-0.8em}
\begin{turn}{90}
\begin{minipage}{0.26\textwidth}
 \hspace{4.5em} Test (III)\vspace{0.12em}
\end{minipage}
\end{turn}%
~ \ 
\begin{subfigure}[t]{0.34\textwidth}
\centering	
\includegraphics[width=\textwidth]{IIIIdentical_Bias1_Chi_square}
%\centerline{(a)}
\caption{\quad (a)\ Without the Bartlett correction}
\end{subfigure}%
~ \ 
\begin{subfigure}[t]{0.34\textwidth}
\centering
\includegraphics[width=\textwidth]{IIIIdentical_Bias1_Max}
%\centerline{(b)}
\caption{\quad (b)\ Without the Bartlett correction}
\end{subfigure}%
~ \ 
\begin{subfigure}[t]{0.34\textwidth}
\centering	
\includegraphics[width=\textwidth]{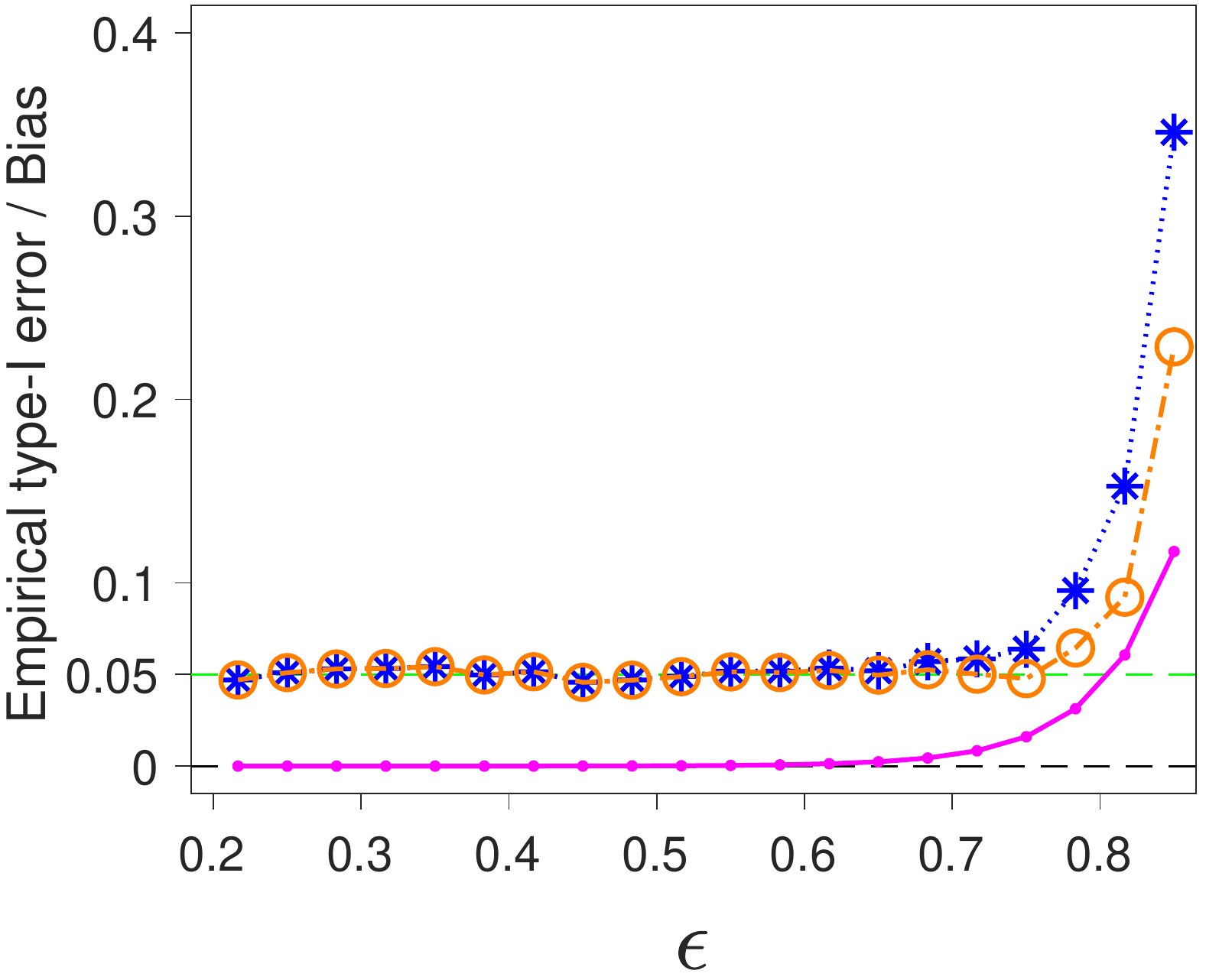}
%\centerline{(c)}
\caption{\quad (c) With the Bartlett correction}
\end{subfigure}%
~ \ 
\begin{subfigure}[t]{0.34\textwidth}
\centering
\includegraphics[width=\textwidth]{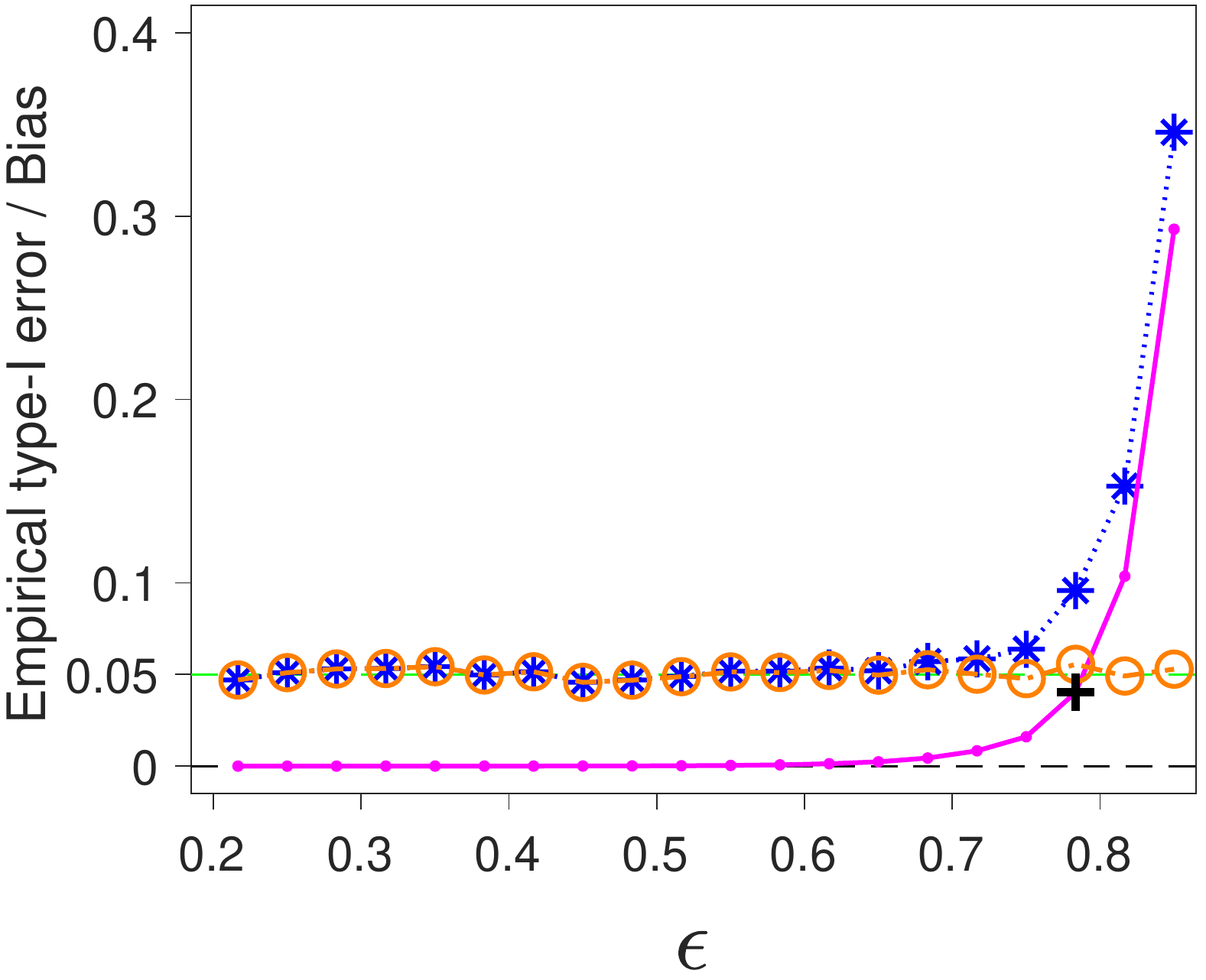}
%\centerline{(d)}
\caption{\quad (d) With the Bartlett correction}
\end{subfigure}
%\caption{Multiple-sample tests (IV)--(VI) when $n=500$. Column (a) without the Bartlett correction, the asymptotic bias in \eqref{eq:chisqapprox} (cross); (b) without the Bartlett correction and the asymptotic bias ; (c) (d) In all subfigures, empirical type-\RNum{1} error versus $\epsilon$ (asterisk); the difference between the asterisk line and the cross line (circle).}
%\caption{One-sample tests (I)--(III) when $n=500$. In each row: (a) Without the Bartlett correction: empirical type-\RNum{1} error versus $\epsilon$ (asterisk); the asymptotic bias in \eqref{eq:chisqapprox} (cross); the difference between the asterisk line and the cross line (circle). (b) Without the Bartlett correction: empirical type-\RNum{1} error versus $\epsilon$ (asterisk); the maximum over the asymptotic biases in \eqref{eq:chisqapprox} and  \eqref{eq:normalbias1} (cross); the difference between the asterisk line and the cross line (circle). (c) With the Bartlett correction: empirical type-\RNum{1} error versus $\epsilon$ (asterisk); the asymptotic bias in \eqref{eq:chisqapproxbartcorr} (cross); the difference between the asterisk line and the cross line (circle). (d) With the Bartlett correction: empirical type-\RNum{1} error versus $\epsilon$ (asterisk); the maximum over the asymptotic biases in \eqref{eq:chisqapproxbartcorr} and \eqref{eq:normalbias2} (cross).}  \label{fig:bias13n500}
\caption{One-sample tests (I)--(III) when $n=500$. Rows 1--3 present the results for tests (I)--(III), respectively. 
For four columns in each row, please see the caption description in Fig. \ref{fig:bias13n100}.} \label{fig:bias13n500}

\end{figure}
\end{landscape}

%\quad 
%
%\newpage

\begin{landscape}
\begin{figure}[!htbp]
\captionsetup[subfigure]{labelformat=empty}
\centering
\begin{turn}{90}
\begin{minipage}{0.26\textwidth}
 \hspace{4.5em} Test (IV)\vspace{0.12em} 
\end{minipage}
\end{turn}%
~ \
\begin{subfigure}[t]{0.34\textwidth}
\centering
\includegraphics[width=\textwidth]{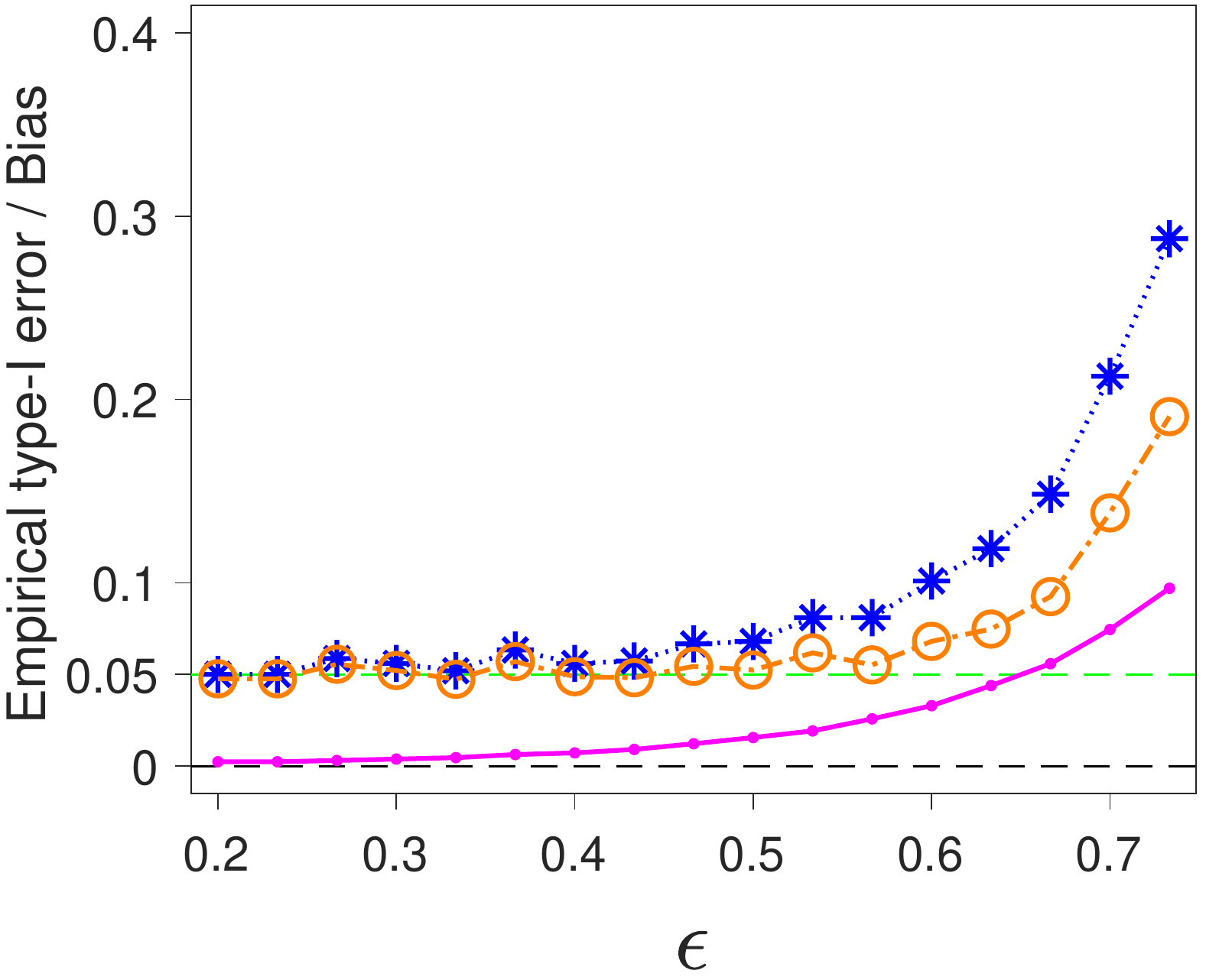}
%\centerline{(a)}
\end{subfigure}%
~ \ 
\begin{subfigure}[t]{0.34\textwidth}
\centering
\includegraphics[width=\textwidth]{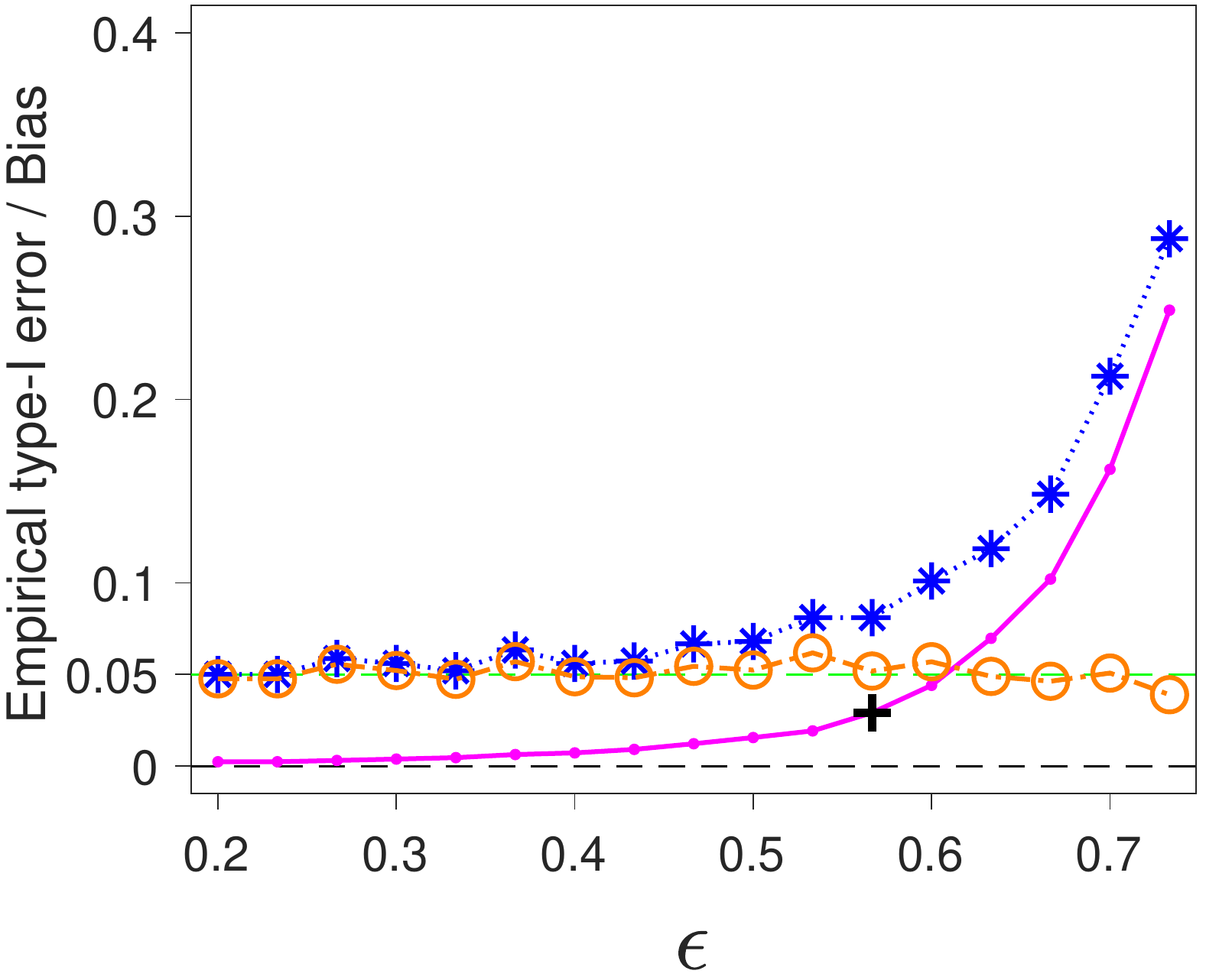}
%\centerline{(b)}
\end{subfigure}
~ \ 
\begin{subfigure}[t]{0.34\textwidth}
\centering
\includegraphics[width=\textwidth]{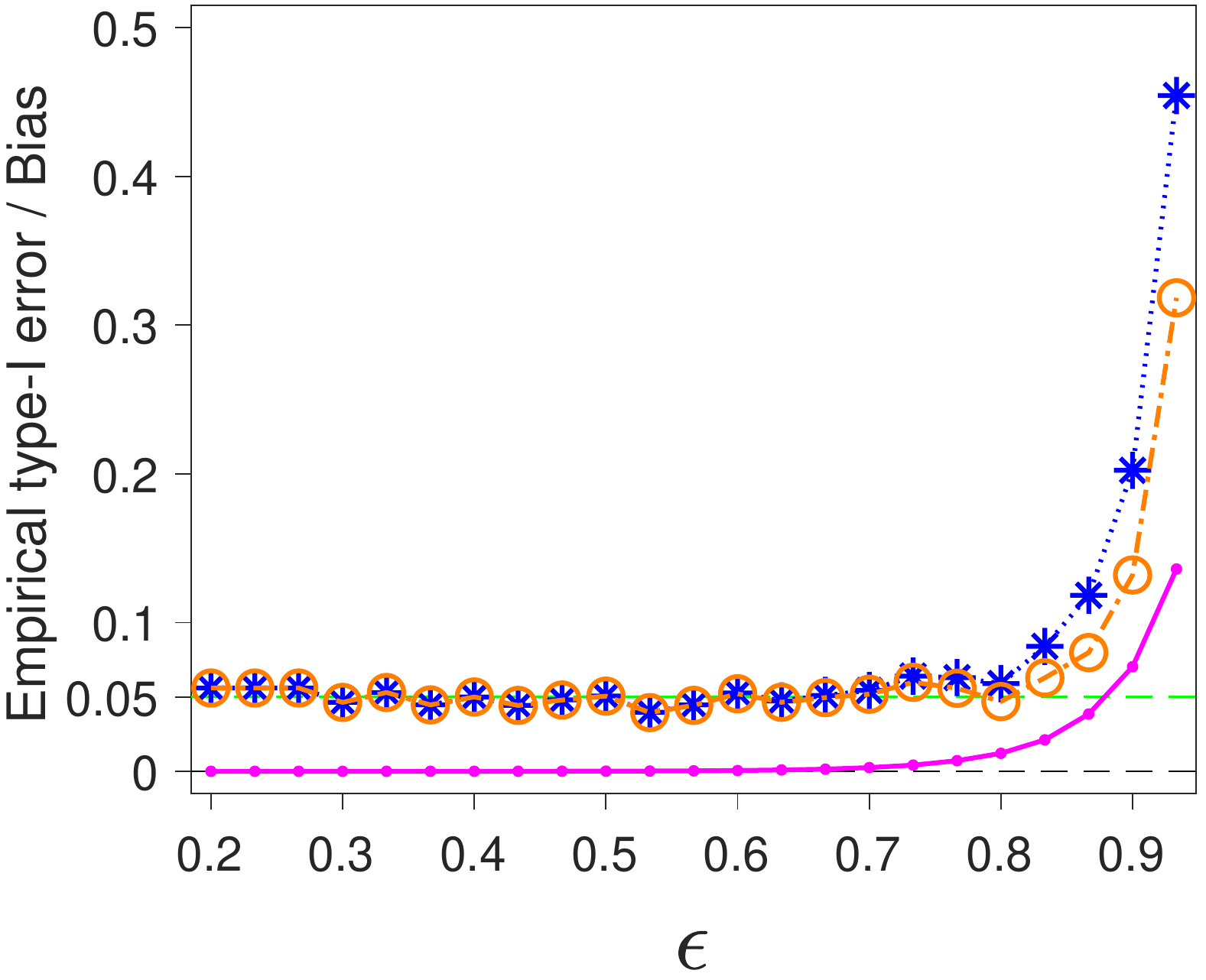}
%\centerline{(c)}
\end{subfigure}%
~ \ 
\begin{subfigure}[t]{0.34\textwidth}
\centering
\includegraphics[width=\textwidth]{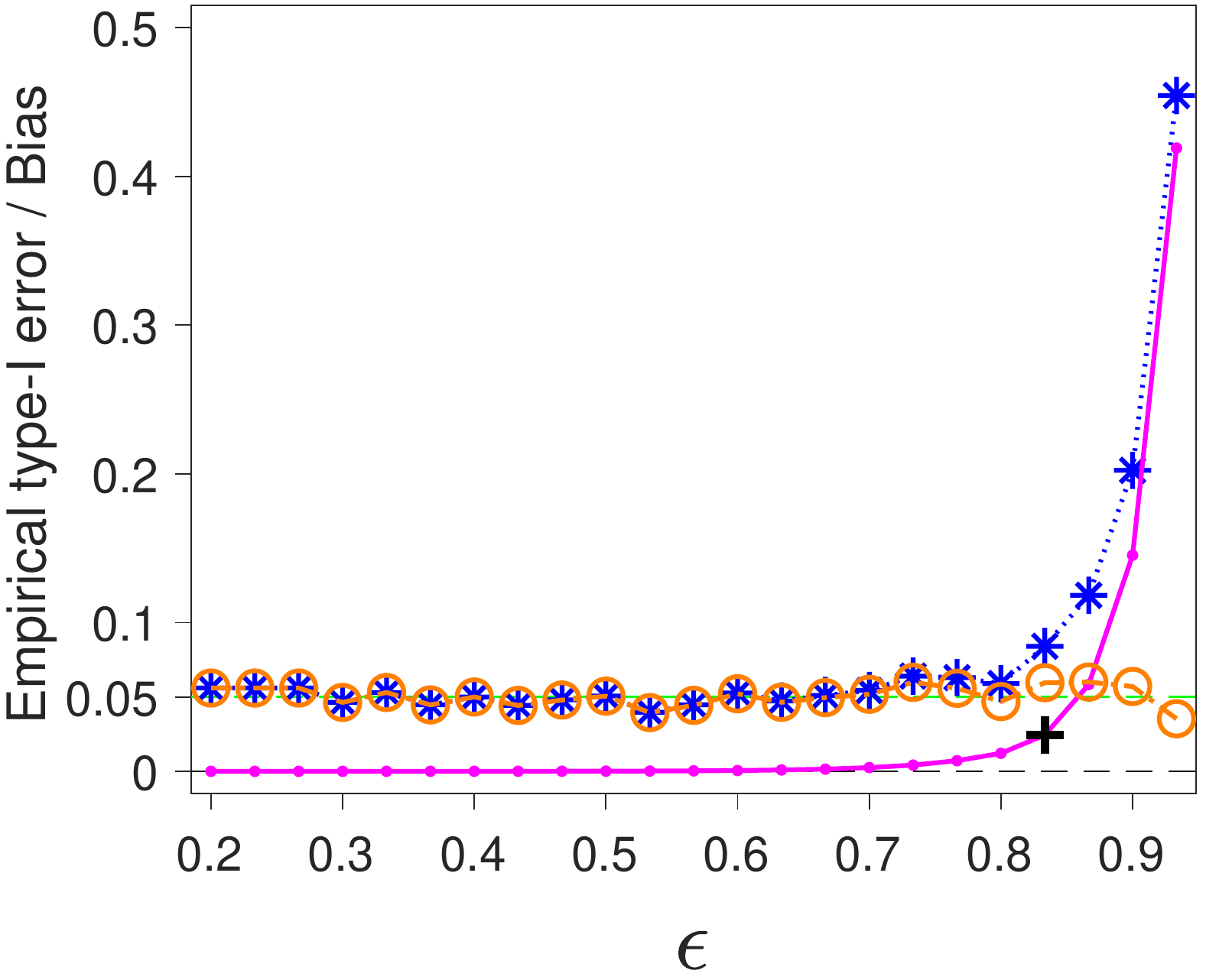}
%\centerline{(d)}
\end{subfigure}
\ \vspace{0.2em}

%\centerline{Test (V)}
%\centering
\begin{turn}{90}
\begin{minipage}{0.26\textwidth}
 \hspace{4.5em} Test (V)\vspace{0.12em} 
\end{minipage}
\end{turn}%
~ \
\begin{subfigure}[t]{0.34\textwidth}
\centering
\includegraphics[width=\textwidth]{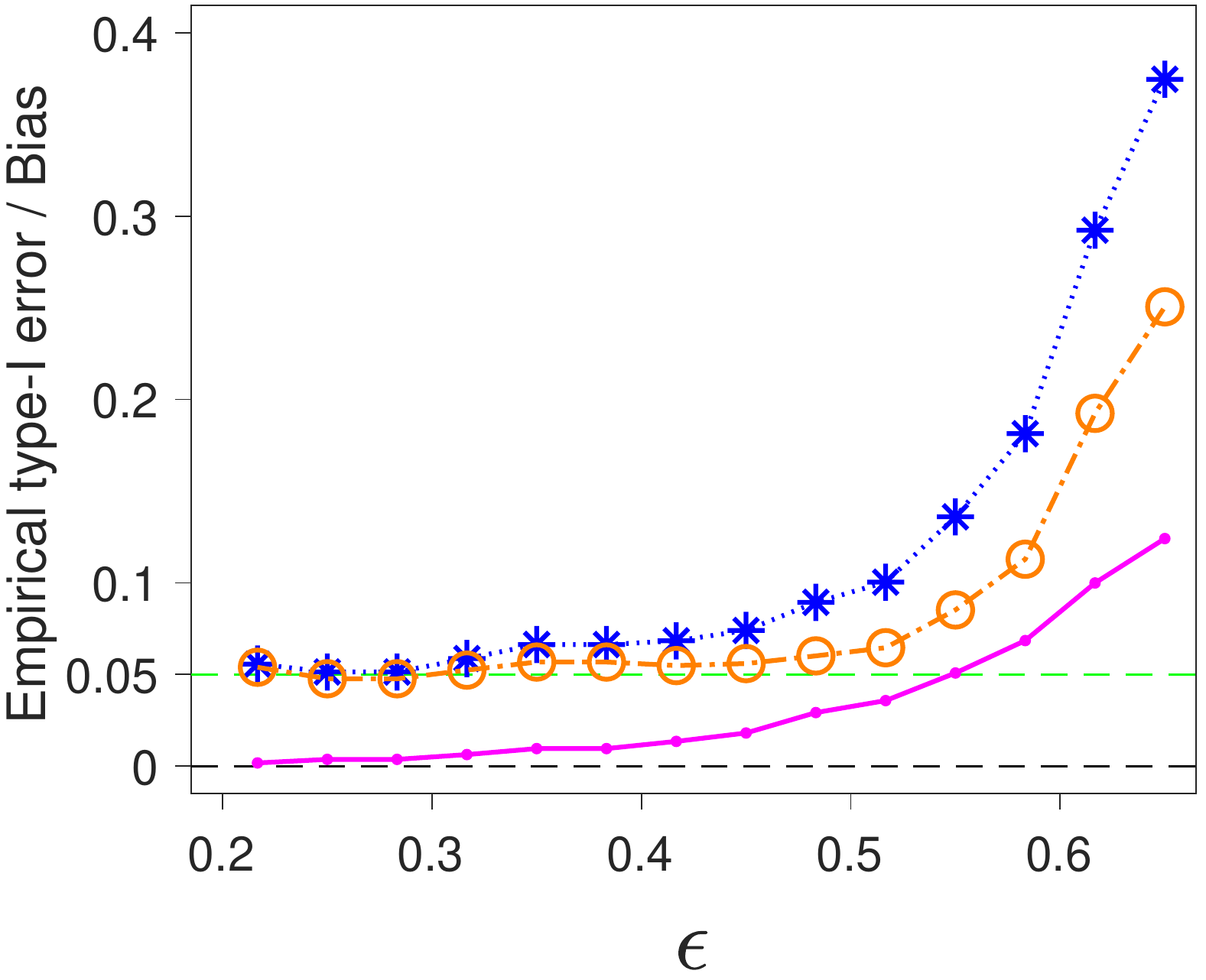}
%\centerline{(a)}
\end{subfigure}%
~ \ 
\begin{subfigure}[t]{0.34\textwidth}
\centering
\includegraphics[width=\textwidth]{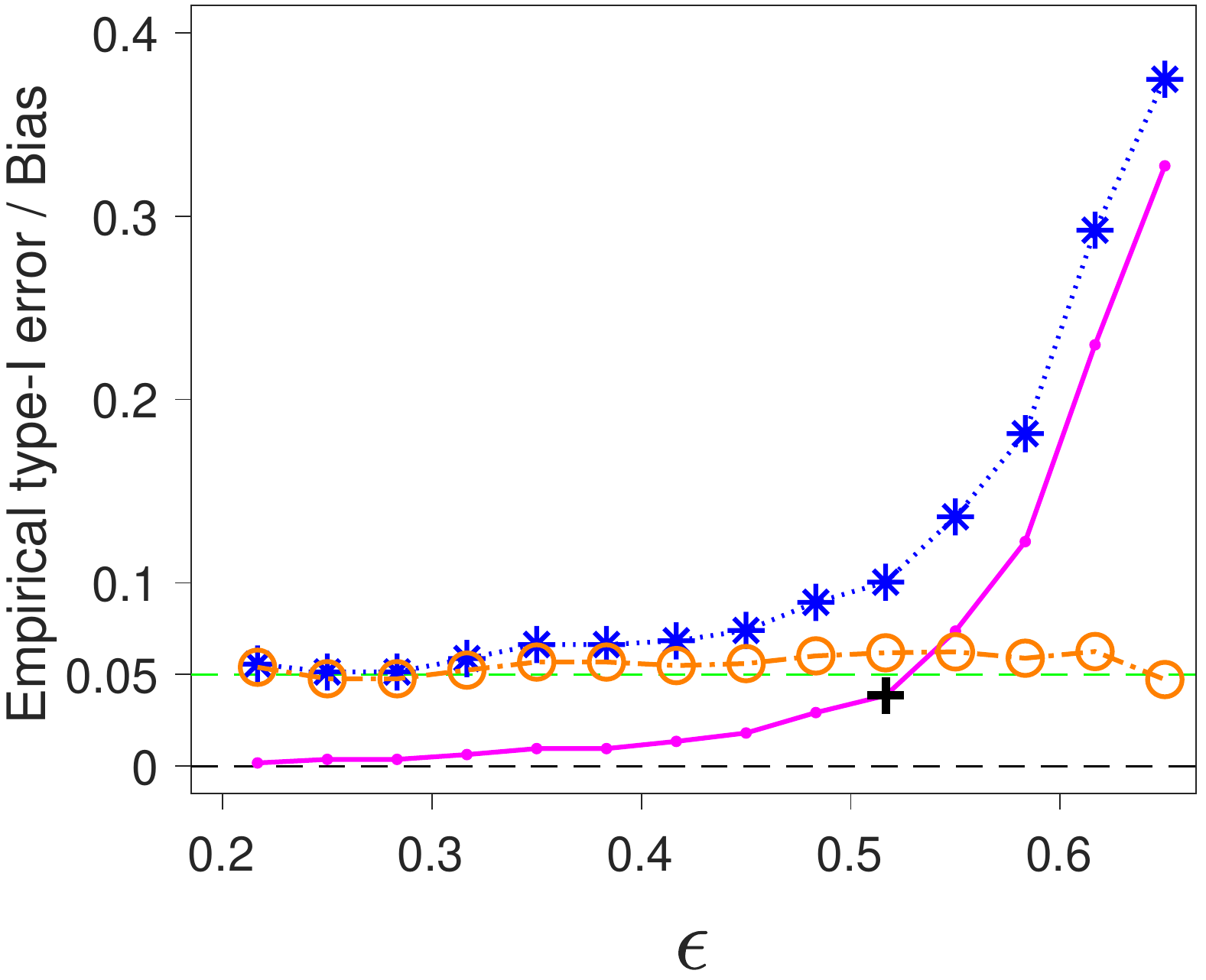}
%\centerline{(b)}
\end{subfigure}
~ \ 
\begin{subfigure}[t]{0.34\textwidth}
\centering
\includegraphics[width=\textwidth]{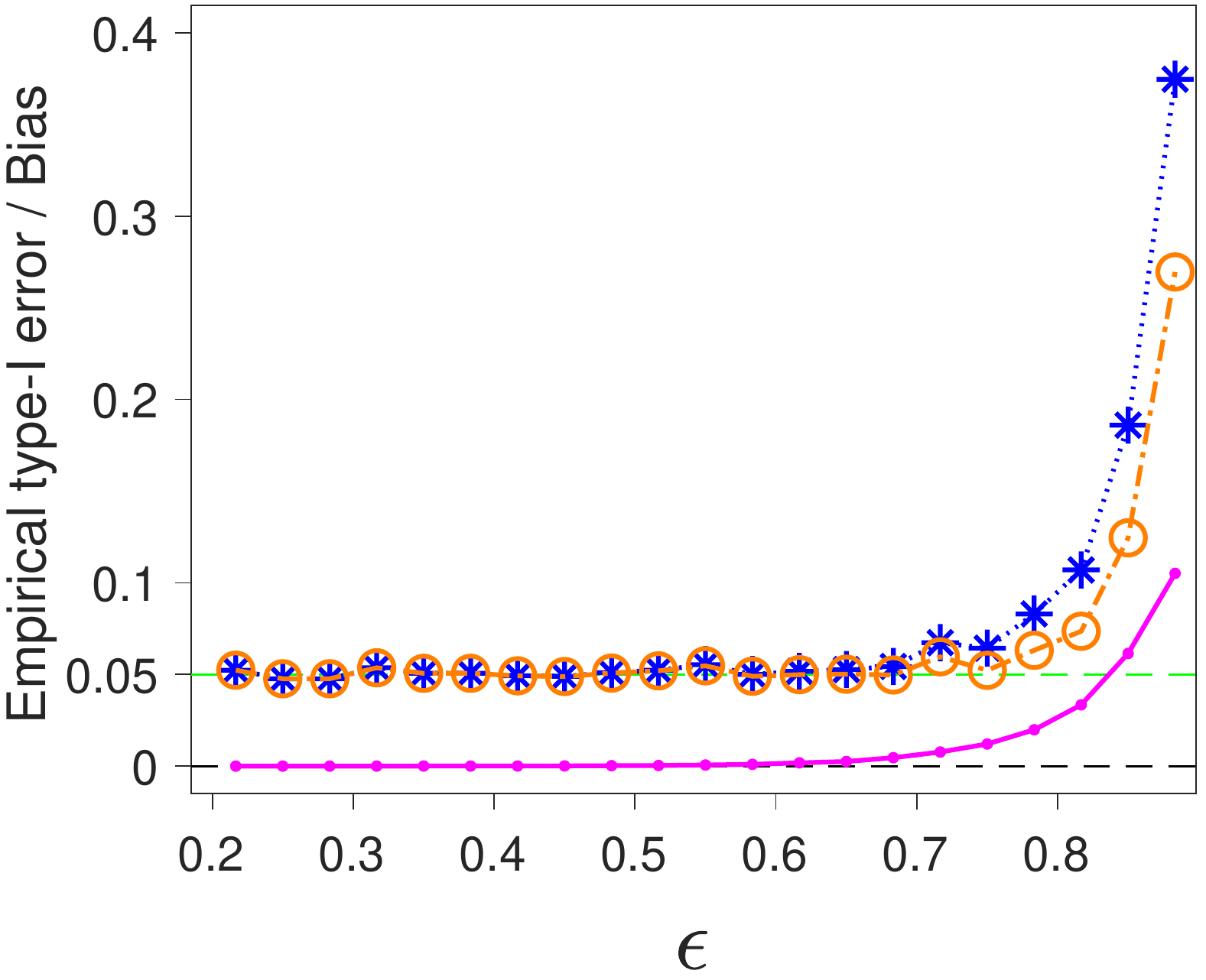}
%\centerline{(c)}
\end{subfigure}%
~ \ 
\begin{subfigure}[t]{0.34\textwidth}
\centering
\includegraphics[width=\textwidth]{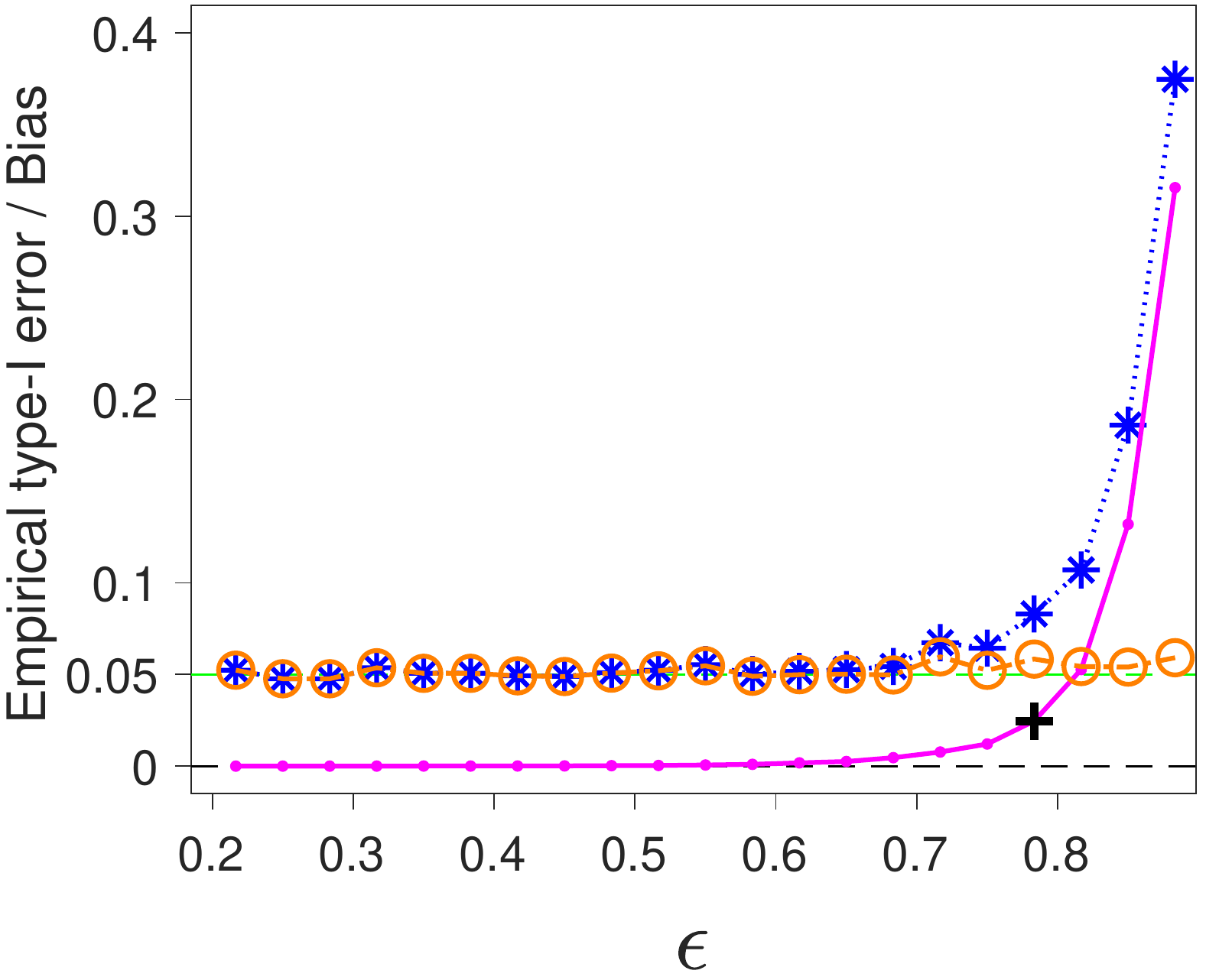}
%\centerline{(d)}
\end{subfigure}
\ \vspace{0.2em}

%\quad 
%\vspace{0.1em}

%\centerline{Test (VI)}
%\centering
\begin{turn}{90}
\begin{minipage}{0.26\textwidth}
 \hspace{4.5em} Test (VI)\vspace{0.12em}
\end{minipage}
\end{turn}%
~ \ 
\begin{subfigure}[t]{0.34\textwidth}
\centering	
\includegraphics[width=\textwidth]{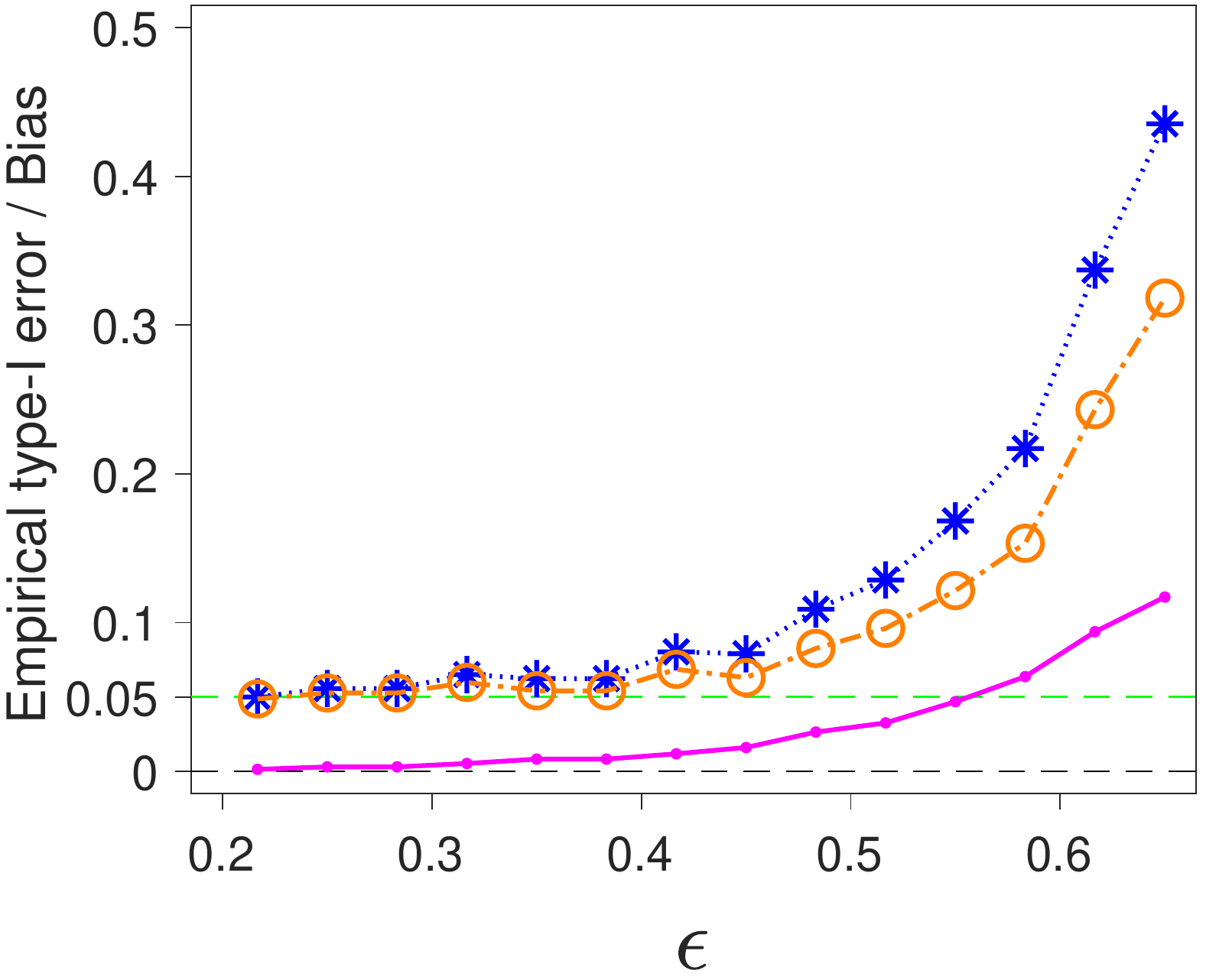}
%\centerline{(a)}
\caption{\quad (a)\ Without the Bartlett correction}
\end{subfigure}%
~ \ 
\begin{subfigure}[t]{0.34\textwidth}
\centering
\includegraphics[width=\textwidth]{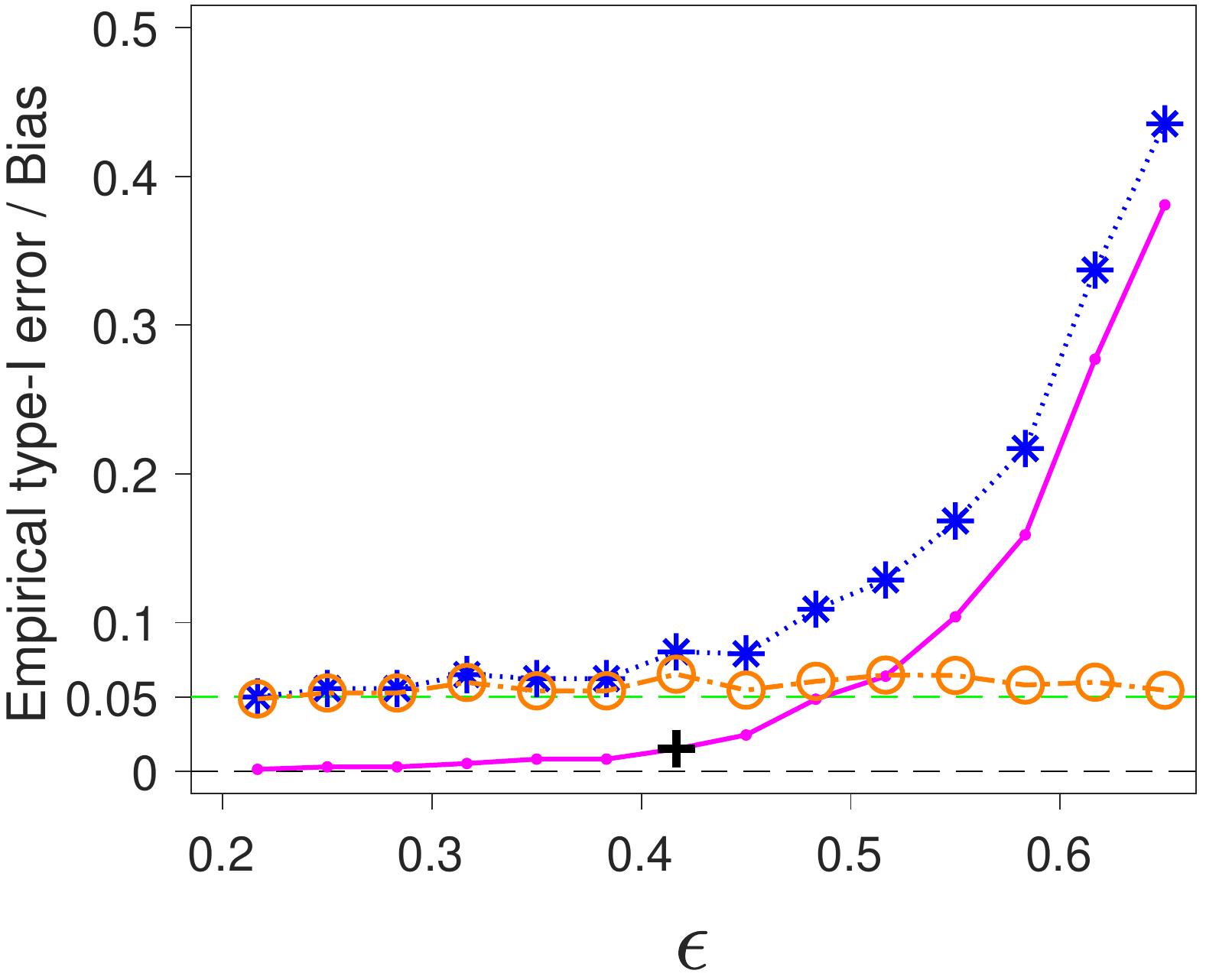}
%\centerline{(b)}
\caption{\quad (b)\ Without the Bartlett correction}
\end{subfigure}%
~ \ 
\begin{subfigure}[t]{0.34\textwidth}
\centering	
\includegraphics[width=\textwidth]{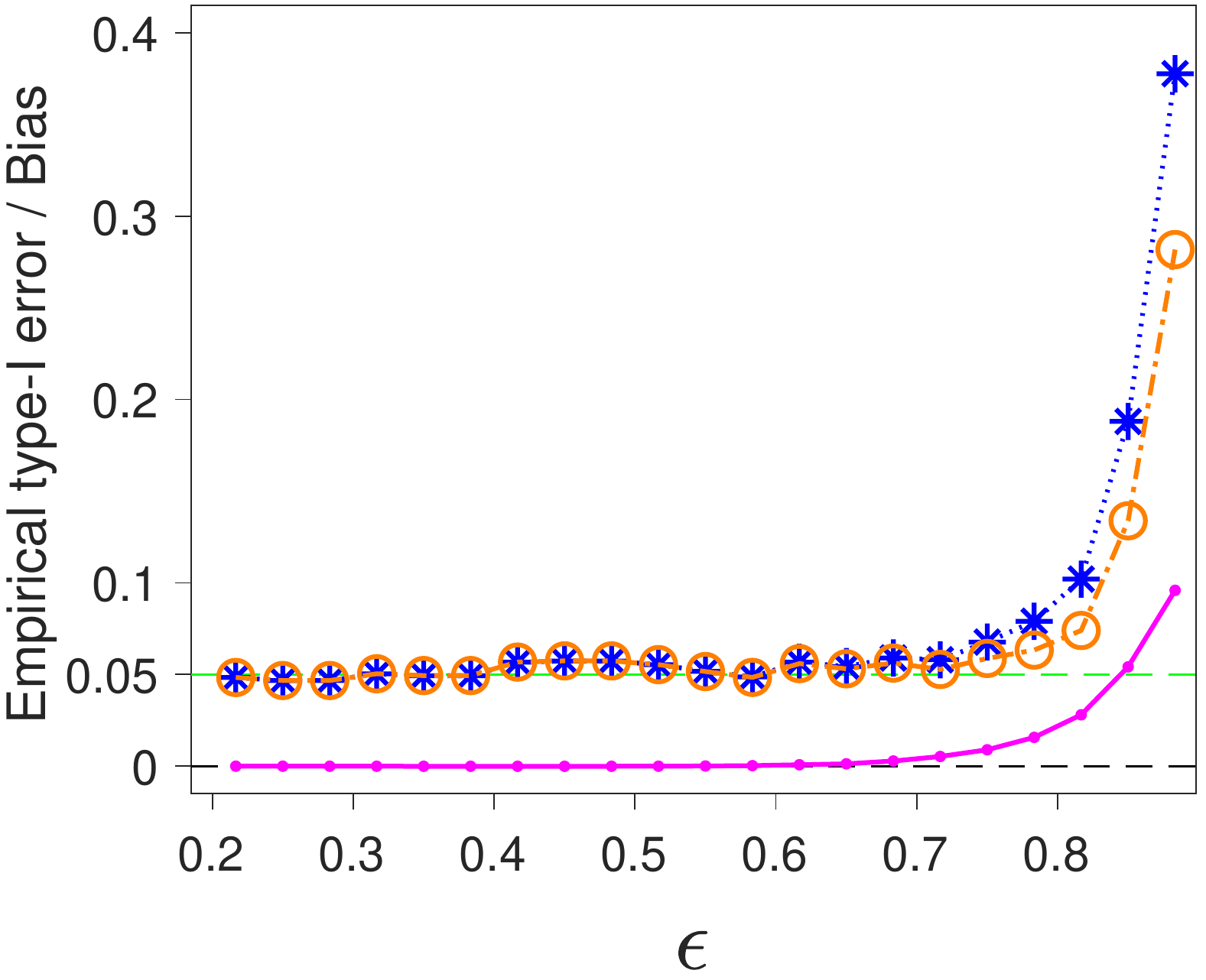}
%\centerline{(c)}
\caption{\quad (c) With the Bartlett correction}
\end{subfigure}%
~ \ 
\begin{subfigure}[t]{0.34\textwidth}
\centering
\includegraphics[width=\textwidth]{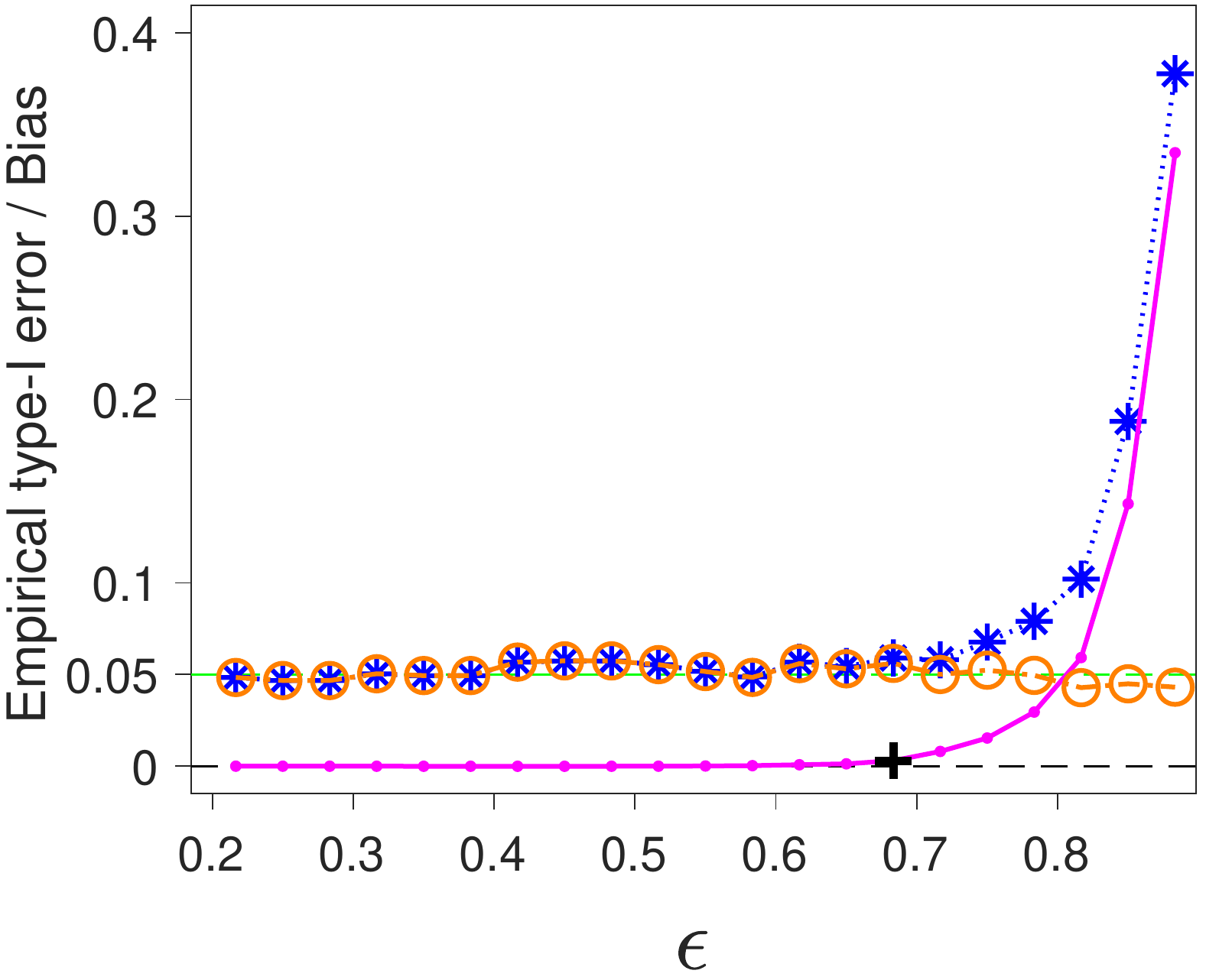}
%\centerline{(d)}
\caption{\quad (d) With the Bartlett correction}
\end{subfigure}
%\caption{Multiple-sample tests (IV)--(VI) when $n=500$. Column (a) without the Bartlett correction, the asymptotic bias in \eqref{eq:chisqapprox} (cross); (b) without the Bartlett correction and the asymptotic bias ; (c) (d) In all subfigures, empirical type-\RNum{1} error versus $\epsilon$ (asterisk); the difference between the asterisk line and the cross line (circle).}

%\caption{Multiple-sample tests (IV)--(VI) when $n_i=100$. In each row: (a) Without the Bartlett correction: empirical type-\RNum{1} error versus $\epsilon$ (asterisk); the asymptotic bias in \eqref{eq:chisqapprox} (cross); the difference between the asterisk line and the cross line (circle). (b) Without the Bartlett correction: empirical type-\RNum{1} error versus $\epsilon$ (asterisk); the maximum over the asymptotic biases in \eqref{eq:chisqapprox} and  \eqref{eq:normalbias1} (cross); the difference between the asterisk line and the cross line (circle). (c) With the Bartlett correction: empirical type-\RNum{1} error versus $\epsilon$ (asterisk); the asymptotic bias in \eqref{eq:chisqapproxbartcorr} (cross); the difference between the asterisk line and the cross line (circle). (d) With the Bartlett correction: empirical type-\RNum{1} error versus $\epsilon$ (asterisk); the maximum over the asymptotic biases in \eqref{eq:chisqapproxbartcorr} and \eqref{eq:normalbias2} (cross).} \label{fig:biasmultin100}
\caption{Multiple-sample tests (IV)--(VI) when $n=100$. Rows 1--3 present the results for tests (IV)--(VI), respectively. 
For four columns in each row, please see the caption description in Fig. \ref{fig:bias13n100}.} \label{fig:biasmultin100}
\end{figure}
\end{landscape}

%\quad
%\newpage

\begin{landscape}
\begin{figure}[!htbp]
\captionsetup[subfigure]{labelformat=empty}
\centering
%\centerline{Test (IV)}
%\begin{subfigure}[t]{0.33\textwidth}
%\centering
%\includegraphics[width=\textwidth]{(IV)Means1.pdf}
%%\caption{(a) }	
%\centerline{(a)}
%\end{subfigure}%
%~
\begin{turn}{90}
\begin{minipage}{0.26\textwidth}
 \hspace{4.5em} Test (IV)\vspace{0.12em} 
\end{minipage}
\end{turn}%
~ \
\begin{subfigure}[t]{0.34\textwidth}
\centering
\includegraphics[width=\textwidth]{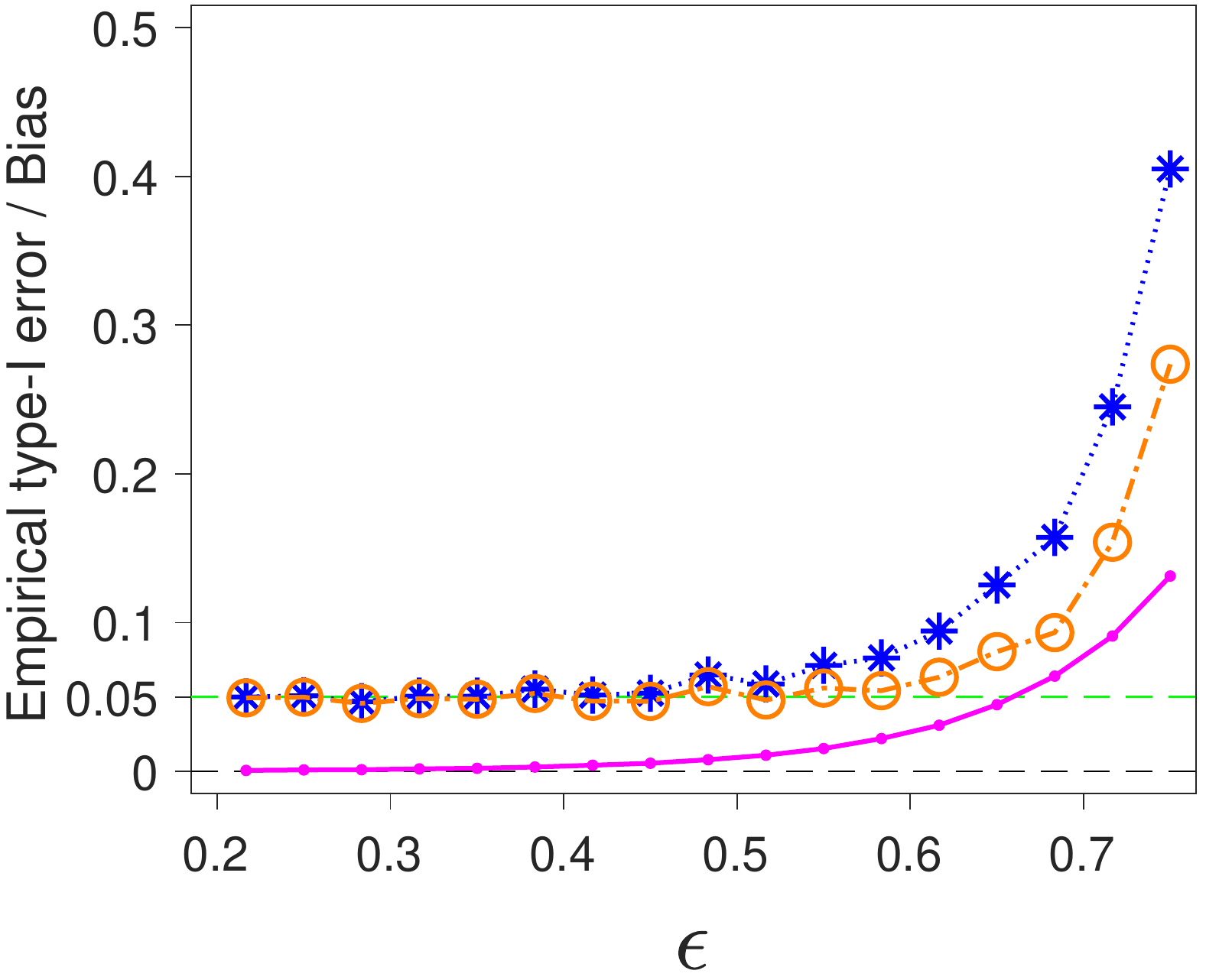}
%\centerline{(a)}
\end{subfigure}%
~ \ 
\begin{subfigure}[t]{0.34\textwidth}
\centering
\includegraphics[width=\textwidth]{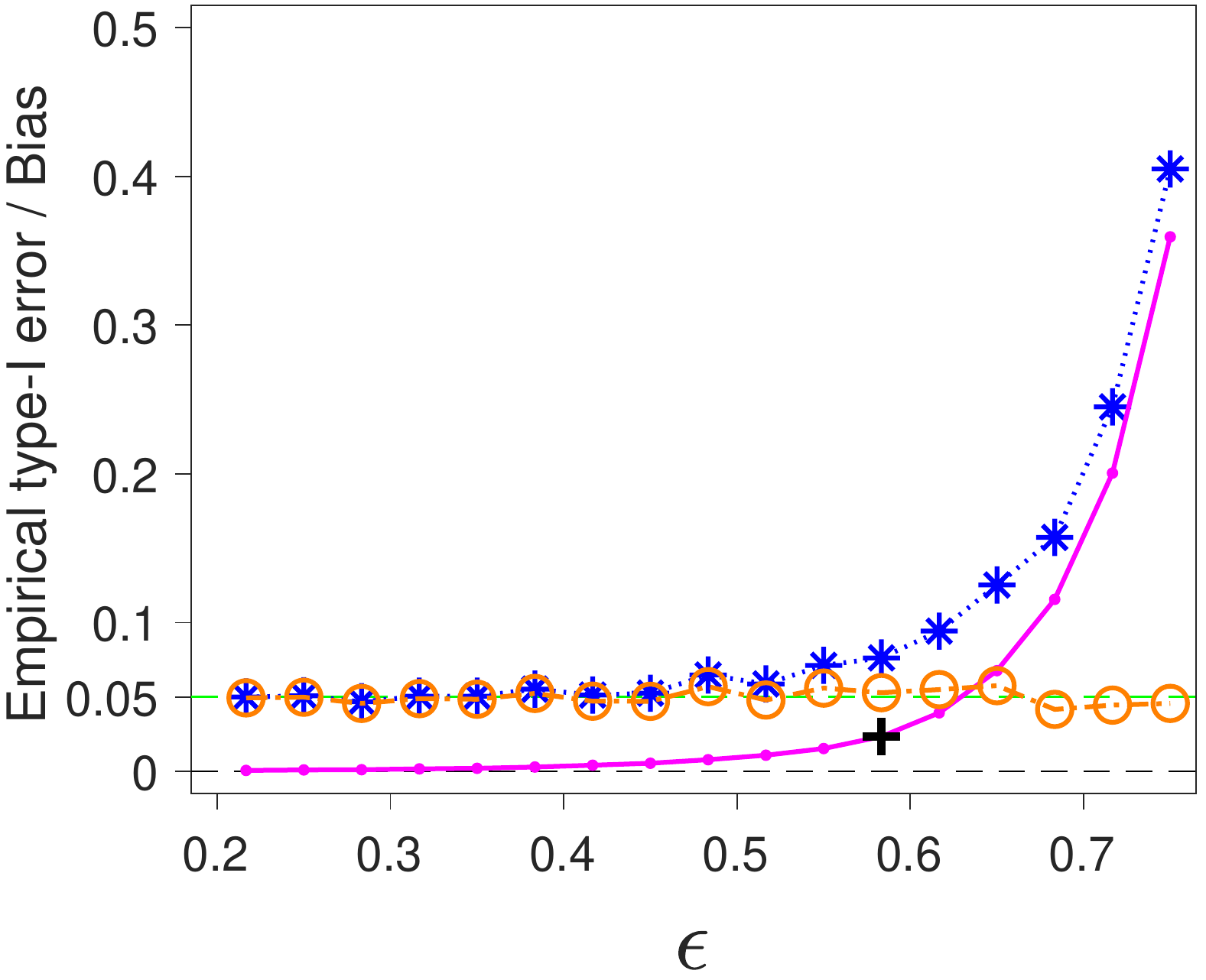}
%\centerline{(b)}
\end{subfigure}
~ \ 
\begin{subfigure}[t]{0.34\textwidth}
\centering
\includegraphics[width=\textwidth]{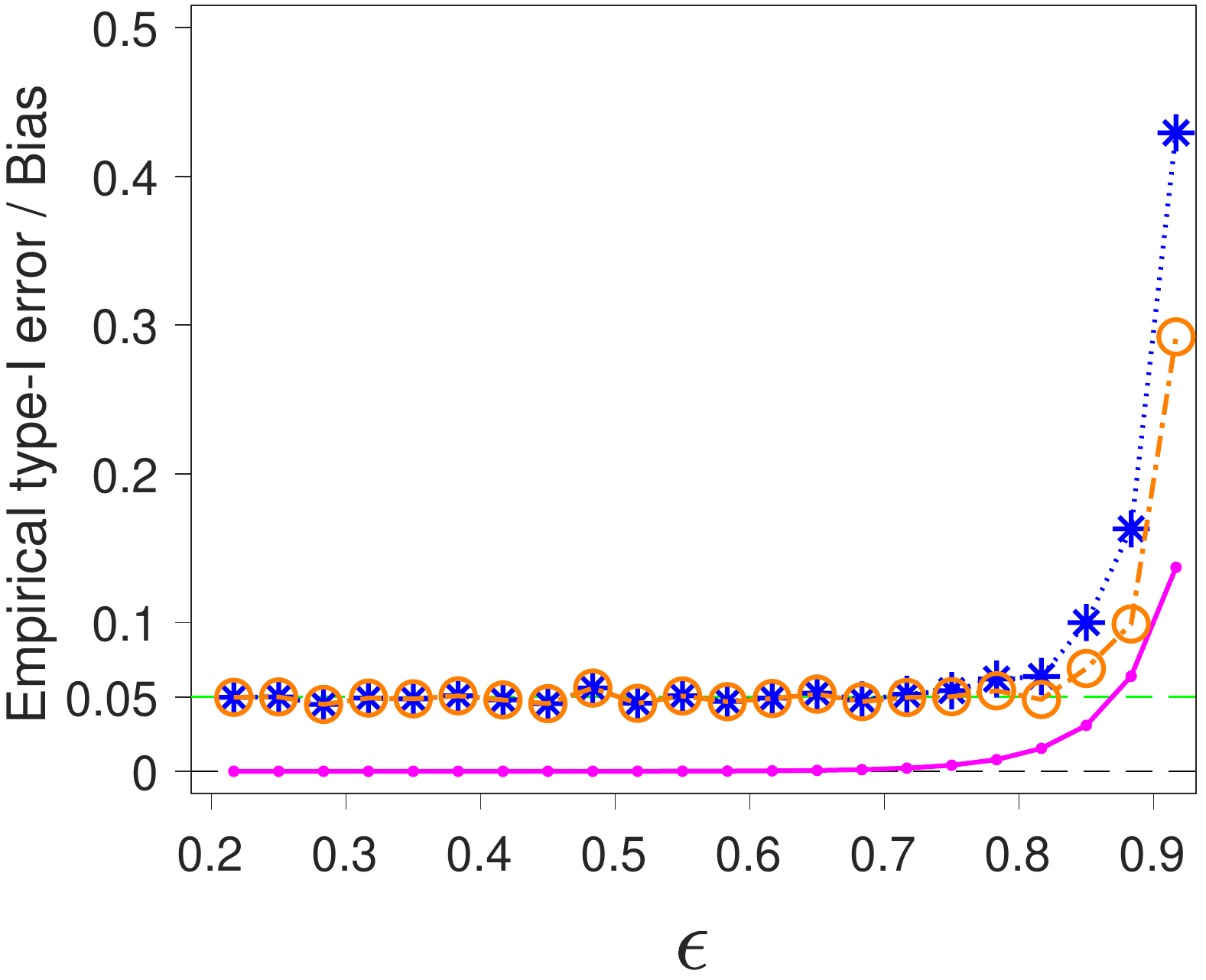}
%\centerline{(c)}
\end{subfigure}%
~ \ 
\begin{subfigure}[t]{0.34\textwidth}
\centering
\includegraphics[width=\textwidth]{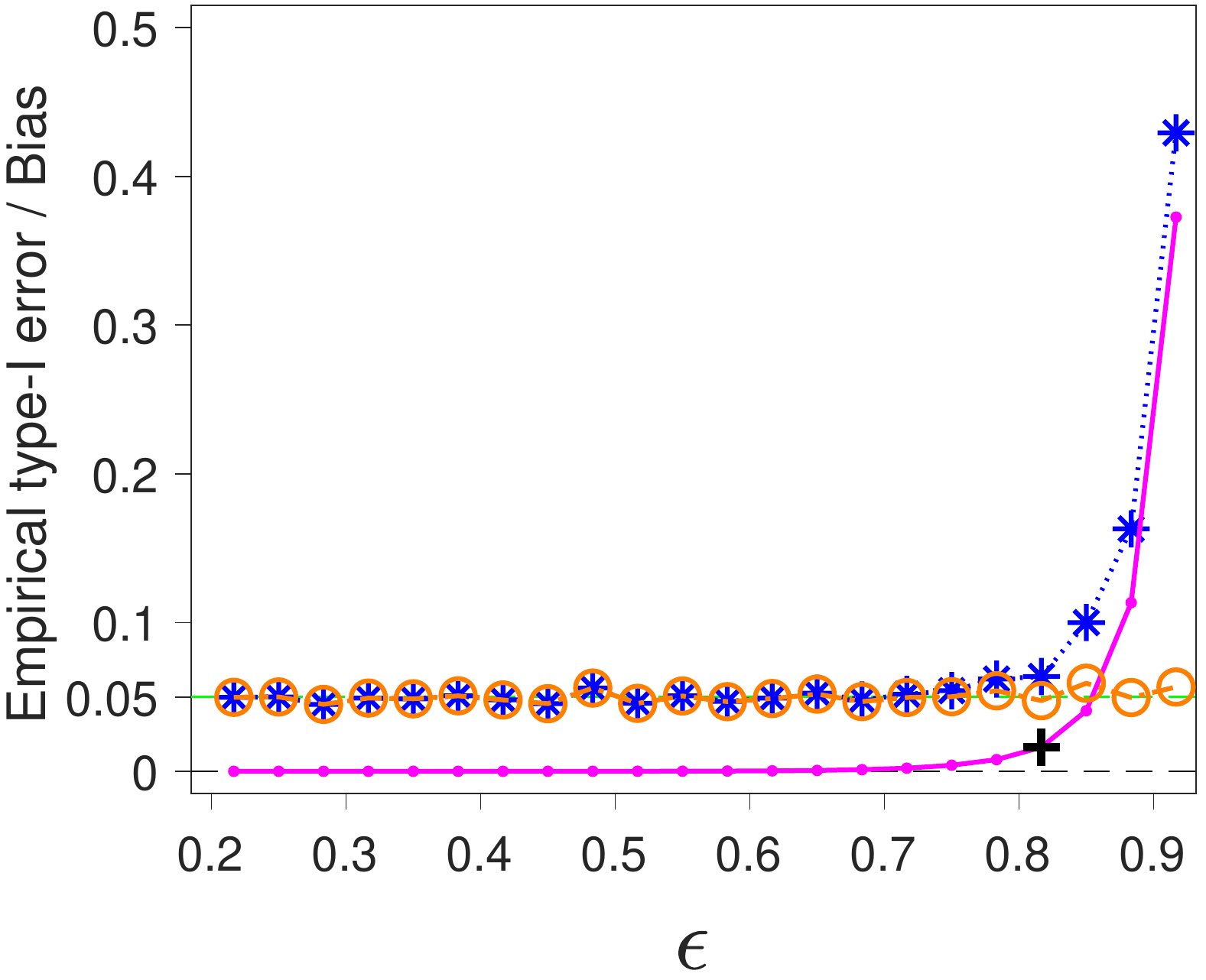}
%\centerline{(d)}
\end{subfigure}
\ \vspace{0.2em}

%\centerline{Test (V)}
%\centering
\begin{turn}{90}
\begin{minipage}{0.26\textwidth}
 \hspace{4.5em} Test (V)\vspace{0.12em} 
\end{minipage}
\end{turn}%
~ \
\begin{subfigure}[t]{0.34\textwidth}
\centering
\includegraphics[width=\textwidth]{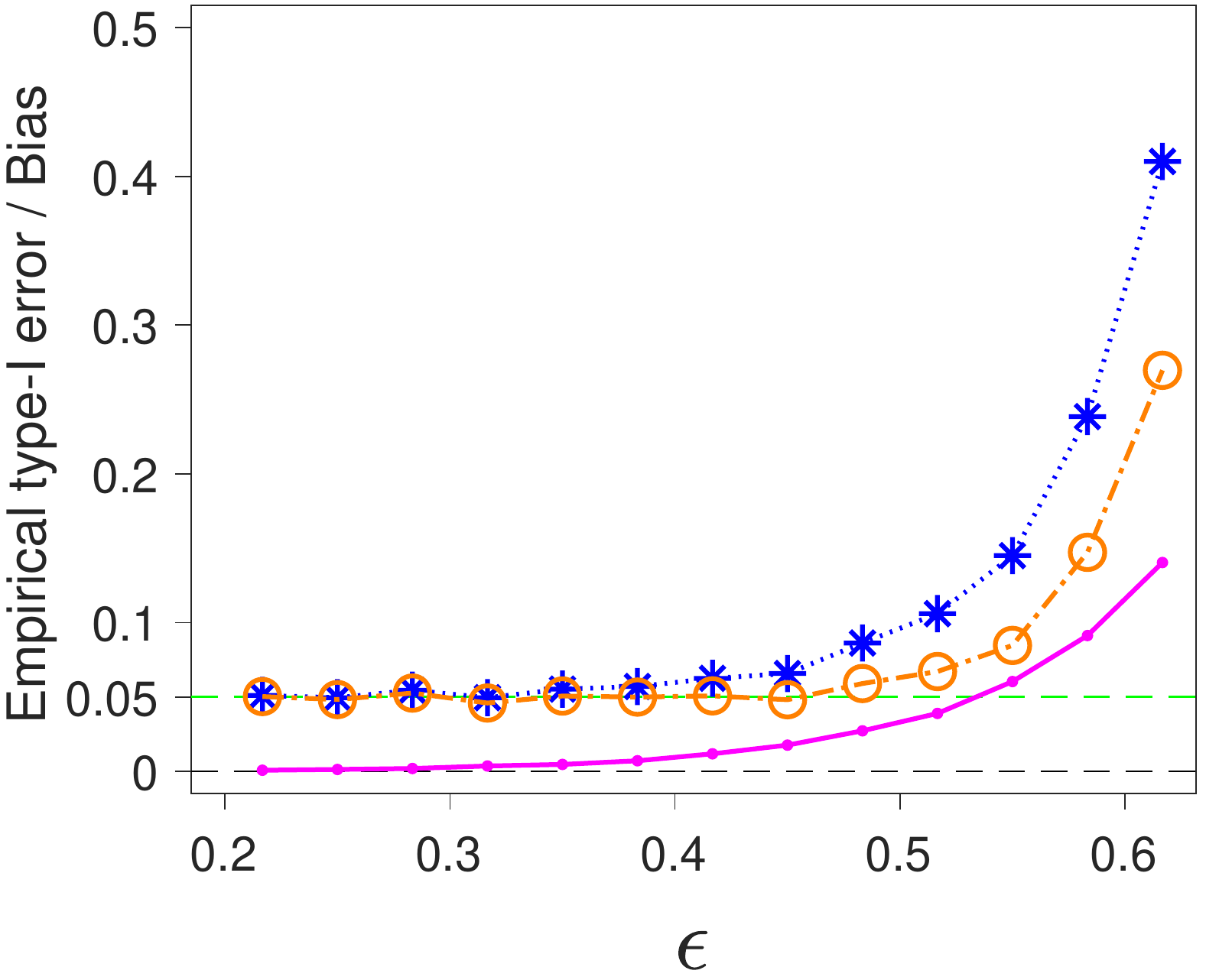}
%\centerline{(a)}
\end{subfigure}%
~ \ 
\begin{subfigure}[t]{0.34\textwidth}
\centering
\includegraphics[width=\textwidth]{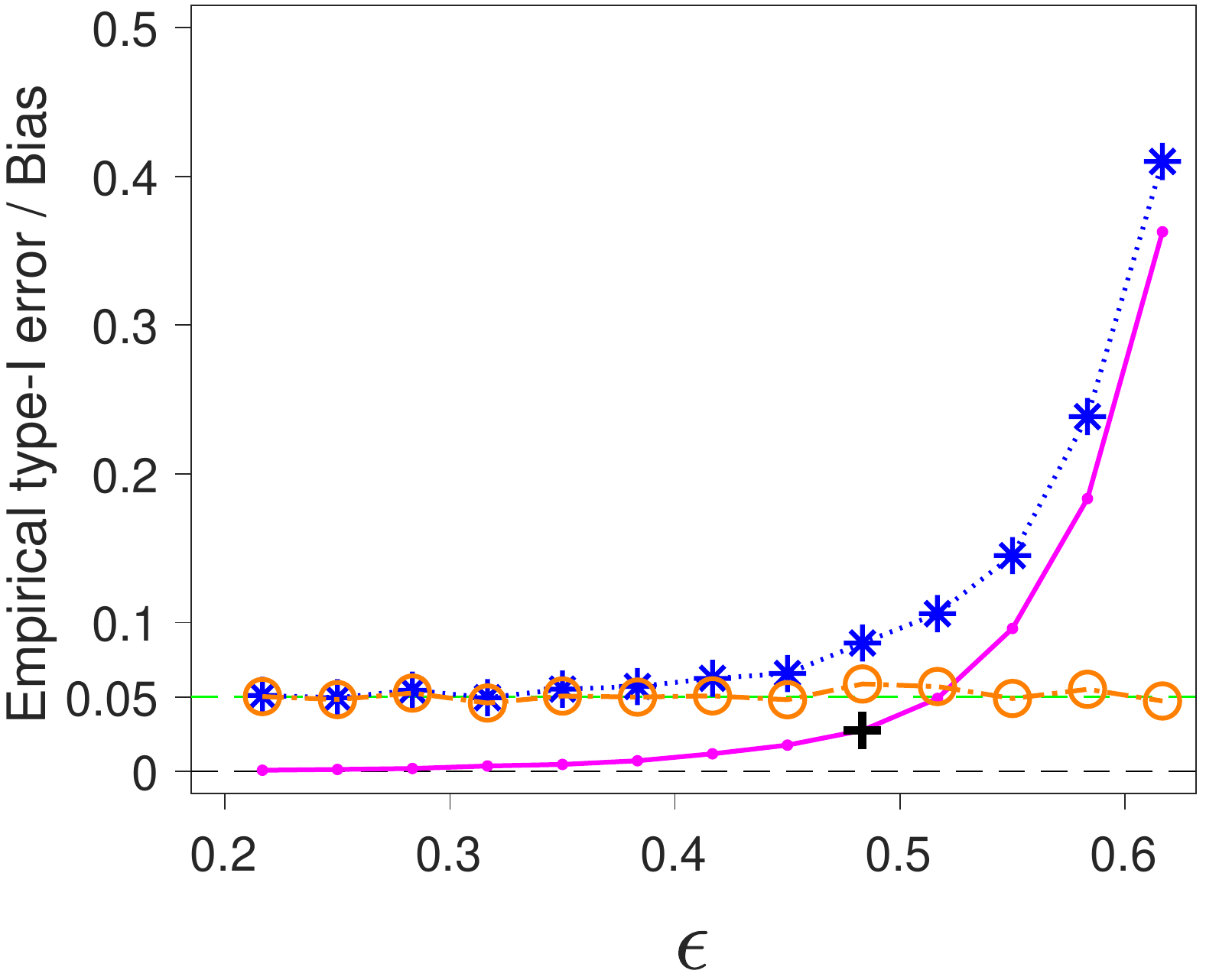}
%\centerline{(b)}
\end{subfigure}
~ \ 
\begin{subfigure}[t]{0.34\textwidth}
\centering
\includegraphics[width=\textwidth]{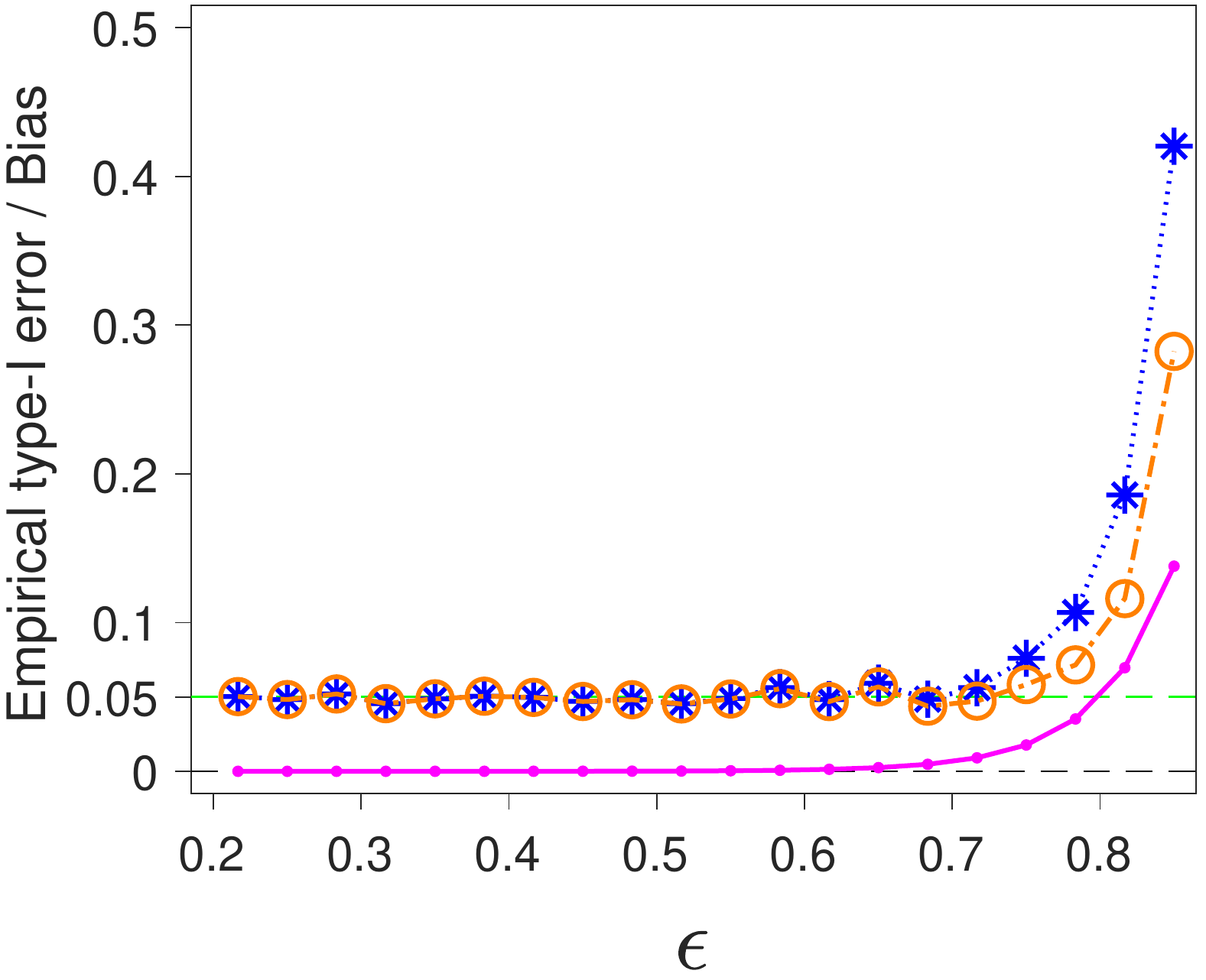}
%\centerline{(c)}
\end{subfigure}%
~ \ 
\begin{subfigure}[t]{0.34\textwidth}
\centering
\includegraphics[width=\textwidth]{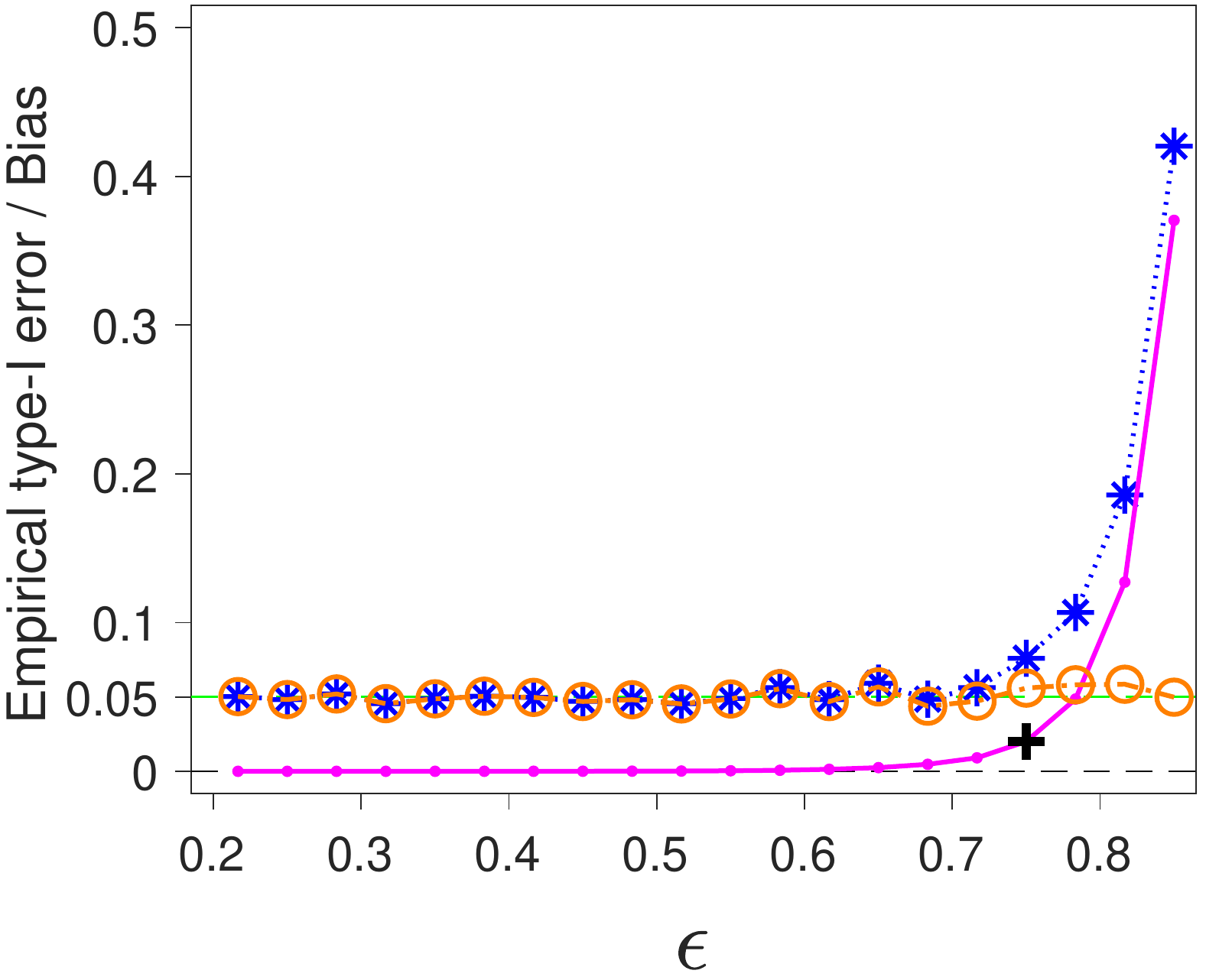}
%\centerline{(d)}
\end{subfigure}
\ \vspace{0.2em}

%\quad 
%\vspace{0.1em}

%\centerline{Test (VI)}
%\centering
\begin{turn}{90}
\begin{minipage}{0.26\textwidth}
 \hspace{4.5em} Test (VI)\vspace{0.12em}
\end{minipage}
\end{turn}%
~ \ 
\begin{subfigure}[t]{0.34\textwidth}
\centering	
\includegraphics[width=\textwidth]{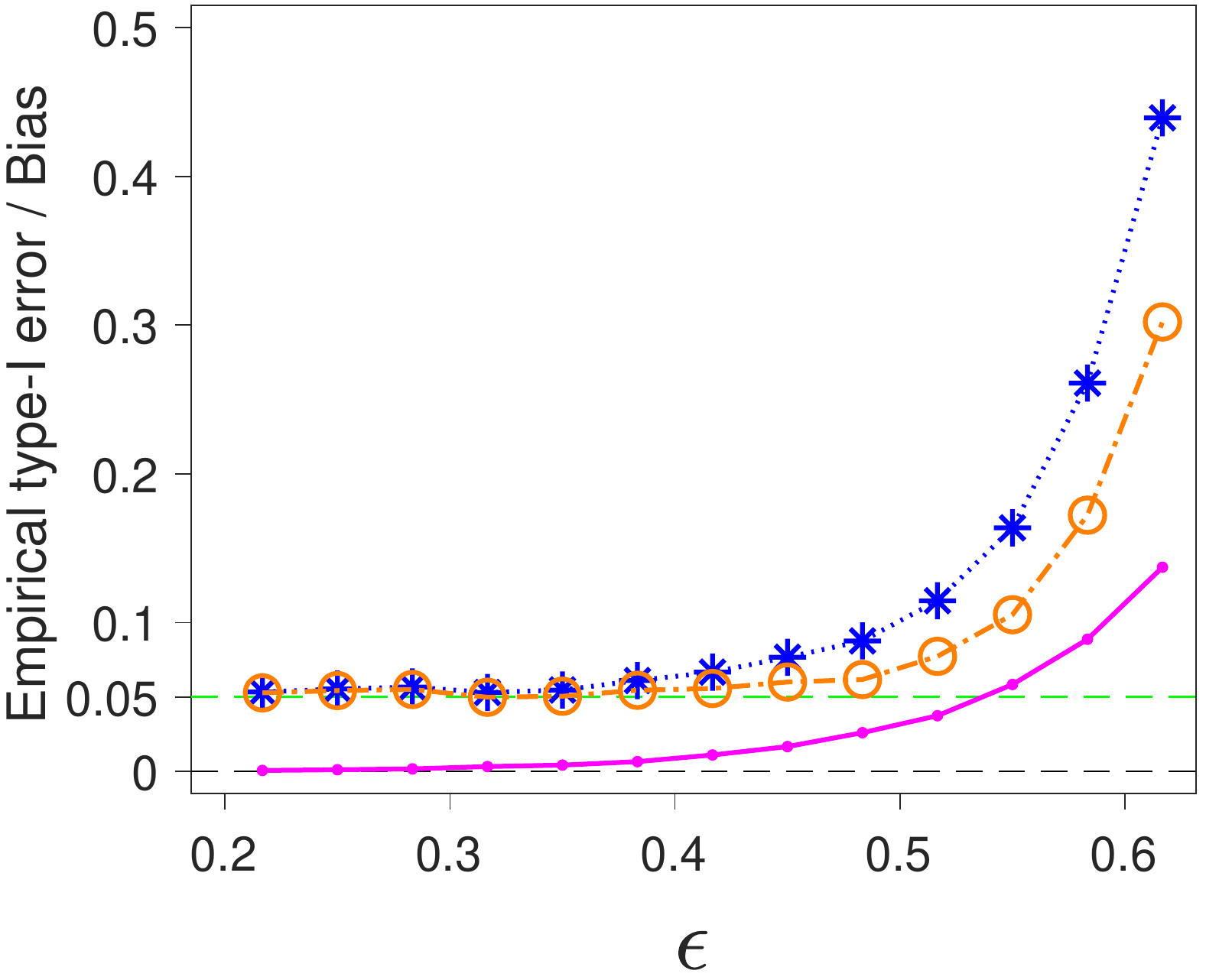}
%\centerline{(a)}
\caption{\quad (a)\ Without the Bartlett correction}
\end{subfigure}%
~ \ 
\begin{subfigure}[t]{0.34\textwidth}
\centering
\includegraphics[width=\textwidth]{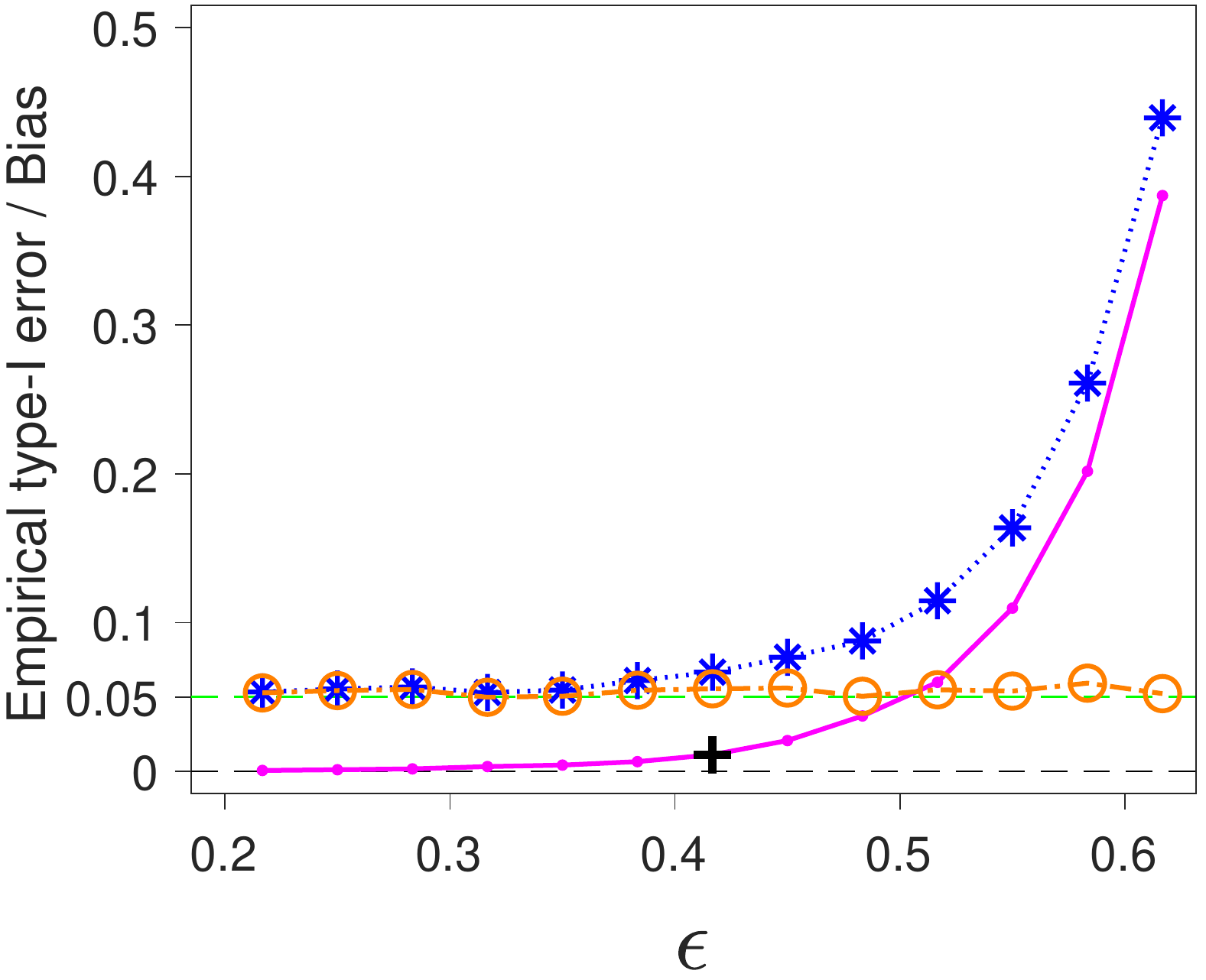}
%\centerline{(b)}
\caption{\quad (b)\ Without the Bartlett correction}
\end{subfigure}%
~ \ 
\begin{subfigure}[t]{0.34\textwidth}
\centering	
\includegraphics[width=\textwidth]{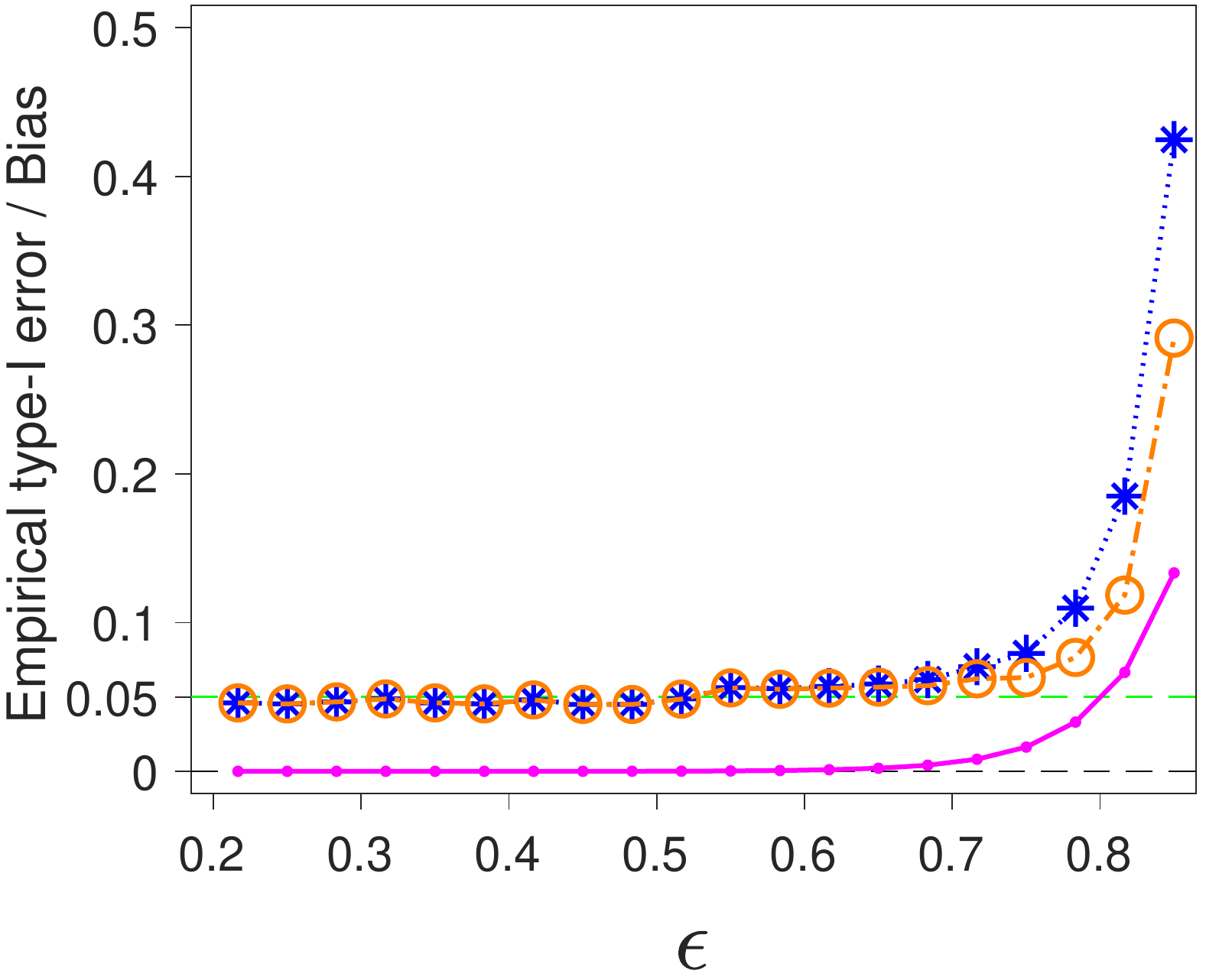}
%\centerline{(c)}
\caption{\quad (c) With the Bartlett correction}
\end{subfigure}%
~ \ 
\begin{subfigure}[t]{0.34\textwidth}
\centering
\includegraphics[width=\textwidth]{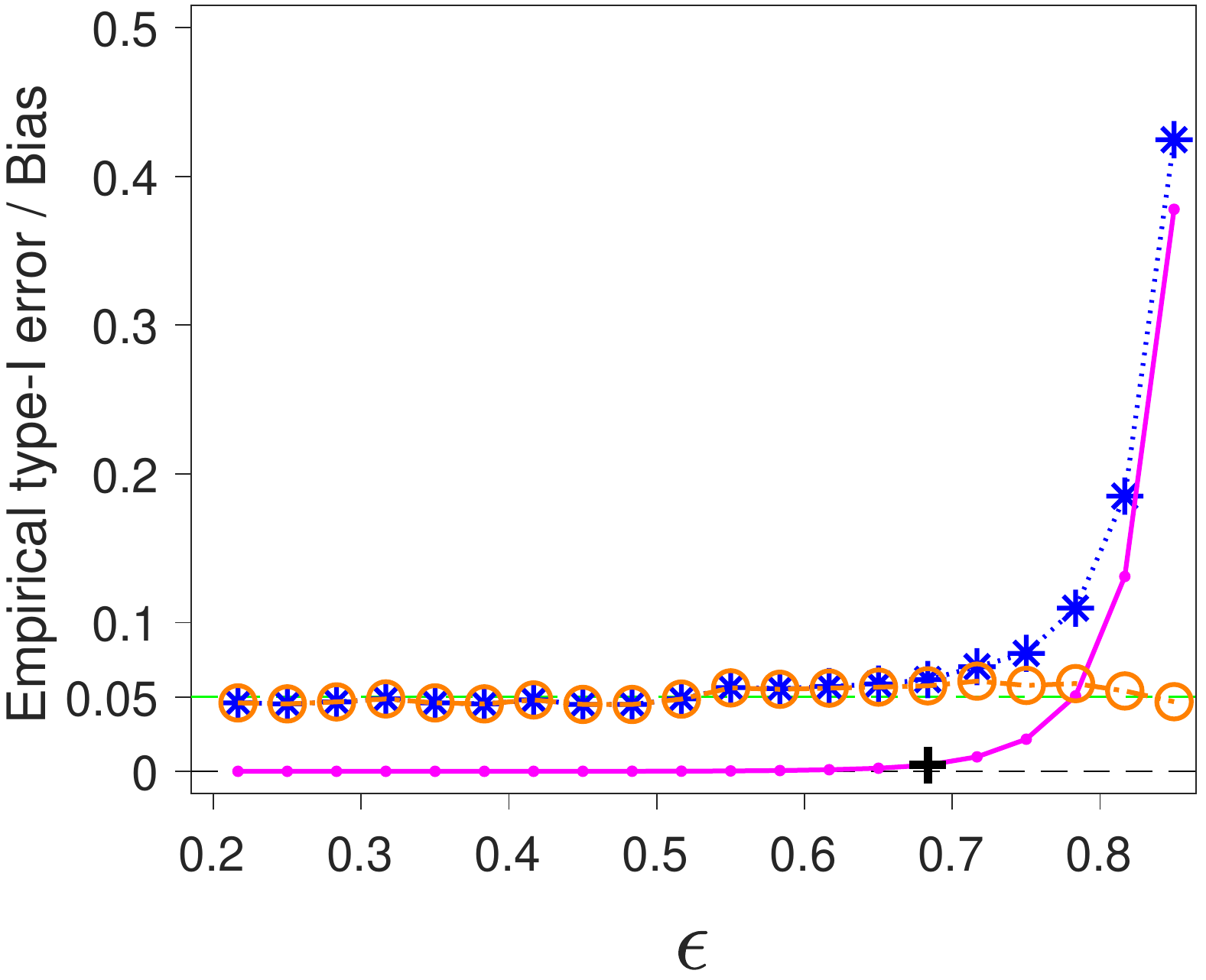}
%\centerline{(d)}
\caption{\quad (d) With the Bartlett correction}
\end{subfigure}
%\caption{Multiple-sample tests (IV)--(VI) when $n=500$. Column (a) without the Bartlett correction, the asymptotic bias in \eqref{eq:chisqapprox} (cross); (b) without the Bartlett correction and the asymptotic bias ; (c) (d) In all subfigures, empirical type-\RNum{1} error versus $\epsilon$ (asterisk); the difference between the asterisk line and the cross line (circle).}

%\caption{Multiple-sample tests (IV)--(VI) when $n_i=500$. In each row: (a) Without the Bartlett correction: empirical type-\RNum{1} error versus $\epsilon$ (asterisk); the asymptotic bias in \eqref{eq:chisqapprox} (cross); the difference between the asterisk line and the cross line (circle). (b) Without the Bartlett correction: empirical type-\RNum{1} error versus $\epsilon$ (asterisk); the maximum over the asymptotic biases in \eqref{eq:chisqapprox} and  \eqref{eq:normalbias1} (cross); the difference between the asterisk line and the cross line (circle). (c) With the Bartlett correction: empirical type-\RNum{1} error versus $\epsilon$ (asterisk); the asymptotic bias in \eqref{eq:chisqapproxbartcorr} (cross); the difference between the asterisk line and the cross line (circle). (d) With the Bartlett correction: empirical type-\RNum{1} error versus $\epsilon$ (asterisk); the maximum over the asymptotic biases in \eqref{eq:chisqapproxbartcorr} and \eqref{eq:normalbias2} (cross).}  \label{fig:biasmultin500}
\caption{Multiple-sample tests (IV)--(VI) when $n=500$. Rows 1--3 present the results for tests (IV)--(VI), respectively. 
For four columns in each row, please see the caption description in Fig. \ref{fig:bias13n100}.}\label{fig:biasmultin500}
%\caption{Simulations on the phase transition boundary and asymptotic biases}
%\caption{Test (V): \ (V.a) Empirical type-\RNum{1} error versus $\epsilon$ with $n=100$ (cross), $500$ (asterisk), $1000$ (square), and $5000$ (triangle);   theoretical phase transition boundary (vertical dashed line; $\epsilon=2/3$ in Panel (A) and $\epsilon=4/5$ in Panel (B)). \ \ (V.b) Empirical type-\RNum{1} error versus $\epsilon$ with $n=500$ (asterisk); asymptotic bias in \eqref{eq:chisqapprox} (cross); the difference between the previous two lines (circle). \ \ (V.c) Empirical type-\RNum{1} error versus $\epsilon$ with $n=500$ (asterisk); the maximum over two asymptotic biases in \eqref{eq:chisqapprox} and \eqref{eq:normalbias1} (cross); the difference between the previous two lines (circle).}
\end{figure}
\end{landscape}

\begin{landscape}

\begin{figure}[!htbp]
\captionsetup[subfigure]{labelformat=empty}
\centering

\begin{subfigure}[t]{0.33\textwidth}
\centering
\includegraphics[width=\textwidth]{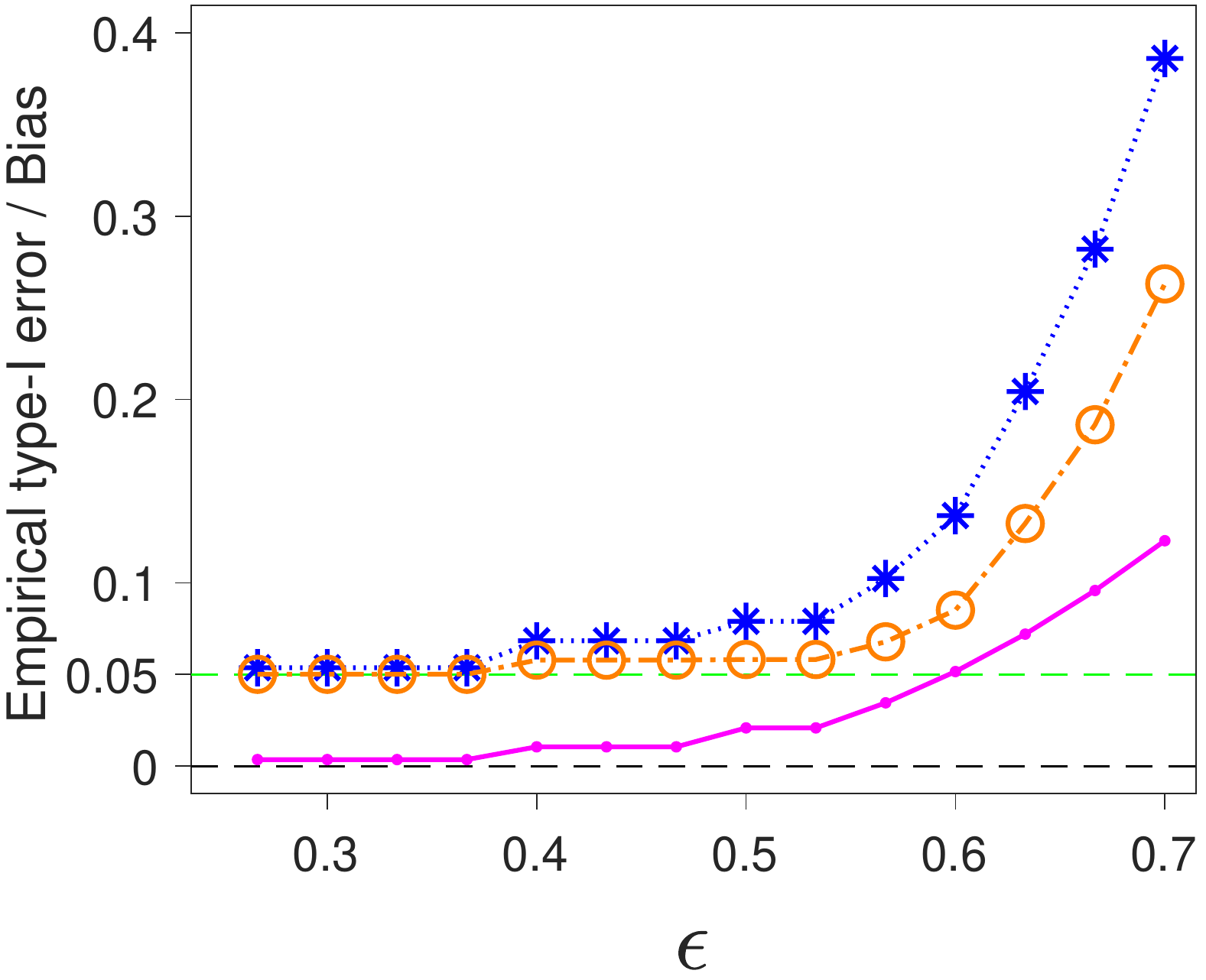}	
\caption{\quad (a) Without the Bartlett correction}
\end{subfigure}%
~\ 
\begin{subfigure}[t]{0.33\textwidth}
\centering
\includegraphics[width=\textwidth]{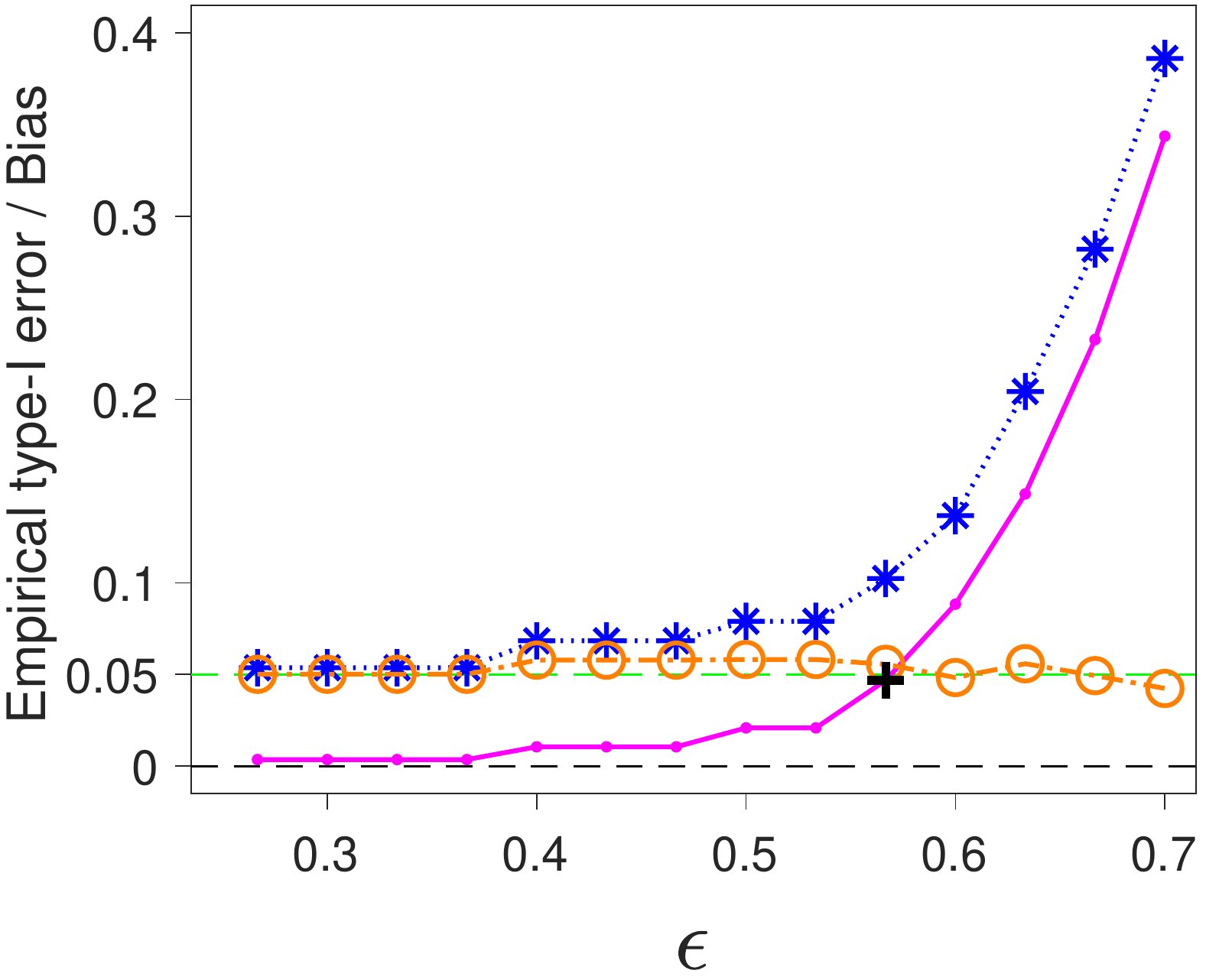}		
\caption{\quad (b) Without the Bartlett correction}
\end{subfigure}%
~\ 
\begin{subfigure}[t]{0.33\textwidth}
\centering
\includegraphics[width=\textwidth]{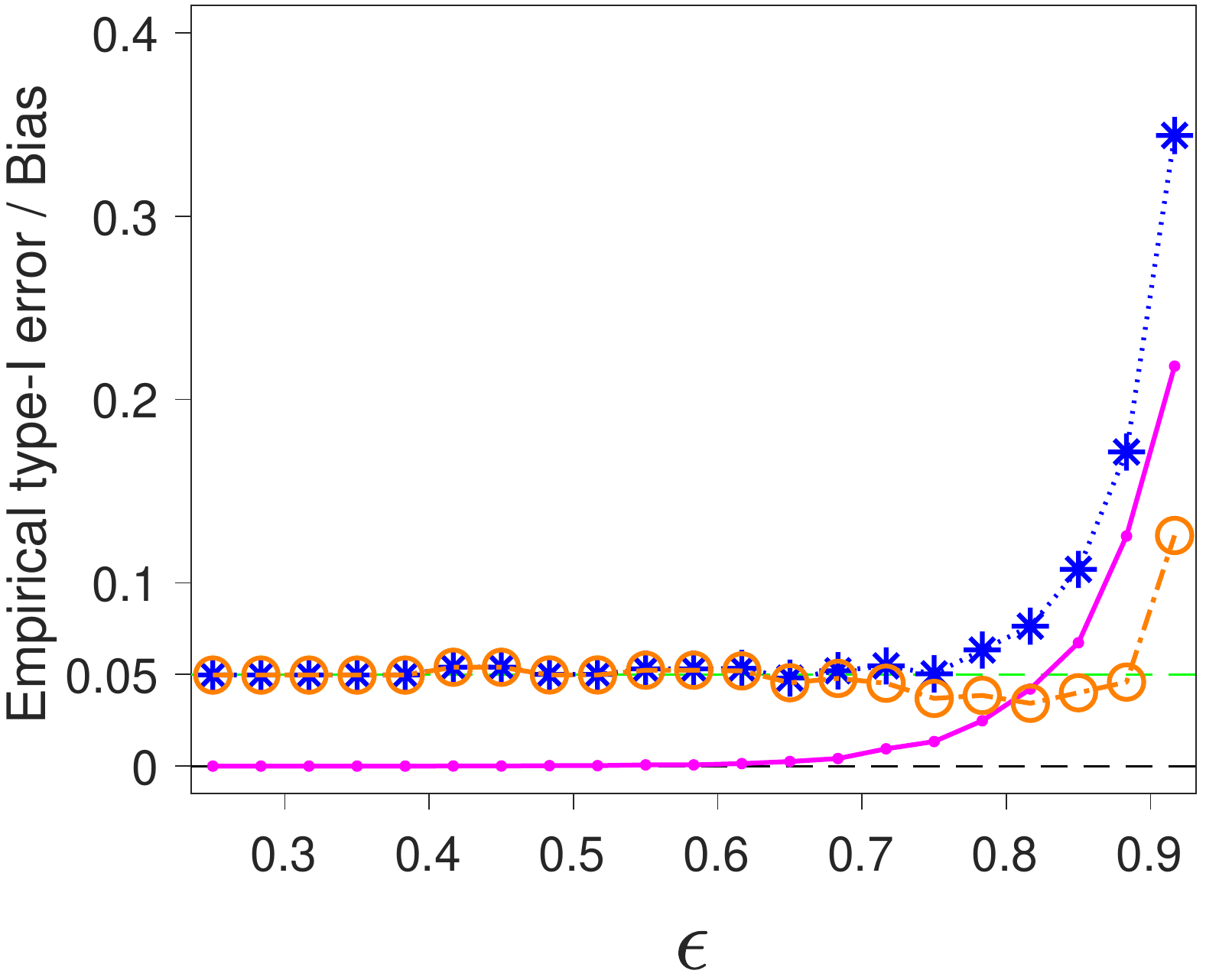}
\caption{\quad \ (c) With the Bartlett correction}
\end{subfigure}%
~\ 
\begin{subfigure}[t]{0.33\textwidth}
\centering
\includegraphics[width=\textwidth]{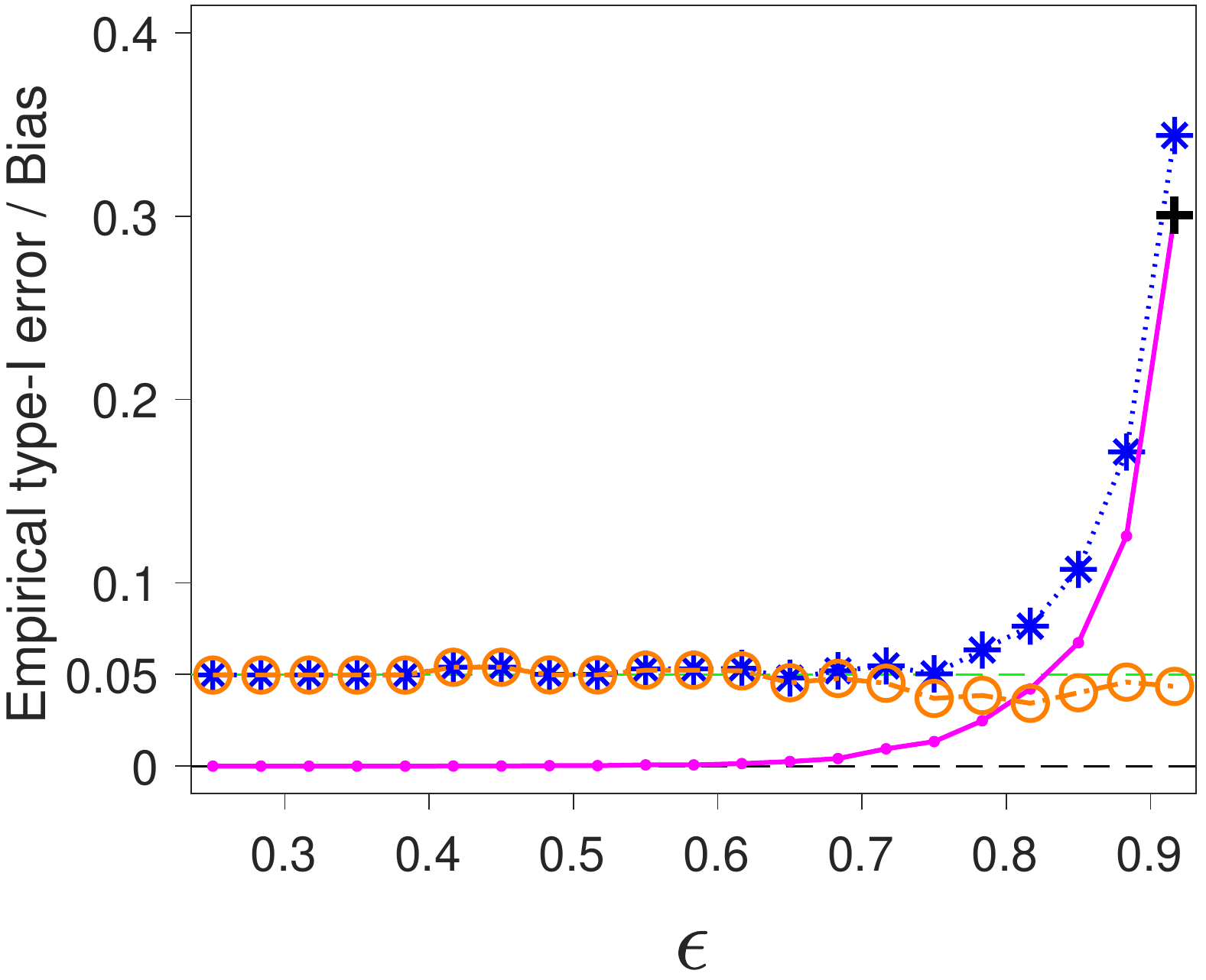}
\caption{\quad (d) With the Bartlett correction}
\end{subfigure}
\caption{Independence test (VII) when $n=100$: for columns (a)--(d),  please see the caption description in Fig. \ref{fig:bias13n100}.}\label{fig:bias7n100}
\end{figure}
	
\quad \\	
	
\begin{figure}[!htbp]
\captionsetup[subfigure]{labelformat=empty}
\centering

\begin{subfigure}[t]{0.33\textwidth}
\centering
\includegraphics[width=\textwidth]{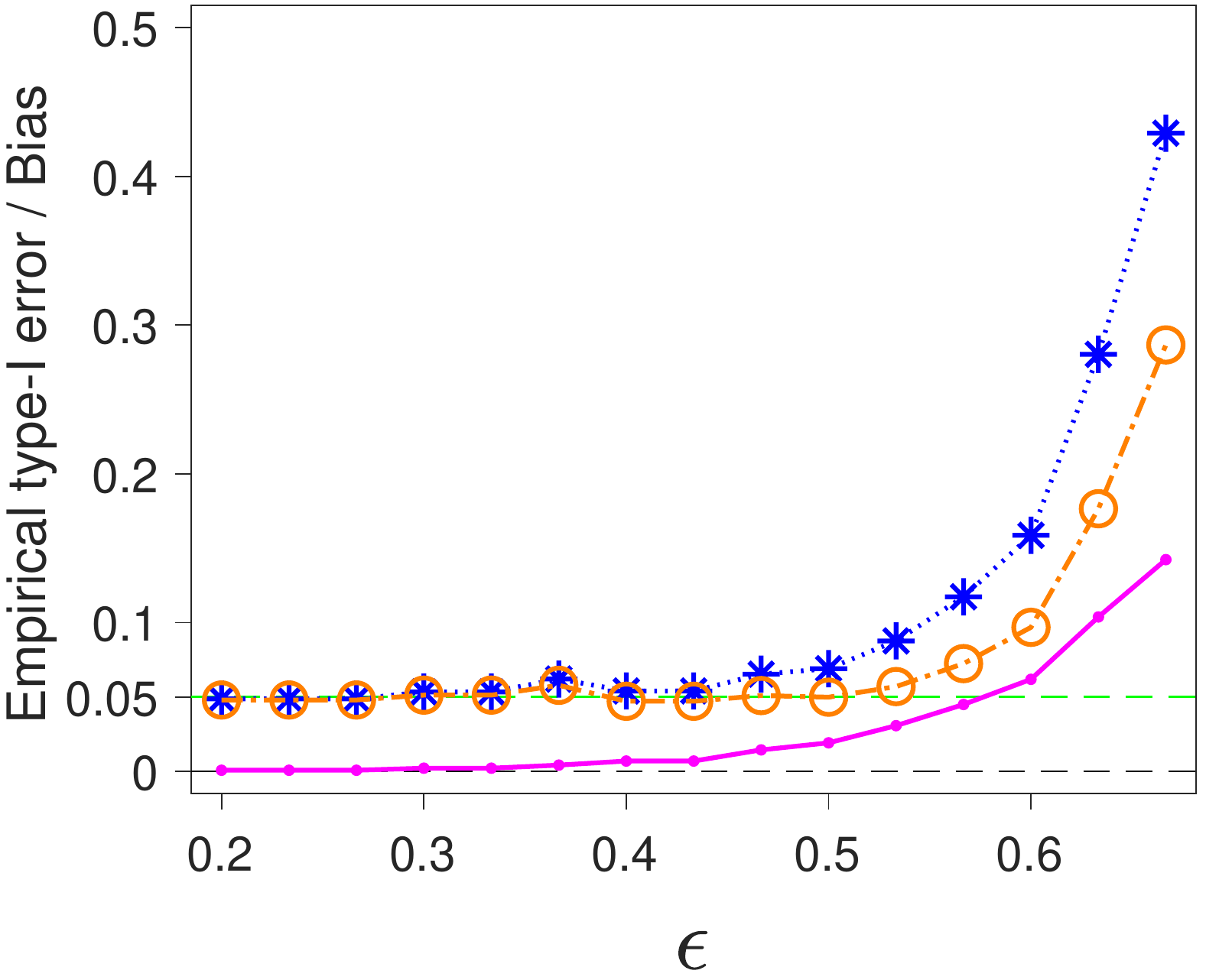}
%\caption{(b)}	
\caption{\quad (a) Without the Bartlett correction}
\end{subfigure}%
~\ 
\begin{subfigure}[t]{0.33\textwidth}
\centering
\includegraphics[width=\textwidth]{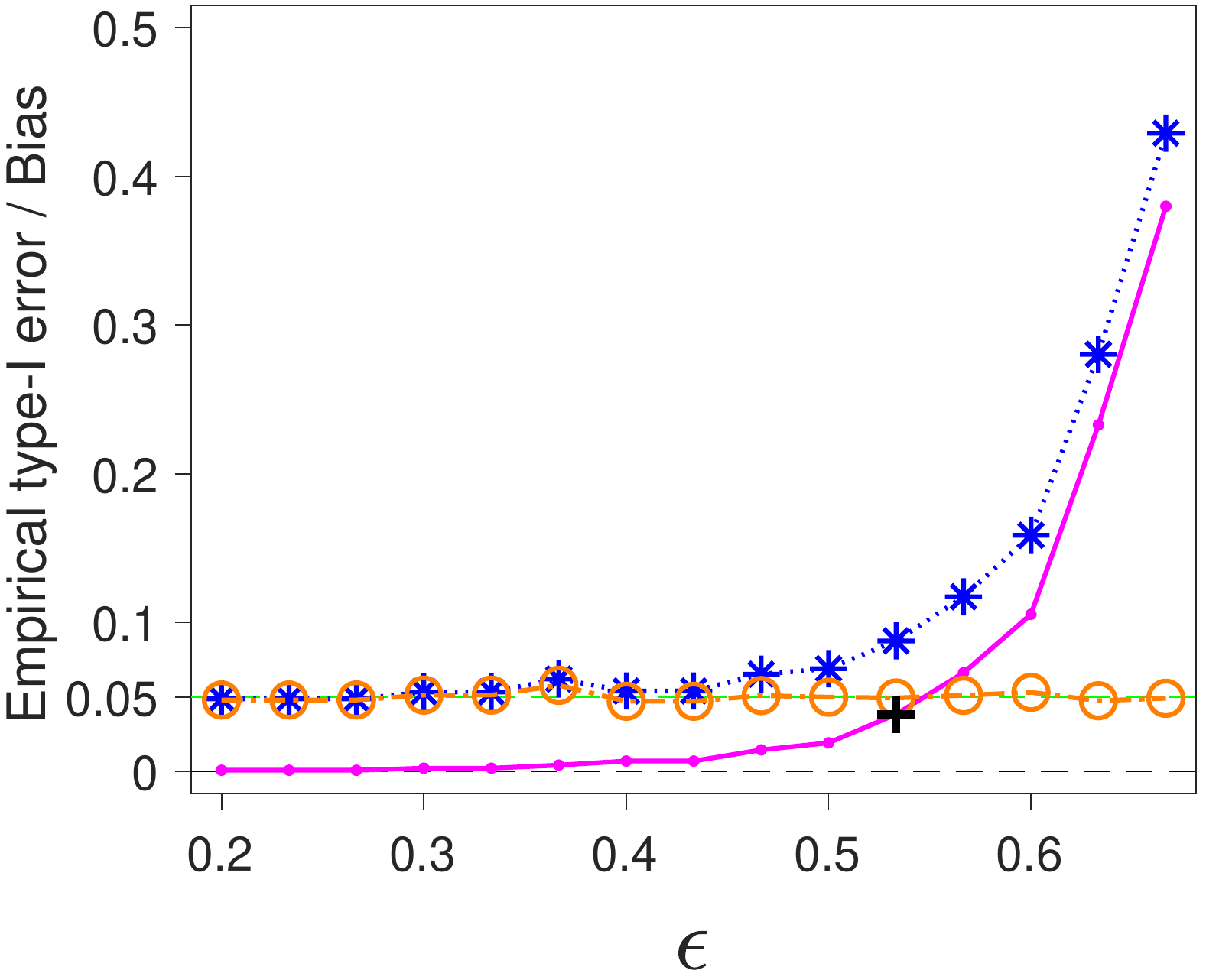}
%\caption{(c)}		
\caption{\quad (b) Without the Bartlett correction}
\end{subfigure}%
~\ 
\begin{subfigure}[t]{0.33\textwidth}
\centering
\includegraphics[width=\textwidth]{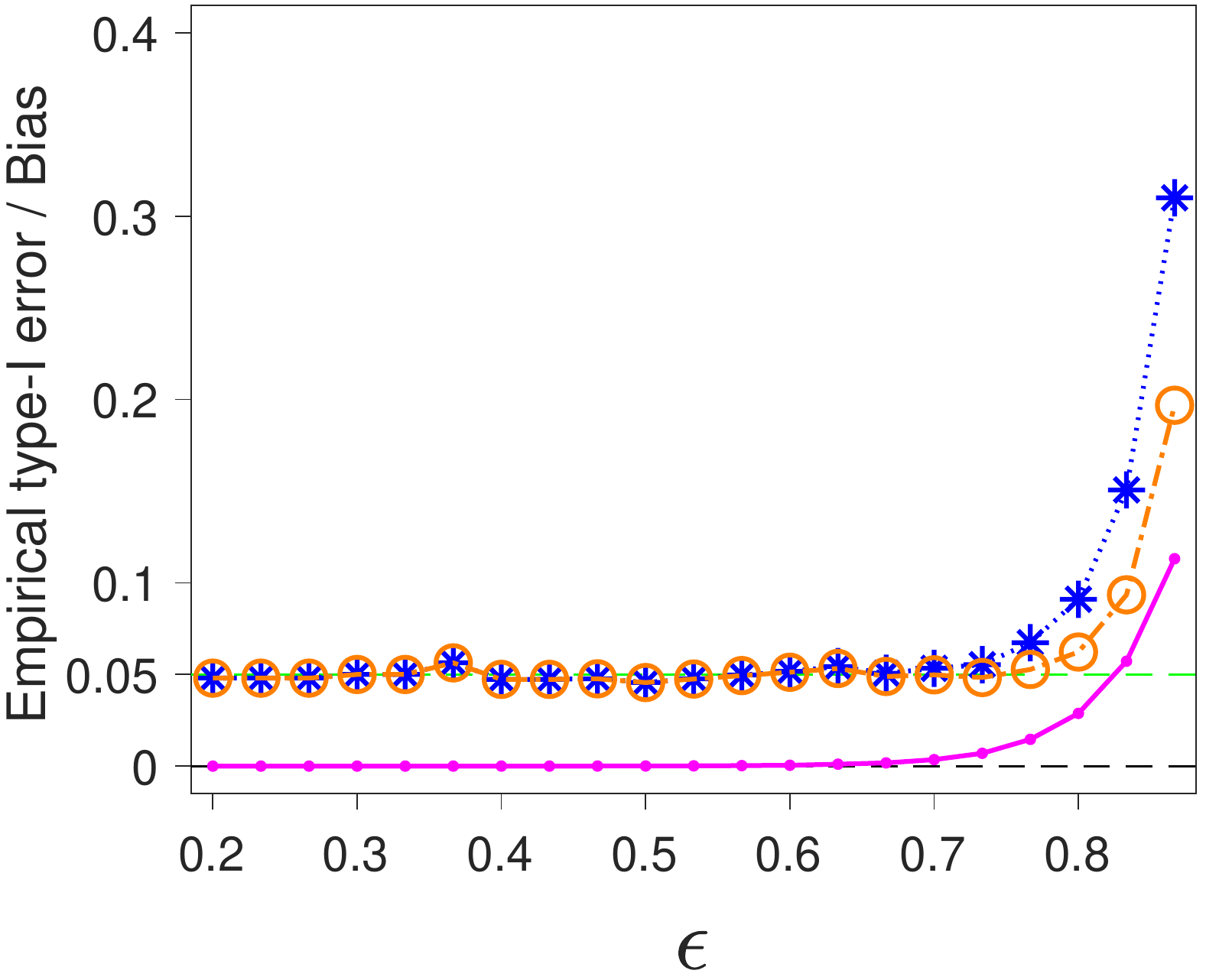}
%\caption{(b)}	
\caption{\quad \ (c) With the Bartlett correction}
\end{subfigure}%
~\ 
\begin{subfigure}[t]{0.33\textwidth}
\centering
\includegraphics[width=\textwidth]{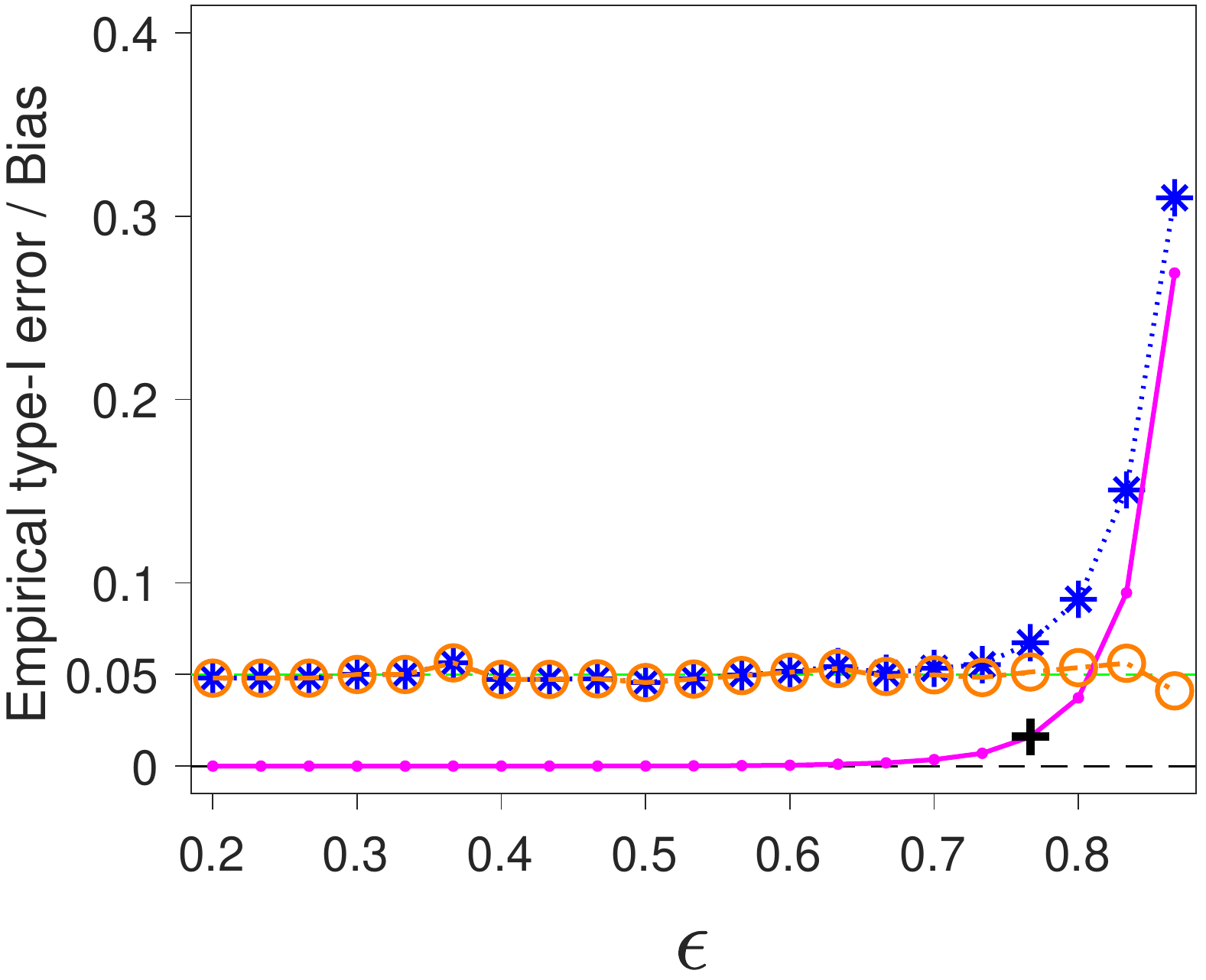}
%\caption{(c)}		
\caption{\quad (d) With the Bartlett correction}
\end{subfigure}
\caption{Independence test (VII) when $n=500$: for columns (a)--(d), please see the caption description in Fig. \ref{fig:bias13n100}.}\label{fig:bias7n500}
\end{figure}

\end{landscape}

\newpage

\section{Proof Illustration with Problem (III)}\label{sec:pfthemexam}

In this section, we illustrate the proofs of Theorems \ref{thm:onesam}--\ref{thm:onesamnormal} by focusing on  the testing problem (III), which jointly tests the the one-sample mean vector and covariance matrix. 
Other testing problems (I)--(II) and (IV)--(VII) can be proved following a similar analysis, and are discussed in Section \ref{sec:pfthms}. 
We define some notation to facilitate the proofs. 
For two sequences of numbers $\{a_n;n\geq1 \}$ and $\{ b_n;n\geq1\}$, $a_n=O(b_n)$ denotes $\limsup_{n\to \infty} |a_n/b_n| < \infty$;  $a_n=o(b_n)$ denotes $\lim_{n\to \infty} a_n/b_n=0$;  $a_n=\Theta(b_n)$ represents that $a_n=O(b_n)$ and $b_n=O(a_n)$ hold simultaneously; $a_n\sim b_n$ denotes $\lim_{n\to \infty} |a_n/b_n| = 1$.

\subsection{Proof of Theorem \ref{thm:onesam} (III)} \label{sec:pfonesam3}

%In particular, we first present the details of the proof of problem (III), which jointly tests the the one-sample mean vector and covariance matrix. The proofs of problems (I) and (II), which test the mean vector and covariance matrix separately,   follow from a similar analysis, and are shown in the following sections.  

When $p$ is fixed, the chi-squared approximations hold by the classical multivariate analysis \citep{anderson2003introduction,Muirhead2009}. Therefore, without loss of generality, the proofs below focus on $p\to \infty$.

Deriving the necessary and sufficient conditions for the chi-squared approximations requires the correct understanding of the limiting behavior of $\log \Lambda_n$ under both low and high dimensions. 
Particularly, we examine the limiting distribution of the log likelihood ratio test statistic $\log \Lambda_n$ based on the moment generating function of $\log \Lambda_n$, that is, $\mathrm{E}\{\exp(t\log \Lambda_n)\}$.
%Particularly, 
%%Under low dimensions, 
%the limiting distribution of the log likelihood ratio test statistic $\log \Lambda_n$ has been  studied based on the moment generating function of $\log \Lambda_n$, that is, $\mathrm{E}\{\exp(h\log \Lambda_n)\}$. 
For $\Lambda_n$ in question (III),  by Theorem 8.5.3  and  Corollary 8.5.4 in \citet{Muirhead2009}, we have that under $H_0$, 
\begin{align}
	\mathrm{E}\{\exp(t\log \Lambda_n)\}=\mathrm{E}(\Lambda_n^t)=\left(\frac{2 e}{n}\right)^{n p t/2}(1+t)^{-n p(1+t)/2}\times \frac{\Gamma_{p}[\{n(1+t)-1\} / 2]}{\Gamma_{p}\{(n-1) / 2\}}, \label{eq:mmtjoint}
\end{align}where $\Gamma_p(\cdot)$ is the multivariate Gamma function; see  Definition 2.1.10 in  \citet{Muirhead2009}.

When $p$ is fixed, the moment generating function of $-2\log \Lambda_n$ approximates that of a chi-squared  variable $\chi^2_f$, where $f=  p(p+3)/2$; 
%and $\rho=1-\{6 n(p+3) \}^{-1}(2 p^{2}+9 p+11)$; 
see, Sections 8.2.4 and 8.5 in \cite{Muirhead2009}.
% To derive the necessary and sufficient condition for the chi-square approximations without/with the Bartlett  correction, it is equivalent to derive the conditions for 
%\begin{align}
%	\Pr\{ -2\log \Lambda_n >\chi_f^2(\alpha)\} \to \alpha, \label{eq:chisqjoint} \\
%	\Pr\{ -2 \rho \log \Lambda_n >\chi_f^2(\alpha)\} \to \alpha.  \label{eq:chisqcorjoint}	
%\end{align}
% \citep[e.g., Theorem 8.5.5][]{Muirhead2009},
%Therefore to derive the 
When $p\to \infty$, \citet{Jiang13} and \cite{Jiang15} derived an approximate expansion of the multivariate Gamma function, and their Theorem 5 utilized  \eqref{eq:mmtjoint} to show that  under the conditions of Theorem \ref{thm:onesam},
\begin{align}
&\mathrm{E}[\exp\{s(-2\log \Lambda_n +2\mu_n)/(2n\sigma_n)\}] \to \exp(s^2/2), \label{eq:mmtnormal1}
\end{align}where $\exp(s^2/2)$ is the moment generating function of $\mathcal{N}(0,1)$, and
\begin{align}
&\mu_{n}=-\frac{1}{4}\left\{n(2 n-2 p-3) \log \left(1-\frac{p}{n-1}\right)+2(n+1) p\right\},\label{eq:munthm21} \\
&\sigma_{n}^{2}=-\frac{1}{2}\left\{\frac{p}{n-1}+\log \left(1-\frac{p}{n-1}\right)\right\}. \label{eq:sigmanthm21} 
\end{align} 
%It follows that $T_{n,p}\equiv (-2\log \Lambda_n+2\mu_n)/(2n\sigma_n)$ converges to $\mathcal{N}(0,1)$ in distribution.  

%To prove (i) in Theorem \ref{thm:onesam},
%we note that 
We next prove (i) in Theorem \ref{thm:onesam} when $p\to \infty$ based on \eqref{eq:mmtnormal1}. Particularly, we write
\begin{align}
\hspace{-0.5em}\sup_{\alpha\in (0,1)}\big| \MYPR\{-2\log \Lambda_n > \chi^2_f(\alpha)\} - \alpha \big|=	\sup_{\alpha\in (0,1)}\Big|\MYPR( T_{n}>q_{n,\alpha})-\bar{\Phi}(q_{n,\alpha})+\bar{\Phi}(q_{n,\alpha}) - \bar{\Phi}(z_{\alpha})\Big|, \label{eq:supprobdiff} 
\end{align}
where 
$T_{n}=(-2\log \Lambda_n+2\mu_n)/(2n\sigma_n)$, 
$q_{n,\alpha}=\{\chi_f^2(\alpha)+2\mu_n\}/(2n\sigma_n)$, and  
$\bar{\Phi}(\cdot)=1-\Phi(\cdot)$ with $\Phi(\cdot)$ being the cumulative distribution function of  $\mathcal{N}(0,1)$. 
Since \eqref{eq:mmtnormal1}  suggests that $T_{n}$ converges to $\mathcal{N}(0,1)$ in distribution, 
and the cumulative distribution function of $\mathcal{N}(0,1)$ is continuous, 
by P\'{o}lya-Cantelli Lemma (see, e.g., Lemma 2.11 in  \cite{van2000asymptotic}),   
we have $\sup_{\alpha \in (0,1)}|\MYPR(T_{n}> q_{n,\alpha})-\bar{\Phi}(q_{n,\alpha})|\to 0$.    
Consequently, $\eqref{eq:supprobdiff} \to 0$ if and only if $\sup_{\alpha\in (0,1)}|\bar{\Phi}(q_{n,\alpha}) - \bar{\Phi}(z_{\alpha}) |\to 0,$ which is equivalent to $\sup_{\alpha\in (0,1)}|q_{n,\alpha} - z_{\alpha}| \to 0$, as $\bar{\Phi}(\cdot)$ is a continuous and strictly decreasing function with bounded derivative.
Since $\chi_f^2$ can be viewed as a summation over $f$ independent $\chi_1^2$ variables, and $f\to \infty$ as $p\to \infty$,  
%When $p\to \infty$, we know $f\to \infty$, and then by applying Berry–Esseen  theorem to $\chi_f^2$ variable, we obtain 
we can apply Berry–Esseen  theorem to $\chi_f^2$ variable, and obtain 
\begin{align}
	\sup_{\alpha\in (0,1)}\big| \{\chi^2_f(\alpha) - f\}/\sqrt{2f} - z_{\alpha}\big| = O(f^{-1/2}). \label{eq:xchisqorder}
\end{align}
%$\sup_{\alpha\in (0,1)}| \{\chi^2_f(\alpha) - f\}/\sqrt{2f} - z_{\alpha}| = O(f^{-1/2})$. 
%$\chi^2_f(\alpha)=f+\sqrt{2f}z_{\alpha}+O(1)$. 
%$q_{\alpha}=\{\sqrt{2f}z_{\alpha}+f+2\mu_n+O(1)\}/(2n\sigma_n)$. 
Therefore, $\sup_{\alpha \in (0,1)}|q_{n,\alpha} - z_{\alpha}|\to 0$  is equivalent to 
%Moreover, when $p\to \infty$, we have $f\to \infty$ and then $\mathrm{E}[ \exp\{ s(\chi_f^2-f)/\sqrt{2f} \}] \to \exp(s^2/2)$ by applying the  central limit theorem to $\chi_f^2$.
%%, and we note that $\exp(t^2/2)$ is the moment generating function of $\mathcal{N}(0,1).$ 
%%$(\chi_f^2-f)/\sqrt{2f}\to \mathcal{N}(0,1)$. 
%If the chi-squared approximation for $\log \Lambda_n$ holds, we know $\mathrm{E}[ \exp\{ s(-2\log \Lambda_n- f)/\sqrt{2f} \}] \to \exp(s^2/2)$, which, given \eqref{eq:mmtnormal1}, is equivalent to 
\begin{align}
	\sqrt{2f}\times (2n\sigma_n)^{-1} \to 1, \label{eq:chisqconvgvar}\\
( O(1) + f+2\mu_n)\times (2n\sigma_n)^{-1} \to 0. \label{eq:chisqconvgmean}
\end{align} 
Following similar analysis, we know that under the conditions of Theorem \ref{thm:onesam}  and $p\to \infty$, for the  chi-squared approximation with the Bartlett correction, $\sup_{\alpha\in (0,1)}| \MYPR\{-2\rho \log \Lambda_n > \chi^2_f(\alpha)\} - \alpha|$  holds if and only if 
\begin{align}
{\sqrt{2f}}\times (2n\rho \sigma_n)^{-1} \to   1, \label{thm13} \\  
(O(1)+f+2\rho \mu_n)\times (2n \rho\sigma_n)^{-1} \to  0. \label{thm14}
\end{align}
We next examine \eqref{eq:chisqconvgvar}--\eqref{eq:chisqconvgmean} and \eqref{thm13}--\eqref{thm14} for the chi-squared  approximation without and with the Bartlett correction, respectively. 

%\begin{align*}
%\biggr|\MYPR\biggr\{ \frac{\chi^2_f-f}{\sqrt{2f}} > z_{\alpha} \biggr\}-	\MYPR\biggr\{ \frac{\chi^2_f-f}{\sqrt{2f}} > \frac{\chi^2_f(\alpha)-f}{\sqrt{2f}} \biggr\}\biggr|=O(f^{-1/2}).
%\end{align*}

\smallskip
\noindent \textit{(III.i) The chi-squared  approximation.} \quad
We next discuss two cases $\lim_{n\to \infty} p/n = 0$ and  $\lim_{n\to \infty} p/n = C \in (0,1]$, respectively. 

\smallskip
\noindent \textit{Case (III.i.1) $\lim_{n\to \infty} p/n = 0$.}\ Under this case,  we prove that \eqref{eq:chisqconvgvar} holds. As $\sqrt{2f}\sim p$, it is equivalent to show that $p/(2n\sigma_n)\to 1$. By Taylor's expansion of $\sigma_n^2$ in \eqref{eq:sigmanthm21}, we have
\begin{equation*}
2{\sigma_n^2}= - \frac{p}{n-1} - \log \left(1-\frac{p}{n-1} \right) = \frac{p^2}{2(n-1)^2} +o\left(\frac{p^2}{n^2} \right), 
\end{equation*}
  and therefore $\sqrt{2f}\times (2n\sigma_n)^{-1} \to 1.$ We next show that \eqref{eq:chisqconvgmean} holds if and only if $p^2/n\to 0$. Given  \eqref{eq:chisqconvgvar}  and $\sqrt{2f}\sim p$, \eqref{eq:chisqconvgmean} is equivalent to $(f+2\mu_n)/p\to 0$.
 By $p/n=o(1)$ and Taylor's expansion of $\log(1-x)$, for $\mu_n$ in \eqref{eq:munthm21}, we have
 \begin{align}
% 	&~ 4\mu_n/p\label{eq:munoverpthm1} \\
 	 4\mu_n/p=&~ -2(n+1)+{n(2n-2p-3)}\left\{\frac{1}{n-1} + \frac{p}{2(n-1)^2} + \frac{p^2}{3(n-1)^3} + O\left(\frac{p^3}{n^4} \right) \right\}	\label{eq:munoverpthm1} \\
 	=&~ - 2(n+1) + (2n-2p-3)\left\{1+\frac{p}{2(n-1)}+\frac{p^2}{3(n-1)^2}\right\}+2+o(1)+O\left(\frac{p^3}{n^2}\right)  \notag \\
 	=&~ -2p-3+ \frac{(2n-2p-3)p}{2(n-1)} + \frac{(2n-2p-3)p^2}{3(n-1)^2} + o(1)+O\left(\frac{p^3}{n^2}\right). \notag 
 \end{align}
 As $2f/p=p+3$, we obtain
\begin{align}
2\times  (f+2\mu_n)/p
%= &~ (p+3)-2(n+1)+ n(2n-2p-3)\left\{\frac{1}{n-1}+\frac{p}{2(n-1)^2}+\frac{p^2}{3(n-1)^3}\right\}+O\left(\frac{p^3}{n^2} \right) \notag \\
%=&~ p-2n+1 + (2n-2p-3)\left\{1+\frac{p}{2(n-1)}+\frac{p^2}{3(n-1)^2}\right\}+\frac{2n-2p-3}{n-1}+o(1)+O\left(\frac{p^3}{n^2} \right)\notag \\
=&~ -p+\frac{\{2(n-1)-2p-1\}p}{2(n-1)} + \frac{2p^2}{3(n-1)}+ o(1) +O\left(\frac{p^3}{n^2} \right) \label{eq:expa1thm1}  \\
=&~ -\frac{p^2}{3(n-1)} + o(1) + O\left(\frac{p^3}{n^2}\right). \notag
\end{align} 
Therefore when $p/n\to 0$, \eqref{eq:chisqconvgmean} holds if and only if $p^2/n\to 0$. 

\smallskip
\noindent \textit{Case (III.i.2)  $\lim_{n\to \infty} p/n = C \in (0,1]$.}\ 
Under this case, we have 
%where  $\lim_{n\to \infty} p/n = C \in (0,1]$, as $\sqrt{2f}\sim p$, we have
%\begin{align}
%	{\sqrt{2f}}\{\sigma_n(n-1)\}^{-1}\sim {p}(n\sigma_n)^{-1}\sim C{\sigma_n}^{-1}.   \label{eq:sigmatermcase2}
%\end{align} 
\begin{align}
	{\sqrt{2f}}\times (2n\sigma_n)^{-1}\sim {p}(2n\sigma_n)^{-1}\sim C{(2\sigma_n)}^{-1}.   \label{eq:sigmatermcase2}
\end{align} 
%\begin{align}
%	\frac{\sqrt{2f}}{\sigma_n(n-1)}\sim \frac{p}{n\sigma_n}\sim \frac{C}{\sigma_n}.  \label{eq:sigmatermcase2}
%\end{align} 
If $C=1$, $\sigma_n^2 \to \infty$ and thus $\eqref{eq:sigmatermcase2} \to 0$. If $C\in (0,1)$, we have $C(2\sigma_n)^{-1} \sim C[-2\{C+ \log(1-C)\}]^{-1/2} <1$ when $0 < C < 1.$  In summary, \eqref{eq:chisqconvgvar} does not hold, which suggests that the chi-squared approximation fails.
%Therefore \eqref{eq:convgalapha} does not hold for all $\alpha$, which shows that the chi-square approximation fails.
% and Theorem \ref{Thm5.1} (i) fails. 

\smallskip
Finally, we consider a general sequence $p/n = p_n/n \in [0,1]$, where we write $p$ as $p_n$ to emphasize that $p$ changes with $n$. Similarly, we also write $f$ as $f_n$.  
Note that a sequence converges if and only if every subsequence converges. 
For the sequence $\{p_n/n\}$, by the Bolzano–Weierstrass theorem, we can further take a subsequence $\{n_t\}$ such that  $p_{n_{t}}/n_{t} \to C\in[0,1]$. 
If $C\in (0,1]$, the above analysis still applies, which shows that the chi-squared approximation fails.
%by the Case (I.2) above, 
%%If there exists a subsequence of $\{n\}$ that converges to $C\in (0,1]$, by the Case 2 above, 
%\begin{align*}
%	 \frac{\sqrt{2f_{n_t}}\{z_\alpha + o(1)\}+f_{n_t}+2\mu_{n_{t}}}{2n_t\sigma_{n_{t}}}
%\end{align*} does not converge to $z_{\alpha}$ for all $\alpha$, and  thus the chi-square approximation fails.  
Alternatively, if all the subsequences of $\{p/n\}$ converge to $0$, we know $p/n \to 0$. In summary, the above analysis shows that \eqref{eq:chisqconvgvar} and \eqref{eq:chisqconvgmean} hold if and only if $p^2/n\to 0.$ 

\medskip

%	\item[\textit{Case 2:}]
%\end{enumerate}

%\textit{(III.i.1) $\lim_{n\to \infty} p/n = 0$}
%\textit{(III.i.2) $\lim_{n\to \infty} p/n = C \in (0,1\mathrm{]}$} 
% $\lim_{n\to \infty} p/n = C \in (0,1\mbox{]}$

%\paragraph{(III.ii) The chi-square approximation with the Bartlett correction}
\smallskip
\noindent \textit{(III.ii) The chi-squared approximation with the Bartlett correction.}
 Similarly to the analysis above, we discuss two cases $\lim_{n\to \infty} p/n = 0$ and  $\lim_{n\to \infty} p/n = C \in (0,1]$, respectively. 

\smallskip

\noindent \textit{Case (III.ii.1) $\lim_{n\to \infty} p/n = 0$.}\  
%When $\lim_{n\to \infty} p/n = 0$, 
Under this case, 
we know \eqref{thm13} holds since $\rho = 1 + O(p/n) \rightarrow 1$ and $p/(2 n \sigma_n) \to 1$ as  shown in Case (III.i.1) above. Given  \eqref{thm13}, deriving the condition for  \eqref{thm14} is equivalent to examine when $p^{-1} ( f+2\rho  \mu_n  ) \to 0$. 
%The Taylor expansion of $\mu_n$ in \eqref{eq:munthm21} is 
%\begin{align*}
%	2\mu_n =-p(n+1)+\frac{n(2n-2p-3)}{2}\sum_{j=1}^{4}\frac{p^j}{j\times (n-1)^j} + O\left(\frac{p^5}{n^3} \right).
%\end{align*}
%Following the analysis similar to \eqref{eq:expa1thm1}, we have
Following the analysis of \eqref{eq:expa1thm1}, we further obtain
\begin{align}
 2\times (f+2\mu_n)/p
= &~ (p+3)-2(n+1)+ n(2n-2p-3)\sum_{j=1}^4 \frac{p^{j-1}}{j(n-1)^j} +O\left(\frac{p^4}{n^3} \right) \label{eq:expstep2thm1} \\
=&~ -\frac{p^2}{3(n-1)}-\frac{p^3}{6(n-1)^2} + O\left(\frac{p^4}{n^3}\right) + o(1). \notag   
\end{align}
We write $\rho=1-\Delta_n$ where $\Delta_n = \{6 n(p+3) \}^{-1}(2 p^{2}+9 p+11)$, which is $O(p/n)$.
By \eqref{eq:expa1thm1}, we have $4\mu_n/p=-p-3-p^2/\{3(n-1)\}+o(1)+O(p^3n^{-2})$.
 Together with \eqref{eq:expstep2thm1}, we have
% Together with \eqref{eq:munoverpthm1} and \eqref{eq:expstep2thm1}, we have
%Then following the analysis similar to \eqref{eq:expa1thm1},  by \eqref{eq:munoverpthm1} and \eqref{eq:expstep2thm1}, we have
%\newpage
\begin{align}
&~2 \times ( f+2\rho  \mu_n  )/p = 2 \times (f+2 \mu_n)/p	- 4\Delta_n\times  \mu_n/p\label{eq:expa1thm2} \\
%= &~ 2 \times (f+2 \mu_n)/p	- 4\Delta \mu_n/p \notag \\
=&~ -\frac{p^2}{3(n-1)} - \frac{p^3}{6(n-1)^2} - \Delta_n \left\{-p-3- \frac{p^2}{3(n-1)}\right\}+O\left(\frac{p^4}{n^3}\right) + o(1) \notag \\
=&~ -\frac{p^2}{3(n-1)} - \frac{p^3}{6(n-1)^2}+\frac{2p^2(p+3)}{6n(p+3)}+\frac{2p^2\times p^2}{6n(p+3)\times 3(n-1)} +O\left(\frac{p^4}{n^3}\right) + o(1)  \notag \\
=&~ -\frac{p^3}{18n^2}+O\left(\frac{p^4}{n^3}\right) + o(1).\notag  
\end{align}Therefore under this case \eqref{thm14} holds if and only if $p^3/n^2\rightarrow 0$. 

\smallskip

%\noindent \textbf{Case 2: $\BSM{\lim_{n\to \infty} p/n = C\in (0,1]}.$\quad } 
\noindent \textit{Case (III.ii.2):} \ 
%Under this case, 
When $\lim_{n\to \infty} p/n = C \in (0,1]$, we have $\rho \to 1- C/3$ and 
\begin{align*}
	\sqrt{2f} \times (  2n\rho \sigma_n )^{-1} \sim C \times (1-C/3)^{-1} (2\sigma_n)^{-1}.
\end{align*} 
%\begin{align*}
%	\sqrt{2f} \times \{ \rho \sigma_n (n-1) \}^{-1} \sim C \times (1-C/3)^{-1} \sigma_n^{-1}.
%\end{align*} 
Similarly to the Case (III.i.2) above, if $C=1$, $\eqref{thm13}\to 0$; if $C\in (0,1)$, we have $C(1-C/3)^{-1} (2\sigma_n)^{-1} \sim C(1-C/3)^{-1}[-2\{C+\log(1-C)\}]^{-1/2} < 1$ when $0<C<1$. In summary,  \eqref{thm13} does not hold, which suggests the failure of the chi-squared approximation with the Bartlett correction. 
%Therefore \eqref{eq:thm11allalphacor} does not hold for all $\alpha$ and Theorem \ref{Thm5.1} (ii) fails. 

%Finally, for a general sequence $p/n = p_n/n \in [0,1]$,  the analysis of taking subsequences above can be applied similarly. In summary, we know for the likelihood ratio test in question (III), the chi-squared  approximation with the Bartlett correction holds if and only if $p^3/n^2 \to 0.$ 
%We also mention that the analysis of taking subsequences can be used similarly in the following proofs, and thus we consider without loss of generality that the sequence $p/n$ has a limit. 

For a general sequence $p/n = p_n/n \in [0,1]$,  the analysis of taking subsequences above can be applied similarly. In summary, we know that for the likelihood ratio test in problem (III), the chi-squared  approximation with the Bartlett correction holds if and only if $p^3/n^2 \to 0.$

\subsection{Proof of Theorem \ref{thm:onesamchisq} (III)} \label{sec:pfthmchisqiii}

%We prove Theorem \ref{thm:onesamchisq} for problem (III) by examining the  characteristic function of $-2\eta \log \Lambda_n$, where $\eta=1$ or $\eta=\rho$, and $\rho$ is the Bartlett correction factor of problem (III) in Section \ref{sec:main}. 

Similarly to $\S$\,\ref{sec:pfonesam3}, in this subsection, 
we prove Theorem  \ref{thm:onesamchisq} for problem (III) as an illustration example, 
while the proofs of other problems are similar and the details are provided in  $\S$\,\ref{sec:pfthm2}.
Particularly, we prove Theorem \ref{thm:onesamchisq} for problem (III) by examining the  characteristic function of 
$-2\eta \log \Lambda_n$, where $\eta=1$ or $\eta=\rho$, and $\rho$ is the corresponding Bartlett correction factor, given in $\S$\,\ref{sec:main}.  
%Here the notation $i$ is reserved to denote the solution of the equation $x^2=-1$, i.e.,  
%the imaginary unit.  
%For each problem (I)--(III), 
%we prove Theorem \ref{thm:onesamchisq} by examining the  characteristic function of 
%%its likelihood ratio test statistic 
%$-2\eta \log \Lambda_n$, where $\eta=1$ or $\eta=\rho$, and $\rho$ is the corresponding Bartlett correction factor, given in Section \ref{sec:main}.  
%Similarly to $\S$\,\ref{sec:pfonesam3}, in this subsection, 
%we prove Theorem  \ref{thm:onesamchisq} for problem (III) as an illustration example, 
%%we first illustrate the proof of problem (III) in Theorem  \ref{thm:onesamchisq}. 
%while the proofs of other problems are similar and the details are provided in Section \ref{sec:pfthm2}.
%Here the notation $i$ is reserved to denote the solution of the equation $x^2=-1$, i.e.,  
%the imaginary unit.  
The following Lemma \ref{lm:lrtiiicharacter} gives an asymptotic expansion for the characteristic function $\mathrm{E}\{\exp( -2 it \eta \log \Lambda_n  ) \}$, 
where the notation $i$ is reserved to denote the solution of the equation $x^2=-1$, i.e.,  
the imaginary unit.  
\begin{lemma}\label{lm:lrtiiicharacter}
Under $H_0$ of the testing problem (III), 
when $\eta=1$ or $\eta=\rho$ with the Bartlett correction factor $\rho$ in $\S$\,\ref{sec:main}, 
the characteristic function of $-2\eta \log \Lambda_n$ satisfies that for a given integer $L$, when $p^{L+2}/n^L\to 0$, 
\begin{align*}
\mathrm{E}\{\exp( -2 it \eta \log \Lambda_n  ) \}	= (1-2it)^{-{f}/{2}}\exp\Biggr[\sum_{l=1}^{L-1}\varsigma_l\big\{(1-2it)^{-l} - 1\big\}+O\biggr(\frac{p^{L+2}}{n^{L}}\biggr)\Biggr], 
\end{align*}		
%\begin{align*}
%\mathrm{E}\{\exp( -2 it \eta \log \Lambda_n  ) \}	= (1-2it)^{-{f}/{2}}\exp\Biggr[\sum_{l=1}^{L-1}\varsigma_l\big\{(1-2it)^{-l} - 1\big\}+O(\varsigma_L)\Biggr], 
%\end{align*}		
where $f=p(p+3)/2$ is the corresponding degrees of freedom, and 
\begin{align}
\varsigma_l=\frac{(-1)^{l+1}}{l(l+1)}\sum_{j=1}^p\Biggr\{B_{l+1}\left(\frac{(1-\eta)n}{2}-\frac{j}{2}\right) -\biggr( \frac{(1-\eta)n}{2} \biggr)^{l+1} \Biggr\}\left(\frac{\eta n}{2}\right)^{-l}.	\label{eq:varsigmaiii}
\end{align} 
%Here $\eta$ can be 1 or the Bartlett correction factor $\rho$ in $\S$\,\ref{sec:main}. 
For any integer $l\geq 1$, $B_{l}(\cdot)$ represents the Bernoulli polynomial of degree $l$;  see, e.g., Eq. (25) in Section 8.2.4 of \cite{Muirhead2009}.   
\end{lemma}
\begin{proof}
Section \ref{sec:lrtiiicharacter} on Page \pageref{sec:lrtiiicharacter}.
\end{proof}

\medskip

\noindent With Lemma \ref{lm:lrtiiicharacter}, we next prove \eqref{eq:chisqapprox} and \eqref{eq:chisqapproxbartcorr} in Theorem \ref{thm:onesamchisq} for the  chi-squared approximations without and with the Bartlett correction, respectively. 

%\noindent With Lemma \ref{lm:lrtiiicharacter}, we next prove \eqref{eq:chisqapprox} for  the chi-squared approximation, and 
%\eqref{eq:chisqapproxbartcorr} for the chi-squared approximation with the Bartlett correction, respectively. 
\medskip

\noindent \textit{(i) The chi-squared approximation.}\quad 
When $\rho =1$, 
as $B_{l+1}(\cdot)$ is a polynomial of order $l+1$,
we have $\varsigma_l=O(p^{l+2}n^{-l})$ for $l\geq 2$,
and we can check that $\varsigma_1=\Theta(p^3n^{-1})$; see \eqref{eq:varsig1}.  
%we  can check that $\varsigma_1=\Theta(p^3n^{-1})$, $\varsigma_2=\Theta(p^4n^{-2}),$ 
%and for $l\geq 2$,  $\varsigma_l=O(p^{l+2}n^{-l})$.  
Thus when $p^2/n\to 0$, $\varsigma_l\to 0$ for $l\geq 2$.
% we have $\varsigma_l\to 0$ when $p^2/n\to 0$. 
Let $\Psi(t)=\mathrm{E}\{\exp( -2 it \log \Lambda_n ) \}$. 
Then by Lemma  \ref{lm:lrtiiicharacter}, 
\begin{align}
\Psi(t)= (1-2it)^{-{f}/{2}}\Biggr\{\exp\Biggr[\sum_{l=1}^{2}\varsigma_l\big\{(1-2it)^{-l} - 1\big\}+ O\big(p^{5}n^{-3}\big) \Biggr]\Biggr\}.\label{eq:charfun1} 
\end{align}
By Taylor' expansion, we can write $\exp[\varsigma_l\{(1-2it)^{-l} - 1\}]=1+V_l(t)$, 
where 
\begin{align}
	V_l(t)=\sum_{v=1}^{\infty}\frac{\varsigma_l^v }{v!} \sum_{w=0}^v \binom{v}{w} (1-2it)^{-lw} (-1)^{v-w}. \label{eq:vltdef}
\end{align}
Then by \eqref{eq:charfun1} and $p^2/n\to 0$, we have $\Psi(t)=\tilde{\Psi}(t)\{1+O(p^5/n^3)\}$, where
\begin{align}
\tilde{\Psi}(t)=&~	(1-2it)^{-{f}/{2}}\big\{1+V_1(t) \big\}\big\{1+V_2(t) \big\}\notag \\
=&~(1-2it)^{-{f}/{2}}+\sum_{v=1}^{\infty}\frac{\varsigma_1^v }{v!} \sum_{w=0}^v \binom{v}{w} (1-2it)^{-f/2-w} (-1)^{v-w}  \label{eq:charfun2}\\
&~\ +\sum_{v=1}^{\infty}\frac{\varsigma_2^v }{v!} \sum_{w=0}^v \binom{v}{w} (1-2it)^{-f/2-2w} (-1)^{v-w}  \notag \\
&~\ +\sum_{\substack{v_1\geq 1; ~ 0\leq w_1\leq v_1\\ v_2\geq 1; ~ 0\leq w_2\leq v_2} }
\frac{\varsigma_1^{v_1}\varsigma_2^{v_2}}{v_1!v_2!}   
% \sum_{w_1=0}^{v_1} 
%\sum_{w_2=0}^{v_2}
 \binom{v_1}{w_1} \binom{v_2}{w_2} (1-2it)^{-f-w_1-2w_2} (-1)^{v_1-w_1+v_2-w_2}.\notag 
\end{align}
Note that $(1-2it)^{-{f}/{2}}$ is the characteristic function of $\chi^2_f$ distribution. 
Following similar analysis to  Section 8.5 in  \cite{anderson2003introduction}, 
we use the inversion property of the characteristic function, and then by  
\eqref{eq:charfun2}, we obtain that
\begin{align}
&~\Pr ( -2\log \Lambda_n \leq x )\label{eq:probexpaniii1} \\
=&~\Biggr\{\Pr (\chi_f^2 \leq x ) + \sum_{v=1}^{\infty} \frac{\varsigma_1^v}{v!} \sum_{w=0}^v \binom{v}{w}\Pr( \chi^2_{f+2w}\leq x ) (-1)^{v-w} \notag \\
	&~ + \sum_{v=1}^{\infty} \frac{\varsigma_2^v}{v!} \sum_{w=0}^v \binom{v}{w}\Pr( \chi^2_{f+4w}\leq x )(-1)^{v-w} \notag \\
	&~ + \sum_{\substack{v_1\geq 1; ~ 0\leq w_1\leq v_1\\ v_2\geq 1; ~ 0\leq w_2\leq v_2} }   \frac{\varsigma_1^{v_1}\varsigma_2^{v_2}}{v_1!v_2!}   \binom{v_1}{w_1} \binom{v_2}{w_2}\Pr(\chi^2_{2f+2w_1+4w_2}\leq x)(-1)^{v_1-w_1+v_2-w_2} \Biggr\}\biggr\{1+O\biggr(\frac{p^5}{n^{3}}\biggr)\biggr\}. \notag 
\end{align}
(From \eqref{eq:charfun2} to \eqref{eq:probexpaniii1}, Fubini's theorem is implicitly used to exchange the order of the infinite sum and the integration of characteristic functions.) 
%Note that $\varsigma_1=\Theta(p^3n^{-1})$ and $\varsigma_2=O(p^4n^{-2})$. 
%When $p$ is fixed,  we have $\varsigma_1=\Theta(n^{-1})$ and $\varsigma_2=O(n^{-2})$. Then by \eqref{eq:probexpaniii1},
%\begin{align}
%\Pr ( -2\log \Lambda_n \leq x )=\Pr (\chi_f^2 \leq x )+\varsigma_1\left\{\Pr(\chi_{f+2}^2\leq x) - \Pr(\chi_{f}^2\leq x) \right\} + o(n^{-1}). \label{eq:chisqgoal12}	
%\end{align}
%When $p\to \infty$, we have  $\varsigma_1=\Theta(p^3n^{-1})$, $\varsigma_2=O(p^4n^{-2})$, and $f=\Theta(p^2)\to \infty$. 
%%we know $f\to \infty$ as $f=\Theta(p^2)$. 
%Under this case, to evaluate  \eqref{eq:probexpaniii1}, we utilize the following Propositions  \ref{prop:infinitesumm} and \ref{prop:infinitesumm2}. 
We next utilize the following Propositions  \ref{prop:infinitesumm} and \ref{prop:infinitesumm2} to evaluate  \eqref{eq:probexpaniii1}.

\begin{proposition}\label{prop:infinitesumm}
%Given the degrees of freedom $f$,
Given an integer $h\in \{ 1, 2, 3, 4\}$,  
 when $x=\chi^2_f(\alpha)$,  there exists a constant $C$ such that   as  $f\to \infty$, 
\begin{align}
\sum_{w=0}^v \binom{v}{w}\Pr( \chi^2_{f+2hw}\leq x )(-1)^{v-w}=&~	O(v!C^vf^{-v/2})  \label{eq:probdifforder1h} 
%\sum_{w=0}^v \binom{v}{w}\Pr( \chi^2_{f+4w}\leq x )(-1)^{v-w}=&~	O(v!f^{-v/2}),   \notag 
\end{align}
uniformly over $v\geq 1$. 
%In addition, 
%for $h_1,h_2\in \{1,2,3\}$ and $h_1\neq h_2$, 
%%$h_1$ and $h_2\geq 1$, 
%\begin{align*}
% \sum_{w_1=0}^{v_1} \sum_{w_2=0}^{v_2}\binom{v_1}{w_1} \binom{v_2}{w_2}\Pr(\chi^2_{2f+2h_1w_1+2h_2w_2}\leq x)(-1)^{v_1-w_1+v_2-w_2}=O\{v_1!v_2!f^{-(v_1+v_2)/2}\}. 
%\end{align*}	
\end{proposition}
\begin{proof}
Please see Section \ref{sec:pfpropiftydiff1} on Page  \pageref{sec:pfpropiftydiff1}.		
\end{proof}

\begin{proposition}\label{prop:infinitesumm2}
For $(h_1,h_2)=(1,2)$ or $(h_1,h_2)=(2,3)$,  when $x=\chi^2_f(\alpha)$,  there exists a constant $C$ such that  as $f\to \infty$,  
%$h_1,h_2\in \{1,2,3\}$ and $h_1\neq h_2$, 
\begin{align}
 &~\sum_{w_1=0}^{v_1} \sum_{w_2=0}^{v_2}\binom{v_1}{w_1} \binom{v_2}{w_2}\Pr(\chi^2_{2f+2h_1w_1+2h_2w_2}\leq x)(-1)^{v_1-w_1+v_2-w_2}\notag\\
  %\label{eq:probdifforder1h1h2} \\
 =&~O\{v_1!v_2!C^{v_1+v_2}f^{-(v_1+v_2)/2}\}  \notag 
\end{align}	
uniformly over $v_1, v_2\geq 1$. 
\end{proposition}
\begin{proof}
Please see  Section \ref{sec:pfpropiftydiff2} on Page  \pageref{sec:pfpropiftydiff2}.		
\end{proof}
%\smallskip

\begin{remark}
In Propositions \ref{prop:infinitesumm}	and \ref{prop:infinitesumm2}, $C$ denotes a universal constant and its value can change.
This is similarly used in the following proofs.  
In addition, for a series $\{b_{v,f}\}$ that depends on positive integers $v$ and $f$,  
 we say $b_{v,f}=O(v!C^vf^{-v/2})$ as $f\to \infty$ and uniformly over $v\geq 1$, if there exists a constant $C$ such that
 $\sup_{v\geq 1}\limsup_{f\to \infty} |b_{v,f}/(v!C^vf^{-v/2})|<\infty$.  
%\begin{align*}
%	\sup_{v\geq 1}\limsup \biggr|\frac{a_{v,f}}{v!C^vf^{-v/2}}\biggr|
%\end{align*} is bounded. 
\end{remark}

\medskip

\noindent 
When $x=\chi_f^2(\alpha)$ and $f\to \infty$, 
we apply Proposition \ref{prop:infinitesumm} with $h=1$ and $h=2$, and Proposition \ref{prop:infinitesumm2} with $(h_1,h_2)=(1,2)$ to \eqref{eq:probexpaniii1}. 
Then as $\varsigma_1=\Theta(p^3n^{-1})$, $\varsigma_2=O(p^4n^{-2})$, and $f=\Theta(p^2)$,  
%for $x=\chi_f^2(\alpha)$,   by Propositions \ref{prop:infinitesumm} and \ref{prop:infinitesumm2}, and \eqref{eq:probexpaniii1}, 
when $p\to \infty$ and  $p^2/n\to 0$, we obtain
\begin{align}
\Pr ( -2\log \Lambda_n \leq x )=\Pr (\chi_f^2 \leq x )+\varsigma_1\left\{\Pr(\chi_{f+2}^2\leq x) - \Pr(\chi_{f}^2\leq x) \right\} + o(p^2/n). \label{eq:chisqgoal1}	
\end{align}
We next compute $\varsigma_1$. 
Particularly, for the chi-squared approximation, $\rho =1$, and then by \eqref{eq:varsigmaiii}, 
%we obtain
\begin{align}
	\varsigma_1=\frac{1}{2}\sum_{j=1}^p B_2\biggr(-\frac{j}{2}\biggr)\Big(\frac{n}{2}\Big)^{-1}=\frac{1}{24n} p\left(2 p^{2}+9 p+11\right),\label{eq:varsig1}
\end{align}
where we use $B_2(z)=z^2-z+1/6$; see, e.g., Eq. (26) in Section 8.2.4 of \cite{Muirhead2009}. 
To finish the proof of \eqref{eq:chisqapprox}, we use the following lemma. 
\begin{lemma}\label{lm:order2diff}
When $x=\chi_f^2(\alpha)$ and $f\to \infty$,  for $h\in \{1,2,3,4\}$,	
\begin{align}
\Pr(\chi_{f+2h}^2\leq x) - \Pr(\chi_{f}^2\leq x)	= &~  - \sum_{k=1}^h \left\{\Gamma\left(\frac{f}{2}+h-k+1\right) \right\}^{-1}\left(\frac{x}{2}\right)^{\frac{f}{2}+h-k}e^{-x/2}\label{eq:probdiff1exact} \\
=&~-  \frac{h}{\sqrt{f\pi }}\exp\left(-\frac{z_{\alpha}^2}{2}\right)\Big\{1+ O(f^{-1/2})\Big\}. \label{eq:probdiff1approx}
\end{align}
\end{lemma}
\begin{proof}
Please see Section \ref{sec:order2diff}	 on Page \pageref{sec:order2diff}. 
\end{proof}
\medskip

\noindent As $p\to \infty$, $f\to \infty$. 
Then by \eqref{eq:chisqgoal1} and  \eqref{eq:varsig1}, and applying	 Lemma \ref{lm:order2diff} with $h=1$,  
\eqref{eq:chisqapprox} is proved, where $\vartheta_1(n,p)=\varsigma_1/\sqrt{f}$.

\medskip

\noindent \textit{(ii) The chi-squared approximation with the Bartlett correction.}\quad 
Similarly to the proof in Part (i) above, 
we prove \eqref{eq:chisqapproxbartcorr} by examining the expansion of the characteristic function in Lemma \ref{lm:lrtiiicharacter}. 
%We first analyze the order of 
In particular, for the chi-squared approximation with the Bartlett correction, 
%coefficients $\varsigma_l$ in \eqref{eq:varsigmaiii}, 
we note that the Bartlett correction factor $\rho$ is chosen such that $\varsigma_1=0$ (see Section 8.5.3 in \cite{Muirhead2009}). 
This can be checked  by plugging $\rho=1-\{6 n(p+3) \}^{-1}(2 p^{2}+9 p+11)$ into   \eqref{eq:varsigmaiii} to calculate $\varsigma_1$. 
In addition, 
by $B_3(z)=z^3-3z^2/2+z/2$ (see, e.g., Eq. (26) in Section 8.2.4 of \cite{Muirhead2009}),  
we  calculate  
\begin{align}
	\varsigma_2=\frac{p(2 p^4 + 18 p^3 + 49 p^2 + 36 p - 13)}{288(p+3)(\rho n)^2}, \label{eq:varsigma2}
\end{align}
and therefore $\varsigma_2=\Theta(p^4n^{-2})$.
%Similarly, since $\rho = 1+ O(p/n)\to 1$ and $B_{l+1}(\cdot)$ is a polynomial of order $l+1$, by \eqref{eq:varsigmaiii}, $\varsigma_l=O(p^{l+2}n^{-l})$  for $l\geq 2$. 
%When $p^3/n^2\to 0$, $\varsigma_l\to 0$ for  $l\geq 4$. 
%We redefine $\Phi(t)=\mathrm{E}\{\exp(-2it\rho \log \Lambda_n) \}$, and 
We  redefine $\Psi(t)=\mathrm{E}\{\exp(-2it\rho \log \Lambda_n) \}$.
% and
Then when $p^3/n^2\to 0$, by Lemma \ref{lm:lrtiiicharacter}, we have
\begin{align}
\Psi(t)= (1-2it)^{-{f}/{2}}\Biggr\{\exp\Biggr[\sum_{l=2}^{3}\varsigma_l\big\{(1-2it)^{-l} - 1\big\}+ O(p^{6}n^{-4}) \Biggr]\Biggr\}, \label{eq:charfun32} 
\end{align} 
where we use $\varsigma_1=0.$ 
Similarly to \eqref{eq:charfun2},
we have $\Psi(t)=(1-2it)^{-f/2}\{1+V_2(t)\}\{1+V_3(t)\} \{1+O(p^6n^{-4})\}$. 
Moreover, similarly to \eqref{eq:probexpaniii1}, we obtain
\begin{align}
&~\MYPR ( -2\rho \log \Lambda_n \leq x )\label{eq:probexpaniii1bart} \\
=&~\Biggr\{\Pr (\chi_f^2 \leq x ) + \sum_{v=1}^{\infty} \frac{\varsigma_2^v}{v!} \sum_{w=0}^v \binom{v}{w}\MYPR( \chi^2_{f+4w}\leq x ) (-1)^{v-w} \notag \\
&~\, + \sum_{v=1}^{\infty} \frac{\varsigma_3^v}{v!} \sum_{w=0}^v \binom{v}{w}\Pr( \chi^2_{f+6w}\leq x )(-1)^{v-w} \notag \\
&~\, + \sum_{\substack{v_2\geq 1; ~ 0\leq w_2\leq v_2\\ v_3\geq 1; ~ 0\leq w_3\leq v_3} }   \frac{\varsigma_2^{v_2}\varsigma_3^{v_3}}{v_2!v_3!}   \binom{v_2}{w_2} \binom{v_3}{w_3}\Pr(\chi^2_{2f+4w_2+6w_3}\leq x)(-1)^{v_2-w_2+v_3-w_3} \Biggr\}\biggr\{1+O\Big(\frac{p^6}{n^{4}}\Big)\biggr\}. \notag 	
\end{align}
When $x=\chi_f^2(\alpha)$ and $f\to \infty$, 
we apply Proposition \ref{prop:infinitesumm} with $h=2$ and $h=3$, and Proposition \ref{prop:infinitesumm2} with $(h_1,h_2)=(2,3)$ to \eqref{eq:probexpaniii1bart}. 
Then as $\varsigma_2=\Theta(p^4/n^2)$, $\varsigma_3=O(p^5/n^3)$, and $f=\Theta(p^2)$,  
we know that when $p\to \infty$ and $p^3/n^2\to 0$, 
%by Propositions \ref{prop:infinitesumm} and \ref{prop:infinitesumm2}, 
%the analysis of \eqref{eq:charfun2},  \eqref{eq:probexpaniii1}, and \eqref{eq:chisqgoal1} in part (i), we can obtain 
\begin{align}
\Pr ( -2\rho\log \Lambda_n \leq x )=\Pr (\chi_f^2 \leq x )+\varsigma_2\left\{\Pr(\chi_{f+4}^2\leq x) - \Pr(\chi_{f}^2\leq x) \right\} + o(p^3/n^2). 	\label{eq:chisqbartgoal1}
\end{align}
%The analysis is very  similar, and the details are not included here due to the limitation of space. 
%Thus we prove  b
By \eqref{eq:varsigma2} and  \eqref{eq:chisqbartgoal1}, and applying  Lemma \ref{lm:order2diff} with $h=2$, we prove \eqref{eq:chisqapproxbartcorr}, where $\vartheta_2(n,p)=2\varsigma_2/\sqrt{f}$.

\subsection{Proof of Theorem \ref{thm:onesamnormal} (III)}\label{sec:pfonsamnomral3}

In this section, we prove Theorem \ref{thm:onesamnormal}  also by examining  the characteristic function of the likelihood ratio test statistic.  
In particular, motivated by the limit in \eqref{eq:mmtnormal1}, 
we study the standardized test statistic $(-2\log\Lambda_n+2\mu_n)(2n\sigma_n)^{-1}$, 
where the values of $\mu_n$ and $\sigma_n$ are given in Theorem \ref{thm:onesamnormal}.
Under $H_0$ of the testing problem (III), 
by \eqref{eq:mmtjoint}, 
the characteristic function of 
$(-2\log\Lambda_n+2\mu_n)/(2n\sigma_n)$ is
\begin{align}
&~\mathrm{E}\Biggr\{\exp\left(is\times \frac{ -2\log \Lambda_n + 2\mu_n}{2n\sigma_n} \right) \Biggr\} \label{eq:fsform} \\
=&~\left(\frac{2 e}{n}\right)^{-n p t i  / 2}(1-ti)^{-n p(1-t i) / 2} \frac{\Gamma_{p}[\{n(1-ti)-1\} / 2]}{\Gamma_{p}\{(n-1) / 2\}}\exp \left(\frac{\mu_ns i}{n\sigma_n}\right),\notag  
\end{align} 
where $i$ denotes the imaginary unit and  $t=s/(n\sigma_n)$. 
Then the proof of Theorem \ref{thm:onesamnormal} utilizes  the following inequality result of the characteristic function.  

\begin{lemma}[Theorem 1.4.9 \citep{ushakov2011selected}] \label{lm:chardiff}
Let $G_1(x)$ and $G_0(x)$ be two distribution functions with characteristic functions $\psi_1(s)$ and $\psi_0(s)$, respectively. If $G_0(x)$ has a derivative and $\sup_x G'_0(x)\leq a <\infty$, then for any positive $T$ and any $b\geq 1/(2\pi)$, 
\begin{align*}
	\sup_x\big|G_1(x)-G_0(x) \big|\leq b\int_{-T}^{T}\biggr|\frac{\psi_1(s)-\psi_0(s)}{s} \biggr|ds + \frac{c}{T},
\end{align*} where $c$ is a constant that depends on $a$ and $b$. 
\end{lemma}

\noindent We next prove \eqref{eq:normalbias1} and  \eqref{eq:normalbias2} in Theorem \ref{thm:onesamnormal} for the chi-squared approximations without and with the Bartlett correction, respectively. 

\medskip

\noindent \textit{(i) Chi-squared approximation.} \quad 
We prove \eqref{eq:normalbias1} 
% Theorem \ref{thm:onesamnormal} is then proved 
 by using Lemma \ref{lm:chardiff} to derive an upper bound of the difference $G_1(x)-G_0(x)$, where we consider
\begin{align*}
G_1(x)=\Pr \left(\frac{-2\log \Lambda_n+2\mu_n}{2n\sigma_n} \leq x  \right), \quad \quad \quad G_0(x)=\Phi(x);  
\end{align*}
here $\Phi(x)$ denotes the cumulative distribution function of the standard normal distribution. 
Then the characteristic function of $G_1(x)$  is $\psi_1(s)= \eqref{eq:fsform}$, and 
the characteristic function of $G_0(x)$ is $\psi_0(s)=\exp(-s^2/2)$. 
To quantify $\psi_1(s)-\psi_0(s)$, we use the following Lemma \ref{lm:chardiffgoal}. 
\begin{lemma}\label{lm:chardiffgoal}
When $s=o(\min \{ (n/p)^{1/2}, f^{1/6}\})$,
\begin{align}
\log \psi_1(s)-\log \psi_0(s)=O\left(\frac{p}{n}\right)s +\left(\frac{1}{p}+\frac{p}{n} \right)O\left( s^2\right)+ O\left( \frac{s^3}{\sqrt{f}}\right). \label{eq:chardiffgoal}
\end{align}	
\end{lemma} 
\begin{proof}
Please see Section \ref{sec:chardiffgoal} on Page \pageref{sec:chardiffgoal}. 
\end{proof}

\medskip
\noindent By Lemmas \ref{lm:chardiff} and \ref{lm:chardiffgoal}, we take $T=\min\{ (n/p)^{(1-\delta)/2}, f^{(1-\delta)/6}\}$, where $\delta \in (0,1)$ is a small constant, and then
 \begin{align}
\sup_x \big|G_1(x)-G_0(x)\big|
\leq b \int_{-T}^{T} \psi_0(s) \biggr\{ O\left(\frac{p}{n}\right) +\left(\frac{1}{p}+\frac{p}{n} \right)O\left( s\right)+ O\left( \frac{s^2}{\sqrt{f}}\right) \biggr\} ds+\frac{c}{T}. \label{eq:supfgdiff1}
\end{align} 
Since $\int_{-T}^{T}\psi_0(s) < \infty$, $\int_{-T}^{T}\psi_0(s)s < \infty$, and $\int_{-T}^{T}\psi_0(s)s^2 < \infty$, by $f=\Theta(p^2)$ and \eqref{eq:supfgdiff1}, 
\begin{align*}
\sup_x \big|G_1(x)-G_0(x)\big|= O\left\{ \left( \frac{p}{n} \right)^{(1-\delta)/2} +  f^{-(1-\delta)/6} \right\}.	
\end{align*}
Consider $x=\{\chi^2_f(\alpha)+2\mu_n\}(2n\sigma_n)^{-1}$, and then $G_1(x)-G_0(x)$ gives
\begin{align}
\Pr \left\{{-2\log \Lambda_n}\leq {\chi^2_f(\alpha)}\right\}-\Phi\Biggr\{ \frac{\chi^2_f(\alpha)+2\mu_n}{2n\sigma_n} \Biggr\}= O\left\{ \left( \frac{p}{n} \right)^{(1-\delta)/2} +  f^{-(1-\delta)/6} \right\}.  \label{eq:diffnormal1}
\end{align} 
Then  \eqref{eq:normalbias1} is proved by $\bar{\Phi}(\cdot)=1-\Phi(\cdot)$ and $\Pr \{{-2\log \Lambda_n}> {\chi^2_f(\alpha)}\} = 1-\Pr \{{-2\log \Lambda_n}\leq {\chi^2_f(\alpha)}\}$. 

%To finish the proof of \eqref{eq:normalbias1}, we combine 
%\eqref{eq:diffnormal1} and \eqref{eq:phiexpannormal}, and use the facts that $\bar{\Phi}(\cdot)=1-\Phi(\cdot)$ and $\Pr \{{-2\log \Lambda_n}> {\chi^2_f(\alpha)}\} = 1-\Pr \{{-2\log \Lambda_n}\leq {\chi^2_f(\alpha)}\}$. 

\medskip

\noindent \textit{(ii) Chi-squared approximation with the Bartlett correction.} \quad 
To prove \eqref{eq:normalbias2}, we still use \eqref{eq:supfgdiff1}.
%To prove
%\eqref{eq:normalbias2}, we similarly use Lemma \ref{lm:chardiff} and then obtain \eqref{eq:supfgdiff1}.
Now consider   
$x=\{\chi^2_f(\alpha)+2\rho\mu_n\}(2\rho n\sigma_n)^{-1}$, and then $G_1(x)-G_0(x)$ gives
\begin{align}
\Pr \left\{{-2\rho \log \Lambda_n}\leq {\chi^2_f(\alpha)}\right\}-\Phi\Biggr\{ \frac{\chi^2_f(\alpha)+2\rho\mu_n}{2\rho n\sigma_n} \Biggr\}= O\left\{ \left( \frac{p}{n} \right)^{(1-\delta)/2} +  f^{-(1-\delta)/6} \right\}. \notag %\label{eq:diffnormal1}
\end{align} 
%Note that $\rho = 1+O(p/n)$. Threfore $2\rho n \sigma_n/\sqrt{2f}=1+O(p/n)$. 
%Similarly to \eqref{eq:phiexpannormal}, we obtain
%\begin{align*}
%\Phi\biggr( \frac{\chi^2_f(\alpha)+2\rho\mu_n}{2\rho n\sigma_n} \biggr)=\Phi\left( z_{\alpha} + \frac{f+2\rho \mu_n}{2\rho n\sigma_n}\right) + O\left(\frac{1}{\sqrt{f}}\right).	
%\end{align*}
%Then \eqref{eq:normalbias2} can be proved similarly as in Part (i) above. 

%Note that $\mu_n$ and $\sigma_n$ in Theorem \ref{thm:onesamnormal} are motivated by the limit in \eqref{eq:mmtnormal1}.  
\begin{remark}\label{rm:comparison}
Although Theorem \ref{thm:onesamnormal} is inspired by the limit in \eqref{eq:mmtnormal1}, which was first established in \cite{Jiang13}, 
%the limit is inspired by \cite{Jiang13}, 
 Theorem \ref{thm:onesamnormal} differs from the existing results  by further characterizing the convergence rate of  \eqref{eq:mmtnormal1} by Lemma \ref{lm:chardiffgoal}. 
Particularly, \cite{Jiang13} proved \eqref{eq:mmtnormal1} when $s$ is considered fixed and the convergence rate is not examined. 
On the other hand, Lemma \ref{lm:chardiffgoal} allows $s$ changes with $n$ and $p$, and the difference between the two characteristic functions is characterized by  \eqref{eq:chardiffgoal}. 
Technically, establishing \eqref{eq:chardiffgoal} requires a careful investigation of the asymptotic expansion of the gamma functions, where the technical details are given in Sections \ref{sec:gammafunc} and \ref{sec:normallemmas}. 
\end{remark}

%are the same as $\mu_n$ in \eqref{eq:munthm21}	and $\sigma_n$ in \eqref{eq:sigmanthm21}, which are motivated by the central limit theorem results in \cite{Jiang13} and \cite{Jiang15}. 

%Particularly,  
%for each likelihood ratio test, 
%when $p/n^{d_1}\to 0$, we have  
%$-(f+2\mu_n)/(2n\sigma_n)=\sqrt{2} \vartheta_1(n,p)+o(p^{1/d_1}n^{-1})$, 
%and then by Taylor's series of $\bar{\Phi}(\cdot)$ at $z_{\alpha}$,  
%%its order is $\Theta(p^{1/d_1}n^{-1})$; please see Remark \ref{rm:order}.  
%%It follows that 
%%\begin{align*}
%%\frac{f+2\mu_n}{2n\sigma_n}=\sqrt{2} \vartheta_1(n,p)	
%%\end{align*} 
%%Then for the chi-squared approximation of  each test, we can check that 
%\begin{align}
%	\bar{\Phi} \biggr(z_{\alpha}+\frac{f+2\mu_n}{2n\sigma_n} \biggr)- \alpha =\frac{\vartheta_1(n,p)}{\sqrt{\pi}} \exp\biggr(-\frac{z_{\alpha}^2}{2}\biggr) + o\biggr( \frac{p^{1/d_1}}{n}\biggr), 
%	\label{eq:biaslink1} 
%\end{align} 
%which suggests that the first two terms in the right hand side of \eqref{eq:normalbias1} are consistent with \eqref{eq:chisqapprox}; 

\begin{remark}\label{rm:order}
Since $\chi^2_f$ can be viewed as a summation over $f$ independent $\chi_1^2$ variables,  
by applying 
the central limit theorem, 
%Berry-Esseen Theorem 
%to $\chi^2_f$, we know that $\chi^2_f(\alpha)$ satisfies 
we have $\chi^2_f(\alpha)=	\sqrt{2f}z_{\alpha}+f + O(1),$
where $z_{\alpha}$ denote the upper $\alpha$-level quantile of the standard normal distribution. 
For the problem (III),
note that $\mu_n$ and $\sigma_n$ in Theorem \ref{thm:onesamnormal} are the same as  \eqref{eq:munthm21}	and  \eqref{eq:sigmanthm21}, respectively.  
Then by the proof of \eqref{eq:chisqconvgvar} in Section \ref{sec:pfonesam3}, we have $2n\sigma_n/\sqrt{2f}=1+O(p/n)$. 
Consequently, 
%in \eqref{eq:diffnormal1}, 
when $f\to \infty$ and $p/n\to 0$,  
\begin{align*}
\Phi\biggr\{ \frac{\chi^2_f(\alpha)+2\mu_n}{2n\sigma_n} \biggr\}=	\Phi\left(z_{\alpha} + \frac{f+2\mu_n}{2n\sigma_n} \right) + O\left( \frac{1}{\sqrt{f}} \right)+O\left( \frac{p}{n} \right).
\end{align*}
%\begin{align}
%\Phi\biggr( \frac{\chi^2_f(\alpha)+2\mu_n}{2n\sigma_n} \biggr)=&~\Phi \left\{ z_{\alpha}+\frac{f+2\mu_n}{2n\sigma_n} + O\left(\frac{1}{\sqrt{f}}\right)\right\}	=\Phi\left(z_{\alpha} + \frac{f+2\mu_n}{2n\sigma_n} \right) + O\left( \frac{1}{\sqrt{f}} \right). \label{eq:phiexpannormal} 
%\end{align} 
%\end{remark}
%\begin{remark} \label{rm:order}
%Note that $\mu_n$ in \eqref{eq:munthm21}	and $\sigma_n$ in \eqref{eq:sigmanthm21} are the same as $\mu_n$ and $\sigma_n$ in \eqref{eq:munthm212}. 
%For the problem (III),
%note that $\mu_n$ and $\sigma_n$ in Theorem \ref{thm:onesamnormal} are the same as  \eqref{eq:munthm21}	and  \eqref{eq:sigmanthm21}, respectively.  
Moreover, by \eqref{eq:expa1thm1},  $(f+2\mu_n)/(2n\sigma_n) \sim -p^2/(6n)$ when $p/n\to 0$. 
%In Section \ref{sec:pfonesam3}, the proof under Case (III.i.1)  shows that
%when $p/n\to 0$,  $(f+2\mu_n)/(2n\sigma_n) \sim -p^2/(6n)$.  
%which is $\Theta(p^{1/d_1}/n)$, where $d_1=1/2$. 
%Moreover, for the problem (III),
%in Theorem \ref{thm:onesamchisq}, $\vartheta_1(n,p)=\sqrt{2}p^2/(12n) + o(p^2/n)$. 
%Therefore, 
Thus $-(f+2\mu_n)/(2n\sigma_n) =\sqrt{2} \vartheta_1(n,p)+o(p^{1/d_1}n^{-1})$, 
which is of the order of $p^{1/d_1}n^{-1}$ with $d_1=1/2$. 
When $p/n^{d_1}\to 0$, by $\alpha = \bar{\Phi}(z_{\alpha})$ and Taylor's series of $\bar{\Phi}(\cdot)$ at $z_{\alpha}$,  
\begin{align}
	\bar{\Phi} \biggr(z_{\alpha}+\frac{f+2\mu_n}{2n\sigma_n} \biggr)- \alpha =\frac{\vartheta_1(n,p)}{\sqrt{\pi}} \exp\biggr(-\frac{z_{\alpha}^2}{2}\biggr) + o\biggr( \frac{p^{1/d_1}}{n}\biggr), 
	\notag %\label{eq:biaslink1} 
\end{align} 
which suggests that the first two terms in the right hand side of \eqref{eq:normalbias1} are consistent with \eqref{eq:chisqapprox}. 
Similarly, for the chi-squared approximation with the Bartlett correction,  
%the proof under Case (III.ii.1) in  Section \ref{sec:pfonesam3} shows that when $p/n\to 0$,  
when $f\to \infty$ and $p/n\to 0$, 
\begin{align*}
\Phi\biggr\{ \frac{\chi^2_f(\alpha)+2\rho \mu_n}{2\rho n\sigma_n} \biggr\}=	\Phi\left(z_{\alpha} + \frac{f+2\rho \mu_n}{2\rho n\sigma_n} \right) + O\left( \frac{1}{\sqrt{f}} \right)+O\left( \frac{p}{n} \right).
\end{align*}
By \eqref{eq:expa1thm2}, 
we have $-(f+2\rho \mu_n)/(2\rho n\sigma_n)=\sqrt{2} \vartheta_2(n,p)+o(p^{2/d_2}n^{-2}),$
which is of the order of $p^{2/d_2}n^{-2}$ with $d_2=2/3$.   
Thus when $p^{2/d_2}n^{-2}\to 0$, 
we also know that the first two terms in the right hand side of \eqref{eq:normalbias2} are consistent with \eqref{eq:chisqapproxbartcorr}. 
%$(f+2\rho \mu_n)/(2\rho n\sigma_n)\sim -p^3/(36n^2)$ when $p/n\to 0$. 
%%which is $\Theta(p^{2/d_2}n^{-2})$ with  $d_2=2/3$.   
%Moreover, $\vartheta_2(n,p)$ in 
%\eqref{eq:vartheta1joint} satisfies $\vartheta_2(n,p)=\sqrt{2}p^3/(72n^2)+o(p^3n^{-2})$. Therefore, $-(f+2\rho \mu_n)/(2\rho n\sigma_n)=\sqrt{2} \vartheta_2(n,p)+o(p^{2/d_2}n^{-2}).$ 
For other likelihood ratio tests (II)--(VI), similar conclusions also hold by the proofs in Section \ref{sec:pfthm1and4}.  
\end{remark}

%The detailed proofs are not included here due to the limitation of space.    
%\medskip

\section{Proofs of Other Problems}\label{sec:pfthms}

In this section, we provide the proofs of other testing problems following similar arguments to that in Section \ref{sec:pfthemexam}. 
Particularly, for tests (I)--(II) and (IV)--(VII),  Theorems \ref{thm:onesam},  \ref{thm:allmult} and \ref{Thm2.1} are proved in Section \ref{sec:pfthm1and4};  
 Theorems \ref{thm:onesamchisq}, \ref{thm:multisamchisq} and \ref{thm:chisqindp} are proved in Section \ref{sec:pfthm2}, 
 Theorems \ref{thm:onesamnormal},  \ref{thm:multsamnormal}, and \ref{thm:indepnormal} are proved in Section \ref{sec:pfthmnormal36}. 
 Propositions \ref{prop:1} and \ref{prop:2} are proved in Section \ref{sec:proofprop}.

\subsection{Proof of Theorems \ref{thm:onesam},  \ref{thm:allmult} \& \ref{Thm2.1}} \label{sec:pfthm1and4}

When $p$ is fixed, the chi-squared approximations hold by the classical multivariate analysis \citep{anderson2003introduction,Muirhead2009}. Therefore, without loss of generality, the proofs below focus on $p\to \infty$.  
In addition, we note that the analysis of taking subsequences in Section \ref{sec:pfonesam3} can be used similarly in the following proofs, and thus we consider without loss of generality that the sequence $p/n$ has a limit below. 
We next study six likelihood ratio tests in the following subsections separately.

\subsubsection{Proof of Theorem \ref{thm:onesam} (I): Testing One-Sample Mean Vector}\label{sec:pfonesam1}
Similarly to the proof above, we derive the necessary and sufficient conditions for the chi-squared  approximations by examining the moment generating functions.
%condition on when the chi-square approximation of $\log \Lambda_n$ differs from its true limiting distribution. 
Note that testing one-sample mean vector can be viewed as testing coefficient vector $\BSM{\mu}$ of the multivariate linear regression $\MBF{x}_i=1\times \BSM{\mu}+\BSM{\epsilon}_i$, where $\BSM{\epsilon}_i \sim \mathcal{N}(\MBF{0},\BSM{\Sigma}).$ 
Motivated by the approximate expansion of multivariate Gamma function in \cite{Jiang13}, \cite{He2018} studied the moment generating function of the likelihood ratio test in high-dimensional multivariate linear regression. Particularly, by Theorem 3 in \cite{He2018}, we know that when $n,p\to \infty$ and $n-p\to \infty$, \eqref{eq:mmtnormal1} holds with 
% $(\log \Lambda_n-\mu_n)/(n\sigma_n) \xrightarrow{D}  \mathcal{N}(0,1)$, 
\begin{align}
%\mu_n=&~\frac{n}{2}\log\biggr(1-\frac{p}{n}\biggr), \label{eq:onesammeanmu} \\
\mu_n=&~\frac{n}{2}\left\{(n-p-3/2)\log\frac{(n-p)(n-1)}{n(n-1-p)}+\log\left(1-\frac{p}{n}\right)+{p}\log\left(1-\frac{1}{n}\right)\right\}, \label{eq:onesammeanmu} \\
	\sigma_n^2=&~\frac{1}{2}\left\{ \log\left(1-\frac{p}{n} \right) - \log \left(1-\frac{p}{n-1} \right)\right\}. \label{eq:onesammeansigma}
\end{align}
 
Following the analysis in Section \ref{sec:pfonesam3}, 
we know that to derive the necessary and sufficient conditions for the chi-squared approximations  without and with the Bartlett correction,  it is equivalent to examine  \eqref{eq:chisqconvgvar}--\eqref{eq:chisqconvgmean} and \eqref{thm13}--\eqref{thm14}, respectively, with $\mu_n$ in \eqref{eq:onesammeanmu} and $\sigma_n$ in \eqref{eq:onesammeansigma}. 
%Similarly to  Section \ref{sec:pfonesam3}, we consider two cases  $\lim_{n\to \infty} p/n = 0$ and  $\lim_{n\to \infty} p/n = C \in (0,1]$ for each approximation.
% respectively. 

%\paragraph{(I.i) The chi-square approximation}  

\smallskip
\noindent \textit{(I.i) The chi-squared approximation.}\ 
When $p/n \to 0$, we apply Theorem 1 in \cite{He2018}, and know that  \eqref{eq:chisqconvgvar}--\eqref{eq:chisqconvgmean} hold if and only if $p^3/n^2\to 0.$ When $p/n\to C\in (0,1]$,
we have 
\begin{align*}
    2\sigma_n^2=\log\left\{ 1+ \Big(1-\frac{p}{n-1}\Big)^{-1}\frac{p}{n(n-1)} \right\}\sim \frac{C}{n(1-C)},
\end{align*}
% \begin{align*}
%     \log\Big(1-\frac{p}{n-1}\Big) - \log\Big(1-\frac{p}{n}\Big) =& \log\Big(1-(1-\frac{p}{n})^{-1}\frac{p}{n(n-1)} \Big) \\
%     \sim &  -(1-\frac{p}{n})^{-1}\frac{p}{n(n-1)} \sim -\frac{C}{n(1-C)}
% \end{align*}
% and $n-p \to \infty,$ 
%by Taylor's expansion of $\log\{1-p/(n-1)\}$ at the point $1-p/n$, we have
% \begin{align*}
%     \log\Big(1-\frac{p}{n}\Big)-\log\Big(1-\frac{p}{n-1}\Big) = \sum_{t=      1}^{\infty} (-1)^{t+1}\Big(1-\frac{p}{n-1}\Big)^{-t}\frac{1}{t}\times \Big(\frac{p}{n(n-1)}\Big)^t
% \end{align*}
%\begin{align*}
%  \log\Big(1-\frac{p}{n-1}\Big)-  \log\Big(1-\frac{p}{n}\Big) =&~ \sum_{t=      1}^{\infty} (-1)^{t+1}\Big(1-\frac{p}{n}\Big)^{-t}\frac{1}{t}\times \Big(-\frac{p}{n(n-1)}\Big)^t \notag \\
%  \to &~-\sum_{t=1}^{\infty}\frac{1}{t}\Big(\frac{C}{n(1-C%)}\Big)^t \sim -\frac{C}{n(1-C)},
%\end{align*} and then
% \begin{align*}
%     \log\Big(1-\frac{p}{n-1}\Big)=\log\Big(1-\frac{p}{n}\Big)+ \sum_{t=      1}^{\infty} (-1)^{t+1}(1-p/n)^{t+1}\times \frac{p}{n(n-1)}
% \end{align*}
% we have 
% and then $\sigma_n^2\sim C/\{2n(1-C)\}$.
% It follows that
and then  
$\sqrt{2f}/(2n\sigma_n)=\sqrt{2p}/(2n\sigma_n)\to \sqrt{1-C}< 1$. Therefore \eqref{eq:chisqconvgvar} fails, which suggests that the classical chi-squared  approximation fails. 

%For general sequences of $\{p,n_i: 1\leq i\leq k\}$, we can apply the  analysis of taking subsequences similar to that in Section \ref{sec:pfonesam3}. This will be applied 
 
%In summary, for the likelihood ratio test in question (I), the chi-square approximation holds if and only if $p^3/n^2\to 0.$

%\noindent \textit{Case (I.i.1)}
%\noindent \textit{Case (I.i.2)}

%\paragraph{(I.ii) The chi-square approximation with the Bartlett correction}  
\smallskip
\noindent \textit{(I.ii) The chi-squared approximation with the Bartlett correction.}
When $p/n \to 0$, we apply Theorem 2 in \cite{He2018}, and know that \eqref{thm13}--\eqref{thm14} hold if and only if $p^5/n^4\to 0.$ When $p/n\to C\in (0,1]$ and $n-p \to \infty,$ we have $\rho \sim 1-C/2$, and then $\sqrt{2f}/(2n\rho\sigma_n)=(1-C/2)^{-1}\sqrt{2p}/(2n\sigma_n)\to (1-C/2)^{-1}\sqrt{1-C}< 1$. Therefore \eqref{thm13} fails, which suggests that the classical chi-squared approximation with the Bartlett correction fails.

%\noindent \textit{Case (I.ii.1)}
%\noindent \textit{Case (I.ii.2)}

\subsubsection{Proof of Theorem \ref{thm:onesam} (II): Testing One-Sample Covariance Matrix}
%under the conditions of Theorem  \ref{thm:onesam} and $p\to \infty$,
Similarly to the proof in Section \ref{sec:pfonesam3},  by Theorem 1 in \cite{Jiang13} and \cite{Jiang15}, we know that under the conditions of our Theorem \ref{thm:onesam} and $p\to \infty,$ \eqref{eq:mmtnormal1} holds with
\begin{align}
	\mu_n = &~ -\frac{(n-1)p}{2}-\frac{n-1}{2}(n-p-{3}/{2})\log \left(1-\frac{p}{n-1} \right), \label{eq:musigma211} \\
\sigma_n^2= &~ -\frac{1}{2}\left\{ \frac{p}{n-1} + \log \left(1-\frac{p}{n-1} \right)\right\}\times \frac{(n-1)^2}{n^2}.\label{eq:musigma21}
\end{align}
Following the analysis above, 
we know that to derive the necessary and sufficient conditions for the chi-squared approximations  without and with the Bartlett correction,  it is equivalent to examine  \eqref{eq:chisqconvgvar}--\eqref{eq:chisqconvgmean} and \eqref{thm13}--\eqref{thm14}, respectively, with $\mu_n$ in \eqref{eq:musigma211} and $\sigma_n$ in \eqref{eq:musigma21}.  
As analyzed in Section \ref{sec:pfonesam3}, it suffices to discuss two cases $\lim_{n\to \infty} p/n = 0$ and  $\lim_{n\to \infty} p/n = C \in (0,1]$ below. 
% similarly as above. 
%$(\log \Lambda_n-\mu_n)/(n\sigma_n) \xrightarrow{D}  \mathcal{N}(0,1)$, where we define
%Similarly to the analysis above, we discuss two cases $\lim_{n\to \infty} p/n = 0$ and  $\lim_{n\to \infty} p/n = C \in (0,1]$, respectively.

%\paragraph{(II.i) The chi-square approximation} \quad
%We discuss two cases $\lim_{n\to \infty} p/n = 0$ and  $\lim_{n\to \infty} p/n = C \in (0,1]$, respectively. 

\smallskip
\noindent \textit{(II.i) The chi-squared approximation.} \

\smallskip
\noindent \textit{Case (II.i.1) $\lim_{n\to \infty} p/n = 0$.}\  
%When $\lim_{n\to \infty} p/n = 0$,
As $\sqrt{2f}\sim p$, and \eqref{eq:musigma21} and \eqref{eq:sigmanthm21}  are asymptotically the same, by the proof in Section \ref{sec:pfonesam3}, we know that \eqref{eq:chisqconvgvar} holds under this case. We next show that \eqref{eq:chisqconvgmean} holds if and only if $p^2/n\to 0.$ 
By \eqref{eq:chisqconvgvar} and $\sqrt{2f}\sim p$,   \eqref{eq:chisqconvgmean} is equivalent to $	p^{-1}(f+2\mu_n) \to 0.$ 
By Taylor's expansion of $\mu_n$ in \eqref{eq:musigma211}, we obtain
\begin{align*}
	\mu_n = &~ -\frac{(n-1)p}{2}+\frac{(n-1)}{2}(n-p-{3}/{2})\left\{\frac{p}{n-1}+\frac{p^2}{2(n-1)^2}+\frac{p^3}{3(n-1)^3}+O\left(\frac{p^4}{n^4}\right)\right\}. 
\end{align*}
Through calculations, we obtain
\begin{align*}
	p^{-1}(f+2\mu_n) = &~ p^{-1}\times \left\{ -\frac{p^2}{2}+\frac{p^2(n-p)}{2(n-1)}+\frac{p^3n}{3(n-1)^2} + o(p)+O\left(\frac{p^4}{n^2} \right) \right\} \notag \\
	=&~p^{-1}\left\{-\frac{p^3}{6n}+ o(p)+O\left(\frac{p^4}{n^2}\right)\right\} = -\frac{p^2}{6n}\{1+o(1)\}+o(1),\notag %\\
%	=&~ -\frac{p^2}{6n}\{1+o(1)\}+o(1),
\end{align*}
which goes to 0 if and only if $p^2/n \to 0$. 

\smallskip
% Following the proof to Section \ref{sec:pfonesam3}, 

\noindent \textit{Case (II.i.2) $\lim_{n\to \infty} p/n = C \in (0,1]$.}\ 
%When $\lim_{n\to \infty} p/n = C \in (0,1]$, 
Similarly, as  \eqref{eq:musigma21} and \eqref{eq:sigmanthm21}  are asymptotically equal, we can apply the analysis same as Section \ref{sec:pfonesam3}, and know that the chi-squared approximation fails under this case.

%\paragraph{(II.ii) The chi-square approximation with the Bartlett correction} \quad 

\smallskip
\noindent \textit{(II.ii) The chi-squared approximation with the Bartlett correction.} \

\noindent \textit{Case (II.ii.1)  $\lim_{n\to \infty} p/n = 0$.}\  Under this case, we know \eqref{thm13} holds since $\rho = 1 + O(p/n) \rightarrow 1$ and $p/(2 n \sigma_n) \to 1$ as  shown above. Given  \eqref{thm13}, to prove \eqref{thm14}, it is equivalent to prove $p^{-1} ( f+2\rho  \mu_n  ) \to 0$. By Taylor's expansion of $\mu_n$ in \eqref{eq:musigma21}, we have
\begin{align*}
\mu_n = & -\frac{p(n-1)}{2}+\frac{(n-p-{3}/{2})(n-1)}{2}\left\{\frac{p}{n-1}+\frac{p^2}{2(n-1)^2}+\frac{p^3}{3(n-1)^3}+\frac{p^4}{4(n-1)^4}+O\left(\frac{p^5}{n^5}\right)\right\}.	
\end{align*}
After calculations, we obtain
%\begin{align*}
%(n-1)\rho\mu_n = &~ -p(p+\frac{1}{2})+\frac{p^3}{3(n-1)} + \frac{p^2(n-p)}{2(n-1)} - \frac{p^3(n-p)}{6(n-1)^2}\\
%&~ +\frac{p^3(n-p)}{3(n-1)^2}-\frac{p^4n}{9(n-1)^3} + \frac{p^4n}{4(n-1)^3} + o(p)+O\left(\frac{p^5}{n^3} \right).
%\end{align*}
\begin{align*}
2\rho\mu_n = &~ -p\left(p+\frac{1}{2}\right)+\frac{p^3}{3(n-1)} + \frac{p^2(n-p)}{2(n-1)} - \frac{p^3(n-p)}{6(n-1)^2}\\
&~ +\frac{p^3(n-p)}{3(n-1)^2}-\frac{p^4n}{9(n-1)^3} + \frac{p^4n}{4(n-1)^3} + o(p)+O\left(\frac{p^5}{n^3} \right).
\end{align*}
%\begin{equation*}
%\begin{aligned}
%\mu_n = & -p+(n-p-{3}/{2})\left\{\frac{p}{n-1}+\frac{p^2}{2(n-1)^2}+\frac{p^3}{3(n-1)^3}+\frac{p^4}{4(n-1)^4}+O\left(\frac{p^5}{n^5}\right)\right\}\\
%(n-1)\mu_n = & -p(n-1)+p(n-p-\frac{3}{2})+\frac{p^2(n-p-\frac{3}{2})}{2(n-1)}+\frac{p^3(n-p-\frac{3}{2})}{3(n-1)^2}\\
%& +\frac{p^4(n-p-\frac{3}{2})}{4(n-1)^3} +O(\frac{p^5}{n^3})\\
%=&-p(p+\frac{1}{2})+\frac{p^2(n-p)}{2(n-1)} + \frac{p^3(n-p)}{3(n-1)^2} + \frac{p^4(n-p)}{4(n-1)^3} + o(p)+O(\frac{p^5}{n^3})\\
%(n-1)\rho\mu_n = & -p(p+\frac{1}{2})+\frac{(2p^2+p+2)(p+\frac{1}{2})}{6(n-1)} + \frac{p^2(n-p)}{2(n-1)} - \frac{p(n-p)(2p^2+p+2)}{12(n-1)^2}\\ &+\frac{p^3(n-p)}{3(n-1)^2}-\frac{p^2(n-p)(2p^2+p+2)}{18(n-1)^3} + \frac{p^4(n-p)}{4(n-1)^3} + o(p)+O(\frac{p^5}{n^3})
%\\
%=& -p(p+\frac{1}{2})+\frac{p^3}{3(n-1)} + \frac{p^2(n-p)}{2(n-1)} - \frac{p^3(n-p)}{6(n-1)^2}\\
%&+\frac{p^3(n-p)}{3(n-1)^2}-\frac{p^4n}{9(n-1)^3} + \frac{p^4n}{4(n-1)^3} + o(p)+O(\frac{p^5}{n^3})
%\end{aligned}
%\end{equation*}
It follows that 
%the numerator of \eqref{thm14} is
%\begin{align*}
%	&~ f + (n-1)\rho\mu_n \notag \\
%=&~ -\frac{p^2}{2}+\frac{p^2(n-p)}{2(n-1)}+\frac{p^3}{3(n-1)}-\frac{p^3(n-p)}{6(n-1)^2}+\frac{p^3(n-p)}{3(n-1)^2} +\frac{5p^4n}{36(n-1)^3}+ o(p)+O\left(\frac{p^5}{n^3}\right)\\
%=&~-\frac{p^4}{36n^2}+ o(p)+O\left(\frac{p^5}{n^3}\right). 
%\end{align*}
\begin{align*}
	&~ f +2\rho\mu_n \notag \\
=&~ -\frac{p^2}{2}+\frac{p^2(n-p)}{2(n-1)}+\frac{p^3}{3(n-1)}-\frac{p^3(n-p)}{6(n-1)^2}+\frac{p^3(n-p)}{3(n-1)^2} +\frac{5p^4n}{36(n-1)^3}+ o(p)+O\left(\frac{p^5}{n^3}\right)\\
=&~-\frac{p^4}{36n^2}+ o(p)+O\left(\frac{p^5}{n^3}\right). 
\end{align*}
%\begin{equation*}
%\begin{aligned}
%f_1 + (n-1)\rho\mu_n =& -\frac{p^2}{2}+\frac{p^2(n-p)}{2(n-1)}+\frac{p^3}{3(n-1)}-\frac{p^3(n-p)}{6(n-1)^2}+\frac{p^3(n-p)}{3(n-1)^2} +\frac{5p^4n}{36(n-1)^3}\\
%&+ o(p)+O(\frac{p^5}{n^3})\\
%=& -\frac{p^3}{2(n-1)}+\frac{p^3}{3(n-1)}+\frac{p^3(n-p)}{6(n-1)^2} +\frac{5p^4n}{36(n-1)^3}+ o(p)+O(\frac{p^5}{n^3})\\
%=&-\frac{p^3}{6(n-1)}+\frac{p^3(n-p)}{6(n-1)^2} +\frac{5p^4n}{36(n-1)^3}+ o(p)+O(\frac{p^5}{n^3})\\
%=&-\frac{p^4}{6(n-1)^2}+\frac{5p^4n}{36(n-1)^3}+ o(p)+O(\frac{p^5}{n^3})\\
%= & -\frac{p^4n}{36(n-1)^3}+ o(p)+O(\frac{p^5}{n^3})\\
%= &-\frac{p^4}{36n^2}+ o(p)+O(\frac{p^5}{n^3})
%\end{aligned}
%\end{equation*}
Therefore $p^{-1}\{f+ \mu_n\rho(n-1)\} \to 0$ if and only if $p^3/n^2 \rightarrow 0$.
%\[
%\frac{f_1+\mu_n\rho(n-1)}{p} = -\frac{p^3}{36n^2}(1+o(1))+o(1) \rightarrow 0
%\]
%if and only if $p^3/n^2 \rightarrow 0$.

\smallskip

\noindent \textit{Case (II.ii.2)  $\lim_{n\to \infty} p/n = C \in (0,1]$.}\ Under this case, we have $\rho \to 1- C/3$. Similarly, as  \eqref{eq:musigma21} \eqref{eq:sigmanthm21}  are asymptotically equal, we can apply the proof same as in Section \ref{sec:pfonesam3}, and know that the chi-squared approximation with the Bartlett correction also fails under this case.

\subsubsection{Proof of Theorem \ref{thm:allmult} (IV): Testing the Equality of Several Mean Vectors}\label{sec:pfmultimean}
Note that testing the equality of several mean vectors can be viewed as testing the coefficient matrix in multivariate linear regression; see, Section 10.7 in \citet{Muirhead2009}. 
Similarly to Section \ref{sec:pfonesam1},  by Theorem 3 in \cite{He2018}, we know that when $n,p\to \infty$ and $n-p\to \infty$, \eqref{eq:mmtnormal1} holds with 
\begin{align}
\mu_n=&~\frac{n}{2}\Biggr\{(n-p-k-{1}/{2})\log\frac{(n-1-p)(n-k)}{(n-p-k)(n-1)}\label{eq:multisammeanmu} \\
&~\quad \quad +(k-1)\log\frac{(n-1-p)}{(n-1)}+{p}\log\frac{(n-k)}{(n-1)}\Biggr\}, \notag \\
\sigma_n^2=&~\frac{1}{2}\left\{ \log\left(1-\frac{p}{n-1} \right) - \log\left( 1-\frac{p}{n-k}\right)\right\}. \label{eq:multisammeansigma}
\end{align}
Following the analysis in Section \ref{sec:pfonesam3}, 
we know to derive the necessary and sufficient conditions for the chi-squared approximations  without and with the Bartlett correction,  it is equivalent to examine  \eqref{eq:chisqconvgvar}--\eqref{eq:chisqconvgmean} and \eqref{thm13}--\eqref{thm14}, respectively,  with $\mu_n$ in \eqref{eq:multisammeanmu} and $\sigma_n$ in \eqref{eq:multisammeansigma}.

%\paragraph{(IV.i) The chi-square approximation} 

\smallskip
\noindent \textit{(IV.i) The chi-squared approximation.}
When $p/n \to 0$, we apply Theorem 1 in \cite{He2018}, and know that  \eqref{eq:chisqconvgvar}--\eqref{eq:chisqconvgmean} hold if and only if $p^3/n^2\to 0.$ When $p/n\to C\in (0,1]$ and $n-p \to \infty,$ we have $\sigma_n^2\sim C(k-1)/\{2n(1-C)\}$, and then  
$\sqrt{2f}/(2n\sigma_n)=\sqrt{2(k-1)p}/(2n\sigma_n)\to \sqrt{1-C}< 1$. Therefore \eqref{eq:chisqconvgvar} fails, which suggests that the classical chi-squared  approximation fails.

%\noindent \textit{Case (I.i.1)}
%\noindent \textit{Case (I.i.2)}

%\paragraph{(IV.ii) The chi-square approximation with the Bartlett correction} 

\smallskip
\noindent \textit{(IV.ii) The chi-squared approximation with the Bartlett correction.}\ 
When $p/n \to 0$, we apply Theorem 2 in \cite{He2018}, and know that \eqref{thm13}--\eqref{thm14} hold if and only if $p^5/n^4\to 0.$ When $p/n\to C\in (0,1]$ and $n-p \to \infty,$ we have $\rho \sim 1-C/2$, and then $\sqrt{2f}/(2n\rho\sigma_n)=(1-C/2)^{-1}\sqrt{2p}/(2n\sigma_n)\to (1-C/2)^{-1}\sqrt{1-C}< 1$. Therefore \eqref{thm13} fails, which suggests that the classical chi-squared approximation with the Bartlett correction fails.

%\noindent \textit{Case (I.ii.1)}
%\noindent \textit{Case (I.ii.2)}

\subsubsection{Proof of Theorem \ref{thm:allmult} (V): Testing the Equality of Several Covariance Matrices}

Similarly to the proof in Section \ref{sec:pfonesam3}, by Theorem 4 in \cite{Jiang13} and \cite{Jiang15}, we know that under the conditions of Theorem \ref{thm:allmult} and $p\to \infty$, \eqref{eq:mmtnormal1} holds with 
\begin{align}
\mu_n=&~\frac{1}{4}\left\{{(n-k)(2 n-2 p-2 k-1) \log \left(1-\frac{p}{n-k}\right)}\right. \label{eq:multicovmun} \\
 &~\left.{\quad\quad -\sum_{i=1}^{k}\left(n_{i}-1\right)\left(2 n_{i}-2 p-3\right) \log \left(1-\frac{p}{n_{i}-1}\right)}\right\}, \notag \\
 \sigma_{n}^{2}=&~\frac{(n-k)^2}{2n^2}\left\{\log \left(1-\frac{p}{n-k}\right)-\sum_{i=1}^{k}\left(\frac{n_{i}-1}{n-k}\right)^{2} \log \left(1-\frac{p}{n_{i}-1}\right)\right\}. \label{eq:multicovsigman}
\end{align}
Following the analysis in Section \ref{sec:pfonesam3}, 
we next derive the equivalent conditions for  \eqref{eq:chisqconvgvar}--\eqref{eq:chisqconvgmean} and \eqref{thm13}--\eqref{thm14}, respectively,  with $\mu_n$ in \eqref{eq:multicovmun} and $\sigma_n$ in \eqref{eq:multicovsigman}.

%we know to derive the necessary and sufficient conditions for the chi-square approximation without/with Bartlett correction,  it is equivalent to examine  \eqref{eq:chisqconvgvar}--\eqref{eq:chisqconvgmean} and \eqref{thm13}--\eqref{thm14} with $\mu_n$ in \eqref{eq:multisammeanmu} and $\sigma_n$ in \eqref{eq:multisammeansigma}. 
%we have 

%\paragraph{(V.i) The chi-square approximation} \quad 

\smallskip
\noindent \textit{(V.i) The chi-squared approximation.}

%In particular, The conditions of Theorem \ref{thm:allmult} implies $n_i=\Theta(n)$.
\noindent \textit{Case (V.i.1) $\lim_{n\to \infty} p/n = 0$.}\  Under this case, we show that \eqref{eq:chisqconvgvar} holds. By Taylor's expansion, 
%we have 
\begin{align*}
\sigma_n^2&=\frac{(n-k)^2}{2n^2}\left[ -\frac{p}{n-k}-\frac{p^2}{2(n-k)^2}+\sum_{i=1}^k\left(\frac{n_i-1}{n-k}\right)^2\left\{\frac{p}{n_i-1}+\frac{p^2}{2(n_i-1)^2}\right\}+O\left(\frac{p^3}{n^3}\right)\right]\\
&=\frac{(n-k)^2}{2n^2}\left\{-\frac{p}{n-k}+\sum_{i=1}^k{\frac{p(n_i-1)}{(n-k)^2}}-\frac{p^2}{2(n-k)^2}+\frac{kp^2}{2(n-k)^2}+O\left(\frac{p^3}{n^3}\right)\right\}\\
&=\frac{(k-1)p^2}{4n^2}\left\{1+o(1)\right\}, 
\end{align*}where we use $n_i=\Theta(n)$.
As $\sqrt{2f}\sim p\sqrt{k-1}$, we have \eqref{eq:chisqconvgvar} holds. Given \eqref{eq:chisqconvgvar}, we know that \eqref{eq:chisqconvgmean} is equivalent to $(2f+4\mu_n)/(2p\sqrt{k-1})\to 0.$ Through Taylor's expansion, we obtain 
\begin{align*}
	4\mu_n=&~ -p(2n-2p-2k-1)-\frac{(n-p)p^2}{n-k}-\frac{2(n-p)p^3}{3(n-k)^2}+o\left(\frac{p^3}{n}\right)+o(p)\\
&~+\sum_{i=1}^{k}p(2 n_{i}-2 p-3)+\sum_{i=1}^{k}\frac{\left(n_{i}-p\right)p^2}{\left(n_{i}-1\right)}+\sum_{i=1}^{k}\frac{2\left(n_{i}-p\right)p^3}{3\left(n_{i}-1\right)^2}+o\left(\frac{p^3}{n}\right)\\
=&~p(p-kp-k+1)+\frac{p^3}{3(n-k)}-\sum_{i=1}^k\frac{p^3}{3(n_i-1)}+o\left(\frac{p^3}{n}\right)+o(p).
\end{align*}
By $f=p(p+1)(k-1)/2$, we have 
\begin{align}
	2f+4\mu_n = \frac{p^3}{3}\left( \frac{1}{n-k}-\sum_{i=1}^k\frac{1}{n_i-1}\right) + o\left(\frac{p^3}{n}\right)+o(p)=\Theta(p^3/n) + o(p), \label{eq:multicovmean1zero}
\end{align}
where we use the fact that $(n-k)^{-1}-\sum_{i=1}^k(n_i-1)^{-1}>0$.
% by the strong convexity of the function $1/x$ when $x>0$. 
It follows that $(2f+4\mu_n)/(2p\sqrt{k-1})=\Theta(p^2/n),$ which converges to 0 if and only if $p^2/n\to 0.$

\smallskip

\noindent \textit{Case (V.i.2) $\lim_{n\to \infty} p/n = C \in (0,1]$.}
Under this case, we show that  \eqref{eq:chisqconvgvar} and \eqref{eq:chisqconvgmean}  do not hold at the same time. Particularly, \eqref{eq:chisqconvgvar} and \eqref{eq:chisqconvgmean} together induce $4(\mu_n+n^2\sigma_n^2)/(2f)\to 0$, which indicates  $2(\mu_n+n^2\sigma_n^2)n^{-2}\to 0$, and thus  $g_1(C) = 0$, where we define
\begin{align*}
	g_1(C)=(2-C)\log(1-C)-\sum_{i=1}^k\delta_i(2\delta_i-C) \log(1-C\delta_i^{-1}),
\end{align*} and we assume $n_i/n\to \delta_i\in (0,1)$ for $ i=1,\ldots, k.$  
%Note that 
As $p/n=(p/n_i)\times (n_i/n) < n_i/n$, we have $0<C\leq \delta_i<1$ for $i=1,\ldots, k$. 
We next show that $g_1(C)>0$ for $C\in (0,\min_{i=1,\ldots, k}\delta_i]$ by taking derivative of $ g_1(C)$. Specifically, by $\sum_{i=1}^k \delta_i=1$ and calculations, we have
\begin{align*}
	g_1'(C)=&~\sum_{i=1}^k\delta_i\left\{-\log(1-C)-{(1-C)^{-1}}+\log(1-C\delta_i^{-1})+{\delta_i}{(\delta_i-C)^{-1}} \right\}, \notag \\
	g_1''(C)=&~ \sum_{i=1}^k\delta_i\times C\left\{ -(1-C)^{-2} + (\delta_i-C)^{-2} \right\}.
\end{align*}When $0<C\leq \delta_i<1$ for $i=1,\ldots, k$, we have $g_1''(C)>0$ and thus $g_1'(C)$ is a monotonically increasing function of $C$. As $g_1'(0)=0$, $g_1'(C)>0$ when $0<C<1$ and then $g_1(C)$ is also monotonically increasing. By $g_1(0)=0$, we further obtain $g_1(C)>0$ when $0<C<1$, which contradicts with  $g_1(C) = 0$. As a result, we know \eqref{eq:chisqconvgvar} and \eqref{eq:chisqconvgmean}  do not hold  simultaneously, which suggests that the chi-squared approximation fails.

\smallskip
\noindent \textit{(V.ii) The chi-squared approximation with the Bartlett correction.}
%When $\lim_{n\to \infty} p/n = 0$, we know \eqref{thm13} holds by $\rho = 1 + O(p/n) \rightarrow 1$ and \eqref{eq:chisqconvgvar}  above. 
When $\lim_{n\to \infty} p/n = 0$, since $\rho = 1 + O(p/n) \rightarrow 1$ and \eqref{eq:chisqconvgvar} is proved above, we know \eqref{thm13} holds.  
Given  \eqref{thm13}, as $f\sim p^2(k-1)/2$, to prove \eqref{thm14}, it is equivalent to show $( 2f+4\rho  \mu_n  )/p\to 0$, which is also equivalent to $(2f+4\mu_n-4\Delta_n\mu_n)/p\to 0,$
where we redefine in this subsection that
%. Note that $\rho=1-\Delta_n$ where 
\begin{align*}
	\Delta_n=\frac{2p^2+3p-1}{6(p+1)(k-1)}\times \tilde{D}_{n,1}, \quad \quad  \tilde{D}_{n,1}=\sum_{i=1}^k\frac{1}{n_i-1}-\frac{1}{n-k}.
\end{align*}
%\begin{align*}
%	\Delta_n=\frac{2p^2+3p-1}{6(p+1)(k-1)}\left( \sum_{i=1}^k\frac{1}{n_i-1}-\frac{1}{n-k}\right).
%\end{align*}
%Thus it is equivalent to show that $p^{-1}()$
%By Taylor's expansion of $\mu_n$ in \eqref{eq:multicovmun}, we have 
Similarly to the analysis of \eqref{eq:multicovmean1zero}, through Taylor's expansion of  $\mu_n$ in \eqref{eq:multicovmun}, we  obtain 
\begin{align}
	2f+4\mu_n=-\frac{p^3}{3}\times \tilde{D}_{n,1}-\frac{p^4}{6}\times \tilde{D}_{n,2}+ o\left(\frac{p^4}{n^2}\right)+o(p),\label{eq:multicovrho1}
\end{align}
%\begin{align*}
%	2f+4\mu_n=-p^3\times D_1/3-p^4\times D_2/6+ o(p^4/n^2)+o(p),
%\end{align*}
%$2f+4\mu_n=-p^3\times D_1/3-p^4\times D_2/6+ o(p^4/n^2)+o(p)$, 
where $\tilde{D}_{n,2}=\sum_{i=1}^k (n_i-1)^{-2}-(n-k)^{-2}.$ Moreover, by \eqref{eq:multicovmean1zero} and  $\Delta_n=O(p/n)=o(1)$,  we have
\begin{align}
	4\Delta_n \mu_n=\Delta_n\left(-\frac{p^3}{3}\times \tilde{D}_{n,1}-2f\right)+o\left(\frac{p^4}{n^2}\right)+o(p),\label{eq:multicovrho2}
\end{align}
%\begin{align*}
%	4\Delta_n \mu_n=\Delta_n(-p^3D_1/3-2f)+o(p^4/n^2)+o(p),
%\end{align*} %where we use $\Delta_n\sim pD_1/\{3(k-1)\}$.  
%where we use $\Delta_n=O(p/n)=o(1)$. 
Combining \eqref{eq:multicovrho1} and \eqref{eq:multicovrho2}, we have
\begin{align}
2f+4\mu_n-4\Delta_n\mu_n =&~-\frac{p^3}{3}\times \tilde{D}_{n,1}-\frac{p^4}{6}\times \tilde{D}_{n,2}+\Delta_n\left(\frac{p^3}{3}\times \tilde{D}_{n,1}+2f\right)+ o\left(\frac{p^4}{n^2}\right)+o(p), \notag\\
%=&~-\frac{p^4}{6}\times D_{n,2}+\frac{p^4}{9(k-1)}D_{n,1}^2+o\left(\frac{p^4}{n^2}\right)+o(p), \\
=&~\frac{p^4}{18(k-1)}\left\{2\tilde{D}_{n,1}^2-3(k-1)\tilde{D}_{n,2} \right\}+o\left(\frac{p^4}{n^2}\right)+o(p), \label{eq:multicovrho3}
\end{align} where we use $\tilde{D}_{n,1}=\Theta(n^{-1})$, $\tilde{D}_{n,2}=\Theta(n^{-2})$, $\Delta_n=p\tilde{D}_{n,1}/\{3(k-1)\}+o(p/n)$, and $2\Delta_n f=p^3\tilde{D}_{n,1}/3+o(p).$ 

We next show that  $\eqref{eq:multicovrho3}=\Theta(p^4n^{-2})$. 
% $2D_{n,1}^2-3(k-1)D_{n,2}= \Theta(n^{-2}).$  
In particular, in this subsection, we redefine  $\delta_i=(n_i-1)/(n-k)$, which satisfies $\sum_{i=1}^k \delta_i=1$. Then by the definitions of $\tilde{D}_{n,1}$ and $\tilde{D}_{n,2}$, we calculate that
\begin{align}
&~(n-k)^2\times \{2\tilde{D}_{n,1}^2-3(k-1)\tilde{D}_{n,2}\}\notag \\
% =&~ 2\Big( \sum_{i=1}^k\delta_i^{-1}-1\Big)^2-3(k-1)\Big( \sum_{i=1}^k\delta_i^{-2}-1\Big) \notag \\
=&~(5-3k)\sum_{i=1}^k\delta_i^{-2}+2\sum_{1\leq i\neq j\leq k}\delta_i^{-1}\delta_j^{-1}-4\sum_{i=1}^k\delta_i^{-1}+3k-1. \label{eq:multicovrho4}
\end{align}
As $2\delta_i^{-1}\delta_j^{-1}\leq \delta_i^{-2}+\delta_j^{-2}$, we have
\begin{align}
 \eqref{eq:multicovrho4}\leq &~  (3-k)\sum_{i=1}^k\delta_i^{-2} -4\sum_{i=1}^k\delta_i^{-1}+3k-1 \notag \\
 \leq &~(3-k)\sum_{i=1}^k\delta_i^{-2} -4k^2+3k-1,\label{eq:multicovrho5}
\end{align}where in the last inequality, we use $\sum_{i=1}^k\delta_i^{-1}\geq k^2 (\sum_{i=1}^k\delta_i )^{-1}=k^2.$  Therefore $\eqref{eq:multicovrho5}<0$ when $k\geq 3.$ When $k=2$, as $\delta_1+\delta_2=1$, we have  $\delta_1^{-1}+\delta_2^{-1}=\delta_1^{-1}\delta_2^{-1}$ and $\eqref{eq:multicovrho4}=-\sum_{i=1}^2\delta_i^{-2}-2\sum_{i=1}^2\delta_i^{-1}+5$. As $\sum_{i=1}^2\delta_i^{-1}\geq 2^2,$  $\eqref{eq:multicovrho4}< - 2\times 2^2+5<0.$ In summary, we know $\eqref{eq:multicovrho4}<0$ for $k\geq 2$, and thus $\eqref{eq:multicovrho3}=\Theta(p^4n^{-2})$. If follows that  $( 2f+4\rho  \mu_n  )/p\to 0$ if and only if $p^3/n^2\to 0.$ In summary, we know for testing problem (V),  the chi-squared approximation with the Bartlett correction works if and only if $p^3/n^2\to 0.$

\subsubsection{Proof of Theorem \ref{thm:allmult} (VI): Joint Testing the Equality of Several Mean Vectors and  Covariance Matrices}

Similarly to the proof in Section \ref{sec:pfonesam3}, by Theorem 3 in \cite{Jiang13} and \cite{Jiang15}, we know that under the conditions of Theorem \ref{thm:allmult} and $p\to \infty$, \eqref{eq:mmtnormal1} holds with 
\begin{align}
\mu_n =&~ \frac{1}{4} \left\{-2kp-\sum_{i=1}^k \frac{p}{n_i}-nL_{n,p}(2p-2n+3)+\sum_{i=1}^{k}n_iL_{n_i-1,p}(2p-2n_i+3)\right\}, \label{eq:multijointmu} \\
%\mu_n =&~ \frac{1}{4} \left\{-2kp-\sum_{i=1}^k \frac{p}{n_i}+nr_n^2(2p-2n+3)-\sum_{i=1}^{k}n_ir_{n_i'}^2(2p-2n_i+3)\right\}, \label{eq:multijointmu} \\
%\sigma_n^2 =&~ \frac{1}{2}\left( \sum_{i=1}^{k}\frac{n_i^2}{n^2} \times r_{n_i'}^2 - r_n^2\right), \label{eq:multijointsigma} \\
\sigma_n^2 =&~ \frac{1}{2}\left( L_{n,p}-\sum_{i=1}^{k}\frac{n_i^2}{n^2} \times L_{n_i-1,p} \right), \label{eq:multijointsigma}
\end{align} where $L_{n,p}=\log (1-p/n)$. 
% where $n_i'=n_i-1$ and $r_x^2=-\log (1-p/x)$ for $x>p$. 
Following 
%the analysis in 
Section \ref{sec:pfonesam3}, 
we next derive the equivalent conditions for  \eqref{eq:chisqconvgvar}--\eqref{eq:chisqconvgmean} and \eqref{thm13}--\eqref{thm14}, respectively,  with $\mu_n$ in \eqref{eq:multijointmu} and $\sigma_n$ in \eqref{eq:multijointsigma}.

%\paragraph{(VI.i) The chi-square approximation}  \quad
%\smallskip

\smallskip
\noindent \textit{(VI.i) The chi-squared approximation.}

\noindent \textit{Case (VI.i.1) $\lim_{n\to \infty} p/n = 0$.}  
Under this case, we show that \eqref{eq:chisqconvgvar} holds. 
As $-\log(1-x)=x+x^2/2+O(x^3)$ and $n_i=\Theta(n)$, we obtain
%Through Taylor's expansion of the function $\log(1-x)$, we obtain
\begin{align*}
	2\sigma_n^2=&~=\sum_{i=1}^{k}\frac{n_i^2}{n^2}\left\{\frac{p}{n_i-1}+\frac{p^2}{2(n_i-1)^2}\right\} - \frac{p}{n} - \frac{p^2}{2n^2} +O\left(\frac{p^3}{n^3}\right) \notag \\
	=&~\sum_{i=1}^{k}\frac{n_i^2}{n^2}\left( \frac{p}{n_i}+\frac{p}{n_i^2}+\frac{p^2}{2n_i^2}\right) - \frac{p}{n} - \frac{p^2}{2n^2} +O\left(\frac{p^3}{n^3}\right) \notag \\
	=&~ \frac{kp}{n^2}+\frac{(k-1)p^2}{2n^2}+O\left(\frac{p^3}{n^3}\right),
\end{align*} where in the second equation, we use $(n_i-1)^{-1}=n_i^{-1}+n_{i}^{-2}+O(n_i^{-3})$ and $(n_i-1)^{-2}=n_i^{-2}+O(n_i^{-3})$. It follows that $2n\sigma_n \sim p\sqrt{k-1}$. By $\sqrt{2f}\sim p\sqrt{k-1}$, we have \eqref{eq:chisqconvgvar}.   
Given \eqref{eq:chisqconvgvar}, we know that \eqref{eq:chisqconvgmean} is equivalent to $(2f+4\mu_n)/p\to 0.$ As $p/n=o(1)$, through Taylor's expansion, we obtain 
\begin{align}
-n(2p-2n+3)L_{n,p}=&~n(2p-2n+3)\left\{\frac{p}{n}+\frac{p^2}{2n^2}+\frac{p^3}{3n^3}+O\left(\frac{p^4}{n^4}\right) \right\} \label{eq:rnexpmulti} \\
%	=&~p(2p-2n+3)\left\{1+\frac{p}{2n}+\frac{p^2}{3n^2}+O\left(\frac{p^3}{n^3} \right) \right\}\notag \\
	=&~p\left\{ 2p+\frac{p^2}{n} -2n -p-\frac{2p^2}{3n}+3 + O\left(\frac{p^3}{n^2} \right) + o(1) \right\}\notag \\
	=&~p\left\{ p+\frac{p^2}{3n} -2n + 3 + O\left(\frac{p^3}{n^2} \right) + o(1) \right\}. \notag
\end{align}
Similarly, by Taylor's expansion and $n_i=\Theta(n)$, we have
\begin{align}
&~ -n_i(2p-2n_i+3)L_{n_i-1,p}\label{eq:rnexpmulti2} \\
=&~ n_i(2p-2n_i+3) \left\{\frac{p}{n_i-1}+\frac{p^2}{2(n_i-1)^2}+\frac{p^3}{3(n_i-1)^3}+O\left(\frac{p^4}{n^4}\right) \right\}\notag \\
	=&~ n_i(2p-2n_i+3) \left\{\frac{p}{n_i}+\frac{p}{n_i^2}+\frac{p^2}{2n_i^2}+\frac{p^3}{3n_i^3}+O\left(\frac{p^4}{n^4}\right)+O\left(\frac{p^2}{n^3}\right) \right\}\notag \\
	=&~p\left\{ p+\frac{p^2}{3n_i} -2n_i +3-2  + O\left(\frac{p^3}{n_i^2} \right) + o(1) \right\}, \notag 
\end{align} where in the second equation, we use $(n_i-1)^{-1}=n_i^{-1}+n_i^{-2}+O(n_i^{-3})$ and $(n_i-1)^{-a}=n_i^{-a}+O(n_i^{-3})$ for integers $a\geq 2$. 
%It follows that
Combining \eqref{eq:rnexpmulti} and \eqref{eq:rnexpmulti2}, we obtain
\begin{align}
	2f+4\mu_n=&~2f -2kp+p\left\{(1-k)p+\frac{p^2}{3}\Big(\frac{1}{n}-\sum_{i=1}^k\frac{1}{n_i} \Big)+3-k\right\}+O\Big(\frac{p^4}{n^2} \Big)+o(p)\label{eq:multiidentmunexp} \\
	=&~\frac{p^3}{3}\Big(\frac{1}{n}-\sum_{i=1}^k\frac{1}{n_i} \Big)+O\Big(\frac{p^4}{n^2} \Big)+o(p). \notag
\end{align} As $n^{-1}-\sum_{i=1}^kn_i^{-1}=\Theta(n^{-1})$, we have  $2f+4\mu_n=\Theta(p^3n^{-1})$. Therefore we know $(2f+4\mu_n)/p\to 0$ if and only if $p^2/n\to 0.$ 

%\begin{align}
%	4\mu_n=&~ -2kp+p\left\{(1-k)p+\frac{p^2}{3}\Big(\frac{1}{n}-\sum_{i=1}^k\frac{1}{n_i} \Big)+3-k\right\}+O\Big(\frac{p^4}{n^2} \Big)+o(p)\label{eq:multiidentmunexp} \\
%	=&~3(1-k)p+(1-k)p^2+\frac{p^3}{3}\Big(\frac{1}{n}-\sum_{i=1}^k\frac{1}{n_i} \Big)+O\Big(\frac{p^4}{n^2} \Big)+o(p), \notag
%\end{align} 
%
%and thus $2f+4\mu_n=\Theta(p^3n^{-1})$. Therefore we know $(2f+4\mu_n)/p\to 0$ if and only if $p^2/n\to 0.$ 

\smallskip

\noindent \textit{Case (VI.i.2) $\lim_{n\to \infty} p/n = C \in (0,1]$.}\  
%Suppose $n_i/n\to \delta_i\in (0,1)$. 
In this subsection, we redefine $\delta_i=n_i/n \in (0,1).$
Then
\begin{align*}
	\frac{4n^2\sigma_n^2}{2f}\to \frac{2}{C^2(k-1)}\times \Big\{\log(1-C)-\sum_{i=1}^k \delta_i^2\log(1-C\delta_i^{-1})\Big\},
\end{align*} where $0<C\leq \delta_i<1.$ 
% $2n\sigma_n/\sqrt{2f}\to 1$
Therefore \eqref{eq:chisqconvgvar}  induces $g_2(C) = 0$, where we define
\begin{align*}
	g_2(C)=\log(1-C)-\sum_{i=1}^k \delta_i^2\log(1-C\delta_i^{-1})-{(k-1)C^2}/{2}.
\end{align*} 
By taking derivative of $g_2(C)$, we obtain $g_2'(0)=0$, $g_2''(0)=0$, and 
\begin{align*}
%	g'_2(C)=&~ -\frac{1}{1-C}+\sum_{i=1}^k\frac{\delta_i^2}{\delta_i-C}-(k-1)C \notag \\
%	g''_2(C)=&~ -\frac{1}{(1-C)^2}+\sum_{i=1}^k\frac{\delta_i^2}{(\delta_i-C)^2}  -(k-1)\notag \\
	g'''_2(C)=&~ \frac{2}{(C-1)^2}-\sum_{i=1}^k\frac{2\delta_i^2}{(C-\delta_i)^3} = \sum_{i=1}^k\frac{2\delta_i(1-\delta_i)(C^3-3\delta_i C+\delta_i^2+\delta_i)}{(1-C)^3(\delta_i-C)^3}. 
\end{align*}
As $C^3-3\delta_i C+\delta_i^2+\delta_i$  is a monotonically decreasing function of $C$ when $0<C\leq \delta_i<1$, and it equals $\delta_i(\delta_i-1)^2>0$ when $C=\delta_i$, we have $g'''_2(C)>0$ for $0<C\leq \delta_i$. 
It follows that $g_2(C)$ is a monotonically increasing function when $0<C\leq \delta_i<1$. 
%In Lemma \ref{lm:multivcovgder}, by taking derivative of $g_2(C)$, we obtain that $g_2(C)$ is a monotonically increasing function when $0<C\leq \delta_i<1$. 
As $g_2(0)=0$, we have 
 $g_2(C)>0$, which contradicts with $g_2(C) = 0$. Therefore, we know that \eqref{eq:chisqconvgvar} does not hold under this case, which implies that the chi-squared approximation fails. 

%We calculate that
%\begin{align*}
%	g'(C)=&~ -\frac{1}{1-C}+\sum_{i=1}^k\frac{\delta_i^2}{\delta_i-C}-(k-1)C \notag \\
%	g''(C)=&~ -\frac{1}{(1-C)^2}+\sum_{i=1}^k\frac{\delta_i^2}{(\delta_i-C)^2}  -(k-1)\notag \\
%	g'''(C)=&~ \frac{2}{(C-1)^2}-\sum_{i=1}^k\frac{2\delta_i^2}{(C-\delta_i)^3} = \sum_{i=1}^k\frac{2\delta_i\left\{ \delta_i(1-C)^3-(\delta_i-C)^3\right\}}{(1-C)^3(\delta_i-C)^3}.
%%	\frac{2}{(C-1)^2}-\sum_{i=1}^k\frac{2\delta_i^2}{(C-\delta_i)^3} =
%\end{align*} We note that $g(0)=0$, $g'(0)=0$, and $g''(0)=0$. Moreover, we have
%\begin{align}
%	 \delta_i(1-C)^3-(\delta_i-C)^3=(1-\delta_i)(C^3-3\delta_i C+\delta_i^2+\delta_i),\label{eq:multicovdev1}
%\end{align} which is a monotonically decreasing function of $C$ when $0<C\leq \delta_i<1$. When $C=\delta_i$, we have $\eqref{eq:multicovdev1}=\delta_i(\delta_i-1)^2>0$ as $0<\delta_i<1$. Therefore $\eqref{eq:multicovdev1} >0$, we know $g''(C)$ is a monotonically decreasing function of $C$ when 

%We note that $g(0)=0$, $g'(0)=0$, $g''(0)=0$

%\paragraph{(VI.ii) The chi-square approximation with the Bartlett correction} 
\smallskip
\noindent \textit{(VI.ii) The chi-squared approximation with the Bartlett correction.} \ 
When $\lim_{n\to \infty} p/n = 0$, since $\rho = 1 + O(p/n) \rightarrow 1$ and \eqref{eq:chisqconvgvar} is proved above, we know \eqref{thm13} holds.  
%we know \eqref{thm13} holds by $\rho = 1 + O(p/n) \rightarrow 1$ and \eqref{eq:chisqconvgvar}  above. 
Given  \eqref{thm13}, as $f\sim p^2(k-1)/2$, to prove \eqref{thm14}, it is equivalent to show $( 2f+4\rho  \mu_n  )/p\to 0$, which is  equivalent to $(2f+4\mu_n-4\Delta_n\mu_n)/p\to 0,$
where in this subsection, we redefine 
\begin{align*}
	\Delta_n=\frac{2p^2+9p+11}{6(p+3)(k-1)}\times D_{n,1}, \quad \quad  D_{n,1}=\sum_{i=1}^k\frac{1}{n_i}-\frac{1}{n}.
\end{align*}
Similarly to \eqref{eq:rnexpmulti}, through Taylor's expansion, we further have
\begin{align*}
	n(2p-2n+3)r_n^2=p\left\{ p -2n + 3+\frac{p^2}{3n}+\frac{p^3}{6n^2} + O\Big(\frac{p^4}{n^3} \Big) + o(1) \right\}. \notag
\end{align*} In addition, similarly to \eqref{eq:rnexpmulti2}, we have
\begin{align}
	n_i(2p-2n_i+3)r_{n_i'}^2
	= p\left\{ p -2n_i +3-2 +\frac{p^2}{3n_i} +\frac{p^3}{6n_i^2} + O\Big(\frac{p^4}{n_i^3} \Big) + o(1) \right\}. \label{eq:multicovtrho4} 
\end{align} 
%It follows that
%\begin{align*}
%	2f+4\mu_n=-\frac{}
%\end{align*}
%where in the second equation, we use $(n_i-1)^{-1}=n_i^{-1}+n_i^{-2}+O(n_i^{-3})$ and $(n_i-1)^{-a}=n_i^{-a}+O(n_i^{-3})$ for integers $a\geq 2$. 
%Similarly to the analysis of \eqref{eq:rnexpmulti}--\eqref{eq:multiidentmunexp}, through Taylor's expansion, we obtain
It follows that 
\begin{align}
	2f+4\mu_n=-\frac{p^3}{3}D_{n,1} - \frac{p^4}{6} D_{n,2} + O\Big(\frac{p^5}{n^3} \Big)+o(p), \label{eq:multiidentrho1}
\end{align}where $D_{n,2}=\sum_{i=1}^k n_i^{-2}-n^{-2}.$ 
Moreover, by \eqref{eq:multiidentmunexp} and $\Delta_n=O(p/n)=o(1)$, 
\begin{align}
	4\Delta_n\mu_n=\Delta_n \Big( -\frac{p^3}{3}D_{n,1} -2f \Big) + O\Big(\frac{p^5}{n^3} \Big) + o(p). \label{eq:multiidentrho2}
\end{align} 
Combining \eqref{eq:multiidentrho1} and \eqref{eq:multiidentrho2}, we obtain
\begin{align}
	2f+4\mu_n-4\Delta_n\mu_n= \frac{p^4}{18(k-1)}\{2D_{n,1}^2 - 3(k-1)D_{n,2}\}+O\Big(\frac{p^5}{n^3} \Big)+o(p), \label{eq:multiidentrho3}
\end{align} where we use $D_{n,1}=\Theta(n^{-1})$, $D_{n,2}=\Theta(n^{-2})$, $\Delta_n=pD_{n,1}/\{3(k-1)\}+o(p/n)$, and $2\Delta_n f=p^3D_{n,1}/3+o(p)$. 
Following the analysis of \eqref{eq:multicovrho4}, we know $\eqref{eq:multiidentrho3}=\Theta(p^4n^{-2}).$ Therefore, $(2f+\rho\mu_n)/p\to 0$ if and only if $p^3/n^2\to 0,$ which suggests that the chi-squared  approximation with the Bartlett correction holds if and only if  $p^3/n^2\to 0.$

\subsubsection{Proof of Theorem \ref{Thm2.1} (VII): Testing  Independence between Multiple Vectors} \label{sec:pfindp}

Similarly to the proof in  Section \ref{sec:pfonesam3}, by Theorem 2 in \cite{Jiang13} and \cite{Jiang15}, we know that under the conditions of Theorem \ref{Thm2.1} and $p\to \infty$, \eqref{eq:mmtnormal1} holds with 
\begin{align}
\mu_{n} =&~\frac{n}{2}\left[-\left(n-p-\frac{3}{2}\right) L_{n-1, p}+\sum_{j=1}^{k}\left\{\left(n-p_{j}-\frac{3}{2}\right) L_{n-1, p_{j}}\right\}\right]\label{eq:indpmu} \\
\sigma_{n}^{2} =&~\frac{1}{2}\biggr(-L_{n-1, p}+\sum_{j=1}^{k} L_{n-1, p_{j}}\biggr).	\label{eq:indpsigma}
\end{align}
%\begin{align}
%\mu_n =&~ \frac{n}{4} \Big\{-(2p-2n+3)r_n^2+\sum_{i=1}^k(2p_i-2n+3)r_{n,i}^2 \Big\}, \label{eq:indpmu} \\
%\sigma_n^2 =&~ \frac{1}{2}\Big( r_n^2 - \sum_{i=1}^{k} r_{n,i}^2\Big), \label{eq:indpsigma}
%\end{align} 
%where we redefine $r_{n}^2=-\log\{1-p/(n-1)\}$ and let $r_{n,i}^2=-\log\{1-p_i/(n-1)\}$ for $i=1,\ldots, k.$ 
%$1\leq i \leq k$. 
%Under the conditions of Theorem \ref{Thm2.1}, for testing question (VII), have the degree of freedom  $f=\Theta(p^2).$
%$n_i'=n_i-1$ and $r_x^2=-\log (1-p/x)$ for $x>p$. 
Following the analysis in Section \ref{sec:pfonesam3}, 
we next derive the equivalent conditions for  \eqref{eq:chisqconvgvar}--\eqref{eq:chisqconvgmean} and \eqref{thm13}--\eqref{thm14}, respectively,  with $\mu_n$ in \eqref{eq:indpmu} and $\sigma_n$ in \eqref{eq:indpsigma}.

%\paragraph{(VII.i) The chi-square approximation} \quad 

\smallskip
\noindent \textit{(VII.i) The chi-squared approximation.}

\noindent \textit{Case (VI.i.1) $\lim_{n\to \infty} p/n = 0$.}  Under this case, we show that \eqref{eq:chisqconvgvar} holds. Through Taylor's expansion, 
%we have
\begin{align*}
	2\sigma_n^2=&~\frac{p}{n-1}+\frac{p^2}{2(n-1)^2}-2\sum_{i=1}^{k}\left\{ \frac{p_i}{n-1} + \frac{p_i^2}{2(n-1)^2}\right\}+O\left(\frac{p^3}{n^3}\right) \notag \\
	=&~\frac{p^2 -\sum_{i=1}^{k}p_i^2}{2(n-1)^2} + O\left(\frac{p^3}{n^3}\right).
\end{align*}Recall that $2f=p^2-\sum_{i=1}^k p_i^2,$ and thus  \eqref{eq:chisqconvgvar}  holds. As $f=\Theta(p^2)$ undert the conditions of Theorem \ref{Thm2.1}, given \eqref{eq:chisqconvgvar}, we know \eqref{eq:chisqconvgmean} is equivalent to $(2f+4\mu_n)/p \to 0.$ 
Similarly to the analysis of \eqref{eq:rnexpmulti2}, through Taylor's expansion, we have
%\begin{align*}
%	n(2p-2n+3)r_n^2 =&~ p \Big\{ p+\frac{p^2}{3n} -2n + 1 + O\Big(\frac{p^3}{n^2} \Big) + o(1)\Big\}, \notag \\
%	n(2p_i-2n+3)r_{n,i}^2 = &~ p_i \Big\{ p_i+\frac{p_i^2}{3n} -2n + 1 + O\Big(\frac{p_i^3}{n^2} \Big) + o(1)\Big\}. \notag
%\end{align*}
%\begin{align*}
%-n(2p-2n+3)L_{n-1,p} =&~ p \Big\{ p+\frac{p^2}{3n} -2n + 1 + O\Big(\frac{p^3}{n^2} \Big) + o(1)\Big\}, \notag \\
%-n(2p_i-2n+3)L_{n-1,p_i} = &~ p_i \Big\{ p_i+\frac{p_i^2}{3n} -2n + 1 + O\Big(\frac{p_i^3}{n^2} \Big) + o(1)\Big\}. \notag
%\end{align*}
\begin{align*}
n(2n-2p-3)L_{n-1,p} =&~ p \Big\{ p+\frac{p^2}{3n} -2n + 1 + O\Big(\frac{p^3}{n^2} \Big) + o(1)\Big\}, \notag \\
n(2n-2p_i-3)L_{n-1,p_i} = &~ p_i \Big\{ p_i+\frac{p_i^2}{3n} -2n + 1 + O\Big(\frac{p_i^3}{n^2} \Big) + o(1)\Big\}. \notag
\end{align*}
It follows that
\begin{align}
&~ 2f+4\mu_n \notag \\
=&~ p^2-\sum_{i=1}^k p_i^2- p \Big( p+\frac{p^2}{3n} -2n + 1\Big)+\sum_{i=1}^k p_i \Big( p_i+\frac{p_i^2}{3n} -2n + 1 \Big)+ O\Big(\frac{p^4}{n^2} \Big) + o(p)\notag \\
=&~\frac{1}{3n} \biggr(\sum_{i=1}^k p_i^3 - p^3 \biggr) + O\Big(\frac{p^4}{n^2} \Big) + o(p). \notag
\end{align} Under the conditions of Theorem \ref{Thm2.1}, we have $\sum_{i=1}^k p_i^3 - p^3 = \Theta(p^3)$. Thus $(2f+4\mu_n)/p \to 0$ if and only if $p^2/n \to 0,$ which suggests that the chi-squared  approximation holds if and only if $p^2/n \to 0.$

%By Taylor's expansion, we obtain
%\begin{align*}
%	2f+4\mu_n=&~
%\end{align*}
\medskip

\noindent \textit{Case (VI.i.2) $\lim_{n\to \infty} p/n = C \in (0,1]$.} 
Under this case, we show that  \eqref{eq:chisqconvgvar} and \eqref{eq:chisqconvgmean}  do not hold at the same time. Particularly, as $f=\Theta(p^2)$ and $p/n\to C \in (0,1]$, \eqref{eq:chisqconvgvar} induces $(2n^2\sigma_n^2-f)/n^2\to 0$, and \eqref{eq:chisqconvgmean} induces $(f+2\mu_n)/n^2 \to 0$. Therefore,  \eqref{eq:chisqconvgvar} and \eqref{eq:chisqconvgmean}  together give $(2n^2\sigma_n^2+2\mu_n)/n^2 \to 0.$ 
Suppose $\lim_{n\to \infty} p_i/n=C_i \in (0,1)$. It follows that $\sum_{i=1}^k C_i=C$, and 
\begin{align}
	(2n^2\sigma_n^2+2\mu_n)/n^2 \to -g_3(C) + \sum_{i=1}^k g_3(C_i),\label{eq:gc3sumidn}
\end{align}
where $g_3(C)=(2-C)\log(1-C).$
%\begin{align*}
%	(2n^2\sigma_n^2+2\mu_n)/n^2 \to -(2-C) \log(1-C)  + \sum_{i=1}^k (2-C_i)\log(1-C_i). 
%\end{align*}
%\eqref{eq:chisqconvgvar} and \eqref{eq:chisqconvgmean} together induce $4(\mu_n+n^2\sigma_n^2)/(2f)\to 0$, 
%\begin{align*}
%	-(2-C) \log(1-C)  + \sum_{i=1}^k (2-C_i)\log(1-C_i)
%\end{align*}
%Suppose $\lim_{n\to \infty} p_i/n=C_i$,  and then $\sum_{i=1}^k C_i=C$. 
Note that $g_3(C)$ is a strictly concave function of $C \in (0,1]$ and $g(0)=0$. By the property  of strictly concave function, we have $$ \sum_{i=1}^k g_3(C_i)=\sum_{i=1}^k g_3(C\times C_i/C) >\sum_{i=1}^k g_3(C)\times C_i/C=g_3(C),$$  where we use $\sum_{i=1}^k C_i=C.$  Therefore when $C\in (0,1]$, the right hand side of $\eqref{eq:gc3sumidn} >0$, which  contradicts with $(2n^2\sigma_n^2+2\mu_n)/n^2 \to 0.$ We thus know that  \eqref{eq:chisqconvgvar} and \eqref{eq:chisqconvgmean}  do not hold simultaneously, which suggests that the chi-squared approximation fails.

%which gives $g_3(C)\to 0$ as $n\to \infty$, where we define
%\begin{align*}
%	g_3(C)= -(2-C) \log(1-C)  + \sum_{i=1}^k (2-C_i)\log(1-C_i),
%\end{align*} and we assume $n_i/n\to \delta_i\in (0,1)$ for $1\leq i\leq k.$  
%%Note that 
%As $p/n=(p/n_i)\times (n_i/n) < n_i/n$, we have $0<C\leq \delta_i<1$ for $1\leq i\leq k$. 
%Suppose $\lim_{n\to \infty} p_i/n=C_i$,  and then $\sum_{i=1}^k C_i=C$. 
%%Under this case, we note that $\sigma_n^2$ is $\Theta(1)$. 
%Note that \eqref{indevar} and \eqref{indemean} imply $D_1:=(2f-\sigma_n^2n^2)/n^2 \to 0$ and $D_2:=2(f+\mu_n n)/n^2\to 0$, respectively. 
%If \eqref{indevar} and \eqref{indemean} hold simultaneously, we have $(D_2-D_1)/2=\mu_n/n + \sigma_n^2/2 \to 0$. In particular, under this Case 2, we have
%\begin{align}
%	\mu_n/n + \sigma_n^2/2 \to  -(2-C) \log(1-C)  + \sum_{i=1}^k (2-C_i)\log(1-C_i).\label{eq:gc1kindep}
%\end{align} 
%Note that $g(C):=(2-C)\log(1-C)$ is a strictly concave function of $C \in (0,1]$. By the property  of strictly convex function, we have $ \sum_{i=1}^k g(C_i)=\sum_{i=1}^k g(C\times C_i/C) >\sum_{i=1}^k g(C)C_i/C=g(C)$, where the last equation uses $\sum_{i=1}^k C_i=C.$  Therefore we know when $C\in (0,1]$, $\eqref{eq:gc1kindep} >0$, which  contradicts with $\mu_n/n + \sigma_n^2/2 \to 0.$ We thus obtain that  \eqref{indevar} and \eqref{indemean}  do not hold simultaneously.  

%\paragraph{(VII.ii) The chi-square approximation with the Bartlett correction} \quad 

\smallskip
\noindent \textit{(VII.ii) The chi-squared approximation with the Bartlett correction.} \ 
When $\lim_{n\to \infty} p/n = 0,$ since $\rho = 1 + O(p/n) \rightarrow 1$ and \eqref{eq:chisqconvgvar} is proved above, we know \eqref{thm13} holds.  
Given  \eqref{thm13}, as $f= \Theta(p^2)$, to prove \eqref{thm14}, it is equivalent to show $( 2f+4\rho  \mu_n  )/p\to 0$, which is  equivalent to $(2f+4\mu_n-4\Delta_n\mu_n)/p\to 0,$ 
where in this subsection, we redefine
\begin{align*}
	\Delta_n=\frac{2\times D_{p,3}+9\times D_{p,2}}{6n\times D_{p,2}}, \quad D_{p,3}=p^3-\sum_{i=1}^{k}p_i^3, \quad D_{p,2}=p^2-\sum_{i=1}^{k}p_i^2.
\end{align*}
Similarly to \eqref{eq:multicovtrho4}, through Taylor's expansion, we further obtain 
%\begin{align*}
%	n(2p-2n+3)r_n^2 =&~ p \Big\{ p-2n + 1+\frac{p^2}{3n} +\frac{p^3}{6n^2}  + O\Big(\frac{p^4}{n^3} \Big) + o(1)\Big\}, \notag \\
%	n(2p_i-2n+3)r_{n,i}^2 = &~ p_i \Big\{ p_i -2n + 1+\frac{p_i^2}{3n}+\frac{p_i^3}{6n^2} + O\Big(\frac{p_i^4}{n^3} \Big) + o(1)\Big\}. \notag
%\end{align*}
\begin{align*}
	n(2n-2p-3)L_{n-1,p} =&~ p \Big\{ p-2n + 1+\frac{p^2}{3n} +\frac{p^3}{6n^2}  + O\Big(\frac{p^4}{n^3} \Big) + o(1)\Big\}, \notag \\
	n(2n-2p_i-3)L_{n-1,p_i}= &~ p_i \Big\{ p_i -2n + 1+\frac{p_i^2}{3n}+\frac{p_i^3}{6n^2} + O\Big(\frac{p_i^4}{n^3} \Big) + o(1)\Big\}. \notag
\end{align*}
It follows that
\begin{align}
&~ 2f+4\mu_n = -\frac{1}{3n}D_{p,3}-\frac{1}{6n^2}D_{p,4}+ O\Big(\frac{p^5}{n^3} \Big) + o(p), \label{eq:munmultind}
%\\
%=&~ p^2-\sum_{i=1}^k p_i^2- p \Big\{ p+\frac{p^2}{3n} -2n + 1\Big\}+\sum_{i=1}^k p_i \Big\{ p_i+\frac{p_i^2}{3n} -2n + 1 \Big\}+ O\Big(\frac{p^4}{n^2} \Big) + o(p)\notag \\
%=&~\frac{1}{3n} \Big(\sum_{i=1}^k p_i^3 - p^3 \Big) + O\Big(\frac{p^4}{n^2} \Big) + o(p). 
\end{align} 
where $D_{p,4}=p^4-\sum_{i=1}^kp_i^4.$ Moreover, as $\Delta_n=\Theta(p/n)$, by \eqref{eq:munmultind} and $2f=D_{p,2}$, we have
\begin{align*}
	4\Delta_n\mu_n=\Delta_n \Big(-\frac{1}{3n}D_{p,3}-D_{p,2}\Big) + O\Big(\frac{p^5}{n^3}\Big) + o(p). 
\end{align*}
As $\Delta_n=D_{p,3}/(3nD_{p,2})+O(n^{-1})$, we calculate that
%through calculations, we have 
\begin{align}
&~ 2f+4\mu_n-4\Delta_n\mu_n\label{eq:mutiequarho1} \\
=&~ -\frac{1}{3n}D_{p,3}-\frac{1}{6n^2}D_{p,4}+ \frac{D_{p,3}}{3nD_{p,2}} \Big(\frac{1}{3n}D_{p,3}+D_{p,2}\Big) + O\Big(\frac{p^5}{n^3}\Big) + o(p)\notag \\
=&~ -\frac{1}{18n^2D_{p,2}}( 3D_{p,4}D_{p,2}-2D_{p,3}^2)+ O\Big(\frac{p^5}{n^3}\Big) + o(p). \notag 
%=&~ -\frac{1}{6n^2}D_{p,4}+\frac{D_{p,3}^2}{9n^2D_{p,2}} + O\Big(\frac{p^5}{n^3}\Big) + o(p). 
\end{align}

We next prove $\eqref{eq:mutiequarho1}=\Theta(p^4n^{-2})$ by showing $ 3D_{p,4}D_{p,2}-2D_{p,3}^2=\Theta(p^6).$  Specifically, by the definitions of $D_{p,2}, D_{p,3}$, and $D_{p,4}$, we write 
\begin{align}
&~ 3D_{p,4}D_{p,2}-2D_{p,3}^2 \label{eq:psumpower} \\
%=&~ p^6 - 3p^4\sum_{i=1}^k p_i^2 +4 p^3 \sum_{i=1}^k p_i^3 - 3p^2\sum_{i=1}^k p_i^4 + 3\sum_{i=1}^k p_i^2 \sum_{i=1}^k p_i^4 - 2\Big(\sum_{i=1}^k p_i^3\Big)^2 \notag \\
=&~p^4\Big( p^2 - \sum_{i=1}^k p_i^2 \Big)+ 2p^3 \Big(- p\sum_{i=1}^k p_i^2+ \sum_{i=1}^k p_i^3 \Big)+ 2p^2\Big( p \sum_{i=1}^k p_i^3 - \sum_{i=1}^k p_i^4\Big) \notag \\
&~ + \Big(-p^2+ \sum_{i=1}^k p_i^2 \Big)\sum_{i=1}^k p_i^4 + 2\left\{\sum_{i=1}^k p_i^2 \sum_{i=1}^k p_i^4- \Big(\sum_{i=1}^k p_i^3\Big)^2\right\}. \notag
\end{align}
Using $p=\sum_{i=1}^k p_i$, we obtain
\begin{align}
p\sum_{i=1}^kp_i^{\alpha}-\sum_{i=1}^kp_i^{\alpha+1}=\sum_{i\neq j}p_i p_j^{\alpha},\quad \quad 	p\sum_{i\neq j} p_ip_j-2\sum_{i\neq j}p_i^2p_j=\sum_{i\neq j\neq l}p_ip_jp_l, \label{eq:decopterm3}
\end{align}
where integer  $1\leq \alpha \leq 5$, and  we use $\sum_{i\neq j}$ and $\sum_{i\neq j\neq l}$ to denote the summation  $\sum_{1\leq i\neq j \leq k}$ and $\sum_{1\leq i\neq j\neq l \leq k}$ for simplicity. By \eqref{eq:decopterm3}, we calculate that
\begin{align*}
%	\eqref{eq:psumpower}=&~p^4\sum_{i\neq j} p_ip_j-2p^3\sum_{i\neq j}p_i^2p_j + 2p^2\sum_{i\neq j}p_i^3p_j-\sum_{i\neq j}p_i p_j\sum_{l=1}^k p_l^4+2\Big(\sum_{i\neq j}p_i^2p_j^4 -  \sum_{i\neq j} p_i^3p_j^3\Big)\notag \\
\eqref{eq:psumpower}=&~p^3\sum_{i\neq j\neq l}p_ip_jp_l+2p^2\sum_{i\neq j}p_i^3p_j-\sum_{i\neq j}p_i p_j\sum_{l=1}^k p_l^4 - 2\sum_{i\neq j} p_i^3p_j^3  +2\sum_{i\neq j}p_i^2p_j^4 \notag \\
>&~ 2p^2\sum_{i\neq j}p_i^3p_j-\sum_{i\neq j}p_i p_j\sum_{l=1}^k p_l^4 - 2\sum_{i\neq j} p_i^3p_j^3\notag \\
=&~2\Big(\sum_{i=1}^k p_i^2 + \sum_{i\neq j} p_i p_j \Big)\sum_{i\neq j}p_i^3p_j- 2\sum_{i\neq j}p_ip_j^5-\sum_{i\neq j\neq l} p_ip_jp_l^4- 2 \sum_{i\neq j} p_i^3p_j^3> 0.
\end{align*}
Therefore $\eqref{eq:psumpower}=\Theta(p^6)$ and then $\eqref{eq:mutiequarho1}=\Theta(p^4n^{-2})$. 
%Then we know $\eqref{eq:psumpower}=\Theta(p^6)$ as
%\begin{align*}
%&~2p^2\sum_{i\neq j}p_i^3p_j-\sum_{i\neq j}p_ip_j\sum_{l=1}^kp_l^4- 2 \sum_{i\neq j} p_i^3p_j^3\notag \\
%=&~2\Big(\sum_{i=1}^k p_i^2 + \sum_{i\neq j} p_i p_j \Big)\sum_{i\neq j}p_i^3p_j- 2\sum_{i\neq j}p_ip_j^5-\sum_{i\neq j\neq l} p_ip_jp_l^4- 2 \sum_{i\neq j} p_i^3p_j^3> 0.
%\end{align*}
%In summary, we obtain $\eqref{eq:mutiequarho1}=\Theta(p^4n^{-2})$. 
Thus  $(2f+4\rho\mu_n)/p \to 0$ if and only if $p^3/n^2\to 0$, which suggests that the chi-squared  approximation with the Bartlett correction holds if and only if $p^3/n^2\to 0$.

\subsection{Proofs of Propositions \ref{prop:1} \& \ref{prop:2}}\label{sec:proofprop}

This section proves Propositions \ref{prop:1} and \ref{prop:2} following similar arguments to that in Sections \ref{sec:pfonesam3} and \ref{sec:pfthm1and4}. 
In particular, we consider without loss of generality that $p\to \infty$ and $p/n$ has a limit. 

\subsubsection{Proof of Proposition \ref{prop:1}}\label{sec:pfprop1}
%\noindent \textit{Proof of Proposition \ref{prop:1}.} 
%Note that testing the equality of several mean vectors can be viewed as testing the coefficient matrix in multivariate linear regression; see, Section 10.7 in \citet{Muirhead2009}. 
Following the analysis in Section \ref{sec:pfmultimean}, 
%by Theorem 3 in \cite{He2018}, 
we know that when $n,p\to \infty$, $n-k\to \infty$,  and $n-p\to \infty$, \eqref{eq:mmtnormal1} holds with $\mu_n$ in \eqref{eq:multisammeanmu} and $\sigma_n^2$ \eqref{eq:multisammeansigma}.
%\begin{align}
%\mu_n=&~\frac{n}{2}\Biggr\{(n-p-k-{1}/{2})\log\frac{(n-1-p)(n-k)}{(n-p-k)(n-1)}\label{eq:multisammeanmu} \\
%&~\quad \quad +(k-1)\log\frac{(n-1-p)}{(n-1)}+{p}\log\frac{(n-k)}{(n-1)}\Biggr\}, \notag \\
%\sigma_n^2=&~\frac{1}{2}\left\{ \log\left(1-\frac{p}{n-1} \right) - \log\left( 1-\frac{p}{n-k}\right)\right\}. \label{eq:multisammeansigma}
%\end{align}
%Following the analysis in Section \ref{sec:pfonesam3}, 
%Similarly to Section \ref{sec:pfonesam3}, 
%we know that 
Moreover, to derive the necessary and sufficient conditions for the chi-squared approximations without and with the Bartlett correction,  it is equivalent to examine  \eqref{eq:chisqconvgvar}--\eqref{eq:chisqconvgmean} and \eqref{thm13}--\eqref{thm14}, respectively, with $\mu_n$ in \eqref{eq:multisammeanmu} and $\sigma_n$ in \eqref{eq:multisammeansigma}.

\smallskip

%\paragraph{(IV.i) The chi-square approximation} 
\noindent \textit{(i) The chi-squared approximation.\ }\  
 (i.1) When $p/n\to 0$ and $k/n\to 0,$ we apply Theorem 1 in \cite{He2018}, and know that  \eqref{eq:chisqconvgvar}--\eqref{eq:chisqconvgmean} hold if and only if  $\sqrt{pk}(p+k)/n \to 0.$ 
 (i.2)  When $p/n\to C\in (0,1]$ and $k/n\to 0$, 
we have $f\sim C(k-1)n$ and $2\sigma_n^2\sim C(k-1)/\{n(1-C)\}$. It follows that $\sqrt{2f}/(2n\sigma_n)\sim \sqrt{1-C} <1$. Thus \eqref{eq:chisqconvgvar} fails, which suggests that the chi-squared approximation fails. 
%the proof of part (IV.i) in Section  \ref{sec:pfmultimean} applies similarly, and we know the chi-square approximation fails. 
(i.3) When $p/n\to 0$ and $k/n\to C\in (0,1]$,  by applying the symmetric substitution technique in Section 10.4 of \cite{Muirhead2009}, we can switch $k$ and $p$ and analyze similarly as in the case  (i.2) above. Therefore we know the  chi-squared approximation also fails here.    
(i.4) When $p/n\to C_1 \in (0,1]$ and $k/n\to C_2\in (0,1]$, we know $0<C_1+C_2\leq 1$ as $p+k<n$. By the constraint, it then suffices to consider $C_1,C_2\in (0,1).$ Note that $2\sigma_n^2 \sim \log\{(1-C_1)(1-C_2)\}-\log(1-C_1-C_2)$ and $2f/n^2\sim 2C_1C_2$.
% and then 
Thus \eqref{eq:chisqconvgvar} 
%$\sqrt{2f}/(2n\sigma_n)=\sqrt{2(k-1)p}/(2n\sigma_n) < 1$ 
induces $g_4(C_1,C_2)=0$ where $g_4(C_1,C_2)=C_1C_2-\log\{(1-C_1)(1-C_2)\}+\log(1-C_1-C_2)$. If $C_1+C_2=1$, $g_4(C_1,C_2)\to -\infty$. We next consider $0<C_1+C_2<1$. By calculations, we have $g_4(0,C_2)=0$,  and 
%\begin{align*}
%	\frac{\mathrm{d}}{\mathrm{d} C_1} g_4(C_1,C_2)|_{C_1=0}= - C_2^2/(1-C_2)<0
%\end{align*} 
%when $C_2 \in (0,1).$ 
\begin{align*}
	\frac{\mathrm{d}}{\mathrm{d} C_1} g_4(C_1,C_2)= \frac{C_2\{(C_1-1)(C_1+C_2)-C_1\}}{(1-C_1)(1-C_1-C_2)}<0,
\end{align*} 
where we use $C_1,C_2\in (0,1)$ and $0<C_1+C_2<1.$
Similarly to the previous  analyses, we know that $g_4(C_1,C_2)$ is monotonically decreasing for $C_1\in (0,1)$ and thus $g_4(C_1,C_2)<0$, as $C_1\in (0,1)$ and $g_4(0,C_2)=0.$ 
% When $p/n\to C\in (0,1]$ and $n-p \to \infty,$ we have $\sigma_n^2\sim C(k-1)/\{2n(1-C)\}$, and then  
%$\sqrt{2f}/(2n\sigma_n)=\sqrt{2(k-1)p}/(2n\sigma_n)\to \sqrt{1-C}< 1$. 
Therefore \eqref{eq:chisqconvgvar} fails, which suggests that the classical chi-squared approximation fails. 

\smallskip
\noindent \textit{(ii) The chi-squared approximation with the Bartlett correction.\ }\   
%\paragraph{(IV.ii) The chi-square approximation with Bartlett correction} 
(ii.1) When $p/n\to 0$ and $k/n\to 0,$ we apply Theorem 2 in \cite{He2018}, and know that \eqref{thm13}--\eqref{thm14} hold if and only if  $\sqrt{pk}(p^2+k^2)/n^2 \to 0$. 
(ii.2) When $p/n\to C\in (0,1]$ and $k/n\to 0$, we have $\rho \sim 1-C/2$, and the proof of part (IV.ii) in Section  \ref{sec:pfmultimean} can be applied here similarly.
%the proof of part (IV.ii) in Section  \ref{sec:pfmultimean} applies similarly, and we know 
Thus the chi-squared approximation fails. 
(ii.3) When $p/n\to 0$ and $k/n\to C\in (0,1]$, we know the  chi-squared approximation also fails by switching $k$ and $p$ symmetrically as in the case (i.3) above. 
%applying the symmetric substitution technique 
%in Section 10.4 of \cite{Muirhead2009}.  
%When $k/n\to 0$ and $p/n\to C\in (0,1]$, we have $\rho \sim 1-C/2$, and then $\sqrt{2f}/(2n\rho\sigma_n)=(1-C/2)^{-1}\sqrt{2p}/(2n\sigma_n)\to (1-C/2)^{-1}\sqrt{1-C}< 1$. Therefore \eqref{thm13} fails, which suggests that the classical chi-square approximation with Bartlett correction fails.
(ii.4) When $p/n\to C_1 \in (0,1]$ and $k/n\to C_2\in (0,1]$, we know $0<C_1+C_2\leq 1$ as $p+k<n$. Similarly to the case (i.4) above, 
%as above, 
we consider $C_1,C_2\in (0,1)$ and $C_1+C_2<1$. Here $\rho \sim 1-(C_1+C_2)/2$ and  then \eqref{thm13}  
induces $g_5(C_1,C_2)=0$,  where $g_5(C_1,C_2)=2C_1C_2-(2-C_1-C_2)[\log\{(1-C_1)(1-C_2)\}-\log(1-C_1-C_2)]$. By calculations, we have $g_5(0,C_2)=0$,  and 
\begin{align*}
&\frac{\mathrm{d}}{\mathrm{d} C_1} g_5(C_1,C_2)|_{C_1=0}= - C_2/(1-C_2)<0, \notag \\
&\frac{\mathrm{d}^2}{\mathrm{d}^2 C_1} g_5(C_1,C_2)=-\frac{C_2\{ (C_1+C_2)(C_2-2)+2\}}{(1-C_1)^2(1-C_1-C_2)^2} < 0,
\end{align*} where we use $(C_1+C_2)(C_2-2)+2>0$ as $0<C_1+C_2<1$ and $-2<C_2-2<-1.$ 
Similarly to the analysis above, we know that $g_5(C_1,C_2)<0$ and thus \eqref{thm13} fails, which suggests that the chi-squared approximation with the Bartlett correction fails.

%\bigskip
%\noindent \textit{Proof of Proposition \ref{prop:2}.}
%\smallskip

\subsubsection{Proof of Proposition \ref{prop:2}}\label{sec:pfprop2}

%\paragraph{(IV.i) The chi-square approximation} 
\noindent \textit{(i) The chi-squared approximation.\ } 
(i.1) When $p_1/n\to 0$ and $p_2/n\to 0,$ we apply Theorem 1 in \cite{He2018}, and know that  \eqref{eq:chisqconvgvar}--\eqref{eq:chisqconvgmean} hold if and only if  $\sqrt{p_1p_2}(p_1+p_2)/n \to 0.$ 
(i.2) When $p_1/n\to C \in (0,1]$ and $p_2/n\to 0$, we have $2f \sim Cnp_2$ and $2\sigma_n^2\sim Cp_2/\{2n(1-C)\}$. Then   $\sqrt{2f}/(2n\sigma_n)\sim \sqrt{1-C}<1$ suggesting the failure of \eqref{eq:chisqconvgvar} and thus the chi-squared  approximation fails.  
% for \eqref{eq:chisqconvgvar},
(i.3) When $p_1/n\to 0$ and $p_2/n\to C \in (0,1]$, the chi-squared  approximation  also fails by  the symmetric substitution technique in Section \ref{sec:pfprop1}.
%in Section 10.4 of \cite{Muirhead2009}. 
(i.4) When $p_1/n\to C_1 \in (0,1]$ and $p_2/n\to C_2\in (0,1]$,  
we have $2\sigma_n^2 \sim \log\{(1-C_1)(1-C_2)\}-\log(1-C_1-C_2)$ and $2f/n^2\sim C_1C_2$. It follows that the analysis in case (i.4) of Section \ref{sec:pfprop1} can be applied similarly, and we obtain the same  conclusion, that is, \eqref{eq:chisqconvgvar} fails and then the chi-squared  approximation fails. 

% which are same as those in part (i) of Section \ref{sec:pfprop1}. Therefore we know that the conclusion holds similarly, that is, the chi-square approximation fails.   

\smallskip
\noindent \textit{(ii) The chi-squared approximation with the Bartlett correction.\ }  
(ii.1) When $p_1/n\to 0$ and $p_2/n\to 0,$ we apply Theorem 2 in \cite{He2018}, and know that \eqref{thm13}--\eqref{thm14} hold if and only if  $\sqrt{p_1p_2}(p_1^2+p_2^2)/n^2 \to 0$. 
(ii.2) When $p_1/n\to C \in (0,1]$ and $p_2/n\to 0$, we have $\rho \sim 1-C/2$,   and then $\sqrt{2f}/(2n\rho\sigma_n)=(1-C/2)^{-1}\sqrt{2p}/(2n\sigma_n)\to (1-C/2)^{-1}\sqrt{1-C}< 1$. Therefore \eqref{thm13} fails, which suggests that the classical chi-squared approximation with the Bartlett correction fails.
(ii.3) When $p_1/n\to 0$ and $p_2/n\to C \in (0,1]$, similar conclusion holds by  the symmetric substitution technique as above. 
(ii.4) When $p_1/n\to C_1 \in (0,1]$ and $p_2/n\to C_2\in (0,1]$, we have $\rho \sim 1-(C_1+C_2)/2$. It follows that the analysis in case (ii.4) of Section \ref{sec:pfprop1} can be applied similarly. Then we obtain the same  conclusion, that is, \eqref{thm13} fails and the chi-squared approximation with the Bartlett correction fails.

%\section*{Appendix 2: Supplementary Simulations}
%\section{Simulations}\label{sec:addsimu}

\subsection{Proofs of Theorems \ref{thm:onesamchisq}, \ref{thm:multisamchisq} \& \ref{thm:chisqindp}} \label{sec:pfthm2}

In this section, we prove the results for other testing problems in Theorems \ref{thm:onesamchisq}, \ref{thm:multisamchisq} \& \ref{thm:chisqindp} following similar analysis to that in Section \ref{sec:pfthmchisqiii}. 
Particularly, 
for each test, 
we consider the  characteristic function of $-2\eta \log \Lambda_n$ 
when $\eta=1$ and $\rho$; here $\rho$ denotes the corresponding Bartlett correction factor of each test.

By Eq. (20)--(23) in Section 8.2.4 of \cite{Muirhead2009}, 
we know that for the testing problems (I)--(II) and (IV)--(VII), 
the characteristic functions of the likelihood ratio test statistics take the following general form:
\begin{align}\label{eq:logexpgeneralform}
\log \mathrm{E}\{\exp( -2it \eta \log \Lambda_n )\}=\varphi(t)-\varphi(0), 
\end{align}
where
\begin{align*}
\varphi(t)=	&~2it\eta \left( \sum_{k=1}^{K_1}\xi_{1,k}\log \xi_{1,k} - \sum_{j=1}^{K_2} \xi_{2,j} \log \xi_{2,j}\right)\notag \\
&~ + \sum_{k=1}^{K_1}\log \Gamma\big\{ \eta \xi_{1,k}(1-2it) +\tau_{1,k}+ \upsilon_{1,k}  \big\} - \sum_{j=1}^{K_2}\log \Gamma\big\{ \eta \xi_{2,j}(1-2it) + \tau_{2,j}+\upsilon_{2,j}\big\}, 
\end{align*}
 $i$ denotes the imaginary unit,  $\tau_{1,k}=(1-\eta)\xi_{1,k}$, and $\tau_{2,j}=(1-\eta)\xi_{2,j}$.
We next consider $\eta=1$ and $\rho$ for the chi-squared approximation without and with the Bartlett correction, respectively.  
The values of $\rho$, $K_1$, $K_2$, $\xi_{1,k}$, $\xi_{2,j}$, $\upsilon_{1,k}$, and $\upsilon_{2,j}$ depend on the testing problem,
and thus take different values in the following subsections.  
Moreover, by \cite{Muirhead2009}, 
 in each problem, we have
%\begin{align}
%\sum_{k=1}^{K_1}\xi_{1,k}=\sum_{j=1}^{K_2}\xi_{2,k}, \label{eq:xiequal}	
%\end{align}
$\sum_{k=1}^{K_1}\xi_{1,k}=\sum_{j=1}^{K_2}\xi_{2,k}$, 
the degrees of freedom $f$ is
\begin{align}
	f=-2\left\{ \sum_{k=1}^{K_1}\upsilon_{1,k} - \sum_{j=1}^{K_2} \upsilon_{2,j} - \frac{1}{2}(K_1-K_2) \right\}, \label{eq:generalfform}
\end{align}
and the Bartlett correction $\rho$ takes the value 
%such that $\varsigma_1=0$, that is, 
\begin{align} \label{eq:generalrhoform}
	\rho=1-\frac{1}{f}\left\{ \sum_{k=1}^{K_1} \frac{\upsilon_{1,k}^2-\upsilon_{1,k}+\frac{1}{6} }{\xi_{1,k}} - \sum_{j=1}^{K_2} \frac{\upsilon_{2,j}^2-\upsilon_{2,j} + \frac{1}{6} }{ \xi_{2,j}} \right\}.
\end{align} 
\smallskip
%We next provide Lemma \ref{lm:generalcharexpan} below 
In the following proofs, we use Lemma \ref{lm:generalcharexpan} below to obtain an asymptotic expansion of each characteristic function.  
\begin{lemma}\label{lm:generalcharexpan}
For a finite integer $L$, when $\eta=1$ or $\rho$, 
 $p/n\to 0$, and $R_{n,L}$ (in \eqref{eq:rnlorder} below) converges to $0$, 
%\begin{align*}
% \mathrm{E}\{\exp( -2it \eta \log \Lambda_n )\}=(1-2it)^{-\frac{f}{2}}\exp\left\{ \sum_{l=1}^{L-1}\varsigma_l(1-2it)^{-l} +O( p^{L+2}/n^{L}) \right\},
%\end{align*}	
\begin{align*}
\log  \mathrm{E}\{\exp( -2it \eta \log \Lambda_n )\}=&~ -\frac{f}{2}\log (1-2it) +\sum_{l=1}^{L-1}\varsigma_l \big\{(1-2it)^{-l}-1\big\}+R_{n,L},	
\end{align*}
%\begin{align*}
%\log  \mathrm{E}\{\exp( -2it \eta \log \Lambda_n )\}=&~ -\frac{f}{2}\log (1-2it) +\sum_{l=1}^{L-1}\varsigma_l(1-2it)^{-l} \notag \\
%&~ +O\Biggr( \sum_{k=1}^{K_1}\frac{|\tau_{1,k}+\upsilon_{1,k}|^{L+1}}{|\eta \xi_{1,k}|^{L}}+ \sum_{j=1}^{K_2} \frac{|\tau_{2,j}+\upsilon_{2,j}|^{L+1}}{|\eta \xi_{1,k}|^{L}}\Biggr),	
%\end{align*}
\end{lemma}
where
\begin{align*}
\varsigma_l=\frac{(-1)^{l+1}}{l(l+1)}\left\{ \sum_{k=1}^{K_1}\frac{B_{l+1}(\tau_{1,k} + \upsilon_{1,k} )}{(\eta \times \xi_{1,k})^l} - \sum_{j=1}^{K_2}   \frac{ B_{l+1}(\tau_{2,j}+ \upsilon_{2,j})}{(\eta \times \xi_{2,j} )^l }\right\},
\end{align*}
 $B_{l+1}(\cdot)$ denotes the $(l+1)$-th Bernoulli polynomial;  see, e.g., Eq. (25) in Section 8.2.4 of \cite{Muirhead2009}, 
 and $R_{n,L}$ denotes the remainder which is of the order of
 \begin{align}\label{eq:rnlorder}
 R_{n,L}=O\Biggr( \sum_{k=1}^{K_1}\frac{|\tau_{1,k}+\upsilon_{1,k}|^{L+1}}{|\eta \xi_{1,k}|^{L}}+ \sum_{j=1}^{K_2} \frac{|\tau_{2,j}+\upsilon_{2,j}|^{L+1}}{|\eta \xi_{2,j}|^{L}}\Biggr).	
 \end{align}

\medskip
\begin{proof}
Please see Section \ref{sec:generalcharexpan} on Page \pageref{sec:generalcharexpan}. 
\end{proof}

\smallskip

\noindent We next examine each testing problem based on  Lemma \ref{lm:generalcharexpan}.

%By \eqref{eq:logexpgeneralform},

%and the values of $\xi_{1,k}$, $\xi_{2,j}$, $\upsilon_{1,k}$, and $\upsilon_{2,j}$ depend on the testing problem. 
%\begin{align*}
%\mathrm{E}\{\exp( -2it \eta \log \Lambda_n )\}=C\left(\frac{ \prod_{j=1}^{K_2} \xi_{2,j}^{\xi_{2,j}} }{\prod_{k=1}^{K_1}\xi_{1,k}^{\xi_{1,k}} } \right)^{-2it \eta}\frac{\prod_{k=1}^{K_1}\Gamma\{\xi_{1,k}(1-2it\eta) + \upsilon_{1,k}  \} }{\prod_{j=1}^{K_2}\Gamma\{\xi_{2,j}(1-2it\eta) + \upsilon_{2,j}\} },
%\end{align*}
%where $i$ is the imaginary unit and $C$ is a normalizing constant. 

\subsubsection{Proof of Theorem \ref{thm:onesamchisq} (I): Testing One-Sample Mean Vector}\label{sec:pfonesamchisq1}
Recall that in Section \ref{sec:pfonesam1}, we mention that testing one-sample mean vector can be viewed as testing coefficient vector of a multivariate linear regression model. 
By Section 10.5 in \cite{Muirhead2009}, we know that in this problem, 
$K_1=1$, $K_2=1$, $\xi_{1,1}=n/2$, $\xi_{2,1}=n/2$, 
$\upsilon_{1,1}=-p/2$, $\upsilon_{2,1}=0$,
$f=p$ and $\rho =1 - (p/2+1)/n$. 
We next discuss the chi-squared approximation without and with the Bartlett correction, respectively. 

%By \eqref{eq:generalfform} and \eqref{eq:generalrhoform}, 
%we have $f=p$ and $\rho =1 - (p/2+1)/n$. 

\medskip

\noindent \textit{(i) Chi-squared approximation.} \quad 
Consider $\rho=1$ and $p^3/n^2\to 0$. 
Then $\tau_{1,1}=\tau_{2,1}=0$,  
\begin{align}
	\varsigma_l=\frac{(-1)^{l+1}}{l(l+1)}\times \frac{1}{(n/2)^l}\left\{ B_{l+1}\left(-\frac{p}{2}\right)-B_{l+1}(0)\right\},\label{eq:varsigonesammean}
\end{align}
and for any finite integer $L$, $R_{n,L}=O( p^{L+1}n^{-L})$. 
Since $B_{l+1}(\cdot)$ is a polynomial of order $l+1$, then $\varsigma_l=O(p^{l+1}/n^l)$. 
By Lemma \ref{lm:generalcharexpan}, when $p^3/n^2\to 0$, $R_{n,3}=O(p^4n^{-3})\to 0$, and 
\begin{align*}
\mathrm{E}\{\exp( -2it \log \Lambda_n )\}=&~(1-2it)^{-\frac{f}{2}}\prod_{l=1}^2\exp\Big[\varsigma_l \big\{ (1-2it)^{-l}-1\big\} \Big] \big\{1+O(p^4n^{-3})\big\}\notag \\
=&~ (1-2it)^{-\frac{f}{2}}\big\{1+V_1(t)+V_2(t)+V_1(t)V_2(t)\big\}\big\{1+O(p^4n^{-3})\big\}, 
\end{align*} 
where $V_l(t)$ is defined as in \eqref{eq:vltdef} on Page \pageref{eq:vltdef}. 
Then similarly to the proof in Section \ref{sec:pfthmchisqiii}, 
by the inversion property of the characteristic function, 
we obtain 
\begin{align}
&~\Pr ( -2\log \Lambda_n \leq x )\label{eq:probexpani1} \\
=&~\Biggr\{\Pr (\chi_f^2 \leq x ) + \sum_{v=1}^{\infty} \frac{\varsigma_1^v}{v!} \sum_{w=0}^v \binom{v}{w}\Pr( \chi^2_{f+2w}\leq x ) (-1)^{v-w} \notag \\
	&~ + \sum_{v=1}^{\infty} \frac{\varsigma_2^v}{v!} \sum_{w=0}^v \binom{v}{w}\Pr( \chi^2_{f+4w}\leq x )(-1)^{v-w} \notag \\
	&~ + \sum_{\substack{v_1\geq 1; ~ 0\leq w_1\leq v_1\\ v_2\geq 1; ~ 0\leq w_2\leq v_2} }   \frac{\varsigma_1^{v_1}\varsigma_2^{v_2}}{v_1!v_2!}   \binom{v_1}{w_1} \binom{v_2}{w_2}\Pr(\chi^2_{2f+2w_1+4w_2}\leq x)(-1)^{v_1-w_1+v_2-w_2} \Biggr\}\biggr\{1+O\Big(\frac{p^4}{n^{3}}\Big)\biggr\}. \notag 
\end{align}
When $x=\chi_f^2(\alpha)$, 
by Propositions \ref{prop:infinitesumm} and \ref{prop:infinitesumm2}, and $\varsigma_l=O(p^{l+1}/n^l)$, 
we have
\begin{align*}
\Pr ( -2\log \Lambda_n \leq x )=\Pr(\chi^2_f\leq x) +\varsigma_1\big\{ \Pr(\chi^2_{f+2}\leq x) -\Pr(\chi^2_f\leq x) \big\} + o(p^{3/2}/n). 	
\end{align*}
%We calculate $\varsigma_1=(p^2+2p)/(4n)$. 
Particularly, by Lemma \ref{lm:order2diff}, 
\begin{align*}
\Pr(\chi_{f+2}^2\leq x) - \Pr(\chi_{f}^2\leq x)	=-  \frac{1}{\sqrt{f\pi }}\exp\left(-\frac{z_{\alpha}^2}{2}\right)\Big\{1+ O(f^{-1/2})\Big\},	
\end{align*}
and we compute $\varsigma_1=(p^2+2p)/(4n)$. 
In Theorem \ref{thm:onesamchisq}, we have $\vartheta_1(n,p)=\varsigma_1/\sqrt{f}$. 

%\begin{lemma}\label{lm:order2diff}
%When $x=\chi_f^2(\alpha)$ and $f\to \infty$,  for $h\in \{1,2,3\}$,	
%\begin{align*}
%\Pr(\chi_{f+2h}^2\leq x) - \Pr(\chi_{f}^2\leq x)	= &~  - \sum_{k=1}^h \left\{\Gamma\left(\frac{f}{2}+h-k+1\right) \right\}^{-1}\left(\frac{x}{2}\right)^{\frac{f}{2}+h-k}e^{-x/2}\notag \\
%=&~-  \frac{h}{\sqrt{f\pi }}\exp\left(-\frac{z_{\alpha}^2}{2}\right)\Big\{1+ O(f^{-1/2})\Big\}. 
%\end{align*}
%\end{lemma}
%\begin{align*}
%\varsigma_1=(p^2+2p)/(4n)	
%\end{align*}
%Note that by  Taylor's expansion,
%\begin{align*}
%\exp\big[\varsigma_l\big\{(1-2it)^{-l}-1\big\}\big]	= 1+\sum_{v=1}^{\infty}\frac{\varsigma_l^v}{v!}\big\{(1-2it)^{-l}-1\big\}^v = 1+V_l(t),
%\end{align*}
%where
%\begin{align}
%	V_l(t)= \sum_{v=1}^{\infty}\frac{\varsigma_l^v }{v!} \sum_{w=0}^v \binom{v}{w} (1-2it)^{-lw} (-1)^{v-w}. \label{eq:vltdef}
%\end{align}
%It follows that 
%\begin{align*}
%\mathrm{E}\{\exp( -2it \log \Lambda_n )\}=(1-2it)^{-\frac{f}{2}}\big\{1+V_1(t)+V_2(t)+V_1(t)V_2(t)\big\}\{1+R_{n,3}\}. 	
%\end{align*}
%When $p^3/n^2\to 0$, $R_{n,3}=O(p^4n^{-3})\to 0$, 
%By Lemma \ref{lm:generalcharexpan} and Taylor's expansion of $\exp[\varsigma_l\{(1-2it)^{-l}-1\}]$, 
%\begin{align*}
%\mathrm{E}\{\exp( -2it \log \Lambda_n )\}=(1-2it)^{-\frac{f}{2}}\big\{1+V_1(t)+V_2(t)+V_1(t)V_2(t)\big\}\{1+R_{n,3}\},
%\end{align*} 
%where for an integer $l$,
%\begin{align}
%	V_l(t)=\exp \big\{\varsigma_l(1-2it)^{-l} \big\}-1 = \sum_{v=1}^{\infty}\frac{\varsigma_l^v }{v!} \sum_{w=0}^v \binom{v}{w} (1-2it)^{-lw} (-1)^{v-w}. \label{eq:vl1def}
%\end{align}

\smallskip

%\medskip

\noindent \textit{(ii) Chi-squared approximation with the Bartlett correction.} \quad 
By choosing the Bartlett correction factor $\rho$ as in \eqref{eq:generalrhoform}, we have $\varsigma_1=0$; see, e.g., Section 8.2.4 in \cite{Muirhead2009}. 
Specifically, in this problem, $\rho =1-(p+2)/(2n)$, $\rho \xi_{1,1}=\rho \xi_{2,1}=n/2-(p+2)/4$,  $\tau_{1,1}=\tau_{2,1}=(p+2)/4$, $\upsilon_{1,1}=-p/2$, $\upsilon_{2,1}=0$, and then
\begin{align*}
	\varsigma_l=\frac{(-1)^{l+1}}{l(l+1)(\rho \times \xi_{1,1} )^l}\left\{ {B}_{l+1}\left(-\frac{p-2}{4}\right) - {B}_{l+1}\left( \frac{p+2}{4} \right)  \right\}.
\end{align*}
We calculate 
$	\varsigma_2={p(p^2-4)}\{48(\rho n)^2\}^{-1}$, $\varsigma_3=0, $
%\begin{align*}
%	\varsigma_2=\frac{p(p^2-4)}{48(\rho n)^2},\quad \quad \varsigma_3=0, 
%\end{align*}
%\begin{align*}
%\varsigma_2=\frac{1}{3(2n-p-2)^2}p(p^2/4-1),\quad \quad \varsigma_3=0, 
%\end{align*}
and $\varsigma_l=O(p^{l+1}n^{-l})$ for $l\geq 4$. 
Similarly to the proof in Section \ref{sec:pfthmchisqiii}, when $p^5/n^4\to 0$,  we have
\begin{align*}
\mathrm{E}\{\exp( -2it\rho \log \Lambda_n )\}=(1-2it)^{-\frac{f}{2}}\big\{1+V_2(t)+V_4(t)+V_2(t)V_4(t) \big\} \big\{1+ O(p^6/n^5)\big\}, 
\end{align*}
and thus
\begin{align}
&~\Pr ( -2\rho \log \Lambda_n \leq x )\label{eq:rhologmean2}  \\
=&~\Biggr\{\Pr (\chi_f^2 \leq x ) + \sum_{v=1}^{\infty} \frac{\varsigma_2^v}{v!} \sum_{w=0}^v \binom{v}{w}\Pr( \chi^2_{f+4w}\leq x ) (-1)^{v-w} \notag \\
&~ + \sum_{v=1}^{\infty} \frac{\varsigma_4^v}{v!} \sum_{w=0}^v \binom{v}{w}\Pr( \chi^2_{f+8w}\leq x )(-1)^{v-w} \notag \\
&~ + \sum_{\substack{v_2\geq 1; ~ 0\leq w_2\leq v_2\\ v_4\geq 1; ~ 0\leq w_4\leq v_4} }   \frac{\varsigma_2^{v_2}\varsigma_4^{v_4}}{v_2!v_4!}   \binom{v_2}{w_2} \binom{v_4}{w_4}\Pr(\chi^2_{2f+4w_2+8w_4}\leq x)(-1)^{v_2-w_2+v_4-w_4} \Biggr\}\biggr\{1+O\Big(\frac{p^6}{n^{5}}\Big)\biggr\}. \notag 	
\end{align}
Note that $\varsigma_2=\Theta(p^3n^{-2})$ and $\varsigma_4=\Theta(p^5n^{-4})$. 
By applying proposition \ref{prop:infinitesumm} with $h=2$, 
\begin{align*}
\sum_{v=1}^{\infty} \frac{\varsigma_2^v}{v!} \sum_{w=0}^v \binom{v}{w}\Pr( \chi^2_{f+4w}\leq x ) (-1)^{v-w} = \sum_{v=1}^{\infty}  \Big\{O \big(\varsigma_2 p^{-1/2}\big)\Big\}^v	= \Theta( p^{5/2}n^{-2}).
\end{align*}
By applying proposition \ref{prop:infinitesumm} with $h=4$, we have
\begin{align*}
\sum_{v=1}^{\infty} \frac{\varsigma_4^v}{v!} \sum_{w=0}^v \binom{v}{w}\Pr( \chi^2_{f+8w}\leq x )(-1)^{v-w} = 	\sum_{v=1}^{\infty}  \Big\{O \big(\varsigma_4 p^{-1/2}\big)\Big\}^v =O\big(p^{9/2}n^{-4}\big)=o(p^{5/2}n^{-2}),
\end{align*}
and 
%\begin{align*}
%&~ \sum_{\substack{v_2\geq 1; ~ 0\leq w_2\leq v_2\\ v_4\geq 1; ~ 0\leq w_4\leq v_4} }   \frac{\varsigma_2^{v_2}\varsigma_4^{v_4}}{v_2!v_4!}   \binom{v_2}{w_2} \binom{v_4}{w_4}\Pr(\chi^2_{2f+4w_2+8w_4}\leq x)(-1)^{v_2-w_2+v_4-w_4} \notag \\
%=&~ \sum_{\substack{v_2\geq 1;\\  0\leq w_2\leq v_2} }  \big\{O\big(\varsigma_2 p^{-1/2} \big)\big\}^{v_2} (v_2!)^{-1} \binom{v_2}{w_2} \sum_{\substack{ v_4\geq 1; \\ 0\leq w_4\leq v_4} } \frac{\varsigma_4^{v_4}}{v_4!} 	 \binom{v_4}{w_4}=o\big(p^{5/2} n^{-2}\big).  
%\end{align*}
\begin{align*}
&~ \sum_{\substack{v_2\geq 1; ~ 0\leq w_2\leq v_2\\ v_4\geq 1; ~ 0\leq w_4\leq v_4} }   \frac{\varsigma_2^{v_2}\varsigma_4^{v_4}}{v_2!v_4!}   \binom{v_2}{w_2} \binom{v_4}{w_4}\Pr(\chi^2_{2f+4w_2+8w_4}\leq x)(-1)^{v_2-w_2+v_4-w_4} \notag \\
=&~ \sum_{\substack{v_2\geq 1} }  \big\{O\big(\varsigma_2 p^{-1/2} \big)\big\}^{v_2} \sum_{\substack{ v_4\geq 1} } \frac{\{O(\varsigma_4)\}^{v_4}}{v_4!} =o\big(p^{5/2} n^{-2}\big).  
\end{align*} 
In summary, by \eqref{eq:rhologmean2}, 
\begin{align*}
\Pr ( -2\rho \log \Lambda_n \leq x )
=\Pr (\chi_f^2 \leq x ) + \varsigma_2\Big\{ \Pr(\chi^2_{f+4}\leq x) -\Pr(\chi^2_f\leq x)\Big\}+o\big(p^{5/2} n^{-2}\big).
\end{align*}
Particularly, by Lemma  \ref{lm:order2diff}, 
\begin{align*}
\Pr(\chi_{f+4}^2\leq x) - \Pr(\chi_{f}^2\leq x)	=~-  \frac{2}{\sqrt{f\pi }}\exp\left(-\frac{z_{\alpha}^2}{2}\right)\Big\{1+ O(f^{-1/2})\Big\}.	
\end{align*}
%Since $f=p$, 
In Theorem \ref{thm:onesamchisq} (I), $\vartheta_2(n,p)=2\varsigma_2/\sqrt{f}$.

%In Theorem \ref{thm:onesamchisq} (I), $\vartheta_2(n,p)=$
%By applying proposition \ref{prop:infinitesumm} with $h=4$, 
%By applying proposition \ref{prop:infinitesumm} with $h=2, 4$, and Proposition \ref{prop:infinitesumm2} with  $(h_1,h_2)=(2,4)$, 
%we have that when $p^5n^{-4}\to 0$, 

\subsubsection{Proof of Theorem \ref{thm:onesamchisq} (II): Testing One-Sample Covariance Matrix}

In this problem, by Section 8.3.3 in \cite{Muirhead2009}, we know $f=(p+2)(p-1)/2$, and 
\begin{itemize}
	\item $K_1=p$, ~$K_2=1$;
	\item $\xi_{1,k}=(n-1)/2$, ~~ $\upsilon_{1,k}=-(k-1)/2$ for $k=1,\ldots, K_1$; 
	\item  $\xi_{2,1}=p(n-1)/2$, \, $\upsilon_{2,1}=0$. 
\end{itemize}
%$K_1=p$, $K_2=1$; $\xi_{1,k}=(n-1)/2$ and $\upsilon_{1,k}=-(k-1)/2$ for $k=1,\ldots, K_1$; $\xi_{2,1}=p(n-1)/2$ and $\upsilon_{2,1}=0$; 
%$f=(p+2)(p-1)/2$. 

\medskip

\noindent \textit{(i) Chi-squared approximation.} \quad
Consider $\rho=1$ and $p^2/n\to 0$. 
Then $\tau_{1,k}=0$ for $k=1,\ldots, K_1$, $\tau_{2,1}=0$, and
\begin{align*}
	\varsigma_l=\frac{(-1)^{l+1}}{l(l+1)}\left\{ \sum_{k=1}^p \left(\frac{2}{n-1} \right)^l {B}_{l+1}\left( -\frac{k-1}{2} \right) - \frac{2}{p(n-1)} {B}_{l+1}(0) \right\}, 
\end{align*} 
which satisfies $\varsigma_l=O(p^{l+2}/n^l)$.  
By Lemma \ref{lm:generalcharexpan},
\begin{align*}
\mathrm{E}\{\exp( -2it\log \Lambda_n )\}=(1-2it)^{-\frac{f}{2}}\big\{1+V_1(t)+V_2(t)+V_1(t)V_2(t) \big\}\big\{1+ O(p^5/n^3)\big\},
\end{align*} 
where $V_l(t)$ is defined as in \eqref{eq:vltdef}.
Similarly to Section \ref{sec:pfthmchisqiii}, by the inversion property of the characteristic functions, and Propositions \ref{prop:infinitesumm} and \ref{prop:infinitesumm2}, 
we obtain \eqref{eq:chisqgoal1}.
We calculate 
\begin{align*}
	\varsigma_1=&~\frac{1}{2}\left[ \sum_{k=1}^p \frac{2}{n-1} \left\{\left(-\frac{k-1}{2} \right)^2-\left(-\frac{k-1}{2} \right) + \frac{1}{6}  \right\}-\frac{2}{p(n-1)}\times \frac{1}{6} \right]\notag \\
%	=&~\frac{p(2p^2+3p-1)}{24(n-1)}-\frac{1}{6p(n-1)}. \\
	=&~\frac{2p^3+3p^2-p-4/p}{24(n-1)}.
\end{align*}	
The conclusion then follows by Lemma \ref{lm:order2diff} and $\vartheta_1(n,p)=\varsigma_1/\sqrt{f}$. 

\medskip

\noindent \textit{(ii) Chi-squared approximation with the Bartlett correction.} \quad
In this problem, consider 
\begin{align*}
	\rho = 1-\frac{2p^2+p+2}{6p(n-1)},
\end{align*} and $p^3/n^2\to 0$. 
Then $\tau_{1,k}=(2p^2+p+2)/(12p)$ for $k=1,\ldots, p$, and $\tau_{2,1}=(2p^2+p+2)/12$. 
It follows that
\begin{align*}
	\varsigma_l=\frac{(-1)^{l+1}}{l(l+1)}\left\{\frac{\rho (n-1)}{2}\right\}^{-l}\Biggr\{\sum_{k=1}^p{B}_{l+1}\left( \frac{2p^2+p+2}{12p} -\frac{k-1}{2} \right) - p^{-l}{B}_{l+1}\left( \frac{2p^2+p+2}{12} \right)\Biggr\}. 
\end{align*} 
%\begin{align*}
%	\varsigma_l=&~ \frac{(-1)^{l+1}}{l(l+1)}\Biggr\{\sum_{k=1}^p\left(\frac{n}{2}-\frac{2p^2+p+2}{12p} \right)^{-l}{B}_{l+1}\left( \frac{2p^2+p+2}{12p} -\frac{k-1}{2} \right) \notag \\
%	&~\quad \quad \quad\quad - \left(\frac{pn}{2}-\frac{2p^2+p+2}{12} \right)^{-l}{B}_{l+1}\left( \frac{2p^2+p+2}{12} \right)	\Biggr\}. 
%\end{align*} 
In particular, 
we calculate
\begin{align*}
%	\varsigma_2=&~-\frac{1}{6}(\rho n/2)^{-2}\Biggr\{ \sum_{k=1}^p {B}_3\Big( \tau_{1,k} -\frac{k-1}{2} \Big) - p^{-2} {B}_3( \tau_{2,1}) \Biggr\} \notag \\
	\varsigma_2=\frac{(p-2)(p-1)(p+2)}{288p^2\rho^2(n-1)^2}(2p^3+6p^2+3p+2).  	
\end{align*}
Similarly to Section \ref{sec:pfthmchisqiii}, by the inversion property of the characteristic functions, and Propositions \ref{prop:infinitesumm} and \ref{prop:infinitesumm2}, 
we obtain \eqref{eq:chisqbartgoal1}. 
The conclusion then follows by Lemma \ref{lm:order2diff} and $\vartheta_2(n,p)=2\varsigma_2/\sqrt{f}$.

\subsubsection{Proof of Theorem \ref{thm:multisamchisq} (IV): Testing the Equality of Several Mean Vectors}\label{sec:multisamchisq}
Recall that in Section \ref{sec:pfmultimean}, we show that this testing problem can be viewed as testing the coefficient matrix in multivariate linear regression. 
Then by Eq. (3) in Section 10.5.3 in \cite{Muirhead2009}, we know that
in this problem, $f=(k-1)p$, and 
\begin{itemize}
\item $K_1=k-1,$\ $K_2=k-1$; 
\item $\xi_{1,j_1}=n/2$, \ $\upsilon_{1,j_1}=-(j_1+p)/2$, \quad $j_1=1,\ldots, k-1$; 
\item $\xi_{2,j_2}=n/2$, \ $\upsilon_{2,j_2}=-j_2/2$, \quad $j_2=1,\ldots, k-1$.
\end{itemize}

\medskip

\noindent \textit{(i) Chi-squared approximation.} \quad
It follows that
\begin{align*}
\varsigma_l=\frac{(-1)^{l+1}}{l(l+1)}\left( \frac{2}{n}\right)^l \left\{\sum_{j_1=1}^{k-1} {B}_{l+1}\left(-\frac{j_1+p}{2} \right)  - \sum_{j_2=1}^{k-1} {B}_{l+1}\left( -\frac{j_2}{2} \right)	\right\},
\end{align*} 
which is $O(p^{l+1}n^{-l})$ when $k$ is finite. 
In particular, we calculate
\begin{align*}
	\varsigma_1=\frac{p(k-1)(p+2+k)}{4n}.
\end{align*}
Applying similar analysis to that in Section \ref{sec:pfonesamchisq1},
the conclusion follows by $\vartheta_1(n,p)=\varsigma_1/\sqrt{f}$. 

\medskip

\noindent \textit{(ii) Chi-squared approximation with the Bartlett correction.} \quad
In this problem,
\begin{align*}
\rho=1-\frac{1}{2n}(p+k+2).	
\end{align*}
It follows that 
\begin{align*}
\varsigma_l=\frac{(-1)^{l+1}}{l(l+1)}\left( \frac{2}{\rho n}\right)^l\left[ \sum_{j_1=1}^{k-1}B_{l+1}\biggr\{ \frac{(1-\rho)n-(j_1+p)}{2} \biggr\} - \sum_{j_2=1}^{k-1}B_{l+1}\biggr\{  \frac{(1-\rho)n-j_2}{2} \biggr\}\right].	
\end{align*}
We calculate that
\begin{align*}
	\varsigma_2=\frac{(k-1)p(p^2+k^2-2k-4)}{48\rho^2 n^2}.
\end{align*}
Similarly to Section  \ref{sec:pfonesamchisq1}, 
the conclusion then follows by $\vartheta_2(n,p)=2\varsigma_2/\sqrt{f}$. 

\subsubsection{Proof of Theorem \ref{thm:multisamchisq} (V): Testing the Equality of Several Covariance Matrices} \label{sec:pfequalmultcov}

In this problem, by Section 8.2.4 in \cite{Muirhead2009}, we have $f=p(p+1)(k-1)/2$, and 
\begin{itemize}
\item $K_1=kp, \ K_2=p$;
	\item $\xi_{1,j_1}=(n_r-1)/2$, \quad $j_1=(r-1)p+1,\ldots, rp$, \ \ $(r=1,\ldots, k)$;
	\item $\upsilon_{1,j_1} = -(r-1)/2$, \quad $j_1=r,p+r,\ldots, (k-1)p+r$,\ \ $(r=1,\ldots, p)$; 
	\item $\xi_{2,j_2}=(n-k)/2$,\ $\upsilon_{2,j_2}=-(j_2-1)/2$, \quad $j_2=1,\ldots, p$. 
\end{itemize}

\medskip

\noindent \textit{(i) Chi-squared approximation.} \quad
Consider $\rho =1$ and $p^2/n\to 0$. Then 
\begin{align*}
	\varsigma_l=\frac{(-1)^{l+1}}{l(l+1)}\Biggr[ 	\sum_{r_1=1}^k\sum_{r_2=1}^{p} \biggr(\frac{2}{n_{r_1}-1}\biggr)^l{B}_{l+1}\biggr(-\frac{r_2-1}{2} \biggr) - \sum_{j=1}^p \biggr(\frac{2}{n-k}\biggr)^l {B}_{l+1}\biggr(-\frac{j-1}{2} \biggr) \Biggr],\notag
\end{align*}
which satisfies $\varsigma_l=O(p^{l+2}/n^l)$.
Particularly, 
\begin{align*}
\varsigma_1=\biggr(\sum_{i=1}^k \frac{1}{n_{i}-1}  -\frac{1}{n-k}\biggr) \frac{1}{24}p(2p^2+3p-1).	
\end{align*}
Following similar analysis to that in Section  \ref{sec:pfthmchisqiii}, 
the conclusion then follows by $\vartheta_1(n,p)=\varsigma_1/\sqrt{f}$. 

\medskip

\noindent \textit{(ii) Chi-squared approximation with the Bartlett correction.} \quad
In this problem, 
\begin{align*}
\rho =1-\frac{(2p^2+3p-1)}{6(p+1)(k-1)}\Biggr(\sum_{i=1}^k\frac{1}{n_i-1} - \frac{1}{n-k}  \Biggr),	
\end{align*}
and we consider $p^3/n^2\to 0$. 
In this problem, 
\begin{align*}
	\varsigma_l=&~\frac{(-1)^{l+1}}{l(l+1)}\Biggr[ \sum_{r_1=1}^{k}\sum_{r_2=1}^p \frac{{B}_{l+1}\{ (1-\rho)(n_{r_1}-1)/2 -(r_2-1)/2 \}}{\{\rho (n_{r_1}-1)/2 \}^l} \notag \\
	&~ \hspace{4.3em} - \sum_{j=1}^{p} \frac{ {B}_{l+1}\{(1-\rho)(n-k)/2 -(j-1)/2 \} }{\{\rho (n-k)/2\}^l}  \Biggr]. 
\end{align*}
Note that $(1-\rho)(n-k)$ and $(1-\rho)(n_{r_1}-1)$ are of the order of $\Theta(p)$, 
${B}_{l+1}(\cdot)$ is a polynomial of order $l+1$, and $k$ is finite. 
Then for $l\geq 2$, $\varsigma_l=O(p^{l+2}/n^l)$. 
In particular, we calculate
\begin{align*}
\varsigma_2=\frac{p(p+1)}{48\rho^2}\Biggr[ (p-1)(p+2)\Biggr\{\sum_{i_1=1}^{k}\frac{1}{(n_{i_1}-1)^2}-\frac{1}{(n-k)^2}\Biggr\} - 6(k-1)(1-\rho)^2 \Biggr].	
\end{align*}
Similarly to Section  \ref{sec:pfthmchisqiii}, 
the conclusion then follows by $\vartheta_2(n,p)=2\varsigma_2/\sqrt{f}$. 

\subsubsection{Proof of Theorem \ref{thm:multisamchisq} (VI): Joint Testing the Equality of Several Mean Vectors and  Covariance Matrices}

In this problem, by Section 10.8.2 in \cite{Muirhead2009}, we have $f=(k-1)p(p+3)/2$, and 
\begin{itemize}
	\item $K_1=kp,$\ $ K_2=p$;
	\item $\xi_{1,j_1}=n_r/2$, \quad $j_1=(r-1)p+1, ~ \ldots,~ rp$, \ \ $(r=1,\ldots, k)$; 
 \item $\upsilon_{1,j_1} = -r/2$, \quad $j_1=r, ~ p+r, \ldots, ~(k-1)p+r$,\ \  $(r=1,\ldots, p)$;
	\item $\xi_{2,j_2}=n/2$,~ $\upsilon_{2,j_2}=-j_2/2$, \quad $(j_2=1,\ldots, p)$.
%	\item $\upsilon_{2,j_2}=-j_2/2$, \quad $j_2=1,\ldots, p$ 
\end{itemize}

\medskip

\noindent \textit{(i) Chi-squared approximation.} \quad
Consider $\rho =1$ and $p^2/n\to 0$. 
It follows that
%\begin{align*}
%	\varsigma_l=\frac{(-1)^{l+1}}{l(l+1)}\Biggr\{\sum_{i_1=1}^k \biggr(\frac{2}{n_{i_1}}\biggr)^l - \biggr( \frac{2}{n} \biggr)^l \Biggr\} \sum_{j=1}^{p} B_{l+1}\Big( -\frac{j}{2} \Big)
%\end{align*}
\begin{align*}
	\varsigma_l=&~\frac{(-1)^{l+1}}{l(l+1)}\left\{ \sum_{r_1=1}^{k}\sum_{r_2=1}^p \frac{{B}_{l+1}( -r_2/2 )}{( n_{r_1}/2 )^l} - \sum_{j=1}^{p} \frac{ {B}_{l+1}(-j/2 ) }{(n/2)^l}  \right\}. 
\end{align*}
%\begin{align*}
%\varsigma_l=\frac{(-1)^{l+1}}{l(l+1)}\left\{ \sum_{k=1}^{K_1}\frac{B_{l+1}(\tau_{1,k} + \upsilon_{1,k} )}{(\eta \times \xi_{1,k})^l} - \sum_{j=1}^{K_2}   \frac{ B_{l+1}(\tau_{2,j}+ \upsilon_{2,j})}{(\eta \times \xi_{2,j} )^l }\right\},
%\end{align*}
Particularly, we compute 
\begin{align*}
	\varsigma_1=\Biggr(\sum_{r=1}^k \frac{1}{n_{r}} - \frac{1}{n} \Biggr)  \frac{1}{24}p\left(2 p^{2}+9 p+11\right). 
\end{align*}
Following similar analysis to that in Section  \ref{sec:pfthmchisqiii}, 
the conclusion then follows by $\vartheta_1(n,p)=\varsigma_1/\sqrt{f}$. 

%\begin{align*}
%	\varsigma_1=\biggr(\sum_{i_1=1}^k \frac{1}{n_{i_1}}- \frac{1}{n}\biggr)\sum_{j=1}^{p}  {B}_{2}\Big( -\frac{j}{2} \Big) = \biggr(\sum_{i_1=1}^k \frac{1}{n_{i_1}}- \frac{1}{n}\biggr)\frac{1}{24} p\left(2 p^{2}+9 p+11\right). 
%\end{align*}

\medskip

\noindent \textit{(ii) Chi-squared approximation with the Bartlett correction.} \quad
In this problem,
\begin{align*}
\rho = 1-\Biggr(\sum_{r=1}^k \frac{1}{n_{r}} - \frac{1}{n} \Biggr)  \frac{\left(2 p^{2}+9 p+11\right)}{6(k-1)(p+3)}. 	
\end{align*}
It follows that $\varsigma_1=0$ and for $l\geq 2$,
\begin{align*}
	\varsigma_l=\frac{(-1)^{l+1}}{l(l+1)}\left\{ \sum_{r_1=1}^{k}\sum_{r_2=1}^p \frac{{B}_{l+1}\{ (1-\rho)n_{r_1}/2 -r_2/2 \}}{(\rho n_{r_1}/2 )^l} - \sum_{j=1}^{p} \frac{ {B}_{l+1}\{(1-\rho)n/2 -j/2 \} }{(\rho n/2)^l}  \right\}. 
\end{align*}
Particularly, we calculate
\begin{align*}
\varsigma_2=\frac{1}{\rho^2}	\Biggr\{\frac{p(p+1)(p+2)(p+3)}{48}\biggr( \sum_{i=1}^k \frac{1}{n_i^2} -\frac{1}{n^2} \biggr)-\frac{p(k-1)(p+3)}{8} (1-\rho)^2\Biggr\}. 	
\end{align*}
Applying similar analysis to that in Section  \ref{sec:pfthmchisqiii}, 
the conclusion then follows by $\vartheta_2(n,p)=2\varsigma_2/\sqrt{f}$.

\subsubsection{Proof of Theorem \ref{thm:chisqindp} (VII): Testing  Independence between Multiple Vectors} \label{sec:pfindpchisq}

In this problem, by Section 11.2.4 in \cite{Muirhead2009}, 
we have $f=(p^2-\sum_{j=1}^k p_j^2)/2$, and
\begin{itemize}
	\item $K_1=p,$\  \ $ K_2=p$;
	\item $\xi_{1,j_1}=n/2$, \ \ $\upsilon_{1,j_1} = -j_1/2$, \quad $j_1=1,\ldots, p$; 
	\item $\xi_{2,~p_1+\ldots+p_{r-1}+j_2}=n/2$,~ $\upsilon_{2,~p_1+\ldots+p_{r-1}+j_2}=-j_2/2$, \quad $r=1,\ldots, k$, \ \ $j_2=1,\ldots, p_r$.
\end{itemize}

\medskip

\noindent \textit{(i) Chi-squared approximation.} \quad
Consider $\rho=1$ and $p^2/n\to 0$. It follows that
\begin{align*}
	\varsigma_l=\frac{(-1)^{l+1}}{l(l+1)}\left\{ \sum_{j_1=1}^{p}\frac{B_{l+1}(-j_1/2)}{(n/2)^l} -\sum_{r=1}^k \sum_{j_2=1}^{p_r}   \frac{ B_{l+1}(-j_2/2)}{(n/2)^l }\right\}.
\end{align*}
Particularly,
\begin{align*}
\varsigma_1=	\frac{2(p^3-\sum_{j=1}^k p_j^3)+9(p^2-\sum_{j=1}^k p_j^2)}{24n}.
\end{align*}
Following similar analysis to that in Section  \ref{sec:pfthmchisqiii}, 
the conclusion then follows by $\vartheta_1(n,p)=\varsigma_1/\sqrt{f}$. 

\bigskip

\noindent \textit{(ii) Chi-squared approximation with the Bartlett correction.} \quad
In this problem, 
\begin{align*}
\rho =1 - \frac{2D_{p,3}+9D_{p,2}}{6nD_{p,2}}	
\end{align*}
where $D_{p,r}=p^r-\sum_{j=1}^k p_j^r$. 
%\begin{align*}
%	\rho = 1-\frac{2(p^3-\sum_{j=1}^k p_j^3)+9(p^2-\sum_{j=1}^k p_j^2)}{6n(p^2-\sum_{j=1}^k p_j^2)}
%\end{align*}
Then 
\begin{align*}
\varsigma_l=\frac{(-1)^{l+1}}{l(l+1)}\left\{ \sum_{j_1=1}^{p}\frac{B_{l+1}\{ (1-\rho)n/2 -j_1/2 \}}{( \rho n/2)^l} - \sum_{r=1}^{k}\sum_{j_2=1}^{p_r}   \frac{ B_{l+1}\{ (1-\rho)n/2 -j_2/2\}}{(\rho n/2 )^l }\right\}.
\end{align*}
In particular, we calculate
\begin{align*}
	\varsigma_2=    \frac{1}{(\rho n)^2}\left(\frac{1}{48}D_{p,4}-\frac{5D_{p,2}}{96} -\frac{D_{p,3}^2}{72D_{p,2}} \right).
\end{align*}
%where $D_{p,r}=p^r-\sum_{j=1}^k p_j^r$. 
Applying similar analysis to that in Section  \ref{sec:pfthmchisqiii}, 
the conclusion then follows by $\vartheta_2(n,p)=2\varsigma_2/\sqrt{f}$.

%\bigskip

\subsection{Proofs of Theorems \ref{thm:onesamnormal},  \ref{thm:multsamnormal}, \& \ref{thm:indepnormal}}\label{sec:pfthmnormal36}

In this section, we prove other problems in Theorems \ref{thm:onesamnormal}, \ref{thm:multsamnormal}, \& \ref{thm:indepnormal} similarly as in Section \ref{sec:pfonsamnomral3}. 
Specifically, we still define $\psi_0(s)=\exp(-s^2/2)$,
and we let $\psi_1(s)$ be  the characteristic function of $(-2\log \Lambda_n + 2\mu_n)/(2n\sigma_n)$, 
where $\Lambda_n$ denotes the corresponding likelihood ratio test statistic, and $\mu_n$ and $\sigma_n$ take the corresponding values given in Theorems \ref{thm:onesamnormal},  \ref{thm:multsamnormal}, \& \ref{thm:indepnormal}. 
%used in the proofs in Section \ref{sec:pfthm1and4}. 
By the analysis in Section \ref{sec:pfonsamnomral3}, 
we know that it suffices to prove the results similar to Lemma \ref{lm:chardiffgoal} on Page \pageref{lm:chardiffgoal}. 
In particular, in the following subsections, we prove that under  $H_0$ of each  test, 
when  $s=o(\min \{ (n/p)^{1/2}, f^{1/6}\})$,
the characteristic functions satisfy
\begin{align}
\log \psi_1(s)-\log \psi_0(s)=O\left(\frac{p}{n}+\frac{1}{\sqrt{f}}\right)s +\left(\frac{1}{p}+\frac{p}{n} \right)O\left( s^2\right)+ O\left( \frac{s^3}{\sqrt{f}}\right). \label{eq:chardiffgoalgeneral}
\end{align}

%In this section, we prove other problems in Theorems \ref{thm:onesamnormal} \& \ref{thm:multsamnormal} similarly to Section \ref{sec:pfonsamnomral3} using Lemma \ref{lm:chardiff}.  
%Specifically, we still define $g(s)=\exp(-s^2/2)$,
%and we let $f(s)$ be the characteristic function of $(-2\log \Lambda_n + 2\mu_n)/(2n\sigma_n)$,
%where $\Lambda_n$ denotes the corresponding likelihood ratio test statistic, and $\mu_n$ and $\sigma_n$ take the corresponding values used in the proofs in Section \ref{sec:pfthm1and4}. 
%Then under each  testing problem, we prove Lemma \ref{lm:chardiffgoal} on Page \pageref{lm:chardiffgoal} in the following subsections,  and the conclusions follow by similar analysis in Section  \ref{sec:pfonsamnomral3}. 
%\medskip
%When $s=o(\min \{ (n/p)^{1/2}, p^{1/3}\})$,
%\begin{align}
%\log f(s)-\log g(s)=O\left(\frac{p}{n}+\frac{1}{\sqrt{f}}\right)s +\left(\frac{1}{p}+\frac{p}{n} \right)O\left( s^2\right)+ O\left( \frac{s^3}{\sqrt{f}}\right). \label{eq:chardiffgoalgeneral}
%\end{align}	

%\begin{align*}
% &~O\left(\frac{p^2t}{n}\right) +\left(\frac{1}{p}+\frac{p}{n} \right)O\big({p^2t^2}\big)+ O\big( p^2 t^3 \big)\notag \\
% =&~	O\left(\frac{p^2s}{n\sqrt{f}}\right) +\left(\frac{1}{p}+\frac{p}{n} \right)O\biggr(\frac{p^2s^2}{f}\biggr)+ O\biggr( \frac{p^2 s^3}{f^{3/2}} \biggr)
%\end{align*}

\subsubsection{Proof of Theorem \ref{thm:onesamnormal} (I): Testing One-Sample Mean Vector}\label{sec:pfonesamnormal1}
%In this subsection, we prove that when $s=o(\min \{ (n/p)^{1/2}, p^{1/6}\})$, 
%\begin{align*}
%\log f(s)-\log g(s)=\left(\frac{1}{p}+\frac{p}{n} \right)O\left( s^2\right)+ O\left( \frac{s^3}{\sqrt{f}}\right). 	
%\end{align*}
Recall that in Section \ref{sec:pfonesam1}, we mention that testing one-sample mean vector can be viewed as testing coefficient vector of a multivariate linear regression model. 
By Section 10.5.3 in \cite{Muirhead2009}, we have 
%\begin{align*}
%	\log \psi_1(s)=\log \frac{\Gamma\big\{\frac{1}{2}(n-p)-nti \big\}}{\Gamma\big\{\frac{1}{2}(n-p)\big\}}-\log \frac{\Gamma\big(\frac{1}{2}n -nti\big)}{\Gamma\big(\frac{1}{2}n\big)} + \frac{\mu_n si}{n\sigma_n},
%\end{align*} where $t=s/(2n\sigma_n)$. 
\begin{align*}
	\log \psi_1(s)=\log \frac{\Gamma\big\{\frac{1}{2}n(1-ti)-\frac{1}{2}p \big\}}{\Gamma\big\{\frac{1}{2}(n-p)\big\}}-\log \frac{\Gamma\big\{\frac{1}{2}n (1-ti)\big\}}{\Gamma\big(\frac{1}{2}n\big)} + \frac{\mu_n si}{n\sigma_n},
\end{align*} where $t=s/(n\sigma_n)$. 
By  \eqref{eq:chisqconvgvar}, $t=s/(n\sigma_n)=O(s/\sqrt{f})$. 
By Lemma \ref{lm:ratiogammapprox} (on Page \pageref{lm:ratiogammapprox}),   
\begin{align*}
\log \frac{\Gamma\big\{\frac{1}{2}n(1-ti)-\frac{1}{2}p \big\}}{\Gamma\big\{\frac{1}{2}(n-p)\big\}}	=&~\left\{\frac{1}{2}(n-p)-\frac{1}{2}nti\right\}\log \left\{\frac{1}{2}(n-p)-\frac{1}{2}nti \right\}+\frac{1}{2}nti\notag \\
&~ - \frac{1}{2}(n-p)\log\left\{\frac{1}{2}(n-p)\right\} +\frac{nti}{2(n-p)}+O\left( \frac{t}{n}+t^2\right).
\end{align*}
%\begin{align*}
%\log \frac{\Gamma\{\frac{1}{2}(n-p)-nti \}}{\Gamma\{\frac{1}{2}(n-p)\}}=&~\left\{\frac{1}{2}(n-p)-nti\right\}\log \left\{\frac{1}{2}(n-p)-nti \right\}+nti\notag \\
%&~ - \frac{1}{2}(n-p)\log\left\{\frac{1}{2}(n-p)\right\} +\frac{nti}{n-p}+O\left( \frac{t}{n}+t^2\right). 
%\end{align*}
Similarly, we have 
\begin{align*}
\log \frac{\Gamma\{\frac{1}{2}n(1-ti)\}}{\Gamma(\frac{1}{2}n)}=&~\left\{\frac{n(1-ti)}{2}\right\}\log \left\{\frac{n(1-ti)}{2} \right\}+\frac{1}{2}nti- \frac{n}{2}\log\left(\frac{n}{2}\right)+\frac{ti}{2}+O\left( \frac{t}{n}+t^2\right). 
\end{align*}
%\begin{align*}
%\log \frac{\Gamma(\frac{1}{2}n -nti)}{\Gamma(\frac{1}{2}n)}=&~\left(\frac{1}{2}n-nti\right)\log \left(\frac{1}{2}n-nti \right)+nti- \frac{1}{2}n\log\left(\frac{1}{2}n\right)+{ti}+O\left( \frac{t}{n}+t^2\right). 
%\end{align*}
It follows that
\begin{align*}
\log \psi_1(s)=g_0\biggr(-\frac{nti}{2}\biggr) - g_0(0)	 + \frac{\mu_n si}{n\sigma_n} + O\left( \frac{pt}{n}+t^2\right),
\end{align*}
where we define in this subsection that $g_0(z)=\{(n-p)/2+z\}\log \{(n-p)/2+z\} - (n/2+z)\log (n/2+z)$. 
Following the proof of Lemma \ref{lm:gjsumapprox} (see Section \ref{sec:gjsumapprox} on Page \pageref{sec:gjsumapprox}), we similarly obtain
%\begin{align*}
%g_0(-nti) - g_0(0)= g_0^{(1)}(0) \times (-nti)-\frac{g_0^{(2)}(0)}{2}n^2t^2 + O(pt^3),  
%\end{align*}
\begin{align*}
g_0\biggr(-\frac{nti}{2}\biggr) - g_0(0)= g_0^{(1)}(0) \times \biggr(-\frac{nti}{2}\biggr)-\frac{g_0^{(2)}(0)}{2}\frac{n^2t^2}{4}+ O(pt^3),  
\end{align*}
where 
\begin{align*}
g_0^{(1)}(0)= \log \left(1-\frac{p}{n}\right),\quad \quad \quad g_0^{(2)}(0)=\frac{2p}{n(n-p)}. 
\end{align*}
Recall that $2n\sigma_n/\sqrt{2f}\to 1$ by \eqref{eq:chisqconvgvar}. 
%By Taylor's series, \eqref{eq:chisqconvgvar}, and $f=p$, 
Then by Taylor's series and $f=p$, 
\begin{align*}
g_0^{(2)}(0)n^2=4n^2\sigma_n^2\left\{1+O\left(\frac{p}{n}\right) \right\} = 4n^2\sigma_n^2+ O\left( \frac{p^2}{n}\right).
\end{align*}
Moreover, by Taylor's series, we have $ng_0^{(1)}(0) - 2\mu_n = O\left({p}/{n}\right).$
%\begin{align*}
%&~\frac{n}{2}g_0^{(1)}(0) - \mu_n \notag \\
%=&~-\frac{n}{2}\left[ (n-p-3/2)\left\{\log \left(1-\frac{1}{n}\right) -\log \left(1-\frac{1}{n-p} \right) \right\}+p\log\left(1-\frac{1}{n}\right) \right]=O\left(\frac{p}{n}\right). 	
%\end{align*}
In summary, by $t=s/(n\sigma_n)$ and $n\sigma_n=\Theta(\sqrt{p})$, we obtain
%Note that $n\sigma_n=\Theta(\sqrt{p})$ by \eqref{eq:chisqconvgvar}. 
%By $t=s/(n\sigma_n)$, we obtain
\begin{align*}
	\log \psi_1(s)=-\frac{\mu_n si}{n\sigma_n}-\frac{4n^2\sigma_n^2}{2}\frac{s^2}{4(n\sigma_n)^2}+\frac{\mu_n si}{n\sigma_n}+O\biggr(\frac{ps}{n}\biggr)+O\biggr(\frac{p}{n}+\frac{1}{p}\biggr)s^2+O\biggr(\frac{s^3}{\sqrt{p}}\biggr).
\end{align*}
Then \eqref{eq:chardiffgoalgeneral} is proved.

%conclusion then follows.  
%\medskip
%s
%\noindent \textit{(i) Chi-squared approximation.} \quad
%
%\medskip
%
%\noindent \textit{(ii) Chi-squared approximation with the Bartlett correction.} \quad
%

\subsubsection{Proof of Theorem \ref{thm:onesamnormal} (II): Testing One-Sample Covariance Matrix}
%In this subsection, we prove that
%when $s=o(\min \{ (n/p)^{1/2}, p^{1/6}\})$,
%\begin{align}
%\log f(s)-\log g(s)=O\left(\frac{p}{n}+\frac{1}{p}\right)s +O\left(\frac{1}{p}+\frac{p}{n} \right)s^2+ O\left( \frac{s^3}{\sqrt{f}}\right). \label{eq:chardiffgoalii}
%\end{align}	
%Let $n'=n-1$. 
By Corollary 8.3.6 in \cite{Muirhead2009}, we have
%\begin{align*}
%\log f(s)=-pnti\log p + \log \frac{\Gamma_p(n'/2-nti)}{\Gamma_p(n'/2)}+\log \frac{\Gamma(pn'/2)}{\Gamma( pn'/2 - pnti )}
%\end{align*}
%\begin{align*}
%\log f(s)=-pnti\log p + \log \frac{\Gamma_p\{(n-1)/2-nti\}}{\Gamma_p\{(n-1)/2\}}+\log \frac{\Gamma\{p(n-1)/2\}}{\Gamma\{ p(n-1)/2 - pnti \}}
%\end{align*}
%\begin{align*}
%\log f(s)=-p(n-1)ti\log p + \log \frac{\Gamma_p\{(n-1)(1/2-ti)\}}{\Gamma_p\{(n-1)/2\}}+\log \frac{\Gamma\{p(n-1)/2\}}{\Gamma\{ p(n-1)(1/2 -ti) \}}+2\mu_n ti. 
%\end{align*}
\begin{align*}
\log \psi_1(s)=-\frac{p(n-1)ti}{2}\log p + \log \frac{\Gamma_p\{\frac{1}{2}(n-1)(1-ti)\}}{\Gamma_p\{\frac{1}{2}(n-1)\}}+\log \frac{\Gamma\{\frac{1}{2}p(n-1)\}}{\Gamma\{ \frac{1}{2}p(n-1)(1 -ti) \}}+\mu_n ti. 
\end{align*}
%\begin{align*}
%	\mathrm{E}\Big\{i\times s \frac{-2\log \Lambda_n +2\mu_n}{2n\sigma_n} \Big\}
%\end{align*}
%Note that $t=s/(n\sigma_n)$ and 
By \eqref{eq:chisqconvgvar} and $f=\Theta(p^2)$, $n\sigma_n=\Theta(p)$. 
Then as $t=s/(n\sigma_n)$, the conditions in  Lemma \ref{lm:gammapaproxexpan} (on Page \pageref{lm:gammapaproxexpan} ) are satisfied and we have
%\begin{align*}
% \log \frac{\Gamma_p\{(n-1)/2-(n-1)ti\}}{\Gamma_p\{(n-1)/2\}}= &~-\beta_{n,1}(n-1)ti+\beta_{n,2}(n-1)^2t^2+\beta_{n,3}\big\{-(n-1)ti\big\}	 \notag \\
% &~+O\left(\frac{p^2t}{n}\right) +\left(\frac{1}{p}+\frac{p}{n} \right)O\big({p^2t^2}\big)+ O\big( p^2 t^3 \big), 
%% &~+O\left(\frac{p^2t}{n^2}\right) +\left(\frac{1}{p}+\frac{p}{n} \right)O\left(\frac{p^2t^2}{n^2}\right)+ O\Big( \frac{p^2 t^3}{n^3} \Big).
%\end{align*}
\begin{align*}
 \log \frac{\Gamma_p\{(n-1)(1-ti)/2\}}{\Gamma_p\{(n-1)/2\}}= &~-\frac{(n-1)\beta_{n,1}ti}{2}+\frac{(n-1)^2\beta_{n,2}t^2}{4}+\beta_{n,3}\biggr\{-\frac{(n-1)ti}{2}\biggr\}	 \notag \\
 &~+O\left(\frac{p^2t}{n}\right) +\left(\frac{1}{p}+\frac{p}{n} \right)O\big({p^2t^2}\big)+ O\big( p^2 t^3 \big), 
\end{align*}
%\begin{align*}
% \log \frac{\Gamma_p\{(n-1)/2-nti\}}{\Gamma_p\{(n-1)/2\}}= &~-\beta_{n,1}nti+\beta_{n,2}n^2t^2+\beta_{n,3}(-nti)	 \notag \\
% &~+O\left(\frac{p^2t}{n}\right) +\left(\frac{1}{p}+\frac{p}{n} \right)O\big({p^2t^2}\big)+ O\big( p^2 t^3 \big).
%% &~+O\left(\frac{p^2t}{n^2}\right) +\left(\frac{1}{p}+\frac{p}{n} \right)O\left(\frac{p^2t^2}{n^2}\right)+ O\Big( \frac{p^2 t^3}{n^3} \Big).
%\end{align*}
where $\beta_{n,1}$, $\beta_{n,2}$, and $\beta_{n,3}(\cdot)$ are defined in Lemma  \ref{lm:gammapaproxexpan}. 
In addition,  we can apply Lemma \ref{lm:ratiogammapprox} and obtain
%the conditions in  Lemma \ref{lm:ratiogammapprox} are satisfied, and we obtain 
\begin{align*}
\log \frac{\Gamma\{p(n-1)/2\}}{\Gamma\{ p(n-1)(1-ti)/2\}}	=&~-p\biggr\{\frac{n-1}{2} (1-ti)\biggr\}\log\biggr[p\biggr\{ \frac{n-1}{2} (1-ti) \biggr\}\biggr] \notag \\
&~+\frac{p(n-1)}{2}\log \frac{p(n-1)}{2}-\frac{p(n-1)ti}{2}-ti+ O\left(\frac{t}{pn} + t^2 \right).
\end{align*}
%\begin{align*}
%\log \frac{\Gamma\{p(n-1)/2\}}{\Gamma\{ p(n-1)/2 - p(n-1)ti \}}	=&~-p\biggr\{\frac{n-1}{2} - (n-1)ti \biggr\}\log\biggr[p\biggr\{ \frac{n-1}{2} - (n-1)ti \biggr\}\biggr] \notag \\
%&~+\frac{p(n-1)}{2}\log \frac{p(n-1)}{2}-p(n-1)ti-2ti+ O\left(\frac{t}{pn} + t^2 \right).
%\end{align*}
%\begin{align*}
%\log \frac{\Gamma\{p(n-1)/2\}}{\Gamma\{ p(n-1)/2 - pnti \}}	=&~-p\biggr(\frac{n-1}{2} - nti \biggr)\log\biggr\{p\biggr( \frac{n-1}{2} - nti \biggr)\biggr\} \notag \\
%&~+\frac{p(n-1)}{2}\log \frac{p(n-1)}{2}-pnti-\frac{2nti}{n-1}+ O\left(\frac{t}{pn} + t^2 \right).
%\end{align*}
By the definition of $\beta_{n,3}(\cdot)$ in Lemma  \ref{lm:gammapaproxexpan}, we have
\begin{align*}
\log \frac{\Gamma\{p(n-1)/2\}}{\Gamma\{ p(n-1)/2 - pnti/2\}}	= -\beta_{n,3}\biggr\{-\frac{(n-1)ti}{2}\biggr\}-\frac{p(n-1)ti(1-\log p)}{2}+ O\left(t + t^2 \right). 
\end{align*}
Since $\mu_n=(\beta_{n,1}+p)(n-1)/2$, 
 $2n^2\sigma^2=\beta_{n,2}(n-1)^2$, $t=s/(n\sigma_n)$, and $n\sigma_n=\Theta(p)$, 
\begin{align*}
	\log \psi_1(s)-\log \psi_0(s)= O\biggr(\frac{p}{n}+\frac{1}{p}\biggr)s+O\biggr(\frac{1}{p}+\frac{p}{n} \biggr)s^2+O\biggr(\frac{s^3}{p} \biggr).
\end{align*}

%Then \eqref{eq:chardiffgoalii} is obtained by $f=\Theta(p^2)$. 
%The conclusion then follows by noting that
%$\mu_n=(\beta_{n,1}+p)(n-1)/2$
%and $2n^2\sigma^2=\beta_{n,2}(n-1)^2$. 
%\begin{align*}
%\log \frac{\Gamma(x+ bi)}{\Gamma(x)}= \left( x + bi \right)\log (x + bi ) - x\log x - bi - \frac{bi}{2x} + O\biggr( \frac{b + b^2}{x^2} \biggr).
%\end{align*}
%\medskip
%
%\noindent \textit{(i) Chi-squared approximation.} \quad
%
%\medskip
%
%\noindent \textit{(ii) Chi-squared approximation with the Bartlett correction.} \quad

\subsubsection{Proof of Theorem \ref{thm:multsamnormal} (IV): Testing the Equality of Several Mean Vectors}
%Recall that in Section \ref{sec:pfmultimean}, we show that the testing problem can be viewed as testing the coefficient matrix in multivariate linear regression. 
%By Section 10.5.3 in \cite{Muirhead2009}, we know that 
%\begin{align*}
%\log f(s)=\sum_{j=1}^{k-1}\log \frac{ \Gamma\big\{ \frac{1}{2}n(1-2it)-\frac{j+p}{2} \big\}}{ \Gamma\big\{ \frac{1}{2}n(1-2it)-\frac{j}{2} \big\} } + \frac{\mu_n s i}{n\sigma_n}.	
%\end{align*}
%\begin{align*}
%\log f(s)=\sum_{j=1}^{k-1}\log \frac{ \Gamma\big\{ \frac{1}{2}(n-j-p) - nti \big\}}{ \Gamma\big\{ \frac{1}{2}(n-j)- nti\big\} } + \frac{\mu_n s i}{n\sigma_n}.	
%\end{align*}
By \eqref{eq:logexpgeneralform} and the analysis in Section \ref{sec:multisamchisq}, we have 
%$\log f(s)=\varphi(t)-\varphi(0)+2\mu_n ti$, where $t=s/(2n\sigma_n)$, and 	%\begin{align*}
%\varphi(t)=\sum_{j=1}^{k-1}\log \frac{ \Gamma\big\{ \frac{1}{2}(n-j-p) - nti \big\}}{ \Gamma\big\{ \frac{1}{2}(n-j)- nti\big\} }
%\end{align*}
\begin{align*}
\log \psi_1(s)=\sum_{j=1}^{k-1}\Biggr[\log\frac{\Gamma\big\{ \frac{1}{2}(n-j-p) - \frac{1}{2}nti \big\}}{\Gamma\big\{ \frac{1}{2}(n-j-p)\big\}} - \log\frac{ \Gamma\big\{ \frac{1}{2}(n-j)- \frac{1}{2}nti\big\}}{ \Gamma\big\{ \frac{1}{2}(n-j)\big\}} \Biggr] 	+\frac{\mu_n si}{n\sigma_n},
\end{align*}
where $t=s/(n\sigma_n)$. 
By Lemma \ref{lm:ratiogammapprox}, 
\begin{align*}
\log\frac{\Gamma\big\{ \frac{1}{2}(n-j-p) - \frac{1}{2}nti \big\}}{\Gamma\big\{ \frac{1}{2}(n-j-p)\big\}} =	&~\biggr\{ \frac{1}{2}(n-j-p) - \frac{1}{2}nti \biggr\}\log \biggr\{ \frac{1}{2}(n-j-p) - \frac{1}{2}nti \biggr\} \notag \\
&~ -\frac{n-j-p}{2}\log \frac{n-j-p}{2}+\frac{nti}{2}+O(t+t^2).
\end{align*}
Applying similar analysis, we obtain
%\begin{align*}
%\log\frac{\Gamma\big\{ \frac{1}{2}(n-j) - nti \big\}}{\Gamma\big\{ \frac{1}{2}(n-j)\big\}} =	&~\biggr\{ \frac{1}{2}(n-j) - nti \biggr\}\log \biggr\{ \frac{1}{2}(n-j) - nti \biggr\} \notag \\
%&~ -\frac{n-j}{2}\log \frac{n-j}{2}+nti+O(t+t^2).	
%\end{align*}
\begin{align*}
\log\frac{\Gamma\big\{ \frac{1}{2}(n-j-nti)\big\}}{\Gamma\big\{ \frac{1}{2}(n-j)\big\}} =	&~\biggr( \frac{n-j-nti}{2} \biggr)\log \biggr(\frac{n-j-nti}{2}\biggr) -\frac{n-j}{2}\log \frac{n-j}{2}+\frac{nti}{2}+O(t+t^2).	
\end{align*}
It follows that
$
\log \psi_1(s)=\sum_{j=1}^{k-1}\{g_j(nti/2)-g_j(0)\}+\mu_n s i/(n\sigma_n)+O(t+t^2), 
$
where we define in this subsection that
\begin{align*}
	g_j(z)=\biggr( \frac{n-j-p}{2} - z \biggr)\log \biggr( \frac{n-j-p}{2} - z \biggr) - \biggr( \frac{n-j}{2} - z \biggr)\log \biggr( \frac{n-j}{2}- z \biggr). 
\end{align*}
%Following similar proof to that of Lemma \ref{lm:gjsumapprox}, we have
%\begin{align*}
%g_j(z)-g_j(0)=g_j^{(1)}(0)nti-\frac{g_j^{(2)}(0)}{2}n^2t^2+	O(pt^3),
%\end{align*}
%where
%\begin{align*}
%\sum_{j=1}^{k-1}g_j^{(1)}(0)=-\biggr(n-p-k-\frac{1}{2}\biggr)\log\biggr(1-\frac{k-1}{n-p-1}\biggr) - (k-1)\left( 1-\frac{1}{n-p} \right)
%\end{align*}
%\begin{align*}
%\sum_{j=1}^{k-1}g_j^{(2)}(0)=	
%\end{align*}
%\begin{lemma}]\label{lm:gjsumv2}
%\begin{align*}
%\sum_{j=1}^{k-1}\{g_j(nti)-g_j(0)\}=\nu_{n,1}nti-\frac{1}{2}\nu_{n,2}n^2t^2+O(pt^3). 
%\end{align*}
%where 
%\begin{align*}
%\nu_{n,1}=	
%\end{align*}	
%\end{lemma}
Following similar proof to that of Lemma \ref{lm:gjsumapprox} (see Section \ref{sec:gjsumapprox}),   we obtain
\begin{align}
\sum_{j=1}^{k-1}\{g_j(nti)-g_j(0)\}=\sum_{j=1}^{k-1}g_j^{(1)}(0) \frac{nti}{2}-\frac{n^2t^2}{8}	\sum_{j=1}^{k-1}g_j^{(2)}(0)+O(pt^3), \label{eq:gjdef2}
\end{align}
where 
\begin{align*}
g_j^{(1)}(0)=\log\biggr( \frac{n-j}{2}\biggr) -\log\biggr( \frac{n-j-p}{2}\biggr), \quad \quad  
g_j^{(2)}(0)=\frac{2}{n-j-p}-\frac{2}{n-j}.
\end{align*}
%By Lemma A.2 in \cite{Jiang15}, we obtain
Note that
\begin{align*}
\frac{1}{2}\sum_{j=1}^{k-1}g_j^{(2)}(0)=\sum_{j=1}^{k-1}\frac{p}{(n-j-p)(n-j)}=\frac{p(k-1)}{(n-p-1)n}\left\{1+O\left(\frac{k}{n}\right)\right\},
\end{align*}
and 
%in this problem,
\begin{align*}
2\sigma_n^2	=\log\left\{ 1+\frac{p(k-1)}{(n-k)(n-p-1)}\right\}=\frac{p(k-1)}{(n-p-1)n}\left\{1+O\left(\frac{k}{n}\right)\right\}.
\end{align*}
Thus $\sum_{j=1}^{k-1}g_j^{(2)}(0)( 4\sigma_n^2)^{-1}=1+O(n^{-1})$. In addition, 
\begin{align*}
\sum_{j=1}^{k-1}g_j^{(1)}(0)=&~	\log \frac{\Gamma(n-1)}{\Gamma(n-k)}-\log \frac{\Gamma(n-p-1)}{\Gamma(n-p-k)}. 
\end{align*}
We then apply Lemma \ref{lm:loggammaexpan} to expand the $\log \Gamma(\cdot)$ function,  and calculate 
\begin{align*}
%\sum_{j=1}^{k-1}g_j^{(1)}(0)=&~\left(n-1-\frac{1}{2}\right)\log(n-1)-\left( n-k-\frac{1}{2}\right)\log(n-k)\notag \\
%&~- \left(n-p-1-\frac{1}{2}\right)\log\left(n-p-1\right) \notag \\
%&~+ \left(n-p-k-\frac{1}{2}\right)\log  \left(n-p-k\right) + O(n^{-1}) \notag \\
\sum_{j=1}^{k-1}g_j^{(1)}(0)=&~	-\left(n-p-k-\frac{1}{2}\right)\left\{\log\left( 1-\frac{p}{n-1}\right)-\log \left(1-\frac{p}{n-k}\right) \right\}\notag \\
&~ -p\log\left(1-\frac{k-1}{n-1}\right)-(k-1)\log \left(1-\frac{p}{n-1} \right)+ O(n^{-1}).
\end{align*}
Therefore  $\sum_{j=1}^{k-1}g_j^{(1)}(0)=-\mu_n/n+O(n^{-1})$. 
Then by \eqref{eq:gjdef2}, $t=s/(n\sigma_n)$, $n\sigma_n=\Theta(f^{1/2})$, and $f=\Theta(p)$, we have
\begin{align*}
	\log \psi_1(s)=&~\big\{-\mu_n/n+O(n^{-1})\big\} nti -\frac{n^2\sigma_n^2t^2}{2}\big\{1+O(n^{-1})\big\}+\mu_n ti+  O\left(t+t^2+pt^3\right)\notag \\
	=&~-\frac{s^2}{2}+O\left(\frac{1}{\sqrt{f}}\right)s+O\left(\frac{p}{n}+\frac{1}{f}\right)s^2+O\left(\frac{s^3}{\sqrt{f}}\right).
\end{align*}
By $\log \psi_0(s)=-s^2/2$, \eqref{eq:chardiffgoalgeneral} is proved.

\subsubsection{Proof of Theorem \ref{thm:multsamnormal} (V): Testing the Equality of Several Covariance Matrices}

%By Section 8.2.4 in \cite{Muirhead2009}, 
By \eqref{eq:logexpgeneralform} and the analysis in Section \ref{sec:pfequalmultcov}, we have
\begin{align*}
\log \psi_1(s) = &~ \log \frac{\Gamma_p\big\{\frac{1}{2}(n-k)\big\} }{  \Gamma_p\big\{\frac{1}{2}(n-k)(1-ti)\big\}}+  \sum_{j=1}^k \log \frac{\Gamma_p \big\{ \frac{1}{2}(n_j-1)(1-ti) \big\} }{ \Gamma_p \big\{ \frac{1}{2}(n_j-1) \big\} }  \notag \\
&~-p\left\{  (n-k)\log(n-k) - \sum_{j=1}^k (n_j-1)\log (n_j-1) \right\}\frac{ti}{2}+ \frac{\mu_n s i}{n\sigma_n},	
\end{align*}
where $t=s/(n\sigma_n)$. 
By Lemma \ref{lm:gammapaproxexpan}, we can expand $\log \Gamma_p(\cdot)$ and obtain 
\begin{align}
\log \psi_1(s) = &~ -\mu_n ti-\frac{n^2\sigma_n^2t^2}{2}+\mu_n ti+R_n(t), \label{eq:logpsiequalcov}
\end{align}
where 
the calculations of $\mu_n$ and $\sigma_n$ are similar to that in Section A.5 of \cite{Jiang15}, and thus the details are skipped here. 
In \eqref{eq:logpsiequalcov}, $R_n(t)$ denotes the remainder term of the expansion. 
Since Lemma \ref{lm:gammapaproxexpan} is used, we know that the remainder term satisfies 
\begin{align*}
	R_n(t)=O\left(\frac{p}{n}\right)s + \left(\frac{1}{p}+\frac{p}{n}\right)s^2+O\left( \frac{s^3}{p}\right).  
\end{align*} 
By $t=s/(n\sigma_n)$ and \eqref{eq:logpsiequalcov}, 
\eqref{eq:chardiffgoalgeneral} is obtained. 

%in \eqref{eq:logpsiequalcov} denotes the remainder term  which satisfies  
%\begin{align*}
%	R_n(t)=O\left(\frac{p}{n}\right)s + \left(\frac{1}{p}+\frac{p}{n}\right)s^2+O\left( \frac{s^3}{p}\right).  
%\end{align*} 
%by Lemma \ref{lm:gjsumapprox}. 

%and then follow the calculations in Section A.5 of \cite{Jiang15}.
%Note that Lemma \ref{lm:gjsumapprox} gives that the remainder terms are of order 
%$O(p^2t/n)+O(1/p+p/n)p^2t^2+ O(p^2t^3)$.  
%with the remainder terms of order $O(p^2t/n)+O(1/p+p/n)p^2t^2+ O(p^2t^3)$.  
%Following similar calculation in Section A.5 of \cite{Jiang15}, 
%Following the calculation in Section A.5 of \cite{Jiang15}, as $t=s/(2n\sigma_n)$ and $n\sigma_n=\Theta(p)$, we obtain
%It follows that 
%\begin{align*}
%\log \psi_1(s)-\log \psi_0(s)= O\left(\frac{p}{n}\right)s + \left(\frac{1}{p}+\frac{p}{n}\right)s^2+O\left( \frac{s^3}{p}\right). 
%\end{align*}

\subsubsection{Proof of Theorem \ref{thm:multsamnormal} (VI): Joint Testing the Equality of Several Mean Vectors and  Covariance Matrices}

By Corollary 10.8.3 in \cite{Muirhead2009}, 
\begin{align*}
\log \psi_1(s) =&~\log \frac{\Gamma_p \big\{ \frac{1}{2} (n-1) \big\} }{\Gamma_p\big\{ \frac{1}{2}(n-1) -\frac{1}{2}nti \big\} } + \sum_{j=1}^k \log \frac{\Gamma_p\big\{ \frac{1}{2}(n_j-1) -\frac{1}{2}n_j ti \big\}}{\Gamma_p\big\{\frac{1}{2}(n_j-1)\big\}} \notag \\
&~-p\biggr(n\log n - \sum_{j=1}^k n_j\log n_j \biggr)\frac{ti}{2}+\frac{\mu_n s i}{n\sigma_n}, 
\end{align*}
where $t=s/(n\sigma_n)$. 
By Lemma \ref{lm:gammapaproxexpan}, 
\begin{align}
 \log \frac{\Gamma_p\big\{ \frac{1}{2}(n_j-1) - \frac{1}{2}n_j ti \big\}}{\Gamma_p\big\{\frac{1}{2}(n_j-1)\big\}}=&~\left[2 p n_{j}+\left(n_{j}- p-\frac{3}{2}\right) n_{j} \log \left(1-\frac{p}{n_{j}-1}\right)\right]\frac{ti}{2}\label{eq:expanloggammproblem4} \\
 &~+\left\{\frac{p}{n_j-1}+\log \left(1-\frac{p}{n_j-1}\right)\right\}\frac{n_j^2t^2}{4} + \varrho_{n_j}(t)+ R_n(t),\notag 
\end{align}
where for an integer $l$, we define
\begin{align}
	\varrho_l(t)=p\left\{ \left(\frac{l-1}{2}+\frac{lt}{2} \right)\log \left(\frac{l-1}{2}+\frac{lt}{2} \right) -\frac{l-1}{2}\log \frac{l-1}{2} \right\}, \label{eq:varrholtdef} 
\end{align} and $R_n(t)$ denotes the remainder term and it is of the order of 
\begin{align}
	R_n(t)=O\left(\frac{pt}{n}\right)+O\left(\frac{1}{p}+\frac{p}{n}\right)p^2t^2+O\big(p^2t^3\big).\label{eq:rntdeforder}
\end{align}
%Then following similar proof to (A.19), (A.20) and Lemma A.7 in  \cite{Jiang15}, we obtain Lemma \ref{lm:rhonksumm} below. 
%calculation to that in Section A.4 of \cite{Jiang15}, we obtain Lemma \ref{lm:rhonksumm} below. 
In addition, to evaluate $\log \psi_1(s)$, we also use  Lemma \ref{lm:rhonksumm} below. 
\begin{lemma}\label{lm:rhonksumm}
Under the conditions of Theorem \ref{thm:multsamnormal}, 
as $p/n\to 0$ and $t=s/(n\sigma_n)=O(s/\sqrt{f})$,  
\begin{align}
n^2t^2\log\left(1-\frac{p}{n-1}\right)=&~n^2t^2\log\left(1-\frac{p}{n}\right)+O\left(\frac{p}{n}\right)t^2,\label{eq:rhonksumm1}\\
\left\{(n-p-3/2) n \log \left(1-\frac{p}{n-1}\right)\right\} t=&~\left\{(n-p-3/2) n \log \left(1-\frac{p}{n}\right)\right\} t- p t	+O\left(\frac{p}{n}\right)t. \notag 
%\left\{(2 n-2 p-3) n \log \left(1-\frac{p}{n-1}\right)\right\} t=&~\left\{(2 n-2 p-3) n \log \left(1-\frac{p}{n}\right)\right\} t-2 p t	+O\left(\frac{p}{n}\right)t. \notag 
\end{align}
Moreover, for $\varrho_l(t)$ defined in \eqref{eq:varrholtdef}, we have
\begin{align}
-\varrho_n(t)+\sum_{j=1}^k \varrho_{n_j}(t)	=\biggr( 1-k-n\log n+\sum_{j=1}^k n_j\log n_j\biggr)\frac{tp}{2}+O\left(\frac{pt}{n}+pt^2\right).\label{eq:rhonksumm2}
\end{align}
\end{lemma}
\begin{proof}
Please see Section \ref{sec:rhonksumm}	on Page \pageref{sec:rhonksumm}.
\end{proof}

\noindent By Lemma \ref{lm:rhonksumm} and the expansions of gamma functions in \eqref{eq:expanloggammproblem4}, we calculate
\begin{align}
&~\log \psi_1(s)\label{eq:psi14}\\
=&~\biggr\{ p-\left(n- p-\frac{3}{2}\right) n \log \left(1-\frac{p}{n}\right)+\sum_{j=1}^{k}\left(n_{j}-p-\frac{3}{2}\right) n_{j} \log \left(1-\frac{p}{n_{j}-1}\right)\biggr\}\frac{ti}{2}\notag \\
&~-\left(n^2L_{n,p}-\sum_{j=1}^k n_j^2L_{n_j-1,p}\right)\frac{t^2}{4}-p\left\{(1-k)-n \log n+\sum_{j=1}^{k} n_{j} \log n_{j}\right\}\frac{ti}{2} \notag \\
&~-p\biggr(n\log n - \sum_{j=1}^k n_j\log n_j \biggr)\frac{ti}{2}+\frac{\mu_n s i}{n\sigma_n} +R_n(t), \notag
\end{align}
where $R_n(t)$ denotes the remainder term of \eqref{eq:psi14}, which is of the order  same as that in \eqref{eq:rntdeforder}, whereas we mention that the exact value of $R_n(t)$ can change.  
Then we obtain \eqref{eq:chardiffgoalgeneral} by $t=s/(n\sigma_n)$ and $n\sigma_n=\Theta(f^{1/2})$. 

%Similar calculation in Section A.4 of \cite{Jiang15} can be applied. 
%To obtain \eqref{eq:psi14}, we need to characterize the remainder terms in (A.19), (A.20) and Lemma A.7 of \cite{Jiang15}.
%First, following similar proof to that of (A.19) and (A.20) in \cite{Jiang15}, we have
%To obtain \eqref{eq:psi14}, we use the refined versions of (A.19), (A.20) and Lemma A.7 in \cite{Jiang15}.
%Particularly, we next refine their results by further characterizing the remainder terms. 
%First, following similar proof to that of (A.19) and (A.20) in \cite{Jiang15}, we have
%We first examine the remainder terms in (A.19) and (A.20) of \cite{Jiang15}. 
%reprove (A.19) and (A.20) in \cite{Jiang15} to characterize the remainder terms. 
%\begin{align*}
%n^2t^2\log\left(1-\frac{p}{n-1}\right)=&~n^2t^2\log\left(1-\frac{p}{n}\right)+O\left(\frac{p}{n}\right)t^2,\\
%\left\{(2 n-2 p-3) n \log \left(1-\frac{p}{n-1}\right)\right\} t=&~\left\{(2 n-2 p-3) n \log \left(1-\frac{p}{n}\right)\right\} t-2 p t	+O\left(\frac{p}{n}\right)t.
%\end{align*}
%In addition, for Lemma A.7 in \cite{Jiang15}. 
%we consider $t=O(1/p)$, and 
%define
%\begin{align*}
%	\varrho_l(t)=p\left\{ \left(\frac{l-1}{2}+\frac{lt}{2} \right)\log \left(\frac{l-1}{2}+\frac{lt}{2} \right) -\frac{l-1}{2}\log \frac{l-1}{2} \right\}. 
%\end{align*}

%\medskip
\subsubsection{Proof of Theorem \ref{thm:indepnormal} (VII): Testing  Independence between Multiple Vectors} \label{sec:pfindpnormal}

By Theorem 11.2.3 in \cite{Muirhead2009}, we know
\begin{align*}
\log \psi_1(s)=\log\frac{\Gamma_p\{\frac{1}{2}(n-1) - \frac{1}{2}nti \}}{\Gamma_p\{\frac{1}{2}(n-1)\}}+\sum_{j=1}^k \frac{\Gamma_{p_j}\{\frac{1}{2}(n-1)\} }{\Gamma_{p_j}\{\frac{1}{2}(n-1) - \frac{1}{2}nti\}}	+ \frac{\mu_n si}{n\sigma_n}, 
\end{align*}
where $t=s/(n\sigma_n)$. 
By Lemma \ref{lm:gammapaproxexpan},  we can expand $\log \Gamma_p(\cdot)$ and obtain
\begin{align*}
\log \psi_1(s)=&~\left[2 p+\left(n-p-\frac{3}{2}\right) L_{n-1,p}-\sum_{j=1}^{k}\biggr\{ 2 p_{j}+\left(n-p_{j}-\frac{3}{2}\right) L_{n-1,p_j}\biggr\}\right]\frac{nti}{2}\notag \\
&~+\left\{\frac{p}{n-1} +L_{n-1,p}-\sum_{j=1}^k\biggr(\frac{p_j}{n-1} + L_{n-1,p_j}\biggr)  \right\}\frac{n^2t^2}{4}\notag \\
&~+\biggr(p-\sum_{j=1}^{k} p_{j}\biggr)\left\{\frac{n(1-ti)}{2} \log \frac{n(1-ti)}{2}-\frac{n}{2} \log \frac{n}{2}\right\}+\frac{\mu_n si}{n\sigma_n} +R_n(t),
\end{align*}
where $R_n(t)$ denotes the remainder term and its order satisfies 
\begin{align*}
	R_n(t)=O\left(\frac{pt}{n}\right)+O\left(\frac{1}{p}+\frac{p}{n}\right)p^2t^2+O\big(p^2t^3\big).
\end{align*}
Then we obtain \eqref{eq:chardiffgoalgeneral} by noticing $p-\sum_{j=1}^k p_j=0$ and $t=O(s/p)$.  

%Following similar analysis to that in Section A.4 of \cite{Jiang15}, we obtain
%\begin{align*}
%\log f(s)=&~\biggr\{ p-\left(n- p-\frac{3}{2}\right) n \log \left(1-\frac{p}{n}\right)+\sum_{j=1}^{k}\left(n_{j}-p-\frac{3}{2}\right) n_{j} \log \left(1-\frac{p}{n_{j}-1}\right)\biggr\}ti\notag \\
%&~-\left(\sum_{j=1}^k n_j^2L_{n_j-1,p}-n^2L_{n,p}\right)t^2-p\left\{(1-k)-n \log n+\sum_{j=1}^{k} n_{j} \log n_{j}\right\}ti \notag \\
%&~-p\biggr(n\log n - \sum_{j=1}^k n_j\log n_j \biggr)ti+ \frac{\mu_n s i}{n\sigma_n} +R_n(t).
%\end{align*}
%\begin{align*}
%\log f(s)=&~\left[2 p+\left(n-p-\frac{3}{2}\right) \log \left(1-\frac{p}{n-1}\right)+\sum_{i=1}^{k}\biggr\{ 2 p_{i}+\left(n-p_{i}-\frac{3}{2}\right) \log \left(1-\frac{p_{i}}{n-1}\right)\biggr\}\right]ti\notag \\
%&~+\left[\frac{p}{n-1} +\log\left(1-\frac{p}{n-1}\right)+\sum_{j=1}^k\biggr\{ \frac{p_j}{n-1} + \log \left( 1-\frac{p_j}{n-1}\right)\biggr\}  \right]t^2
%\end{align*}

\bigskip

\section{Proofs of Assisted Lemmas}\label{sec:assitedlemmas}

\subsection{Results on Asymptotic Expansions of the Gamma Functions}\label{sec:gammafunc}

In this section, we provide some results on asymptotic expansions of the gamma functions, 
which are repeatedly used in the proofs.  
% we use Stirling's formula to expand the  logarithm of the gamma function;
We first give the following Lemma \ref{lm:loggammaexpan} on the expansion of $\log \Gamma(z)$, which also provides the basis for 
other lemmas below. 
%the following lemmas. 
Lemma \ref{lm:loggammaexpan} and its proof can be found in 12.33 of  \cite{whittaker1996course}. 
%asymptotic expansion property of the logarithm of the gamma function, 
\begin{lemma}\label{lm:loggammaexpan}
Suppose that a complex number $z$ satisfies $\text{Re}(z)\geq \epsilon_1 >0$ and $|\mathrm{arg}(z)|\leq \pi/2-\epsilon_2$ with $\epsilon_1>0$ and $0<\epsilon_2< \pi/4$ being given in advance.  
%Suppose a complex number $z$ has a positive real part and satisfies $|\mathrm{arg}(z)|\leq \pi/2-\epsilon$ with $\epsilon>0$ being given in advance. 
When $|z|\to \infty$, and an even integer $L$, we have	
\begin{align}\label{eq:gammaexpanapprox2}
\log \Gamma(z)=	\left(z-\frac{1}{2}\right)\log z - z + \log \sqrt{2\pi}+\sum_{l=1}^{L-1}\frac{(-1)^{l+1}B_{l+1}(0)}{l(l+1)z^{l}}+R_L(z),
\end{align}
where $B_{l+1}(\cdot)$ represents the Bernoulli polynomial of order $l+1$, and
\begin{align*}
	|R_L(z)|=O\left( \frac{ |B_{L+2}(0)|}{(L+1)(L+2)|z|^{L+1}}\right).
\end{align*}
Particularly, we know $B_{l}(0)=0$ when $l$ is odd and $l\geq 3$.
% $B_{l+1}(0)=0$ when $l$ is even and $\geq 3$. 
%\begin{align*}
%	|R_L(z)|=O\left( \frac{ |\beta_{L}|}{(2L-1)(2L)|z|^{2L-1}}\right).
%\end{align*}
%Here $B_{l+1}(0)=0$ when $l$ is even, and $B_{2l}(0)=(-1)^{l-1}\beta_l$. 
%Equivalently, we can also write
%\begin{align}\label{eq:gammaexpanapprox}
%\log \Gamma(z)=	\left(z-\frac{1}{2}\right)\log z - z + \log \sqrt{2\pi}+\sum_{l=1}^{L-1}\frac{(-1)^{l-1}{\beta}_{l} }{(2l-1)(2l) z^{2l-1}}+R_L(z),
%\end{align}
%where $\beta_l=|B_{2l}(0)|$.
\end{lemma}

\medskip

In Lemma \ref{lm:loggammaexpan}, if we take $L=2$ and $z$ as a real number, by $B_2(0)=1/6$, we have
\begin{align}
	\log \Gamma(z)=	\left(z-\frac{1}{2}\right)\log z - z + \log \sqrt{2\pi}+\frac{1}{12z}+O(z^{-2}).\label{eq:logammaexpanl2} 
\end{align}

\noindent Given Lemma \ref{lm:loggammaexpan}, we next prove two additional lemmas on asymptotic expansions of the gamma functions. 

\begin{lemma}\label{lm:logzaexpan}
%Suppose a complex number $z+a$ has a positive real part, and satisfies $\mathrm{arg}(z+a)\leq \pi-\epsilon$  with $\epsilon>0$ being given in advance. 
Suppose a complex number $z+a$ satisfies $\text{Re}(z+a)\geq \epsilon_1 >0$ and $|\mathrm{arg}(z+a)|\leq \pi/2-\epsilon_2$ with $\epsilon_1>0$ and $0<\epsilon_2\leq \pi/4$ being given in advance. 
Assume $|a|\to \infty$ as $|z|\to \infty$ and $|a|=o(|z|)$.  
For a finite even  $L$,  when $|a|^{L+1}/|z|^L \to 0$, 
%When $|a|\to \infty$ as $|z|\to \infty$ and $|a|=o(|z|)$, for a finite even  $L$, 
\begin{align*}
	\log \Gamma(z+a) = \left(z+a-\frac{1}{2}\right)\log z-z + \log \sqrt{2\pi}+\sum_{l=1}^{L-1}\frac{(-1)^{l+1}B_{l+1}(a)}{l(l+1)z^l}+O\left(\frac{|a|^{L+1}}{|z|^{L}} \right).
\end{align*}	
%\begin{align*}
%	\log \Gamma(z+a) = \left(z+a-\frac{1}{2}\right)\log z-z + \log \sqrt{2\pi}+\sum_{l=1}^{L-1}\frac{(-1)^{l+1}B_{l+1}(a)}{l(l+1)z^l}+R(z,a).
%\end{align*}	
%where 
%\begin{align*}
%	R(z,a)=
%\end{align*}
%which is $O(|a|^{L+1}/|z|^{L})$. 
\end{lemma}
\begin{proof}
Please see Section \ref{sec:logzaexpan} on Page \pageref{sec:logzaexpan}. 
\end{proof}
\medskip

\begin{lemma}\label{lm:ratiogammapprox}
For a real number $x\to \infty$ and a real number $b=o(x)$, 
\begin{align*}
\log \frac{\Gamma(x+ bi)}{\Gamma(x)}= \left( x + bi \right)\log (x + bi ) - x\log x - bi - \frac{bi}{2x} + O\biggr( \frac{b + b^2}{x^2} \biggr),
\end{align*}
where $i$ denotes the imaginary unit. 
\end{lemma}
\begin{proof}
Please see Section \ref{sec:ratiogammapprox} on Page \pageref{sec:ratiogammapprox}. 
\end{proof}

\subsubsection{Proof of Lemma \ref{lm:logzaexpan} (on Page \pageref{lm:logzaexpan})}\label{sec:logzaexpan}

By \eqref{eq:gammaexpanapprox2}, for a finite even $L$, we have
\begin{align}
&~\log \Gamma(z+a)\label{eq:logammaza} \\
 =&~ \left(z+a-\frac{1}{2}\right)\log (z+a)-(z+a) +	\log \sqrt{2\pi}+\sum_{l=1}^{L-1}\frac{(-1)^{l+1}B_{l+1}(0)}{l(l+1)(z+a)^l}+O\big(|z+a|^{-L-1}\big)\notag \\
 =&~\left(z+a-\frac{1}{2}\right)\log z - z + \left(z+a-\frac{1}{2}\right)\log \left(1+\frac{a}{z}\right) -a + \log \sqrt{2\pi} \notag \\
 &~+\sum_{l=1}^{L-1}\frac{(-1)^{l+1}B_{l+1}(0)}{l(l+1)z^{l}} \left(1+ \frac{a}{z}\right)^{-l}+O\big(|z+a|^{-L-1}\big).\notag 
\end{align}
By Taylor's expansion, 
%\begin{align}
%\left(z+a-\frac{1}{2}\right)\log \left(1+\frac{a}{z}\right) -a 
%=\sum_{k=1}^{\infty} \frac{(-1)^{k+1}}{z^k} \left\{ \frac{a^{k+1}}{k(k+1)} - \frac{1}{2k} a^k \right\}. \label{eq:logazpartexpan}
%\end{align}
\begin{align}
\left(z+a-\frac{1}{2}\right)\log \left(1+\frac{a}{z}\right) -a 
=\sum_{k=1}^{L-1} \frac{(-1)^{k+1}}{z^k} \left\{ \frac{a^{k+1}}{k(k+1)} - \frac{1}{2k} a^k \right\}+O\left(\frac{|a|^{L+1}}{|z|^{L}} \right). \label{eq:logazpartexpan}
\end{align}
%\begin{align}
%\left(z+a-\frac{1}{2}\right)\log \left(1+\frac{a}{z}\right) -a 
%=\sum_{k=1}^{L-1} \frac{(-1)^{k+1}}{z^k} \left\{ \frac{a^{k+1}}{k(k+1)} - \frac{1}{2k} a^k \right\}+R_1(a,z). \label{eq:logazpartexpan}
%\end{align}
Note that $B_0(0)=1$ and $B_1(0)=-1/2$. 
Thus
\begin{align}
\eqref{eq:logazpartexpan}=\sum_{k=1}^{L-1} \frac{(-1)^{k+1}}{k(k+1)z^k}\left\{B_0(0)a^{k+1} + \binom{k+1}{1}B_1(0) a^k\right\}+O\left(\frac{|a|^{L+1}}{|z|^{L}} \right).\label{eq:logazpartexpan2}
\end{align}
In addition, by Taylor's expansion, when $L$ is finite,  
%\begin{align*}
%\left(1+ \frac{a}{z}\right)^{-l}=\sum_{s=0}^{\infty} (-1)^{s} \binom{l+s-1}{s} \frac{a^s}{z^s}. 	
%\end{align*}
%It follows that
%\begin{align}
% &~\sum_{l=1}^{L}\frac{(-1)^{l+1}B_{l+1}(0)}{l(l+1)z^{l}} \left(1+ \frac{a}{z}\right)^{-l} \label{eq:blazterm} \\
%=&~	\sum_{l=1}^{L}\frac{(-1)^{l+1}B_{l+1}(0)}{l(l+1)z^{l}} \sum_{s=0}^{\infty} (-1)^{s} \binom{l+s-1}{s} \frac{a^s}{z^s}\notag \\
%=&~\sum_{k=1}^{\infty} \sum_{t=1}^k \frac{(-1)^{k+1}B_{t+1}(0)}{t(t+1)z^k} \frac{(k-1)!}{(t-1)!(k-t)!}a^{k-t}\notag \\
%=&~\sum_{k=1}^{\infty} \sum_{t=2}^{k+1} \frac{(-1)^{k+1}B_{t}(0)}{k(k+1)z^k}\binom{k+1}{t}a^{k+1-t}.\notag 
%\end{align}
\begin{align}
 &~\sum_{l=1}^{L-1}\frac{(-1)^{l+1}B_{l+1}(0)}{l(l+1)z^{l}} \left(1+ \frac{a}{z}\right)^{-l} \label{eq:blazterm} \\
=&~	\sum_{l=1}^{L-1}\frac{(-1)^{l+1}B_{l+1}(0)}{l(l+1)z^{l}} \left\{\sum_{s=0}^{L-1-l} (-1)^{s} \binom{l+s-1}{s} \frac{a^s}{z^s} +O\left( \left|{a}/{z}\right|^{L-l}\right) \right\}\notag \\
=&~\sum_{k=1}^{L-1} \sum_{t=1}^k \frac{(-1)^{k+1}B_{t+1}(0)}{t(t+1)z^k} \frac{(k-1)!}{(t-1)!(k-t)!}a^{k-t}+ O\left( \left|{a}/{z}\right|^{L}\right)\notag \\
=&~\sum_{k=1}^{L-1} \sum_{t=2}^{k+1} \frac{(-1)^{k+1}B_{t}(0)}{k(k+1)z^k}\binom{k+1}{t}a^{k+1-t}+ O\left( \left|{a}/{z}\right|^{L}\right).\notag 
\end{align}
Combining \eqref{eq:logammaza}, \eqref{eq:logazpartexpan2}, and \eqref{eq:blazterm}, we obtain
\begin{align*}
	\log \Gamma(z+a)=&~\left(z+a-\frac{1}{2}\right)\log z - z + \log \sqrt{2\pi}  \notag \\
	&~+\sum_{k=1}^{L-1} \frac{(-1)^{k+1}}{k(k+1)z^k}\left\{\sum_{t=0}^{k+1} \binom{k+1}{t}B_t(0)a^{k+1-t}\right\}+O\left(\frac{|a|^{L+1}}{|z|^{L}} \right).
\end{align*}
%where we use
By the property of the Bernoulli polynomials, $B_{k+1}(a)=\sum_{t=0}^{k+1} B_t(0)a^{k+1-t};$ see, e.g., 
Eq. (13) on Page 21 in \cite{luke1969special}.  
%\begin{align*}
%	B_{k+1}(a)=\sum_{t=0}^{k+1} B_t(0)a^{k+1-t}. 
%\end{align*}
Therefore the lemma is proved. 

%\bigskip

%\begin{align*}
%\log \Gamma(x)=\left(x+bi-\frac{1}{2}\right)\log (x+bi) - (x+bi) + \log \sqrt{2\pi}+	2\int_{0}^{\infty} \frac{\mathrm{arc tan} (t/x) }{ e^{2\pi t} -1 }\mathrm{d} t.
%\end{align*} 
%Note that for any integer $L\geq 1$,
%\begin{align*}
%\mathrm{arc tan} (x)=\sum_{l=0}^{L-1}\frac{(-1)^lx^{2l+1}}{2l+1}+(-1)^L\int_0^{x} \frac{u^{2L}}{1+u^2}\mathrm{d}u
%\end{align*}

\subsubsection{Proof of Lemma \ref{lm:ratiogammapprox} (on Page \pageref{lm:ratiogammapprox})}\label{sec:ratiogammapprox}
By Binet's second formula of the gamma function, it can be obtained that for a complex number $z$ with positive real part, and any integer $L\geq 1$,
%\begin{align*}
%	\log \Gamma(z)=\left(z-\frac{1}{2}\right)\log z- z+\log\sqrt{2\pi}+\sum_{l=1}^{L}\frac{(-1)^{l-1}\beta_l}{(2l-1)(2l)z^{2l-1}}+\frac{2(-1)^{L}}{z^{2L-1}}\int_0^{t}\frac{u^{2L}\mathrm{d} u }{u^2+z^2};
%\end{align*}
\begin{align*}
	\log \Gamma(z)=\left(z-\frac{1}{2}\right)\log z- z+\log\sqrt{2\pi}+\sum_{l=1}^{L}\frac{B_{2l}(0)}{(2l-1)(2l)z^{2l-1}}+\frac{2(-1)^{L}}{z^{2L-1}}\int_0^{\infty} \int_0^{t}\frac{u^{2L}\mathrm{d} u }{u^2+z^2}\frac{	\mathrm{d}t}{e^{2\pi t} -1};
\end{align*}
please see Page 252 in \cite{whittaker1996course} for details. 
Take $L=1$, and by $B_2(0)=1/6$, we have
%\begin{align*}
%\log \Gamma(x)=\left(x-\frac{1}{2}\right)\log x - x + \log \sqrt{2\pi}+2\int_0^{\infty}\left(\frac{t}{x} - \int_0^{t/x}\frac{u^2}{1+u^2}	\mathrm{d}u \right) \frac{1}{e^{2\pi t} -1}	\mathrm{d}t
%\end{align*}
%By
%\begin{align*}
%	\int_0^{\infty}\frac{t^{2L-1}}{e^{2\pi t} -1}\mathrm{d}t = \frac{\beta_L}{4L},
%\end{align*}
\begin{align*}
\log \Gamma(x)=\left(x-\frac{1}{2}\right)\log x - x + \log \sqrt{2\pi}+\frac{1}{12x}	-\frac{2}{x} \int_0^{\infty}\left(\int_0^{t} \frac{u^2}{x^2+u^2}	\mathrm{d}u \right)\frac{	\mathrm{d}t}{e^{2\pi t} -1}.
\end{align*}
Similarly, we have
\begin{align*}
	\log \Gamma(x+bi)=&~\left(x+bi-\frac{1}{2}\right)\log (x+bi) - (x+bi) + \log \sqrt{2\pi}\notag \\
	&~+\frac{1}{12(x+bi)}	-\frac{2}{x+bi} \int_0^{\infty}\left(\int_0^{t} \frac{u^2}{(x+bi)^2+u^2}	\mathrm{d}u \right)\frac{	\mathrm{d}t}{e^{2\pi t} -1}.
\end{align*}
%By $\beta_1=1/6$,
It follows that
\begin{align}
&~\log \frac{ \Gamma(x+bi)}{\Gamma(x)}\label{eq:logamabir2} \\
=&~ (x+ bi)\log (x+ bi) - x\log x -bi - \frac{1}{2}\log\left(1+\frac{bi}{x}\right)+\frac{1}{12}\left( \frac{1}{x+bi}-\frac{1}{x}\right)+\tilde{R}_2,\notag 
%&~-2\int_0^{\infty}\int_0^{t} \left[\frac{1}{(x+bi)\{(x+bi)^2+u^2\}}-\frac{1}{x(x^2+u^2)}\right]u^2 \mathrm{d}u \frac{	\mathrm{d}t}{e^{2\pi t} -1}.
\end{align}
where
\begin{align*}
\tilde{R}_2=-2\int_0^{\infty}\int_0^{t} \left[\frac{u^2}{(x+bi)\{(x+bi)^2+u^2\}}-\frac{u^2}{x(x^2+u^2)}\right] \mathrm{d}u \frac{	\mathrm{d}t}{e^{2\pi t} -1}.
\end{align*}
To evaluate $\tilde{R}_2$, we note that 
\begin{align}
&~	\frac{u^2}{(x+bi)\{(x+bi)^2+u^2\}}-\frac{u^2}{x(x^2+u^2)} \notag \\[3pt]
%=&~	-\frac{(x+bi)\{(x+bi)^2+u^2\}-x(x^2+u^2)}{(x+bi)\{(x+bi)^2+u^2\} x(x^2+u^2)}
%=&~\frac{u^2}{x}\left[ \frac{1}{(1+bi/x)\{ x^2(1+bi/x)^2+u^2 \}}-\frac{1}{x^2+u^2}  \right] \notag \\[4pt]
%=&~\frac{u^2}{x}\left[ \frac{1}{(1+b_xi)\{ x^2(1+b_xi)^2+u^2 \}}-\frac{1}{x^2+u^2}  \right] \notag \\[4pt]
%%=&~-\frac{u^2}{x}\times \frac{(1+bi/x)\{ x^2(1+bi/x)^2+u^2 \} -(x^2+u^2)}{(1+bi/x)\{ x^2(1+bi/x)^2+u^2 \}(x^2+u^2)}\notag \\[4pt]
%=&~-\frac{u^2}{x^3}\times \frac{(1+bi/x)\{ (1+bi/x)^2+(u/x)^2 \} -\{1+(u/x)^2\}}{(1+bi/x)\{ (1+bi/x)^2+(u/x)^2 \}(1+(u/x)^2)}\notag \\[5pt]
%=&~-\frac{u^2}{x^3}\times  \frac{\frac{2bi}{x}-\frac{b^2}{x^2}+\frac{bi}{x}\left\{(1+\frac{bi}{x})^2+ (\frac{u}{x})^2 \right\} }{(1+\frac{bi}{x})\left\{(1+\frac{bi}{x})^2+ (\frac{u}{x})^2 \right\}\left\{(1+ (\frac{u}{x})^2 \right\}  } \notag \\[5pt]
=&~ -\frac{u^2}{x^3}\times \frac{2b_x i - b_x^2 + b_xi\{(1+b_xi)^2 + u_x^2 \}}{(1+b_xi) \{ (1+b_xi)^2+u_x^2 \}(1+u_x^2)} \notag \\[3pt] 
=&~ -\frac{u^2b_x}{x^3(1+b_xi)(1+u_x^2)}\times \left[\frac{ 2 i - b_x }{ (1+b_xi)^2+u_x^2  } +i \right],\notag %\label{eq:difftermgammaratio}
\end{align}
where for easy presentation, we let $b_x=b/x$ and $u_x=u/x$. 
Since $b=o(x)$, $|(1+b_xi)^{-1}|$ is bounded. Moreover, we also know $(1+u_x^2)^{-1}$ and $|\{(1+b_x i)^2 + u_x^2\}^{-1}|$ are bounded. 
% and then 
%\begin{align*}
%\eqref{eq:difftermgammaratio}=	
%\end{align*}
It follows that there exists a constant $C$ such that 
\begin{align*}
|\tilde{R}_2|\leq \frac{Cb_x }{x^{3}}\int_0^{\infty}\left(\int_0^t  {u^2}\mathrm{d}u	\right)\frac{	\mathrm{d}t}{e^{2\pi t} -1} =O\biggr(\frac{b}{x^4} \biggr),
\end{align*}
where we use $\int_0^{\infty} {t^3}(e^{2\pi t} -1)^{-1}\mathrm{d}t$ is a constant; see 7.2 in \cite{whittaker1996course}.  
%In addition, 
Lemma \ref{lm:ratiogammapprox} is then obtained by \eqref{eq:logamabir2} and 
\begin{align*}
\log\biggr(1+\frac{bi}{x}\biggr)=	\frac{bi}{x} + O\biggr(\frac{b^2}{x^2}\biggr), \quad \quad \quad \frac{1}{x+bi}-\frac{1}{x}=O\biggr(\frac{b}{x^2}\biggr).   
\end{align*}

\smallskip
\subsection{Lemmas for Theorems \ref{thm:onesamchisq}, \ref{thm:multisamchisq} \& \ref{thm:chisqindp}} \label{sec:chisqlemma}

\subsubsection{Proof of Lemma \ref{lm:lrtiiicharacter} (on Page \pageref{lm:lrtiiicharacter})} \label{sec:lrtiiicharacter}

By \eqref{eq:mmtjoint}, we can write
%\begin{align*}
%\mathrm{E}\{\exp(t\log \Lambda_n)\}=\mathrm{E}(\Lambda_n^t)=\left(\frac{2 e}{n}\right)^{n p t/2}(1+t)^{-n p(1+t)/2}\times \frac{\Gamma_{p}((n(1+t)-1) / 2)}{\Gamma_{p}((n-1) / 2)}	
%\end{align*}
%\begin{align*}
%\mathrm{E}\{\exp( -2 it \eta \log \Lambda_n  ) \}=	\left(\frac{2 e}{n}\right)^{-i\eta n p t}(1+t)^{-n p(1-2i \eta t)/2}\times \frac{\Gamma_{p}\{(n(1-2i\eta t)-1) / 2\}}{\Gamma_{p}\{(n-1) / 2\}}	
%\end{align*}
\begin{align*}
\log \mathrm{E}\{\exp( -2 it \eta \log \Lambda_n  ) \}=G_1+G_2+G_3,	
\end{align*}
where in this subsection, we let
\begin{align*}
	&G_1= -i\eta n p t\log \left(\frac{2 e}{n}\right), \quad   
	G_2= -\frac{np}{2}(1-2 i \eta t)\log (1-2i\eta t), \notag \\
	&G_3=  \log \Gamma_p \left(\frac{n-1}{2}-\eta n it \right)  - \log\Gamma_p \Big(\frac{n-1}{2}\Big). \notag %\\
%G_3 = &~ \log \Gamma_p \left\{\frac{1}{2}(n-1-2\eta n it) \right\}  - \log\Gamma_p \left\{\frac{1}{2}(n-1) \right\}.
\end{align*}
By the property of multivariate gamma function; see, e.g.,  Theorem 2.1.12 in  \citet{Muirhead2009}, 
we obtain
\begin{align*}
G_3=&~	\sum_{j=1}^p \log  \Gamma \left\{\frac{n}{2}(1-2\eta  it) - \frac{j}{2} \right\}-\sum_{j=1}^p\log \Gamma\left(\frac{n}{2}- \frac{j}{2}  \right)\notag \\
=&~\sum_{j=1}^p \left[	\log  \Gamma \left\{\frac{\eta n}{2}(1-2it)+ \frac{n(1-\eta)-j}{2} \right\}- \log \Gamma \left\{\frac{\eta n}{2}+ \frac{n(1-\eta)-j}{2}  \right\}\right]. \notag 
\end{align*}

We first examine $G_3$. When $\eta=1$ or $\eta=\rho$, 
for $1\leq j\leq p$, 
 $n(1-\eta)-j=O(p)$ and  $\eta n= \Theta(n)$.
As $p=o(n)$,  
$|\{n(1-\eta)-j\}\{\eta n (1-2it)\}^{-1}|=O(p/n)=o(1)$. 
Then we can apply Lemma \ref{lm:logzaexpan} on Page \pageref{lm:logzaexpan}, and obtain 
\begin{align*}
&~	\log  \Gamma \left\{\frac{\eta n}{2}(1-2it)+ \frac{n(1-\eta)-j}{2} \right\}\notag \\
=&~ \left\{\frac{\eta n}{2}(1-2it)+ \frac{n(1-\eta)-j-1}{2} \right\}\log \left\{ \frac{\eta n}{2}(1-2it) \right\} -  \frac{\eta n}{2}(1-2it)+\log \sqrt{2\pi} \notag \\
&~ + \sum_{l=1}^{L-1} \frac{(-1)^{l+1}}{l(l+1)}B_{l+1}\left\{\frac{n(1-\eta)}{2}-\frac{j}{2}\right\}\left\{ \frac{\eta n}{2}(1-2it) \right\}^{-l}+O\biggr(\frac{p^{L+1}}{n^{L}}\biggr),
%O\big\{ (p/n)^{L}\big\},
\end{align*}
and 
\begin{align*}
&~	\log  \Gamma \left\{\frac{\eta n}{2}+ \frac{n(1-\eta)-j}{2} \right\}\notag \\	
=&~\left\{\frac{\eta n}{2}+ \frac{n(1-\eta)-j-1}{2} \right\}\log \frac{\eta n}{2} - \frac{\eta n}{2}+\log \sqrt{2\pi} \notag \\
&~+ \sum_{l=1}^{L-1} \frac{(-1)^{l+1}}{l(l+1)}B_{l+1}\left\{\frac{n(1-\eta)}{2}-\frac{j}{2}\right\}\left(\frac{\eta n}{2}\right)^{-l}+O\biggr(\frac{p^{L+1}}{n^{L}}\biggr). 
%O\big\{ (p/n)^{L}\big\}.
\end{align*}
%Applying Lemma  \ref{lm:logzaexpan} similarly, and we can calculate
It follows that 
\begin{align*}
	G_3=&~-\eta pn ti \log \left(\frac{n}{2e} \right) - p\eta n it \log \eta +\frac{pn}{2}(1-2i\eta t)\log(1-2it)-\sum_{j=1}^p\frac{j+1}{2}\log(1-2it) \notag \\
&~+ \sum_{l=1}^{L-1} \frac{(-1)^{l+1}}{l(l+1)}\sum_{j=1}^pB_{l+1}\left\{\frac{n(1-\eta)}{2}-\frac{j}{2}\right\}\left(\frac{\eta n}{2}\right)^{-l}\left\{ (1-2it)^{-l}-1 \right\}+O\biggr(\frac{p^{L+2}}{n^{L}}\biggr).  
\end{align*}
%\bigskip
We next examine $G_2$. 
By $1-2i\eta t = \eta (1-2it) + 1-\eta$, and Taylor's expansion,
\begin{align*}
&~	 (1-2i\eta t)\log (1-2i\eta t) \notag \\
=&~ \{\eta(1-2it)+1-\eta \}\log\{\eta(1-2it) \}\notag \\
&~+ 1-\eta + (1-\eta) \sum_{l=1}^{L-1}\frac{(-1)^{l+1}}{l(l+1)}\biggr( \frac{1-\eta}{\eta}\biggr)^l(1-2it)^{-l}+ O\big\{(1-\eta)^{L+1}\big\}. 
\end{align*}
As $\log (1)=(1-2i\eta \times 0)\log (1-2i\eta \times 0)=0$, by applying Taylor's expansion similarly as above, 
\begin{align*}
&~(1-2i\eta t)\log (1-2i\eta t) - \log(1) \notag \\ 
=&~	-2i\eta t\log \eta(1-2it)  +\log (1-2it) \notag \\
&~+(1-\eta) \sum_{l=1}^{L-1}\frac{(-1)^{l+1}}{l(l+1)}\biggr( \frac{1-\eta}{\eta}\biggr)^l \big\{(1-2it)^{-l}-1\big \}+O\{(1-\eta)^{L+1}\}.\notag 	
\end{align*}
As $(1-\eta)/\eta = \{(1-\eta) n/2 \}/(\eta n/2)$,
\begin{align*}
G_2=&~	- \sum_{l=1}^{L-1}\frac{(-1)^{l+1}}{l(l+1)} \sum_{j=1}^p \left\{\frac{(1-\eta)n}{2}\right\}^{l+1}\left(\frac{\eta n}{2}\right)^{-l} \big\{(1-2it)^{-l}-1\big \}\notag \\
&~ +i \eta npt \log \eta(1-2it) -\frac{np }{2}\log (1-2it)+O\big\{ (1-\eta)^{L+1}pn\big\}.
\end{align*}
%\begin{align*}
%G_2=&~  i \eta npt \log \eta(1-2it) -\frac{np }{2}\log (1-2it) \notag \\
%&~ - \sum_{l=1}^{L-1}\frac{(-1)^{l+1}}{l(l+1)} \sum_{j=1}^p \left\{\frac{(1-\eta)n}{2}\right\}^{l+1}\left(\frac{\eta n}{2}\right)^{-l} \big\{(1-2it)^{-l}-1\big \}+O\big\{ (1-\eta)^{L+1}pn\big\}.\notag	
%\end{align*} 
In summary, as $1-\eta=O(p/n)$ when $\eta=1$ or $\rho$, we have 
\begin{align*}
	G_1+G_2+G_3=-\sum_{j=1}^p \frac{j+1}{2}\log(1-2it)+\sum_{l=1}^{L-1}\varsigma_l\big\{(1-2it)^{-l}-1\big\}+O\biggr(\frac{p^{L+2}}{n^{L}}\biggr),
\end{align*}
where
\begin{align*}
\varsigma_l= \frac{(-1)^{l+1}}{l(l+1)} \sum_{j=1}^p	\left[ B_{l+1}\left\{\frac{(1-\eta)n}{2}-\frac{j}{2}\right\}- \left\{\frac{(1-\eta)n}{2}\right\}^{l+1}\right]\left(\frac{\eta n}{2}\right)^{-l}. 
\end{align*}
Particularly, as $B_{l+1}(\cdot)$ is a polynomial of order $l+1$ and $(1-\eta)n=O(p)$, we have $\varsigma_l=O(p^{l+2}n^{-l})$.

%when $\varsigma=1$ or $\varsigma=\rho$, since $(1-\eta)n=O(p)$, $\eta n=\Theta(n)$, and

%\subsubsection{Notation and results of the finite difference}\label{sec:notationfinite}
\subsubsection{Notation of the finite difference and computation rules}\label{sec:notationfinite} 
In the following, we  prove Propositions \ref{prop:infinitesumm} and \ref{prop:infinitesumm2} and Lemma \ref{lm:order2diff} based on the calculus of the finite difference.  
To facilitate the proofs, we introduce some notation. 
%To facilitate the proofs of  Propositions \ref{prop:infinitesumm} and \ref{prop:infinitesumm2} and Lemma \ref{lm:order2diff}, we define some notation. 
%a function with respect to the degrees of freedom $f$ as
Given $x$, define a function with respect to the degrees of freedom  $f$ as $F_x(f)=P(\chi_f^2\leq x)$.
Let $\Delta_{2h}$ represent a forward difference operator with step $2h$, that is, 
$\Delta_{2h}(F_{x},f)=F_x(f+2h)-F_x(f).$ 
%\begin{align*}
%\Delta_{2h}(F_{x},f)=F_x(f+2h)-F_x(f).	
%\end{align*}
For an integer $v\geq 1$, it follows that the $v$-th order forward difference is
\begin{align*}
	\Delta^v_{2h}(F_{x}, f)= \sum_{w=0}^v\binom{v}{w}(-1)^{v-w} F(f+2hw),
\end{align*}
where $\Delta^1_{2h}(F_{x}, f)=\Delta_{2h}(F_{x}, f)$. 
Particularly, when $h=1$, we have
\begin{align*}
	\Delta^v_2(F_{x}, f)= \sum_{w=0}^v \binom{v}{w}(-1)^{v-w}P(\chi^2_{f+2w}\leq x);	
\end{align*}
when $h=2$, 
\begin{align*}
	\Delta^v_4(F_{x}, f)= \sum_{w=0}^v \binom{v}{w}(-1)^{v-w}P(\chi^2_{f+4w}\leq x). 
\end{align*}
%We also write $\Delta^v_{2h}\{F_{x}(f)\}$ without ambiguity. 
In the following proofs, we use several rules of the finite difference operator listed in Lemmas \ref{lm:Leibnizfinite}--\ref{lm:exchangerule} below, which can be found in  Section 3.7 of \cite{zwillinger2002crc}. 
%the Leibniz rule for the finite difference operator. 
\begin{lemma}[Leibniz rule] \label{lm:Leibnizfinite}
For two functions $F(f)$ and $G(f)$, and two positive integers $v$ and $h$, 
\begin{align*}
	\Delta_{h}^v(FG,f)=\sum_{w=0}^{v}\binom{v}{w} \Delta_h^{w}(F,f)\Delta_h^{v-w}(G, f+hw). 
\end{align*}	
\end{lemma}
\begin{lemma}[Linearity rule] \label{lm:linearityrulefinite}
For two constants $C_1$ and $C_2$, two functions $F(f)$ and $G(f)$,  and two positive integers $v$ and $h$, the linear combination $C_1F(f)+C_2G(f)$ satisfies
\begin{align*}
	\Delta_h^v (C_1F + C_2G, f)=C_1\Delta_h^v(F)+C_2\Delta_h^v(G). 
\end{align*}
\end{lemma}
\begin{lemma} \label{lm:exchangerule}
For a function $F(f)$ and positive integers $v_1$, $v_2$, $h_1$, and $h_2$, 
\begin{align*}
	\Delta_{h_2}^{v_2} \Delta_{h_1}^{v_1}(F,f)=\Delta_{h_1}^{v_1}\Delta_{h_2}^{v_2} (F,f) = \Delta_{h_2}^{v_2} \Delta_{h_1}^{v_1-1} \{ \Delta_{h_1}(F,f) \}=\Delta_{h_1}^{v_1}\Delta_{h_2}^{v_2-1}\{\Delta_{h_2}(F,f)\}.
\end{align*}
\end{lemma}
\medskip

\noindent Based on the notation and lemmas on the finite difference, we first prove Lemma \ref{lm:order2diff} in Section \ref{sec:order2diff}, and then use Lemma \ref{lm:order2diff}  to prove Propositions \ref{prop:infinitesumm} and \ref{prop:infinitesumm2} in Sections \ref{sec:pfpropiftydiff1} and  \ref{sec:pfpropiftydiff2}, respectively. 

\subsubsection{Proof of Lemma \ref{lm:order2diff} (on Page \pageref{lm:order2diff})}\label{sec:order2diff}

%To prove Lemma \ref{lm:order2diff}, we next show that
%Note that t
We prove \eqref{eq:probdiff1exact} in Lemma \ref{lm:order2diff} from the  cumulative distribution function of the chi-squared distribution.
% $\chi_f^2$. 
%The cumulative distribution function of $\chi_f^2$ is
In particular, by the probability density of $\chi_f^2$, we have 
% for $\chi_f^2$, 
% $\Pr( \chi^2_f \leq x ) = {\gamma( f/2, x/2 ) }\{\Gamma(f/2) \}^{-1}$, 
%\begin{align*}
%		\Pr\left( \chi^2_f \leq x \right) =\gamma\Big(\frac{f}{2}, \frac{x}{2} \Big) \Big\{\Gamma\Big(\frac{f}{2}\Big)\Big\}^{-1}
%\end{align*}
\begin{align*}
	\Pr\left( \chi^2_f \leq x \right) = \frac{\gamma( f/2, x/2 ) }{\Gamma(f/2)}, 
\end{align*}
where $\gamma(m,x)$ is the lower incomplete gamma function defined as $\gamma(m,x)=\int_{0}^{x} t^{m-1}e^{-t}\mathrm{d} t$, and $\Gamma(m)$ is the gamma function defined as $\Gamma(m)=\int_{0}^{\infty} t^{m-1}e^{-t}\mathrm{d}t$; see, e.g., Section 6.2 in \cite{press1992numerical}. 
Thus for an integer $h$, 
%by the notation in Section \ref{sec:notationfinite}, we write  $\Pr( \chi^2_{f+2h} \leq x )-\Pr( \chi^2_f \leq x )$ as $\Delta_{2h}^1(F_x,f)$, and 
%recall that we write $\Pr( \chi^2_{f+2h} \leq x )-\Pr( \chi^2_f \leq x )$ as $\Delta_{2h}^1(F_x,f)$, and then
\begin{align}
\Delta_{2h}^1(F_x,f)=\frac{ \Gamma\big(\frac{f}{2}\big)\gamma\big(\frac{f}{2}+h, \frac{x}{2}\big)- \Gamma\big(\frac{f}{2}+h\big)\gamma\big(\frac{f}{2}, \frac{x}{2}\big) }{\Gamma\big(\frac{f}{2}+h\big)\Gamma\big(\frac{f}{2}\big)},  \notag %\label{eq:chisqprobdiff2h} 
\end{align}
where $\Delta_{2h}^1(F_x,f)=\Pr( \chi^2_{f+2h} \leq x )-\Pr( \chi^2_f \leq x )$ following the notation in Section \ref{sec:notationfinite}. 
%\begin{align}
%\Pr\left( \chi^2_{f+2h} \leq x \right)-\Pr\left( \chi^2_f \leq x \right)=\frac{ \Gamma\big(\frac{f}{2}\big)\gamma\big(\frac{f}{2}+h, \frac{x}{2}\big)- \Gamma\big(\frac{f}{2}+h\big)\gamma\big(\frac{f}{2}, \frac{x}{2}\big) }{\Gamma\big(\frac{f}{2}+h\big)\Gamma\big(\frac{f}{2}\big)}.  \label{eq:chisqprobdiff2h} 
%\end{align}
%\begin{align*}
%\Pr\left( \chi^2_{f+2h} \leq x \right)-\Pr\left( \chi^2_f \leq x \right)=&~ \frac{\gamma\big(\frac{f}{2}+h, \frac{x}{2} \big) }{\Gamma\big( \frac{f}{2}+h\big) } -\frac{\gamma\big(\frac{f}{2}, \frac{x}{2} \big) }{\Gamma\big( \frac{f}{2}\big) } \notag \\
%=&~\frac{ \Gamma\left(\frac{f}{2}\right)\gamma\left(\frac{f}{2}+h, \frac{x}{2}\right)- \Gamma\left(\frac{f}{2}+h\right)\gamma\left(\frac{f}{2}, \frac{x}{2}\right) }{\Gamma\big(\frac{f}{2}+h\big)\Gamma\big(\frac{f}{2}\big)}. %\left\{ \Gamma\left(\frac{f}{2}\right)\gamma\left(\frac{f}{2}+h, \frac{x}{2}\right)- \Gamma\left(\frac{f}{2}+h\right)\gamma\left(\frac{f}{2}, \frac{x}{2}\right) \right\}
%\end{align*}
By integration by parts, we have 
%$\Gamma(m+1)=m\Gamma(m)$, 
\begin{align}
	\Gamma(m+1)=m\Gamma(m), \quad \text{ and then } \quad \Gamma(m+h)=\prod_{k=1}^h (m+h-k)\Gamma(m). \label{eq:recursgamma}
\end{align}
%and then $\Gamma(m+h)=\prod_{k=1}^h (m+h-k)\Gamma(m)$. 
Similarly, we have $\gamma(m+1,x)=m \gamma(m,x) - x^me^{-x}$, and then 
\begin{align*}
	\gamma(m+h, x)=\prod_{k=1}^h (m+h-k)\gamma(m,x) - \sum_{k=1}^h \prod_{t=1}^{k-1}(m+h-t) x^{m+h-k}e^{-x};
\end{align*}
this recurrence formulas can also be found in Sections 6.3 and 6.5 in \cite{abramowitz1948handbook}. 
It follows that
\begin{align*}
%\eqref{eq:chisqprobdiff2h} 	
\Delta_{2h}^1(F_x,f)=  -\frac{\sum_{k=1}^h\prod_{t=1}^{k-1}({f}/{2}+h-t)({x}/{2})^{\frac{f}{2}+h-k}e^{-x/2}}{ \prod_{t=1}^h ({f}/{2}+h-t) \times  \Gamma({f}/{2})}=- \sum_{k=1}^h \frac{\left({x}/{2}\right)^{\frac{f}{2}+h-k}e^{-x/2}}{\Gamma\left({f}/{2}+h-k+1\right)}. %\\
%=&~ - \sum_{k=1}^h \left\{\Gamma\left(\frac{f}{2}+h-k+1\right) \right\}^{-1}\left(\frac{x}{2}\right)^{\frac{f}{2}+h-k}e^{-x/2}.
\end{align*}
Therefore \eqref{eq:probdiff1exact} is proved. 
%By integration by parts, we have the following recurrence relations
%\begin{align*}
%	\gamma(s+1,x)=s\gamma(s,x) - x^se^{-x}; \quad \quad \quad \Gamma(s+1)=s\Gamma(s).
%\end{align*}

%To prove,
We next prove  \eqref{eq:probdiff1approx} in Lemma \ref{lm:order2diff} based on  \eqref{eq:probdiff1exact} by discussing $h\in \{1,2,3,4\}$, respectively. \\ 
\textit{(1).} We first consider $h=1$.  
Under this case, 
\begin{align}
%\Pr(\chi_{f+2}^2\leq x) - \Pr(\chi_{f}^2\leq x) 
\Delta_{2}^1(F_x,f) = - \frac{ (x/2)^{f/2}e^{-x/2}}{\Gamma(f/2+1)}.	 \label{eq:probdiff1exacth2}
\end{align}
%Note that $f/2+h-k+1 \to \infty$ as $f\to \infty$ and $k$ and $h$ are finite. 
By \eqref{eq:logammaexpanl2}, as $f\to \infty$,  $\Gamma(f/2)=(f/2)^{f/2-1/2}e^{-f/2}\sqrt{2\pi}\{1+O(f^{-1})\}$.
Moreover, by $\Gamma(f/2+1)=\Gamma(f/2)f/2$, we have
% and $\Gamma(f/2+1)=\Gamma(f/2)f/2$.
%Then
\begin{align*}
%	\frac{1}{\Gamma(f/2+1)}(x/2)^{f/2}e^{-x/2}=&~\frac{2}{f}\left(\frac{ex}{f}\right)^{f/2}\frac{\sqrt{f}}{2\sqrt{\pi }}e^{-x/2}\{1+O(f^{-1})\}\notag \\[2pt]
	\frac{1}{\Gamma(f/2+1)}(x/2)^{f/2}e^{-x/2}=&~\frac{1}{\sqrt{f\pi }}\left(\frac{x}{f}\right)^{{f}/{2}}\exp\left\{\frac{f-x}{2}+O(f^{-1})\right\} \notag \\
=&~ \frac{1}{\sqrt{f\pi}}\exp\left\{ \frac{f-x}{2} +\frac{f}{2}\log \biggr(1+\frac{x-f}{f}\biggr)  + O(f^{-1})\right\}. \notag 
\end{align*}
When $x=\chi_f^2(\alpha)$, we have $x=f+\sqrt{2f}\{z_{\alpha}+O(f^{-1/2})\}$ by \eqref{eq:xchisqorder}.  
Then by Taylor's series, 
%\begin{align*}
%\Gamma\left({f}/{2}+h-k+1\right)=\left({f}/{2}+h-k+1\right)^{\left({f}/{2}+h-k+1\right)-1/2}	
%\end{align*}
\begin{align*}
%\eqref{eq:h1diffprob}
\Delta_{2}^1(F_x,f) =\frac{1}{\sqrt{f\pi}}\exp\left\{ -\frac{(x-f)^2}{4f} + O(f^{-1/2})\right\} =	\frac{1}{\sqrt{f\pi}}\exp \left(-\frac{z_{\alpha}^2}{2}\right) \{1 + O(f^{-1/2})\}. 
\end{align*}
%\begin{align*}
%\eqref{eq:chisqprobdiff2h}=	
%\end{align*}
\textit{(2).} When $h=2$, by \eqref{eq:probdiff1exact}, \eqref{eq:recursgamma}, and $x=f+\sqrt{2f}\{z_{\alpha}+O(f^{-1/2})\}$, we have
\begin{align*}
%\Delta_{4}^1(F_x,f) =&~- \frac{1}{\Gamma(\frac{f}{2}+2)}\left(\frac{x}{2}\right)^{\frac{f}{2}+1}e^{-x/2} +\Delta_{2}^1(F_{x}, f) \notag \\
\Delta_{4}^1(F_x,f) =&~ -\frac{\frac{x}{2}}{\frac{f}{2}+1} \times \Delta_{2}^1(F_{x}, f) + \Delta_{2}^1(F_{x}, f)= -\frac{2}{\sqrt{f\pi}}\exp \left(-\frac{z_{\alpha}^2}{2}\right) \{1+O(f^{-1/2})\}. 
\end{align*}
%where we use $x=f+\sqrt{2f}\{z_{\alpha}+O(f^{-1/2})\}$ in the second equation.  
\textit{(3).} When $h=3$, similarly by \eqref{eq:probdiff1exact}, \eqref{eq:recursgamma}, and $x=f+\sqrt{2f}\{z_{\alpha}+O(f^{-1/2})\}$, we have
\begin{align*}
%\Delta_{6}^1(F_{x}, f)=&~ - \frac{1}{\Gamma(\frac{f}{2}+3)}\left(\frac{x}{2}\right)^{\frac{f}{2}+2}e^{-x/2} + \Delta_{4}^1[F_{x}](f)\notag \\
\Delta_{6}^1(F_{x}, f)=- \frac{(\frac{x}{2})^2}{ (\frac{f}{2}+2)(\frac{f}{2}+1) }  \Delta_{2}^1(F_{x},f)+ \Delta_{4}^1(F_{x},f)=-\frac{3}{\sqrt{f\pi}}\exp \left(-\frac{z_{\alpha}^2}{2}\right) \{1+O(f^{-1/2})\}.\notag	
\end{align*}
\textit{(4).} When $h=4$, similarly by \eqref{eq:probdiff1exact}, \eqref{eq:recursgamma}, and $x=f+\sqrt{2f}\{z_{\alpha}+O(f^{-1/2})\}$, we have
\begin{align*}
\Delta_{8}^1(F_{x}, f)=&~- \frac{(\frac{x}{2})^3}{(\frac{f}{2}+3)(\frac{f}{2}+2)(\frac{f}{2}+1) }  \Delta_{2}^1(F_{x},f)+ \Delta_{6}^1(F_{x},f)\notag \\
=&~-\frac{4}{\sqrt{f\pi}}\exp \left(-\frac{z_{\alpha}^2}{2}\right) \{1+O(f^{-1/2})\}.\notag		
\end{align*}
In summary, \eqref{eq:probdiff1approx} is proved. 

%\bigskip

\subsubsection{Proof of Proposition \ref{prop:infinitesumm} (on Page \pageref{prop:infinitesumm})} \label{sec:pfpropiftydiff1}
We prove Proposition \ref{prop:infinitesumm} based on the notation in Section \ref{sec:notationfinite} and Lemma \ref{lm:order2diff}, which is proved in Section \ref{sec:order2diff} above. 
%Following the notation defined in Section \ref{sec:notationfinite}, 
Particularly, we write the left hand side of  \eqref{eq:probdifforder1h} as $\Delta_{2h}^v(F_x,f)$ below.  
By \eqref{eq:probdiff1approx}, we know \eqref{eq:probdifforder1h} holds for $v=1$ and  $h\in \{1,2,3,4\}$.  We next prove \eqref{eq:probdifforder1h} for $v\geq 2$ when $h\in \{1,2,3,4\}$, respectively. 

\medskip
\noindent \textit{(Part I) Proof for $h=1$}. \quad 
When $v=2$, by \eqref{eq:probdiff1exact}, we have
\begin{align*}
\Delta_{2}^2(F_{x}, f)= -\frac{1}{\Gamma(\frac{f}{2}+2)}\left(\frac{x}{2}\right)^{\frac{f}{2}+1}e^{-x/2} + \frac{1}{\Gamma(\frac{f}{2}+1)}\left(\frac{x}{2}\right)^{\frac{f}{2}}e^{-x/2}.	
\end{align*}
%By the recursion property of the gamma function in \eqref{eq:recursgamma},
Then we can write  $\Delta_{2}^2(F_{x}, f)= A_1(f)Q_1(f)$, where we define 
\begin{align}
	Q_1(f)=\Delta_2^1(F_x,f), \quad \text{ and } \quad A_1(f)=x/(f+2)-1. \label{eq:a1q1def}
\end{align}
%$Q_1(f)=\Delta_2^1(F_x,f)$ and $A_1(f)=x/(f+2)-1$.
Note that $Q_1(f)=O(f^{-1/2})$ by \eqref{eq:probdiff1approx}, and $A_1(f)=O(f^{-1/2})$ by \eqref{eq:xchisqorder} when $x=\chi_f^2(\alpha)$. 
%as   $x=f+\sqrt{2f}\{z_{\alpha}+O(f^{-1/2})\}$. 
Therefore, \eqref{eq:probdifforder1h}  holds for $h=1$ and $v=2$. 
%For easy presentation, we define $Q_1(f)=\Delta_2^1(F_x,f)$ and $A_1(f)=1-x/(f+2)$. By the recursion of property of gamma function, i.e., $\Gamma(m+1)=m\Gamma(m)$, we have
%\begin{align}
%	\Delta_{2}^2(F_{x}, f)= \Delta_{2}^1(Q_1, f)= A_1(f)Q_1(f). \label{eq:qfinitedif1}
%\end{align}

We next prove \eqref{eq:probdifforder1h} for $h=1$ and $v> 2$ by the mathematical induction. 
Assume that there exists some constant $C$ such that uniformly for integers $1\leq k\leq v-1$, 
%Assume uniformly for integers $1\leq k\leq v-1$, we have for some constant $C$, 
\begin{align*}
\Delta^k_2(F_{x}, f)	 = O( k! C^k  f^{-k/2}), 
\end{align*}
that is, uniformly for integers $1 \leq k\leq v-1$, 
\begin{align}
	\Delta_2^{k-1}(Q_1, f)=O( k! C^k f^{-k/2}). \label{eq:qformkm1} 
\end{align}
We next prove $\Delta^v_2(F_{x},f) =  O(v!C^vf^{-v/2})$.
By the definition of $Q_1(f)$ and $A_1(f)$, we have
\begin{align*}
\Delta^v_2(F_{x},f) = \Delta^{v-1}_2(Q_1, f)=\Delta^{v-2}_2(A_1Q_1, f). 
\end{align*}
By Lemma \ref{lm:Leibnizfinite}, 
\begin{align}
\Delta^{v-2}_2(A_1Q_1, f)=\sum_{w=0}^{v-2}\binom{v-2}{w} \Delta_2^w (A_1,f)\Delta_2^{v-2-w}(Q_1, {f+2w}). \label{eq:prodqalebniz}
\end{align}
To evaluate \eqref{eq:prodqalebniz}, by \eqref{eq:qformkm1}, for $0\leq w\leq v-2$, we have
$$\Delta_2^{v-2-w}(Q_1, {f+2w})=O\big\{(v-w-1)!C^{v-w-1}f^{-(v-w-1)/2}\big\}.$$ 
In addition, to evaluate $\Delta_2^w (A_1,f)$ in \eqref{eq:prodqalebniz}, we use the following Lemma \ref{lm:a1order}. 
\begin{lemma} \label{lm:a1order}
When $x=\chi_f^2(\alpha)$ and $f\to \infty$, $A_1(f)=\sqrt{2}z_{\alpha}f^{-1/2}\{1+O(f^{-1})\},$
%\begin{align}
%	A_1(f)=\sqrt{2}z_{\alpha}f^{-1/2}\{1+O(f^{-1})\}. \label{eq:a1forderf}
%\end{align}
and for any integer $w\geq 1$,  
\begin{align}
\Delta_2^w (A_1, f)= 	x\times (-1)^{w}2^ww!\frac{1}{\prod_{k=1}^{w+1}(f+2k)}.  \label{eq:a1korderdif}
\end{align} 
%which is of order $O(k! f^{-k})$ as $f\to \infty$, and is of order $O(2^k/k)$ as $k\to \infty$. 	
Thus there exists a constant $C$ such that \eqref{eq:a1korderdif} is of the order of $O(w! C^w f^{-w})$ as $f\to \infty$ uniformly for $w\geq 1$. 	
\end{lemma}
\begin{proof}
Please see Section \ref{sec:a1order} on Page \pageref{sec:a1order}.	
\end{proof}
\medskip

\noindent By Lemma \ref{lm:a1order}, 
%$A_1(f)=O(f^{-1/2})$ and $\Delta_2^w (A_1,f)=O(w!C^wf^{-w})$ for $w\geq 1.$
%\begin{align*}
%A_1(f)=O(f^{-1/2}), \quad \text{and} \quad	\Delta_2^w (A_1,f)=w!C^wO(f^{-w}) \ \text{ for } \ w\geq 1. 
%\end{align*}
%Therefore, 
\eqref{eq:prodqalebniz} gives that as $f\to \infty$, 
\begin{align}
&~ \Delta^{v-2}_2(A_1Q_1, f)\label{eq:a1q1vdiff2}\\
 =&~ O(f^{-1/2})\times O\big\{(v-1)!C^{v-1}f^{-(v-1)/2}\big\}  \notag \\
 &~ +	\sum_{w=1}^{v-2}\binom{v-2}{w}O\big( w! C^wf^{-w}\big) \times O\big\{(v-w-1)!C^{v-w-1}f^{-(v-w-1)/2}\big\}\notag \\
  =&~(v-1)! C^{v-1}O(f^{-v/2}) + \sum_{w=1}^{v-2} (v-2)!(v-w-1)C^{v-1}O\{f^{-(w-1)/2}\times f^{-v/2}\}  \notag \\
  =&~ O(v!C^v f^{-v/2}),\notag
\end{align}
where in the last equation, we use $v-w-1\leq v-1$ and $O\{f^{-(w-1)/2}\}=O(1)$ when $w\geq 1$. 
We note that there exists a constant $C$ such that the last equation in \eqref{eq:a1q1vdiff2} holds uniformly for $v\geq 1$. 
%the derived order also holds uniformly for $v\geq 1$. 
In summary, we obtain \eqref{eq:probdifforder1h} for $h=1$.  

\bigskip

\noindent \textit{(Part II) Proof for $h=2$}. \quad 
%By \eqref{eq:probdiff1exact}, 
%\begin{align}
%	\Delta_{4}^1(F_{x}, f) =&~ A_2(f)Q_1(f)+ \Delta_{2}^1(F_{x}, f), \label{eq:linearformvh24}
%\end{align}
%where $A_2(f)=x/(f+2)$.  
%Then by \eqref{eq:linearformvh24} and the linearity  rule of  finite difference, we have
%\begin{align*}
%	\Delta_{4}^v(F_{x}, f) =&~\Delta_{4}^{v-1} (A_2Q_1, f)+ \Delta_{2}^v(F_{x},f).
%\end{align*}
%By \eqref{eq:probdiff1exact}, there exists a constant $C$ such that $\Delta_{2}^v(F_{x},f)= O(v!C^vf^{-v/2})$.
%Therefore, to prove 
% \eqref{eq:probdifforder1h} for $h=2$,  it suffices to prove 
%% $\Delta_{4}^{v-1} (Q_2, f)=O(v!C^vf^{-v/2})$. 
%\begin{align}
%	\Delta_{4}^{v-1} (A_2Q_1,f)= v!C^v O(f^{-v/2}).  \label{eq:q2difforder}
%\end{align}
%By Lemma \ref{lm:Leibnizfinite},  
%\begin{align*}
%\Delta_{4}^{v-1} (A_2Q_1,f)=\sum_{w=0}^{v-1} \binom{v-1}{w} \Delta^{w}_4(A_2) 	
%\end{align*}
By \eqref{eq:probdiff1exact}, \eqref{eq:probdiff1exacth2} and \eqref{eq:a1q1def}, 
\begin{align}
	\Delta_{4}^1(F_{x}, f) =&~ Q_2(f)+ Q_1(f), \label{eq:linearformvh24}
\end{align}
where we define 
\begin{align*}
Q_2(f)=- \frac{1}{\Gamma(\frac{f}{2}+2)}\left(\frac{x}{2}\right)^{\frac{f}{2}+1}e^{-x/2}.	
\end{align*}
Then by \eqref{eq:linearformvh24} and Lemma \ref{lm:linearityrulefinite}, we have
\begin{align*}
	\Delta_{4}^v(F_{x}, f) =&~\Delta_{4}^{v-1} (Q_2, f)+ \Delta_{4}^{v-1}(Q_1,f).
\end{align*}
%By \eqref{eq:probdiff1exact}, 
%\begin{align}
%	\Delta_{4}^1(F_{x}, f) =&~ Q_2(f)+ \Delta_{2}^1(F_{x}, f), \label{eq:linearformvh24}
%\end{align}
%where we define
%\begin{align*}
%Q_2(f)=- \frac{1}{\Gamma(\frac{f}{2}+2)}\left(\frac{x}{2}\right)^{\frac{f}{2}+1}e^{-x/2}.	
%\end{align*}
%Then by \eqref{eq:linearformvh24} and Lemma \ref{lm:linearityrulefinite}, we have
%%the linearity  rule of  finite difference, we have
%\begin{align*}
%	\Delta_{4}^v(F_{x}, f) =&~\Delta_{4}^{v-1} (Q_2, f)+ \Delta_{2}^v(F_{x},f).
%\end{align*}
%By \eqref{eq:probdiff1exact}, there exists a constant $C$ such that $\Delta_{2}^v(F_{x},f)= O(v!C^vf^{-v/2})$.
Therefore, to prove 
 \eqref{eq:probdifforder1h} for $h=2$,  it suffices to prove 
% $\Delta_{4}^{v-1} (Q_2, f)=O(v!C^vf^{-v/2})$. 
%\begin{align}
%\Delta_{4}^{v-1} (Q_r,f)= &~ O(v!C^vf^{-v/2}).  \label{eq:q2difforder}	
%\end{align}
\begin{align}
	\Delta_{4}^{v-1} (Q_1,f)= &~ O(v!C^vf^{-v/2}), \notag \\
	\Delta_{4}^{v-1} (Q_2,f)= &~ O(v!C^vf^{-v/2}).  \label{eq:q2difforder}
\end{align}

As $Q_1(f)=Q_2(f-2)$, it suffices to prove \eqref{eq:q2difforder}, and we next use the mathematical induction. 
%We next prove \eqref{eq:q2difforder}  by the mathematical induction.
Note that \eqref{eq:q2difforder} holds for $v=1$ since $\Delta_{4}^{0} (Q_2, f) =Q_2(f)=O(f^{-1/2})$ by the proof of \eqref{eq:probdiff1approx}. 
%% Lemma \ref{lm:difforder1}. 
% In addition, 
%Particularly, we first show \eqref{eq:q2difforder} holds for $v=1$ and $2$. 
In addition, for $v=2$, we have
\begin{align}
	\Delta_{4}^{1} (Q_2, f) =&~ Q_2(f+4)- Q_2(f) = A_2(f)Q_2(f), \label{eq:a4diffq2prod}
\end{align} 
where
\begin{align*}
A_2(f) = \frac{(\frac{x}{2})^2}{(\frac{f}{2}+3)(\frac{f}{2}+2)} -1.
\end{align*}
Note that $Q_2(f)=O(f^{-1/2})$, and when $x=\chi^2_f(\alpha)$, we have 
%Note that $Q_2(f)=O(f^{-1/2})$, and 
$A_2(f)=O(f^{-1/2})$ by \eqref{eq:xchisqorder}. 
%as $x=f+\sqrt{2f}\{z_{\alpha} + O(f^{-1/2})\}$.
Therefore, $\Delta_{4}^{1} (Q_2, f)=O(f^{-1})$, i.e.,  \eqref{eq:q2difforder} holds for  $v=2$. 
For $v\geq 3$, we next use the mathematical induction, where we assume for integers $0\leq w \leq v-2$, 
\begin{align}
	\Delta_{4}^{w} (Q_2,f)= O\{(w+1)!C^{w+1} f^{-(w+1)/2}\},\label{eq:q2difforderh4}
\end{align}
and prove \eqref{eq:q2difforder}. 
By \eqref{eq:a4diffq2prod}, $\Delta_{4}^{v-1} (Q_2,f)= \Delta_{4}^{v-2}(A_2Q_2,f)$.
Then by Lemma \ref{lm:Leibnizfinite}, 
\begin{align}
 \Delta_{4}^{v-2} (A_2Q_2,f) =\sum_{w=0}^{v-2}\binom{v-2}{w} \Delta_4^w(A_2,f)	\Delta_4^{v-2-w}(Q_2,f+4w). \label{eq:a2q2orderh4diff}	
 \end{align}
We next prove \eqref{eq:q2difforder} by \eqref{eq:q2difforderh4}, \eqref{eq:a2q2orderh4diff}	 and the following Lemma \ref{lm:a2orderf}. 

\medskip 
\begin{lemma} \label{lm:a2orderf}
When $x=\chi_f^2(\alpha)$, 
$A_2(f)=2\sqrt{2}z_{\alpha}f^{-1/2} \{1+O(f^{-1/2})\}$. 
%\begin{align}
%	A_2(f)=-2\sqrt{2}z_{\alpha}f^{-1/2} \{1+O(f^{-1})\}. \label{eq:a2orderf}
%\end{align}
Moreover, there exists a constant $C$ such that uniformly for any integer $w\geq 1$,
\begin{align}
	\Delta_4^w(A_2,f)= O\biggr\{(w+1)! C^w \prod_{t=1}^w(f+2t)^{-1} \biggr\}. \label{eq:ordera2fh4diff}
\end{align} 
%\begin{align}
%	\Delta_4^w(A_2,f)= O(w! C^wf^{-w}). \label{eq:ordera2fh4diff}
%\end{align} 
%uniformly for $w\geq 1$. 
\end{lemma} 
\begin{proof}
Please see Section \ref{sec:a2orderf} on Page  \pageref{sec:a2orderf}.
\end{proof}

\medskip

%\noindent By Lemma \ref{lm:a2orderf} and \eqref{eq:q2difforderh4}, 
%\eqref{eq:q2difforder} can be proved similarly  to \eqref{eq:a1q1vdiff2}. 

\noindent By Lemma \ref{lm:a2orderf} and \eqref{eq:q2difforderh4},  we have
 \begin{align}
 	\eqref{eq:a2q2orderh4diff}=& O(f^{-1/2})\times O\big\{(v-1)!C^{v-1}f^{-(v-1)/2}\big\}\notag \\
& + \sum_{w=1}^{v-2} \binom{v-2}{w} O\biggr\{ (w+1)!C^w\prod_{t=1}^w(f+2t)^{-1}(v-1-w)! C^{(v-1-w)} f^{-\frac{v-1-w}{2}}\biggr\} \notag \\
=& O\big\{(v-1)!C^{v-1}f^{-\frac{v}{2}}\big\}+\sum_{w=1}^{v-2}O\big\{(v-2)!(v-1-w)C^{v-1} f^{-\frac{v}{2}}\big\}\frac{(w+1)f^{\frac{w+1}{2}}}{\prod_{t=1}^{w}(f+2t)}. \label{eq:a2q2orderh4diff2} 
%=&~ O(v!C^vf^{-v/2}).
 \end{align}
%and $w  - (w+1)/2 > 1$ when $w$ is sufficiently large, 
To evaluate \eqref{eq:a2q2orderh4diff2}, 
we note that when $w=1$ and $2$, $(w+1)f^{(w+1)/2}\{\prod_{t=1}^{w}(f+2t)\}^{-1}=O(f^{(1-w)/2})$; 
when $w\geq 3$, as $f\to \infty$, 
\begin{align*}
	\frac{(w+1)f^{(w+1)/2}}{\prod_{t=1}^{w}(f+2t)}\leq \frac{w+1}{2w}\frac{f^{(w+1)/2}}{f^{w-1}} = O(1)
\end{align*} 
uniformly over $w\geq 3$.
Moreover, by $\sum_{w=1}^{v-2}(v-2)!(v-1-w)\leq v!$, 
we obtain $\eqref{eq:a2q2orderh4diff}=O(v!C^vf^{-v/2}).$

%\begin{align*}
%\eqref{eq:a2q2orderh4diff}	=	
%\end{align*}
% \begin{align*}
% 	\eqref{eq:a2q2orderh4diff}	=&~ O(f^{-1/2})\times O\big\{(v-1)!C^{v-1}f^{-(v-1)/2}\big\}\notag \\
%&~ + \sum_{w=1}^{v-2} \binom{v-2}{w} O\big(w!C^wf^{-w}\big)\times O\big\{(v-1-w)! C^{(v-1-w)} f^{-(v-1-w)/2}\big\} \notag \\
%=&~ O\big\{(v-1)!C^{v-1}f^{-v/2}\big\}+\sum_{w=1}^{v-2}O\big\{(v-2)!(v-1-w)C^{v-1} f^{-(v+w-1)/2}\big\}\notag \\
%=&~ O(v!C^vf^{-v/2}).
% \end{align*}
%In summary, \eqref{eq:q2difforder} is proved. 
  
 \bigskip
 
\noindent \textit{(Part III) Proof for $h=3$}. \quad 
 By \eqref{eq:probdiff1exact}, 
\begin{align}
	\Delta_{6}^1(F_{x}, f) =&~ Q_3(f)+ Q_2(f)+Q_1(f), \label{eq:linearformvh26}
\end{align}
where we define
\begin{align*}
	Q_3(f)=- \frac{1}{\Gamma(\frac{f}{2}+3)}\left(\frac{x}{2}\right)^{\frac{f}{2}+2}e^{-x/2}.	
\end{align*}
Then by \eqref{eq:linearformvh26} and Lemma \ref{lm:linearityrulefinite}, 
\begin{align*}
\Delta_6^v(F_x,f)= \Delta_6^{v-1}(Q_3,f) + 	 \Delta_6^{v-1}(Q_2,f) +  \Delta_6^{v-1}(Q_1,f). 
\end{align*}
Since $Q_2(f)=Q_3(f-2)$ and $Q_1(f)=Q_3(f-4)$, it suffices to prove 
 \begin{align}
 	 \Delta_6^{v-1}(Q_3,f)=O(v!C^vf^{-v/2}). \label{eq:q3difforder}
 \end{align}

% By \eqref{eq:probdiff1exact}, 
%\begin{align}
%	\Delta_{6}^1(F_{x}, f) =&~ Q_3(f)+ \Delta_{4}^1(F_{x}, f), \label{eq:linearformvh26}
%\end{align}
%where we define
%\begin{align*}
%	Q_3(f)=- \frac{1}{\Gamma(\frac{f}{2}+3)}\left(\frac{x}{2}\right)^{\frac{f}{2}+2}e^{-x/2}.	
%\end{align*}
%Then by \eqref{eq:linearformvh26} and Lemma \ref{lm:linearityrulefinite}, 
%\begin{align*}
%	\Delta_6^v(F_x,f)= \Delta_6^{v-1}(Q_3,f) + \Delta_4^v(F_x,f). 
%\end{align*}
%Since we already obtain  $\Delta_4^v(F_x,f)= O(v!C^vf^{-v/2})$ above, it remains to prove 
% \begin{align}
% 	 \Delta_6^{v-1}(Q_3,f)=O(v!C^vf^{-v/2}). \label{eq:q3difforder}
% \end{align}
 
We next prove \eqref{eq:q3difforder} by the mathematical induction. Note that \eqref{eq:q3difforder} holds for $v=1$ since $ \Delta_6^{0}(Q_3,f)=Q_3(f)=O(f^{-1/2})$ by the proof of \eqref{eq:probdiff1approx}  in Section \ref{sec:order2diff}. In addition, for $v=2$, 
%we have
\begin{align}
\Delta_6^1(Q_3,f)=Q_3(f+6)-Q_3(f)=A_3(f)Q_3(f),	\label{eq:a3diffq3prod}
\end{align} 
where
\begin{align*}
	A_3(f)= \prod_{k=1}^3 A_{3,k}(f) -1, \quad \quad A_{3,k}(f)=\frac{x}{f+4+2k}.
\end{align*} 
Note that $A_3(f)=O(f^{-1/2})$ when $x=\chi^2_f(\alpha)$ by \eqref{eq:xchisqorder}. 
%as $x=f+\sqrt{2f}\{z_{\alpha} + O(f^{-1/2})\}$.
Moreover, as $Q_3(f)=O(f^{-1/2})$, $\Delta_6^1(Q_3,f)=O(f^{-1})$, i.e., \eqref{eq:q3difforder} holds for $v=2$. 
For $v\geq 3$, 
we next use the mathematical induction, where we assume for integers $0\leq w\leq v-2$, 
\begin{align}
	\Delta_{6}^{w} (Q_3,f)= O\{(w+1)!C^{w+1} f^{-(w+1)/2}\},\label{eq:q3difforderh6}
\end{align}
and prove \eqref{eq:q3difforder}. 
By \eqref{eq:a3diffq3prod}, $\Delta_6^{v-1}(Q_3,f)=\Delta_6^{v-2}(A_3Q_3,f).$ 
Then by Lemma \ref{lm:Leibnizfinite}, 
\begin{align}
 \Delta_{6}^{v-2} (A_3Q_3,f) =\sum_{w=0}^{v-2}\binom{v-2}{w} \Delta_6^w(A_3,f)	\Delta_6^{v-2-w}(Q_3,f+6w). \label{eq:a3q3orderh6diff}	
 \end{align}
We next prove \eqref{eq:a3q3orderh6diff}	 by \eqref{eq:q3difforderh6}  and the following Lemma \ref{lm:a3orderf}. 
 
\begin{lemma} \label{lm:a3orderf}
When $x=\chi_f^2(\alpha)$, $A_3(f)=3\sqrt{2}z_{\alpha}f^{-1/2} \{1+O(f^{-1/2})\}$. 
Moreover, there exists a constant $C$ such that uniformly for any integer $w\geq 1$,
\begin{align*}
\Delta_6^w(A_3,f)=&~O\biggr\{(w+2)!C^w \prod_{t=1}^{w}(f+2t)^{-1}\biggr\}. 
\end{align*}
 %\begin{align*}
%\Delta_6^w(A_3,f)=O(w!C^wf^{-w}). 	
%\end{align*}
\end{lemma}
\begin{proof}
Please see Section \ref{sec:a3orderf}	on Page \pageref{sec:a3orderf}. 
\end{proof} 
\medskip

\noindent Then by \eqref{eq:q3difforderh6} and  Lemma \ref{lm:a3orderf}, 
\begin{align*}
\eqref{eq:a3q3orderh6diff}=&~~ O(f^{-1/2})\times O\big\{(v-1)!C^{v-1}f^{-(v-1)/2}\big\}\notag \\
&~+\sum_{w=1}^{v-2}\binom{v-2}{w}O\biggr\{ (w+2)!C^w \prod_{t=1}^{w}(f+2t)^{-1} 	(v-1-w)!C^{v-1-w}f^{-(v-1-w)/2}\biggr\} \notag \\
=&~ O\big\{(v-1)!C^{v-1}f^{-v/2}\big\}+\sum_{w=1}^{v-2}O\big\{(v-1)!C^vf^{-v/2}\big\}\frac{(w+2)(w+1)f^{(w+1)/2}}{\prod_{t=1}^{w}(f+2t)}.
\end{align*}
Note that when $w\leq 4$, $(w+2)(w+1)f^{(w+1)/2} \prod_{t=1}^{w}(f+2t)^{-1}=O\{f^{(1-w)/2}\}$; 
when $w\geq 5$,
\begin{align*}
 \frac{(w+2)(w+1)f^{(w+1)/2}}{\prod_{t=1}^{w}(f+2t)} \leq \frac{(w+2)(w+1)}{w(w-1)}f^{(5-w)/2} = O(1)
\end{align*} as $f\to \infty$ uniformly over $w\geq 5.$
It follows that $\eqref{eq:a3q3orderh6diff}=O(v!C^vf^{-v})$ and thus \eqref{eq:q3difforder} is proved. 

%\noindent It folllows that \eqref{eq:q3difforder} can be proved by  \eqref{eq:q3difforderh6} and  Lemma \ref{lm:a3orderf} similarly  to \eqref{eq:a1q1vdiff2}.

%Parts (I) and (II) above. 
%Combining  \eqref{eq:q3difforderh6}   Lemma \ref{lm:a3orderf}, the foll 
%\bigskip
%\noindent \textit{(Part III) Proof for $h=3$}. \quad  

\bigskip

\noindent \textit{(Part IV) Proof for $h=4$}. \quad 
 By \eqref{eq:probdiff1exact}, 
\begin{align}
	\Delta_{8}^1(F_{x}, f) =&~ Q_4(f)+Q_3(f)+Q_2(f)+Q_1(f), \label{eq:linearformvh28}
\end{align}
where we define
\begin{align*}
	Q_4(f)=-\frac{1}{\Gamma(\frac{f}{2}+4)}\left(\frac{x}{2}\right)^{\frac{f}{2}+3}e^{-x/2}. 
\end{align*}
Then by \eqref{eq:linearformvh28} and Lemma \ref{lm:Leibnizfinite}, 
\begin{align*}
	\Delta_8^v(F_x,f)=	\Delta_8^{v-1}(Q_4,f) + 	\Delta_8^{v-1}(Q_3,f) + 	\Delta_8^{v-1}(Q_2,f) + 	\Delta_8^{v-1}(Q_1,f). 
\end{align*}
Since $Q_3(f)=Q_4(f-2)$, $Q_2(f)=Q_4(f-4)$, and $Q_1(f)=Q_4(f-6)$, it suffices to prove 
\begin{align}
 	 \Delta_8^{v-1}(Q_4,f)=O(v!C^vf^{-v/2}). \label{eq:q4difforder}
 \end{align}

% By \eqref{eq:probdiff1exact}, 
%\begin{align}
%	\Delta_{8}^1(F_{x}, f) =&~ Q_4(f)+ \Delta_{6}^1(F_{x}, f), \label{eq:linearformvh28}
%\end{align}
%where we define
%\begin{align*}
%	Q_4(f)=-\frac{1}{\Gamma(\frac{f}{2}+4)}\left(\frac{x}{2}\right)^{\frac{f}{2}+3}e^{-x/2}. 
%\end{align*}
%Then by \eqref{eq:linearformvh28} and Lemma \ref{lm:Leibnizfinite}, 
%\begin{align*}
%	\Delta_8^v(F_x,f)=	\Delta_8^{v-1}(Q_4,f) + \Delta_6^v(F_x,f). 
%\end{align*}
%Since we already obtain $ \Delta_6^v(F_x,f)=O(v!C^vf^{-v/2})$ in Part (III) above, it remains to prove 
%\begin{align}
% 	 \Delta_8^{v-1}(Q_4,f)=O(v!C^vf^{-v/2}). \label{eq:q4difforder}
% \end{align}

 We next prove \eqref{eq:q4difforder} by the mathematical induction. Note that \eqref{eq:q4difforder} holds for $v=1$ since $ \Delta_8^{0}(Q_4,f)=Q_4(f)=O(f^{-1/2})$ by the proof of \eqref{eq:probdiff1approx} in Section \ref{sec:order2diff}. In addition, for $v=2$, we have
\begin{align}
\Delta_8^1(Q_4,f)=Q_4(f+8)-Q_4(f)=A_4(f)Q_4(f),	\label{eq:a4diffq4prod}
\end{align} 
where
\begin{align*}
	A_4(f)= \prod_{k=1}^4 A_{4,k}(f) -1, \quad \quad A_{4,k}(f)=\frac{x}{f+6+2k}.
\end{align*} 
Note that $A_4(f)=O(f^{-1/2})$ as $x=f+\sqrt{2f}\{z_{\alpha} + O(f^{-1/2})\}$.
Moreover, as $Q_4(f)=O(f^{-1/2})$, $\Delta_8^1(Q_4,f)=O(f^{-1})$, i.e., \eqref{eq:q4difforder} holds for $v=2$. 
For $v\geq 3$, 
we next use the mathematical induction, where we assume for integers $0\leq w\leq v-2$, 
\begin{align}
	\Delta_{8}^{w} (Q_4,f)= O\{(w+1)!C^{w+1} f^{-(w+1)/2}\},\label{eq:q4difforderh8}
\end{align}
and prove \eqref{eq:q4difforder}. 
By \eqref{eq:a4diffq4prod}, $\Delta_8^{v-1}(Q_4,f)=\Delta_8^{v-2}(A_4Q_4,f).$ 
Then by Lemma \ref{lm:Leibnizfinite}, 
\begin{align}
 \Delta_{8}^{v-2} (A_4Q_4,f) =\sum_{w=0}^{v-2}\binom{v-2}{w} \Delta_8^w(A_4,f)	\Delta_8^{v-2-w}(Q_4,f+8w). \label{eq:a4q4orderh8diff}	
 \end{align}
We next prove \eqref{eq:a4q4orderh8diff}	 by \eqref{eq:q4difforderh8}, \eqref{eq:a4q4orderh8diff}	  and the following Lemma \ref{lm:a4orderf}. 
 
\begin{lemma} \label{lm:a4orderf}
When $x=\chi_f^2(\alpha)$, $A_4(f)=4\sqrt{2}z_{\alpha}f^{-1/2} \{1+O(f^{-1/2})\}$. 
Moreover, there exists a constant $C$ such that as $f\to \infty$,
\begin{align*}
\Delta_8^w(A_4,f)=&~O\biggr\{(w+3)!C^w \prod_{t=1}^{w}(f+2t)^{-1}\biggr\}  
\end{align*}
holds uniformly for any integer $w\geq 1$
 %\begin{align*}
%\Delta_6^w(A_3,f)=O(w!C^wf^{-w}). 	
%\end{align*}
\end{lemma}
\begin{proof}
Please see Section \ref{sec:a4orderf}	on Page \pageref{sec:a4orderf}. 
\end{proof} 
\medskip

\noindent Then by \eqref{eq:q4difforderh8} and  Lemma \ref{lm:a4orderf}, 
\begin{align*}
\eqref{eq:a4q4orderh8diff}=&~~ O(f^{-1/2})\times O\big\{(v-1)!C^{v-1}f^{-(v-1)/2}\big\}\notag \\
&~+\sum_{w=1}^{v-2}\binom{v-2}{w}O\biggr\{ (w+3)!C^w \prod_{t=1}^{w}(f+2t)^{-1} 	(v-1-w)!C^{v-1-w}f^{-(v-1-w)/2}\biggr\} \notag \\
=&~ O\big\{(v-1)!C^{v-1}f^{-v/2}\big\}+\sum_{w=1}^{v-2}O\big\{(v-1)!C^vf^{-v/2}\big\}\frac{(w+3)(w+2)(w+1)f^{\frac{w+1}{2}}}{\prod_{t=1}^{w}(f+2t)}. 
\end{align*}
Note that when $w\leq 6$, $(w+3)(w+2)(w+1)f^{(w+1)/2} \prod_{t=1}^{w}(f+2t)^{-1}=O\{f^{(1-w)/2}\}$; 
when $w\geq 7$, as $f\to \infty$, 
\begin{align*}
 \frac{(w+3)(w+2)(w+1)f^{(w+1)/2}}{\prod_{t=1}^{w}(f+2t)} \leq \frac{(w+3)(w+2)(w+1)}{w(w-1)(w-2)}f^{(7-w)/2} = O(1)
\end{align*} 
holds uniformly over $w\geq 7.$
It follows that $\eqref{eq:a4q4orderh8diff}=O(v!C^vf^{-v})$ and thus \eqref{eq:q4difforder} is proved. 

%\medskip

\subsubsection{Proof of Proposition \ref{prop:infinitesumm2} (on Page \pageref{prop:infinitesumm2})} \label{sec:pfpropiftydiff2}
Similar to the proof of Proposition  \ref{prop:infinitesumm} in Section \ref{sec:pfpropiftydiff2}, 
we prove Proposition  \ref{prop:infinitesumm2} using the notation in Section \ref{sec:notationfinite} and Lemma \ref{lm:order2diff}. 
We next discuss $(h_1,h_2)=(1,2)$ and $(h_1,h_2)=(2,3)$ in \textit{(Part I)} and  \textit{(Part II)} below, respectively. 

%The proof of Proposition  \ref{prop:infinitesumm2} is similar to the proof of Proposition  \ref{prop:infinitesumm} in Section \ref{sec:pfpropiftydiff2}. 

\medskip

\noindent \textit{(Part I) Proof for $h_1=1$ and $h_2=2$}. \quad  
Based on the notation in Section \ref{sec:notationfinite}, it is equivalent to prove that there exists some constant $C$ such that when  $x=\chi_f^2(\alpha)$, as $f\to \infty$, 
\begin{align}
	\Delta^{v_2}_4\Delta^{v_1}_2(F_{x}, f) = O\big\{v_1!v_2! C^{v_1+v_2} f^{-(v_1+v_2)/2}\big\},  \label{eq:lfinitediffh2f}
\end{align}  uniformly for integers $v_1, v_2 \geq 1$.

When $v_1=0$ or $v_2=0$, \eqref{eq:lfinitediffh2f} holds by Proposition \ref{prop:infinitesumm}. 
%the conclusion holds by \eqref{eq:probdifforder1h} .
 When $v_1=v_2=1$, by \eqref{eq:probdiff1exacth2}, we have
\begin{align*}
\Delta^1_4\Delta^1_2(F_x,f)= &~	 - \frac{1}{\Gamma(\frac{f}{2}+3)}\left(\frac{x}{2}\right)^{\frac{f}{2}+2}e^{-x/2}+ \frac{1}{\Gamma(\frac{f}{2}+1)}\left(\frac{x}{2}\right)^{\frac{f}{2}}e^{-x/2}=D_{2,4}(f)\Delta^1_2(F_x,f), 
\end{align*}
where
\begin{align}
	D_{2,4}(f) =  \frac{x^2}{(f+4)(f+2)} -1.\label{eq:a3def}
\end{align}
As $D_{2,4}(f)=O(f^{-1/2})$ and $\Delta_2^1(F_x,f)=O(f^{-1/2})$, \eqref{eq:lfinitediffh2f} holds for $v_1=v_2=1$.  
%Recall that $Q_1(f)=\Delta^1_2(F_x,f)$. 
%By the definition of forward difference operator, we can write
%\begin{align*}
%\Delta^{v_2}_4\Delta^{v_1}_2(F_x,f)=&~\Delta^{v_2}_4\Delta^{v_1-1}_2(Q_1,f).	\notag 	
%\end{align*}
We next prove \eqref{eq:lfinitediffh2f} by the mathematical induction. 
Particularly, 
we assume for integers $s_1\leq v_1$ and $s_2\leq v_2$, 
\begin{align}
	\Delta^{s_2}_4\Delta^{s_1}_2(F_x,f)=O\big\{s_1!s_2!C^{s_1+s_2}f^{-(s_1+s_2)/2}\big\},\label{eq:mathinducs1s2cond}
\end{align}
%We next prove 
and prove that \eqref{eq:mathinducs1s2cond} also holds for $(s_1,s_2)=(v_1+1,v_2)$ and $(s_1,s_2)=(v_1, v_2+1)$, i.e., $\Delta^{v_2}_4\Delta^{v_1+1}_2(F_x,f)$ and $\Delta^{v_2+1}_4\Delta^{v_1}_2(F_x,f)$, respectively. 
%Recall that $Q_1(f)=\Delta^1_2(F_x,f)$, and then 
%%By the definition of forward difference operator, we can write
%\begin{align*}
%\Delta^{v_2}_4\Delta^{v_1}_2(F_x,f)=&~\Delta^{v_2}_4\Delta^{v_1-1}_2(Q_1,f).	\notag 	
%\end{align*}

\bigskip
%\noindent \textit{Step 1. $\Delta^{v_2}_4\Delta^{v_1}_2(Q_1,f)$.} \quad 

\noindent \textit{Step I.1. $\Delta^{v_2}_4\Delta^{v_1+1}_2(F_x,f)$.} \quad 
Recall that we define $Q_1(f)=\Delta^1_2(F_x,f)$. 
It follows that 
\eqref{eq:mathinducs1s2cond} gives that for integers $ s_1\leq v_1-1$ and $ s_2\leq v_2$
\begin{align}
\Delta^{s_2}_4\Delta^{s_1}_2(Q_1,f)=	O\big\{(s_1+1)! s_2!C^{s_1+s_2+1} f^{-(s_1+s_2+1)/2}\big\}. \label{eq:assumes1s2q}
\end{align}
It is then equivalent to prove that \eqref{eq:assumes1s2q} holds for $(s_1,s_2)=(v_1,v_2)$, i.e., $\Delta^{v_2}_4\Delta^{v_1}_2(Q_1,f)$.
% and $\Delta^{v_2+1}_4\Delta^{v_1-1}_2(Q_1,f)$, respectively. 
%Recall that $\Delta^{1}_2(Q_1,f)=A_1(f)Q_1(f)$, 
%where $A_1(f)$ and $Q_1(f)$ are defined in \eqref{eq:a1q1def}.
%Then by  Lemmas \ref{lm:Leibnizfinite} and \ref{lm:linearityrulefinite},  
By $\Delta^{1}_2(Q_1,f)=A_1(f)Q_1(f)$, 
(see the definitions in  \eqref{eq:a1q1def}), 
%where $A_1(f)$ and $Q_1(f)$ are defined in \eqref{eq:a1q1def},
and Lemmas \ref{lm:Leibnizfinite} and \ref{lm:linearityrulefinite},  
\begin{align}
&~\Delta^{v_2}_4\Delta^{v_1}_2(Q_1,f)\notag \\
=&~	 \sum_{w_1=0}^{v_1-1}\binom{v_1-1}{w_1}\Delta^{v_2}_4\Big\{  \Delta_2^{w_1}(A_1,f) \Delta_2^{v_1-1-w_1}(Q_1, f+2w_1) \Big\} \notag \\
=&~		\sum_{w_1=0}^{v_1-1}\binom{v_1-1}{w_1} \sum_{w_2=0}^{v_2} \binom{v_2}{w_2} \Delta_4^{w_2}\Delta_2^{w_1}(A_1,f)  \Delta_4^{v_2-w_2} \Delta_2^{v_1-1-w_1}(Q_1, f+2w_1+4w_2). \label{eq:v1v2updateform}
\end{align}
To evaluate \eqref{eq:v1v2updateform}, we use the following Lemma \ref{lm:doublediffa1}. 

\begin{lemma}\label{lm:doublediffa1}
For two integers $w_1$ and $w_2$ satisfying $w_1+w_2 \geq 1$, there exists some constant $C$ such that as $f\to \infty$, 
\begin{align*}
\Delta_4^{w_2}\Delta_2^{w_1}(A_1,f)=(w_1+w_2)! O\Biggr(C^{w_1+w_2}\prod_{k=1}^{w_1+w_2}\frac{1}{f+2k}\Biggr)
\end{align*}	
uniformly over $w_1+w_2\geq 1$. 
\end{lemma}
\begin{proof}
Please see Section \ref{sec:doublediffa1} on Page \pageref{sec:doublediffa1}. 
\end{proof}

\medskip
\noindent By Lemma \ref{lm:doublediffa1} and the assumption \eqref{eq:assumes1s2q}, we have
\begin{align*}
\eqref{eq:v1v2updateform}=&~ \sum_{w_1=0}^{v_1-1}\binom{v_1-1}{w_1} \sum_{w_2=0}^{v_2} \binom{v_2}{w_2} (w_1+w_2)! O\Biggr(\prod_{k=1}^{w_1+w_2}\frac{1}{f+2k}\Biggr)C^{v_1+v_2+1} \notag \\
&~ \times (v_2-w_2)!(v_1-w_1)!O\big\{(f+2w_1+4w_2)^{-(v_1-w_1+v_2-w_2)/2}\big\}  \notag \\
=&~ \sum_{w_1=0}^{v_1-1}(v_1-1)!(v_1-w_1) \sum_{w_2=0}^{v_2} v_2! C^{v_1+v_2+1} O\{f^{-(v_1+v_2+1)/2}\} \notag \\
&~ \times \frac{(w_1+w_2)!	}{w_1!w_2!}\prod_{k=1}^{w_1+w_2}\frac{1}{f+2k} \times f^{(w_1+w_2+1)/2}.  \notag
\end{align*}
We next use the following Lemma \ref{lm:evalwqw2}. 
\begin{lemma} \label{lm:evalwqw2}
For integers $w_1$, $w_2$, and $f$, 
\begin{align*}
\frac{(w_1+w_2)!	}{w_1!w_2!}\prod_{k=1}^{w_1+w_2}\frac{1}{f+2k}\times f^{(w_1+w_2+1)/2}=	O\big\{ 2^{-(w_1+w_2-1)/2} \big\}.
\end{align*}	
\end{lemma}
\begin{proof}
Please see Section \ref{sec:evalwqw2} on Page \pageref{sec:evalwqw2}. 
\end{proof}
\medskip

\noindent It follows that by Lemma \ref{lm:evalwqw2},
\begin{align}
\eqref{eq:v1v2updateform}=&~ \sum_{w_1=0}^{v_1-1}(v_1-1)!(v_1-w_1)  \sum_{w_2=0}^{v_2} v_2! \frac{1}{(\sqrt{2})^{w_1+w_2-1}}C^{v_1+v_2+1}O\{f^{-(v_1+v_2+1)/2}\} \notag \\
=&~ O\{v_1! v_2!C^{v_1+v_2+1} f^{-(v_1+v_2+1)/2}\}, \label{eq:h1h2inductorder1} 
\end{align}
which is $ O\{(v_1+1)! v_2! C^{v_1+v_2+1} f^{-(v_1+v_2+1)/2}\}$ as $v_1<v_1+1$. Therefore, we obtain $\Delta^{v_2}_4\Delta^{v_1}_2(Q_1,f)=O\{(v_1+1)! v_2! C^{v_1+v_2+1} f^{-(v_1+v_2+1)/2}\}$.
% is obtained.   

\bigskip
\noindent \textit{Step I.2. $\Delta^{v_2+1}_4\Delta^{v_1}_2(F_x,f)$.} \quad 
%\noindent \textit{Step 2. $\Delta^{v_2+1}_4\Delta^{v_1-1}_2(Q_1,f)$.} \quad 
%Note that 
%Recall that $\Delta^{1}_4(Q_1,f)=Q_2(f)+\Delta_2^1(F_x,f).$
%Then 
By \eqref{eq:linearformvh24},
\begin{align*}
\Delta^{v_2+1}_4\Delta^{v_1}_2(F_x,f)=&~	\Delta^{v_1}_2\Delta^{v_2+1}_4(F_x,f)= \Delta^{v_1}_2 \Delta^{v_2}_4(Q_2,f) +  \Delta^{v_2}_4\Delta^{v_1+1}_2(F_x,f).\notag 
\end{align*}
By \eqref{eq:h1h2inductorder1}, we have $\Delta^{v_2}_4\Delta^{v_1+1}_2(F_x,f)=O\{v_1!v_2!C^{v_1+v_2+1}f^{-(v_1+v_2+1)/2}\}$.  
Therefore, it remains to prove $\Delta^{v_1}_2 \Delta^{v_2}_4(Q_2,f)=O\{v_1!(v_2+1)!C^{v_1+v_2+1}f^{-(v_1+v_2+1)/2}\}$. 
%We next prove this similarly as above. 
By \eqref{eq:a4diffq2prod} and Lemma \ref{lm:Leibnizfinite}, 
\begin{align}
&~\Delta^{v_1}_2 \Delta^{v_2}_4(Q_2,f) \notag \\
=&~\Delta^{v_1}_2\Biggr\{\sum_{w_2=0}^{v_2-1}\binom{v_2-1}{w_2} \Delta_4^{w_2}(A_2,f)	\Delta_4^{v_2-1-w_2}(Q_2,f+4w_2)\Biggr\}\notag \\
=&~\sum_{w_2=0}^{v_2-1}\binom{v_2-1}{w_2}\sum_{w_1=0}^{v_1}\binom{v_1}{w_1} \Delta^{w_1}_2 \Delta_4^{w_2}(A_2,f)\Delta^{v_1-w_1}_2\Delta_4^{v_2-1-w_2}(Q_2,f+4w_2+2w_1). \label{eq:v1v2updateform2}
\end{align}
To evaluate \eqref{eq:v1v2updateform2} through the mathematical induction, by \eqref{eq:linearformvh24} and  
\eqref{eq:mathinducs1s2cond}, we can assume that or integers $ s_1\leq v_1$ and $ s_2\leq v_2-1$,
%To evaluate \eqref{eq:v1v2updateform2}, 
%we note that by \eqref{eq:linearformvh24}, 
%\eqref{eq:mathinducs1s2cond} gives that for integers $ s_1\leq v_1$ and $ s_2\leq v_2-1$,
\begin{align}
\Delta^{s_1}_2\Delta^{s_2}_4(Q_2,f)=	O\big\{s_1! (s_2+1)!C^{s_1+s_2+1} f^{-(s_1+s_2+1)/2}\big\}. \label{eq:assumes1s2q2}
\end{align}
In addition, we use the following Lemma \ref{lm:doublediffa2}. 
\begin{lemma}\label{lm:doublediffa2}
For two integers $w_1$ and $w_2$ satisfying $w_1+w_2 \geq 1$,
\begin{align*}
\Delta_4^{w_2}\Delta_2^{w_1}(A_2,f)=(w_1+w_2+1)! O\Biggr(C^{w_1+w_2+1}\prod_{k=1}^{w_1+w_2}\frac{1}{f+2k}\Biggr).
\end{align*}	
\end{lemma}
\begin{proof}
Please see Section \ref{sec:doublediffa2}	on Page \pageref{sec:doublediffa2}. 
\end{proof}
\medskip

\noindent Combining \eqref{eq:assumes1s2q2} and Lemma \ref{lm:doublediffa2}, 
we obtain $\Delta^{v_1}_2 \Delta^{v_2}_4(Q_2,f)=O\{v_1!v_2!C^{v_1+v_2+1}f^{-(v_1+v_2+1)/2}\}$ similarly to \eqref{eq:h1h2inductorder1} in \textit{Step I.1}. As $v_2<v_2+1$, we have $\Delta^{v_1}_2 \Delta^{v_2}_4(Q_2,f)=O\{v_1!(v_2+1)! C^{v_1+v_2+1}\times f^{-(v_1+v_2+1)/2}\}$.

\bigskip

\noindent \textit{(Part II) Proof for $h_1=2$ and $h_2=3$}. \quad   
In this part, we prove  
\begin{align}
	\Delta^{v_2}_6\Delta^{v_1}_4(F_{x}, f) = O\big\{v_1!v_2! C^{v_1+v_2} f^{-(v_1+v_2)/2}\big\},  \label{eq:lfinitediffh46f}
\end{align} as $f\to \infty$ and uniformly for integers $v_1, v_2 \geq 1$.

When $v_1=0$ or $v_2=0$, \eqref{eq:lfinitediffh46f} holds by Proposition \ref{prop:infinitesumm}. 
%\eqref{eq:probdifforder1h}. 
When $v_1=v_2=1$, 
note that  $\Delta^1_4(F_x,f)=Q_1(f)+Q_2(f)$ by  \eqref{eq:linearformvh24}. 
%since $\Delta^1_4(F_x,f)=Q_1(f)+Q_2(f)$, by  \eqref{eq:linearformvh24}, we have 
Then we have
%the proof in Section \ref{sec:order2diff}, we have
$
\Delta^1_6\Delta^1_4(F_x,f)= \Delta^1_6 (Q_1,f) + \Delta^1_6(Q_2,f).
$ 
Particularly, 
\begin{align}
&\Delta^1_6(Q_1,f)=D_{2,6}(f)Q_1(f), \quad \ D_{2,6}(f)=\frac{x^3}{(f+6)(f+4)(f+2)} -1; \label{eq:d26def} \\
&\Delta^1_6(Q_2,f)=D_{4,6}(f)Q_2(f), \quad \ D_{4,6}(f)=\frac{x^3}{(f+8)(f+6)(f+4)} -1.\notag
\end{align}
By the proof of \eqref{eq:probdiff1approx}, $Q_1(f)=O(f^{-1/2})$ and $Q_2(f)=O(f^{-1/2})$. 
In addition, 
%as $x=f+\sqrt{2f}\{z_{\alpha} + O(f^{-1/2})\}$, 
for $x=\chi^2_f(\alpha)$, by \eqref{eq:xchisqorder}, 
$D_{2,6}(f)=O(f^{-1/2})$ and $D_{4,6}(f)=O(f^{-1/2})$. 
Therefore, \eqref{eq:lfinitediffh46f} holds for $v_1=1$ and $v_2=1$. 
When $v_1>1$ or $v_2>1$, by \eqref{eq:linearformvh24}, 
\begin{align*}
	\Delta^{v_2}_6\Delta^{v_1}_4(F_{x}, f) = \Delta^{v_2}_6\Delta^{v_1-1}_4(Q_1,f) +  \Delta^{v_2}_6\Delta^{v_1-1}_4(Q_2,f). 
\end{align*}
%To prove \eqref{eq:lfinitediffh46f}, it
It suffices to prove  
%$\Delta^{v_2}_6\Delta^{v_1-1}_4(Q_1,f)$ and $\Delta^{v_2}_6\Delta^{v_1-1}_4(Q_2,f)$ are $O\big\{v_1!v_2! C^{v_1+v_2} f^{-(v_1+v_2)/2}\big\}$, respectively. 
%We next prove the conclusions by the mathematical induction. 
\begin{align}
\Delta^{v_2}_6\Delta^{v_1-1}_4(Q_1,f) =&~ O\big\{v_1!v_2! C^{v_1+v_2} f^{-(v_1+v_2)/2}\big\}, \label{eq:prv14v26q1} \\
\Delta^{v_2}_6\Delta^{v_1-1}_4(Q_2,f) = &~ O\big\{v_1!v_2! C^{v_1+v_2} f^{-(v_1+v_2)/2}\big\}. \label{eq:prv14v26q2} 
\end{align}
We next prove \eqref{eq:prv14v26q1} and \eqref{eq:prv14v26q2}  by the mathematical induction, respectively. 

\medskip

%First, to prove the conclusion for $\Delta^{v_2}_6\Delta^{v_1-1}_4(Q_1,f)$, 

First, to prove \eqref{eq:prv14v26q1}, 
we apply the mathematical induction considering increasing $v_1$ and $v_2$ in the following \textit{Step II.1} and \textit{Step II.2}, respectively.

\medskip

\noindent \textit{Step II.1.} \quad We assume for  $0\leq s_1\leq v_1-2$ and $0\leq s_2\leq v_2$, 
\begin{align}
\Delta^{s_2}_6\Delta^{s_1}_4(Q_1,f)=O\{(s_1+1)!s_2!C^{s_1+s_2+1}f^{-(s_1+s_2+1)/2}\}, \label{eq:prop2s1swconc1}	
\end{align}
and then prove  \eqref{eq:prv14v26q1}.
Note that $\Delta^{1}_4(Q_1,f)=D_{2,4}(f)Q_1(f)$, where  $D_{2,4}(f)$ is defined in \eqref{eq:a3def}. 
Then by the Leibniz rule in Lemma \ref{lm:Leibnizfinite}, 
\begin{align}
&~\Delta^{v_2}_6\Delta^{v_1-1}_4(Q_1,f)\notag \\
=&~\Delta^{v_2}_6\Delta^{v_1-2}_4(D_{2,4}Q_1,f)\notag \\
=&~\sum_{k_2=0}^{v_2} \sum_{k_1=0}^{v_1-2} \binom{v_1-2}{k_1}\binom{v_2}{k_2} \Delta_6^{k_2}\Delta_4^{k_1}	(D_{2,4},f)\times \Delta_6^{v_2-k_2}\Delta_4^{v_1-2-k_1}(Q_1,f+4k_1+6k_2).\label{eq:leib46v1v2}
\end{align}
To evaluate \eqref{eq:leib46v1v2}, we use the following Lemma \ref{lm:d24order}. 
\begin{lemma}\label{lm:d24order}
For integers $k_1+k_2\geq 1$, there exists some constant $C$ such that
\begin{align*}
\Delta_6^{k_2}\Delta_4^{k_1}	(D_{2,4},f)=	(k_1+k_2+1)!O\left(C^{k_1+k_2}\prod_{t=1}^{k_1+k_2}\frac{1}{f+2t}\right), 
\end{align*}	
as $f\to \infty$ and uniformly over $k_1+k_2\geq 1$. 
\end{lemma}
\begin{proof}
Please see Section \ref{sec:d24order} on Page \pageref{sec:d24order}.
\end{proof}
\medskip

\noindent Then applying similar analysis to that of \eqref{eq:v1v2updateform} and \eqref{eq:h1h2inductorder1} in \textit{Part I} above,  
we obtain \eqref{eq:prv14v26q1} by the assumption \eqref{eq:prop2s1swconc1}	and Lemma \ref{lm:d24order}.   

\medskip

\noindent \textit{Step II.2.} \quad We assume for  $0\leq s_1\leq v_1-1$ and $0\leq s_2\leq v_2-1$, \eqref{eq:prop2s1swconc1} holds, 	
and then prove  \eqref{eq:prv14v26q1}.
By \eqref{eq:d26def} and the Leibniz rule in Lemma \ref{lm:Leibnizfinite}, 
\begin{align}
&~\Delta^{v_2}_6\Delta^{v_1-1}_4(Q_1,f)\label{eq:leib46v1v22} \\
=&~\Delta^{v_1-1}_4\Delta^{v_2-1}_6(D_{2,6}Q_1,f)\notag \\	
=&~ \sum_{k_1=0}^{v_1-1}\sum_{k_2=0}^{v_2-1}\binom{v_1-1}{k_1}\binom{v_2-1}{k_2}\Delta_6^{k_2}\Delta_4^{k_1}	(D_{2,6},f)\times \Delta_6^{v_2-1-k_2}\Delta_4^{v_1-1-k_1}(Q_1,f+4k_1+6k_2). \notag
\end{align} 
Similarly to the analysis of \eqref{eq:leib46v1v2}, 
we use the following Lemma \ref{lm:d26order} to evaluate \eqref{eq:leib46v1v22}. 
%next evaluate $\Delta_6^{k_2}\Delta_4^{k_1}	(D_{2,6},f)$. 

\begin{lemma}\label{lm:d26order}
For integers $k_1+k_2\geq 1$, there exists a constant $C$ such that
\begin{align*}
\Delta_6^{k_2}\Delta_4^{k_1}	(D_{2,6},f)=(k_1+k_2+2)! O\left(C^{k_1+k_2}\prod_{t=1}^{k_1+k_2}\frac{1}{f+2t}\right),
\end{align*}	
as $f\to \infty$ and uniformly over $k_1+k_2\geq 1$. 
\end{lemma}
\begin{proof}
Please see Section \ref{sec:d26order} on Page \pageref{sec:d26order}. 
\end{proof}
\medskip

\noindent 
Since we assume \eqref{eq:prop2s1swconc1} holds for  $0\leq s_1\leq v_1-1$ and $0\leq s_2\leq v_2-1$, then by Lemma \ref{lm:d26order}, 
%by \eqref{eq:leib46v1v22} and Lemma \ref{lm:d26order},
%By \eqref{eq:prop2s1swconc1}, \eqref{eq:leib46v1v22} for $ $, and Lemma \ref{lm:d26order},  
%we have
\begin{align*}
%\Delta^{v_2}_6\Delta^{v_1-1}_4(Q_1,f)
\eqref{eq:leib46v1v22}=&~	 \sum_{k_1=0}^{v_1-1}\sum_{k_2=0}^{v_2-1}\binom{v_1-1}{k_1}\binom{v_2-1}{k_2} (k_1+k_2+2)!(v_2-1-k_2)!(v_1-k_1)! \notag \\
&~\times C^{v_1+v_2} f^{-(v_1+v_2)/2} O\left( f^{-(k_1+k_2+1)/2}\prod_{t=1}^{k_1+k_2}\frac{1}{f+2t}\right) \notag \\
=&~C^{v_1+v_2} f^{-(v_1+v_2)/2}(v_1-1)!(v_2-1)!\sum_{k_2=0}^{v_2} \sum_{k_1=0}^{v_1-1}  (v_1-k_1)\notag \\
&~ \times \frac{ (k_1+k_2+2)!}{k_1!k_2!}O\left( f^{-(k_1+k_2+1)/2}\prod_{t=1}^{k_1+k_2}\frac{1}{f+2t}\right). \notag 
%\\=&~C^{v_1+v_2} f^{-(v_1+v_2)/2}(v_1+1)!v_2!
\end{align*} 
We next use the following Lemma \ref{lm:k1kwprodorder} to evaluate \eqref{eq:leib46v1v22}. 
%$\Delta^{v_2}_6\Delta^{v_1-1}_4(Q_1,f)$. 
%\eqref{eq:leib46v1v22}. 
\begin{lemma}\label{lm:k1kwprodorder}
For integers $k_1+k_2\geq 1$, as $f\to \infty$, 
\begin{align*}
\frac{ (k_1+k_2+2)!}{k_1!k_2!}O\left\{ f^{-(k_1+k_2+1)/2}\prod_{t=1}^{k_1+k_2}\frac{1}{f+2t}\right\}	= O\{2^{-(k_1+k_2-1)/2} \}.
\end{align*} 		
\end{lemma}
\begin{proof}
Please see Section \ref{sec:k1kwprodorder} on Page \pageref{sec:k1kwprodorder}. 	
\end{proof}
\medskip

\noindent Then by Lemma \ref{lm:k1kwprodorder}, we obtain $\Delta^{v_2}_6\Delta^{v_1-1}_4(Q_1,f)=O\{ v_1!v_2!C^{v_1+v_2} f^{-(v_1+v_2)/2}\}$ similarly to \eqref{eq:h1h2inductorder1}.  
In summary,  combining \textit{Step II.1} and \textit{Step II.2},  we finish the proof of \eqref{eq:prv14v26q1}.

%\smallskip
%
%\noindent In summary, we finish the proof of \eqref{eq:prv14v26q1}. 

\medskip

Second, to prove \eqref{eq:prv14v26q2}, we can use the mathematical induction  similarly to the proof of \eqref{eq:prv14v26q1}. 
The analysis would be very similar and the details are thus skipped.

\medskip 
 
\subsubsection{Proof of Lemma \ref{lm:a1order} (on Page \pageref{lm:a1order})}\label{sec:a1order}
%By $x=f+\sqrt{2f}\{z_{\alpha}+O(f^{-1/2})\}$, we have $A_1(f)=\sqrt{2}z_{\alpha}f^{-1/2}\{1+O(f^{-1})\}.$ 
When $x=\chi_f^2(\alpha)$, by \eqref{eq:xchisqorder}, we have $x=f+\sqrt{2f}\{z_{\alpha}+O(f^{-1/2})\}$,   
and then $A_1(f)=\sqrt{2}z_{\alpha}f^{-1/2}\{1+O(f^{-1})\}.$ 
We next prove \eqref{eq:a1korderdif} by the mathematical induction.
For $w=1$, we compute
\begin{align*}
\Delta_2^1 (A_1,f)= A_{1}(f+2)-A_1(f)= - x\times 2 \times \frac{1}{(f+2)(f+4)}.	
\end{align*} 
Therefore \eqref{eq:a1korderdif} holds when $w=1$. 
We next assume \eqref{eq:a1korderdif} holds, and prove the conclusion holds for $\Delta_2^{w+1}(A_1, f)$. 
Particularly, 
%it still holds when we replace $k$ with $k+1$. In particular, 
\begin{align*}
\Delta_2^{w+1}(A_1, f)=	&~ x\times(-1)^{w}2^ww!\left\{\frac{1}{\prod_{k=2}^{w+2}(f+2k) } - \frac{1}{\prod_{k=1}^{w+1}(f+2k)} \right\} \notag \\
=&~x\times (-1)^{w+1} 2^{w+1}(w+1)!\frac{1}{\prod_{k=1}^{w+2}(f+2k)}.
\end{align*} 	
In summary, Lemma \ref{lm:a1order} is proved. 

\subsubsection{Proof of Lemma \ref{lm:a2orderf} (on Page \pageref{lm:a2orderf})} \label{sec:a2orderf}
When $x=\chi_f^2(\alpha)$, by \eqref{eq:xchisqorder}, we have $x=f+\sqrt{2f}\{z_{\alpha}+O(f^{-1/2})\}$,   
and then  $A_2(f)=2\sqrt{2}z_{\alpha}f^{-1/2} \{1+O(f^{-1})\}$. 
We next prove \eqref{eq:ordera2fh4diff}. % by the mathematical induction.
Note that we can write $A_2(f)=A_{2,1}(f)A_{2,2}(f)-1$, where we define
\begin{align*}
	A_{2,1}(f)=\frac{x}{f+4} \quad \text{ and }\quad A_{2,2}(f)=\frac{x}{f+6}.
\end{align*}
By Lemmas \ref{lm:Leibnizfinite} and \ref{lm:linearityrulefinite}, when $w\geq 1$, 
%By the linearity rule of the finite difference in Lemma \ref{lm:linearityrulefinite}, $\Delta_4^w(A_2,f)= \Delta_4^w(A_{2,1}A_{2,2},f)$. By Lemma \ref{lm:Leibnizfinite}, 
\begin{align}
\Delta_4^w(A_2,f)=\Delta_4^w(A_{2,1}A_{2,2},f)=\sum_{k=0}^w \binom{w}{k} \Delta_4^k(A_{2,1},f)\Delta^{w-k}_4(A_{2,2},f+4k). \label{eq:leibnizb1b2}
\end{align}
To prove $\eqref{eq:leibnizb1b2}=O(w!C^{w}f^{-w})$,
we next evaluate $\Delta_4^k(A_{2,1},f)$ and $\Delta_4^{w-k}(A_{2,2},f+4k)$. 

%It suffices to prove that there exists a constant $C$ such that $\Delta_4^k(A_{1,1},f)=O(k!C^kf^{-k})$ and $\Delta_4^k(A_{2,2},f)=O(k!C^kf^{-k})$ as $f\to \infty$ and uniformly for integers $k\geq 1$. 

In particular, we prove that
\begin{align}
	\Delta_4^k(A_{2,1},f)=(-1)^{k}4^{k}k!x\times \frac{1}{\prod_{t=1}^{k+1}(f+4t)} \label{eq:b1sdfifforder}
\end{align}
by the mathematical induction.
When $k=1$, 
\begin{align*}
	\Delta_4^1(A_{2,1},f)=\frac{x}{f+8}-\frac{x}{f+4}=\frac{x\times (-4)}{(f+4)(f+8)}.
\end{align*}
Thus \eqref{eq:b1sdfifforder} holds for $k=1$. 
We next assume \eqref{eq:b1sdfifforder} holds and prove the conclusion for $\Delta_4^{k+1}(A_{2,1},f)$. 
Specifically, 
\begin{align*}
\Delta_4^{k+1}(A_{2,1},f)=&~(-1)^{k}4^{k}k!x\left\{\frac{1}{\prod_{t=2}^{k+2}(f+4t)}  - \frac{1}{\prod_{t=1}^{k+1}(f+4t)}\right\} \notag \\
=&~(-1)^{k+1}4^{k+1}(k+1)!x \frac{1}{\prod_{t=1}^{k+2}(f+4t)}.
\end{align*}
In summary, \eqref{eq:b1sdfifforder} is proved. 
Moreover, as $A_{2,2}(f)=A_{2,1}(f+2),$ we have
\begin{align*}
	\Delta_4^{k}(A_{2,2},f)=\Delta_4^{k}(A_{2,1},f+2)=(-1)^{k}4^{k}k!x \frac{1}{\prod_{t=1}^{k+1}(f+2+4t)}.
\end{align*}
It follows that $\Delta_4^{w-k}(A_{2,2},f+4k)=(-1)^{w-k}4^{w-k}(w-k)!x\{\prod_{t=k+1}^{w+1}(f+2+4t)\}^{-1}$. 
%Since $x=O(f)$, we know $\Delta_4^k(A_{1,1},f)$ and $\Delta_4^k(A_{2,2},f)$ are $O(k!C^kf^{-k})$ as $f\to \infty$ and uniformly for integers $k\geq 1$.
Then by \eqref{eq:leibnizb1b2},
there exists a constant $C$ such that 
\begin{align*}
\big|\Delta_4^w(A_{2,1}A_{2,2},f)\big|=&~\Biggr|\sum_{k=0}^w \binom{w}{k} \frac{(-4)^{w}k!(w-k)!x^2}{\prod_{t=1}^{k+1}(f+4t)\prod_{t=k+1}^{w+1}(f+2+4t)}\Biggr| \leq w!C^w \sum_{k=0}^w \frac{x^2}{\prod_{t=1}^{w+2}(f+2t)}.
%=&~\sum_{k=0}^w O(w!C^{w}f^{-w})\frac{x^2f^w}{\prod_{t=1}^{k+1}(f+4t)\prod_{t=k+1}^{w+1}(f+2+4t) }\notag \\
%=&~\sum_{k=0}^w O(w!C^{w}f^{-w})\frac{1}{\prod_{t=1}^{k+1}(1+4t/f)\prod_{t=k+1}^{w+1}(f+2/f+4t/f) }
\end{align*}
As $x=\chi^2_f(\alpha)=O(f)$, we obtain that \eqref{eq:ordera2fh4diff}  holds as $f\to \infty$ and uniformly for any integer $w\geq 1$. 

%\begin{align*}
%|\Delta_4^w(A_{2,1}A_{2,2},f)|\leq &~ \sum_{k=0}^w O(w!C^{w}f^{-w})\frac{1}{\prod_{t=1}^{w+2}(1+2t/f)}	= O(w!C^{w}f^{-w}). 
%\end{align*}
%\begin{align*}
%|\Delta_4^w(A_{2,1}A_{2,2},f)|\leq &~  O(w!C^{w}f^{-w})\frac{(w+1)f^{w+2}}{\prod_{t=1}^{w+2}(f+2t)}.
%%	= O(w!C^{w}f^{-w}). 
%\end{align*}
%As $f>0$, 
%\begin{align*}
%\frac{(w+1)f^{w+2}}{\prod_{t=1}^{w+2}(f+2t)}\leq \frac{w+1}{(1+2/f)^{w+2}}	\leq \frac{w+1}{1+2(w+2)/f} \leq \frac{1}{(w+1)^{-1}+2/f}.
%\end{align*}
%
%we obtain that \eqref{eq:ordera2fh4diff}  holds uniformly for any integer $w\geq 1$. 

%Therefore, when $f\to \infty$, $(w+1)f^{w+2}\{\prod_{t=1}^{w+2}(f+2t)\}^{-1}=O(1)$ uniformly over $w\geq 1$. 
%It follows that $\Delta_4^w(A_{2,1}A_{2,2},f)=O(w!C^{w}f^{-w})$, and we obtain \eqref{eq:ordera2fh4diff}. 

%As $\prod_{t=1}^{w+2}(f+2t) \geq f^{w+2}+2(w+2)!$, 
%we have $(w+1)f^{w+2}\{\prod_{t=1}^{w+2}(f+2t)\}^{-1} = O(1)$ as $f\to \infty$ and 
%
%\begin{align*}
%\prod_{t=1}^{w+2}\frac{(w+1)^{\frac{1}{w+2}} }{(1+2t/f)}=	
%\end{align*}
%\begin{align*}
%\prod_{t=1}^{w+2}\frac{(1+2t/f)}{(w+1)^{\frac{1}{w+2}} }= \prod_{t=1}^{w+2}\left\{\frac{w+1}{(1+2t/f)^{w+2}} \right\}^{\frac{1}{w+2}}	
%\end{align*}
%\bigskip
%
%Then by \eqref{eq:leibnizb1b2}, we obtain \eqref{eq:ordera2fh4diff}.

\subsubsection{Proof of Lemma \ref{lm:a3orderf} (on Page  \pageref{lm:a3orderf})} \label{sec:a3orderf}
When $x=\chi_f^2(\alpha)$, by \eqref{eq:xchisqorder}, we have $x=f+\sqrt{2f}\{z_{\alpha}+O(f^{-1/2})\}$,   
and then  $A_3(f)=3\sqrt{2}z_{\alpha}f^{-1/2} \{1+O(f^{-1/2})\}$. 
%By $x=f+\sqrt{2f}\{z_{\alpha}+O(f^{-1/2})\}$, we have 
% $A_3(f)=3\sqrt{2}z_{\alpha}f^{-1/2} \{1+O(f^{-1/2})\}$. 
We next consider $\Delta_6^w(A_3,f)$ for $w\geq 1.$ 
%prove  the conclusion for $w\geq 1.$ 
As $A_3(f)=\prod_{l=1}^{3}A_{3,l}(f)-1$, 
\begin{align*}
\Delta_6^w(A_3,f)=\sum_{k_1=0}^{w} \sum_{k_2=0}^{k_1} \binom{k_1}{k_2}\binom{w}{k_1}\Delta_6^{k_2}(A_{3,1},f)	\Delta_6^{k_1-k_2}(A_{3,2},f+6k_2)	\Delta_6^{w-k_1}(A_{3,3}, f+6k_1). 
\end{align*}
Similarly to the proofs of Lemma \ref{lm:a1order} in Section \ref{sec:a1order},  
for $A_{3,l}(f)$, $l\in \{1,2,3\}$,  
we can obtain that for any integer $w\geq 1$ and $l\in \{1,2,3\}$ 
\begin{align*}
\Delta_6^w(A_{3,l},f)=(-6)^w w! x \times \frac{1}{\prod_{t=0}^{w}(f+4+2l+6t)}.
\end{align*}
It follows that
\begin{align*}
\Delta_6^w(A_3,f)=&~\sum_{k_1=0}^{w} \sum_{k_2=0}^{k_1} \binom{k_1}{k_2}\binom{w}{k_1}(-6)^w k_2!(k_1-k_2)!(w-k_1)! \notag \\
&~	\times x^3\Biggr\{ \prod_{t=1}^{k_2+1} (f+6t) \prod_{t=k_2+1}^{k_1+1}(f+6t+2) \prod_{t=k_1+1}^{w+1} (f+6t+4) \Biggr\}^{-1}.
% \notag \\=&~ 
\end{align*}
As $\binom{k_1}{k_2}\binom{w}{k_1} k_2!(k_1-k_2)!(w-k_1)!=w!$, $\sum_{k_1=0}^{w} \sum_{k_2=0}^{k_1} 1 \leq (w+1)^2$, 
  and $x=\chi^2_f(\alpha)=O(f)$, there exists a constant $C$ such that as $f\to \infty$ and uniformly over $w\geq 1$, 
\begin{align*}
\Delta_6^w(A_3,f)=&~O\biggr\{(w+2)!C^w \prod_{t=1}^{w}(f+2t)^{-1}\biggr\}. 
\end{align*}

% \bigskip
 
\subsubsection{Proof of Lemma \ref{lm:a4orderf} (on Page \pageref{lm:a4orderf})} \label{sec:a4orderf}
When $x=\chi_f^2(\alpha)$, by \eqref{eq:xchisqorder}, we have $x=f+\sqrt{2f}\{z_{\alpha}+O(f^{-1/2})\}$,   
and then $A_4(f)=4\sqrt{2}z_{\alpha}f^{-1/2} \{1+O(f^{-1/2})\}$.
%By $x=f+\sqrt{2f}\{z_{\alpha}+O(f^{-1/2})\}$, we have 
% $A_4(f)=4\sqrt{2}z_{\alpha}f^{-1/2} \{1+O(f^{-1/2})\}$. 
We next prove  the conclusion for $w\geq 1.$ 
As $A_4(f)=\prod_{l=1}^{4}A_{4,l}(f)-1$, 
\begin{align*}
\Delta_8^w(A_4,f)=&~\sum_{k_1=0}^{w}\sum_{k_2=0}^{k_1}\sum_{k_3=0}^{k_2} \binom{w}{k_1} \binom{k_1}{k_2}\binom{k_2}{k_3} \Delta_8^{k_3}(A_{4,1}, f)\times \Delta_8^{k_2-k_3}(A_{4,2}, f+8k_3)\notag \\
&~ \times \Delta_8^{k_1-k_2}(A_{4,3}, f+8k_2)\times \Delta_8^{w-k_1}(A_{4,4}, f+8k_1).
\end{align*}
%\begin{align*}
%\Delta_8^w(A_4,f)=-\sum_{k_1=0}^{w} \sum_{k_2=0}^{k_1} \binom{k_1}{k_2}\binom{w}{k_1}\Delta_8^{k_2}(A_{4,1},f)	\Delta_8^{k_1-k_2}(A_{4,2},f+6k_2)	\Delta_6^{w-k_1}(A_{3,3}, f+6k_1). 
%\end{align*}
Similarly to the proof of Lemma \ref{lm:a1order} in Section \ref{sec:a1order},  
for $A_{4,l}(f)$, $l\in \{1,2,3, 4\}$,  
we can obtain that for any integer $w\geq 1$,  
%and $l\in \{1,2,3,4\}$ 
\begin{align*}
\Delta_8^w(A_{4,l},f)=(-8)^w w! x \times \frac{1}{\prod_{t=0}^{w}(f+6+2l+8t)}.
\end{align*}
It follows that
\begin{align*}
\Delta_8^w(A_4,f)=&~\sum_{k_1=0}^{w}\sum_{k_2=0}^{k_1}\sum_{k_3=0}^{k_2} \binom{w}{k_1} \binom{k_1}{k_2}\binom{k_2}{k_3} (-8)^w k_3!(k_2-k_3)!(k_1-k_2)!(w-k_1)! \notag \\
&~ \times x^4 \biggr\{ \prod_{t=1}^{k_3+1}(f+8t) \prod_{t=k_3+1}^{k_2+1}(f+8t+2) \prod_{t=k_2+1}^{k_1+1}(f+8t+4)\prod_{k_1+1}^{w+1}(f+8t+6)\biggr\}^{-1}. 
\end{align*}
As $ \binom{w}{k_1} \binom{k_1}{k_2}\binom{k_2}{k_3} k_3!(k_2-k_3)!(k_1-k_2)!(w-k_1)!=w!$, $\sum_{k_1=0}^{w}\sum_{k_2=0}^{k_1}\sum_{k_3=0}^{k_2} 1 \leq (w+1)^3$,  and $x=O(f)$, there exists a constant $C$ such that 
\begin{align*}
\Delta_8^w(A_4,f)=&~O\biggr\{(w+3)!C^w \prod_{t=1}^{w}(f+2t)^{-1}\biggr\}. 
\end{align*}

%\begin{align*}
%\Delta_6^w(A_3,f)=&~\sum_{k_1=0}^{w} \sum_{k_2=0}^{k_1} \binom{k_1}{k_2}\binom{w}{k_1}(-6)^w k_2!(k_1-k_2)!(w-k_1)! \notag \\
%&~	\times x^3\Biggr\{ \prod_{t=1}^{k_2+1} (f+6t) \prod_{t=k_2+1}^{k_1+1}(f+6t+2) \prod_{t=k_1+1}^{w+1} (f+6t+4) \Biggr\}^{-1}.
%% \notag \\=&~ 
%\end{align*}
%As $\binom{k_1}{k_2}\binom{w}{k_1}(-6)^w k_2!(k_1-k_2)!(w-k_1)!=w!$, $\sum_{k_1=0}^{w} \sum_{k_2=0}^{k_1} 1 \leq (w+1)^2$, 
%  and $x=O(f)$, there exists a constant $C$ such that 
%\begin{align*}
%\Delta_6^w(A_3,f)=&~O\biggr\{(w+2)!C^w \prod_{t=1}^{w}(f+2t)^{-1}\biggr\}. 
%\end{align*}

%\smallskip
 
\subsubsection{Proof of Lemma \ref{lm:doublediffa1} (on Page \pageref{lm:doublediffa1})} \label{sec:doublediffa1}
By the proof of  Lemma \ref{lm:a1order}, we have
\begin{align*}
\Delta_2^{w_1} (A_1, f)=&~ (-1)^{w_1}2^{w_1}w_1! x\prod_{s=1}^{w_1+1}A_{1,s}(f),	
\end{align*}
where $A_{1,s}(f)=1/(f+2s)$.
It follows that
\begin{align}
	\Delta_4^{w_2}\big\{\Delta_2^{w_1} (A_1, f)\big\}=x (-2)^{w_1}w_1! 	\Delta_4^{w_2}\Biggr\{\prod_{s=1}^{w_1+1}A_{1,s}(f)\Biggr\}. \label{eq:a1diffw1w2}
\end{align} 
%As $x=O(f)$, to prove Lemma \ref{lm:doublediffa1}, 
To prove Lemma \ref{lm:doublediffa1}, by $x=\chi^2_f(\alpha)=O(f)$ and \eqref{eq:a1diffw1w2}, 
it suffices to prove 
\begin{align}
\Delta_4^{w_2}\Biggr\{\prod_{s=1}^{w_1+1}A_{1,s}(f)\Biggr\}=\frac{(w_1+w_2)!}{w_1!}	O\left\{ C^{w_1+w_2}\prod_{s=1}^{w_1+w_2+1} (f+2s)^{-1} \right\}.\label{eq:a1sprodval}
\end{align}

%Particularly, we prove \eqref{eq:a1sprodval} by the mathematical induction, 
We next prove \eqref{eq:a1sprodval} by the mathematical induction.
Consider $w_1=0$ first. 
Similarly to the proof of  Lemma \ref{lm:a1order}, for each integer $1\leq s\leq w_1+1$, we have
\begin{align}
	\Delta_4^{w_2}(A_{1,s},f)=w_2!(-4)^{w_2}\prod_{k=0}^{w_2}(f+2s+4k).\label{eq:valdel3w2a1s}
\end{align}
Thus \eqref{eq:a1sprodval} holds for $w_1=0$. 
We then assume for integers $1\leq l\leq w_1$, 
\begin{align}
\Delta_4^{w_2}\Biggr\{ \prod_{s=1}^l A_{1,s}(f)\Biggr\}=\frac{(w_2+l-1)!}{(l-1)!}O\left\{ \prod_{k=1}^{w_2+l}(f+2k)^{-1}\right\}, \label{eq:a1sprodvalassume}
\end{align}
and prove \eqref{eq:a1sprodval}. 
%By  \eqref{eq:a1diffw1w2}  and  the Leibniz rule in Lemma \ref{lm:Leibnizfinite}, 
By the Leibniz rule in Lemma \ref{lm:Leibnizfinite}, 
\begin{align}\label{eq:a1diffw1w222}
\Delta_4^{w_2}\Biggr\{\prod_{s=1}^{w_1+1}A_{1,s}(f)\Biggr\}=\sum_{k_2=0}^{w_2}\binom{w_2}{k_2}\Delta_4^{k_2}\left\{ \prod_{s_1=1}^{w_1}A_{1,s_1}(f) \right\}\Delta_4^{w_2-k_2}(A_{1,w_1+1},f+4k_2). 	
\end{align}
%\begin{align*}
%\Delta_4^{w_2}\big\{\Delta_2^{w_1} (A_1, f)\big\}=&~\sum_{k_2=0}^{w_2}\binom{w_2}{k_2}\Delta_4^{k_2}\left\{ \prod_{s_1=1}^{w_1}A_{1,s_1}(f) \right\}\Delta_4^{w_2-k_2}(A_{1,w_1+1},f+4k_2).\notag 
%%=&~ \sum_{k_2=0}^{w_2}\binom{w_2}{k_2} \frac{(k_2+w_1-1)!}{(w_1-1)!}O \Biggr( \prod_{s_1=1}^{w_1+k_2}\frac{1}{f+2s_1}\Biggr) \notag \\ 
%%&~\quad \times O\Biggr\{ (w_2-k_2)!\prod_{s_2=0}^{w_2-k_2+1}\frac{1}{f+4k_2+2(w_1+1) + 4s_2} \Biggr\}
%% \notag \\	
%\end{align*}
%Similarly to the proof of  Lemma \ref{lm:a1order},  we have for each integer $0\leq k_2\leq w_2$
%\begin{align*}
%	\Delta_4^{w_2-k_2}(A_{1,w_1+1},f+4k_2)= O\Biggr\{ (w_2-k_2)!\prod_{s_2=0}^{w_2-k_2+1}\frac{1}{f+4k_2+2(w_1+1) + 4s_2} \Biggr\}.
%\end{align*}
Then by \eqref{eq:valdel3w2a1s} and \eqref{eq:a1sprodvalassume}, we obtain
\begin{align*}
\eqref{eq:a1diffw1w222}=&~ \sum_{k_2=0}^{w_2}\binom{w_2}{k_2} \frac{(k_2+w_1-1)!}{(w_1-1)!}O \Biggr( C^{w_1+k_2} \prod_{s_1=1}^{w_1+k_2}\frac{1}{f+2s_1}\Biggr) \notag \\ 
&~\quad \times O\Biggr\{ (w_2-k_2)!C^{w_2-k_2}\prod_{s_2=0}^{w_2-k_2}\frac{1}{f+4k_2+2(w_1+1) + 4s_2} \Biggr\}
 \notag \\	
%=&~ \sum_{k_2=0}^{w_2}\frac{w_2!(k_2+w_1-1)!}{k_2!(w_1-1)!}O\Biggr\{\prod_{s_1=0}^{w_1+k_2-1}\frac{1}{f+2s_1}\prod_{s_2=w_1+k_2}^{w_2+w_1}\frac{1}{f+2s_2}\Biggr\} \notag \\	
=&~ C^{w_1+w_2}\sum_{k_2=0}^{w_2} w_2!\binom{k_2+w_1-1}{k_2} O\Biggr(\prod_{s=1}^{w_1+w_2+1}\frac{1}{f+2s}\Biggr). \notag
 \end{align*}
By the hockey-stick identity, $ \sum_{k_2=0}^{w_2} \binom{k_2+w_1-1}{k_2}=\binom{w_1+w_2}{w_2}$. 
Therefore, \eqref{eq:a1sprodval} is proved and then \eqref{eq:a1diffw1w2} follows.

%As $\binom{k_1}{k_2}\binom{w}{k_1}(-6)^w k_2!(k_1-k_2)!(w-k_1)!=w!$,  $6t \geq 2$, $x=O(f)$, and $\sum_{k_1=0}^{w} \sum_{k_2=0}^{k_1} 1 \leq (w+1)^2$, 
%\begin{align*}
%|\Delta_6^w(A_3,f)|\leq O(w!6^wf^{-w}) \times \frac{(w+1)^2f^{w+3}}{(f+2)^{w+3}}. 	
%\end{align*}
%%As $f>0$, 
%Since
%\begin{align*}
%\frac{(w+1)^2f^{w+3}}{(f+2)^{w+3}} = \frac{(w+1)^2}{(1+2/f)^{w+3}}\leq \frac{(w+2)^2}{}
%\end{align*}
%\begin{align*}
%\frac{(w+1)^2f^{w+3}}{\prod_{t=1}^{w+3} (f+2t)} \leq 	
%\end{align*}
%\bigskip
%there exists a constant $C$ such that 
%$\Delta_6^{w}(A_{3,k},f)=O(w!C^wf^{-w})$ uniformly for $w\geq 1.$  
%It follows that  
%\begin{align*}
%\Delta_6^w(A_3,f)=\sum_{k_1=0}^{w} \sum_{k_2=0}^{k_1} \binom{k_1}{k_2}\binom{w}{k_1}k_2!(k_1-k_2)!(w-k_1)!O(C^wf^{-w})=\sum_{k_1=0}^{w} \sum_{k_2=0}^{k_1}O(w!C^wf^{-w})
%\end{align*} 
  
\smallskip 
\subsubsection{Proof of Lemma \ref{lm:evalwqw2} (on Page  \pageref{lm:evalwqw2})} \label{sec:evalwqw2}
We next prove Lemma \ref{lm:evalwqw2} by discussing the cases when $w_1+w_2$ is odd and even, respectively. \\
(1) When $w_1+w_2$ is odd, $(w_1+w_2+1)/2$ is an integer, and then  
\begin{align*}
(w_1+w_2)!\prod_{k=1}^{w_1+w_2}\frac{1}{f+2k}\times f^{(w_1+w_2+1)/2} \leq &~ (w_1+w_2)!\prod_{k=(w_1+w_2+1)/2 + 1}^{w_1+w_2}	\frac{1}{2k} \notag \\
%\leq &~2^{-(w_1+w_2-1)/2}\times \Gamma \left( \frac{w_1+w_2-1}{2} \right).  
\leq &~2^{-(w_1+w_2-1)/2} \prod_{k=1}^{(w_1+w_2+1)/2}k. 
\end{align*} 
To prove Lemma \ref{lm:evalwqw2}, it now suffices to prove that there exists a constant $C$ such that 
\begin{align}\label{eq:w1w2boundc}
 \frac{1}{w_1!w_2!}\prod_{k=1}^{(w_1+w_2+1)/2}k	 \leq C.
\end{align}
To prove \eqref{eq:w1w2boundc}, we  use the following Lemma \ref{lm:wfactorialbound}. 
%Then we use the following Lemma to prove \eqref{eq:w1w2boundc}. 
%Assume without loss of generality that $w_2\geq w_1$, and then we use the following Lemma to prove \eqref{eq:w1w2boundc}. 
\begin{lemma}[Factorial bound]\label{lm:wfactorialbound}
For any integer $w\geq 1$, 
\begin{align*}
\left(\frac{w}{e}\right)^{w} e \leq w ! \leq\left(\frac{w+1}{e}\right)^{w+1} e. 	
\end{align*}	
\end{lemma}
\begin{proof}
This is a known bound on factorial in literature, and is obtained by  
$\int_{1}^{w} \ln x \mathrm{d} x \leq \sum_{x=1}^{w} \ln x \leq \int_{0}^{w} \ln (x+1) \mathrm{d} x$	. 
\end{proof}
\medskip

\noindent Assume without loss of generality that $w_2\geq w_1$, and then  by Lemma \ref{lm:wfactorialbound}, 
\begin{align}
  &~ \frac{1}{w_1!w_2!}\prod_{k=1}^{(w_1+w_2+1)/2}k\notag \\
  \leq &~ \frac{1}{e} \left(\frac{e}{w_1}\right)^{w_1}\left(\frac{e}{w_2}\right)^{w_2} \left( \frac{w_1+w_2+3}{2e} \right)^{(w_1+w_2+3)/2} \notag \\
 =&~\frac{1}{e} \left(  \frac{e^2}{w_1w_2} \frac{w_1+w_2+3}{2e} \right)^{w_1} \left(\frac{e^2}{w_2^2} \frac{w_1+w_2+3}{2e} \right)^{(w_2-w_1)/2}\left( \frac{w_1+w_2+3}{2e} \right)^{3/2}. \label{eq:upperbound1prod}
 \end{align}
As $w_1+w_2+3\leq 4w_2$, there exists a constant $C$ such that
\begin{align*}
\eqref{eq:upperbound1prod}\leq C\left(\frac{2e}{w_1} \right)^{w_1}\left(\frac{2e}{w_2} \right)^{(w_2-w_1)/2}(w_1+w_2+3)^{3/2}. 	
\end{align*} 
When $w_2-w_1\geq 3$, 
\begin{align*}
\eqref{eq:upperbound1prod}\leq C\left(\frac{2e}{w_1} \right)^{w_1}\left(\frac{2e}{w_2} \right)^{(w_2-w_1-3)/2} \left\{ \frac{2e(w_1+w_2+3)}{w_2} \right\}^{3/2},	
\end{align*} 
which is bounded. When  $0\leq w_2-w_1\leq 2$, 
\begin{align*}
\eqref{eq:upperbound1prod}\leq C \left(\frac{2e}{w_1} \right)^{w_1}(2w_1+5)^{3/2},
\end{align*} 
which is also bounded. In summary, \eqref{eq:upperbound1prod} is bounded. 
 
%As $w_1+w_2\leq 2w_2$, 
%%\begin{align*}
%% \frac{e^2}{w_1w_2} \frac{w_1+w_2+1}{2e} \leq \frac{e}{w_1}.	
%%\end{align*}
%%Therefore,
%\begin{align*}
%\eqref{eq:upperbound1prod}\leq \sqrt{\frac{e}{2}}\left(\frac{e}{w_1} \right)^{w_1}\left(\frac{e}{w_2} \right)^{(w_2-w_1)/2}\sqrt{w_1+w_2+1}.
%\end{align*}
%When $w_2-w_1\geq 1$, 
%\begin{align*}
%\eqref{eq:upperbound1prod}\leq \sqrt{\frac{e}{2}}\left(\frac{e}{w_1} \right)^{w_1}	\left(\frac{e}{w_2} \right)^{(w_2-w_1-1)/2}\left\{ \frac{e(w_1+w_2+1)}{w_2} \right\}^{1/2},
%\end{align*}
%which is bounded. 
%%\begin{align*}
%%\eqref{eq:upperbound1prod}\leq \sqrt{\frac{e}{2}}\left(\frac{e}{w_1} \right)^{w_1-1}	\left(\frac{e}{w_2} \right)^{(w_2-w_1-1)/2}\left\{ \frac{e^2(w_1+w_2+1)}{w_1w_2} \right\}^{1/2}.
%%\end{align*}
%When $w_2-w_1=0$,
%\begin{align*}
%\eqref{eq:upperbound1prod}\leq \sqrt{\frac{e}{2}}\left(\frac{e}{w_1} \right)^{w_1} \sqrt{2w_1+1},	
%\end{align*} which is also bounded. In summary, \eqref{eq:upperbound1prod} is bounded. 

%which is bounded when $w_1$ and $w_2$ are sufficiently large. 
\medskip

\noindent (2) When $w_1+w_2$ is even, similarly, we have
\begin{align*}
(w_1+w_2)!\prod_{k=1}^{w_1+w_2}\frac{1}{f+2k}\times f^{(w_1+w_2+1)/2} \leq 	2^{-(w_1+w_2)/2-1}	\prod_{k=1}^{(w_1+w_2)/2+1}k.
\end{align*}
%\begin{align*}
%(w_1+w_2)!\prod_{k=1}^{w_1+w_2}\frac{1}{f+2k}\times f^{(w_1+w_2+1)/2} \leq &~ (w_1+w_2)!\prod_{k=(w_1+w_2)/2 + 2}^{w_1+w_2}	\frac{1}{2k} \notag \\
%\leq &~2^{-(w_1+w_2)/2-1}	\prod_{k=1}^{(w_1+w_2)/2+1}k.
%\end{align*} 
%\begin{align*}
%(w_1+w_2)!\prod_{k=1}^{w_1+w_2}\frac{1}{f+2k}\times f^{(w_1+w_2+1)/2} \leq 	 \prod_{k=(w_1+w_2)/2 + 2}	\frac{1}{2k}
%\end{align*}
To prove Lemma \ref{lm:evalwqw2}, it now suffices to prove that there exists a constant $C$ such that 
\begin{align*}
 \frac{1}{w_1!w_2!}\prod_{k=1}^{(w_1+w_2)/2+1}k \leq C. 
\end{align*}
Similar analysis can be applied and the conclusions follow.

%\smallskip

\subsubsection{Proof of Lemma \ref{lm:doublediffa2} (on Page \pageref{lm:doublediffa2})}\label{sec:doublediffa2}
When $w_1=0$, we know  Lemma  \ref{lm:doublediffa2} holds  by Lemma \ref{lm:a2orderf}. 
%Lemma \ref{lm:a2orderf} suggests that Lemma  \ref{lm:doublediffa2} holds when $w_1=0$. 
Recall that we write $A_2(f)=A_{2,1}(f)A_{2,2}(f)-1$ in Section \ref{sec:a2orderf}. 
Thus when $w_1+w_2\geq 1$,
\begin{align*}
\Delta_2^{w_1}\Delta_4^{w_2}(A_2,f)=\Delta_2^{w_1}\Delta_4^{w_2}(A_{2,1}A_{2,2},f).
\end{align*}
By the Leibniz rule in Lemma \ref{lm:Leibnizfinite}, 
\begin{align}
&~\Delta_2^{w_1}\Delta_4^{w_2}(A_{2,1}A_{2,2},f)\notag \\
=&~\sum_{k_1=0}^{w_1} \sum_{k_2=0}^{w_2} \binom{w_1}{k_1}\binom{w_2}{k_2}\Delta_2^{k_1}\Delta_4^{k_2}(A_{2,1},f)\Delta_2^{w_1-k_1}\Delta^{w_2-k_2}_4(A_{2,2},f+2k_1+4k_2).\label{eq:k1k2a21a22leib}
\end{align}
Following the proof of Lemma \ref{lm:doublediffa1}, 
we have when $k_1+k_2\geq 1$, 
\begin{align*}
\Delta_2^{k_1}\Delta_4^{k_2}(A_{2,1},f) 	= (k_1+k_2)! O\left(C^{k_1+k_2}\prod_{s=1}^{k_1+k_2}\frac{1}{f+2s}\right),
\end{align*}
and when $w_1+w_2-k_1-k_2\geq 1$, 
\begin{align*}
\Delta_2^{w_1-k_1}\Delta^{w_2-k_2}_4(A_{2,2},f+2k_1+4k_2)=(w_1+w_2-k_1-k_2)! O\left(C^{w_1+w_2-k_1-k_2} \prod_{s=k_1+k_2+1}^{w_1+w_2} \frac{1}{f+2s} \right).	
\end{align*}
Therefore,
\begin{align*}
\eqref{eq:k1k2a21a22leib}=w_1!w_2!\sum_{k_1=0}^{w_1} \sum_{k_2=0}^{w_2} \binom{k_1+k_2}{k_1}\binom{w_1+w_2-k_1-k_2}{w_1-k_1} 	O\left(C^{w_1+w_2}\prod_{s=1}^{w_1+w_2}\frac{1}{f+2s}\right).
\end{align*}
By the Chu–Vandermonde identity, 
\begin{align*}
\sum_{k_1=0}^{w_1} \sum_{k_2=0}^{w_2} \binom{k_1+k_2}{k_1}\binom{w_1+w_2-k_1-k_2}{w_1-k_1} =&~ \sum_{m=0}^{w_1+w_2}\sum_{s_1=0}^{w_1}  \binom{m}{s_1}\binom{w_1+w_2-m}{w_1-s_1}\notag \\
=&~(w_1+w_2+1)\binom{w_1+w_2}{w_1}. 
\end{align*}
Then $\Delta_2^{w_1}\Delta_4^{w_2}(A_2,f)=(w_1+w_2+1)!O\{C^{w_1+w_2}\prod_{s=1}^{w_1+w_2}(f+2s)^{-1}\}.$ 

%$\Delta_2^{w_1}\Delta_4^{w_2}(A_{2,1}A_{2,2},f)=(w_1+w_2+1)!O\{\prod_{s=1}^{w_1+w_2}(f+2s)^{-1}\}$. 
%\bigskip

\subsubsection{Proof of Lemma \ref{lm:d24order} (on Page \pageref{lm:d24order})} \label{sec:d24order}
%Note that $D_{2,4}(f)=A_3(f-2)$. 
By the definition of $D_{2,4}(f)$, when $k_1+k_2\geq 1$, 
\begin{align*}
\Delta_6^{k_2}\Delta_4^{k_1}	(D_{2,4},f)=	x^2\Delta_6^{k_2}\Delta_4^{k_1}(A_{1,1}A_{1,2},f),
\end{align*}
where recall that we define $A_{1,t}=1/(f+2t)$ for integers $t$. 
By the Leibniz rule in Lemma \ref{lm:Leibnizfinite}, 
\begin{align*}
\Delta_6^{k_2}\Delta_4^{k_1}(A_{1,1}A_{1,2},f)=\sum_{s_2=0}^{k_2}\sum_{s_1=0}^{k_1}\binom{k_1}{s_1}\binom{k_2}{s_2} 	\Delta_6^{s_2}\Delta_4^{s_1}(A_{1,1},f)	\Delta_6^{k_2-s_2}\Delta_4^{k_1-s_1}(A_{1,2},f+4s_1+6k_2)
\end{align*}
Following the proof of Lemma \ref{lm:doublediffa1} in Section \ref{sec:doublediffa1}, we similarly have
\begin{align*}
\Delta_6^{s_2}\Delta_4^{s_1}(A_{1,1},f)=(s_1+s_2)!O\Biggr(C^{s_1+s_2} \prod_{k=1}^{s_1+s_2+1}\frac{1}{f+2k}\Biggr). 
\end{align*}
Then following the proof of Lemma \ref{lm:doublediffa2} in Section  \ref{sec:doublediffa2}, 
we obtain Lemma \ref{lm:d24order}. 
The analysis will be very similar and thus the details are skipped. 

%\bigskip

\subsubsection{Proof of Lemma \ref{lm:d26order} (on Page \pageref{lm:d26order})}\label{sec:d26order}
Note that we can write $D_{2,6}(f)=x^3 \prod_{k=1}^3 A_{1,k}(f)-1$. 
By the Leibniz rule in Lemma \ref{lm:Leibnizfinite}, 
%\begin{align*}
%	\Delta_6^{k_2}	(D_{2,6},f)=\sum_{s_1=0}^{k_2}\sum_{s_2=0}^{s_1} \binom{k_2}{s_1}\binom{s_1}{s_2} \Delta_6^{s_2}(A_{3,1},f)\Delta_6^{s_1-s_2}(A_{3,2},f+6s_2)\Delta_6^{k_2-s_1}(A_{3,3},f+6s_1)
%\end{align*}
\begin{align*}
\Delta_4^{k_1}\Delta_6^{k_2}(D_{2,6},f)=&~ \sum_{s_1=0}^{k_2}\sum_{s_2=0}^{s_1} \binom{k_1}{s_1}\binom{s_1}{s_2} \sum_{t_1=0}^{k_1} \sum_{t_2=0}^{t_1}\binom{k_2}{t_1}\binom{t_1}{t_2}x^3\times \Delta_4^{t_2} \Delta_6^{s_2}(A_{3,1},f) \notag \\
&~ \times \Delta_4^{t_1-t_2}\Delta_6^{s_1-s_2}(A_{3,2},f+6s_2+4t_2) \Delta_4^{k_1-t_1}\Delta_6^{k_2-s_1}(A_{3,3},f+6s_1+4t_1).
\end{align*}
Following the proof of Lemma \ref{lm:doublediffa1} in Section \ref{sec:doublediffa1}, we similarly have that 
for integers $t+s\geq 1$, and $l\in \{1,2,3\}$, 
\begin{align*}
	\Delta_4^t\Delta_6^s(A_{3,l})=(t+s)!O\left(C^{t+s}\prod_{m=1}^{t+s+1}\frac{1}{f+2m}\right).
\end{align*}
By $x=\chi^2_f(\alpha)=O(f)$, 
\begin{align*}
\Delta_4^{k_1}\Delta_6^{k_2}(D_{2,6},f)=&~ 	 \sum_{s_1=0}^{k_2}\sum_{s_2=0}^{s_1} \sum_{t_1=0}^{k_1} \sum_{t_2=0}^{t_1} \binom{k_1}{s_1}\binom{s_1}{s_2}  \binom{k_2}{t_1}\binom{t_1}{t_2} (t_2+s_2)!(t_1+s_1-t_2-s_2)! \notag \\
&~\times  (k_1+k_2-t_1-s_1) \times O\left( \prod_{m=1}^{k_1+k_2}\frac{1}{f+2m}\right). 
\end{align*}
Similarly to the proof of Lemma \ref{lm:doublediffa2} in Section \ref{sec:doublediffa2}, 
by the Chu–Vandermonde identity,  we obtain
\begin{align*}
\Delta_4^{k_1}\Delta_6^{k_2}(D_{2,6},f)=&~	\sum_{s_1=0}^{k_2} \sum_{t_1=0}^{k_1} k_1!k_2!\binom{k_1+k_2-s_1-t_1}{k_1-s_1}\binom{s_1+t_1}{s_1}(s_1+t_1+1)\notag \\
=&~(k_1+k_2+2)!\times O\left( \prod_{m=1}^{k_1+k_2}\frac{1}{f+2m}\right), 
\end{align*}
where we use $s_1+s_2+1\leq k_1+k_2+1$ in the second equation. 

%\begin{align*}
%\Delta_4^{k_1}\Delta_6^{k_2}(D_{2,6},f)=&~\sum_{s_1=0}^{k_2} \sum_{t_1=0}^{k_1} k_1!k_2!\binom{k_1+k_2-s_1-t_1}{k_1-s_1}\binom{s_1+t_1}{s_1}(s_1+s_2+1)\notag \\
%=&~(k_1+k_2+2)!\times O\left( \prod_{m=1}^{k_1+k_2}\frac{1}{f+2m}\right),
%\end{align*}

\subsubsection{Proof of Lemma \ref{lm:k1kwprodorder} (on Page \pageref{lm:k1kwprodorder})}\label{sec:k1kwprodorder}
We prove  Lemma \ref{lm:k1kwprodorder} similarly to the proof of Lemma \ref{lm:evalwqw2} in Section \ref{sec:evalwqw2} by discussing $k_1+k_2$ is odd and even, respectively. 

\smallskip

\noindent (1) When $k_1+k_2$ is odd,  similarly to the analysis of  \eqref{eq:upperbound1prod}, we assume without loss of generality that $k_2\geq k_1$, and obtain
\begin{align}
&~\frac{(k_1+k_2+2)!}{k_1!k_2!}O\left( f^{-(k_1+k_2+1)/2}\prod_{t=1}^{k_1+k_2}\frac{1}{f+2t}\right)\notag  \\
%\leq  &~(k_1+k_2+2)! \prod_{t=(k_1+k_2+1)/2+1}^{k_1+k_2} \frac{1}{2t} \notag \\
\leq &~\frac{2^{-(k_1+k_2-1)/2} }{k_1!k_2!} (k_1+k_2+2)(k_1+k_2+1) \prod_{t=1}^{(k_1+k_2+1)/2}t.\label{eq:k1k2bound0}
\end{align}
Note that
\begin{align}
&~\frac{(k_1+k_2+2)(k_1+k_2+1) }{k_1!k_2!}	 \prod_{t=1}^{(k_1+k_2+1)/2}t \notag \\
\leq &~C \left(  \frac{e^2}{k_1k_2} \frac{k_1+k_2+3}{2e} \right)^{k_1} \left(\frac{e^2}{k_2^2} \frac{k_1+k_2+3}{2e} \right)^{(k_2-k_1)/2}(k_1+k_2+3)^{5/2}\notag \\
\leq &~C \left(\frac{2e}{k_1} \right)^{k_1}	\left(\frac{2e}{k_2} \right)^{(k_2-k_1-5)/2}\left\{\frac{2e(k_1+k_2+3)}{k_2}\right\}^{5/2}. \label{eq:k1k2bound}
\end{align}
When $k_2-k_1\geq 5$, we can see that \eqref{eq:k1k2bound} is bounded. 
When $k_2-k_1\leq 4$, we have
\begin{align*}
\eqref{eq:k1k2bound}	\leq C\left(\frac{k_2}{k_1}\right)^{(5-k_2+k_1)/2}\left(\frac{2e}{k_1} \right)^{(k_1+k_2-5)/2}\left(\frac{k_1+k_2+3}{k_2}\right)^{5/2},
\end{align*}
which suggests that \eqref{eq:k1k2bound} is bounded. 
In summary, we know  \eqref{eq:k1k2bound} is bounded, and therefore $\eqref{eq:k1k2bound0}=O\{2^{-(k_1+k_2-1)/2} \}$. 	
\medskip

\noindent (2) When  $k_1+k_2$ is even, similar analysis can be applied, and then Lemma \ref{lm:k1kwprodorder} is proved. 

\medskip

\subsubsection{Proof of Lemma \ref{lm:generalcharexpan} (on Page \pageref{lm:generalcharexpan})}\label{sec:generalcharexpan}
We prove  Lemma \ref{lm:generalcharexpan}  based on  \eqref{eq:logexpgeneralform}. 
In each testing problem, we have $|\tau_{1,k}+ \upsilon_{1,k}|/|\eta \xi_{1,k}|=o(1)$; see Sections \ref{sec:pfonesamchisq1}--\ref{sec:pfindpchisq}.
Then under the conditions of Lemma \ref{lm:generalcharexpan}, we can apply Lemma \ref{lm:logzaexpan} and obtain for $1\leq k\leq K_1$, 
\begin{align*}
&~	\log \Gamma\big\{ \eta \xi_{1,k}(1-2it) +\tau_{1,k}+ \upsilon_{1,k}  \big\} \notag \\
=&~	\biggr\{ \eta \xi_{1,k}(1-2it) +\tau_{1,k}+ \upsilon_{1,k} -\frac{1}{2}  \biggr\}\log \big\{\eta \xi_{1,k}(1-2it) \big\}- \eta \xi_{1,k}(1-2it) +\log \sqrt{2\pi}\notag \\
&~+\sum_{l=1}^{L-1}\frac{(-1)^{l+1}B_{l+1}(\tau_{1,k} + \upsilon_{1,k})}{l(l+1)}\Big\{ \eta \xi_{1,k}(1-2it)\Big\}^{-l}+O\Big( |\tau_{1,k}+\upsilon_{1,k}|^{L+1}/|\eta \xi_{1,k}|^{L}\Big).
\end{align*}
Applying similar expansion to $\log \Gamma( \eta\xi_{1,k} + \tau_{1,k}+\upsilon_{1,k})$, we obtain
\begin{align*}
&~ \log \Gamma\big\{ \eta \xi_{1,k}(1-2it) +\tau_{1,k}+ \upsilon_{1,k}  \big\} -\log \Gamma\big(\eta \xi_{1,k}+\tau_{1,k}+ \upsilon_{1,k}  \big) \notag \\
=&~\left(\eta \xi_{1,k} +\tau_{1,k}+\upsilon_{1,k}-\frac{1}{2}\right)\log(1-2it)-2it\eta \xi_{1,k}\log \big\{ \eta \xi_{1,k} (1-2it ) \big\}+2it\eta \xi_{1,k} \notag \\
&~ +\sum_{l=1}^{L-1}\frac{(-1)^{l+1}B_{l+1}(\tau_{1,k} + \upsilon_{1,k})}{l(l+1)( \eta \xi_{1,k})^{l}}\Big\{(1-2it)^{-l}-1\Big\}+O\Big( |\tau_{1,k}+\upsilon_{1,k}|^{L+1}/|\eta \xi_{1,k}|^{L}\Big).
\end{align*}
Similarly, for $1\leq j\leq K_2$, we have
\begin{align*}
&~\log \Gamma\big\{ \eta \xi_{2,j}(1-2it) + \tau_{2,j}+\upsilon_{2,j}\big\}-\log \Gamma\big( \eta \xi_{2,j} + \tau_{2,j}+\upsilon_{2,j}\big)\notag \\
=&~\biggr( \eta \xi_{2,j} + \tau_{2,j}+\upsilon_{2,j}-\frac{1}{2}\biggr)\log(1-2it)-2it\eta \xi_{2,j}\log\big\{\eta \xi_{2,j}(1-2it) \big\}+2it\eta \xi_{2,j}\notag \\
&~+\sum_{l=1}^{L-1}\frac{(-1)^{l+1}B_{l+1}(\tau_{2,j} + \upsilon_{2,j})}{l(l+1)( \eta \xi_{2,j})^{l}}\Big\{(1-2it)^{-l}-1\Big\}+O\Big( |\tau_{2,j}+\upsilon_{2,j}|^{L+1}/|\eta \xi_{2,j}|^{L}\Big).
\end{align*}
Then by the form of $\varphi(t)$ in \eqref{eq:logexpgeneralform}, we  calculate
\begin{align*}
\eqref{eq:logexpgeneralform}=&~2it\eta \left( \sum_{k=1}^{K_1}\xi_{1,k}\log \xi_{1,k} - \sum_{j=1}^{K_2} \xi_{2,j} \log \xi_{2,j}\right) \notag \\
&~+\biggr\{\sum_{k=1}^{K_1}(\xi_{1,k} +\tau_{1,k}+\upsilon_{1,k}-1/2) - \sum_{j=1}^{K_2}(\xi_{2,j}+\tau_{2,j}+\upsilon_{2,j}-1/2)\biggr\}\log(1-2it) \notag \\
&~ -2it\eta \left(\sum_{k=1}^{K_1}\xi_{1,k}\log \xi_{1,k} - \sum_{j=1}^{K_2}\xi_{2,j}\log \xi_{2,j} \right) -2it\eta  (\log \eta - 1)\left(\sum_{k=1}^{K_1}\xi_{1,k}-\sum_{j=1}^{K_2}\xi_{2,j}\right)\notag \\
&~ +\sum_{l=1}^{L-1}\varsigma_l\Big\{(1-2it)^{-l}-1\Big\} +O\Biggr( \sum_{k=1}^{K_1}\frac{|\tau_{1,k}+\upsilon_{1,k}|^{L+1}}{|\eta \xi_{1,k}|^{L}}+ \sum_{j=1}^{K_2} \frac{|\tau_{2,j}+\upsilon_{2,j}|^{L+1}}{|\eta \xi_{2,j}|^{L}}\Biggr). 
\end{align*}
%\begin{align*}
%\eqref{eq:logexpgeneralform}=&~\biggr(\sum_{k=1}^{K_1}\upsilon_{1,k}-\sum_{j=1}^{K_2}\upsilon_{1,j}-\frac{K_1-K_2}{2} \biggr)\log (1-2it)+\sum_{l=1}^{L-1}\varsigma_l\Big\{(1-2it)^{-l}-1\Big\} \notag \\
%&~+\, O\Biggr( \sum_{k=1}^{K_1}\frac{|\tau_{1,k}+\upsilon_{1,k}|^{L+1}}{|\eta \xi_{1,k}|^{L}}+ \sum_{j=1}^{K_2} \frac{|\tau_{2,j}+\upsilon_{2,j}|^{L+1}}{|\eta \xi_{1,k}|^{L}}\Biggr). 
%\end{align*}
By the facts that $\tau_{1,k}=\eta \xi_{1,k}$, $\tau_{2,j}=\eta \xi_{2,j}$, and $\sum_{k=1}^{K_1}\xi_{1,k}=\sum_{j=1}^{K_2}\xi_{2,k}$, Lemma  \ref{lm:generalcharexpan}  is proved.

\bigskip
\subsection{Lemmas for Theorems \ref{thm:onesamnormal},  \ref{thm:multsamnormal}, \& \ref{thm:indepnormal}} \label{sec:normallemmas}

\subsubsection{Proof of Lemma \ref{lm:chardiffgoal} (on Page \pageref{lm:chardiffgoal})}\label{sec:chardiffgoal}

By \eqref{eq:fsform} on Page \pageref{eq:fsform}, 
\begin{align*}
\log \psi_1(s)=-\frac{pnti}{2}	\log \frac{2e}{n}-\frac{pn(1-ti)}{2}\log (1-ti) + \log \frac{\Gamma_{p}\{(n-1)/2-nti/2\}}{\Gamma_{p}\{(n-1) / 2\}} + \mu_n ti,
\end{align*}
where $t=s/(n\sigma_n)$.
We next examine $\log \psi_1(s)$ by the following Lemma \ref{lm:gammapaproxexpan}. 
%\begin{align}
%&~\mathrm{E}\Biggr\{\exp\left(is\times \frac{ -2\log \Lambda_n + 2\mu_n}{2n\sigma_n} \right) \Biggr\} \notag \\
%=&~\left(\frac{2 e}{n}\right)^{-n p t i  / 2}(1-ti)^{-n p(1-t i) / 2} \frac{\Gamma_{p}\{(n(1-ti)-1) / 2\}}{\Gamma_{p}\{(n-1) / 2\}}\exp \left(\frac{\mu_ns i}{n\sigma_n}\right),\notag  
%\end{align} 

\begin{lemma} \label{lm:gammapaproxexpan}
Let $\left\{p=p_{n}; n \geq 1\right\},\left\{m=m_{n}; n \geq 1\right\}$, $\{t_n; n\geq 1\}$, and $\{s_n; n\geq 1\}$ satisfy that (i) $p_n\to \infty$ and $p_n=o(n)$; (ii) there exists $\epsilon \in (0,1)$ such that $\epsilon \leq m_{n} / n \leq \epsilon^{-1}$;  (iii)  $t=t_n=O(ns/p)$;  (iv) $s=s_n=o( \min\{ (n/p)^{1/2}, f^{1/6}\} )$. Then as $n\to \infty$, 
\begin{align}
&~\log \frac{\Gamma_{p}\left(\frac{m-1}{2}+ti\right)}{\Gamma_{p}\left(\frac{m-1}{2}\right)}\notag \\
=&~\beta_{m,1} ti-\beta_{m,2} t^{2}+\beta_{m,3}(ti)+O\left(\frac{p^2t}{m^2}\right) +\left(\frac{1}{p}+\frac{p}{m} \right)O\left(\frac{p^2t^2}{m^2}\right)+ O\Big( \frac{p^2 t^3}{m^3} \Big), \notag %\label{eq:gammplogapprox}  
\end{align}
where 
\begin{align*}
&\beta_{m,1}=-\left\{2 p+\left(m-p-\frac{3}{2}\right) \log \left(1-\frac{p}{m-1}\right)\right\}; \quad \beta_{m,2}=-\left\{\frac{p}{m-1}+\log \left(1-\frac{p}{m-1}\right)\right\}; \\
&\beta_{m,3}(ti)=p\left\{\left(\frac{m-1}{2}+ti\right) \log \left(\frac{m-1}{2}+ti\right)-\frac{m-1}{2} \log \frac{m-1}{2}\right\}. 
\end{align*}
\end{lemma}
\begin{proof}
Please see Section \ref{sec:gammapaproxexpan} on Page \pageref{sec:gammapaproxexpan}. 
\end{proof}

\medskip

\noindent 
By \eqref{eq:chisqconvgvar} and $f=\Theta(p^2)$, we know $t=s/(n\sigma_n)=O(s/p)$. 
Thus we can apply Lemma \ref{lm:gammapaproxexpan} and expand 
%By Lemma \ref{lm:gammapaproxexpan}, 
\begin{align*}
 \log \frac{\Gamma_{p}\{(n-1)/2-nti/2\}}{\Gamma_{p}\{(n-1) / 2\}}=&~-\frac{n\beta_{n,1}ti}{2}-\frac{\beta_{n,2}n^2t^2}{4}+\beta_{n,3}\left(-\frac{nti}{2}\right)\notag \\
 &~+O\left(\frac{p^2t}{n}\right) +\left(\frac{1}{p}+\frac{p}{n} \right)O\left(p^2t^2\right)+ O\big({p^2 t^3} \big).
\end{align*}
\smallskip

\noindent We next use the following Lemma \ref{lm:betan3order} to evaluate $\beta_{n,3}(-nti/2)$. 

\begin{lemma}\label{lm:betan3order}
When $p=p_n\to \infty$, $p=o(n)$, and $t=t_n=O(s/p)$ with $s=s_n=o( \min\{ (n/p)^{1/2}, f^{1/6}\} )$, 
\begin{align*}
\beta_{n,3}\biggr(-\frac{nti}{2}\biggr)=-\frac{pnti}{2}\log\frac{n}{2}+\frac{pn(1-ti)}{2}\log(1-ti)+\frac{pti}{2}+O\biggr( pt^2+\frac{pt}{n}\biggr).
\end{align*}
%\begin{align*}
%\beta_{n,3}=-\frac{pnti}{2}\log\frac{n}{2}+\frac{pn(1-ti)}{2}\log(1-ti)+\frac{pti}{2}+O\biggr( pt^2+\frac{pt}{n}\biggr).
%\end{align*}	
\end{lemma}
\begin{proof}
Please see Section \ref{sec:betan3order} on Page \pageref{sec:betan3order}. 
\end{proof}

\medskip

\noindent It follows that 
\begin{align*}
\log \psi_1(s)=&~-\frac{\{p(n-1)+n\beta_{n,1}\}ti}{2}-\frac{\beta_{n,2}n^2t^2}{4} +\mu_n ti \notag \\
&~+O\biggr(\frac{p^2t}{n}\biggr) +\left(\frac{1}{p}+\frac{p}{n} \right)O\left(p^2t^2\right)+ O\big({p^2 t^3} \big).
\end{align*}
Since $\sigma_n^2=\beta_{n,2}/2$, $\mu_n=\{p(n-1)+n\beta_{n,1}\}/2$,  and $t=s/(n\sigma_n)$, 
\begin{align*}
	\log \psi_1(s)=-\frac{s^2}{2}+O\biggr(\frac{ps}{n}\biggr)+\left(\frac{1}{p}+\frac{p}{n} \right)O(s^2)+O\biggr(\frac{s^3}{p}\biggr),
\end{align*}
where we  use $t=O(s/p)$. 
As $\log \psi_0(s)=-s^2/2$, \eqref{eq:chardiffgoal} is proved. 

\subsubsection{Proof of Lemma \ref{lm:gammapaproxexpan} (on Page \pageref{lm:gammapaproxexpan})}\label{sec:gammapaproxexpan}

By the property of the multivariate gamma function; see, e.g.,  Theorem 2.1.12 in  \citet{Muirhead2009}, 
\begin{align}
\log \frac{\Gamma_{p}\left(\frac{m-1}{2}+ ti\right)}{\Gamma_{p}\left(\frac{m-1}{2}\right)} 
=	\sum_{j=1}^{p} \log \frac{\Gamma\left(\frac{m-j}{2}+ti\right)}{\Gamma\left(\frac{m-j}{2}\right)}. \label{eq:multivgammasumexpan} 	
\end{align}
Then by Lemma \ref{lm:ratiogammapprox} on Page \pageref{lm:ratiogammapprox}, 
\begin{align}
	\log \frac{\Gamma\left(\frac{m-j}{2}+ti\right)}{\Gamma\left(\frac{m-j}{2}\right)}=&~	\sum_{j=1}^p\Biggr[\left(\frac{m-j}{2}+ti\right)\log\left(\frac{m-j}{2}+ti\right) -  \left(\frac{m-j}{2}\right)\log\left(\frac{m-j}{2}\right)\label{eq:onegammajexpan} \\
&~ \quad \quad  -ti - \frac{ti}{m-j} + O\biggr\{ \frac{t + t^2}{(m-j)^2} \biggr\}\Biggr], \notag
\end{align}
as $m\to \infty$ uniformly for all $1\leq j\leq p.$
Note that $t/(m-j)=t/m+ (t/m)\times \{j/(m-j)\}$, and then 
\begin{align}
	\sum_{j=1}^p \frac{ti}{m-j} = \frac{pti}{m} + O\left(\frac{p^2}{m^2} \right) ti. \label{eq:sumjtiapprox}
\end{align}
By \eqref{eq:onegammajexpan} and \eqref{eq:sumjtiapprox}, we obtain as $m\to \infty$, 
\begin{align}
\eqref{eq:multivgammasumexpan} =&~	\sum_{j=1}^p\Biggr[\left(\frac{m-j}{2}+ti\right)\log\left(\frac{m-j}{2}+ti\right) -  \left(\frac{m-j}{2}\right)\log\left(\frac{m-j}{2}\right)\Biggr] \label{eq:summ1g} \\
&~ \quad \quad -\frac{(m+1)pti}{m}+ O\left(\frac{p^2}{m^2}t+\frac{p}{m^2}t^2\right). \notag
\end{align}
For $1\leq j\leq p$, define 
\begin{align*}
g_j(z)=\left(\frac{m-j}{2}+z\right) \log \left(\frac{m-j}{2}+z\right)-\left(\frac{m-1}{2}+z\right) \log \left(\frac{m-1}{2}+z\right),	
\end{align*} where the real part of $z > - (m-p)/2$. 
It follows that the ``$\sum_{j=1}^p$'' term in the first row of \eqref{eq:summ1g} is equal to 
\begin{align}
&~p \left\{\left(\frac{m-1}{2}+ti\right) \log \left(\frac{m-1}{2}+ti\right)-\frac{m-1}{2} \log \frac{m-1}{2}\right\}+\sum_{j=1}^p \{g_j(ti)-g_j(0) \}. \label{eq:summ1g2}
\end{align}
To evaluate \eqref{eq:summ1g2}, we use the following Lemma \ref{lm:gjsumapprox}. 
\begin{lemma}\label{lm:gjsumapprox}
Let $p=p_m$ such that $1\leq p < m$, $p\to \infty$ and $p/m \to 0$ as $m\to \infty$. 
%For $1\leq j\leq p$, define 
%\begin{align*}
%g_j(z)=\left(\frac{m-j}{2}+z\right) \log \left(\frac{m-j}{2}+z\right)-\left(\frac{m-1}{2}+z\right) \log \left(\frac{m-1}{2}+z\right)	,
%\end{align*}where the real part of $z>-(m-p)/2$. 
%Let 
When $t=t_m=O(ms/p)$ with $s=s_m=o( \min\{ (m/p)^{1/2}, p^{1/3}\})$,  we have that, as $m\to \infty$,
%When $t=t_m=O(ms/\sqrt{f})$ with $f=f_m\to \infty$ and  $s=s_m=o( \min\{ (m/p)^{1/2}, f^{1/6}\})$,  we have that, as $m\to \infty$,
\begin{align*}
\sum_{j=1}^p\{g_j(ti)-g_j(0)\}=\nu_{1,m} t i-\frac{\nu_{2,m}^2}{2}t^2 + O\left(\frac{p^2t}{m^2}\right) +\left(\frac{1}{p}+\frac{p}{m} \right)O\left(\frac{p^2t^2}{m^2}\right)+ O\Big( \frac{p^2 t^3}{m^3} \Big),
\end{align*} 
where
\begin{align}
	\nu_{1,m}=&~\left( p-m+\frac{3}{2}\right)\log \left(1-\frac{p}{m-1}\right)-\frac{m-1}{m}p,\label{eq:defsigmammum}  \\ 
	\nu_{2,m}^2=&~-2\left\{\frac{p}{m-1}+\log \left(1-\frac{p}{m-1}\right) \right\}.\notag
\end{align} 	
\end{lemma}
\begin{proof}
Please see Section \ref{sec:gjsumapprox}	 on Page \pageref{sec:gjsumapprox}. 
\end{proof}
\bigskip

\noindent Then by Lemma \ref{lm:gjsumapprox}, 
\begin{align*}
\eqref{eq:summ1g2}=&~p \left\{\left(\frac{m-1}{2}+ti\right) \log \left(\frac{m-1}{2}+ti\right)-\frac{m-1}{2} \log \frac{m-1}{2}\right\} \notag \\
&~ + \nu_{1,m} t i-\frac{\nu_{2,m}^2}{2}t^2 + O\left(\frac{p^2t}{m^2}\right) +\left(\frac{1}{p}+\frac{p}{m} \right)O\left(\frac{p^2t^2}{m^2}\right)+ O\Big( \frac{p^2 t^3}{m^3} \Big).	
\end{align*}
In summary, 
%\eqref{eq:gammplogapprox} 
Lemma \ref{lm:gammapaproxexpan} can be proved by noticing 
\begin{align*}
&\beta_{m,1}=\nu_{1,m}-\frac{(m+1)p}{m}, \quad \quad \beta_{m,2}=\nu_{2,m}^2/2 \notag \\
&\beta_{m,3}(ti)=	p \left\{\left(\frac{m-1}{2}+ti\right) \log \left(\frac{m-1}{2}+ti\right)-\frac{m-1}{2} \log \frac{m-1}{2}\right\}.
\end{align*}

%\begin{align*}
%&\beta_{m,1}=-\left\{2 p+\left(m-p-\frac{3}{2}\right) \log \left(1-\frac{p}{m-1}\right)\right\}; \quad \beta_{m,2}=-\left\{\frac{p}{m-1}+\log \left(1-\frac{p}{m-1}\right)\right\}; \\
%&\beta_{m,3}(ti)=p\left\{\left(\frac{m-1}{2}+ti\right) \log \left(\frac{m-1}{2}+ti\right)-\frac{m-1}{2} \log \frac{m-1}{2}\right\}. 
%\end{align*}

\medskip

\subsubsection{Proof of Lemma \ref{lm:betan3order} (on Page \pageref{lm:betan3order})} \label{sec:betan3order}
%\begin{align*}
%\beta_{n,3}\left(-\frac{nti}{2}\right)=&~p\left\{\left(\frac{n-1}{2}-\frac{nti}{2}\right) \log \left(\frac{n-1}{2}-\frac{nti}{2}\right)-\frac{n-1}{2} \log \frac{n-1}{2}\right\}	
%\end{align*}
By Taylor's series,
\begin{align*}
p^{-1}\beta_{n,3}(-nti/2)=&~	 -\frac{nti}{2}\log \frac{n}{2}	-\frac{nti}{2} \log \left(1-ti-\frac{1}{n}\right) + \frac{n-1}{2}\log\left( 1-ti-\frac{ti}{n-1} \right) \notag \\
=&~	-\frac{nti}{2}\log \frac{n}{2}-\frac{nti}{2}\log(1-ti)+\frac{nti}{2n(1-ti)}+O\biggr( \frac{nt}{n^2} \biggr)\notag \\
&~+\frac{n-1}{2}\log(1-ti)-\frac{n-1}{2}\frac{ti}{(n-1)(1-ti)}+\frac{n-1}{2}O\biggr(\frac{t^2}{n^2}\biggr)\notag \\
=&~-\frac{nti}{2}\log \frac{n}{2}+\frac{n(1-ti)-1}{2}\log(1-ti)+O\biggr(\frac{t+t^2}{n}\biggr).
\end{align*}
It follows that
\begin{align*}
	\beta_{n,3}(-nti/2)=-\frac{pnti}{2}\log \frac{n}{2}+\frac{pn(1-ti)}{2}\log(1-ti) + \frac{pti}{2}+O\biggr( pt^2+\frac{pt}{n}\biggr).
\end{align*}
%where we  $pt^2=o(1)$ and $pt/n=o(1)$. 
%and Lemma  \ref{lm:betan3order} is proved by $t=O(s/p)$. 

\medskip

\subsubsection{Proof of Lemma \ref{lm:gjsumapprox} (on Page \pageref{lm:gjsumapprox})} \label{sec:gjsumapprox}	

The first-order derivatives of $g_j(z)$ is
\begin{align*}
	g_j^{(1)}(z)=&~ \log\left( \frac{m-j}{2} +z \right)-\log\left( \frac{m-1}{2} +z \right), \notag 
\end{align*}
and for $l\geq 2$, the $l$-th order derivatives of $g_j(z)$ is
\begin{align*}
g_j^{(l)}(z)=&~(-1)^{l-2}(l-2)!\left\{\left( \frac{m-j}{2}+z\right)^{-(l-1)} -\left( \frac{m-1}{2}+z\right)^{-(l-1)} \right\}\notag \\
=&~ (-1)^{l-2}(l-2)!\left( \frac{m-1}{2}+z\right)^{-(l-1)} \sum_{v=1}^{l-1}\binom{l-1}{v}\left( \frac{j-1}{m-j+2z} \right)^v. 
\end{align*}
By Taylor's expansion, 
%$g_j(ti)-g_j(0)=\sum_{l=1}^{2}g_j^{(l)}(0) z^l/l! + O\{ z^3\}$. 
$g_j(ti)-g_j(0)=\sum_{l=1}^{\infty}g_j^{(l)}(0) z^l/l!$. 
In particular, 
\begin{align*}
g_j^{(1)}(0)=\log (m -j)	 - \log (m-1), \quad \quad  g_j^{(2)}(0)=\frac{2}{m-j}-\frac{2}{m-1}.
\end{align*}
When $z=ti$, $t=t_m=O(ms/p)$, and $l\geq 3$, as $j-1/(m-j+2z)=O(p/m)=o(1)$, 
\begin{align*}
g_j^{(l)}(0)z^l/l!=O\biggr(\frac{1}{m^{l-1}}\frac{p}{m} t^l \biggr)=O \biggr( \frac{p}{m^l} \biggr) t^l.
\end{align*} 
As $t/m=O(s/p)=o(1)$, 
\begin{align*}
\sum_{j=1}^p \{ g_j(ti)-g_j(0)\}=\sum_{j=1}^p g_j^{(1)}(0) ti -\frac{1}{2}\sum_{j=1}^p g_j^{(2)}(0)t^2+ O\Big( \frac{p^2 t^3}{m^3} \Big). 
\end{align*}
By Lemma A.2 in \cite{Jiang15}, 
\begin{align*}
\sum_{j=1}^p g_j^{(1)}(0)=\nu_{1,m} + O(\nu_{2,m}^2), \quad \quad \sum_{j=1}^p g_j^{(2)}(0)=\nu_{2,m}^2\biggr\{1+O\left(\frac{1}{p}+\frac{p}{m} \right) \biggr\}, 
\end{align*}
where $\nu_{1,m}$ and $\nu_{2,m}^2$ are defined in \eqref{eq:defsigmammum}.
In summary,
\begin{align*}
\sum_{j=1}^p\{g_j(ti)-g_j(0)\}=\nu_{1,m} t i-\frac{\nu_{2,m}^2}{2}t^2 + O(\nu_{2,m}^2)t +\nu_{2,m}^2O\left(\frac{1}{p}+\frac{p}{n} \right)t^2+ O\Big( \frac{p^2 t^3}{m^3} \Big). 
\end{align*}
Then Lemma \ref{lm:gjsumapprox} follows by $\nu_{2,m}^2=O(p^2/m^2).$

\subsubsection{Proof of Lemma \ref{lm:rhonksumm} (on Page \pageref{lm:rhonksumm})}\label{sec:rhonksumm}
By Taylor's series, we have \eqref{eq:rhonksumm1}. 
In addition, for \eqref{eq:rhonksumm2}, 
note that we can write
\begin{align*}
p^{-1}\varrho_l(t)=\frac{l-1}{2}\log\left(1+\frac{lt}{l-1}\right) +\frac{lt}{2}\log \left(\frac{l-1}{2}+\frac{lt}{2} \right). \notag	
\end{align*}
By Taylor's series
%\begin{align*}
%\log x = \log a + \sum_{l=1}^{\infty} \frac{(-1)^{l-1}}{la^l} (x-a)^{l}	
%\end{align*}
$
\log x = \log a + \sum_{l=1}^{L-1} {(-1)^{l-1}}l^{-1}(x/a-1)^{l}	+O\{( x/a-1)^L\}
$, we obtain 
%\begin{align*}
%x\log x=a\log a -\frac{1}{2}(1+\log a)+\sum_{l=2}^{\infty}\frac{(-1)^l}{la^{l-1}}(x-a)^l	
%\end{align*}
\begin{align*}
%p^{-1}\varrho_l(t)=&~\frac{l-1}{2}\log\left(1+\frac{lt}{l-1}\right) +\frac{lt}{2}\log \left(\frac{l-1}{2}+\frac{lt}{2} \right) \notag \\
\frac{\varrho_l(t)}{p}=&~\frac{l}{2}\log \left(1 +t+\frac{t}{l-1}\right)-\frac{1}{2}\log\left(1+\frac{lt}{l-1}\right)+\frac{lt}{2}\log \left\{\frac{l(1+t)}{2}-\frac{1}{2} \right\}\notag \\
=&~\frac{l}{2}\log(1+t)+ \frac{lt}{2(l-1)(1+t)} -\frac{lt}{2(l-1)}+ \frac{lt}{2}\log \frac{l(1+t)}{2}-\frac{t}{2(1+t)}+O\left(\frac{t}{l}+t^2\right)\notag \\
=&~\frac{l(1+t)}{2}\log(1+t)+\frac{lt}{2}\log\frac{l}{2}-\frac{t}{2}+O\left(\frac{t}{l}+t^2\right).
\end{align*}
Then by $n=\sum_{j=1}^k n_j$, we have
\begin{align*}
-\varrho_n(t)+\sum_{j=1}^k \varrho_{n_j}(t)	=\biggr( 1-k-n\log n+\sum_{j=1}^k n_j\log n_j\biggr)\frac{tp}{2}+O\left(\frac{pt}{n}+pt^2\right).
\end{align*}
%In summary, as $p/n\to 0$, the remainder terms in (A.19), (A.20) and Lemma A.7 of \cite{Jiang15} are $O(ps/n)+O(1/p+p/n)s^2$. 

\vspace{2em}

\bibliographystyle{chicago}
\bibliography{testrefer}

\end{document}